\documentclass[10pt,a4paper,twoside,openright,english]{book}

\usepackage{tesi}
\usepackage{microtype}
\usepackage{fullpage}
\usepackage{amsmath,amsfonts,amssymb,amsthm}
\usepackage[utf8x]{inputenc}
\usepackage[english]{babel}
\usepackage[T1]{fontenc}
\usepackage{lmodern}
\usepackage{textcomp}
\usepackage[nottoc]{tocbibind}
\usepackage{titletoc}
\usepackage{graphicx}
\usepackage[all]{xy}
\usepackage{latexsym}
\usepackage{fancyhdr}
\usepackage{lscape}
\usepackage{url}
\usepackage{ifthen}
\usepackage{booktabs}
\usepackage{paralist}
\usepackage{array}
\usepackage{multirow}
\usepackage{color}
\usepackage{dcolumn}

\makeatletter
\def\thickhrulefill{\leavevmode \leaders \hrule height 1ex \hfill \kern \z@}
\def\@makechapterhead#1{%
  \vspace*{10\p@}%
  {\parindent \z@ \raggedleft \reset@font
            {\scshape \large \@chapapp{} \thechapter}
        \par\nobreak
        \interlinepenalty\@M
    \vspace{12pt}
    \huge \bfseries #1\par\nobreak
    \hrulefill
    \par\nobreak
    \vskip 40\p@
  }}
\def\@makeschapterhead#1{%
  \vspace*{10\p@}%
  {\parindent \z@ \raggedleft \reset@font
            {\scshape \large \vphantom{\@chapapp{} \thechapter}}
        \par\nobreak
        \interlinepenalty\@M
    \vspace{12pt}
    \huge \bfseries #1\par\nobreak
    \hrulefill
    \par\nobreak
    \vskip 40\p@
  }}

\newcommand{\K}{{\mathbb{K}}}
\newcommand{\Z}{{\mathbb{Z}}}
\newcommand{\C}{{\mathbb{C}}}
\newcommand{\R}{{\mathbb{R}}}
\newcommand{\Q}{{\mathbb{Q}}}
\newcommand{\N}{{\mathbb{N}}}
\newcommand{\CP}{{\mathbb{CP}}}
\newcommand{\T}{{\mathbb{T}}}

\newcommand{\st}{\;:\;}

\newcommand{\pf}{pure-and-full}

\newcommand{\p}{pure}
\newcommand{\f}{full}
\newcommand{\Cpf}{$\mathcal{C}^\infty$-\pf}

\newcommand{\Cp}{$\mathcal{C}^\infty$-\p}
\newcommand{\Cf}{$\mathcal{C}^\infty$-\f}

\newcommand{\module}[1]{\left|#1\right|}
\newcommand{\phit}[1]{\varphi^{#1}_{\mathbf{t}}}
\newcommand{\bphit}[1]{\bar{\varphi}^{#1}_{\mathbf{t}}}

\newcommand{\supp}{\textrm{supp}\,}

\newcommand{\card}{\textrm{card}\,}

\newcommand{\sspace}{\cdot}
\newcommand{\ssspace}{{\cdot\cdot}}

\def\Re{{\mathfrak{Re}}}
\def\Im{{\mathfrak{Im}}}

\def\Hom{{\mathrm{Hom}}}

\newcommand{\Nij}{{\mathrm{Nij}}}

\newcommand{\Cinf}{{\mathcal{C}^{\infty}}}

\newcommand{\correnti}{{\mathcal{D}}}

\newcommand{\scalar}[2]{\left\langle #1 \;|\; #2 \right\rangle}

\newcommand{\del}{\partial}
\newcommand{\delbar}{{\overline\partial}}

\newcommand{\duale}[1]{{#1}^{*}}

\newcommand{\trasposta}[1] {{{#1}^\mathrm{t}}}

\newcommand{\Lie}[1] {{\mathcal{L}_{#1}}}

\newcommand{\g}{\mathfrak{g}}
\newcommand{\h}{\mathfrak{h}}

\DeclareMathOperator{\id}{id}

\DeclareMathOperator{\tr}{tr}

\DeclareMathOperator{\End}{End}

\DeclareMathOperator{\rk}{rk}

\DeclareMathOperator{\im}{i}

\DeclareMathOperator{\imm}{im}

\DeclareMathOperator{\vol}{vol}

\DeclareMathOperator{\SL}{SL}
\DeclareMathOperator{\GL}{GL}

\DeclareMathOperator{\ad}{ad}

\DeclareMathOperator{\esp}{e}

\DeclareMathOperator{\pr}{pr}

\DeclareMathOperator{\de}{d}
\DeclareMathOperator{\Vol}{Vol}
\DeclareMathOperator{\interno}{int}

\DeclareMathOperator{\Sing}{Sing}

\def\zero{\mathbf{0}}

\theoremstyle{plain}
\newtheorem{thm}{Theorem}[chapter]
\newtheorem{lem}[thm]{Lemma}
\newtheorem{prop}[thm]{Proposition}
\newtheorem{cor}[thm]{Corollary}
\newtheorem{conj}[thm]{Conjecture}
\newtheorem{quest}[thm]{Question}

\theoremstyle{definition} 
\newtheorem{defi}[thm]{Definition}
\newtheorem{ex}[thm]{Example}
\newtheorem{rem}[thm]{Remark}

\theoremstyle{remark} 
\newtheorem{question}[thm]{Question}

\newcommand{\I}{\mathbb{I}}
\newcommand{\para}{$\mathbf{D}$}

\newcommand{\cp}{\smile}
\newcommand{\kal}{K\"{a}hler}

\newcommand{\proj}{\pi}

\newcommand{\kth}[1]{\ifthenelse{\equal{#1}{1}}{$#1^\text{st}$}{\ifthenelse{\equal{#1}{2}}{$#1^\text{nd}$}{\ifthenelse{\equal{#1}{3}}{$#1^\text{rd}$}{$#1^\text{th}$}}}}

\newcommand{\formepm}[2]{\wedge^{#1,\,#2}_{+\,-}\,}
\newcommand{\correntipm}[2]{\correnti_{#1,\,#2}^{+\,-}\,}
\newcommand{\correntipmalto}[2]{\correnti^{#1,\,#2}_{+\,-}\,}

\newcommand{\modulo}{\,\textrm{mod}\,}

\newcommand{\opiccolo}[1]{\mathrm{o}\left(\left|#1\right|\right)}

\newcommand{\opiccolouno}{\mathrm{o}\left(1\right)}

\newcommand{\tempo}{\mathbf{t}}

\newcommand{\Prim}{{\mathrm{P}\wedge}}
\newcommand{\Primcorrenti}{{\mathrm{P}\correnti}}
\newcommand{\PH}{\mathrm{P}H}

\newcommand{\normaL}[2]{\left\|#1\right\|_{\mathrm{L}^2_{#2}}}
\newcommand{\normazeroL}[1]{\left\|#1\right\|_{\mathrm{L}^2}}
\newcommand{\scalarL}[3]{\left\langle  #1 \,\left|\, #2  \right. \right\rangle_{\mathrm{L}^2_{#3}}}
\newcommand{\scalardL}[2]{\left\langle #1 \,\left|\, #2 \right. \right\rangle}

\newcommand{\ssum}[1]{\widetilde{\sum_{#1}}}

\newcommand{\Leb}[2]{\ifthenelse{\equal{#2}{0}}{\mathrm{L}^{2}\left(X;\wedge^{#1}T^*X\right)}{\mathrm{L}^{2}_{#2}\left(X;\wedge^{#1}T^*X\right)}}
\newcommand{\LebK}[2]{\ifthenelse{\equal{#2}{0}}{\mathrm{L}^{2}\left(K;\wedge^{#1}T^*K\right)}{\mathrm{L}^{2}_{#2}\left(K;\wedge^{#1}T^*K\right)}}
\newcommand{\Lebloc}[1]{\mathrm{L}^{2}_{\text{loc}}\left(X;\wedge^{#1}T^*X\right)}
\newcommand{\Cinfk}[1]{\ifthenelse{\equal{#1}{0}}{\mathcal{C}^\infty\left(X;\R\right)}{\mathcal{C}^\infty\left(X;\wedge^{#1}T^*X\right)}}
\newcommand{\Cinfkc}[1]{\ifthenelse{\equal{#1}{0}}{\mathcal{C}^\infty_{\mathrm{c}}\left(X;\R\right)}{\mathcal{C}^\infty_{\mathrm{c}}\left(X;\wedge^{#1}T^*X\right)}}
\newcommand{\Sob}[3]{\ifthenelse{\equal{#2}{0}}{\mathrm{W}^{#3, 2}\left(X;\wedge^{#1}T^*X\right)}{\mathrm{W}^{#3, 2}_{#2}\left(X;\wedge^{#1}T^*X\right)}}
\newcommand{\Sobloc}[2]{\mathrm{W}^{#1, 2}_{\textrm{loc}}\left(X;\wedge^{#2}T^*X\right)}
\newcommand{\SobK}[3]{\ifthenelse{\equal{#2}{0}}{\mathrm{W}^{#3, 2}\left(K;\wedge^{#1}T^*K\right)}{\mathrm{W}^{#3, 2}_{#2}\left(K;\wedge^{#1}T^*K\right)}}

\newcommand{\sign}[2]{
 \mathrm{sign}
 \left(
 \begin{array}{c}
   #1 \\
   #2
 \end{array}
 \right)
\,}

\newcommand{\der}[2]{\frac{\del #1}{\del x^{#2}}\,}

\newcommand{\expp}[1]{\exp\left(#1\right)\,}

\newcommand{\In}[1]{\left\{1,\ldots,#1\right\}}

\newcommand{\Endom}[1]{\mathrm{End}\left(#1\right)}

\newcommand{\Grass}[3]{\mathrm{G}_{#1}\left(#2,\,#3\right)}

\newcommand{\Homom}[2]{\mathrm{Hom}\left(#1,\,#2\right)}

\newcommand{\sgn}[1]{\mathrm{sgn}\left(#1\right)}

\newcommand{\Sym}[1]{\mathrm{Sym}^{2}\left(#1\right)}

\newcommand{\destar}[2]{\duale{\de}_{#1,\,#2}}

\DeclareMathOperator{\dom}{dom}

\newcommand{\intern}[1]{\mathrm{int}\left({#1}\right)}

\newcommand{\paragrafod}[2]{\smallskip \noindent \texttt{Step {#1}} -- {\itshape #2}.\ }
\newcommand{\paragrafodclaim}[2]{\smallskip \noindent \texttt{Claim {#1}} -- {\itshape #2}.\ }

\DeclareMathOperator{\Hess}{Hess}
\DeclareMathOperator{\PSH}{PSH}
\DeclareMathOperator{\Levi}{L}
\DeclareMathOperator{\bordo}{b}
\DeclareMathOperator{\ord}{ord}

\DeclareMathOperator{\spec}{spec}

\DeclareMathOperator{\Span}{span}

\makeatletter
\@addtoreset{equation}{section}
\makeatother
\renewcommand{\theequation}%
  {\thesection.\arabic{equation}}

  \titolo{\textbf{Cohomological aspects}\\
          \textbf{of non-K\"ahler manifolds}}
  \laureando{Daniele Angella}
  \annoaccademico{2011/2012}

 \ateneo{Pisa}
 \facolta{Scienze di Base ``Galileo Galilei''}
 \dottorato{Matematica}
 \ciclo{xxv}
  \relatore[Prof.]{Adriano Tomassini}
 \ssd{MAT/03 Geometria}
 \direttorescuola[Prof.]{Fabrizio Broglia}

\begin{document}
\selectlanguage{english}

\normalfont\selectfont

\frontmatter

\maketitle

\tableofcontents

\chapter*{Introduction}
\addcontentsline{toc}{chapter}{Introduction}
\markboth{Introduction}{Introduction}
\vspace{-12pt}

By a remarkable result by W.~L. Chow, \cite[Theorem V]{chow} (see also \cite{serre-gaga}), \emph{projective manifolds} (that is, compact complex submanifolds of $\CP^n:=\left. \left(\C^{n+1}\setminus\{0\}\right) \right\slash \left(\C\setminus\{0\}\right)$, for $n\in\N$) are in fact \emph{algebraic} (that is, they can be described as the zero set of finitely many homogeneous holomorphic polynomials). One is hence interested in relaxing the projective condition, looking for special properties on compact manifolds sharing a weaker structure than projective manifolds. For example, a large amount of developed analytic techniques allows to prove strong cohomological properties for compact \emph{K\"ahler manifolds} (that is, compact complex manifolds endowed with a K\"ahler metric, namely, a Hermitian metric admitting a local potential function), \cite{schouten-vandantzig, kahler}, see also \cite{weil}, which are, in a certain sense, the ``analytic-versus-algebraic'', \cite[Theorem V]{chow}, or the
``$\R$-versus-$\Q$'', \cite[Theorem 4]{kodaira-embedding}, version of projective manifolds.
K\"ahler manifolds are in fact endowed with three different structures, interacting each other: a \emph{complex} structure, a \emph{symplectic} structure, and a \emph{metric} structure; it is the strong linking between them that allows to develop many analytic tools and hence to derive the very special properties of K\"ahler manifolds. In order to better understand any of such properties, it is natural to ask what of these three structures is actually involved and required. Therefore, one is led to study complex, symplectic, and metric contribution separately, possibly weakening either the interactions between them, or one of these structures.
For example, by relaxing the metric condition, one could ask what properties of a compact complex manifold can be deduced by the existence of special Hermitian metrics defined by conditions similar to, but weaker than, the defining condition of the K\"ahler metrics (for example, metrics being \emph{balanced} in the sense of M.~L. Michelsohn \cite{michelsohn}, \emph{pluriclosed} \cite{bismut--math-ann-1989}, \emph{astheno-K\"ahler} \cite{jost-yau, jost-yau-correction}, \emph{Gauduchon} \cite{gauduchon}, \emph{strongly-Gauduchon} \cite{popovici-proj}); by relaxing the complex structure, one is led to study properties of \emph{almost-complex} manifolds, possibly endowed with compatible symplectic structures.

\medskip

In particular, we are concerned with studying cohomological properties of compact (almost-)complex manifolds, and of manifolds endowed with special structures, e.g., symplectic structures, \para-complex structures in the sense of F.~R. Harvey and H.~B. Lawson, exhaustion functions satisfying positivity conditions. Part of the original results have been published or will appear in \cite{angella-tomassini-1, angella-tomassini-2, angella,  angella-tomassini-3, angella-rossi, angella-tomassini-zhang, angella-calamai}; some other results have been collected in a preprint, see \cite{angella-tomassini-4}; some more results have not yet been submitted for publication.

\medskip

We recall that a \emph{complex} manifold $X$ is endowed with a natural \emph{almost-complex} structure, that is, an endomorphism $J\in\End(TX)$ of the tangent bundle of $X$ such that $J^2=-\id_{TX}$, which actually satisfies a further integrability condition, \cite[Theorem 1.1]{newlander-nirenberg}. By considering the decomposition into eigen-spaces, just the datum of the almost-complex structure yields a splitting of the complexified tangent bundle, namely,
$$ TX\otimes \C \;=\; T^{1,0}X \oplus T^{0,1}X \;, $$
and hence it induces also a splitting of the bundle of complex differential forms, namely,
$$ \wedge^\bullet X \otimes_\R \C \;=\; \bigoplus_{p+q=\bullet} \wedge^{p,q} X \;. $$
Furthermore, on a complex manifold, the integrability condition of such an almost-complex structure yields a further structure on $\wedge^{\bullet,\bullet}X$, namely, a structure of double complex $\left(\wedge^{\bullet,\bullet}X,\, \del,\, \delbar\right)$, where $\del$ and $\delbar$ are the components of the $\C$-linear extension of the exterior differential $\de$.

Hence, on a complex manifold $X$, one can consider both the de Rham cohomology
$$ H^\bullet_{dR}\left(X;\C\right) \;:=\; \frac{\ker\de}{\imm\de} $$
and the Dolbeault cohomology
$$ H^{\bullet,\bullet}_{\delbar}(X) \;:=\; \frac{\ker\delbar}{\imm\delbar} \;;$$
whenever $X$ is compact, the Hodge theory assures that they have finite-dimension as $\C$-vector spaces.
On a compact complex manifold, in general, no natural map between $H^{\bullet,\bullet}_{\delbar}(X)$ and $H^\bullet_{dR}\left(X;\C\right)$ exists; on the other hand, the structure of double complex of $\left(\wedge^{\bullet,\bullet}X,\, \del,\, \delbar\right)$ gives rise to a spectral sequence
$$ E^{\bullet,\bullet}_1 \;\simeq\; H^{\bullet,\bullet}_{\delbar}(X) \;\Rightarrow\; H^{\bullet}_{dR}(X;\C) \;, $$
from which one gets the Fr\"olicher inequality, \cite[Theorem 2]{frolicher}: for every $k\in\N$,
$$ \dim_\C H^k_{dR}(X;\C) \;\leq\; \sum_{p+q=k} \dim_\C H^{p,q}_{\delbar}(X) \;.$$

On a complex manifold, a ``bridge'' between the Dolbeault and the de Rham cohomology is provided, in a sense, by the \emph{Bott-Chern cohomology},
$$ H^{\bullet,\bullet}_{BC}(X) \;:=\; \frac{\ker\del\cap\ker\delbar}{\imm\del\delbar} \;,$$
and the \emph{Aeppli cohomology},
$$ H^{\bullet,\bullet}_{A}(X) \;:=\; \frac{\ker\del\delbar}{\imm\del+\imm\delbar} \;.$$
In fact, the identity induces the maps of (bi-)graded $\C$-vector spaces
$$
\xymatrix{
 & H^{\bullet,\bullet}_{BC}(X) \ar[d]\ar[ld]\ar[rd] & \\
 H^{\bullet,\bullet}_{\del}(X) \ar[rd] & H^{\bullet}_{dR}(X;\C) \ar[d] & H^{\bullet,\bullet}_{\delbar}(X) \ar[ld] \\
 & H^{\bullet,\bullet}_{A}(X) & 
}
$$
which, in general, are neither injective nor surjective.

We recall that, whenever $X$ is compact, the Hodge theory can be performed also for Bott-Chern and Aeppli cohomologies, \cite[\S2]{schweitzer}, yielding their finite-dimensionality; more precisely, one has that, on a compact complex manifold $X$ of complex dimension $n$ endowed with a Hermitian metric,
$$ H^{\bullet,\bullet}_{BC}(X) \;\simeq\; \tilde\Delta_{BC} \qquad \text{ and } \qquad H^{\bullet,\bullet}_{A}(X) \;\simeq\; \tilde\Delta_{A} \;, $$
where $\tilde \Delta_{BC}$ and $\tilde \Delta_{A}$ are \kth{4} order self-adjoint elliptic differential operators; furthermore, the Hodge-$*$-operator associated to any Hermitian metric on $X$ induces an isomorphism $H^{p,q}_{BC}(X) \simeq H^{n-q,n-p}_{A}(X)$, for every $p,q\in\N$.

By the definitions, the map $H_{BC}^{\bullet,\bullet}(X)\to H^{\bullet}_{dR}(X;\C)$ is injective if and only if every $\del$-closed $\delbar$-closed $\de$-exact form is $\del\delbar$-exact: a compact complex manifold fulfilling this property is said to satisfy the \emph{$\del\delbar$-Lemma}; see \cite{deligne-griffiths-morgan-sullivan} by P. Deligne, Ph.~A. Griffiths, J. Morgan, and D.~P. Sullivan, where consequences of the validity of the $\del\delbar$-Lemma on the real homotopy type of a compact complex manifold are investigated. When the $\del\delbar$-Lemma holds, it turns out that actually all the above maps are isomorphisms, \cite[Lemma 5.15, Remark 5.16, 5.21]{deligne-griffiths-morgan-sullivan}: in particular, one gets a decomposition
$$ H^\bullet_{dR}(X;\C) \;\simeq\; \bigoplus H^{\bullet,\bullet}_\delbar(X) \qquad \text{ such that } \qquad H^{\bullet_1,\bullet_2}_\delbar(X) \;\simeq\; \overline{H^{\bullet_2,\bullet_1}_\delbar(X)} \;. $$

A very remarkable property of compact K\"ahler manifolds is that they satisfy the $\del\delbar$-Lemma, \cite[Lemma 5.11]{deligne-griffiths-morgan-sullivan}: this follows from the K\"ahler identities, which can be proven as a consequence of the fact that the K\"ahler metrics osculate to order $2$ the standard Hermitian metric of $\C^n$ at every point. Therefore, the above decomposition holds true, in particular, for compact K\"ahler manifolds, \cite[Théorème IV.3]{weil}.

In particular, if $X$ is a compact complex manifold satisfying the $\del\delbar$-Lemma, then, for every $k\in\N$,
$$ \dim_\C H^k_{dR}(X;\C) \;=\; \sum_{p+q=k} \dim_\C H^{p,q}_{BC}(X) \;. $$

\medskip

In the first chapter, we study cohomological properties of compact complex manifolds, studying in particular the Bott-Chern and Aeppli cohomologies, and their relation with the $\del\delbar$-Lemma.

In fact, the first result we prove states a Fr\"olicher-type inequality for the Bott-Chern and Aeppli cohomologies, which provides also a characterization of the compact complex manifolds satisfying the $\del\delbar$-Lemma just in terms of the dimensions of the Bott-Chern cohomology groups, \cite[Theorem A, Theorem B]{angella-tomassini-3}; a key tool in the proof of the Fr\"olicher-type inequality relies on exact sequences by J. Varouchas, \cite{varouchas}. More precisely, we state the following result.

\smallskip
\noindent{\bfseries Theorem (see Theorem \ref{thm:frol-bc} and Theorem \ref{thm:caratterizzazione-bc-numbers}).}
{\itshape  Let $X$ be a compact complex manifold. Then, for every $k\in\N$, the following inequality holds:
$$ \sum_{p+q=k} \left( \dim_\C H^{p,q}_{BC}(X) + \dim_\C H^{p,q}_{A}(X)\right) \;\geq\; 2\, \dim_\C H^k_{dR}(X;\C) \;. $$
The equality
$$ \dim_\C H^k_{BC}(X) + \dim_\C H^k_{A}(X) \;=\; 2\, \dim_\C H^k_{dR}(X;\C) $$
holds for every $k\in\N$ if and only if $X$ satisfies the $\del\delbar$-Lemma.
}
\smallskip

Note that the equality $\sum_{p+q=k} \dim_\C H^{p,q}_{\delbar}(X) = \dim_\C H^k_{dR}(X;\C)$ for every $k\in\N$ (which is equivalent to the degeneration of the Hodge and Fr\"olicher spectral sequence at the first step, $E_1\simeq E_{\infty}$) is not sufficient to let $X$ satisfy the $\del\delbar$-Lemma: in some sense, the above result states that the Bott-Chern cohomology, together with its dual, the Aeppli cohomology, encodes ``more informations'' on the double complex $\left(\wedge^{\bullet,\bullet}X,\, \del,\, \delbar\right)$ than just the Dolbeault cohomology.

As a straightforward consequence of the previous theorem, we obtain another proof, see \cite[Corollary 2.7]{angella-tomassini-3}, of the following result, see \cite[Proposition 9.21]{voisin}, \cite[Theorem 5.12]{wu}, \cite[\S B]{tomasiello}.

\smallskip
\noindent{\bfseries Corollary (see Corollary \ref{cor:stab-del-delbar-lemma}).}
{\itshape
 Satisfying the $\del\delbar$-Lemma is a stable property under small deformations of the complex structure, that is, if $\left\{X_t\right\}_{t\in B}$ is a complex-analytic family of compact complex manifolds and $X_{t_0}$ satisfies the $\del\delbar$-Lemma for some $t_0\in B$, then $X_t$ satisfies the $\del\delbar$-Lemma for every $t$ in an open neighbourhood of $t_0$ in $B$.
}
\smallskip

\medskip

A class of manifolds that turns out to be particularly interesting in non-K\"ahler geometry, as a fruitful source of examples, is provided by the class of \emph{nilmanifolds}, and, more in general, of \emph{solvmanifolds}, namely, compact quotients of connected simply-connected nilpotent, respectively solvable, Lie groups by co-compact discrete subgroups. In fact, on the one hand, non-tori nilmanifolds admit no
K\"ahler structure, \cite[Theorem A]{benson-gordon-nilmanifolds}, \cite[Theorem 1, Corollary]{hasegawa_pams}, and, on the other hand, focusing on left-invariant geometric structures on solvmanifolds, one can often reduce their study at the level of the associated Lie algebra; this turns out to hold true, in particular, for the de Rham cohomology of completely-solvable solvmanifolds, \cite{nomizu, hattori}, and for the Dolbeault cohomology of nilmanifolds endowed with certain left-invariant complex structures, \cite{sakane, cordero-fernandez-gray-ugarte, console-fino, rollenske, rollenske-survey}, see, 
e.g., \cite{console, rollenske-survey}.

More precisely, on a nilmanifold $X=\left. \Gamma \right\backslash G$, the inclusion of the subcomplex composed of the $G$-left-invariant forms on $X$ (which is isomorphic to the complex $\left(\wedge^\bullet\duale{\mathfrak{g}},\,\de\right)$, where $\mathfrak{g}$ is the associated Lie algebra) turns out to be a quasi-isomorphism, \cite[Theorem 1]{nomizu}, that is,
$$ i\colon H_{dR}^\bullet\left(\mathfrak{g};\R\right) \;:=\; H^\bullet\left(\wedge^\bullet\duale{\mathfrak{g}},\, \de\right) \;\stackrel{\simeq}{\to}\; H^\bullet_{dR}(X;\R) \;;$$
a similar result holds true also for completely-solvable solvmanifolds, \cite[Corollary 4.2]{hattori}, and for the Dolbeault cohomology of nilmanifolds endowed with left-invariant complex structures belonging to certain classes, \cite[Theorem 1]{sakane}, \cite[Main Theorem]{cordero-fernandez-gray-ugarte}, \cite[Theorem 2, Remark 4]{console-fino}, \cite[Theorem 1.10]{rollenske}, \cite[Corollary 3.10]{rollenske-survey}.

As a matter of notation, denote by $H^{\bullet,\bullet}_{\sharp}\left(\mathfrak{g}_\C\right)$, for $\sharp\in\left\{\delbar,\, \del,\, BC,\, A\right\}$, the cohomology of the corresponding subcomplex of $G$-left-invariant forms on a solvmanifold $X=\left.\Gamma \right\backslash G$, with Lie algebra $\mathfrak{g}$, endowed with a $G$-left-invariant complex structure.
The following result states a Nomizu-type theorem also for the Bott-Chern and Aeppli cohomologies, \cite[Theorem 3.7, Theorem 3.8, Theorem 3.9]{angella}.

\smallskip
\noindent{\bfseries Theorem (see Theorem \ref{thm:bc-invariant}, Theorem \ref{thm:bc-invariant-cor}, Remark \ref{rem:aeppli-inv}, and Theorem \ref{thm:bc-invariant-open}).}
{\itshape
 Let $X=\left.\Gamma\right\backslash G$ be a solvmanifold endowed with a $G$-left-invariant complex structure $J$, and denote the Lie algebra naturally associated to $G$ by $\mathfrak{g}$. Suppose that the inclusions of the subcomplexes of $G$-left-invariant forms on $X$ into the corresponding complexes of differential forms on $X$ yield the isomorphisms
$$ i\colon H^\bullet_{dR}(\mathfrak{g};\C)\stackrel{\simeq}{\to} H^\bullet_{dR}(X;\C) \qquad \text{ and } \qquad i\colon H^{\bullet,\bullet}_{\delbar}\left(\mathfrak{g}_\C\right)\stackrel{\simeq}{\to} H^{\bullet,\bullet}_{\delbar}(X) \;;$$
 in particular, this holds true if one of the following conditions holds:
\begin{itemize}
 \item $X$ is holomorphically parallelizable;
 \item $J$ is an Abelian complex structure;
 \item $J$ is a nilpotent complex structure;
 \item $J$ is a rational complex structure;
 \item $\mathfrak{g}$ admits a torus-bundle series compatible with $J$ and with the rational structure induced by $\Gamma$;
 \item $\dim_\R\mathfrak{g}=6$ and $\mathfrak{g}$ is not isomorphic to $\mathfrak{h}_7:=\left(0^3,\, 12,\, 13,\, 23\right)$.
\end{itemize}
 Then also
 $$ i\colon H^{\bullet,\bullet}_{BC}\left(\mathfrak{g}_\C\right) \stackrel{\simeq}{\to} H^{\bullet,\bullet}_{BC}(X)
    \qquad \text{ and } \qquad
    i\colon H^{\bullet,\bullet}_{A}\left(\mathfrak{g}_\C\right) \stackrel{\simeq}{\to} H^{\bullet,\bullet}_{A}(X)$$
 are isomorphisms.

 Furthermore, if $\mathcal{C}\left(\mathfrak{g}\right)$ denotes the set of $G$-left-invariant complex structures on $X$, then the set
 $$ \mathcal{U} \;:=\; \left\{ J\in\mathcal{C}\left(\mathfrak{g}\right) \st i\colon H^{\bullet,\bullet}_{\sharp_J}\left(\mathfrak{g}_\C\right) \;\stackrel{\simeq}{\hookrightarrow}\; H^{\bullet,\bullet}_{\sharp_J}\left(X\right) \right\} $$
 is open in $\mathcal{C}\left(\mathfrak{g}\right)$, for $\sharp\in\{\del,\,\delbar,\,BC,\,A\}$.
}
\smallskip

The above result allows to explicitly compute the Bott-Chern cohomology for the \emph{Iwasawa manifold}
$$ \mathbb{I}_3 \;:=\; \left. \mathbb{H}\left(3;\Z\left[\im\right]\right) \right\backslash \mathbb{H}(3;\C) $$
and for its small deformations, where
$$ \mathbb{H}(3;\mathbb{C}) \;:=\; \left\{
\left(
\begin{array}{ccc}
 1 & z^1 & z^3 \\
 0 &  1  & z^2 \\
 0 &  0  &  1
\end{array}
\right) \in \mathrm{GL}(3;\mathbb{C}) \st z^1,\,z^2,\,z^3 \in\C \right\}
\quad \text{ and } \quad
\mathbb{H}\left(3;\Z\left[\im\right]\right) \;:=\; \mathbb{H}(3;\C)\cap\GL\left(3;\Z\left[\im\right]\right) \;.
$$
The Iwasawa manifold is one of the simplest example of compact non-K\"ahler complex manifold: as an example of a complex-parallelizable manifold, it has been studied by I. Nakamura, \cite{nakamura}, who computed its Kuranishi space and classified the small deformations of $\mathbb{I}_3$ by means of the dimensions of their Dolbeault cohomology groups.

In \S\ref{sec:bott-chern-iwasawa}, \cite[\S5.3]{angella}, we explicitly compute the Bott-Chern cohomology of the small deformations of the Iwasawa manifold, showing that it makes possible to give a finer classification of the small deformations $\left\{X_\tempo\right\}_{\tempo\in\Delta(\zero,\varepsilon)\subset \C^6}$ of $\I_3$ than the Dolbeault cohomology: more precisely, classes {\itshape (ii)} and {\itshape (iii)} in I. Nakamura's classification \cite[\S3]{nakamura} are further subdivided into subclasses {\itshape (ii.a)} and {\itshape (ii.b)}, respectively {\itshape (iii.a)} and {\itshape (iii.b)}, according to the value of $\dim_\C H^{2,2}_{BC}\left(X_\tempo\right)$.

\medskip

Another class that could provide several interesting  examples is given by complex orbifolds of the type $\tilde X=\left. X \right\slash G$, where $X$ is a complex manifold and $G$ is a finite group of biholomorphisms of $X$. Orbifolds of such a global-quotient-type have been considered and studied, e.g., by D.~D. Joyce in constructing examples of compact $7$-dimensional manifolds with holonomy $G_2$, \cite{joyce-jdg-1-2-g2} and \cite[Chapters 11-12]{joyce-ext}, and examples of compact $8$-dimensional manifolds with holonomy ${\rm Spin}(7)$, \cite{joyce-invent, joyce-jdg-spin7} and \cite[Chapters 13-14]{joyce-ext}.

One can define the space of differential forms $\wedge^{\bullet,\bullet}\tilde X$ on a complex orbifold of the type $\tilde X=\left. X \right\slash G$ as the space of $G$-invariant differential forms on $X$; hence, one can define the de Rham, Dolbeault, Bott-Chern, and Aeppli cohomologies also for $\tilde X$. Analogously, one can define the space of currents $\mathcal{D}^{\bullet,\bullet}\tilde X$ on $\tilde X$ as the space of $G$-invariant currents on $X$, as well as a Hermitian metric on $\tilde X$ as a $G$-invariant Hermitian metric on $X$.

As a first tool to investigate the Bott-Chern and Aeppli cohomologies of compact complex orbifolds of global-quotient-type, we obtain the following result.

\smallskip
\noindent{\bfseries Theorem (see Theorem \ref{thm:bc}).}
{\itshape
 Let $\tilde X=\left.X \right\slash G$ be a compact complex orbifold of complex dimension $n$, where $X$ is a complex manifold and $G$ is a finite group of biholomorphisms of $X$.
 Then, for any $p,q\in\N$, there are canonical isomorphisms
 $$
  H^{p,q}_{BC}(\tilde X) \;\simeq\; \frac{\ker\left(\del \colon \correnti^{p,q}\tilde X \to \correnti^{p+1,q}\tilde X\right) \cap \ker \left(\delbar\colon \correnti^{p,q}\tilde X \to \correnti^{p,q+1}\tilde X\right)}{\imm\left(\del\delbar \colon \correnti^{p-1,q-1}\tilde X \to \correnti^{p,q}\tilde X\right)} \;.
 $$

 Furthermore, given a Hermitian metric on $X$, there are canonical isomorphisms
 $$ H^{\bullet,\bullet}_{BC}(\tilde X) \;\simeq\; \ker \tilde\Delta_{BC} \qquad \text{ and } \qquad H^{\bullet,\bullet}_{A}(\tilde X) \;\simeq\; \ker\tilde\Delta_A \;. $$

 In particular, the Hodge-$*$-operator induces an isomorphism
 $$ H^{\bullet_1,\bullet_2}_{BC}(\tilde X) \;\simeq\; H^{n-\bullet_2,n-\bullet_1}_{A}(\tilde X) \;.$$
}
\smallskip

\medskip

In the second chapter, we do not require the integrability of the almost-complex structure, and we study cohomological properties of almost-complex manifolds, that is, differentiable manifolds endowed with a (possibly non-integrable) almost-complex structure. In this case, the Dolbeault cohomology is not defined. However, following T.-J. Li and W. Zhang \cite{li-zhang}, one can consider, for every $p,q\in\N$, the subgroups
$$ H^{(p,q),(q,p)}_J(X;\R) \;:=\; \left\{ \left[\alpha\right]\in H^{p+q}_{dR}(X;\R) \st \alpha \in \left(\wedge^{p,q}X\oplus\wedge^{q,p}X\right) \cap \wedge^{p+q}X \right\} \subseteq H^{p+q}_{dR}(X;\R) \;, $$
and their complex counterpart
$$ H^{(p,q)}_J(X;\C) \;:=\; \left\{ \left[\alpha\right]\in H^{p+q}_{dR}(X;\C) \st \alpha \in \wedge^{p,q}X \right\} \subseteq H^{p+q}_{dR}(X;\C) \;.$$
If $X$ is a compact K\"ahler manifold, then $H^{(p,q)}_J(X;\C)\simeq H^{p,q}_{\delbar}(X)$ for every $p,q\in\N$, \cite[Lemma 2.15, Theorem 2.16]{draghici-li-zhang}; therefore these subgroups can be considered, in a sense, as a generalization of the Dolbeault cohomology groups to the non-K\"ahler, or to the non-integrable, case.

Two remarks need to be pointed out. Firstly, note that, in general, neither the equality in
$$ \sum_{\substack{p+q=k\\p\leq q}} H^{(p,q),(q,p)}_J(X;\R) \;\subseteq\; H^{p+q}_{dR}(X;\R) \;, \qquad \text{ or } \qquad \sum_{p+q=k} H^{(p,q)}_J(X;\C) \;\subseteq\; H^{p+q}_{dR}(X;\C) \;, $$
holds, nor the sum is direct, nor there are relations between the equality holding and the sum being direct, see, e.g., Proposition \ref{prop:Cf-Cp-non-related}. Hence, one may be interested in studying compact almost-complex manifolds for which one of the above properties holds, at least for a fixed $k\in\N$, see \cite{li-zhang, draghici-li-zhang, draghici-li-zhang-2, fino-tomassini, angella-tomassini-1, angella-tomassini-2, zhang, angella-tomassini-zhang, draghici-zhang, tan-wang-zhang-zhu, hind-medori-tomassini, li-tomassini, draghici-li-zhang-survey}. A remarkable result by T. Dr\v{a}ghici, T.-J. Li, and W. Zhang, \cite[Theorem 2.3]{draghici-li-zhang}, states that every almost-complex structure $J$ on a compact $4$-dimensional manifold $X^4$ satisfies the cohomological decomposition
$$ H^2_{dR}\left(X^4;\R\right) \;=\; H^{(2,0),(0,2)}_J\left(X^4;\R\right) \oplus H^{(1,1)}_J\left(X^4;\R\right) \;.$$

Secondly, note that $J\lfloor_{\wedge^2X}$ satisfies $\left(J\lfloor_{\wedge^2X}\right)^2=\id_{\wedge^2X}$, therefore the above subgroups of $H^2_{dR}(X;\R)$ can be interpreted as the subgroup represented by $J$-invariant forms,
$$ H^+_J(X) \;:=\; H^{(1,1)}_J(X;\R) \;=\;\left\{ \left[\alpha\right]\in H^2_{dR}(X;\R) \st J\alpha=\alpha \right\} \;, $$
and the subgroup represented by $J$-anti-invariant forms,
$$ H^-_J(X) \;:=\; H^{(2,0),(0,2)}_J(X;\R) \;=\;\left\{ \left[\alpha\right]\in H^2_{dR}(X;\R) \st J\alpha=-\alpha \right\} \;. $$
Note also that, if $g$ is any Hermitian metric on $X$ whose associated $(1,1)$-form $\omega:=g(J\sspace,\, \ssspace)\in\wedge^{1,1}X\cap\wedge^2X$ is $\de$-closed (namely, $g$ is an \emph{almost-K\"ahler} metric on $X$), then $\left[\omega\right]\in H^+_J(X)$.

In fact, T.-J. Li and W. Zhang's interest in studying such subgroups and \emph{\Cpf} almost-complex structures (that is, almost-complex structures for which the decomposition
$$ H^2_{dR}(X;\R) \;=\; H^+_J(X) \oplus H^-_J(X) $$
holds, \cite[Definition 2.2, Definition 2.3, Lemma 2.2]{li-zhang}) arises in investigating the symplectic cones of an almost-complex manifolds, that is, the \emph{$J$-tamed cone}
$$ \mathcal{K}^t_J \;:=\; \left\{ \left[\omega\right]\in H^2_{dR}(X;\R) \st \omega_x\left(v_x, \, J_xv_x\right)>0 \text{ for every } v_x\in T_xX\setminus\{0\} \text{ and for every }x\in X\right\} $$
and the \emph{$J$-compatible cone}
$$ \mathcal{K}^c_J \;:=\; \left\{ \left[\omega\right]\in H^2_{dR}(X;\R) \st \omega_x\left(v_x, \, J_xv_x\right)>0 \text{ for every } v_x\in T_xX\setminus\{0\} \text{ and for every }x\in X, \text{ and } J\omega=\omega\right\} \;. $$
Indeed, they proved in \cite[Theorem 1.1]{li-zhang} that, given a \Cpf\ almost-K\"ahler structure on a compact manifold $X$, the $J$-anti-invariant subgroup $H^-_J(X)$ of $H^2_{dR}(X;\R)$ measures the quantitative difference between the $J$-tamed cone and the $J$-compatible cone, namely,
$$ \mathcal{K}^t_J \;=\; \mathcal{K}^c_J \oplus H^{-}_J(X) \;. $$
A natural question concerns the qualitative comparison between the tamed cone and the compatible cone: more precisely, one could ask whether, whenever an almost-complex structure $J$ admits a $J$-tamed symplectic form, there exists also a $J$-compatible symplectic form. This turns out to be false, in general, for non-integrable almost-complex structures in dimension greater than $4$, \cite{migliorini-tomassini, tomassini-forummath}; on the other hand, it is not known whether, for almost-complex structures on compact $4$-dimensional manifolds, as asked by S.~K. Donaldson, \cite[Question 2]{donaldson}, or for complex structures on compact manifolds of complex dimension greater than or equal to $3$, as asked by T.-J. Li and W. Zhang, \cite[page 678]{li-zhang}, and by J. Streets and G. Tian, \cite[Question 1.7]{streets-tian}, it holds that $\mathcal{K}^c_J$ is non-empty if and only if $\mathcal{K}^t_J$ is non-empty. We prove the following result, stating that no counterexample can be found among $6$-dimensional 
non-tori nilmanifolds endowed with left-invariant complex structures, \cite[Theorem 3.3]{angella-tomassini-1}; note that the same holds true, more in general, for higher dimensional nilmanifolds, as proven by N. Enrietti, A. Fino, and L. Vezzoni, \cite[Theorem 1.3]{enrietti-fino-vezzoni}.

\smallskip
\noindent{\bfseries Theorem (see Theorem \ref{thm:nilmfd-are-not-tamed}).}
{\itshape
Let $X=\Gamma\backslash G$ be a $6$-dimensional nilmanifold endowed with a $G$-left-invariant complex structure $J$.
If $X$ is not a torus, then there is no $J$-tamed symplectic structure on $X$.
}
\smallskip

One can study further cones in cohomology, which are related to special metrics, other than K\"ahler metrics; a key tool is provided by the theory of cone structures on differentiable manifolds developed by D.~P. Sullivan, \cite{sullivan}.
In order to compare, in particular, the cone associated to \emph{balanced metrics} (that is, Hermitian metrics whose associated $(1,1)$-form is co-closed, \cite[Definition 1.4, Theorem 1.6]{michelsohn}) and the cone associated to \emph{strongly-Gauduchon metrics} (that is, Hermitian metrics whose associated $(1,1)$-form $\omega$ satisfies the condition that $\del\left(\omega^{\dim_\C X-1}\right)$ is $\delbar$-exact \cite[Definition 3.1]{popovici-proj}), we give the following result, \cite[Theorem 2.9]{angella-tomassini-2}, which is the semi-K\"ahler counterpart of \cite[Theorem 1.1]{li-zhang}. (We refer to \S\ref{subsec:balanced-cones} for the definitions of the cones $\mathcal{K}b^t_J$ and $\mathcal{K}b^c_J$ on a manifold $X$ endowed with an almost-complex structure $J$.)

\smallskip
\noindent{\bfseries Theorem (see Theorem \ref{thm:kbc-kbt}).}
{\itshape
 Let $X$ be a compact $2n$-dimensional manifold endowed with an almost-complex structure $J$.
 Assume that $\mathcal{K}b^c_J\neq\varnothing$ (that is, there exists a semi-K\"ahler structure on $X$) and that $0\not\in\mathcal{K}b^t_J$. Then
 \begin{equation*}
  \mathcal{K}b^t_J\cap H^{(n-1,n-1)}_J(X;\R) \;=\; \mathcal{K}b^c_J
 \end{equation*}
 and
 \begin{equation*}
  \mathcal{K}b^c_J+H^{(n,n-2),(n-2,n)}_J(X;\R) \;\subseteq\; \mathcal{K}b^t_J \;. 
 \end{equation*}
 Moreover, if the equality $H^{2n-2}_{dR}(X;\R) = H^{(n,n-2),(n-2,n)}_J(X;\R) + H^{(n-1,n-1)}_J(X;\R)$ holds, then
 \begin{equation*}
 \mathcal{K}b^c_J+H^{(n,n-2),(n-2,n)}_J(X;\R) \;=\; \mathcal{K}b^t_J \;.
 \end{equation*}
}
\smallskip

\medskip

In order to better understand cohomological properties of compact almost-complex manifolds, and in view of the Hodge decomposition theorem for compact K\"ahler manifolds, it could be interesting to investigate the subgroups $H^{(p,q),(q,p)}_J(X;\R)$ for almost-complex manifolds endowed with special structures. For example, we prove the following result, \cite[Proposition 4.1]{angella-tomassini-zhang}, providing a strong difference between the K\"ahler case and the almost-K\"ahler case.

\smallskip
\noindent{\bfseries Proposition (see Proposition \ref{prop:almost-kahler-non-cf}).}
{\itshape
The differentiable manifold $X$ underlying the Iwasawa manifold $\mathbb{I}_3 := \left. \mathbb{H}\left(3;\Z\left[\im\right]\right) \right\backslash \mathbb{H}(3;\C)$ admits a non-\Cpf\ almost-K\"ahler structure.
}
\smallskip

A further study on almost-K\"ahler structures $\left(J,\, \omega,\, g\right)$ on a compact $2n$-dimensional manifold $X$ yields the following result, \cite[Theorem 2.3]{angella-tomassini-zhang}, which relates \Cpf ness with the \emph{Lefschetz-type property on $2$-forms} firstly considered by W. Zhang, that is, the property that the Lefschetz operator
$$ \omega^{n-2}\wedge\sspace \colon \wedge^2X\to\wedge^{2n-2}X $$
takes $g$-harmonic $2$-forms to $g$-harmonic $(2n-2)$-forms.

\smallskip
\noindent{\bfseries Theorem (see Theorem \ref{thm:almost-kahler-cpf}).}
{\itshape
 Let $X$ be a compact manifold endowed with an almost-K\"ahler structure $\left(J,\,\omega,\,g\right)$. Suppose that there exists a basis of $H^2_{dR}(X;\R)$ represented by $g$-harmonic $2$-forms which are of pure type with respect to $J$. Then the Lefschetz-type property on $2$-forms holds on $X$.
}
\smallskip

As a tool to study explicit examples, we provide a Nomizu-type theorem for the subgroups $H^{(p,q),(q,p)}_J(X;\R)$ of a completely-solvable solvmanifold $X = \left. \Gamma \right\backslash G$ endowed with a $G$-left-invariant almost-complex structure $J$, \cite[Theorem 5.4]{angella-tomassini-zhang}, see Proposition \ref{prop:linear-cpf-invariant-cpf-J}, and Corollary \ref{cor:linear-cpf-cpf}.

\medskip

A remarkable result by K. Kodaira and D.~C. Spencer states that the K\"ahler property on compact complex manifolds is stable under small deformations of the complex structure, \cite[Theorem 15]{kodaira-spencer-3}: more precisely, it states that, given a compact complex manifold admitting a K\"ahler structure, every small deformation still admits a K\"ahler structure; it can be proven as a consequence of the semi-continuity properties for the dimensions of the cohomology groups of a compact K\"ahler manifold. Hence, a natural question in non-K\"ahler geometry is to investigate the (in)stability of weaker properties than being K\"ahler.
As a first result in this direction, L. Alessandrini and G. Bassanelli proved that, given a compact complex manifold, the property of admitting a \emph{balanced metric} (that is, a Hermitian metric whose associated $(1,1)$-form is co-closed) is not stable under small deformations of the complex structure, \cite[Proposition 4.1]{alessandrini-bassanelli}; on the other hand, they proved that the class of balanced manifolds is stable under modifications, \cite[Corollary 5.7]{alessandrini-bassanelli--complex-geometry-trento}.
Another result in this context is the stability of the property of satisfying the $\del\delbar$-Lemma under small deformations of the complex structure, as already recalled, see Corollary \ref{cor:stab-del-delbar-lemma}.

Therefore, it is natural to investigate stability properties for the cohomological decomposition by means of the subgroups $H^{(p,q),(q,p)}_J(X;\R)$ on (almost-)complex manifolds $\left(X,\, J\right)$. More precisely, we consider the Iwasawa manifold $\mathbb{I}_3 := \left. \mathbb{H}\left(3;\Z\left[\im\right]\right) \right\backslash \mathbb{H}(3;\C)$, showing that the subgroups $H^{(p,q),(q,p)}_J(X;\R)$ provide a cohomological decomposition for $\mathbb{I}_3$ but not for some of its small deformations, Theorem \ref{thm:instability-iwasawa}. We prove the following result, \cite[Theorem 3.2]{angella-tomassini-1}.

\smallskip
\noindent{\bfseries Theorem (see Theorem \ref{thm:instability}).}
{\itshape
 The properties of being \Cpf\ is not stable under small deformations of the complex structure.
}
\smallskip

More in general, one could try to study directions along which the curves of almost-complex structures on a differentiable manifold preserve the property of being \Cpf. Using a procedure by J. Lee, \cite[\S1]{lee}, to construct curves of almost-complex structures through an almost-complex structure $J$, by means of $J$-anti-invariant real $2$-forms, we provide the following result, \cite[Theorem 4.1]{angella-tomassini-1}.

\smallskip
\noindent{\bfseries Theorem (see Theorem \ref{thm:lee-curves}).}
{\itshape
 There exists a compact manifold $N^6(c)$ endowed with an almost-complex structure $J$ and a $J$-Hermitian metric $g$ such that:
\begin{enumerate}
\item[(i)] $J$ is \Cpf;
\item[(ii)] each $J$-anti-invariant $g$-harmonic form gives rise to a curve
$\left\{J_t\right\}_{t\in\left(-\varepsilon,\varepsilon\right)}$ of \Cpf\ almost-complex structures on $N^6(c)$ (where $\varepsilon>0$ is small enough);
\item[(iii)] furthermore, the function
$$
\left(-\varepsilon,\varepsilon\right) \;\ni\; t\mapsto \dim_\R H^{(2,0),(0,2)}_{J_t}\left(N^6(c);\R\right) \;\in\; \N
$$
is upper-semi-continuous at $0$.
\end{enumerate}
}
\smallskip

Another problem in deformation theory is the study of semi-continuity properties for the dimensions of the subgroups $H^+_J(X)$ and $H^-_J(X)$. As a consequence of the Hodge theory for compact $4$-dimensional manifolds, T. Dr\v{a}ghici, T.-J. Li, and W. Zhang proved in \cite[Theorem 2.6]{draghici-li-zhang-2} that, given a curve $\{J_t\}_{t\in I\subseteq \R}$ of (\Cpf) almost-complex structures on a compact $4$-dimensional manifold $X$, the functions
$$ I \;\ni\; t\mapsto \dim_\R H^{-}_{J_t}(X) \;\in\; \N  \qquad \text{ and } \qquad I \;\ni\; t\mapsto \dim_\R H^{+}_{J_t}(X) \;\in\; \N $$
are, respectively, upper-semi-continuous and lower-semi-continuous.

In higher dimension this fails to be true, as we show in explicit examples. We provide hence the following result, \cite[Proposition 4.1, Proposition 4.3]{angella-tomassini-2}.

\smallskip
\noindent{\bfseries Proposition (see Proposition \ref{prop:no-scs-h-} and Proposition \ref{prop:no-sci-h+}).}
{\itshape
 In dimension higher than $4$, there exist compact manifolds $X$ endowed with families $\left\{J_t\right\}_{t\in I}$ of almost-complex structures such that either the function $I \ni t \mapsto \dim_\R H^-_{J_t}\left(X\right) \in \N$ is not upper-semi-continuous, or the function $I \ni t \mapsto \dim_\R H^+_{J_t}\left(X\right) \in \N$ is not lower-semi-continuous.
}
\smallskip

Motivated by such counterexamples, we study a stronger semi-continuity property on almost-complex manifolds (namely, that, for every $\de$-closed $J$-invariant real $2$-form $\alpha$, there exists a $\de$-closed $J_t$-invariant real $2$-form $\eta_t=\alpha+\opiccolouno$, depending real-analytically in $t$, for $t\in\left(-\varepsilon,\,\varepsilon\right)$ with $\varepsilon>0$ small enough): we give a formal characterization of the curves of almost-complex structures satisfying such a property, see Proposition \ref{prop:semicont-forte}, and we provide also a counterexample to such a stronger semi-continuity property, see Proposition \ref{prop:semi-cont-strong-counterex}.

\medskip

In the third chapter, motivated by the problem to study cohomological obstructions induced by special structures on differentiable manifolds, we investigate cohomological properties of symplectic manifolds, \para-complex manifolds, and strictly $p$-convex domains.

\medskip

We recall that compact K\"ahler manifolds have special cohomological properties not just in the complex framework, but also from the symplectic viewpoint. More precisely, another important result, other than the Hodge decomposition theorem, \cite[Théorème IV.3]{weil}, is the Lefschetz decomposition theorem, \cite[Théorème IV.5]{weil}, which states a decomposition in terms of \emph{primitive} subgroups of the cohomology, namely,
$$ H^{\bullet}_{dR}(X;\C) \;=\; \bigoplus_{r\in\N} L^r \left(\ker\left(\Lambda\colon H^{\bullet-2r}_{dR}(X;\C)\to H^{\bullet-2r-2}_{dR}(X;\C)\right)\right) \;, $$
where $\Lambda$ is the adjoint operator of the Lefschetz operator $L:=\omega\wedge\sspace \colon \wedge^\bullet X \to \wedge^{\bullet+2}X$ with respect to the pairing induced by $\omega$. Hence, after having investigated cohomological properties of almost-complex manifolds, we turn our attention to cohomological properties of symplectic manifolds.

In particular, in \S\ref{sec:sympl}, we provide a symplectic counterpart of T.-J. Li and W. Zhang's cohomological theory for almost-complex-manifolds, studying compact symplectic manifolds $\left(X,\, \omega\right)$ for which the Lefschetz decomposition on differential forms,
$$ \wedge^\bullet X \;=\; \bigoplus_{r\in\N} L^r \, \Prim^{\bullet-2r}X \;, $$
(where $\Prim^\bullet X := \ker\Lambda$ is the space of \emph{primitive} forms,) gives rise to a decomposition of the de Rham cohomology by means of the subgroups
$$ H^{(r,s)}_\omega(X;\R) \;:=\; \left\{\left[L^r\,\beta^{(s)}\right]\in H^{2r+s}_{dR}(X;\R) \st \beta^{(s)}\in \Prim^sX \right\} \;\subseteq\; H^{2r+s}_{dR}(X;\R) \;.$$

In particular, we provide the following result, \cite[Theorem 2.6]{angella-tomassini-4}, which gives a symplectic counterpart to T. Dr\v{a}ghici, T.-J. Li, and W. Zhang's decomposition theorem \cite[Theorem 2.3]{draghici-li-zhang} in the almost-complex setting (in fact, without the restriction to dimension $4$).

\smallskip
\noindent{\bfseries Theorem (see Theorem \ref{thm:sympl-decomp-H2}).}
{\itshape
 Let $X$ be a compact manifold endowed with a symplectic structure $\omega$. Then
 \begin{eqnarray*}
  H^2_{dR}(X;\R) &=& H^{(1,0)}_\omega(X;\R) \oplus H^{(0,2)}_\omega(X;\R) \;.
 \end{eqnarray*}
}
\smallskip

A Nomizu-type theorem for the subgroups $H^{(r,s)}_{\omega}(X;\R)$ of a completely-solvable solvmanifold $X = \left.\Gamma \right\backslash G$ endowed with a $G$-left-invariant symplectic structure $\omega$ is provided, see Proposition \ref{prop:linear-cpf-invariant-cpf-omega}, giving an useful tool in order to investigate explicit examples.

\medskip

In a sense, \para-complex Geometry provides a ``hyperbolic analogue'' of Complex Geometry. An \emph{almost-\para-complex} structure is, by definition, the datum of an endomorphism $K\in\End(TX)$ of the tangent bundle of a differentiable manifold $X$ such that $K^2=\id_{TX}$ and with the additional property that the eigen-bundles $T^+X$ and $T^-X$ have the same rank; a natural notion of \emph{integrability} can be defined by requiring that the two distributions $T^+X$ and $T^-X$ are involutive. Many connections between \para-complex Geometry and other problems both in Mathematics and Physics (in particular, concerning product structures, bi-Lagrangian geometry, and optimal transport theory) have been investigated in the last years: see, e.g., \cite{harvey-lawson, alekseevsky-medori-tomassini, cortes-mayer-mohuapt-saueressig-1, cortes-mayer-mohuapt-saueressig-2, cortes-mayer-mohuapt-saueressig-3, cruceanu-fortuny-gadea, kim-mccann-warren, andrada-barberis-dotti-ovando, andrada-salamon, krahe, rossi-1, rossi-2}
 and the references therein for further details on \para-complex structures and motivations for their study.

We study cohomological decomposition for compact manifolds $X$ endowed with (almost-)\para-complex structures $K$. Note that the elliptic theory in the complex setting has not a \para-complex counterpart: for example, a \para-complex counterpart of the Dolbeault cohomology is possibly infinite-dimensional, even if the manifold is compact. This fact makes natural to consider the \para-complex counterpart $H^{(p,q)}_K\left(X;\R\right)$ of T.-J. Li and W. Zhang's subgroups $H^{(p,q),(q,p)}_J(X;\R)$ as a possible substitute of the \para-Dolbeault cohomology, and hence to study the subgroups
$$ H^{2\,+}_K(X;\R) \;:=\; \left\{\left[\alpha\right]\in H^2_{dR}\left(X;\R\right) \st K\alpha=\alpha\right\} \;\subseteq\; H^2_{dR}(X;\R) $$
and
$$ H^{2\,-}_K(X;\R) \;:=\; \left\{\left[\alpha\right]\in H^2_{dR}\left(X;\R\right) \st K\alpha=-\alpha\right\} \;\subseteq\; H^2_{dR}(X;\R) \;. $$

Nevertheless, several important differences arise between the complex and the \para-complex cases. For example, after having stated and proved a Nomizu-type result for the subgroups $H^{(p,q)}_K\left(X;\R\right)$ of a completely-solvable solvmanifold $X = \left. \Gamma \right\backslash G$ endowed with a $G$-left-invariant \para-complex structure $K$, we are able to prove the following result, \cite[Proposition 3.3]{angella-rossi}, which turns out to be very different from the complex case, see \cite[Proposition 2.1]{li-zhang}, or \cite[Lemma 2.15, Theorem 2.16]{draghici-li-zhang}. (Recall that a \emph{\para-K\"ahler} structure on a \para-complex manifold is the datum of an anti-invariant symplectic form with respect to the \para-complex structure.)

\smallskip
\noindent{\bfseries Proposition (see Proposition \ref{prop:para-kahler-non-cpf}).}
{\itshape
 Admitting a \para-K\"ahler structure is not a sufficient condition for either the sum
 $$ H^{2\,+}_K(X;\R) + H^{2\,-}_K(X;\R) \;\subseteq\; H^2_{dR}\left(X;\R\right) $$
 being direct, or the equality holding.
}
\smallskip

A partial \para-complex counterpart of T. Dr\v{a}ghici, T.-J. Li, and W. Zhang's decomposition theorem \cite[Theorem 2.3]{draghici-li-zhang} is provided by the following result, \cite[Theorem 3.17]{angella-rossi}.

\smallskip
\noindent{\bfseries Theorem (see Theorem \ref{thm:4-dim}).}
{\itshape
 Every left-invariant \para-complex structure on a $4$-dimensional nilmanifold satisfies the cohomological decomposition
 $$ H^2_{dR}\left(X;\R\right) \;=\; H^{2\,+}_K(X;\R) \oplus H^{2\,-}_K(X;\R) \;. $$
}
\smallskip

Note that the hypothesis in Theorem \ref{thm:4-dim} can not be weakened, as Example \ref{es 2.5} and Example \ref{es 2.6}, Example \ref{es 2.17}, and Example \ref{es 2.8} show.

\medskip

Concerning deformations of the \para-complex structure, we provide another strong difference with the complex case: in contrast with the stability theorem of K. Kodaira and D.~C. Spencer, \cite[Theorem 15]{kodaira-spencer-3}, we prove the following result in the \para-complex context, \cite[Theorem 4.2]{angella-rossi}.

\smallskip
\noindent{\bfseries Theorem (see Theorem \ref{thm:para-kahler-deformations}).}
{\itshape
 The property of being \para-K\"ahler is not stable under small deformations of the \para-complex structure.
}
\smallskip

Analogously to Theorem \ref{thm:instability} for almost-complex structures, we provide also the following instability result, \cite[Proposition 4.3]{angella-rossi}.

\smallskip
\noindent{\bfseries Proposition (see Proposition \ref{prop:para-cpf-instability}).}
{\itshape
 The properties of either the sum
 $$ H^{2\,+}_K(X;\R) + H^{2\,-}_K(X;\R) \;\subseteq\; H^2_{dR}\left(X;\R\right) $$
 being direct, or the equality holding are not stable under small deformations of the \para-complex structure.
}
\smallskip

Finally, we prove that, even in the \para-complex case, no general result on semi-continuity holds for the dimensions of the $K$-(anti-)invariant subgroups of the de Rham cohomology, \cite[Proposition 4.6]{angella-rossi}.

\smallskip
\noindent{\bfseries Proposition (see Proposition \ref{prop:para-semi-cont}).}
{\itshape
 Let $X$ be a compact manifold and let $\left\{K_t\right\}_{t\in I \subseteq \R}$ be a curve of \para-complex structures on $X$. Then, in general, the functions
$$ I\ni t \mapsto \dim_\R H^{2\,+}_{K_t}(X;\R)\in\N \qquad \text{ and } \qquad  I\ni t \mapsto \dim_\R H^{2\,-}_{K_t}(X;\R)\in\N $$
are not upper-semi-continuous or lower-semi-continuous.
}
\smallskip

\medskip

Finally, motivated by A. Andreotti and H. Grauert's vanishing result for the higher Dolbeault cohomology groups of a \emph{$q$-complete} domain in $\C^n$ (that is, a domain in $\C^n$ admitting a smooth proper exhaustion function whose Levi form has at least $n-q+1$ positive eigen-values), we turn our interest to study cohomological properties of Riemannian manifolds endowed with exhaustion functions whose Hessian satisfies positivity conditions.

In particular, a first case to be considered is the case of \emph{strictly $p$-convex} domains in $\R^n$ in the sense of F.~R. Harvey and H.~B. Lawson, \cite{harvey-lawson-1, harvey-lawson-2}, that is, domains in $\R^n$ admitting a smooth proper exhaustion function $u$ such that, at every point, every sum of $p$ different eigenvalues of the Hessian of $u$ is positive.
Adapting the $\mathrm{L}^2$-techniques developed by L. H\"ormander, \cite{hormander-acta}, and used also by A. Andreotti and E. Vesentini, \cite{andreotti-vesentini, andreotti-vesentini-erratum}, (and which could be hopefully applied in a wider context,) we give a different proof of a vanishing result following from J.-P. Sha's theorem \cite[Theorem 1]{sha}, and from H. Wu's theorem \cite[Theorem 1]{wu-indiana}, for the de Rham cohomology of strictly $p$-convex domains in $\R^n$ in the sense of F.~R. Harvey and H.~B. Lawson; more precisely, the following result holds, \cite[Theorem 3.1]{angella-calamai}, see \cite[Theorem 1]{sha}, \cite[Theorem 1]{wu-indiana}, \cite[Proposition 5.7]{harvey-lawson-2}.

\smallskip
\noindent{\bfseries Theorem (see Theorem \ref{thm:cauchy} and Theorem \ref{thm:vanishing}).}
{\itshape
 Let $X$ be a strictly $p$-convex domain in $\R^n$, and fix $k\in\N$ such that $k\geq p$. Then, every $\de$-closed $k$-form is $\de$-exact, that is,
 $$ H^k_{dR}(X;\R) \;=\; \{0\} $$
 for every $k\geq p$.
}
\smallskip

\bigskip

The plan of the thesis is as follows.

In Chapter \ref{chapt:preliminaries}, which contains no original material, we collect the basic notions concerning almost-complex, complex, and symplectic structures, we recall the main results on Hodge theory for K\"ahler manifolds, and we summarize the classical results on deformations of complex structures, on currents and de Rham homology, and on solvmanifolds.

In Chapter \ref{chapt:complex}, we study cohomological properties of compact complex manifolds, and in particular the Bott-Chern cohomology, \cite{angella-tomassini-3, angella}.
By using exact sequences introduced by J. Varouchas, \cite{varouchas}, we prove a Fr\"olicher-type inequality for the Bott-Chern cohomology, Theorem \ref{thm:frol-bc}, which also provides a characterization of the validity of the $\del\delbar$-Lemma in terms of the dimensions of the Bott-Chern cohomology groups, Theorem \ref{thm:caratterizzazione-bc-numbers}. We then prove a Nomizu-type result for the Bott-Chern cohomology, showing that, for certain classes of complex structures on nilmanifolds, the Bott-Chern cohomology is completely determined by the associated Lie algebra endowed with the induced linear complex structure, Theorem \ref{thm:bc-invariant}, Theorem \ref{thm:bc-invariant-cor}, and Theorem \ref{thm:bc-invariant-open}.
As an application, in \S\ref{sec:computations-iwasawa}, we explicitly study the Bott-Chern and Aeppli cohomologies of the Iwasawa manifold and of its small deformations. Finally, we study the Bott-Chern cohomology of complex orbifolds of the type $\left. X \right\slash G$, where $X$ is a compact complex manifold and $G$ a finite group of biholomorphisms of $X$, Theorem \ref{thm:bc}.

In Chapter \ref{chapt:almost-complex}, we study cohomological properties of almost-complex manifolds, \cite{angella-tomassini-1, angella-tomassini-2, angella-tomassini-zhang}. Firstly, in \S\ref{sec:definition}, we recall the notion of \Cpf\ almost-complex structure, which has been introduced by T.-J. Li and W. Zhang in \cite{li-zhang} in order to investigate the relations between the compatible and the tamed symplectic cones on a compact almost-complex manifold and with the aim to throw light on a question by S.~K. Donaldson, \cite[Question 2]{donaldson}. In particular, we are interested in studying when certain subgroups, related to the almost-complex structure, let a splitting of the de Rham cohomology of an almost-complex manifold, and their relations with cones of metric structures. In \S\ref{sec:classes-Cpf}, we focus on \Cpf ness on several classes of (almost-)complex manifolds, e.g., solvmanifolds endowed with left-invariant almost-complex structures, semi-K\"ahler manifolds, almost-K\"ahler 
manifolds. In \S\ref{sec:deformations-cpf}, we study the behaviour of \Cpf ness under small deformations of the complex structure and along curves of almost-complex structures, investigating properties of stability, Theorem \ref{thm:instability}, Theorem \ref{thm:lee-curves}, and of semi-continuity for the dimensions of the invariant and anti-invariant subgroups of the de Rham cohomology with respect to the almost-complex structure, Proposition \ref{prop:no-scs-h-}, Proposition \ref{prop:no-sci-h+}, Proposition \ref{prop:semicont-forte}, Proposition \ref{prop:semi-cont-strong-counterex}. In \S\ref{sec:cones}, we consider the cone of semi-K\"ahler structures on a compact almost-complex manifold and, in particular, by adapting the results by D.~P. Sullivan on cone structures, \cite{sullivan}, we compare the cones of balanced metrics and of strongly-Gauduchon metrics on a compact complex manifold, Theorem \ref{thm:kbc-kbt}.

In Chapter \ref{chapt:special}, we study the cohomological properties of (differentiable) manifolds endowed with special structures, other than (almost-)complex structures, \cite{angella-tomassini-4, angella-rossi, angella-calamai}. More precisely, in Section \ref{sec:sympl}, we investigate the cohomology of symplectic manifolds; in Section \ref{sec:paracomplex}, we study cohomological decompositions on \para-complex manifolds in the sense of F.~R. Harvey and H.~B. Lawson; finally, in Section \ref{sec:p-convex}, we consider domains in $\R^n$ endowed with a smooth proper strictly $p$-convex exhaustion function, and, using $\mathrm{L}^2$-techniques, we give another proof of a consequence of J.-P. Sha's theorem \cite[Theorem 1]{sha}, and of H. Wu's theorem \cite[Theorem 1]{wu-indiana}, on the vanishing of the higher degree de Rham cohomology groups.

\clearpage
{\fontsize{9}{10}\selectfont
\section*{Acknowledgments}

My first thanks goes to Adriano Tomassini: he has been my adviser thrice, and it is thanks to him if I have understood something in --- and not only in --- Mathematics. He supported and stood me during the past years, and his patience, guidance, suggestions, together with all the helpful discussions with him, have made me grow a lot.

I also wish to thank Jean-Pierre Demailly for his hospitality at Institut Fourier: his encouragement has been as important as his explanations and his kindly answers to my questions.

I wish to thank Marco Abate for all his kindness and for his very useful suggestions, which improved a lot the presentation of this thesis.

A particular thanks goes to the director of the PhD school, Fabrizio Broglia, for his support and his help, and for loving his work much more than complaining about it.

\medskip

I think Mathematics is better when played in two, or more, and I wish to thank all my collaborators, starting from my \emph{fratello accademico} Federico Alberto Rossi (with whom I shared a lot of conferences and workshops --- and several soccer matches ---, and to whom I am also indebted for my understanding of $\mathbf{D}$-complex Geometry), continuing with Simone Calamai (who has been always friendly, both in talking about Mathematics and in walking along the Lungarni) and Weiyi Zhang, and finishing with my \emph{sorellina} Maria Giovanna Franzini: needless to say, I hope this may be the starting point for enjoying new projects all together.

My growth as a mathematician is due to very many useful conversations and discussions (on Mathematics, and beyond), sometimes really brief, sometimes everlasting, but always very inspiring, especially with (in an almost-completely random list) Lucia Alessandrini, S\"onke Rollenske, Greg Kuperberg, Junyan Cao, Cristiano Spotti, Daniele Marconi, Valentina Disarlo, Costantino Medori, Jasmin Raissy, Isaia Nisoli, Serena Guarino Lo Bianco, Luca Battistella, Maria Beatrice Pozzetti, Pietro Tortella, Marco Pasquali, Paolo Baroni, Leonardo Biliotti, Matteo Serventi, Amedeo Altavilla, Laura Cremaschi, Alberto Gioia, Alberto Della Vedova, Anna Fino, Nicola Enrietti, Maura Macrì, Tedi Dr\v{a}ghici, Gunnar \TH{}ór Magnússon, Mickael Bordonaro, Samuele Mongodi, Carlo Perrone, Marco Spinaci, Roberto Mossa, Chih-Wei Chen, Herman Stel, Minh Nguyet Mach, David Petrecca, Andrea Villa, Matteo Ruggiero, John Mandereau, Carlo Collari, Andrea Marchese, Luis Ugarte, \dots (and many others: I apologize for forgetting them!).

I would also like to mention how inspiring have been all the teachings by Fulvio Lazzeri, Giuseppe Tomassini, Marco Abate, Angelo Vistoli, Luigi Ambrosio, Adriano Tomassini, Claudio Arezzo, Leonardo Biliotti José Francisco Fernando Galv\'{a}n, Luca Lorenzi, Mario Servi, Cristina Reggiani, Giovanni Ferrero, Celestina Cotti Ferrero, Fabio Zuddas, Alberto Arosio, Pietro Celada, Gianluca Crippa, Domenico Mucci, Carlo Marchini, Alessandra Lunardi, Alessandro Zaccagnini.

The talks at the ``Seminario di Geometria'' organized in Parma in the academic year 2011/2012 revealed to be a very fruitful source of inspiration, and I wish also to thank the Geometry
research group in Parma.

In Pisa, I had the opportunity to attend, not only to the very interesting talks in ``Seminario di Geometria'' and ``Seminario di Calcolo delle Variazioni e Analisi Geometrica'', but also to the PhD students' informal seminars: a particular thanks goes to Abramo Bertucco, Cristina Pagliantini, Valentina Disarlo, Flavia Poma, Alessandro Cobbe, Gian Maria Dall'Ara for having accepted the task to organize ``Seminario dei baby-geometri'' and ``Seminari Informali di Matematica dei Dottorandi''.

I also thank Carlo Petronio, who gave me the opportunity to start knowing the ``teaching side'' of Mathematics, and has been a great example of teacher; my thanks go also to Elena Bucchioni and Valentina Bigini of Liceo Scientifico ``Leonardo da Vinci'' in Villafranca in Lunigiana, and Carlo, Irene, Giovanni, Michele, Giorgio, Luisa, Ilaria, who let me learn how to convey the beauty of Mathematics.

\medskip

Many thanks are due to all the members and the staff of the three Departments of Mathematics where I spend part of my life: the Dipartimento di Matematica ``Leonida Tonelli'' of the University of Pisa; the Dipartimento di Matematica of the University of Parma (our \emph{condominio}); the Institut Fourier in Grenoble. Many thanks to the secretary of the Scuola di Dottorato ``Galileo Galilei'' of the Università di Pisa, Ilaria Fierro.

\medskip

I wish to thank all my PhD colleagues and friends in Pisa, and in particular: Serena (for her nice \emph{serenità}, for her patience in standing my ideas, and \emph{gracias} also for having taught me, among many other things, how to fold \emph{origàmi}, without knowing how to do them!), Jasmin (no word is better than this sentence by Isaia: for ``teaching me how to behave''), Andrea (for his likability, for liking \emph{le renard}, and for his teachings --- and also for his coffees!), Maria Beatrice (for as many ping-pong matches as Tomassini's and Lazzeri's classes attended together, and for having taught me how to properly tumble down Monte Forato) and Luca (for believing in the beauty of silence, for appreciating always my efforts, and especially for the jam!), Minh (for not hating me, even if I starved most of her plants\dots), Isaia, John, Andrea, David, Ana, Tiziano, Andrea, Maria Rosaria, Giovanni, Marco, Sara, Giuseppe, Matteo, Valentina, Laura, Paolo, Laura, Flavia, Stefano, Cristina, Abramo, \dots

For my friends in Parma, with whom I started and (sometimes) continued studying Mathematics, I reserve special thanks to: Daniele, Matteo, Michele, Chiara, Alberto, Marco, Cristiano, Paolo, Simone, Laura, Pietro, Fabio, Amedeo, Eridano, Martino, Alessandra, Giandomenico, Matteo, Luigi, Magda, Francesca, Sara, Michele, Marcello, Emiliano, Luca, Yuri, and everybody I am forgetting to recall. See you soon at the ``Seminario degli ex-studenti''!

My best wishes to the boys and girls working in Mathematics and Physics and living in our beautiful town, Pontremoli, \emph{la Città del Libro}: I hope Francesca, Sara, Michelangelo, Sara, Daniele could reveal and teach to anybody their passion for Science.

\medskip

I wish to end these acknowledgments recalling how important has been the support of my family and, in particular, of my parents Angelo and Antonella. My last thought goes to the memory of my grandfather, \emph{il maestro} Giovanni: he was my first teacher, in letting me learn that even if you may feel frustrating when having no word, Mathematics is always an enjoying game!

}

\mainmatter

\setcounter{chapter}{-1}

\chapter{Preliminaries on (almost-)complex manifolds}\label{chapt:preliminaries}

In this preliminary chapter (which contains no original material), we summarize the basic notions and the classical results concerning (almost-)complex and symplectic structures. In particular, we start by setting some definitions and notation concerning (almost-)complex structures, \S\ref{sec:alm-cplx}, and symplectic structures, \S\ref{sec:symplectic}; then we recall the main results in the Hodge theory for K\"ahler manifolds, \S\ref{sec:hodge-kahler}, and in the Kodaira, Spencer, Nirenberg, and Kuranishi theory of deformations of complex structures, \S\ref{sec:deformations}; furthermore, we summarize the basic definitions and some useful facts about currents and de Rham homology, \S\ref{sec:currents}, and about solvmanifolds, \S\ref{sec:solvmanifolds}, in order to set the notation for the following chapters. (As a matter of notation, unless otherwise stated, by ``manifold'' we mean ``connected differentiable manifold'', and by ``compact manifold'' we mean ``closed manifold''.)

\section{Almost-complex structures and integrability}\label{sec:alm-cplx}
The tangent bundle of a complex manifold $X$ is naturally endowed with an endomorphism $J\in\End(TX)$ such that $J^2=-\id_{TX}$, satisfying a further integrability property. It is hence natural to study differentiable manifolds endowed with such an endomorphism, the so-called \emph{almost-complex} manifolds. It turns out that the vanishing of the \emph{Nijenhuis tensor} $\Nij_J$ characterizes the almost-complex structures $J$ on $X$ naturally induced by a structure of complex manifold, \cite[Theorem 1.1]{newlander-nirenberg}.

In this section, we recall the notions of almost-complex structure, complex manifold, and Dolbeault cohomology, and some of their properties.

\subsection{Almost-complex structures}
Let $X$ be a (differentiable) manifold endowed with an \emph{almost-complex structure} $J$, namely, an endomorphism $J\in\End(TX)$ such that $J^2=-\id_{TX}$.

Extending $J$ by $\C$-linearity to $TX\otimes \C$, we get the decomposition
$$ TX\otimes \C \;=\; T^{1,0}X\oplus T^{0,1}X \;,$$
where $T^{1,0}X$ (respectively, $T^{0,1}X$) is the sub-bundle of $TX\otimes \C$ given by the $\im$-eigen-spaces (respectively, the $\left(-\im\right)$-eigen-spaces) of $J\in\End\left(TX\otimes\C\right)$: that is, for every $x\in X$,
$$ \left(T^{1,0}X\right)_x \;=\; \left\{v_x-\im\, J_x v_x \st v_x\in T_xX \right\} \;, \qquad \left(T^{0,1}X\right)_x \;=\; \left\{v_x+\im\, J_x v_x \st v_x \in T_xX \right\} \;. $$
Considering the dual of $J$, again denoted by $J\in\End\left(T^*X\right)$, we get analogously a decomposition at the level of the cotangent bundle:
$$ T^*X\otimes \C \;=\; \duale{\left(T^{1,0}X\right)}\oplus \duale{\left(T^{0,1}X\right)} \;,$$
where $\duale{\left(T^{1,0}X\right)}$ (respectively, $\duale{\left(T^{0,1}X\right)}$) is the sub-bundle of $T^*X\otimes \C$ given by the $\im$-eigen-spaces (respectively, the $\left(-\im\right)$-eigen-spaces) of the $\C$-linear extension $J\in\End\left(T^*X\otimes\C\right)$.
Extending the endomorphism $J$ to the bundle $\wedge^\bullet \left(T^*X\right)\otimes \C$ of complex-valued differential forms, we get, for every $k\in\N$, the bundle decomposition
$$ \wedge^k \left(T^*X\right) \otimes \C \;=\; \bigoplus_{p+q=k} \wedge^p \duale{\left(T^{1,0}X\right)} \otimes \wedge^{q} \duale{\left(T^{0,1}X\right)} \;.$$
As a matter of notation, we will denote by $\mathcal{C}^\infty\left(X;F\right)$ the space of smooth sections of a vector bundle $F$ over $X$, and, for every $k\in\N$ and $p,q\in\N$, we will denote by $\wedge^kX:=\mathcal{C}^\infty \left(X;\,\wedge^k\left(T^*X\right)\right)$ the space of smooth sections of $\wedge^k\left(T^*X\right)$ over $X$ and by $\wedge^{p,q} X:=:\wedge^{p,q}_JX:=\mathcal{C}^\infty \left(X;\,\wedge^p\duale{\left(T^{1,0}X\right)}\otimes \wedge^q\duale{\left(T^{0,1}X\right)}\right)$ the space of smooth sections of $\wedge^p\duale{\left(T^{1,0}X\right)}\otimes \wedge^q\duale{\left(T^{0,1}X\right)}$ over $X$.

\medskip

Since $\de \left(\wedge^0X \otimes_\R \C \right) \subseteq \wedge^{1,0}X\oplus \wedge^{0,1}X$ and $\de \left(\wedge^1X \otimes_\R \C \right) \subseteq \wedge^{2,0}X\oplus \wedge^{1,1}X \oplus \wedge^{0,2}X$, since every differential form is locally a finite sum of decomposable differential forms, and by the Leibniz rule, the $\C$-linear extension of the exterior differential, $\de\colon\wedge^{\bullet}X\otimes\C\to\wedge^{\bullet+1}X\otimes\C$, splits into four components:
$$ \de \;=\; A + \del + \delbar + \bar A \,$$
where
$$
A \colon \wedge^{\bullet,\bullet}X \to \wedge^{\bullet+2,\bullet-1}X \;, \qquad
\del \colon \wedge^{\bullet,\bullet}X \to \wedge^{\bullet+1,\bullet}X \;, \qquad 
\delbar \colon \wedge^{\bullet,\bullet}X \to \wedge^{\bullet,\bullet+1}X \;, \qquad
\bar A \colon \wedge^{\bullet,\bullet}X \to \wedge^{\bullet-1,\bullet+2}X \;;
$$
in terms of these components, the condition $\de^2=0$ is written as
$$
\left\{
\begin{array}{rcl}
 A^2 &=& 0 \\[5pt]
 A\,\del+\del\,A &=& 0 \\[5pt]
 A\,\delbar+\del^2+\delbar\,A &=& 0 \\[5pt]
 A\,\bar A+\del\,\delbar+\delbar\,\del+A\,\bar A &=& 0 \\[5pt]
 \del\,\bar A+\delbar^2+\bar A\,\del &=& 0 \\[5pt]
 \bar A\,\delbar+\delbar\,\bar A &=& 0 \\[5pt]
 \bar A^2 &=& 0
\end{array}
\right.
\;.
$$

\subsection{Complex structures, and Dolbeault cohomology}
If $X$ is a complex manifold, then there is a natural almost-complex structure on $X$: locally, in a holomorphic coordinate chart $\left(U,\, \left\{z^\alpha=:x^{2\alpha-1}+\im \, x^{2\alpha}\right\}_{\alpha\in\{1,\ldots,\dim_\C X\}}\right)$, with $\left(U,\, \left\{x^\alpha\right\}_{\alpha\in\left\{1,\ldots,2 \dim_\C X\right\}}\right)$ a (differential) coordinate chart, one defines, for every $\alpha\in\left\{1,\ldots,\dim_\C X\right\}$,
$$ J\left(\frac{\del}{\del x^{2\alpha-1}}\right) \;\stackrel{\text{loc}}{:=}\; \frac{\del}{\del x^{2\alpha}} \;, \qquad J\left(\frac{\del}{\del x^{2\alpha}}\right) \;\stackrel{\text{loc}}{:=}\; -\frac{\del}{\del x^{2\alpha-1}} \;;$$
note that this local definition does not depend on the coordinate chart, by the Cauchy and Riemann equations.

Conversely, an almost-complex structure on a manifold $X$ is called \emph{integrable} if it is the natural almost-complex structure induced by a structure of complex manifold on $X$. The following theorem by A. Newlander and L. Nirenberg characterizes the integrable almost-complex structures on a manifold $X$ in terms of the \emph{Nijenhuis tensor} $\Nij_J$, defined as
$$ \Nij_J (\sspace,\,\ssspace) \;:=\; \left[\sspace,\,\ssspace\right]+J\left[J\,{\sspace},\,\ssspace\right]+J\left[\sspace,\,J\,{\ssspace}\right]-\left[J\sspace,\,J\,{\ssspace}\right] \;. $$

\begin{thm}[{\cite[Theorem 1.1]{newlander-nirenberg}}]
 Let $X$ be a manifold. An almost-complex structure $J$ on $X$ is integrable if and only if $\Nij_J=0$.
\end{thm}

By a straightforward computation, the integrability of an almost-complex structure $J$ turns out to be equivalent to the vanishing of the components $A$ and $\bar A$ of the exterior differential, equivalently, to $\left(\wedge^{\bullet,\bullet}X,\,\del,\,\delbar\right)$ being a double complex of $\Cinf\left(X;\C\right)$-modules (see, e.g., \cite[\S2.6]{wells}, \cite[Proposition 8.2]{moroianu}).

\medskip

Therefore, for a complex manifold $X$, one can consider, for every $p\in\N$, the differential complex $\left( \wedge^{p,\bullet}X,\, \delbar\right)$ and its cohomology, defining the \emph{Dolbeault cohomology}, as the bi-graded $\C$-vector space
$$ H^{\bullet,\bullet}_{\delbar}(X) \;:=\; \frac{\ker\delbar}{\imm\delbar} \;.$$

For every $p, q\in\N$, denote by $\mathcal{A}^{p,q}_X$ the (fine) sheaf of germs of $(p,q)$-forms on $X$.
For every $p\in\N$, denote by $\Omega^p_X$ the sheaf of germs of \emph{holomorphic $p$-forms} on $X$, that is, the kernel sheaf of the map $\delbar\colon \mathcal{A}^{p,0}_X\to \mathcal{A}^{p,1}_X$. By the Dolbeault and Grothendieck Lemma, see, e.g., \cite[I.3.29]{demailly-agbook}, one has that
$$ 0 \to \Omega^p_X \to \mathcal{A}^{p,\bullet}_X $$
is a fine resolution of $\Omega^p_X$; hence, one gets the following result.

\begin{thm}[{Dolbeault theorem, \cite{dolbeault}}]
 Let $X$ be a complex manifold. For every $p,\, q\in\N$,
 $$ H^{p,q}_{\delbar}(X) \;\simeq\; \check{H}^q\left(X;\Omega^p\right) \;.$$
\end{thm}

This gives a sheaf-theoretic interpretation of the Dolbeault cohomology. On the other hand, also an analytic interpretation can be provided.

Suppose $X$ is a compact complex manifold of complex dimension $n$, and fix $g$ a Hermitian metric on $X$ and $\vol$ the induced volume form on $X$ (recall that every complex manifold is orientable, see, e.g., \cite[pages 17--18]{griffiths-harris}); denote by $\omega:=g(J\sspace,\, \ssspace)\in\wedge^{1,1}X\cap\wedge^2X$ the associated $(1,1)$-form to $g$. Recall that $g$ induces a Hermitian inner product $\left\langle \sspace,\, \ssspace \right\rangle$ on the space $\wedge^{\bullet,\bullet}X$ of global differential forms on $X$, and that the \emph{Hodge-$*$-operator} associated to $g$ is the $\C$-linear map
$$ *\lfloor_{\wedge^{p,q}X}\colon \wedge^{p,q}X\to \wedge^{n-q,n-p}X $$
defined requiring that, for every $\alpha,\beta\in\wedge^{p,q}X$,
$$ \alpha\wedge * \bar\beta \;=\; \left\langle \alpha,\, \beta\right\rangle \, \vol \;.$$

Define
$$ \delbar^* \;:=\; -*\,\del\,* \colon \wedge^{\bullet,\bullet}X\to \wedge^{\bullet,\bullet-1}X \;; $$
the operator $\delbar^* \colon \wedge^{\bullet,\bullet}X\to \wedge^{\bullet,\bullet-1}X$ is the adjoint of $\delbar \colon \wedge^{\bullet,\bullet}X\to \wedge^{\bullet,\bullet+1}X$ with respect to $\left\langle \sspace,\, \ssspace\right\rangle$. Define
$$ \overline\square \;:=\; \left[\delbar,\,\delbar^*\right] \;:=\; \delbar\,\delbar^*+\delbar^*\,\delbar \colon \wedge^{\bullet,\bullet}X\to \wedge^{\bullet,\bullet}X \;; $$
$\overline\square$ being a \kth{2} order self-adjoint elliptic differential operator, (see, e.g., \cite[Theorem 3.16]{kodaira}), one gets the following result.

\begin{thm}[{Hodge theorem, \cite{hodge}}]
 Let $X$ be a compact complex manifold endowed with a Hermitian metric. There is an orthogonal decomposition
 $$ \wedge^{\bullet,\bullet}X \;=\; \ker \overline\square \,\stackrel{\perp}{\oplus}\, \delbar\wedge^{\bullet,\bullet-1}X \,\stackrel{\perp}{\oplus}\, \delbar^* \wedge^{\bullet,\bullet+1}X \;,$$
 and hence an isomorphism
 $$ H^{\bullet,\bullet}_{\delbar}(X) \;\simeq\; \ker\overline\square \;.$$
 In particular, $\dim_\C H^{\bullet,\bullet}_{\delbar}(X)<+\infty$.
\end{thm}

Note that, for any $p,q\in\N$, the Hodge-$*$-operator $*\colon \wedge^{p,q}X\to \wedge^{n-q,n-p}X$ sends a $\overline\square$-harmonic $(p,q)$-form $\psi$ (that is, $\psi\in\wedge^{p,q}X$ is such that $\overline\square\psi=0$) to a $\square$-harmonic $(n-q,n-p)$-form $*\psi$, where $\square:=\left[\del,\,\del^*\right]:=\del\del^*+\del^*\del\in\End\left(\wedge^{\bullet,\bullet}X\right)$ is the conjugate operator to $\overline\square$, and hence, by conjugating, one gets a $\overline\square$-harmonic $(n-p,n-q)$-form $\overline{*\psi}$. Hence, one gets the following result.

\begin{thm}[{Serre duality, \cite[Théorème 4]{serre-duality}}]
Let $X$ be a compact complex manifold of complex dimension $n$, endowed with a Hermitian metric. For every $p,\, q\in\N$, the Hodge-$*$-operator induces an isomorphism
$$ * \colon H^{p,q}_\delbar (X) \stackrel{\simeq}{\to} \overline{H^{n-p,n-q}_{\delbar}(X)} \;.$$
\end{thm}

\medskip

Since a $\delbar$-closed form is not necessarily $\de$-closed, Dolbeault cohomology classes do not define, in general, de Rham cohomology classes, that is, in general, on a compact complex manifold, there is no natural map between the Dolbeault cohomology and the de Rham cohomology (as we will see, in the special case of compact K\"ahler manifolds, or more in general of compact complex manifolds satisfying the $\del\delbar$-Lemma, the de Rham cohomology actually can be decomposed by means of the Dolbeault cohomology groups, \cite[Théorème IV.3]{weil}, \cite[Lemma 5.15, Remark 5.16, 5.21]{deligne-griffiths-morgan-sullivan}). Nevertheless, the Fr\"olicher inequality provides a relation between the dimension of the Dolbeault cohomology and the dimension of the de Rham cohomology; it follows by considering the Hodge and Fr\"olicher spectral sequence, which we recall here.

The structure of double complex of $\left(\wedge^{\bullet,\bullet}X,\,\del,\,\delbar\right)$ gives rise to two natural filtrations of $\wedge^{\bullet}X\otimes\C$, namely, (for $p,q\in N$ and for $k\in\N$,)
$$ 'F^p\left(\wedge^{k}X\otimes\C\right) \;:=\; \bigoplus_{\substack{r+s=k\\r\geq p}}\wedge^{r,s}X \quad \text{ and }\quad
  ''F^q\left(\wedge^{k}X\otimes\C\right) \;:=\; \bigoplus_{\substack{r+s=k\\s\geq q}}\wedge^{r,s}X \;; $$
these filtrations induce two spectral sequences (see, e.g., \cite[\S2.4]{mccleary}, \cite[\S3.5]{griffiths-harris}),
$$ \left\{\left(E^{\bullet,\bullet}_r,\, \de_r\right) \;:=:\; \left('E^{\bullet,\bullet}_r,\, {'\de}_r \right)\right\}_{r\in\N} \qquad \text{ and, respectively, } \qquad \left\{\left(''E^{\bullet,\bullet}_r,\, {''\de}_r \right)\right\}_{r\in\N} \;,$$
called \emph{Hodge and Fr\"olicher spectral sequences} (or \emph{Hodge to de Rham spectral sequences}): one has
$$ 'E^{\bullet,\bullet}_1 \;\simeq\; H^{\bullet,\bullet}_{\delbar}(X)\;\Rightarrow\; H^{\bullet}_{dR}(X;\C) \quad\text{ and }\quad
  ''E^{\bullet,\bullet}_1 \;\simeq\; H^{\bullet,\bullet}_{\del}(X)\;\Rightarrow\; H^{\bullet}_{dR}(X;\C) \;.$$
An explicit description of $\left\{\left(E_r,\, \de_r\right)\right\}_{r\in\N}$ is given in \cite{cordero-fernandez-ugarte-gray_spectral}: for any $p,\,q\in\N$ and $r\in\N$, its terms are
$$ E^{p,q}_r \;\simeq\; \frac{\mathcal{X}^{p,q}_r}{\mathcal{Y}^{p,q}_r} \;,$$
where, for $r=1$,
$$
 \mathcal{X}^{p,q}_1 \;:=\; \left\{\alpha\in\wedge^{p,q}X \st \delbar \alpha=0 \right\}\;,
 \qquad
 \mathcal{Y}^{p,q}_1 \;:=\; \delbar \wedge^{p,q-1}X \;,
$$
and, for $r\geq 2$,
\begin{eqnarray*}
 \mathcal{X}^{p,q}_r &:=& \left\{\alpha^{p,q}\in\wedge^{p,q}X \st \delbar \alpha^{p,q}=0 \text{ and, for any }i\in\{1, \ldots, r-1\}\text{, there exists } \alpha^{p+i,q-i}\in\wedge^{p+i,q-i}X \right.\\[5pt]
 && \left. \text{ such that }\del\alpha^{p+i-1,q-i+1}+\delbar\alpha^{p+i,q-i}=0 \right\}\;, \\[10pt]
 \mathcal{Y}^{p,q}_r &:=& \left\{\del\beta^{p-1,q}+\delbar\beta^{p,q-1}\in\wedge^{p,q}X \st \text{for any }i\in\{2, \ldots, r-1\}\text{, there exists } \beta^{p-i,q+i-1}\in\wedge^{p-i,q+i-1}X \right. \\[5pt]
 && \left. \text{ such that } \del\beta^{p-i,q+i-1}+\delbar\beta^{p-i+1,q+i-2}=0 \text{ and }\delbar\beta^{p-r+1, q+r-2}=0\right\} \;,
\end{eqnarray*}
see \cite[Theorem 1]{cordero-fernandez-ugarte-gray_spectral}, and, for any $r\geq 1$, the map $\de_r\colon E^{\bullet,\bullet}_r\to E^{\bullet+r, \bullet-r+1}_r$ is given by
$$ \de_r\colon \left\{\left[\alpha^{p,q}\right] \in E^{p,q}_r\right\}_{p,q\in\N} \mapsto \left\{\left[\del\alpha^{p+r-1, q-r+1}\right] \in E^{p+r, q-r+1}_r\right\}_{p,q\in\N} \;,$$
see \cite[Theorem 3]{cordero-fernandez-ugarte-gray_spectral}.

As a consequence of $'E^{\bullet,\bullet}_1 \simeq H^{\bullet,\bullet}_{\delbar}(X)\;\Rightarrow\; H^{\bullet}_{dR}(X;\C)$, one gets the following  inequality by A. Fr\"olicher.

\begin{thm}[{Fr\"olicher inequality, \cite[Theorem 2]{frolicher}}]
Let $X$ be a compact complex manifold. Then, for every $k\in\N$,
$$ \dim_\C H^k_{dR}(X;\C) \;\leq\; \sum_{p+q=k} \dim_\C H^{p,q}_{\delbar}(X) \;. $$
\end{thm}

As a matter of notation, for $k\in\N$ and $p,q\in\N$, we will denote by $b_k := \dim_\R H^k_{dR}(X;\R)$, respectively $h^{p,q}_{\delbar} := \dim_\C H^{p,q}_{\delbar}(X)$, the \emph{\kth{k} Betti number}, respectively the \emph{\kth{(p,q)} Hodge number} of $X$.

In the next chapter, we will provide a Fr\"olicher-type inequality also for the Bott-Chern cohomology, Theorem \ref{thm:frol-bc}, showing that it allows to characterize the compact complex manifolds satisfying the $\del\delbar$-Lemma just in terms of the dimensions of the Bott-Chern cohomology and of the de Rham cohomology, Theorem \ref{thm:caratterizzazione-bc-numbers}.

\begin{rem}
 Other than the Dolbeault cohomology, other cohomologies can be defined for a complex manifold $X$; more precisely, since, for every $p,q\in\N$,
 $$ \wedge^{p-1,q-1}X \stackrel{\del\delbar}{\to} \wedge^{p,q}X \stackrel{\del+\delbar}{\to} \wedge^{p+1,q}X\oplus\wedge^{p,q+1}X \qquad \text{ and } \qquad \wedge^{p-1,q}X \oplus \wedge^{p,q-1}X \stackrel{\left(\del,\, \delbar\right)}{\to} \wedge^{p,q}X \stackrel{\del\delbar}{\to} \wedge^{p+1,q+1}X $$
 are complexes, one can define the \emph{Bott-Chern cohomology} $H^{\bullet,\bullet}_{BC}(X)$ and the \emph{Aeppli cohomology} $H^{\bullet,\bullet}_{A}(X)$ of $X$ as
 $$ H^{\bullet,\bullet}_{BC}(X) \;:=\; \frac{\ker\del\cap\ker\delbar}{\imm\del\delbar} \qquad \text{ and }\qquad H^{\bullet,\bullet}_{A}(X) \;:=\; \frac{\ker\del\delbar}{\imm\del+\imm\delbar} \;;$$
 we refer to \S\ref{sec:preliminaries-bott-chern-deldelbar} for further details.
\end{rem}

\section{Symplectic structures}\label{sec:symplectic}
In this section, we recall some definitions and results concerning symplectic manifolds, that is, differentiable manifolds endowed with a non-degenerate $\de$-closed $2$-form. An interesting class of examples of symplectic manifolds is provided by the K\"ahler manifolds. Moreover, given a differentiable manifold $X$, its cotangent bundle $T^*X$ is endowed with a natural symplectic structure (see, e.g., \cite[\S2]{cannasdasilva}): in fact, Symplectic Geometry has applications and motivations in the study of Hamiltonian Mechanics, see, e.g., \cite[Part VII]{cannasdasilva}.

\medskip

Let $X$ be a compact $2n$-dimensional manifold endowed with a \emph{symplectic form}, namely, a non-degenerate $\de$-closed $2$-form $\omega\in\wedge^2X$.

The main difference between Symplectic Geometry and Riemannian Geometry is provided by G. Darboux's theorem.

\begin{thm}[{Darboux theorem, \cite{darboux}}]
Let $X$ be a $2n$-dimensional manifold endowed with a symplectic form $\omega$. Then, for every $x\in X$, there exists a coordinate chart $\left(U, \, \left\{x^j\right\}_{j\in\{1,\ldots,2n\}}\right)$, with $x\in U$, such that
 $$ \omega \;\stackrel{\text{loc}}{=}\; \sum_{j=1}^{n} \de x^{2j-1}\wedge \de x^{2j} \;.$$
 \end{thm}

\medskip

By exploiting the parallelism with Riemannian Geometry, one can try to develop a Hodge theory also for compact symplectic manifolds, \cite{brylinski}. The first tool that can be introduced is an analogue of the Hodge-$*$-operator.

Note that every symplectic manifold is orientable, $\frac{\omega^n}{n!}$ giving a canonical orientation.

Denote by $I\colon TX\to T^*X$ the natural isomorphism of vector bundles induced by $\omega$, namely, $I(v)(\sspace):=\omega(v,\sspace)\in\Hom\left(T_xX;\R\right)$, for every $v\in T_xX$ and $x\in X$. Then, for every $k\in\N$, the form $\omega$ gives rise to a bi-$\mathcal{C}^\infty(X;\R)$-linear form on $\wedge^k X$ denoted by $\left(\omega^{-1}\right)^k$, which is skew-symmetric, respectively symmetric, according that $k$ is odd, respectively even, and defined on the simple elements $\alpha^1\wedge\ldots \wedge \alpha^k,\, \beta^1\wedge\ldots \wedge \beta^k\in\wedge^kX$ as
$$
\left(\omega^{-1}\right)^k\left(\alpha^1\wedge\ldots \wedge \alpha^k,\beta^1\wedge\ldots \wedge \beta^k\right) \;:=\; \det\left(\omega^{-1}\left(\alpha^\ell,\beta^m\right)\right)_{\ell,m\in\{1,\ldots,k\}} \;,
$$
where $\omega^{-1}\left(\alpha^\ell,\beta^m\right) \;:=\; \omega\left(I^{-1}\left(\alpha^\ell\right),I^{-1}\left(\beta^m\right)\right)$ for every $\ell,m\in\{1,\ldots,k\}$.
In a Darboux coordinate chart $\left(U,\, \left\{x^j\right\}_{j\in\{1,\ldots, 2n\}}\right)$, the canonical \emph{Poisson bi-vector} $\Pi:=\omega^{-1}\in\wedge^2TX$ associated to $\omega$ is written as $\omega^{-1}\stackrel{\text{loc}}{=}\sum_{j=1}^{n}\frac{\del}{\del x^{2j-1}}\wedge\frac{\del}{\del x^{2j}}$.

The \emph{symplectic-$\star$-operator}
$$
\star_\omega\colon \wedge^\bullet X\to \wedge^{2n-\bullet}X \;,
$$
introduced by J.-L. Brylinski, \cite[\S2]{brylinski}, is defined requiring that, for every $k\in\N$, and for every $\alpha,\,\beta\in\wedge^kX$,
$$ \alpha\wedge\star_\omega \beta \;=\; \left(\omega^{-1}\right)^k\left(\alpha,\beta\right)\,\frac{\omega^n}{n!} \;.$$

\medskip

As for (almost-)complex manifolds, on a symplectic manifold $X$ one has a decomposition of differential forms in symplectic-type components, the so-called Lefschetz decomposition; it is a consequence of a $\mathfrak{sl}(2;\R)$-representation on $\wedge^2X$ by means of operators related to the symplectic structure.

More precisely, define the operators $L,\, \Lambda,\, H\in\End^\bullet\left(\wedge^\bullet X\right)$ as
\begin{eqnarray*}
L\colon \wedge^\bullet X\to \wedge^{\bullet+2}X\;, \quad && \alpha\mapsto \omega\wedge\alpha \;,\\[5pt]
\Lambda\colon \wedge^\bullet X\to \wedge^{\bullet-2}X\;, \quad && \alpha \mapsto -\iota_\Pi\alpha \;,\\[5pt]
H\colon \wedge^\bullet X\to \wedge^\bullet X\;, \quad && \alpha\mapsto \sum_k \left(n-k\right)\,\pi_{\wedge^kX}\alpha
\end{eqnarray*}
(where $\iota_{\xi}\colon\wedge^{\bullet}X\to\wedge^{\bullet-2}X$ denotes the interior product with $\xi\in\wedge^2\left(TX\right)$, and, for $k\in\N$, the map $\pi_{\wedge^kX}\colon\wedge^\bullet X\to\wedge^kX$ denotes the natural projection onto $\wedge^kX$).
Note that, using the symplectic-$\star$-operator $\star_\omega$, one can write, \cite[Lemma 1.5]{yan},
$$
\Lambda \;=\; -\star_\omega\,L\,\star_\omega \;.
$$

The following result holds.

\begin{thm}[{\cite[Corollary 1.6]{yan}}]
Let $X$ be a manifold endowed with a symplectic structure. Then
$$ \left[L,\, H\right] \;=\; 2\, L \;, \qquad \left[\Lambda,\, H\right] \;=\; -2\, \Lambda \;, \qquad  \left[L,\, \Lambda\right] \;=\; H \;, $$
and hence
$$ \mathfrak{sl}(2;\R) \;\simeq\; \left\langle L, \, \Lambda, \, H \right\rangle \to \End^\bullet\left(\wedge^\bullet X\right) $$
gives a $\mathfrak{sl}(2;\R)$-representation on $\wedge^\bullet X$.
\end{thm}

(See, e.g., \cite[\S7]{humphreys} for general results concerning $\mathfrak{sl}(2;\R)$-representations.)

The above $\mathfrak{sl}(2;\R)$-representation, having finite $H$-spectrum, induces a decomposition of the space of the differential forms.

\begin{thm}[{\cite[Corollary 2.6]{yan}}]
Let $X$ be a manifold endowed with a symplectic structure. Then one has the \emph{Lefschetz decomposition} on differential forms,
$$ \wedge^\bullet X \;=\; \bigoplus_{r\in\N} L^r \, \Prim^{\bullet-2r}X \;, $$
where
$$ \Prim^\bullet X \;:=\; \ker\Lambda $$
is the space of \emph{primitive forms}.
\end{thm}

Note (see, e.g., \cite[Proposition 1.2.30(v)]{huybrechts}) that, for every $k\in\N$,
$$ \Prim^kX \;=\; \ker L^{n-k+1}\lfloor_{\wedge^kX} \;.$$
In general, see, e.g., \cite[pages 7--8]{tseng-yau-2}, the Lefschetz decomposition of $A^{(k)}\in\wedge^kX$ reads as
$$ A^{(k)} \;=\; \sum_{r\geq \max\left\{k-n,\,0\right\}}\frac{1}{r!}\, L^r\, B^{(k-2r)} $$
where, for $r\geq\max\left\{k-n,\,0\right\}$,
$$ B^{(k-2r)} \;:=\; \left(\sum_{\ell\in\N}a_{r,\ell,(n,k)}\,\frac{1}{\ell!}\, L^\ell\, \Lambda^{r+\ell}\right)\, A^{(k)} \;\in\; \Prim^{k-2r}X $$
and, for $r\geq\max\left\{k-n,\,0\right\}$ and $\ell\in\N$,
$$ a_{r,\ell,(n,k)} \;:=\; \left(-1\right)^{\ell} \cdot \left(n-k+2r+1\right)^2 \cdot \prod_{i=0}^{r}\frac{1}{n-k+2r+1-i} \cdot \prod_{j=0}^{\ell}\frac{1}{n-k+2r+1+j} \in \Q \;. $$

\medskip

We recall that
$$ L\lfloor_{\bigoplus_{k=-1}^{n-2}\wedge^{n-k-2}X} \colon \bigoplus_{k=-1}^{n-2}\wedge^{n-k-2}X \to \wedge^{n-k}X $$
is injective, \cite[Corollary 2.8]{yan}, and that, for every $k\in\N$,
$$ L^k\colon \wedge^{n-k}X\to\wedge^{n+k}X $$
is an isomorphism, \cite[Corollary 2.7]{yan}.

\medskip

Since $\left[L,\, \de\right]=0$, for any $k\in\N$, the map $L^k\colon \wedge^{n-k}X\to\wedge^{n+k}X$ induces a map $L^k\colon H^{n-k}_{dR}(X;\R)\to H^{n+k}_{dR}(X;\R)$ in cohomology. One says that $X$ satisfies the \emph{Hard Lefschetz Condition}, shortly \emph{\textsc{hlc}}, if
\begin{equation}
\tag{HLC}
 \text{for every } k\in\N\;, \qquad L^k\colon H^{n-k}_{dR}(X;\R) \stackrel{\simeq}{\to} H^{n+k}_{dR}(X;\R) \;.
\end{equation}

\medskip

By continuing in the parallelism between Riemannian Geometry and Symplectic Geometry, one can introduce the $\de^\Lambda$ operator with respect to a symplectic structure $\omega$ as
$$ \de^\Lambda\lfloor_{\wedge^kX} \;:=\; (-1)^{k+1}\star_\omega \de \star_\omega $$
for any $k\in\N$, and interpret it as the symplectic counterpart of the Riemannian $\de^*$ operator with respect to a Riemannian metric. In light of this, J.-L. Brylinski proposed in \cite{brylinski} a Hodge theory for compact symplectic manifolds, conjecturing that, on a compact manifold endowed with a symplectic structure $\omega$, every de Rham cohomology class admits a (possibly non-unique) \emph{$\omega$-symplectically-harmonic representative}, namely, a $\de$-closed $\de^\Lambda$-closed representative, \cite[Conjecture 2.2.7]{brylinski}. (Note that $\de\de^\Lambda+\de^\Lambda\de=0$, \cite[Theorem 1.3.1]{brylinski}, \cite[page 265]{koszul}, provides a strong difference in the parallelism between Symplectic Geometry and Riemannian geometry; in particular, it follows that a $\omega$-symplectically-harmonic representative, whenever it exists, is not unique.)

For an almost-K\"ahler structure $\left(J,\,\omega,\,g\right)$ on a compact manifold $X$ (that is, $\omega\in\wedge^2X$ is a symplectic form on $X$, $J\in\End\left(TX\right)$ is an almost-complex structure on $X$, and $g$ is a $J$-Hermitian metric on $X$ such that $\omega$ is the associated $(1,1)$-form to $g$), the symplectic-$\star$-operator $\star_\omega$ and the Hodge-$*$-operator $*_g$ are related by
$$ \star_\omega \;=\; J\,*_g \;,$$
and hence
$$ \de^\Lambda \;=\; -\left(\de^c\right)^{*_g} $$
where $\de^c:=J^{-1}\,\de\,J$ and $\left(\de^c\right)^{*_g}\lfloor_{\wedge^kX}:=\left(-1\right)^{k+1}*_g\,\de\,*_g$ for every $k\in\N$ (note that, when $J$ is integrable, then $\de^c=-\im\left(\del-\delbar\right)$). Moreover, on a compact manifold $X$ endowed with a K\"ahler structure $\left(J,\, \omega,\, g\right)$, by the Hodge decomposition theorem, \cite[Théorème IV.3]{weil}, the pure-type components with respect to $J$ of the harmonic representatives of the de Rham cohomology classes are themselves harmonic. Hence, it follows that Brylinski's conjecture holds true for compact K\"ahler manifolds, \cite[Corollary 2.4.3]{brylinski}.

O. Mathieu in \cite{mathieu}, and D. Yan in \cite{yan}, provided counterexamples to Brylinski's conjecture, characterizing the compact symplectic manifolds satisfying Brylinski's conjecture in terms of the validity of the Hard Lefschetz Condition. Furthermore, S.~A. Merkulov in \cite{merkulov}, see also \cite{cavalcanti}, and V. Guillemin in \cite{guillemin}, proved that the Hard Lefschetz Condition on compact symplectic manifolds is equivalent to satisfying the $\de\de^\Lambda$-Lemma, namely, to every $\de$-exact $\de^\Lambda$-closed form being $\de\de^\Lambda$-exact. Summarizing, we recall the following result. 
\begin{thm}[{\cite[Corollary 2]{mathieu}, \cite[Theorem 0.1]{yan}, \cite[Proposition 1.4]{merkulov}, \cite{guillemin}, \cite[Theorem 5.4]{cavalcanti}}]
 Let $X$ be a compact manifold endowed with a symplectic structure $\omega$. The following conditions are equivalent:
 \begin{enumerate}[\itshape (i)]
  \item every de Rham cohomology class admits a representative being both $\de$-closed and $\de^\Lambda$-closed (i.e., Brylinski's conjecture \cite[Conjecture 2.2.7]{brylinski} holds true on $X$);
  \item $X$ satisfies the Hard Lefschetz Condition;
  \item $X$ satisfies the $\de\de^\Lambda$-Lemma.
 \end{enumerate}
\end{thm}

Note that, by the Lefschetz decomposition theorem, \cite[Théorème IV.5]{weil} (see \S\ref{sec:hodge-kahler}), compact K\"ahler manifolds satisfy the Hard Lefschetz Condition.

\medskip

\begin{rem}
The Complex Generalized Geometry, introduced by N.~J. Hitchin in \cite{hitchin} and developed, among others, by M. Gualtieri, \cite{gualtieri-phd, gualtieri}, and G.~R. Cavalcanti, \cite{cavalcanti}, see also \cite{hitchin-introduction, cavalcanti-impa}, allows to frame symplectic structures and complex structures in the same context (in a sense, this add more significance to the term ``symplectic'', which was invented by H. Weyl, \cite[\S VI]{weyl}, substituting the Greek root in the term ``complex'' with the corresponding Latin root).
In such a framework, the $\de^\Lambda$ operator associated to a symplectic structure should be interpreted as the symplectic counterpart of the operator $\de^c:=-\im\,\left(\del-\delbar\right)$ associated to a complex structure, \cite{cavalcanti}.
\end{rem}

\section{K\"ahler structures and cohomological decomposition}\label{sec:hodge-kahler}
Note that, given a manifold $X$ endowed with a symplectic form $\omega$, there is always a (possibly non-integrable) almost-complex structure $J$ on $X$ such that $g:=\omega(\sspace, \, J\ssspace)$ is a Hermitian metric on $X$ with $\omega$ as the associated $(1,1)$-form, see, e.g., \cite[Corollary 12.7]{cannasdasilva} (in fact, the set of such almost-complex structures is contractible, see, e.g., \cite[Corollary II.1.1.7]{audin-lafontaine}, \cite[Proposition 13.1]{cannasdasilva}; see also \cite[Corollary 2.3.C$_2'$]{gromov}, which proves that the space of almost-complex structures on $X$ tamed by a given $2$-form on $X$ is contractible). Instead, the datum of an integrable almost-complex structure with the above property yields a K\"ahler structure on $X$. The notion of K\"ahler manifold has been studied for the first time by J.~A. Schouten and D. van Dantzig \cite{schouten-vandantzig}, see also \cite{schouten}, and by E. K\"ahler \cite{kahler}, and the terminology has been fixed by A. Weil \cite{weil}.

K\"ahler structures can be defined in different ways, according to the point of view which is stressed, \S\ref{subsec:kahler-def}. The presence of three different structures (complex, symplectic, and Riemannian) allows to make use of the tools available for any of them; in addition, the relations between such structures make available further tools, which yield many interesting results on Hodge theory, \S\ref{subsec:kahler-hodge}. Finally, we will study a cohomological property of compact K\"ahler manifolds, namely, the $\del\delbar$-Lemma, \S\ref{subsec:kahler-deldelbar}: other than being a very useful tool in K\"ahler Geometry (compare, e.g., its role in S.-T. Yau's proof \cite{yau-proc, yau} of E. Calabi's conjecture \cite{calabi}), it provides obstructions to the existence of K\"ahler structures on differentiable manifolds, by means of the notion of formality introduced by D.~P. Sullivan, \cite[\S12]{sullivan-ihes}.

\subsection{K\"ahler metrics}\label{subsec:kahler-def}

Let $X$ be a compact complex manifold of complex dimension $n$, and denote by $J$ its natural integrable almost-complex structure.

A \emph{K\"ahler metric} on $X$ is a Hermitian metric $g$ such that the associated $(1,1)$-form $\omega:=g(J\sspace,\, \ssspace)$ is $\de$-closed (that is, $\omega$ is a symplectic form on $X$).

\begin{rem}
 Let $X$ be a complex manifold endowed with a K\"ahler metric $g$, and denote the associated $(1,1)$-form to $g$ by $\omega$. By the Poincaré lemma, see, e.g., \cite[I.1.22, Theorem I.2.24]{demailly-agbook}, and the Dolbeault and Grothendieck lemma, see, e.g., \cite[I.3.29]{demailly-agbook}, the property that $\de\omega=0$ is equivalent to ask that, for every $x\in X$, there exist an open neighbourhood $U$ in $X$ with $x\in U$ and a smooth function $u\in\mathcal{C}^\infty(U;\R)$ such that $\omega \stackrel{\text{loc}}{=} \im \del\delbar u$ in $U$, that is, the metric has a local potential, \cite{kahler} (see, e.g., \cite[Proposition 8.8]{moroianu}).
\end{rem}

\begin{rem}
For every $n\in\N$, the complex projective space $\CP^n$ admits a K\"ahler metric, the so-called \emph{Fubini and Study metric}, \cite{fubini, study}, which is induced by the fibration $\mathbb{S}^1 \hookrightarrow \mathbb{S}^{2n+1} \to \CP^n$; more precisely, by using the homogeneous coordinates $\left[z_0 \,:\, \cdots \,:\, z_n\right]$, one has that the associated $(1,1)$-form $\omega_{\text{FS}}$ to the Fubini and Study metric is
$$ \omega_{\text{FS}} \;=\; \frac{\im}{2\pi} \, \del\delbar \log\left(\sum_{\ell=0}^{n} \left|z^\ell\right|^2\right) \;.$$
It follows that complex projective manifolds provide examples of K\"ahler manifolds. Conversely, by the Kodaira embedding theorem \cite[Theorem 4]{kodaira-embedding}, if $X$ is a compact complex manifold endowed with a K\"ahler metric $\omega$ such that $\left[\omega\right]\in H^2_{dR}(X;\R) \cap \imm\left(H^2(X;\Z)\to H^2_{dR}(X;\R)\right)$, then there exists a complex-analytic embedding of $X$ into a complex projective space $\CP^N$ for some $N\in\N$. In a sense, this suggest that projective manifolds are to K\"ahler manifolds as $\Q$ is to $\R$. Hence, it is natural to ask if every compact K\"ahler manifold is a deformation of a projective manifold (which is known as the \emph{Kodaira problem}). Since Riemann surfaces are projective, this is trivially true in complex dimension $1$. Furthermore, K. Kodaira proved in \cite[Theorem 16.1]{kodaira-surfaces-3} that every compact K\"ahler surface is a deformation of an algebraic  surface, as conjectured by W. Hodge; another proof, which does not make use of the 
classification of elliptic surfaces, has been given by N. Buchdahl, \cite[Theorem]{buchdahl-deformations}. In higher dimension, a negative answer to the Kodaira problem has been given by C. Voisin, who constructed examples of compact K\"ahler manifolds, of any complex dimension greater than or equal to $4$, which do not have the homotopy type of a complex projective manifold, \cite[Theorem 2]{voisin-invent} (indeed, recall that, by Ehresmann's theorem, if two compact complex manifolds can be obtained by deformation, then they are homeomorphic, and hence they have the same homotopy type). The examples in \cite{voisin-invent} being, by construction, bimeromorphic to manifolds that can be deformed to projective manifolds, one could ask (as done by N. Buchdahl, F. Campana, S.-T. Yau) whether, in higher dimension, a birational version of the Kodaira problem may hold true; in \cite[Theorem 3]{voisin-jdg}, C. Voisin provided a negative answer to the birational version of the Kodaira problem, proving that, in any 
even complex 
dimension greater that or equal to $10$, there exist compact K\"ahler manifolds $X$ such that, for any compact K\"ahler manifold $X'$ bimeromorphic to $X$, $X'$ does not have the homotopy type of a projective complex manifold.
\end{rem}

\medskip

In the definition of a K\"ahler manifold, three different structures are involved: a complex structure, a symplectic structure, and a metric structure. Therefore, changing the point of view allows to give several equivalent definitions of K\"ahler structure (see, e.g., \cite[Theorem 4.17]{ballmann}): we review here two of these characterizations.

Firstly, it is straightforward to prove that a Hermitian metric $g$ on a compact complex manifold $X$ is a K\"ahler metric if and only if, for every point $x\in X$, there exists a holomorphic coordinate chart $\left(U,\,\left\{z^j\right\}_{j\in\{1,\ldots, n\}}\right)$, with $x\in U$, such that
$$ g \;=\; \sum_{\alpha,\beta=1}^{n} \left(\delta_{\alpha\beta}+\opiccolo{z}\right) \,\de z^\alpha\odot\de \bar z^\beta \qquad \text{ at } x \;,$$
that is, $g$ \emph{osculates to order $2$} the standard Hermitian metric of $\C^n$ (see, e.g., \cite[pages 107--108]{griffiths-harris}, \cite[Proposition 1.3.12]{huybrechts}, \cite[Theorem 11.6]{moroianu}).

As regards the second characterization, we recall that, on a compact complex manifold $X$ endowed with a Hermitian metric $g$, there is a unique connection $\nabla^{C}$ such that
\begin{enumerate}[\itshape (i)]
 \item $\nabla^Cg=0$,
 \item $\nabla^CJ=0$, and
 \item $\pi_{\wedge^{0,1}X}\nabla^C\lfloor_{\mathcal{C}^\infty\left(X;\C\right)}=\delbar\lfloor_{\mathcal{C}^\infty\left(X;\C\right)}$;
\end{enumerate}
such a connection is called the \emph{Chern connection} of $X$ (see, e.g., \cite[Proposition 4.2.14]{huybrechts}, \cite[Theorem 3.18]{ballmann}, \cite[Theorem 10.3]{moroianu}).
Let $g$ be a Hermitian metric on a compact complex manifold $X$, and set $\omega:=g(J\sspace,\, \ssspace)$ its associated $(1,1)$-form, where $J$ is the natural integrable almost-complex structure on $X$; consider the Levi Civita connection $\nabla^{LC}$. One can prove that, for every $x,y,z\in\mathcal{C}^\infty\left(X;TX\right)$,
$$ \de\omega(x,y,z) \;=\; g\left(\left(\nabla^{LC}_xJ\right)y,\, z\right) + g\left(\left(\nabla^{LC}_yJ\right)z,\, x\right) + g\left(\left(\nabla^{LC}_zJ\right)x,\, y\right) \;, $$
and
$$ 2\, g\left(\left(\nabla^{LC}_xJ\right)y,\, z\right) \;=\; \de\omega\left(x,\, y,\, z\right) - \de\omega\left(x, Jy,\, Jz\right) - g\left(\Nij_J\left(y,\,Jz\right),\, x\right) \;;$$
(see, e.g., \cite[Theorem 4.16]{ballmann}, \cite[Proposition 1.5]{tian}); in particular, it follows that $g$ is a K\"ahler metric if and only if $\nabla^{LC}J=0$ if and only if the Chern connection is the Levi Civita connection (see, e.g., \cite[Theorem 4.17]{ballmann}, \cite[Proposition 11.8]{moroianu}).

\subsection{Hodge theory for K\"ahler manifolds}\label{subsec:kahler-hodge}

The complex, symplectic, and metric structures being related on a K\"ahler manifold, one gets the following identities concerning the corresponding operators (see, e.g., \cite[Proposition 3.1.12]{huybrechts}); see also \cite{hodge-plms, hodge}. (In \cite[Theorem 1.1, Theorem 2.12]{demailly-lnm}, commutation relations on arbitrary Hermitian manifolds are provided; see also \cite{griffiths-ajm}, \cite[\S VI.6.2]{demailly-agbook}.)

\begin{thm}[{K\"ahler identities, \cite[Théorème II.1, Théorème II.2, Corollaire II.1]{weil}}]
 Let $X$ be a compact K\"ahler manifold. Consider the differential operators $\del$ and $\delbar$ associated to the complex structure, the symplectic operators $L$ and $\Lambda$ associated to the symplectic structure, and the Hodge-$*$-operator associated to the Hermitian metric. Then, these operators are related as follows:
 \begin{enumerate}[\itshape (i)]
  \item $\left[\delbar,\, L\right]=\left[\del,\, L\right]=0$ and $\left[\Lambda,\, \delbar^*\right]=\left[\Lambda,\, \del^*\right]=0$;
  \item $\left[\delbar^*,\, L\right]=\im\,\del$ and $\left[\del^*,\, L\right]=-\im\,\delbar$, and $\left[\Lambda,\, \delbar\right]=-\im\,\del^*$ and $\left[\Lambda,\, \del\right]=\im\,\delbar^*$.
 \end{enumerate}
Therefore, considering the \kth{2} order self-adjoint elliptic differential operators $\square:=\left[\del,\, \del^*\right]$, $\overline\square:=\left[\delbar,\, \delbar^*\right]$, and $\Delta:=\left[\de,\, \de^*\right]$, one gets that
 \begin{enumerate}[\itshape (i)] \setcounter{enumi}{2}
  \item $\square=\overline\square=\frac{1}{2}\Delta$, and $\Delta$ commutes with $*$, $\del$, $\delbar$, $\del^*$, $\delbar^*$, $L$, $\Lambda$.
 \end{enumerate}
\end{thm}

The previous identities can be proven either using the $\mathfrak{sl}\left(2;\C\right)$ representation $\left\langle L,\, \Lambda,\, H \right\rangle \to \End^\bullet\left(\wedge^\bullet X\otimes \C\right)$, or reducing to prove the corresponding identities on $\C^n$ with the standard K\"ahler structure (which are known as Y. Akizuki and S. Nakano's identities, \cite[\S3]{akizuki-nakano}) and hence using that every K\"ahler metric osculates to order $2$ the standard Hermitian metric on $\C^n$.

As a consequence, one gets the following theorems, stating a decomposition of the de Rham cohomology of a K\"ahler manifold related to the complex, respectively symplectic, structure (see, e.g., \cite[Corollary 3.2.12]{huybrechts}, respectively \cite[Proposition 3.2.13]{huybrechts}).

\begin{thm}[{Hodge decomposition theorem, \cite[Théorème IV.3]{weil}}]
 Let $X$ be a compact complex manifold endowed with a K\"ahler structure. Then there exist a decomposition
 $$ H^\bullet_{dR}(X;\C) \;\simeq\; \bigoplus_{p+q=\bullet} H^{p,q}_\delbar(X) \;,$$
 and, for every $p,q\in\N$, an isomorphism
 $$ H^{p,q}_\delbar(X) \;\simeq\; \overline{H^{q,p}_\delbar(X)} \;.$$
\end{thm}

\begin{thm}[{Lefschetz decomposition theorem, \cite[Théorème IV.5]{weil}}]
 Let $X$ be a compact complex manifold, of complex dimension $n$, endowed with a K\"ahler structure. Then there exist a decomposition
 $$ H^{\bullet}_{dR}(X;\C) \;=\; \bigoplus_{r\in\N} L^r \left(\ker\left(\Lambda\colon H^{\bullet-2r}_{dR}(X;\C)\to H^{\bullet-2r-2}_{dR}(X;\C)\right)\right)\;, $$
 and, for every $k\in\N$, an isomorphism
 $$ L^k\colon H^{n-k}_{dR}(X;\C) \stackrel{\simeq}{\to} H^{n+k}_{dR}(X;\C) \;.$$
\end{thm}

\subsection{$\partial\overline{\partial}$-Lemma and formality for compact K\"ahler manifolds}\label{subsec:kahler-deldelbar}

The Hodge decomposition theorem and the Lefschetz decomposition theorem provide obstructions to the existence of K\"ahler structures on a compact complex manifold. In this section, we study another property of compact K\"ahler manifolds, namely, formality, which provides an obstruction to the existence of a K\"ahler structure on a compact (differentiable) manifold. Such a property turns out to be a consequence of the validity of the $\del\delbar$-Lemma on compact complex manifolds.

Firstly, we need to recall some general notions regarding homotopy theory of differential algebras; we will then summarize some results concerning the homotopy type of K\"ahler manifolds: by the classical result by P. Deligne, Ph.~A. Griffiths, J. Morgan, and D.~P. Sullivan, \cite[Main Theorem]{deligne-griffiths-morgan-sullivan}, the real homotopy type of a K\"ahler manifold $X$ is a formal consequence of its cohomology ring $H^\bullet_{dR}(X;\R)$.

\medskip

We recall that a \emph{differential graded algebra} (shortly, \emph{dga}) over a field $\mathbb{K}$ is a graded $\mathbb{K}$-algebra $A^\bullet$ (where the structure of $\mathbb{K}$-algebra is induced by an inclusion $\mathbb{K}\subseteq A^0$) being graded-commutative (that is, for every $x\in A^{\deg x}$ and $y\in A^{\deg y}$, it holds $x\cdot y = \left(-1\right)^{\deg x\cdot \deg y}\, y\cdot x$) and endowed with a differential $\de\colon A^\bullet\to A^{\bullet+1}$ satisfying the graded-Leibniz rule (that is, for every $x\in A^{\deg x}$ and $y\in A^{\deg y}$, it holds $\de\left(x\cdot y\right) = \de x \cdot y + \left(-1\right)^{\deg x} \, x \cdot \de y$). A \emph{morphism of differential graded algebras} $F\colon \left(A^\bullet, \, \de_{A^\bullet}\right) \to \left(B^\bullet, \, \de_{B^\bullet}\right)$ is a morphism $A^\bullet\to B^\bullet$ of $\mathbb{K}$-algebras such that $F\circ \de_{A^\bullet}=\de_{B^\bullet} \circ F$.

Given a dga $\left(A^\bullet,\, \de\right)$ over $\mathbb{K}$, the cohomology $H^\bullet\left(A^\bullet,\, \de\right):=\frac{\ker\de}{\imm\de}$ endowed with the zero differential has a natural structure of dga over $\mathbb{K}$; furthermore, every morphism $F\colon \left(A^\bullet, \, \de_{A^\bullet}\right) \to \left(B^\bullet, \, \de_{B^\bullet}\right)$ of dgas induces a morphism $F^*\colon \left(H^\bullet\left(A^\bullet, \, \de_{A^\bullet}\right), \, 0\right) \to \left(H^\bullet\left(B^\bullet, \, \de_{B^\bullet}\right),\, 0\right)$ of dgas in cohomology; a morphism $F\colon \left(A^\bullet, \, \de_{A^\bullet}\right) \to \left(B^\bullet, \, \de_{B^\bullet}\right)$ of dgas is called a \emph{quasi-isomorphism} (shortly, \emph{qis}) if the corresponding morphism $F^*\colon \left(H^\bullet\left(A^\bullet, \, \de_{A^\bullet}\right), \, 0\right) \to \left(H^\bullet\left(B^\bullet, \, \de_{B^\bullet}\right),\, 0\right)$ is an isomorphism.

The de Rham complex $\left(\wedge^\bullet X,\, \de\right)$ of a compact (differentiable) manifold $X$ has a structure of dga over $\R$, whose cohomology is the dga $\left(H^\bullet_{dR}(X;\R),\, 0\right)$.

\medskip

Given a dga $\left(A^\bullet, \, \de_{A^\bullet}\right)$ over $\mathbb{K}$, the differential $\de_{A^\bullet}$ is called \emph{decomposable} if
$$ \de_{A^\bullet}\left(A^{\bullet}\right)\subseteq \left(\bigoplus_{k\in\N\setminus\{0\}}A^k\right) \cdot \left(\bigoplus_{k\in\N\setminus\{0\}}A^k\right) \;.$$
Given a dga $\left(A^\bullet, \, \de_{A^\bullet}\right)$ over $\mathbb{K}$, an \emph{elementary extension} of $\left(A^\bullet, \, \de_{A^\bullet}\right)$ is a dga $\left(B^\bullet, \, \de_{B^\bullet}\right)$ over $\mathbb{K}$ such that
\begin{enumerate}[\itshape (i)]
 \item $B^\bullet=A^\bullet \otimes_\K \wedge^\bullet V_k$ for $V_k$ a finite-dimensional $\mathbb{K}$-vector space and $k>0$, where $\wedge^\bullet V_k$ is the free graded $\mathbb{K}$-algebra generated by $V_k$, the elements of $V_k$ having degree $k$, and
 \item $\de_{B^\bullet}\lfloor_{A^\bullet}=\de_{A^\bullet}$ and $\de_{B^\bullet}\left(V_k\right)\subseteq A^\bullet$.
\end{enumerate}

A dga $\left(M^\bullet, \, \de_{M^\bullet}\right)$ over $\mathbb{K}$ is called \emph{minimal} if it can be written as an increasing union of sub-dga,
$$ \left(\mathbb{K},\, 0\right) \;=\; \left(M^\bullet_0, \, \de_{M^\bullet_0}\right) \;\subset\; \left(M^\bullet_1, \, \de_{M^\bullet_1}\right) \;\subset\; \left(M^\bullet_2, \, \de_{M^\bullet_2}\right) \;\subseteq\; \cdots \;, \qquad \left(M^\bullet, \, \de_{M^\bullet}\right) \;=\; \bigcup_{j\in\N} \left(M^\bullet_j, \, \de_{M^\bullet_j}\right) \;, $$
such that
\begin{enumerate}[\itshape (i)]
 \item for any $j\in\N$, the dga $\left(M^\bullet_{j+1}, \, \de_{M^\bullet_{j+1}}\right)$ is an elementary extension of the dga $\left(M^\bullet_j, \, \de_{M^\bullet_j}\right)$, and
 \item $\de_{M^\bullet}$ is decomposable.
\end{enumerate}

A \emph{minimal model} for a dga $\left(A^\bullet, \, \de_{A^\bullet}\right)$ over $\mathbb{K}$ is the datum of a minimal dga $\left(M^\bullet, \, \de_{M^\bullet}\right)$ over $\mathbb{K}$ and a quasi-isomorphism $\rho\colon \left(M^\bullet, \, \de_{M^\bullet}\right) \stackrel{\text{qis}}{\to} \left(A^\bullet, \, \de_{A^\bullet}\right)$ of dgas.

\medskip

Two dgas $\left(A^\bullet, \, \de_{A^\bullet}\right)$ and $\left(B^\bullet, \, \de_{B^\bullet}\right)$ over $\mathbb{K}$ are \emph{equivalent} if there exist an integer $n\in\N\setminus\{0\}$, a family $\left\{\left(C^\bullet_j, \, \de_{C^\bullet_j}\right)\right\}_{j\in\{0,\ldots,2n\}}$ of dgas over $\mathbb{K}$ with $\left(C^\bullet_0, \, \de_{C^\bullet_0}\right)=\left(A^\bullet, \, \de_{A^\bullet}\right)$ and $\left(C^\bullet_{2n}, \, \de_{C^\bullet_{2n}}\right)=\left(B^\bullet, \, \de_{B^\bullet}\right)$, and a family
$$ \left\{\left(C^\bullet_{2j+1}, \, \de_{C^\bullet_{2j+1}}\right) \stackrel{\text{qis}}{\to} \left(C^\bullet_{2j}, \, \de_{C^\bullet_{2j}}\right),\; \left(C^\bullet_{2j+1}, \, \de_{C^\bullet_{2j+1}}\right) \stackrel{\text{qis}}{\to} \left(C^\bullet_{2j+2}, \, \de_{C^\bullet_{2j+2}}\right)\right\}_{j\in\{0,\ldots,n-1\}} $$
of quasi-isomorphisms.
A dga $\left(A^\bullet, \, \de_{A^\bullet}\right)$ over $\mathbb{K}$ is called \emph{formal} if it is equivalent to a dga $\left(B^\bullet,\, 0\right)$ over $\mathbb{K}$ with zero differential, that is, if it is equivalent to $\left(H^\bullet\left(A^\bullet,\, \de_{A^\bullet}\right),\, 0\right)$.

A compact manifold $X$ is called \emph{formal} if its de Rham complex $\left(\wedge^\bullet X,\, \de\right)$ is a formal dga over $\R$.

\medskip

Let $\left(A^\bullet, \, \de_{A^\bullet}\right)$ be a dga over $\mathbb{K}$. Given
$$ \left[\alpha_{12}\right]\in H^{\deg \alpha_{12}}\left(A^\bullet,\, \de_{A^\bullet}\right) \;, \qquad \left[\alpha_{23}\right]\in H^{\deg \alpha_{23}}\left(A^\bullet,\, \de_{A^\bullet}\right) \;, \qquad \text{ and } \qquad \left[\alpha_{34}\right]\in H^{\deg \alpha_{34}}\left(A^\bullet,\, \de_{A^\bullet}\right) $$
such that
$$ \left[\alpha_{12}\right] \cdot \left[\alpha_{23}\right] \;=\; 0 \qquad \text{ and } \qquad \left[\alpha_{23}\right] \cdot \left[\alpha_{34}\right] \;=\;0 \;, $$
let $\alpha_{13}\in A^{\deg \alpha_{12}+\deg \alpha_{23}-1}$ and $\alpha_{24}\in A^{\deg \alpha_{23}+\deg\alpha_{34}-1}$
be such that
$$ \left(-1\right)^{\deg \alpha_{12}}\, \alpha_{12}\cdot \alpha_{23} \;=\; \de_{A\bullet} \alpha_{13} \qquad \text{ and } \qquad  \left(-1\right)^{\deg\alpha_{23}}\, \alpha_{23} \cdot \alpha_{34} \;=\; \de_{A\bullet} \alpha_{24} \;;$$
one can then define the \emph{triple Massey product} $\left\langle \left[\alpha_{12}\right],\, \left[\alpha_{23}\right],\, \left[\alpha_{34}\right] \right\rangle$ as
\begin{eqnarray*}
\lefteqn{\left\langle \left[\alpha_{12}\right],\, \left[\alpha_{23}\right],\, \left[\alpha_{34}\right] \right\rangle \;:=\; \left[ \left(-1\right)^{\deg\alpha_{12}}\, \alpha_{12}\cdot\alpha_{24}+\left(-1\right)^{\deg\alpha_{13}}\, \alpha_{13}\cdot\alpha_{34} \right]} \\[5pt]
&\in& \frac{H^{\deg \alpha_{12}+\deg \alpha_{23}+\deg \alpha_{34}-1}\left(A^\bullet,\, \de_{A^\bullet}\right)}{H^{\deg \alpha_{12}}\left(A^\bullet,\, \de_{A^\bullet}\right) \cdot H^{\deg \alpha_{23}+\deg\alpha_{34}-1}\left(A^\bullet,\, \de_{A^\bullet}\right) + H^{\deg\alpha_{34}}\left(A^\bullet,\, \de_{A^\bullet}\right) \cdot H^{\deg\alpha_{12}+\deg \alpha_{23}-1}\left(A^\bullet,\, \de_{A^\bullet}\right)} \;.
\end{eqnarray*}

One can define the higher order Massey product by induction. Fixed $m\in\N$ such that $m\geq4$, and given
$$ \left[\alpha_{12}\right]\in H^{\deg\alpha_{12}}\left(A^\bullet,\, \de_{A^\bullet}\right) \;, \qquad \dots \;, \qquad \left[\alpha_{m,m+1}\right]\in H^{\deg\alpha_{m,m+1}}\left(A^\bullet,\, \de_{A^\bullet}\right) $$
such that all the Massey products of order lower than or equal to $m-1$ vanish, let $\left\{\alpha_{rs}\right\}_{1\leq r<s\leq m+1}\subseteq A^{\bullet}$ be such that
$$ \sum_{h<\ell<k}\left(-1\right)^{\deg  \alpha_{h\ell}}\, \alpha_{h\ell}\cdot  \alpha_{\ell k} \;=\; \de \alpha_{hk} \;,$$
for any $h,k\in\{1,\ldots, m+1\}$ with $k-h<m$. Then define the \emph{\kth{m} order Massey product} as
$$ \left\langle \left[\alpha_{12}\right],\, \ldots, \left[\alpha_{m,m+1}\right] \right\rangle \;:=\; \left[ \sum_{1<\ell<m+1} \left(-1\right)^{\deg \alpha_{1k}}\, \alpha_{1k}\cdot \alpha_{k,m+1} \right] $$
belonging to a quotient of $H^\bullet\left(A^\bullet,\, \de_{A^\bullet}\right)$.

As a direct consequence of the definitions, the Massey products (of any order) on a formal dga are zero.

\medskip

Now, let $X$ be a compact manifold endowed with a K\"ahler structure.

The K\"ahler identities allow to prove the following result, known as \emph{$\partial\overline{\partial}$-Lemma} (see, e.g., \cite[Corollary 3.2.10]{huybrechts}), which, in a sense, summarizes many of the cohomological properties of compact K\"ahler manifolds.

\begin{thm}[{$\partial\overline{\partial}$-Lemma for compact K\"ahler manifolds, \cite[Lemma 5.11]{deligne-griffiths-morgan-sullivan}}]
 Let $X$ be a compact K\"ahler manifold. Then every $\del$-closed, $\delbar$-closed, $\de$-exact form is also $\del\delbar$-exact.
\end{thm}

Using the differential operator $\de^c:=J^{-1}\, \de\, J=-\im\, \left(\del-\delbar\right)$ (where $J$ is the integrable almost-complex structure naturally associated to the structure of complex manifold on $X$), and noting that $\ker \del \cap \ker\delbar=\ker\de\cap \ker \de^c$ and $\imm\del\delbar=\imm\de\de^c$, the following equivalent formulation can be provided.

\begin{thm}[{$\de\de^c$-Lemma for compact K\"ahler manifolds, \cite[Lemma 5.11]{deligne-griffiths-morgan-sullivan}}]
 Let $X$ be a compact K\"ahler manifold. Then every $\de$-closed, $\de^c$-closed, $\de$-exact form is also $\de\de^c$-exact.
\end{thm}

Actually, the $\del\delbar$-Lemma holds true for a larger class of compact complex manifolds than the compact K\"ahler manifolds: indeed, it holds, for examples, for any compact complex manifold that can be blown up to a K\"ahler manifold, \cite[Theorem 5.22]{deligne-griffiths-morgan-sullivan}, e.g., for compact complex manifolds in class $\mathcal{C}$ of Fujiki, or for Mo\v{\i}{\v{s}}ezon manifolds; we refer to \S\ref{subsec:deldelbar-lemma} for further results concerning the $\del\delbar$-Lemma for compact complex manifolds.

\medskip

If $X$ is a compact K\"ahler manifold (or, more in general, any compact complex manifold for which the $\del\delbar$-Lemma, equivalently the $\de\de^c$-Lemma, holds), then one has the following quasi-isomorphisms of dgas:
$$
\xymatrix{
 & \left(\ker \de^c,\, \de\lfloor_{\ker \de^c}\right) \ar[ld]^{\text{qis}} \ar[rd]_{\text{qis}} & \\
 \left(\wedge^\bullet X,\, \de\right) & & \left(\frac{\ker \de^c}{\imm \de^c},\, 0\right) \;;
}
$$
in particular, the dga $\left(\wedge^\bullet X,\, \de\right)$ is equivalent to a dga with zero differential, and hence it is formal. This proves the following result by P. Deligne, Ph.~A. Griffiths, J. Morgan, and D.~P. Sullivan.

\begin{thm}[{\cite[Main Theorem]{deligne-griffiths-morgan-sullivan}}]
 Let $X$ be a compact complex manifold for which the $\del\delbar$-Lemma holds (e.g., a compact K\"ahler manifold, or a manifold in class $\mathcal{C}$ of Fujiki). Then the differentiable manifold underlying $X$ is formal (that is, the differential graded algebra $\left(\wedge^\bullet X,\, \de\right)$ is formal).
\end{thm}

In particular, all Massey products (of any order) on a compact complex manifold satisfying the $\del\delbar$-Lemma are zero, \cite[Corollary 1]{deligne-griffiths-morgan-sullivan}. This provide an obstruction to the existence of K\"ahler structures on compact differentiable manifolds.

\section{Deformations of complex structures}\label{sec:deformations}

A natural way to construct new complex structures on a manifold is by ``deforming'' a given complex structure. Natural questions arise naturally from this construction, concerning, for example, what properties (e.g., the existence of some special metric) remain still valid after such a small deformation.

We recall in this section the basic notions and the classical results concerning the K. Kodaira, D.~C. Spencer, L. Nirenberg, and M. Kuranishi theory of deformations of complex manifolds, \cite{kodaira-spencer-1-2, kodaira-spencer-3, kodaira-nirenberg-spencer, kuranishi-annals}, referring to \cite{huybrechts}, see also, e.g., \cite{kodaira, kodaira-morrow}.

\medskip

Let $B$ be a complex (respectively, differentiable) manifold. A family $\left\{X_t\right\}_{t\in B}$ of compact complex manifolds is said to be a \emph{complex-analytic} (respectively, \emph{differentiable}) \emph{family of compact complex manifolds} if there exist a complex (respectively, differentiable) manifold $\mathcal{X}$ and a surjective holomorphic (respectively, smooth) map $\pi\colon \mathcal{X}\to B$ such that
\begin{inparaenum}[\itshape (i)]
 \item $\pi^{-1}(t)=X_t$ for any $t\in B$, and
 \item $\pi$ is a proper holomorphic (respectively, smooth) submersion.
\end{inparaenum}
A compact complex manifold $X$ is said to be a \emph{deformation} of a compact complex manifold $Y$ if there exist a complex-analytic family $\left\{X_t\right\}_{t\in B}$ of compact complex manifolds, and $b_0,\, b_1 \in B$ such that $X_{b_0}=X_s$ and $X_{b_1}=X_t$.

A complex-analytic (respectively, differentiable) family $\mathcal{X}\stackrel{\pi}{\to}B$ of compact complex manifolds is said to be \emph{trivial} if $\mathcal{X}$ is bi-holomorphic (respectively, diffeomorphic) to $B\times X_{b}\stackrel{\pi_B}{\to}B$ for some $b\in B$ (where $\pi_B\colon B\times X_{b}\to B$ denotes the natural projection onto $B$); it is said to be \emph{locally trivial} if, for any $b\in B$, there exists an open neighbourhood $U$ of $b$ in $B$ such that $\pi^{-1}(U) \stackrel{\pi\lfloor_{\pi^{-1}(U)}}{\to} U$ is trivial.
The following theorem by C. Ehresmann states the local triviality of a differentiable family of compact complex manifolds (see, e.g., \cite[Theorem 2.3, Theorem 2.5]{kodaira}, \cite[Theorem 1.4.1]{kodaira-morrow}).

\begin{thm}[{Ehresmann theorem, \cite{ehresmann-2}}]
 Let $\left\{X_t\right\}_{t\in B}$ be a differentiable family of compact complex manifolds. For any $s,\, t\in B$, the manifolds $X_s$ and $X_t$ are diffeomorphic.
\end{thm}

As a consequence of Ehresmann's theorem, a complex-analytic family $\left\{X_t\right\}_{t\in B}$ of compact complex manifolds with $B$ contractible can be viewed as a family of complex structures on a compact differentiable manifold.

\medskip

We recall some other useful definitions, see, e.g., \cite[\S6.2]{huybrechts}. Let $\pi\colon \mathcal{X}\to B$ be a complex-analytic family of compact complex manifolds, deformations of $X:=\pi^{-1}(0)$. We recall that, given $f\colon \left(B',\, 0'\right) \to \left(B,\, 0\right)$ a morphism of germs with a distinguished point, the pull-back $f^*\mathcal{X} := \mathcal{X} \times_B B'$ gives a complex-analytic family of deformations of $X$. The complex-analytic family $\pi\colon \mathcal{X}\to B$ of deformations of $X$ is called \emph{complete} if, for any complex-analytic family $\pi'\colon \mathcal{X}'\to B'$ of deformations of $X$, there exists a morphism $f\colon B'\to B$ of germs with a distinguished point such that $\mathcal{X'} = f^*\mathcal{X}$. The complex-analytic family $\pi\colon \mathcal{X}\to B$ of deformations of $X$ is called \emph{universal} if, for any complex-analytic family $\pi'\colon \mathcal{X}'\to B'$ of deformations of $X$, there exists a unique morphism $f\colon B'\to B$ of germs 
with a distinguished point such that $\mathcal{X'} = f^*\mathcal{X}$.
The complex-analytic family $\pi\colon \mathcal{X}\to B$ of deformations of $X$ is called \emph{versal} if, for any complex-analytic family $\pi'\colon \mathcal{X}'\to B'$ of deformations of $X$, there exists a morphism $f\colon B'\to B$ of germs with a distinguished point such that $\mathcal{X'} = f^*\mathcal{X}$ and such that $\de f\colon T_{0'}B'\to T_0S$ is uniquely determined.

\medskip

The theory of complex-analytic deformations of compact complex manifolds has been introduced by K. Kodaira and D.~C. Spencer, \cite{kodaira-spencer-1-2, kodaira-spencer-3}, and developed also by L. Nirenberg, \cite{kodaira-nirenberg-spencer}, and M. Kuranishi, \cite{kuranishi-annals, kuranishi-proceedings}, see also \cite{kodaira, kodaira-morrow}.
In recalling the main results of this theory, we follow the approach in \cite{huybrechts}, based on the construction of a differential graded Lie algebra structure on $\mathcal{C}^\infty\left(X;\, T^{1,0}X\otimes \wedge^{0,\bullet}X\right)$, see also \cite{manetti-lectures}.

Let $X$ be a compact manifold endowed with an integrable almost-complex structure $J$. Every section $s\in \mathcal{C}^\infty\left(X; T^{1,0}_JX\otimes \wedge^{0,1}_JX\right)$
near to the zero section determines an almost-complex structure $J'$, defined in such a way that $\wedge^{1,0}_{J'}X$ is the graph of $-s\colon \wedge^{1,0}_JX\to\wedge^{0,1}_JX$; it turns out that $J'$ is integrable if and only if the \emph{Maurer and Cartan equation}
\begin{equation}\tag{\text{MC}}\label{eq:MC}
\delbar s+\frac{1}{2}\, \left[s,s\right] \;=\; 0
\end{equation}
holds (see, e.g., \cite[Lemma 6.1.2]{huybrechts}), where
\begin{itemize}
 \item $\left[\sspace, \ssspace\right]\colon \mathcal{C}^\infty\left(X;\, T^{1,0}_JX\otimes\wedge^{0,p}_JX\right) \times \mathcal{C}^\infty\left(X;\, T^{1,0}_JX\otimes\wedge^{0,q}_JX\right) \to \mathcal{C}^\infty\left(X;\, T^{1,0}_JX\otimes\wedge^{0,p+q}_JX\right)$ is defined as
$$ \left[X\otimes \bar\alpha,\, Y\otimes\bar\beta\right] \;:=\; X \otimes \left(\bar \beta \wedge \Lie{Y}\bar\alpha\right) + Y \otimes  \left(\bar\alpha\wedge\Lie{X}\bar\beta\right) + \left[X,\, Y\right] \otimes \left(\bar\alpha\wedge\bar\beta\right) \;,$$
where $\Lie{W}\varphi:=\iota_W\de\varphi+\de\left(\iota_W\varphi\right)$ is the Lie derivative of $\varphi$ along $W$;
locally, in a chart with holomorphic coordinates $\left\{z^j\right\}_j$, one has
$$ \left[w\otimes \de\bar z^{\ell_1}\wedge\cdots\wedge \de\bar z^{\ell_p},\, w'\otimes \de\bar z^{m_1}\wedge\cdots\wedge \de\bar z^{m_q}\wedge\right] \;\stackrel{\text{loc}}{=}\; \left[w, \, w'\right] \otimes \de \bar z^{\ell_1}\wedge\cdots\wedge \de\bar z^{\ell_p}\wedge \de\bar z^{m_1}\wedge\cdots\wedge \de\bar z^{m_q} \;;$$
 \item $\delbar \colon \mathcal{C}^\infty\left(X;\, T^{1,0}_JX\otimes\wedge^{0,p}_JX\right) \to \mathcal{C}^\infty\left(X;\, T^{1,0}_JX \otimes \wedge^{0,p+1}_JX \right)$ is defined as
$$ \delbar\varphi\left(\bar Z, \, \bar W\right) \;:=\; \left[\bar Z, \, \varphi\left(\bar W\right)\right]^{1,0}-\left[\bar W, \, \varphi\left(\bar Z\right)\right]^{1,0}-\varphi\left(\left[\bar Z, \, \bar W\right]\right) \;,$$
where $X^{1,0}:=X-\im \, J\, X$ is the $(1,0)$-component of $X$; locally, in a chart with holomorphic coordinates $\left\{z^j\right\}_j$, one has
$$ \delbar \left(\frac{\del}{\del z^\ell}\otimes\alpha\right) \;\stackrel{\text{loc}}{=}\; \frac{\del}{\del z^\ell}\otimes\delbar\alpha \;.$$
\end{itemize}

Hence, to study complex-analytic families of infinitesimal deformations of a compact complex manifold $X$, it suffices to study complex-analytic families $\left\{s(\tempo)\right\}_{\tempo\in\Delta(0,\varepsilon)\subset \C^m}\subseteq \mathcal{C}^\infty\left(X;\; T^{1,0}X\otimes\wedge^{0,1}X\right)$ (where $\varepsilon>0$ is small enough) with $s(0)=0$. Consider the power series expansion in $\tempo$ of $s(\tempo)$,
 $$ s(\tempo) \;=:\; \sum_{k\in\N}s_k(\tempo) \;,$$
where $s_k(\tempo)\in\mathcal{C}^\infty\left(X;\, T^{1,0}X\otimes\wedge^{0,1}X\right)$ is homogeneous of degree $k$ in $\tempo$, and $s_0(\tempo)=0$. Then the Maurer and Cartan equation \eqref{eq:MC} can be rewritten, for every $\tempo\in\Delta(0,\varepsilon)$, as the system
$$
\left\{
\begin{array}{lclcl}
 \delbar s_1(\tempo) &=& 0 && \\[5pt]
 \delbar s_k(\tempo) &=& -\sum_{1\leq j \leq k-1}\left[s_j(\tempo),\, s_{k-j}(\tempo)\right] & \text{ for } & k\geq2
\end{array}
\right. \;;
$$
in particular, $s_1(\tempo)$ defines a class in $H^{0,1}\left(X;\, \Theta_X\right)$, where $\Theta_X$ denotes the sheaf of the germs of holomorphic vector fields on $X$; up to the action of $\mathrm{Diff}(X)$, one has that $s_1(\tempo)$ is uniquely determined by its class in $H^{0,1}\left(X;\; \Theta_X\right)$ (see, e.g., \cite[Lemma 6.14]{huybrechts}).

Fix now a Hermitian metric $g$ on $X$. Consider the decomposition
$$ T^{1,0}X\otimes \wedge^{0,1}X \;=\; \left(T^{1,0}X\otimes \ker\overline\square\lfloor_{\wedge^{0,1}X}\right) \,\oplus\, \left(T^{1,0}X\otimes\delbar\wedge^{0,0}X\right) \,\oplus\, \left(T^{1,0}X\otimes\delbar^*\wedge^{0,2}X\right) \;,$$
and the corresponding projections
$$ H_\delbar\colon T^{1,0}X\otimes \wedge^{0,1}X \to T^{1,0}X\otimes \ker\overline\square\lfloor_{\wedge^{0,1}X} \;, \qquad P_\delbar\colon T^{1,0}X\otimes \wedge^{0,1}X \to T^{1,0}X\otimes \delbar\wedge^{0,0}X \;.$$
In order that $s(\tempo)$ satisfies \eqref{eq:MC}, for every $\tempo\in\Delta(0,\varepsilon)$, one should have
$$ \delbar s_k(\tempo) \;=\; -P_\delbar \left(\sum_{1\leq j\leq k-1}\left[s_j(\tempo),\, s_{k-j}(\tempo)\right]\right) \;.$$
Hence, one gets
$$ \delbar s(\tempo) + \left[s(\tempo),\, s(\tempo)\right] \;=\; H_\delbar\left(\left[s(\tempo),\, s(\tempo)\right]\right) \;.$$

Therefore, define the map
$$ \mathrm{obs}\colon H^{0,1}\left(X;\, \Theta_X\right) \to H^{0,2}\left(X;\, \Theta_X\right) $$
as follows. Let $\left\{X_j\otimes \bar\omega^k\right\}_{\substack{j\in\{1,\ldots,n\}\\k\in\{1,\ldots,m\}}}$ be a basis of $H^{0,1}\left(X;\, \Theta_X\right)$. Given $\mu:=:\sum_{\substack{j\in\{1,\ldots,n\}\\k\in\{1,\ldots,m\}}} t^j_k \, X_j\otimes\bar\omega^k$, denote $\tempo:=:\left(t^j_k\right)_{\substack{j\in\{1,\ldots,n\}\\k\in\{1,\ldots,m\}}}$, and define $s_1(\tempo):=\mu$ and $s_k(\tempo)$ such that $\delbar s_k(\tempo):=-P_\delbar\left(\sum_{1\leq j\leq k-1}\left[s_j(\tempo),\, s_{k-j}(\tempo)\right]\right)$ for $k\geq 2$; hence, define the formal power series $s(\tempo):=\sum_{k\in\N}s_k(\tempo)$. Define
$$ \mathrm{obs}\left(\mu\right) \;:=\; H_\delbar\left(\left[s(\tempo),\, s(\tempo)\right]\right) \;.$$

Hence, one has then that $\left\{s(\tempo)\right\}_{\tempo\in\Delta(0,\varepsilon)\subset\C^m}\subseteq \mathcal{C}^\infty\left(X; T^{1,0}_JX\otimes \wedge^{0,1}_JX\right)$ (where $\varepsilon>0$ is small enough) defines an infinitesimal family of compact complex manifolds if $\mathrm{obs}\left(s_1\left(\tempo\right)\right)=0$ for every $\tempo\in\Delta(0,\varepsilon)$ (indeed, for $\varepsilon>0$ small enough, the formal power series converges, see, e.g., \cite[\S5.3]{kodaira}, \cite[\S2.3]{kodaira-morrow}).

One gets the following result by M. Kuranishi.

\begin{thm}[{\cite[Theorem 2]{kuranishi-annals}}]
 Let $X$ be a compact complex manifold. Then $X$ admits a versal complex-analytic family of deformations.
\end{thm}

Fixed a Hermitian metric on $X$, such a family of deformations, which is called the \emph{Kuranishi space} $\mathrm{Kur}(X)$ of $X$, is parametrized by
 $$ \mathrm{Kur}(X) \;=\; \left\{\mu \in H^{0,1}\left(X;\, \Theta_X\right) \st \left\|\mu\right\| \ll 1,\; \mathrm{obs}(\mu)=0 \right\} \;.$$

\begin{rem}
 A compact complex manifold $X$ is called \emph{non-obstructed} if $\mathrm{Kur}(X)$ is non-singular.
 In particular, if $H^{0,2}\left(X;\, \Theta_X\right)=\{0\}$, then $X$ is non-obstructed.
 There are other interesting cases in which the Kuranishi space turns out to be non-singular: as announced by F.~A. Bogomolov, \cite{bogomolov-cy}, and proven by G. Tian, \cite{tian-cy}, and, independently, by A.~N. Todorov, \cite[Theorem 1]{todorov-cy}, this happens for \emph{Calabi-Yau manifolds} (that is, compact complex manifolds $X$ of complex dimension $n$ endowed with a K\"ahler structure $\left(J,\, \omega,\, g\right)$ and with a nowhere vanishing $\epsilon\in\wedge^{n,0}X$ such that
 \begin{inparaenum}[\itshape (i)]
   \item $\nabla^{LC}\epsilon=0$,where $\nabla^{LC}$ denotes the Levi Civita connection associated to $g$, and
   \item $\epsilon\wedge\bar\epsilon=\left(-1\right)^{\frac{n(n+1)}{2}}\,\im^n\,\frac{\omega^n}{n!}$
 \end{inparaenum}).
 In \cite{debartolomeis-tomassini-2}, P. de Bartolomeis and A. Tomassini introduced the notion of \emph{quantum inner state manifold}, \cite[Definition 2.2]{debartolomeis-tomassini-2}, as a possible generalization of Calabi-Yau manifolds, proving that, under a suitable hypothesis, the moduli space of quantum inner state deformations of a compact Calabi-Yau manifold is totally unobstructed, \cite[Theorem 3.6]{debartolomeis-tomassini-2}.
 On the other hand, in \cite{rollenske-jems}, S. Rollenske studied the Kuranishi space of holomorphically parallelizable nilmanifolds, proving that it is cut out by polynomial equations of degree at most equal to the step of nilpotency of the nilmanifold, \cite[Theorem 4.5]{rollenske-jems}, and it is smooth if and only if the associated Lie algebra is a free $2$-step nilpotent Lie algebra, \cite[Corollary 4.9]{rollenske-jems}.
\end{rem}

\medskip

It could be interesting to study what properties are, in a sense, compatible with the construction of small deformations of the complex structure. In such a context, a property $\mathcal{P}$ concerning compact complex manifolds is called \emph{open under (holomorphic) deformations of the complex structure} (or \emph{stable under small deformations of the complex structure}) if, for every complex-analytic family $\left\{X_t\right\}_{t\in B}$ of compact complex manifolds, and for every $b_0\in B$, if $X_{b_0}$ has the property $\mathcal{P}$, then $X_b$ has the property $\mathcal{P}$ for every $b$ in an open neighbourhood of $b_0$; it is called \emph{closed under (holomorphic) deformations of the complex structure} if, for every complex-analytic family $\left\{X_t\right\}_{t\in B}$ of compact complex manifolds, and for every converging sequence $\left\{b_k\right\}_{k\in\N} \subset B$ with $b_\infty:=\lim_{k\to+\infty}b_k\in B$, if $X_{b_k}$ has the property $\mathcal{P}$ for every $k\in\N$, then $X_{b_\infty}$ 
has the property $\mathcal{P}$.

We recall here the following classical result by K. Kodaira and D.~C. Spencer, stating that admitting a K\"ahler metric is a stable property under deformations of the complex structure.

\begin{thm}[{\cite[Theorem 15]{kodaira-spencer-3}}]
 Let $\left\{X_t\right\}_{t\in B}$ be a differentiable family of compact complex manifolds. If $X_{t}$ admits a K\"ahler metric for some $t\in B$, then $X_s$ admits a K\"ahler metric for every $s$ in an open neighbourhood of $t$ in $B$. Moreover, given any K\"ahler metric $\omega$ on $X_t$, one can choose an open neighbourhood $U$ of $t$ in $B$ and a K\"ahler metric $\omega_s$ on $X_s$ for any $s\in U$ such that $\omega_s$ depends differentiably in $s$ and $\omega_t=\omega$.
\end{thm}

\begin{rem}
 In \cite{hironaka}, it is proven that admitting a K\"ahler structure is not a closed property under deformations of the complex structure: in fact, H. Hironaka provided an explicit example of a complex-analytic family of compact complex manifolds of complex dimension $3$ such that 
  \begin{inparaenum}[\itshape (i)]
    \item one of the complex manifold is non-K\"ahler (indeed, it carries a positive $1$-cycle algebraically equivalent to zero), and
    \item the others are K\"ahler and, in fact, bi-regularly embedded in a projective space (and hence projective, \cite[Theorem 11]{moishezon-transl}),
  \end{inparaenum}
\cite[Theorem]{hironaka}.
(Note that, in complex dimension $2$, the K\"ahler property is also closed under small deformations of the complex structure, since a compact complex surface is K\"ahler if and only if its \kth{1} Betti number is even, by \cite{kodaira-structure-I, miyaoka, siu}, or \cite[Corollaire 5.7]{lamari}, or \cite[Theorem 11]{buchdahl}.)
It is not known whether the limit of compact K\"ahler manifolds admits some special structure;
J.-P. Demailly and M. P\v{a}un conjectured that, given a complex-analytic family $\left\{X_t\right\}_{t\in S}$ of compact complex manifolds such that one of the fibers, $X_{t_0}$, is endowed with a K\"ahler structure, then there exists a countable union $S'\subsetneq S$ of analytic subsets in the base such that $X_t$ admits a K\"ahler structure for $t\in S\setminus S'$, \cite[Conjecture 5.1]{demailly-paun}; they also guessed that a ``natural expectation'' is that the remaining fibres, $X_t$ for $t\in S'$, are in class $\mathcal{C}$ of Fujiki, \cite[page 1272]{demailly-paun}.
In \cite{popovici-proj, popovici-moish}, D. Popovici studied limits of projective, respectively Mo\v{\i}\v{s}ezon manifolds under holomorphic deformations of complex structures, stating, in particular, (by means of a class of Hermitian metrics called strongly-Gauduchon metrics,) that the limit of Mo\v{\i}\v{s}ezon manifolds is still Mo\v{\i}\v{s}ezon. C. LeBrun and Y.~S. Poon \cite{lebrun-poon}, and F. Campana \cite{campana} showed that being in class $\mathcal{C}$ of Fujiki is not a stable property under small deformations of the complex structures, \cite[Theorem 1]{lebrun-poon}, \cite[Corollary 3.13]{campana}, studying twistor spaces. It is conjectured that being in class $\mathcal{C}$ of Fujiki is a closed property under deformations of the complex structure, see, e.g., \cite[Standard Conjecture 1.17]{popovici-openness}.
\end{rem}

We refer to \cite{popovici-openness} for a review on the behaviour under holomorphic deformations of properties concerning, e.g., the existence of various types of Hermitian metrics on compact complex manifolds. See also Corollary \ref{cor:stab-del-delbar-lemma}, Theorem \ref{thm:instability} for some results concerning stability or instability of special properties of complex manifolds, and Theorem \ref{thm:instability-almK}, Theorem \ref{thm:para-kahler-deformations} for other instability results for almost-complex or \para-complex manifolds.

\section{Currents and de Rham homology}\label{sec:currents}
In this section, we recall the basic notions and results concerning currents on (differentiable) manifolds and de Rham homology: they turn out to be a useful tool to study the geometry of complex manifolds (as an example, we recall F.~R. Harvey and H.~B. Lawson's intrinsic characterization of K\"ahler manifolds by means of currents, \cite[Proposition 12, Theorem 14]{harvey-lawson}, or M.~L. Michelsohn's intrinsic characterization of balanced manifolds by means of currents \cite[Theorem 4.7]{michelsohn}, see also Theorem \ref{thm:caratterizzazione-intrinseca}, or J.~P. Demailly and M. P\v{a}un's characterization of compact complex manifolds in class $\mathcal{C}$ of Fujiki by means of K\"ahler currents \cite[Theorem 3.4]{demailly-paun}). We refer, e.g., to \cite[Chapter 3]{derham}, \cite[\S I.2]{demailly-agbook}, and \cite{federer} (see also \cite{alessandrini-correnti-1, alessandrini-correnti-2}) for further details.

\medskip

Let $X$ be a $m$-dimensional oriented differentiable manifold.

For every compact set $L\subseteq X$ and for every $s\in\N$, define the semi-norm $\rho^s_L$ on $\wedge^\bullet X$ as follows:
chosen $\left(U,\, \left\{x^j\right\}_{j\in\{1,\ldots,m\}}\right)$ a coordinate chart with $U\supset L$, and given
$$ \varphi \;\stackrel{\text{loc}}{:=:}\; \sum_{\substack{\left\{i_1,\ldots,i_k\right\}\subseteq\left\{1,\ldots,m\right\}\\i_1<\cdots<i_k}} \varphi_I\,\de x^{i_1}\wedge\cdots\wedge\de x^{i_k} \;\in\; \wedge^\bullet X\;, $$
set
$$ \rho^s_L(\varphi) \;:=\; \sup_L \, \sup_{\substack{\left\{i_1,\ldots,i_k\right\}\subseteq\left\{1,\ldots,m\right\}\\i_1<\cdots<i_k}} \, \sup_{\substack{\left(\alpha_1,\ldots,\alpha_m\right)\in\N^m\\\alpha_1+\cdots+\alpha_m\leq s}} \, \left|\frac{\del^{\alpha_1+\cdots+\alpha_m}\varphi_I}{\left(\del x^{1}\right)^{\alpha_1}\cdots \left(\del x^{m}\right)^{\alpha_m}}\right| \;\in\; \R \;. $$
Consider $\wedge^{\bullet}X$ endowed with the topology induced by the family of semi-norms $\rho_{L}^s$, varying $L$ among the compact sets in $X$, and $s\in\N$: the manifold $X$ being second-countable, $\wedge^\bullet X$ has a structure of a Fréchet space.
Let $\wedge^\bullet_{\text{c}}X$ be the topological subspace of $\wedge^\bullet X$ consisting of differential forms with compact support in $X$.

\medskip

For any $k\in\N$, the space of \emph{currents} of dimension $k$ (or degree $m-k$), denoted by
$$ \correnti_kX\;:=:\;\correnti^{m-k}X \;,$$
is defined as the topological dual space of $\wedge^k_{\text{c}}X$; the space $\correnti_\bullet X$ is endowed with the weak-$*$ topology.

Two basic examples of currents are the following.
\begin{itemize}
 \item If $Z$ is a (possibly non-closed) $k$-dimensional oriented compact submanifold of $X$, then
  $$ [Z] \;:=\; \int_Z \sspace \;\in\; \correnti_k X $$
  is a current of dimension $k$.
 \item If $\varphi\in\wedge^kX$, then
  $$ T_\varphi \;:=\; \int_X \varphi \wedge \sspace \;\in\; \correnti^kX $$
  is a current of degree $k$.
\end{itemize}

\medskip

The exterior differential $\de\colon\wedge^\bullet X\to \wedge^{\bullet+1}X$ induces a differential on $\correnti_\bullet X$ by duality:
$$ \de\colon \correnti_{\bullet}X\to\correnti_{\bullet-1}X $$
is defined, for every $T\in\correnti^kX$, as
$$ \de T \;:=\; \left(-1\right)^{k+1} T\left(\de\sspace\right) \;.$$

In particular, if $Z$ is a $k$-dimensional oriented closed submanifold of $X$, then $\de\left[Z\right]=\left(-1\right)^{m-k+1}\left[\bordo Z\right]$, where $\bordo$ is the boundary operator; if $\varphi\in\wedge^kX$, then $\de T_\varphi=T_{\de\varphi}$.

\medskip

By definition, the \emph{de Rham homology} $H^{dR}_\bullet(X;\R)$ of $X$ is the homology of the differential complex $\left(\correnti_\bullet X,\,\de\right)$.
By means of a regularization process, \cite[Theorem 12]{derham}, (see also \cite[\S 2.D.3, \S 2.D.4]{demailly-agbook},) one can prove, \cite[Theorem 14]{derham}, that
$$ H^\bullet_{dR}(X;\R) \;\simeq\; H^{dR}_{2n-\bullet}(X;\R) \;.$$

Since, for every $k\in\N$, the sheaf $\mathcal{A}^k_X$ is a sheaf of $\mathcal{C}^\infty_X$-module over a paracompact space (where $\mathcal{C}^\infty_X$ denotes the sheaf of germs of smooth functions over $X$), and by the Poincaré lemma for forms, see, e.g., \cite[I.1.22]{demailly-agbook}, one has that
$$ 0 \to \underline{\R}_{X} \to \left(\mathcal{A}^\bullet_{X},\, \de\right) $$
is a fine (and hence acyclic, see, e.g., \cite[Corollary IV.4.19]{demailly-agbook}) resolution of the constant sheaf $\underline{\R}_{X}$, and hence
$$ \check H^\bullet\left(X;\underline{\R}_X\right) \;\simeq\; \frac{\ker\left(\de\colon \wedge^\bullet X \to \wedge^{\bullet+1}X\right)}{\imm\left(\de\colon \wedge^{\bullet-1} X \to \wedge^{\bullet}X\right)} \;=:\; H^\bullet_{dR}\left(X;\R\right) \;, $$
see, e.g., \cite[IV.6.4]{demailly-agbook}.

Analogously, the regularization process \cite[Theorem 12]{derham} allows to prove the analogue of the Poincaré lemma for currents, see, e.g., \cite[Theorem I.2.24]{demailly-agbook}, and hence, the sheaf $\mathcal{D}^k_X$ being fine for every $k\in\N$ since it is a sheaf of $\mathcal{C}^\infty_X$-module over a paracompact space, one has that
$$ 0 \to \underline{\R}_{X} \to \left(\mathcal{D}^\bullet_{X},\, \de\right) $$
is a fine (and hence acyclic, see, e.g., \cite[Corollary IV.4.19]{demailly-agbook}) resolution of the constant sheaf $\underline{\R}_{X}$ over $X$, and hence
$$ \check H^\bullet\left(X;\underline{\R}_X\right) \;\simeq\; \frac{\ker\left(\de\colon \correnti^\bullet X \to \correnti^{\bullet+1}X\right)}{\imm\left(\de\colon \correnti^{\bullet-1} X \to \correnti^{\bullet}X\right)} \;=:\; H_{2n-\bullet}^{dR}\left(X;\R\right) \;, $$
see, e.g., \cite[IV.6.4]{demailly-agbook}.

If $X$ is compact, then it follows that the map $T_\sspace \colon \wedge^\bullet X \to \correnti^\bullet X$ is injective and a quasi-isomorphism of differential complexes: indeed, fixed a Riemannian metric $g$ on $X$, if $\alpha$ is a $\Delta$-harmonic form (i.e., a $\de$-closed $\de^*$-closed form), then $T_\alpha(*\alpha)=\left\|\alpha\right\|^2$.

\medskip

Suppose now that $X$ is a $2n$-dimensional manifold endowed with an almost-complex structure $J\in\End(TX)$. Considering the induced endomorphisms $J\in\End\left(\wedge^\bullet X\right)$ and $J\in\End\left(\wedge^\bullet_{\text{c}} X\right)$, one can define $J\in\End\left(\correnti^\bullet X\right)$ by duality. In the same way as $J\in\End\left(\wedge^\bullet X\right)$ defines a bi-graduation on $\wedge^\bullet X\otimes\C$, one has that $J\in\End\left(\correnti^\bullet X\right)$ defines the splitting
$$ \correnti_\bullet X \otimes \C \;=\; \bigoplus_{p,q\in\N} \correnti_{p,q}X \;;$$
note that $\correnti_{p,q}X:=:\correnti^{n-p,n-q}X$ is the topological dual of $\wedge^{p,q}X\cap \left(\wedge^\bullet_{\text{c}}X\otimes_\R \C\right)$, for every $p,q\in\N$.

\section{Solvmanifolds}\label{sec:solvmanifolds}
Nilmanifolds and solvmanifolds provide an important class of examples in non-K\"ahler geometry. Indeed, on the one hand, in studying their properties, one often can reduce to study left-invariant objects on them, which is the same to study linear objects on the corresponding Lie algebra (this allows, for example, to reduce the study of the de Rham cohomology of a nilmanifold to the study of the cohomology of a complex of finite-dimensional vector spaces, \cite[Theorem 1]{nomizu}); on the other hand, they do not admit too strong structures, e.g., they do not admit any K\"ahler structure. 

In this section, we recall the main definitions and results concerning the theory of nilmanifolds and solvmanifolds, setting also the notation for the following chapters.

\medskip

A \emph{nilmanifold}, respectively \emph{solvmanifold}, $X=\left.\Gamma\right\backslash G$ is a compact quotient of a connected simply-connected nilpotent, respectively solvable, Lie group $G$ by a co-compact discrete subgroup $\Gamma$. A solvmanifold $X=\left.\Gamma\right\backslash G$ is called \emph{completely-solvable} if, for any $g\in G$, all the eigenvalues of $\mathrm{Ad} g\in\mathrm{End}(\mathfrak{g})$ are real, equivalently, if, for any $X\in\mathfrak{g}$, all the eigenvalues of $\mathrm{ad} X\in\mathrm{End}(\mathfrak{g})$ are real.

\medskip

Given a $2n$-dimensional solvmanifold $X=\left.\Gamma\right\backslash G$, consider $\left(\mathfrak{g},\left[\sspace,\ssspace\right]\right)$ the Lie algebra naturally associated to the Lie group $G$; given a basis $\left\{e_1,\ldots,e_{2n}\right\}$ of $\mathfrak{g}$, the Lie algebra structure of $\mathfrak{g}$ is characterized by the \emph{structure constants} $\left\{c_{\ell m}^k\right\}_{\ell,m,k\in\{1,\ldots,2n\}}\subset\R$: namely, for any $k\in\left\{1,\ldots,2n\right\}$,
$$ \de_{\mathfrak{g}} e^k =: \sum_{\ell,m} c_{\ell m}^k \, e^\ell\wedge e^m \;, $$
where $\left\{e^1,\ldots,e^{2n}\right\}$ is the dual basis of $\duale{\mathfrak{g}}$ of $\left\{e_1,\ldots,e_{2n}\right\}$ and $\de_{\mathfrak{g}}\colon\duale{\mathfrak{g}}\to \wedge^2\duale{\mathfrak{g}}$ is defined by
$$ \duale{\mathfrak{g}} \; \ni \; \alpha \mapsto \de_{\mathfrak{g}}\alpha(\sspace,\,\ssspace):=-\alpha\left(\left[\sspace,\,\ssspace\right]\right) \;\in\; \wedge^2\duale{\mathfrak{g}} \;.$$

To shorten the notation, as in \cite{salamon}, we will refer to a given solvmanifold $X=\left.\Gamma\right\backslash G$ writing the structure equations of its Lie algebra: for example, writing
$$ X \;:=\; \left(0^4,\; 12,\; 13\right)\;, \qquad \text{ (or } \mathfrak{g} \;:=\; \left(0^4,\; 12,\; 13\right)\;, \text{)} $$
we mean that $X=\left.\Gamma\right\backslash G$ and there exists a basis of the Lie algebra $\mathfrak{g}$ naturally associated to $G$, let us say $\left\{e_1,\,\ldots,\,e_6\right\}$, whose dual will be denoted by $\left\{e^1,\,\ldots,\,e^6\right\}$, such that the structure equations with respect to such basis are
$$
\left\{
\begin{array}{l}
 \de e^1 \;=\; \de e^2 \;=\; \de e^3 \;=\; \de e^4 \;=\; 0 \\[5pt]
 \de e^5 \;=\; e^{1}\wedge e^{2} \;=:\; e^{12} \\[5pt]
 \de e^6 \;=\; e^{1}\wedge e^{3} \;=:\; e^{13}
\end{array}
\right. \;,
$$
where we also shorten $e^{AB}:=e^A\wedge e^B$.

\medskip

The following theorem by A.~I. Mal'tsev characterizes the nilpotent Lie algebras $\mathfrak{g}$ for which the naturally associated connected simply-connected Lie group admits a co-compact discrete subgroup, and hence such that there exists a nilmanifold with $\mathfrak{g}$ as Lie algebra.

\begin{thm}[{\cite[Theorem 7]{malcev}}]
 In order that a simply-connected connected nilpotent Lie group contain a discrete co-compact Lie group it is necessary and sufficient that the Lie algebra of this group have rational constant structures with respect to an appropriate basis.
\end{thm}

\medskip

Dealing with \emph{$G$-left-invariant} objects on $X$, we mean objects induced by objects on $G$ being invariant under the left-action of $G$ on itself given by left-translations. By means of left-translations, $G$-left-invariant objects will be identified with objects on the Lie algebra $\mathfrak{g}$.

For example, a $G$-left-invariant complex structure $J\in\End(TX)$ on $X$ is uniquely determined by a linear complex structure $J\in\End(\mathfrak{g})$ on $\mathfrak{g}$ satisfying the integrability condition $\Nij_J=0$, \cite[Theorem 1.1]{newlander-nirenberg}, where
$$ \Nij_J(\sspace,\, \ssspace) \;:=\; \left[\sspace,\, \ssspace\right] + J\left[J\sspace,\, \ssspace\right] + J\left[\sspace,\, J\ssspace\right] - \left[J\sspace,\, J\ssspace\right] \;\in\; \wedge^2\duale{\mathfrak{g}}\otimes_\R \mathfrak{g} \;;$$
we will denote the set of $G$-left-invariant complex structures on $X$ by
$$ \mathcal{C}\left(\mathfrak{g}\right) := \left\{ J\in\End\left(\mathfrak{g}\right) \st J^2=-\id_{\mathfrak{g}} \;\text{ and }\;\Nij_J=0 \right\} \;. $$

\medskip

By the Leibniz rule, the map $\de_\mathfrak{g}\colon\wedge^1\duale{\mathfrak{g}}\to\wedge^2\duale{\mathfrak{g}}$ induces a differential operator $\de\colon \wedge^{\bullet}\duale{\mathfrak{g}} \to \wedge^{\bullet+1}\duale{\mathfrak{g}}$ giving a graded differential algebra $\left(\wedge^\bullet\duale{\mathfrak{g}},\, \de\right)$, and hence a differential complex $\left(\wedge^\bullet\duale{\mathfrak{g}},\,\de\right)$; we will denote by $H_{dR}^\bullet\left(\mathfrak{g};\R\right):=H^\bullet\left(\wedge^\bullet\duale{\mathfrak{g}},\, \de\right)$ the cohomology of such a differential complex.

In general, on a solvmanifold, the inclusion $\left(\wedge^\bullet\duale{\mathfrak{g}},\,\de\right) \hookrightarrow \left(\wedge^\bullet X,\, \de\right)$ induces an injective map in cohomology, $i\colon H_{dR}^{\bullet}\left(\mathfrak{g};\R\right)\hookrightarrow H^{\bullet}_{dR}(X;\R)$ (compare \cite[Lemma 9]{console-fino} and Lemma \ref{lemma:inj}, for the Dolbeault, respectively Bott-Chern, cohomology), which is not always an isomorphism, as the example in \cite[Corollary 4.2, Remark 4.3]{debartolomeis-tomassini} shows.
On the other hand, the following theorem by K. Nomizu says that the de Rham cohomology of a nilmanifold can be computed as the cohomology of the subcomplex of left-invariant forms (some results in this direction have been provided also by Y. Matsushima in \cite[Theorem 5, Theorem 6]{matsushima}).

\begin{thm}[{\cite[Theorem 1]{nomizu}}]
 Let $X=\left. \Gamma \right\backslash G$ be a nilmanifold and denote the Lie algebra naturally associated to $G$ by $\mathfrak{g}$. The complex $\left(\wedge^{\bullet}\mathfrak{g}^*,\de\right)$ is a minimal model for $\left(\wedge^\bullet X,\, \de\right)$. In particular, the map $\left(\wedge^\bullet\mathfrak{g}^*,\,\de\right)\to\left(\wedge^\bullet X,\,\de\right)$ of differential complexes is a quasi-isomorphism:
$$ i\colon H^{\bullet}_{dR}(\mathfrak{g};\R) \stackrel{\simeq}{\to} H^\bullet_{dR}(X;\R) \;.$$
\end{thm}

The proof rests on an inductive argument, which can be performed since every nilmanifold can be seen as a principal torus-bundle over a lower dimensional nilmanifold, see \cite[Lemma 4]{malcev}, \cite[Theorem 3]{matsushima}.

A similar result holds also in the case of completely-solvable solvmanifolds, as proven by A. Hattori, as a consequence of the Mostow structure theorem, \cite[Theorem 2]{mostow, mostow-errata}.

\begin{thm}[{\cite[Corollary 4.2]{hattori}}]
 Let $X=\left. \Gamma \right\backslash G$ be a completely-solvable solvmanifold and denote the Lie algebra naturally associated to $G$ by $\mathfrak{g}$. The map $\left(\wedge^\bullet\mathfrak{g}^*,\,\de\right)\to\left(\wedge^\bullet X,\,\de\right)$ of differential complexes is a quasi-isomorphism:
$$ i\colon H^{\bullet}_{dR}(\mathfrak{g};\R) \stackrel{\simeq}{\to} H^\bullet_{dR}(X;\R) \;.$$
\end{thm}

(For some results concerning the de Rham cohomology of (non-necessarily completely-solvable) solvmanifolds, see \cite{guan, console-fino--ann-sns-2011}.)

\medskip

In some cases, we will see that the study of (properties of) geometric structures on a solvmanifold is reduced to the study of the corresponding (properties of) geometric structures on the associated Lie algebra (see, e.g., Theorem \ref{thm:nilmfd-are-not-tamed}, Proposition \ref{prop:linear-cpf-invariant-cpf-J}, Proposition \ref{prop:linear-cpf-invariant-cpf-omega}, Proposition \ref{prop:linear-cpf-invariant-cpf-K}, Theorem \ref{thm:instability-almK}). To this aim, we need the following lemma by J. Milnor. (Recall that a Lie group $G$, with associated Lie algebra $\mathfrak{g}$, is called \emph{unimodular} if, for all $X\in \mathfrak{g}$, it holds $\tr\ad X=0$.)

\begin{lem}[{\cite[Lemma 6.2]{milnor}}]
 Any connected Lie group that admits a discrete subgroup with compact quotient is unimodular and in particular admits a bi-invariant volume form $\eta$.
\end{lem}

We will also need the following trick by F.~A. Belgun (see also \cite[Theorem 2.1]{fino-grantcharov}).

\begin{lem}[{F.~A. Belgun's symmetrization trick, \cite[Theorem 7]{belgun}}]
Let $X=\left.\Gamma\right\backslash G$ be a solvmanifold, and denote the Lie algebra naturally associated to $G$ by $\mathfrak{g}$. Let $\eta$ be a $G$-bi-invariant volume form on $G$ such that $\int_X\eta=1$, whose existence follows from J. Milnor's lemma \cite[Lemma 6.2]{milnor}. Up to identifying $G$-left-invariant forms on $X$ and linear forms over $\duale{\mathfrak{g}}$ through left-translations, define the \emph{F.~A. Belgun's symmetrization map}
$$ \mu\colon \wedge^\bullet X \to \wedge^\bullet \duale{\mathfrak{g}}\;,\qquad \mu(\alpha)\;:=\;\int_X \alpha\lfloor_m \, \eta(m) \;.$$
One has that
$$ \mu\lfloor_{\wedge^\bullet \duale{\mathfrak{g}}}\;=\;\id\lfloor_{\wedge^\bullet \duale{\mathfrak{g}}} \;, $$
and that
$$  \de \circ \mu \;=\; \mu \circ \de \;.$$
\end{lem}

In particular, the symmetrization map $\mu$ induces a map $\mu\colon \left(\wedge^\bullet X,\, \de\right)\to \left(\wedge^\bullet\duale{\mathfrak{g}},\, \de\right)$ of differential complexes, and hence a map $\mu\colon H^\bullet_{dR}(X;\R)\to H_{dR}^\bullet\left(\mathfrak{g};\R\right)$ in cohomology. Since $\mu\lfloor_{\wedge^\bullet \duale{\mathfrak{g}}} = \id\lfloor_{\wedge^\bullet \duale{\mathfrak{g}}}$, if the inclusion $\left(\wedge^\bullet\duale{\mathfrak{g}},\, \de\right)\hookrightarrow \left(\wedge^\bullet X,\, \de\right)$ is a quasi-isomorphism (for example, if $X$ is a nilmanifold, by \cite[Theorem 1]{nomizu}, or a completely-solvable solvmanifold, by \cite[Corollary 4.2]{hattori}), then the map $\mu\colon \left(\wedge^\bullet X,\, \de\right)\to \left(\wedge^\bullet\duale{\mathfrak{g}},\, \de\right)$ turns out to be a quasi-isomorphism.

\medskip

K. Nomizu's theorem \cite[Theorem 1]{nomizu}, A. Hattori's theorem \cite[Corollary 4.2]{hattori}, and F.~A. Belgun's theorem \cite[Theorem 7]{belgun} suggest that nilmanifolds, and, more in general, solvmanifolds, may provide a very useful and interesting class of examples in non-K\"ahler geometry. On the other hand, another reason for this statement is given by the following results by Ch. Benson and C.~S. Gordon, and by K. Hasegawa.

\begin{thm}[{\cite[Theorem A]{benson-gordon-nilmanifolds}}]
 Let $X$ be a nilmanifold endowed with a symplectic structure $\omega$ such that the Hard Lefschetz Condition holds. Then $X$ is diffeomorphic to a torus.
\end{thm}

(Actually, one can prove that any $2n$-dimensional nilmanifold $X$ endowed with a symplectic structure $\omega$ such that the map $\left[\omega\right]^{n-1}\colon H^1_{dR}(X;\R)\to H^{2n-1}_{dR}(X;\R)$ is an isomorphism is diffeomorphic to a torus, \cite{lupton-oprea}, see, e.g., \cite[Theorem 4.98]{felix-oprea-tanre}. A minimal model proof of Ch. Benson and C.~S. Gordon's theorem \cite[Theorem A]{benson-gordon-nilmanifolds} is due to G. Lupton and J. Oprea, \cite[Theorem 3.5]{lupton-oprea}.)

\begin{thm}[{\cite[Theorem 1, Corollary]{hasegawa_pams}}]
 Let $X$ be a nilmanifold. If $X$ is formal, then $X$ is diffeomorphic to a torus.
\end{thm}

In particular, since compact K\"ahler manifolds satisfy the Hard Lefschetz Condition, \cite[Théorème IV.5]{weil}, and are formal, \cite[Main Theorem]{deligne-griffiths-morgan-sullivan}, it follows that a nilmanifold admits a K\"ahler structure if and only if it is diffeomorphic to a torus (compare also \cite[Theorem II, Footnote 1]{hano}). More in general, compact completely-solvable K\"ahler solvmanifolds are tori, as proven by A. Tralle and J. Kedra in \cite[Theorem 1]{tralle-kedra}, solving a conjecture by Ch. Benson and C.~S. Gordon, \cite[page 972]{benson-gordon-solvmanifolds}. In fact, the following result by K. Hasegawa gives a complete characterization of K\"ahler solvmanifolds.

\begin{thm}[{\cite[Main Theorem]{hasegawa-3}}] 
Let $X$ be a compact homogeneous space of solvable Lie group, that is, a compact differentiable manifold on which a connected solvable Lie group acts transitively.
Then $X$ admits a K\"ahler structure if and only if it is a finite quotient of a complex torus which has a structure of a complex torus-bundle over a complex torus.
In particular, a completely-solvable solvmanifold has a K\"ahler structure if and only if it is a complex torus.
\end{thm}

\chapter{Cohomology of complex manifolds}\label{chapt:complex}

In this chapter, we study cohomological properties of compact complex manifolds. In particular, we are concerned with studying the \emph{Bott-Chern cohomology}, which, in a sense, constitutes a bridge between the de Rham cohomology and the Dolbeault cohomology of a complex manifold.

In \S\ref{sec:preliminaries-bott-chern-deldelbar}, we recall some definitions and results on the \emph{Bott-Chern} and \emph{Aeppli} cohomologies, see, e.g., \cite{schweitzer}, and on the \emph{$\del\delbar$-Lemma}, referring to \cite{deligne-griffiths-morgan-sullivan}.
In \S\ref{sec:cohomology-complex}, we provide a Fr\"olicher-type inequality for the Bott-Chern cohomology, Theorem \ref{thm:frol-bc}, which also allows to characterize the validity of the \emph{$\del\delbar$-Lemma} in terms of the dimensions of the Bott-Chern cohomology groups, Theorem \ref{thm:caratterizzazione-bc-numbers}; the proof of such inequality is based on two exact sequences, firstly considered by J. Varouchas in \cite{varouchas}.
In \S\ref{sec:computations-nilmfds}, we show that, for certain classes of complex structures on \emph{nilmanifolds} (that is, compact quotients of connected simply-connected nilpotent Lie groups by co-compact discrete subgroups), the Bott-Chern cohomology is completely determined by the associated Lie algebra endowed with the induced linear complex structure, Theorem \ref{thm:bc-invariant-cor}, giving a sort of Nomizu-type result for the Bott-Chern cohomology.
This will allow us to explicitly study the Bott-Chern and Aeppli cohomologies of the \emph{Iwasawa manifold} and of its small deformations, in \S\ref{sec:computations-iwasawa}.
In \S\ref{sec:orbifolds}, we investigate the Bott-Chern cohomology of complex \emph{orbifolds} of the type $\left. X \right\slash G$, where $X$ is a compact complex manifold and $G$ a finite group of biholomorphisms of $X$, Theorem \ref{thm:bc}.

Some of the original results of this chapter have been obtained in \cite{angella}, and jointly with A. Tomassini in \cite{angella-tomassini-3}; \S\ref{sec:orbifolds} contains some original results that have not yet been submitted for publication.

\section{Cohomologies of complex manifolds}\label{sec:preliminaries-bott-chern-deldelbar}
The \emph{Bott-Chern cohomology} and the \emph{Aeppli cohomology} provide important invariants for the study of the geometry of compact (especially, non-K\"ahler) complex manifolds. These cohomology groups have been introduced by R. Bott and S.~S. Chern in \cite{bott-chern}, and by A. Aeppli in \cite{aeppli}, and hence studied by many authors, e.g., B. Bigolin \cite{bigolin, bigolin-2} (both from the sheaf-theoretic and from the analytic viewpoints), A. Andreotti and F. Norguet \cite{andreotti-norguet} (to study cycles of algebraic manifolds), J. Varouchas \cite{varouchas} (to study the cohomological properties of a certain class of compact complex manifolds), M. Abate \cite{abate} (to study annular bundles), L. Alessandrini and G. Bassanelli \cite{alessandrini-bassanelli--complex-geometry-trento} (to investigate the properties of balanced metrics), S. Ofman \cite{ofman-6, ofman-5, ofman-3} (in view of applications to integration on analytic cycles),
S. Boucksom \cite{boucksom} (in order to extend divisorial Zariski decompositions to compact complex manifolds),
J.-P. Demailly \cite{demailly-agbook} (as a tool in Complex Geometry), M. Schweitzer \cite{schweitzer} (in the context of cohomology theories), L. Lussardi \cite{lussardi} (in the non-compact K\"ahler case), R. Kooistra \cite{kooistra} (in the framework of cohomology theories), J.-M. Bismut \cite{bismut-cras11, bismut-preprint} (in the context of Chern characters), L.-S. Tseng and S.-T. Yau \cite{tseng-yau-3} (in the framework of Generalized Geometry and type II String Theory).

In this preliminary section, we recall the basic notions and classical results concerning cohomologies of complex manifolds. More precisely, we recall the definitions of the Bott-Chern and Aeppli cohomologies, and some results on Hodge theory, referring to \cite{schweitzer}; then, we recall the notion of $\del\delbar$-Lemma, referring to \cite{deligne-griffiths-morgan-sullivan}.

\subsection{The Bott-Chern cohomology}
Let $X$ be a complex manifold.
The \emph{Bott-Chern cohomology} of $X$ is the bi-graded algebra
$$ H^{\bullet,\bullet}_{BC}(X) \;:=\; \frac{\ker \del \cap \ker \delbar}{\imm \del\delbar} \;.$$

\medskip

Unlike in the case of the Dolbeault cohomology groups, for every $p,q\in\N$, the conjugation induces an isomorphism
$$ H^{p,q}_{BC}(X)\;\simeq\; H^{q,p}_{BC}(X) \;.$$
Furthermore, since $\ker\del\cap\ker\delbar\subseteq\ker\de$ and $\imm\del\delbar\subseteq\imm\de$, one has the natural map of graded $\C$-vector spaces
$$ \bigoplus_{p+q=\bullet}H^{p,q}_{BC}(X)\to H^{\bullet}_{dR}(X;\C) \;, $$
and, since $\ker\del\cap\ker\delbar\subseteq\ker\delbar$ and $\imm\del\delbar\subseteq\imm\delbar$, one has the natural map of bi-graded $\C$-vector spaces
$$ H^{\bullet,\bullet}_{BC}(X)\to H^{\bullet,\bullet}_{\delbar}(X) \;.$$
In general, even for compact complex manifolds, these maps are neither injective nor surjective: see, e.g., the examples in \cite[\S1.c]{schweitzer} or in \S\ref{sec:bott-chern-iwasawa}. A case of special interest is when $X$ is a compact complex manifold satisfying the $\del\delbar$-Lemma, namely, the property that every $\del$-closed $\delbar$-closed $\de$-exact form is also $\del\delbar$-exact, \cite{deligne-griffiths-morgan-sullivan}, that is, the natural map $H^{\bullet,\bullet}_{BC}(X)\to H^\bullet_{dR}(X;\C)$ is injective (we recall that compact K\"ahler manifolds and, more in general, manifolds in class $\mathcal{C}$ of Fujiki, \cite{fujiki}, that is, compact complex manifolds admitting a proper modification from a K\"ahler manifold, satisfy the $\del\delbar$-Lemma, \cite[Lemma 5.11, Corollary 5.23]{deligne-griffiths-morgan-sullivan}; we refer to \S\ref{subsec:deldelbar-lemma} for further details). In fact, we recall the following result.

\begin{thm}[{\cite[Lemma 5.15, Remark 5.16, 5.21]{deligne-griffiths-morgan-sullivan}}]
 Let $X$ be a compact complex manifold. If $X$ satisfies the $\del\delbar$-Lemma, then the natural maps
 $$ \bigoplus_{p+q=\bullet}H^{p,q}_{BC}(X)\to H^{\bullet}_{dR}(X;\C) \qquad \text{ and } \qquad H^{\bullet,\bullet}_{BC}(X)\to H^{\bullet,\bullet}_{\delbar}(X) $$
 induced by the identity are isomorphisms.
\end{thm}

\medskip

As for the de Rham and the Dolbeault cohomologies, a Hodge theory can be developed also for the Bott-Chern cohomology for compact complex manifolds: we recall here some results, referring to \cite[\S2]{schweitzer} (see also \cite[\S5]{bigolin}, and \cite{lussardi}).

Suppose that $X$ is a compact complex manifold.
Fix a Hermitian metric on $X$, and define the differential operator
$$ \tilde\Delta_{BC} \;:=\;
\left(\del\delbar\right)\left(\del\delbar\right)^*+\left(\del\delbar\right)^*\left(\del\delbar\right)+\left(\delbar^*\del\right)\left(\delbar^*\del\right)^*+\left(\delbar^*\del\right)^*\left(\delbar^*\del\right)+\delbar^*\delbar+\del^*\del \;, $$
see \cite[Proposition 5]{kodaira-spencer-3} (where it is used to prove the stability of the K\"ahler property under small deformations of the complex structure), and also \cite[\S2.b]{schweitzer}, \cite[\S5.1]{bigolin}. One has the following result.

\begin{thm}[{\cite[Proposition 5]{kodaira-spencer-3}, see also \cite[\S2.b]{schweitzer}}]
Let $X$ be a compact complex manifold endowed with a Hermitian metric. The operator $\tilde\Delta_{BC}$ is a \kth{4} order self-adjoint elliptic differential operator, and
$$ \ker \tilde\Delta_{BC} \; = \; \ker\del \cap \ker \delbar \cap \ker \delbar^*\del^* \;. $$
\end{thm}

Therefore, as a consequence of the general theory of self-adjoint elliptic differential operators, see, e.g., \cite[page 450]{kodaira}, the following result holds.
\begin{thm}[{\cite[Th\'eor\`eme 2.2]{schweitzer}, \cite[Corollaire 2.3]{schweitzer}}]
 Let $X$ be a compact complex manifold, endowed with a Hermitian metric. Then there exist an orthogonal decomposition
$$ \wedge^{\bullet,\bullet}X \;=\; \ker\tilde\Delta_{BC} \,\oplus\, \imm\del\delbar \,\oplus\, \left(\imm\del^*\,+\,\imm\delbar^*\right) $$
and an isomorphism
$$ H^{\bullet,\bullet}_{BC}(X) \;\simeq\; \ker\tilde\Delta_{BC} \;.$$
In particular, the Bott-Chern cohomology groups of $X$ are finite-dimensional $\C$-vector spaces.
\end{thm}

\medskip

Another consequence of general results in spectral theory, see, e.g., \cite[Theorem 4]{kodaira-spencer-3}, \cite[Theorem 7.3]{kodaira}, is the semi-continuity property for the dimensions of the Bott-Chern cohomology.

\begin{thm}[{\cite[Lemme 3.2]{schweitzer}}]
 Let $\left\{X_t\right\}_{t\in B}$ a complex-analytic family of compact complex manifolds. Then, for every $p,q\in\N$, the function
 $$ B \ni t \mapsto \dim_\C H^{p,q}_{BC}\left(X_t\right) \in \N $$
 is upper-semi-continuous.
\end{thm}

\medskip

By using the K\"ahler identities (in particular, the fact that $\overline\square=\square$ and that $\del^*\delbar+\delbar\del^*=0=\delbar^*\del+\del\delbar^*$), one can prove that, on a compact K\"ahler manifold,
$$ \tilde\Delta_{BC} \;=\; \overline{\square}^2+\del^*\,\del+\delbar^*\,\delbar \;,$$
\cite[Proposition 6]{kodaira-spencer-3}, \cite[Proposition 2.4]{schweitzer}, and hence $\ker\tilde\Delta_{BC}=\ker\overline{\square}=\ker\Delta$; in particular, it follows that, on a compact K\"ahler manifold, the de Rham cohomology, the Dolbeault cohomology, and the Bott-Chern cohomology are isomorphic (actually, since the $\del\delbar$-Lemma holds on every compact K\"ahler manifold, one gets an isomorphism that does not depend on the choice of the Hermitian metric).

\subsection{The Aeppli cohomology}
Let $X$ be a complex manifold.
Dualizing the definition of the Bott-Chern cohomology, one can define another cohomology on $X$, the \emph{Aeppli cohomology}: it is the bi-graded $H^{\bullet, \bullet}_{BC}(X)$-module
$$ H^{\bullet,\bullet}_{A}(X) \;:=\; \frac{\ker \del\delbar}{\imm\del \,+\, \imm\delbar} \;.$$

\medskip

As for the Bott-Chern cohomology, the conjugation induces, for every $p,q\in\N$, the isomorphism
$$ H^{p,q}_{A}(X)\;\simeq\; H^{q,p}_{A}(X) \;. $$
Furthermore, since $\ker\de\subseteq\ker\del\delbar$ and $\imm\de\subseteq\imm\del+\imm\delbar$, one has the natural map of graded $\C$-vector spaces
$$ H^{\bullet}_{dR}(X;\C)\to \bigoplus_{p+q=\bullet}H^{p,q}_{A}(X) \;,$$
and, since $\ker\delbar\subseteq\ker\del\delbar$ and $\imm\delbar\subseteq\imm\del+\imm\delbar$, one has the natural map of bi-graded $\C$-vector spaces
$$ H^{\bullet,\bullet}_{\delbar}(X)\to H^{\bullet,\bullet}_{A}(X) \;;$$
as we have noted for the Bott-Chern cohomology, such maps are, in general, neither injective nor surjective, but they are isomorphisms whenever $X$ is compact and satisfies the $\del\delbar$-Lemma, \cite[Lemma 5.15, Remark 5.16, 5.21]{deligne-griffiths-morgan-sullivan}, and hence, in particular, if $X$ is a compact complex manifold admitting a K\"ahler structure, \cite[Lemma 5.11]{deligne-griffiths-morgan-sullivan}, or if $X$ is a compact complex manifold in class $\mathcal{C}$ of Fujiki, \cite[Corollary 5.23]{deligne-griffiths-morgan-sullivan}.

\begin{rem}
 On a compact K\"ahler manifold $X$, the associated $(1,1)$-form $\omega$ of the K\"ahler metric defines a non-zero class in $H^2_{dR}(X;\R)$. For general Hermitian manifolds, special classes of metrics are often defined in terms of closedness of powers of $\omega$, so they define classes in the Bott-Chern or Aeppli cohomology groups (e.g., a Hermitian metric on a complex manifold of complex dimension $n$ is said \emph{balanced} if $\de\omega^{n-1}=0$ \cite{michelsohn}, \emph{pluriclosed} if $\del\delbar\omega=0$ \cite{bismut--math-ann-1989}, \emph{astheno-K\"ahler} if $\del\delbar\omega^{n-2}=0$ \cite{jost-yau, jost-yau-correction}, \emph{Gauduchon} if $\del\delbar\omega^{n-1}=0$ \cite{gauduchon}). (Note that, they define possibly the zero class in the Bott-Chern or Aeppli cohomologies: for the balanced case, see \cite[Corollary 1.3]{fu-li-yau}, where it is shown that, for $k\geq 2$, the complex structures on $\sharp_{j=1}^{k} \left(\mathbb{S}^3\times\mathbb{S}^3\right)$ constructed from the conifold 
transitions admit balanced metrics.)
\end{rem}

\medskip

We refer to \cite[\S2.c]{schweitzer} for the following results, concerning Hodge theory for the Aeppli cohomology on compact complex manifolds.

Suppose that $X$ is a compact complex manifold. Once fixed a Hermitian metric on $X$, one defines the differential operator
$$ \tilde\Delta_{A} \;:=\; \del\del^*+\delbar\delbar^*+\left(\del\delbar\right)^*\left(\del\delbar\right)+\left(\del\delbar\right)\left(\del\delbar\right)^*+\left(\delbar\del^*\right)^*\left(\delbar\del^*\right)+\left(\delbar\del^*\right)\left(\delbar\del^*\right)^* \;, $$
which turns out to be a \kth{4} order self-adjoint elliptic differential operator such that
$$ \ker \tilde\Delta_{A} \;=\; \ker\del\delbar \cap \ker\del^* \cap \ker\delbar^* \;.$$
Hence one has an orthogonal decomposition
$$ \wedge^{\bullet,\bullet}X \;=\; \ker\tilde\Delta_{A} \,\oplus\, \left(\imm\del \,+\, \imm\delbar\right)
\,\oplus\, \imm\left(\del\delbar\right)^* $$
from which one gets an isomorphism
$$ H^{\bullet,\bullet}_{A}(X) \;\simeq\; \ker\tilde\Delta_{A} \;;$$
in particular, this proves that the Aeppli cohomology groups of a compact complex manifold are finite-dimensional $\C$-vector spaces.

Furthermore, as for the Bott-Chern cohomology, if $\left\{X_t\right\}_{t\in B}$ is a complex-analytic family of compact complex manifolds, with $B$ a complex manifold, then, for every $p,q\in\N$, the function $B \ni t \mapsto \dim_\C H^{p,q}_{A}\left(X_t\right) \in \N$ is upper-semi-continuous.

Once again, whenever $X$ is a compact K\"ahler manifold, by using the K\"ahler identities, one has
$$ \tilde\Delta_A \;=\; \overline\square^2 + \del\del^* + \delbar\delbar^* \;; $$
indeed, recall that $\overline\square=\square$ and that $\del^*\delbar=\im\,\left[\Lambda,\delbar\right]\,\delbar=-\im\,\delbar\,\Lambda\,\delbar=-\im\,\delbar\,\left[\Lambda,\delbar\right]=-\delbar\del^*$, and hence $\delbar^*\del=-\del\delbar^*$; therefore
\begin{eqnarray*}
\overline\square^2 \;=\; \overline\square \, \square &=& \phantom{+}\delbar\delbar^*\del\del^*+\delbar\delbar^*\del^*\del+\delbar^*\delbar\del\del^*+\delbar^*\delbar\del^*\del \\[5pt]
 &=& -\delbar\del\delbar^*\del^*-\delbar\del^*\delbar^*\del-\delbar^*\del\delbar\del^*-\delbar^*\del^*\delbar\del \\[5pt]
 &=& \phantom{+}\del\delbar\delbar^*\del^*+\delbar\del^*\del\delbar^*+\del\delbar^*\delbar\del^*+\delbar^*\del^*\del\delbar \\[5pt]
 &=& \tilde\Delta_A - \del\del^*-\delbar\delbar^* \;.
\end{eqnarray*}
In particular, it follows that, on a compact K\"ahler manifold, $\ker\tilde\Delta_{A}=\ker\overline{\square}=\ker\Delta$, and hence the de Rham cohomology, the Dolbeault cohomology, and the Aeppli cohomology are isomorphic (actually, since the $\del\delbar$-Lemma holds on every compact K\"ahler manifold, one gets an isomorphism that does not depend on the choice of the Hermitian metric).

\medskip

In fact, since $\ker \tilde\Delta_{BC} = \ker\del \cap \ker \delbar \cap \ker \delbar^*\del^*$ and  $\ker \tilde\Delta_{A} = \ker\del\delbar \cap \ker\del^* \cap \ker\delbar^*$, one has the following isomorphism between the Bott-Chern cohomology and the Aeppli cohomology.

\begin{thm}[{\cite[\S2.c]{schweitzer}}]
 Let $X$ be a compact complex manifold of complex dimension $n$. For any $p,q\in\N$, the Hodge-$*$-operator associated to a Hermitian metric on $X$ induces an isomorphism,
 $$ *\colon H^{p,q}_{BC}(X) \simeq H^{n-q,n-p}_{A}(X) \;, $$
 between the Bott-Chern and the Aeppli cohomologies.
\end{thm}

\begin{rem}
We refer to \cite[\S VI.12]{demailly-agbook}, \cite[\S4]{schweitzer}, \cite[\S3.2, \S3.5]{kooistra} for a sheaf cohomology interpretation of the Bott-Chern and Aeppli cohomologies (see also Remark \ref{rem:hypoercoh-bc-orbifolds}).
\end{rem}

\subsection{The $\partial\overline{\partial}$-Lemma}\label{subsec:deldelbar-lemma}
Let $X$ be a compact complex manifold, and consider its complex de Rham $H^{\bullet}_{dR}(X;\C)$, Dolbeault $H^{\bullet,\bullet}_{\delbar}(X)$, conjugate Dolbeault $H^{\bullet,\bullet}_{\del}(X)$, Bott-Chern $H^{\bullet,\bullet}_{BC}(X)$, and Aeppli $H^{\bullet,\bullet}_{A}(X)$ cohomologies.

The identity map induces the following natural maps of (bi-)graded $\C$-vector spaces:
$$
\xymatrix{
 & H^{\bullet,\bullet}_{BC}(X) \ar[d]\ar[ld]\ar[rd] & \\
 H^{\bullet,\bullet}_{\del}(X) \ar[rd] & H^{\bullet}_{dR}(X;\C) \ar[d] & H^{\bullet,\bullet}_{\delbar}(X) \ar[ld] \\
 & H^{\bullet,\bullet}_{A}(X)  & 
}
$$
In general, these maps are neither injective nor surjective: see, e.g., the examples in \cite[\S1.c]{schweitzer} or in \S\ref{sec:bott-chern-iwasawa}.

\medskip

By \cite[Lemma 5.15, Proposition 5.17]{deligne-griffiths-morgan-sullivan}, it turns out that, if one of the above map is an isomorphism, then all the maps are isomorphisms, \cite[Remark 5.16]{deligne-griffiths-morgan-sullivan}; this is encoded in the notion of \emph{$\del\delbar$-Lemma}, which can be introduced in the more general setting of bounded double complexes of vector spaces. We start by recalling the following general result by P. Deligne, Ph.~A. Griffiths, J. Morgan, and D.~P. Sullivan, \cite{deligne-griffiths-morgan-sullivan}.

\begin{prop}[{\cite[Lemma 5.15]{deligne-griffiths-morgan-sullivan}}]
 Let $\left(K^{\bullet,\bullet},\, \de',\, \de''\right)$ be a bounded double complex of vector spaces (or, more in general, of objects of any Abelian category), and let $\left(K^\bullet,\, \de\right)$ be the associated simple complex, where $\de:=\de'+\de''$. For each $h\in\N$, the following conditions are equivalent:
\begin{enumerate}[$(a)_{h}$]
 \item $\ker\de'\cap\ker\de''\cap\imm\de=\imm\de'\de''$ in $K^h$;
 \item $\ker\de''\cap\imm\de'=\imm\de'\de''$ and $\ker\de'\cap\imm\de''=\imm\de'\de''$ in $K^h$;
 \item $\ker\de'\cap\ker\de''\cap\left(\imm\de'+\imm\de''\right)=\imm\de'\de''$ in $K^h$;
\end{enumerate}
\begin{enumerate}[$(a^*)_{h-1}$]
 \item $\imm\de'+\imm\de''+\ker\de=\ker\de'\de''$ in $K^{h-1}$;
 \item $\imm\de''+\ker\de'=\ker\de'\de''$ and $\imm\de'+\ker\de''=\ker\de'\de''$ in $K^{h-1}$;
 \item $\imm\de'+\imm\de''+\left(\ker\de'\cap\ker\de''\right)=\ker\de'\de''$ in $K^{h-1}$.
\end{enumerate}
\end{prop}

The above equivalent conditions define the validity of the $\de'\de''$-Lemma for a double complex.

\begin{defi}[{\cite{deligne-griffiths-morgan-sullivan}}]
 Let $\left(K^{\bullet,\bullet},\, \de',\, \de''\right)$ be a bounded double complex of vector spaces (or, more in general, of objects of any Abelian category), and let $\left(K^\bullet,\, \de\right)$ be the associated simple complex, where $\de:=\de'+\de''$. One says that $\left(K^{\bullet,\bullet},\, \de',\, \de''\right)$ \emph{satisfies the $\de'\de''$-Lemma} if, for every $h\in\N$, the equivalent conditions in \cite[Lemma 5.15]{deligne-griffiths-morgan-sullivan} hold.
\end{defi}

The following result by P. Deligne, Ph.~A. Griffiths, J. Morgan, and D.~P. Sullivan, \cite{deligne-griffiths-morgan-sullivan}, gives a characterization for the validity of the $\de'\de''$-Lemma.

\begin{thm}[{\cite[Proposition 5.17]{deligne-griffiths-morgan-sullivan}}]
 Let $\left(K^{\bullet,\bullet},\, \de',\, \de''\right)$ be a bounded double complex of vector spaces, and let $\left(K^\bullet,\, \de\right)$ be the associated simple complex, where $\de:=\de'+\de$. The following conditions are equivalent:
 \begin{enumerate}[(\itshape i\upshape)]
  \item $\left(K^{\bullet,\bullet},\, \de',\, \de''\right)$ satisfies the $\de'\de''$-Lemma;
  \item $K^{\bullet,\bullet}$ is a sum of double complexes of the following two types:
    \begin{description}
     \item[(dots)] complexes which have only a single component, with $\de'=0$ and $\de''=0$;
     \item[(squares)] complexes which are a square of isomorphisms,
         $$
         \xymatrix{
          K^{p-1,q} \ar[r]^{\de'} & K^{p,q} \\
          K^{p-1,q-1} \ar[u]_{\simeq}^{\de''} \ar[r]^{\de'} & K^{p,q-1} \ar[u]_{\simeq}^{\de''}
         }
         $$
    \end{description}
    \item the spectral sequence defined by the filtration associated to either degree (denoted by ${{'}F}$ or ${{''}F}$) degenerates at $E_1$ (namely, $E_1=E_\infty$) and, for every $h\in\N$, the two induced filtrations are $h$-opposite on $H^h_{dR}(X;\C)$, i.e., ${{'}F}^p\oplus{{''}F}^q\stackrel{\simeq}{\to}H^h_{dR}(X;\C)$ for $p+q-1=h$.
 \end{enumerate}
\end{thm}

\medskip

In particular, we are interested in dealing with compact complex manifolds $X$, where one considers the double complex $\left(\wedge^{\bullet,\bullet}X,\, \del,\, \delbar\right)$.

\begin{defi}[{\cite{deligne-griffiths-morgan-sullivan}}]
 A compact complex manifold $X$ is said to \emph{satisfy the $\del\delbar$-Lemma} if $\left(\wedge^{\bullet,\bullet}X,\, \del,\, \delbar\right)$ satisfies the $\del\delbar$-Lemma, namely, if
 $$ \ker \del \cap \ker \delbar \cap \imm \de \;=\; \imm \del\delbar \;, $$
 that is, in other words, if the natural map $H^{\bullet,\bullet}_{BC}(X)\to H^\bullet_{dR}(X;\C)$ of graded $\C$-vector spaces induced by the identity is injective.
\end{defi}

\begin{rem}
 Let $X$ be a compact complex manifold. By considering the differential operator
 $$ \de^c \;:=\; -\im\,\left(\del-\delbar\right) \;,$$
 one can say that $X$ \emph{satisfies the $\de\de^c$}, by definition, if
 $$ \imm\de\cap\ker\de^c \;=\; \imm\de\de^c \;.$$
 Since $\de\de^c=2\,\im\,\del\delbar$, and $\del=\frac{1}{2}\left(\de+\im\de^c\right)$ and $\delbar=\frac{1}{2}\left(\de-\im\de^c\right)$, one has
 $$ \ker \de \cap \ker\de^c \;=\; \ker\del\cap\ker\delbar \qquad \text{ and } \qquad \imm\de\de^c=\imm\del\delbar \;;$$
 and hence $X$ satisfies the $\de\de^c$-Lemma if and only if $X$ satisfies the $\del\delbar$-Lemma.
\end{rem}

\medskip

For compact complex manifolds, P. Deligne, Ph.~A. Griffiths, J. Morgan, and D.~P. Sullivan's characterization \cite[Proposition 5.17]{deligne-griffiths-morgan-sullivan} is rewritten as follows.

\begin{thm}[{\cite[5.21]{deligne-griffiths-morgan-sullivan}}]
 A compact complex manifold $X$ satisfies the $\del\delbar$-Lemma if and only if
 \begin{inparaenum}[\itshape (i)]
  \item the Hodge and Fr\"olicher spectral sequence degenerates at the first step (that is, $E_1\simeq E_\infty$), and
  \item the natural filtration on $\left(\wedge^{\bullet,\bullet}X,\, \del,\, \delbar\right)$ induces, for every $k\in\N$, a Hodge structure of weight $k$ on $H^k_{dR}(X;\C)$ (that is, $H^k_{dR}(X;\C)=\bigoplus_{p+q=k} F^pH^k_{dR}(X;\C) \cap \bar{F}^qH^k_{dR}(X;\C)$, where $F^\bullet H^\bullet_{dR}(X;\C)$ is the filtration induced by $F^\bullet\wedge^{\bullet_1,\bullet_2}X:=\bigoplus_{p\geq \bullet ,\; q} \wedge^{p,q}X$ on $H^\bullet_{dR}(X;\C)$ and $\bar{F}^\bullet H^\bullet_{dR}(X;\C)$ is the conjugated filtration to $F^\bullet H^\bullet_{dR}(X;\C)$).
 \end{inparaenum}
\end{thm}

Another characterization for the validity of the $\del\delbar$-Lemma, in terms of the dimensions of the Bott-Chern cohomology, will be given in Theorem \ref{thm:caratterizzazione-bc-numbers}.

Actually, as already mentioned, if a compact complex manifold satisfies the $\del\delbar$-Lemma, then all the natural maps between cohomologies induced by the identity turn out to be isomorphisms.

\begin{thm}[{\cite[Lemma 5.15, Remark 5.16, 5.21]{deligne-griffiths-morgan-sullivan}}]
 A compact complex manifold $X$ satisfies the $\del\delbar$-Lemma if and only if all the natural maps
 $$
 \xymatrix{
  & H^{\bullet,\bullet}_{BC}(X) \ar[d]\ar[ld]\ar[rd] & \\
  H^{\bullet,\bullet}_{\del}(X) \ar[rd] & H^{\bullet}_{dR}(X;\C) \ar[d] & H^{\bullet,\bullet}_{\delbar}(X) \ar[ld] \\
  & H^{\bullet,\bullet}_{A}(X)  & 
 }
 $$
 induced by the identity are isomorphisms.
\end{thm}

\medskip

We recall that if $X$ is a compact complex manifold endowed with a K\"ahler structure, then $X$ satisfies the $\del\delbar$-Lemma, \cite[Lemma 5.11]{deligne-griffiths-morgan-sullivan}. Moreover, one has the following result.

\begin{thm}[{\cite[Theorem 5.22]{deligne-griffiths-morgan-sullivan}}]
 Let $X$ and $Y$ be compact complex manifolds of the same dimension, and let $f\colon X\to Y$ be a holomorphic birational map. If $X$ satisfies the $\del\delbar$-Lemma, then also $Y$ satisfies the $\del\delbar$-Lemma.
\end{thm}

Indeed, one has that, if $X$ and $Y$ are complex manifolds of the same dimension, and $\pi\colon X\to Y$ is a proper surjective holomorphic map, then the maps
$$ \pi^*\colon H^\bullet_{dR}\left(Y;\C\right) \to H^\bullet_{dR}\left(X;\C\right) \qquad \text{ and } \qquad \pi^*\colon H^{\bullet,\bullet}_{\delbar}\left(Y\right) \to H^{\bullet,\bullet}_{\delbar}\left(X\right) $$
induced by $\pi\colon X\to Y$ are injective, see, e.g., \cite[Theorem 3.1]{wells-pacific}; then one can use the characterization in \cite[5.21]{deligne-griffiths-morgan-sullivan}.

In particular, it follows that \emph{Mo\v{\i}{\v{s}}ezon manifolds} (that is, compact complex manifolds $X$ such that the degree of transcendence over $\C$ of the field of meromorphic functions over $X$ is equal to the complex dimension of $X$, \cite{moishezon-transl}, equivalently, compact complex manifolds admitting a proper modification from a projective manifold, \cite[Theorem 1]{moishezon-transl}), and, more in general, manifolds \emph{in class $\mathcal{C}$ of Fujiki} (that is, compact complex manifolds admitting a proper modification from a K\"ahler manifold, \cite{fujiki}) satisfy the $\del\delbar$-Lemma. (We recall that a proper holomorphic map $f\colon X\to Y$ from the complex manifold $X$ to the complex manifold $Y$ is called a \emph{modification} if there exists a nowhere dense closed analytic subset $B\subset Y$ such that $f\lfloor_{X\setminus f^{-1}(B)}\colon X\setminus f^{-1}(B) \to Y\setminus B$ is a biholomorphism.)

\begin{cor}[{\cite[Lemma 5.11, Corollary 5.23]{deligne-griffiths-morgan-sullivan}}]
 The $\del\delbar$-Lemma holds for compact K\"ahler manifolds, for Mo\v{\i}{\v{s}}ezon manifolds, and for manifolds in class $\mathcal{C}$ of Fujiki.
\end{cor}

\begin{rem}
 In \cite{hironaka}, H. Hironaka provided an example of a non-K\"ahler Mo\v{\i}\v{s}ezon manifold of complex dimension $3$ with arbitrary small deformations being projective (in fact, as stated by D. Popovici, the limit of projective manifolds under holomorphic deformations is Mo\v{\i}\v{s}ezon, \cite[Theorem 1.1]{popovici-proj}, and, more in general, the limit of Mo\v{\i}\v{s}ezon manifolds under holomorphic deformations is Mo\v{\i}\v{s}ezon, \cite[Theorem 1.1]{popovici-moish}); in particular, H. Hironaka's manifold provides an example of a non-K\"ahler manifold satisfying the $\del\delbar$-Lemma.
 Studying twistor spaces, C. LeBrun and Y.~S. Poon, and F. Campana, showed that being in class $\mathcal{C}$ of Fujiki is not a stable property under small deformations of the complex structures, \cite[Theorem 1]{lebrun-poon}, \cite[Corollary 3.13]{campana}; since the property of satisfying the $\del\delbar$-Lemma is stable under small deformations of the complex structure, Corollary \ref{cor:stab-del-delbar-lemma}, or \cite[Proposition 9.21]{voisin}, or \cite[Theorem 5.12]{wu}, or \cite[\S B]{tomasiello}, C. LeBrun and Y.~S. Poon's, and F. Campana's, result yields examples of compact complex manifolds satisfying the $\del\delbar$-Lemma and not belonging to class $\mathcal{C}$ of Fujiki.
\end{rem}

\medskip

Finally, we recall the following obstructions to the existence of complex structures satisfying the $\del\delbar$-Lemma on a compact (differentiable) manifold.

\begin{thm}[{\cite[Main Theorem, Corollary 1]{deligne-griffiths-morgan-sullivan}}]
 Let $X$ be a compact manifold. If $X$ admits a complex structure such that the $\del\delbar$-Lemma holds, then the differential graded algebra $\left(\wedge^{\bullet}X,\, \de\right)$ is formal. In particular, all the Massey products of any order are zero.
\end{thm}

Indeed, if $X$ satisfies the $\del\delbar$-Lemma, equivalently, the $\de\de^c$-Lemma, then the inclusion $\ker\de^c\to \wedge^\bullet X$ and the projection $\ker\de^c\to\frac{\ker\de^c}{\imm\de^c}$ induce the quasi-isomorphisms
$$
\xymatrix{
 & \left(\ker \de^c,\, \de\right) \ar[dl]^{\text{qis}} \ar[dr]_{\text{qis}} & \\
 \left(\wedge^\bullet X,\, \de\right) & & \left(\frac{\ker \de^c}{\imm\de^c},\, 0\right)
}
$$
of differential graded algebras, proving that $\left(\wedge^{\bullet}X,\, \de\right)$ is equivalent to $\left(\frac{\ker \de^c}{\imm\de^c},\, 0\right)$, and hence formal.

\section{Cohomological properties of compact complex manifolds and the $\partial\overline{\partial}$-Lemma}\label{sec:cohomology-complex}
In this section, we study some cohomological properties of compact complex manifolds, especially in relation with the $\del\delbar$-Lemma. More precisely, we prove a Fr\"olicher-type inequality for the Bott-Chern cohomology, Theorem \ref{thm:frol-bc}, and we characterize the validity of the $\del\delbar$-Lemma in terms of the dimensions of the Bott-Chern cohomology groups, Theorem \ref{thm:caratterizzazione-bc-numbers}. This has been the matter of a joint work with A. Tomassini, \cite{angella-tomassini-3}.

\medskip

Let $X$ be a compact complex manifold of complex dimension $n$.

As a matter of notation, for every $p,q\in\N$, for every $k\in\N$, and for $\sharp\in\left\{\delbar,\,\del,\,BC,\,A\right\}$, we will denote
$$ h^{p,q}_{\sharp} \;:=\; \dim_\C H^{p,q}_{\sharp}(X) \;<\; +\infty \qquad \text{ and }\qquad h^{k}_{\sharp} \;:=\; \sum_{p+q=k}h^{p,q}_{\sharp} \;<\; +\infty \;, $$
while recall that the Betti numbers are denoted by
$$ b_k \;:=\; \dim_\C H^{k}_{dR}(X;\C) \;<\; +\infty \;.$$
Recall that, for every $p,q\in\N$, the conjugation induces the isomorphisms $H^{p,q}_{BC}(X)\stackrel{\simeq}{\to}H^{q,p}_{BC}(X)$, $H^{p,q}_{A}(X)\stackrel{\simeq}{\to}H^{q,p}_{A}(X)$, and $H^{p,q}_{\delbar}(X)\stackrel{\simeq}{\to}H^{q,p}_{\del}(X)$, and the Hodge-$*$-operator associated to any given Hermitian metric induces the isomorphisms $H^{p,q}_{BC}(X)\stackrel{\simeq}{\to}H^{n-q,n-p}_{A}(X)$ and $H^{p,q}_{\delbar}(X)\stackrel{\simeq}{\to}H^{n-q,n-p}_{\del}(X)$; hence, for every $p,q\in\N$, one has the equalities
$$ h^{p,q}_{BC} \;=\; h^{q,p}_{BC} \;=\; h^{n-p,n-q}_{A} \;=\; h^{n-q,n-p}_{A} \quad \text{ and }\quad h^{p,q}_{\delbar} \;=\; h^{q,p}
_{\del} \;=\; h^{n-p,n-q}_{\delbar} \;=\; h^{n-q,n-p}_{\del} \;,$$
and therefore, for every $k\in\N$, one has the equalities
$$ h^k_{BC} \;=\; h^{2n-k}_A \quad \text{ and } \quad h^k_{\delbar} \;=\; h^k_{\del} \;=\; h^{2n-k}_{\delbar} \;=\; h^{2n-k}_{\del} \;;$$
Finally, recall that the Hodge-$*$-operator (of any given Riemannian metric and volume form on $X$) yields, for every $k\in\N$, the isomorphism $H^k_{dR}(X;\R)\stackrel{\simeq}{\to}H^{2n-k}_{dR}(X;\R)$, and hence the equality
$$ b_k \;=\; b_{2n-k} \;.$$

\subsection{J. Varouchas' exact sequences}

In order to prove a Fr\"olicher-type inequality for the Bott-Chern and Aeppli cohomologies and to give therefore a characterization of compact complex manifolds satisfying the $\del\delbar$-Lemma in terms of the dimensions of their Bott-Chern cohomology groups, we need to recall two exact sequences from \cite{varouchas}.

\medskip

Following J. Varouchas, one defines the (finite-dimensional) bi-graded $\C$-vector spaces
$$ A^{\bullet,\bullet} \;:=\; \frac{\imm\delbar\cap\imm\del}{\imm\del\delbar} \;,\qquad B^{\bullet,\bullet}\;:=\; \frac{\ker\delbar\cap\imm\del}{\imm\del\delbar} \;,\qquad C^{\bullet,\bullet}\;:=\; \frac{\ker\del\delbar}{\ker\delbar+\imm\del} $$
and
$$ D^{\bullet,\bullet} \;:=\; \frac{\imm\delbar\cap\ker\del}{\imm\del\delbar} \;,\qquad E^{\bullet,\bullet}\;:=\; \frac{\ker\del\delbar}{\ker\del+\imm\delbar} \;,\qquad F^{\bullet,\bullet}\;:=\; \frac{\ker\del\delbar}{\ker\delbar+\ker\del} \;.$$
For every $p,q\in\N$ and $k\in\N$, we will denote their dimensions by
\begin{eqnarray*}
 a^{p,q}\;:=\; \dim_\C A^{p,q} \;,\qquad & \displaystyle b^{p,q}\;:=\; \dim_\C B^{p,q} \;,\qquad & c^{p,q}\;:=\; \dim_\C C^{p,q} \;, \\[5pt]
 d^{p,q}\;:=\; \dim_\C D^{p,q} \;,\qquad & \displaystyle e^{p,q}\;:=\; \dim_\C E^{p,q} \;,\qquad & f^{p,q}\;:=\; \dim_\C F^{p,q} \;,
\end{eqnarray*}
and
\begin{eqnarray*}
 a^k\;:=\; \sum_{p+q=k}a^{p,q} \;,\qquad & \displaystyle b^k\;:=\; \sum_{p+q=k}b^{p,q} \;,\qquad & c^k\;:=\; \sum_{p+q=k}c^{p,q} \;, \\[5pt]
 d^k\;:=\; \sum_{p+q=k}d^{p,q} \;,\qquad & \displaystyle e^k\;:=\; \sum_{p+q=k}e^{p,q} \;,\qquad & f^k\;:=\; \sum_{p+q=k}f^{p,q} \;.
\end{eqnarray*}

The previous vector spaces give the following exact sequences, by J. Varouchas.

\begin{thm}[{\cite[\S3.1]{varouchas}}]
The sequences
\begin{equation}\label{eq:succesatta-1}
0 \to A^{\bullet,\bullet} \to B^{\bullet,\bullet} \to H^{\bullet,\bullet}_{\delbar}(X) \to H^{\bullet,\bullet}_{A}(X) \to C^{\bullet,\bullet} \to 0
\end{equation}
and
\begin{equation}\label{eq:succesatta-2}
0 \to D^{\bullet,\bullet} \to H^{\bullet,\bullet}_{BC}(X) \to H^{\bullet,\bullet}_{\delbar}(X) \to E^{\bullet,\bullet} \to F^{\bullet,\bullet} \to 0
\end{equation}
are exact sequences of finite-dimensional bi-graded $\C$-vector spaces.
\end{thm}

\begin{proof}
 We first prove the exactness of \eqref{eq:succesatta-1}.
 Since $\imm\delbar\subseteq\ker\delbar$, the map $A^{\bullet,\bullet} \to B^{\bullet,\bullet}$ is injective. The kernel of the map $B^{\bullet,\bullet} \to H^{\bullet,\bullet}_{\delbar}(X)$ is $\frac{\ker\delbar\cap\imm\del\cap\imm\delbar}{\imm\del\delbar}=\frac{\imm\delbar\cap\imm\del}{\imm\del\delbar}$, that is, the image of the map $A^{\bullet,\bullet} \to B^{\bullet,\bullet}$. The kernel of the map $H^{\bullet,\bullet}_{\delbar}(X) \to H^{\bullet,\bullet}_{A}(X)$ is $\frac{\ker\delbar\cap\imm\del}{\imm\delbar}$, that is, the image of the map $B^{\bullet,\bullet} \to H^{\bullet,\bullet}_{\delbar}(X)$. The kernel of the map $H^{\bullet,\bullet}_{A}(X) \to C^{\bullet,\bullet}$ is $\frac{\ker\delbar\cap\ker\del\delbar}{\imm\del+\imm\delbar}=\frac{\ker\delbar}{\imm\del+\imm\delbar}$, that is, the image of the map $H^{\bullet,\bullet}_{\delbar}(X) \to H^{\bullet,\bullet}_{A}(X)$. Finally, since $\imm\del+\imm\delbar\subseteq\ker\delbar+\imm\del$, the map $H^{\bullet,\bullet}_{A}(X) \to C^{\bullet,\bullet}$ is 
surjective.
 In particular, since $H^{\bullet,\bullet}_{A}(X)\to C^{\bullet,\bullet}$ is surjective, then $C^{\bullet,\bullet}$ has finite dimension; since the identity induces an injective map $B^{\bullet,\bullet}\to H^{\bullet,\bullet}_{BC}(X)$, then $B^{\bullet,\bullet}$ has finite dimension; hence, since $A^{\bullet,\bullet}\to B^{\bullet,\bullet}$ is injective, then also $A^{\bullet,\bullet}$ has finite dimension.

 We prove now the exactness of \eqref{eq:succesatta-2}.
 Since $\imm\delbar\subseteq\ker\delbar$, the map $D^{\bullet,\bullet} \to H^{\bullet,\bullet}_{BC}(X)$ is injective. The kernel of the map $H^{\bullet,\bullet}_{BC}(X) \to H^{\bullet,\bullet}_{\delbar}(X)$ is $\frac{\ker\del\cap\ker\delbar\cap\imm\delbar}{\imm\del\delbar}=\frac{\imm\delbar\cap\ker\del}{\imm\del\delbar}$, that is, the image of the map $D^{\bullet,\bullet} \to H^{\bullet,\bullet}_{BC}(X)$. The kernel of the map $H^{\bullet,\bullet}_{\delbar}(X) \to E^{\bullet,\bullet}$ is $\frac{\ker\delbar\cap \left(\ker\del+\imm\delbar\right)}{\imm\delbar}=\frac{\ker\delbar\cap \ker\del}{\imm\delbar}$, that is, the image of the map $H^{\bullet,\bullet}_{BC}(X) \to H^{\bullet,\bullet}_{\delbar}(X)$. The kernel of the map $E^{\bullet,\bullet} \to F^{\bullet,\bullet}$ is $\frac{\ker\del\delbar\cap\left(\ker\delbar+\ker\del\right)}{\ker\del+\imm\delbar}=\frac{\ker\del\delbar \cap \ker\delbar}{\ker\del+\imm\delbar}$, that is, the image of the map $H^{\bullet,\bullet}_{\delbar}(X) \to E^{\bullet,\bullet}$. 
Finally, since $\ker\del+\imm\delbar \subseteq \ker\delbar + \ker\del$, the map $E^{\bullet,\bullet} \to F^{\bullet,\bullet}$ is surjective.
 In particular, since $D^{\bullet,\bullet}\to H^{\bullet,\bullet}_{BC}(X)$ is injective, then $D^{\bullet,\bullet}$ has finite dimension; since the identity induces a surjective map $H^{\bullet,\bullet}_{A}(X)\to E^{\bullet,\bullet}$, then $E^{\bullet,\bullet}$ has finite dimension; hence, since $E^{\bullet,\bullet}\to F^{\bullet,\bullet}$ is surjective, then also $F^{\bullet,\bullet}$ has finite dimension.

\end{proof}

Note, \cite[\S3.1]{varouchas},
that the conjugation yields, for every $p,q\in\N$, the equalities
\begin{equation}\label{eq:apq=aqp}
a^{p,q}\;=\;a^{q,p}\;, \qquad f^{p,q} \;=\; f^{q,p}\;, \qquad d^{p,q}\;=\;b^{q,p}\;, \qquad e^{p,q}\;=\;c^{q,p} \;,
\end{equation}
and the isomorphisms $\delbar\colon C^{\bullet,\bullet}\stackrel{\simeq}{\to}D^{\bullet,\bullet+1}$ and $\del\colon E^{\bullet,\bullet}\stackrel{\simeq}{\to} B^{\bullet+1,\bullet}$ yield the equalities
$$ c^{p,q}\;=\;d^{p,q+1}\;,\qquad e^{p,q}\;=\;b^{p+1,q}\;;$$
hence, for every $k\in\N$, one gets the equalities
$$ d^k\;=\;b^k\;,\qquad e^k\;=\;c^k \;, \qquad \text{ and } \qquad c^k\;=\; d^{k+1}\;,\qquad e^k\;=\;b^{k+1}\;.$$

\begin{rem}
 Following the same argument used in \cite{schweitzer} to prove the duality between Bott-Chern and Aeppli cohomology groups, we can prove the duality between $A^{\bullet,\bullet}$ and $F^{\bullet,\bullet}$, and, similarly, between $C^{\bullet,\bullet}$ and $\overline{D^{\bullet,\bullet}}$.

 Indeed, note that the pairing
 $$ A^{\bullet,\bullet}\times F^{\bullet, \bullet} \to \C \;, \qquad \left(\left[\alpha\right],\, \left[\beta\right]\right) \mapsto \int_X \alpha\wedge \overline \beta \;, $$
 is non-degenerate: choose a Hermitian metric $g$ on $X$; if $\left[\alpha\right]\in A^{\bullet,\bullet}\subseteq H^{\bullet,\bullet}_{BC}(X)$, then there exists a $\tilde\Delta_{BC}$-harmonic representative $\tilde\alpha$ in $\left[\alpha\right]\in A^{\bullet,\bullet}$, by \cite[Corollaire 2.3]{schweitzer}, that is, $\del\tilde\alpha=\delbar\tilde\alpha=\del\delbar*\tilde\alpha=0$; hence, $\left[*\tilde\alpha\right]\in F^{\bullet,\bullet}$, and $\left(\left[\tilde\alpha\right],\, \left[*\tilde\alpha\right]\right)=\int_X \tilde\alpha\wedge \overline{*\tilde\alpha}$ is zero if and only if $\tilde\alpha$ is zero if and only if $\left[\alpha\right]\in A^{\bullet,\bullet}$ is zero.

 Analogously, the pairing
 $$ C^{\bullet,\bullet}\times \overline{D^{\bullet, \bullet}} \to \C \;, \qquad \left(\left[\alpha\right],\, \left[\beta\right]\right) \mapsto \int_X \alpha\wedge \overline \beta \;, $$
 is non-degenerate: indeed, choose a Hermitian metric $g$ on $X$; if $\left[\alpha\right]\in \overline{D^{\bullet,\bullet}} \subseteq \overline{H^{\bullet,\bullet}_{BC}(X)}$, then there exists a $\tilde\Delta_{BC}$-harmonic representative $\tilde\alpha$ in $\left[\alpha\right]\in \overline{D^{\bullet,\bullet}}$, by \cite[Corollaire 2.3]{schweitzer}, that is, $\del\tilde\alpha=\delbar\tilde\alpha=\del\delbar*\tilde\alpha=0$; hence, $\left[*\tilde\alpha\right]\in C^{\bullet,\bullet}$, and $\left(\left[\tilde\alpha\right],\, \left[*\tilde\alpha\right]\right)=\int_X \tilde\alpha\wedge \overline{*\tilde\alpha}$ is zero if and only if $\tilde\alpha$ is zero if and only if $\left[\alpha\right]\in \overline{D^{\bullet,\bullet}}$ is zero.
\end{rem}

\subsection{A Fr\"olicher-type inequality for the Bott-Chern cohomology}

We can now state and prove a Fr\"olicher-type inequality for the Bott-Chern and Aeppli cohomologies, Theorem \ref{thm:frol-bc}.

\medskip

Firstly, we recall that, on a compact complex manifold $X$, the \emph{Fr\"olicher inequality} \cite[Theorem 2]{frolicher} relates the Hodge numbers and the Betti numbers.

\begin{thm}[{\cite[Theorem 2]{frolicher}}]
 Let $X$ be a compact complex manifold. Then, for every $k\in\N$, the following inequality holds:
 $$ \sum_{p+q=k} \dim_\C H^{p,q}_{\delbar}(X) \;\geq\; \dim_\C H^k_{dR}(X;\C) \;.$$ 
\end{thm}

The equality $\sum_{p+q=k} \dim_\C H^{p,q}_{\delbar}(X) = \dim_\C H^k_{dR}(X;\C)$ holds for every $k\in\N$ if and only if the Hodge and Fr\"olicher spectral sequence $\left\{\left(E_r,\, \de_r\right)\right\}_{r\in\N}$ degenerates at the first step.

\medskip

It is in general not true that $h^k_{BC}$ (respectively, $h^k_{A}$) is higher than the \kth{k} Betti number of $X$ for every $k\in\N$: an example is provided by the small deformations of the Iwasawa manifold $\mathbb{I}_3 := \left. \mathbb{H}\left(3;\Z\left[\im\right]\right) \right\backslash \mathbb{H}(3;\C)$ (see \S\ref{subsec:iwasawa}). In the following table, we summarize the dimensions of the Bott-Chern and Aeppli cohomology groups for $\mathbb{I}_3$ (which have been computed in \cite[Proposition 1.2]{schweitzer}) and for the small deformations of $\mathbb{I}_3$ (see \S\ref{sec:bott-chern-iwasawa}). We recall that the small deformations of the Iwasawa manifold, according to I. Nakamura's classification, \cite[\S3]{nakamura}, are divided into three classes, {\itshape (i)}, {\itshape (ii)}, and {\itshape (iii)}, in terms of their Hodge numbers; it turns out that the Bott-Chern cohomology yields a finer classification of the Kuranishi space of $\mathbb{I}_3$, allowing a further subdivision of class
{\itshape (ii)}, respectively class {\itshape (iii)}, into subclasses {\itshape (ii.a)}, {\itshape (ii.b)}, respectively {\itshape (iii.a)}, {\itshape (iii.b)}, see \S\ref{sec:deformations-iwasawa}.

\smallskip
\begin{center}
\begin{small}
\begin{tabular}{c||*{3}{c}|*{3}{c}|*{3}{c}|*{3}{c}|*{3}{c}||}
\toprule
  {\bfseries classes} & $\mathbf{h^1_{\delbar}}$ & $\mathbf{h^1_{BC}}$ & $\mathbf{h^1_{A}}$ & $\mathbf{h^2_{\delbar}}$ & $\mathbf{h^2_{BC}}$ & $\mathbf{h^2_{A}}$ & $\mathbf{h^3_{\delbar}}$ & $\mathbf{h^3_{BC}}$ & $\mathbf{h^3_{A}}$ & $\mathbf{h^4_{\delbar}}$ & $\mathbf{h^4_{BC}}$ & $\mathbf{h^4_{A}}$ & $\mathbf{h^5_{\delbar}}$ & $\mathbf{h^5_{BC}}$ & $\mathbf{h^5_{A}}$\\
\midrule[0.02em]\midrule[0.02em]
{\itshape (i)} & 5 & 4 & 6 & 11 & 10 & 12 & 14 & 14 & 14 & 11 & 12 & 10 & 5 & 6 & 4\\
\midrule[0.02em]
{\itshape (ii.a)} & 4 & 4 & 6 & 9 & 8 & 11 & 12 & 14 & 14 & 9 & 11 & 8 & 4 & 6 & 4\\
{\itshape (ii.b)} & 4 & 4 & 6 & 9 & 8 & 10 & 12 & 14 & 14 & 9 & 10 & 8 & 4 & 6 & 4\\
\midrule[0.02em]
{\itshape (iii.a)} & 4 & 4 & 6 & 8 & 6 & 11 & 10 & 14 & 14 & 8 & 11 & 6 & 4 & 6 & 4\\
{\itshape (iii.b)} & 4 & 4 & 6 & 8 & 6 & 10 & 10 & 14 & 14 & 8 & 10 & 6 & 4 & 6 & 4\\
\midrule[0.02em]\midrule[0.02em]
 & \multicolumn{3}{c|}{$\mathbf{b_1=4}$} & \multicolumn{3}{c|}{$\mathbf{b_2=8}$} & \multicolumn{3}{c|}{$\mathbf{b_3=10}$} & \multicolumn{3}{c|}{$\mathbf{b_4=8}$} & \multicolumn{3}{c||}{$\mathbf{b_5=4}$}\\
\bottomrule
\end{tabular}
\end{small}
\end{center}
\smallskip

\medskip

The following result, \cite[Theorem A]{angella-tomassini-3}, gives a Fr\"olicher-type inequality for the Bott-Chern cohomology. (We recall that, on a compact complex manifold $X$ of complex dimension $n$, for any $p,q\in\N$, one has the equality $\dim_\C H^{p,q}_{BC}(X)=\dim_\C H^{n-q,n-p}_A(X)$, and, for any $k\in\N$, the equality $\sum_{p+q=k}\dim_\C H^{p,q}_{BC}(X) = \sum_{r+s=2n-k} \dim_\C H^{r,s}_{A}(X)$, \cite[\S2.c]{schweitzer}.)

\begin{thm}
\label{thm:frol-bc}
 Let $X$ be a compact complex manifold. Then, for every $p,q\in\N$, the following inequality holds:
\begin{equation}\label{eq:frol-bc-1}
\dim_\C H^{p,q}_{BC}(X) + \dim_\C H^{p,q}_{A}(X) \;\geq\; \dim_\C H^{p,q}_{\delbar}(X) + \dim_\C H^{p,q}_\del(X) \;.
\end{equation}
In particular, for every $k\in\N$, the following inequality holds:
\begin{equation}\label{eq:frol-bc-2}
\sum_{p+q=k} \left( \dim_\C H^{p,q}_{BC}(X) + \dim_\C H^{p,q}_{A}(X) \right) \;\geq\; 2\, \dim_\C H^k_{dR}(X;\C) \;.
\end{equation}
\end{thm}

\begin{proof}
Fix $p,q\in\N$. The exact sequences \eqref{eq:succesatta-1}, respectively \eqref{eq:succesatta-2}, yield the equality
$$ h^{p,q}_{A}  \;=\; h^{p,q}_\delbar + c^{p,q} + a^{p,q} - b^{p,q} \;,$$
respectively
$$ h^{p,q}_{BC} \;=\; h^{p,q}_\delbar + d^{p,q} + f^{p,q} - e^{p,q} \;;$$
using also the symmetries $h^{p,q}_A=h^{q,p}_A$ and $h^{p,q}_{\delbar}=h^{q,p}_{\del}$, and the equalities \eqref{eq:apq=aqp}, we get
\begin{eqnarray*}
h^{p,q}_{BC}+h^{p,q}_{A} &=& h^{p,q}_{BC}+h^{q,p}_{A} \\[5pt]
&=& h^{p,q}_{\delbar} + h^{q,p}_{\delbar} + f^{p,q} + a^{q,p} + d^{p,q} - b^{q,p} - e^{p,q} + c^{q,p} \\[5pt]
&=& h^{p,q}_{\delbar} + h^{p,q}_{\del} + f^{p,q} + a^{p,q}  \\[5pt]
&\geq& h^{p,q}_{\delbar} + h^{p,q}_{\del} \;,
\end{eqnarray*}
which proves \eqref{eq:frol-bc-1}.

Now, fix $k\in\N$; summing over $\left(p,q\right)\in\N\times\N$ such that $p+q=k$, we get
\begin{eqnarray*}
 h^k_{BC} + h^k_{A} &=& \sum_{p+q=k} \left( h^{p,q}_{BC} + h^{p,q}_{A} \right) \\[5pt]
&\geq& \sum_{p+q=k} \left( h^{p,q}_{\delbar} + h^{p,q}_{\del}\right) \;=\; h^k_{\delbar} + h^k_{\del} \\[5pt]
&\geq& 2\, b_k \;,
\end{eqnarray*}
from which we get \eqref{eq:frol-bc-2}.
\end{proof}

\begin{rem}
 Note that small deformations of the Iwasawa manifold show that both the inequalities \eqref{eq:frol-bc-1} and \eqref{eq:frol-bc-2} can be strict.

 For example, for small deformations of $\mathbb{I}_3$ in class {\itshape (i)}, one has,
 $$        h^1_{BC} + h^1_{A} \;=\; 10 \;>\;  8 \;=\; 2 \cdot b_1 \;,
    \qquad h^2_{BC} + h^2_{A} \;=\; 22 \;>\; 16 \;=\; 2 \cdot b_2 \;,
    \qquad h^3_{BC} + h^3_{A} \;=\; 28 \;>\; 20 \;=\; 2 \cdot b_3 \;,
 $$
 showing that \eqref{eq:frol-bc-2} is strict for every $k\in\{1,\, 2,\, 3,\, 4,\, 5\}$.

 On the other hand, for small deformations of $\mathbb{I}_3$ in class {\itshape (ii)} or in class {\itshape (iii)}, one has
 $$ h^{1,0}_{BC} + h^{1,0}_{A} \;=\; \frac{1}{2} \left(h^1_{BC} + h^1_{A} \right) \;=\; 5 \; > \; 4 \;=\; h^1_\delbar \;=\; h^{1,0}_\delbar + h^{0,1}_\delbar \;=\; h^{1,0}_\delbar + h^{1,0}_\del \;,$$
 showing that \eqref{eq:frol-bc-1} is strict, for example, for $(p,q)=(1,0)$.

 (For further examples among the small deformations of the Iwasawa manifold, compare the computations in \S\ref{sec:bott-chern-iwasawa}, which are summarized in \S\ref{subsec:chart}.)
\end{rem}

\begin{rem}
 Note that, in the proof of Theorem \ref{thm:frol-bc}, we have actually shown that, for every $k\in\N$,
\begin{equation*}
 h^k_{BC}+h^k_{A} \;=\; 2\,h^k_{\delbar}+a^k+f^k \;.
\end{equation*}
\end{rem}

\subsection{A characterization of the $\partial\overline{\partial}$-Lemma in terms of the Bott-Chern cohomology}

This section is devoted to give a characterization of the validity of the $\del\delbar$-Lemma in terms of the Bott-Chern cohomology.

Note that, if a compact complex manifold $X$ satisfies the $\del\delbar$-Lemma, then, for every $k\in\N$, it holds $h^k_{BC}=h^k_{A}=h^k_{\delbar}=h^k_{\del}=b_k$, and hence \eqref{eq:frol-bc-2} is actually an equality. In fact, we prove now that also the converse holds true: more precisely, the equality in \eqref{eq:frol-bc-2} holds for every $k\in\N$ if and only if the $\del\delbar$-Lemma holds; in particular, this gives a characterization of the validity of the $\del\delbar$-Lemma just in terms of $\left\{h^k_{BC}\right\}_{k\in\N}$, \cite[Theorem B]{angella-tomassini-3}.

\begin{thm}
\label{thm:caratterizzazione-bc-numbers}
 Let $X$ be a compact complex manifold. The equality
$$ \sum_{p+q=k} \left(\dim_\C H^{p,q}_{BC}(X) + \dim_\C H^{p,q}_{A}(X)\right) \;=\; 2\, \dim_\C H^k_{dR}(X;\C) $$
holds for every $k\in\N$ if and only if $X$ satisfies the $\del\delbar$-Lemma.
\end{thm}

\begin{proof}
 If $X$ satisfies the $\del\delbar$-Lemma, then the natural maps $H^{\bullet,\bullet}_{BC}(X)\to H^\bullet_{dR}(X;\C)$, $H^{\bullet,\bullet}_{BC}(X) \to H^{\bullet,\bullet}_{\delbar}(X)$, and $H^{\bullet,\bullet}_{\delbar}(X) \to H^{\bullet,\bullet}_{A}(X)$, $H^\bullet_{dR}(X;\C)\to H^{\bullet,\bullet}_{A}(X)$ induced by the identity are isomorphisms, \cite[Remark 5.16]{deligne-griffiths-morgan-sullivan}, and hence, for every $k\in\N$, one has
$$ h^k_{BC} \;=\; h^k_{A} \;=\; h^k_{\delbar} \;=\; b_k $$
and hence, in particular,
$$ h^k_{BC}+h^k_{A} \;=\; 2\,b_k \;.$$
 We split the proof of the converse into the following claims.

 \paragrafodclaim{1}{If $h^k_{BC}+h^k_{A}=2\,b_k$ holds for every $k\in\N$, then the Hodge and Fr\"olicher spectral sequences degenerate at the first step (namely, ${E_1}\simeq {E_{\infty}}$, that is, $h^k_{\delbar}=b_k$ for every $k\in\N$) and $a^k=0=f^k$ for every $k\in\N$}\\
Since, for every $k\in\N$, we have
$$ 2\,b_k \;=\; h^k_{BC}+h^k_{A} \;=\; 2\,h^k_{\delbar}+a^k+f^k \;\geq\; 2\,b_k \;,$$
then $h^k_{\delbar}=b_k$ and $a^k=0=f^k$ for every $k\in\N$.

 \paragrafodclaim{2}{Fix $k\in\N$. If $a^{k+1}:=\sum_{p+q=k+1}\dim_\C A^{p,q}=0$, then the natural map
$$
\bigoplus_{p+q=k}H^{p,q}_{BC}(X)\to H^{k}_{dR}(X;\C)
$$ 
is surjective}\\
 Let $\mathfrak{a}=\left[\alpha\right]\in H^k_{dR}(X;\C)$. We have to prove that $\mathfrak{a}$ admits a representative whose pure-type components are $\de$-closed. Consider the pure-type decomposition of $\alpha$:
$$ \alpha \;=:\; \sum_{j=0}^{k} \left(-1\right)^j\, \alpha^{k-j,j} \;,$$
where $\alpha^{k-j,j}\in \wedge^{k-j,j}X$. Since $\de\alpha=0$, we get that
$$ \del\alpha^{k,0}=0\;,\qquad \delbar\alpha^{k-j,j}-\del\alpha^{k-j-1,j+1}=0\text{ for }j\in\{0,\ldots,k-1\}\;,\qquad \delbar\alpha^{0.k}=0 \;; $$
by the hypothesis $a^{k+1}=0$, for every $j\in\{0,\ldots,k-1\}$, we get that,
$$ \delbar\alpha^{k-j,j}\;=\;\del\alpha^{k-j-1,j+1}\;\in\;\left(\imm\delbar\cap\imm\del\right)\cap\wedge^{k-j, j+1}X \;=\; \imm\del\delbar\cap\wedge^{k-j, j+1}X $$
and hence there exists $\eta^{k-j-1,j}\in\wedge^{k-j-1,j}X$ such that
$$ \delbar\alpha^{k-j,j} \;=\; \del\delbar\eta^{k-j-1,j} \;=\; \del\alpha^{k-j-1,j+1} \;.$$
Define
$$ \eta \;:=\; \sum_{j=0}^{k-1}\left(-1\right)^j\,\eta^{k-j-1,j} \;\in\;\wedge^{k-1}X \otimes_\R \C \;.$$
The claim follows noting that
\begin{eqnarray*}
\mathfrak{a} &=& \left[\alpha\right] \;=\; \left[\alpha+\de\eta\right] \\[5pt]
&=& \left[\left(\alpha^{k,0}+\del\eta^{k-1,0}\right)+\sum_{j=1}^{k-1}
\left(-1\right)^{j}\,\left(\alpha^{k-j,j}+\del\eta^{k-j-1,j}-\delbar\eta^{k-j,j-1}\right) + \left(-1\right)^k\,\left(\alpha^{0,k}-\delbar\eta^{0,k-1}\right)\right] \\[5pt]
&=& \left[\alpha^{k,0}+\del\eta^{k-1,0}\right]+\sum_{j=1}^{k-1}\left(-1\right)^{j}\,\left[\alpha^{k-j,j}+\del\eta^{k-j-1,j}-\delbar\eta^{k-j,j-1}\right] + \left(-1\right)^k\,\left[\alpha^{0,k}-\delbar\eta^{0,k-1}\right] \;,
\end{eqnarray*}
that is, each of the pure-type components of $\alpha+\de\eta$ is both $\del$-closed and $\delbar$-closed.

 \paragrafodclaim{3}{If $h^k_{BC}\geq b_k$ and $h^k_{BC}+h^k_{A}=2\,b_k$ for every $k\in\N$, then $h^k_{BC}=b_k$ for every $k\in\N$}\\
If $n$ is the complex dimension of $X$, then, for every $k\in\N$, we have
$$ b_k \;\leq\; h^k_{BC} \;=\; h^{2n-k}_{A} \;=\; 2\,b_{2n-k}-h^{2n-k}_{BC} \;\leq\; b_{2n-k} \;=\; b_k$$
and hence $h^k_{BC}=b_k$ for every $k\in\N$.

\smallskip

Now, by \texttt{Claim 1}, we get that $a^k=0$ for each $k\in\N$; hence, by \texttt{Claim 2}, for every $k\in\N$ the natural map
$$ \bigoplus_{p+q=k}H^{p,q}_{BC}(X)\to H^k_{dR}(X;\C) $$
induced by the identity is surjective, and hence, in particular, $h^k_{BC}\geq b_k$. By \texttt{Claim 3} we get therefore that $h^k_{BC}=b_k$ for every $k\in\N$. Hence, the natural map $H^{\bullet,\bullet}_{BC}(X)\to H^{\bullet}_{dR}(X;\C)$ is actually an isomorphism, which is equivalent to say that $X$ satisfies the $\del\delbar$-Lemma.
\end{proof}

\begin{rem}
We note that, using the exact sequences \eqref{eq:succesatta-2} and \eqref{eq:succesatta-1}, one can prove that, on a compact complex manifold $X$ and for every $k\in\N$,
\begin{eqnarray*}
e^k &=& \left(h^{k}_{\delbar}-h^{k}_{BC}\right)+f^k+c^{k-1} \\[5pt]
&=& \left(h^{k}_{\delbar}-h^{k}_{BC}\right)-\left(h^{k-1}_{\delbar}-h^{k-1}_{A}\right) + f^k-a^{k-1}+e^{k-2} \;.
\end{eqnarray*}
\end{rem}

\begin{rem}
 Note that ${E_1} \simeq {E_\infty}$ is not sufficient to have the equality $h^k_{BC}+h^k_{A}=2\,b_k$ for every $k\in\N$ (and hence the $\del\delbar$-Lemma): a counterexample is provided by small deformations of the Iwasawa manifold.

 Indeed, for small deformations of $\mathbb{I}_3$ in class {\itshape (iii)}, since
 $$ h^1_\delbar \;=\; 4 \;=\; b_1 \;, \qquad h^2_\delbar \;=\; 8 \;=\; b_2 \;, h^3_\delbar \;=\; 10 \;=\; b_3 \;, $$
 the Hodge and Fr\"olicher spectral sequences degenerate at the first step, but
 $$ h^1_{BC}+h^1_{A} \;=\; 10 \;>\; 8 \;=\; 2\,b_1 \;, \qquad h^2_{BC}+h^2_{A} \;=\; 16 \;=\; 2\,b_2 \;, \qquad h^3_{BC}+h^3_{A} \;=\; 28 \;>\; 20 \;=\; 2\,b_3 \;. $$
\end{rem}

\medskip

Using Theorem \ref{thm:caratterizzazione-bc-numbers}, we get another proof of the stability of the $\del\delbar$-Lemma under small deformations of the complex structure, \cite[Corollary 2.7]{angella-tomassini-3}; for different proofs of the same result by means of other techniques see, e.g., \cite[Proposition 9.21]{voisin}, \cite[Theorem 5.12]{wu}, \cite[\S B]{tomasiello}.

\begin{cor}
\label{cor:stab-del-delbar-lemma}
 Satisfying the $\del\delbar$-Lemma is a stable property under small deformations of the complex structure.
\end{cor}

\begin{proof}
 Let $\left\{X_t\right\}_{t\in B}$ be a complex-analytic family of compact complex manifolds. Since, for every $k\in\N$, the dimensions $h^{k}_{BC}(X_t)$ and $h^{k}_{A}(X_t)$ are upper-semi-continuous functions in $t$, \cite[Lemme 3.2]{schweitzer}, while the dimensions $b_k(X_t)$ are constant in $t$ by Ehresmann's theorem, one gets that, if $X_{t_0}$ satisfies the equality $h^k_{BC}\left(X_{t_0}\right) + h^k_{A}\left(X_{t_0}\right) = 2\,b_k\left(X_{t_0}\right)$ for every $k\in\N$, the same holds true for $X_t$ with $t$ near $t_0$.
\end{proof}

\medskip

We recall that \cite[5.21]{deligne-griffiths-morgan-sullivan} by P. Deligne, Ph.~A. Griffiths, J. Morgan, and D.~P. Sullivan characterizes  the validity of the $\del\delbar$-Lemma on a compact complex manifold in terms of the degeneracy of the Hodge and Fr\"olicher spectral sequence and of the existence of Hodge structures in cohomology.
In particular, if follows that, on a compact complex manifold satisfying the $\del\delbar$-Lemma, one has the equality $b_k=\sum_{p+q=k} h^{p,q}_\delbar$ for every $k\in\N$ (which is equivalent to the degeneracy of the Hodge and Fr\"olicher spectral sequence) and the symmetry $h^{p,q}_\delbar=h^{q,p}_\delbar$ for every $p,q\in\N$.

Note that, on a compact complex surface $X$, since the Hodge and Fr\"olicher spectral sequence degenerates at the first step (see, e.g., \cite[Theorem IV.2.8]{barth-hulek-peters-vandeven}) if $h^{1,0}_\delbar=h^{0,1}_\delbar$ then $b_1=2\,h^{1,0}_\delbar$ is even, and hence $X$ is K\"ahler, by \cite{kodaira-structure-I, miyaoka, siu}, or \cite[Corollaire 5.7]{lamari}, or \cite[Theorem 11]{buchdahl}.
As already remarked, the small deformations of $\mathbb{I}_3$ in class {\itshape (iii)} satisfy the degeneracy condition of the Hodge and Fr\"olicher spectral sequence, but they do not satisfy either the $\del\delbar$-Lemma, or the symmetry of the Hodge numbers.

It could hence be interesting to construct a compact complex manifold (of any complex dimension greater than or equal to $3$) such that ${E_1} \simeq  {E_{\infty}}$ and $h^{p,q}_{\delbar}=h^{p,q}_{\del}$ for every $p,q\in\N$ but for which the $\del\delbar$-Lemma does not hold.
A compact complex manifold $X$ whose double complex $\left(\wedge^{\bullet,\bullet}X,\,\del,\,\delbar\right)$ has the form in Figure \ref{fig:conj} (where dots denote generators of the $\mathcal{C}^\infty(X;\R)$-module $\wedge^{\bullet,\bullet}X$, horizontal arrows are meant as $\del$, vertical ones as $\delbar$ and zero arrows are not depicted) provides such an example.

\smallskip
\begin{figure}[ht]
 \centering
 \includegraphics[width=6cm,natwidth=800,natheight=600]{conj.eps}
 \caption{An abstract example}
 \label{fig:conj}
\end{figure}
\smallskip

\begin{rem}
 L. Ugarte informed us that M. Ceballos, A. Otal, he himself, and R. Villacampa have found such an example among the $6$-dimensional nilmanifolds endowed with left-invariant complex structures: they provided a complete classification, up to equivalence, of the linear integrable complex structures on $6$-dimensional nilpotent Lie algebras in \cite{ceballos-otal-ugarte-villacampa}, where they also studied some applications of their classification.
\end{rem}

\section{Cohomology computations for special nilmanifolds}\label{sec:computations-nilmfds}

We are now interested in studying the Bott-Chern and Aeppli cohomologies in the special case of left-invariant complex structures on nilmanifolds and solvmanifolds.

In this section, we firstly recall some results concerning the computation of the de Rham cohomology and of the Dolbeault cohomology, for nilmanifolds and solvmanifolds, endowed with left-invariant complex structures, \S\ref{subsec:cohomology-computation-derham-dolbeault}, referring to \cite{nomizu, hattori}, respectively \cite{sakane, cordero-fernandez-gray-ugarte, console-fino, rollenske, rollenske-survey}; then, we state and prove the results obtained in \cite{angella} about the computation of the Bott-Chern and Aeppli cohomologies, Theorem \ref{thm:bc-invariant}, Theorem \ref{thm:bc-invariant-open}. Using these tools, one can compute the de Rham, Dolbeault, Bott-Chern and Aeppli cohomologies for the Iwasawa manifold and for its small deformations, \S\ref{sec:derham-iwasawa}, \S\ref{sec:dolbeault-iwasawa}, \S\ref{sec:bott-chern-iwasawa}.

\subsection{Left-invariant complex structures on solvmanifolds}

We start by recalling some facts and notations concerning left-invariant complex structures on solvmanifolds.

\medskip

Let $X=\left. \Gamma \right\backslash G$ be a solvmanifold, that is, a compact quotient of a connected simply-connected solvable Lie group $G$ by a discrete and co-compact subgroup $\Gamma$; the Lie algebra naturally associated to $G$ will be denoted by $\mathfrak{g}$ and its complexification by $\mathfrak{g}_\C:=\mathfrak{g}\otimes_\R\C$. We recall that, dealing with $G$-left-invariant objects on $X$, we mean objects on $X$ obtained by objects on $G$ that are invariant under the action of $G$ on itself given by left-translations; note that $G$-left-invariant objects on $X$ are uniquely determined by objects on $\mathfrak{g}$. In particular, a $G$-left-invariant complex structure $J$ on $X$ is uniquely determined by a linear complex structure $J$ on $\mathfrak{g}$ satisfying the integrability condition $\Nij_J=0$, \cite[Theorem 1.1]{newlander-nirenberg}; the set of $G$-left-invariant complex structures on $X$ is denoted by
$$ \mathcal{C}\left(\mathfrak{g}\right) := \left\{ J\in\End\left(\mathfrak{g}\right) \st J^2=-\id_{\mathfrak{g}} \;\text{ and }\;\Nij_J=0 \right\} \;. $$

\medskip

Recall that the exterior differential $\de$ on $X$ can be written using only the action of $\Gamma(X;\,TX)$ on $\mathcal{C}^\infty(X)$ and the Lie bracket of the Lie algebra of vector fields on $X$: more precisely, recall that, if $\varphi\in\wedge^k X$ and $X_0,\ldots,X_k\in\mathcal{C}^\infty\left(X;TX\right)$, then
\begin{eqnarray*}
\lefteqn{\de\varphi\left(X_0,\ldots,X_k\right) \;=\; \sum_{j=0}^k \left(-1\right)^j\, X_j\,\varphi\left(X_0,\ldots,X_{j-1},X_{j+1},\ldots, X_k\right)}\\[5pt]
 && + \sum_{0\leq j<h\leq k} \left(-1\right)^{j+h-1}\, \varphi\left(\left[X_j,X_h\right],X_0,\ldots,X_{j-1},X_{j+1},\ldots,X_{h-1},X_{h+1},\ldots,X_k\right) \;.
\end{eqnarray*}
Hence one has a differential complex $\left(\wedge^\bullet\mathfrak{g}^*,\,\de\right)$, which is isomorphic, as a differential complex, to the differential subcomplex $\left(\wedge^\bullet_{\text{inv}}X,\,\de\lfloor_{\wedge^\bullet_{\text{inv}}X}\right)$ of $\left(\wedge^\bullet X,\,\de\right)$ given by the $G$-left-invariant forms on $X$.

If a $G$-left-invariant complex structure on $X$ is given, then one also has the double complex $\left(\wedge^{\bullet,\bullet}\mathfrak{g}_{\C}^*,\,\del,\,\delbar\right)$, which is isomorphic, as a double complex, to the double subcomplex $\left(\wedge^{\bullet,\bullet}_{\text{inv}}X,\,\del\lfloor_{\wedge^{\bullet,\bullet}_{\text{inv}}X},\,\delbar\lfloor_{\wedge^{\bullet,\bullet}_{\text{inv}}X}\right)$ of $\left(\wedge^{\bullet,\bullet}X,\,\del,\,\delbar\right)$ given by the $G$-left-invariant forms on $X$.

Finally, given a $G$-left-invariant complex structure on $G$ and fixed $p,q\in\N$, one also has the following complexes and the following maps of complexes:
\begin{equation}\label{eq:bc-complessi}
\begin{gathered}
\xymatrix{
\wedge^{p-1,q-1}\mathfrak{g}_\C^* \ar[r]^{\del\delbar} \ar[d]^{\simeq} & \wedge^{p,q}\mathfrak{g}_\C^* \ar[r]^{\de} \ar[d]^{\simeq} & \wedge^{p+q+1}\mathfrak{g}_\C^* \ar[d]^{\simeq} \\
\wedge^{p-1,q-1}_{\text{inv}}X \ar[r]^{\del\delbar} \ar@{^{(}->}[d]^{i} & \wedge^{p,q}_{\text{inv}}X \ar[r]^{\de} \ar@{^{(}->}[d]^{i} & \wedge^{p+q+1}_{\text{inv}}(X;\C) \ar@{^{(}->}[d]^{i} \\
\wedge^{p-1,q-1}X \ar[r]^{\del\delbar} & \wedge^{p,q}X \ar[r]^{\de} & \wedge^{p+q+1}(X;\C)
}
\end{gathered}
\;,
\end{equation}
and
\begin{equation}\label{eq:a-complessi}
\begin{gathered}
\xymatrix{
\wedge^{p-1,q}\mathfrak{g}_\C^* \oplus \wedge^{p,q-1}\mathfrak{g}_\C^* \ar[r]^{\qquad\del+\delbar} \ar[d]^{\simeq} & \wedge^{p,q}\mathfrak{g}_\C^* \ar[r]^{\del\delbar} \ar[d]^{\simeq} & \wedge^{p+1,q+1}\mathfrak{g}_\C^* \ar[d]^{\simeq} \\
\wedge^{p-1,q}_{\text{inv}}X \oplus \wedge^{p,q-1}_{\text{inv}}X \ar[r]^{\qquad\del+\delbar} \ar@{^{(}->}[d]^{i} & \wedge^{p,q}_{\text{inv}}X \ar[r]^{\del\delbar} \ar@{^{(}->}[d]^{i} & \wedge^{p+1,q+1}_{\text{inv}}X \ar@{^{(}->}[d]^{i} \\
\wedge^{p-1,q}X \oplus \wedge^{p,q-1}X \ar[r]^{\qquad \del+\delbar} & \wedge^{p,q}X \ar[r]^{\del\delbar} & \wedge^{p+1,q+1}X
}
\end{gathered}
\;.
\end{equation}

For $\sharp\in\left\{\delbar,\,\del,\,BC,\,A\right\}$ and $\mathbb{K}\in\left\{\R,\C\right\}$, we will write $H_{dR}^\bullet\left(\mathfrak{g};\mathbb{K}\right):=:H^\bullet_{dR}\left(\mathfrak{g};\mathbb{K}\right)$ and $H^{\bullet,\bullet}_{\sharp}\left(\mathfrak{g}_\C\right)$ to denote the cohomology groups of the corresponding complexes of forms on $\mathfrak{g}$, which are isomorphic to the cohomology groups of the corresponding complexes of $G$-left-invariant forms on $X$. The rest of this section is devoted to the problem whether these cohomologies are isomorphic to the corresponding cohomologies on $X$.

\subsection{Classical results on computations of the de Rham and Dolbeault cohomologies}\label{subsec:cohomology-computation-derham-dolbeault}

In this section, we collect some results, by K. Nomizu \cite{nomizu}, A. Hattori \cite{hattori}, S. Console and A. Fino \cite{console-fino}, Y. Sakane \cite{sakane}, L.~A. Cordero, M. Fern\'{a}ndez, A. Gray, and L. Ugarte \cite{cordero-fernandez-gray-ugarte}, S. Rollenske \cite{rollenske, rollenske-survey, rollenske-jlms}, concerning the computation of the de Rham cohomology and the Dolbeault cohomology for nilmanifolds and solvmanifolds, endowed with left-invariant complex structures.

\medskip

First of all, we recall the following result, concerning the de Rham cohomology: it was firstly proven by K. Nomizu for nilmanifolds, and then generalized by A. Hattori to the case of completely-solvable solvmanifolds.

\begin{thm}[{\cite[Theorem 1]{nomizu}, \cite[Corollary 4.2]{hattori}}]
 Let $X=\left. \Gamma \right\backslash G$ be a nilmanifold, or, more in general, a completely-solvable solvmanifold, and denote the Lie algebra naturally associated to $G$ by $\mathfrak{g}$. The map of differential complexes $\left(\wedge^\bullet\mathfrak{g}^*,\,\de\right)\to\left(\wedge^\bullet X,\,\de\right)$ is a quasi-isomorphism:
$$ i\colon H^{\bullet}_{dR}(\mathfrak{g};\R) \stackrel{\simeq}{\to} H^\bullet_{dR}(X;\R) \;.$$
\end{thm}

A counterexample in the non-completely-solvable case was provided by P. de Bartolomeis and A. Tomassini in \cite[Corollary 4.2, Remark 4.3]{debartolomeis-tomassini}, studying the Nakamura manifold, \cite[\S2]{nakamura}.

\medskip

Similar results hold for the Dolbeault cohomology of nilmanifolds endowed with certain left-invariant complex structures; \cite{console} and \cite{rollenske-survey} are recent surveys on the known results. (Some results about the Dolbeault cohomology of solvmanifolds have been recently proven by H. Kasuya, see \cite{kasuya}.)

First of all, we recall the following lemma by S. Console and A. Fino, \cite{console-fino}: the argument used in the proof can be generalized to Bott-Chern and Aeppli cohomologies, see Lemma \ref{lemma:inj}.
\begin{lem}[{\cite[Lemma 9]{console-fino}}]
Let $X=\left.\Gamma\right\backslash G$ be a nilmanifold endowed with a $G$-left-invariant complex structure $J$, and denote the Lie algebra naturally associated to $G$ by $\mathfrak{g}$. For any $p\in\N$, the map of complexes $\left(\wedge^{p,\bullet}\mathfrak{g}_{\C}^*,\,\delbar\right) 
\hookrightarrow\left(\wedge^{p,\bullet}X,\,\delbar\right)$ induces an injective homomorphism $i$ in cohomology:
$$ i\colon H^{\bullet,\bullet}_{\delbar}\left(\mathfrak{g}_\C\right) \hookrightarrow H^{\bullet,\bullet}_{\delbar}(X) \;.$$
\end{lem}

For an arbitrary $G$-left-invariant complex structure on a nilmanifold $X=\left. \Gamma \right\backslash G$, it is not known whether $i\colon H^{\bullet,\bullet}_{\delbar}\left(\mathfrak{g}_\C\right) \hookrightarrow H^{\bullet,\bullet}_{\delbar}(X)$ actually is an isomorphism, but some results are known for certain classes of $G$-left-invariant complex structures.

\begin{thm}[{\cite[Theorem 1]{sakane}, \cite[Main Theorem]{cordero-fernandez-gray-ugarte}, \cite[Theorem 2, Remark 4]{console-fino}, \cite[Theorem 1.10]{rollenske}, \cite[Corollary 3.10]{rollenske-survey}}]
Let $X=\left. \Gamma \right\backslash G$ be a nilmanifold endowed with a $G$-left-invariant complex structure $J$, and denote the Lie algebra naturally associated to $G$ by $\mathfrak{g}$.
Then, for every $p\in\N$, the map of complexes
\begin{equation}\label{eq:dolb-complessi}
\left(\wedge^{p,\bullet}\mathfrak{g}_{\C}^*,\,\delbar\right) 
\hookrightarrow\left(\wedge^{p,\bullet}X,\,\delbar\right)
\end{equation}
is a quasi-isomorphism, namely,
$$
i \colon H^{\bullet,\bullet}_{\delbar}\left(\mathfrak{g}_\C\right) \stackrel{\simeq}{\to} H^{\bullet,\bullet}_{\delbar}(X) \;,$$
provided one of the following conditions holds:
\begin{itemize}
 \item $X$ is holomorphically parallelizable;
 \item $J$ is an Abelian complex structure;
 \item $J$ is a nilpotent complex structure;
 \item $J$ is a rational complex structure;
 \item $\mathfrak{g}$ admits a torus-bundle series compatible with $J$ and with the rational structure induced by $\Gamma$;
 \item $\dim_\R\mathfrak{g}=6$ and $\mathfrak{g}$ is not isomorphic to $\mathfrak{h}_7:=\left(0^3,\, 12,\, 13,\, 23\right)$.
\end{itemize}
\end{thm}

We recall, (see, e.g., \cite[Definition 1.5]{rollenske},) that, given a nilpotent Lie algebra $\mathfrak{g}$, a \emph{rational structure} for $\mathfrak{g}$ is a $\Q$-vector space $\mathfrak{g}_\Q$ such that $\mathfrak{g}_\Q \otimes_\Q \R = \mathfrak{g}$. A sub-algebra $\mathfrak{h}$ of $\mathfrak{g}$ is said to be \emph{rational with respect to a rational structure $\mathfrak{g}_\Q$} if the $\Q$-vector space
$\mathfrak{h} \cap \mathfrak{g}_\Q$ of $\mathfrak{h}$ is a rational structure for $\mathfrak{h}$.
If $G$ is the connected simply-connected Lie group associated to $\mathfrak{g}$, then any discrete co-compact subgroup $\Gamma$ of $G$ induces a rational structure for $\mathfrak{g}$, given by $\Q \log \Gamma$.

Consider a $G$-left-invariant complex structure on a nilmanifold $X=\left. \Gamma \right\backslash G$  with associated Lie algebra $\mathfrak{g}$; we recall that:
\begin{itemize}
 \item $J$ is called \emph{holomorphically parallelizable} if the the holomorphic tangent bundle is holomorphically trivial, see, e.g., \cite{wang, nakamura};
 \item $J$ is called \emph{Abelian} if $\left[Jx,\,Jy\right]=\left[x,\,y\right]$ for any $x,y\in\mathfrak{g}$, see, e.g., \cite{barberis-dotti-miatello, andrada-barberis-dotti};
 \item $J$ is called \emph{nilpotent} if there exists a $G$-left-invariant co-frame $\left\{\omega^1,\ldots,\omega^n\right\}$ for $\left(T^{1,0}X\right)^*$ with respect to which the structure equations of $X$ are of the form
$$ \de\omega^j\;=\; \sum_{h<k<j}A_{hk}^j\,\omega^h\wedge\omega^k+\sum_{h,k<j}B_{h k}^j\,\omega^h\wedge\bar\omega^k $$
with $\left\{A_{hk}^j,\,B_{h k}^j\right\}_{j,h,k}\subset\C$, see, e.g, \cite{cordero-fernandez-gray-ugarte};
 \item $J$ is called \emph{rational} if $J\left(\mathfrak{g}_\Q\right)\subseteq \mathfrak{g}_\Q$ where $\mathfrak{g}_\Q$ is the rational structure for $\mathfrak{g}$ induced by $\Gamma$, see, e.g., \cite{console-fino}.
\end{itemize}
 
 We recall also the following definitions, \cite[Definition 1.8]{rollenske}. An ascending filtration $\left\{\mathcal{S}^j\mathfrak{g}\right\}_{j\in\{0,\ldots,k\}}$ on $\mathfrak{g}$ is called a \emph{torus-bundle series compatible with a linear complex structure $J$ on $\mathfrak{g}$ and a rational structure $\mathfrak{g}_\Q$ for $\mathfrak{g}$} if, for every $j\in\{1,\ldots,k\}$, it holds that
\begin{inparaenum}[\itshape (i)]
 \item $\mathcal{S}^j\mathfrak{g}$ is rational with respect to $\mathfrak{g}_\Q$ and an ideal in $\mathcal{S}^{j+1}\mathfrak{g}$,
 \item $J\mathcal{S}^j\mathfrak{g}=\mathcal{S}^j\mathfrak{g}$, and
 \item $\left. \mathcal{S}^{j+1}\mathfrak{g} \right\slash \mathcal{S}^j\mathfrak{g}$ is Abelian.
\end{inparaenum}
If, in addition, it holds that
\begin{inparaenum}[\itshape (iv)]
 \item $\left. \mathcal{S}^{j+1}\mathfrak{g} \right\slash \mathcal{S}^j\mathfrak{g}$ is contained in the center of $\left. \mathfrak{g} \right\slash \mathcal{S}^j\mathfrak{g}$,
\end{inparaenum}
then $\left\{\mathcal{S}^j\mathfrak{g}\right\}_{j\in\{0,\ldots,k\}}$ is called a \emph{principal torus-bundle series compatible with $J$ and $\mathfrak{g}_\Q$}. Finally, an ascending filtration $\left\{\mathcal{S}^j\mathfrak{g}\right\}_{j\in\{0,\ldots,k\}}$ on $\mathfrak{g}$ is called a \emph{stable (principal) torus-bundle series} if it is a (principal) torus-bundle series compatible with $J$ and $\mathfrak{g}_\Q$ for any complex structure $J$ and for any rational structure $\mathfrak{g}_\Q$.
By S. Rollenske's theorem \cite[Theorem B]{rollenske}, every $6$-dimensional nilpotent Lie algebra except $\h_7:=\left(0^3,\, 12,\, 13,\, 23\right)$ admits a stable torus-bundle series.

\medskip

The property of computing the Dolbeault cohomology using just left-invariant forms turns out to be open along curves of left-invariant complex structures: this was proven by S. Console and A. Fino, \cite{console-fino}.

\begin{thm}[{\cite[Theorem 1]{console-fino}}]\label{thm:dolbeault-invariant-open}
Let $X=\left.\Gamma\right\backslash G$ be a nilmanifold endowed with a $G$-left-invariant complex structure $J$, and denote the Lie algebra naturally associated to $G$ by $\mathfrak{g}$.
Let $\mathcal{U}\subseteq\mathcal{C}(\mathfrak{g})$ be the subset containing the $G$-left-invariant complex structures $J$ on $X$ such that the inclusion $i$ is an isomorphism:
$$ \mathcal{U} \;:=\; \left\{ J\in\mathcal{C}\left(\mathfrak{g}\right) \st i\colon H^{\bullet,\bullet}_{\delbar}\left(\mathfrak{g}_\C\right)\stackrel{\simeq}{\hookrightarrow}H^{\bullet,\bullet}_{\delbar}(X)\right\} \;\subseteq\; \mathcal{C}\left(\mathfrak{g}\right) \;. $$
Then $\mathcal{U}$ is an open set in $\mathcal{C}\left(\mathfrak{g}\right)$.
\end{thm}

The strategy of the proof consists in proving that the dimension of the orthogonal of $H^{\bullet,\bullet}_{\delbar}\left(\mathfrak{g}_\C\right)$ in $H^{\bullet,\bullet}_{\delbar}(X)$ with respect to a given $J$-Hermitian $G$-left-invariant metric on $X=\left.\Gamma\right\backslash G$ is an upper-semi-continuous function in $J\in\mathcal{C}\left(\mathfrak{g}\right)$ and thus, if it is zero for a given $J\in\mathcal{C}\left(\mathfrak{g}\right)$, then it remains equal to zero in an open neighbourhood of $J$ in $\mathcal{C}\left(\mathfrak{g}\right)$.
We will use the same argument in proving Theorem \ref{thm:bc-invariant-open}, which is a slight modification of the previous result in the case of the Bott-Chern cohomology.

\medskip

The aforementioned results suggest the following conjecture.

\begin{conj}[{\cite[Conjecture 1]{rollenske-survey}; see also \cite[page 5406]{cordero-fernandez-gray-ugarte}, \cite[page 112]{console-fino}}]\label{conj:dolbeault}
 Let $X=\left. \Gamma \right\backslash G$ be a nilmanifold endowed with a $G$-left-invariant complex structure $J$, and denote the Lie algebra naturally associated to $G$ by $\mathfrak{g}$.
Then, for any $p\in\N$, the map of complexes \eqref{eq:dolb-complessi} is a quasi-isomorphisms, that is,
$$ i\colon H^{\bullet,\bullet}_{\delbar}\left(\mathfrak{g}_\C\right) \stackrel{\simeq}{\to} H^{\bullet,\bullet}_{\delbar}(X) \;.$$
\end{conj}

Note that, since $i$ is always injective by \cite[Lemma 9]{console-fino}, this is equivalent to asking that
$$ \dim_\C \left(H^{\bullet,\bullet}_{\delbar}\left(\mathfrak{g}_\C\right)\right)^\perp=0 \;,$$
where the orthogonality is meant with respect to the inner product induced by a given $J$-Hermitian $G$-left-invariant metric $g$ on $X$.

\medskip

Finally, as an application of the previous results, we recall the following theorem by S. Rollenske, concerning the deformations of left-invariant complex structures on nilmanifolds.

\begin{thm}[{\cite[Theorem 2.6]{rollenske-jlms}}]
 Let $X=\left.\Gamma\right\backslash G$ be a nilmanifold endowed with a $G$-left-invariant complex structure $J$, and denote the Lie algebra naturally associated to $G$ by $\mathfrak{g}$. Suppose that, for $p=1$, the map of complexes \eqref{eq:dolb-complessi} is a quasi-isomorphism: $i\colon H^{1,q}_\delbar\left(\mathfrak{g}_\C\right) \stackrel{\simeq}{\to} H^{1,q}_\delbar(X)$ for every $q\in\N$. Then all small deformations of the complex structure $J$ are again $G$-left-invariant complex structures. More precisely, the Kuranishi family of $X$ contains  only $G$-left-invariant complex structures.
\end{thm}

\subsection{The Bott-Chern cohomology on solvmanifolds}\label{subsec:cohomology-computation-bott-chern}
We recall here the results obtained in \cite{angella}, concerning the computation of the Bott-Chern cohomology for nilmanifolds and solvmanifolds.

\medskip

Firstly, we prove a slight modification of \cite[Lemma 9]{console-fino} proven by S. Console and A. Fino for the Dolbeault cohomology: we repeat here their argument for the case of the Bott-Chern cohomology, \cite[Lemma 3.6]{angella}.

\begin{lem}
\label{lemma:inj}
 Let $X=\left. \Gamma \right\backslash G$ be a solvmanifold endowed with a $G$-left-invariant complex structure $J$, and denote the Lie algebra naturally associated to $G$ by $\mathfrak{g}$.
 The map of complexes \eqref{eq:bc-complessi} induces an injective homomorphism
$$ i\colon H^{\bullet,\bullet}_{BC}\left(\mathfrak{g}_\C\right) \hookrightarrow H^{\bullet,\bullet}_{BC}(X) \;.$$
\end{lem}

\begin{proof}
 Fix $p,q\in\N$. Let $g$ be a $J$-Hermitian $G$-left-invariant metric on $X$ and consider the induced inner product $\left\langle\left.\sspace\right|\ssspace\right\rangle$ on $\wedge^{\bullet,\bullet} X$. Hence, both $\del$, $\delbar$, and their adjoints $\del^*$, $\delbar^*$ preserve the $G$-left-invariant forms on $X$ and therefore also $\tilde\Delta_{BC}$ does. In such a way, we get a Hodge decomposition also at the level of $G$-left-invariant forms:
 $$
 \wedge^{p,q}\mathfrak{g}_\C^* = \ker\tilde\Delta_{BC}\lfloor_{\wedge^{p,q}\mathfrak{g}_\C^*} \,\oplus\, \imm\del\delbar\lfloor_{\wedge^{p-1,q-1}\mathfrak{g}_\C^*} \oplus\, \left(\imm\del^*\lfloor_{\wedge^{p+1,q}\mathfrak{g}_\C^*} \,+\, \imm\delbar^*\lfloor_{\wedge^{p,q+1}\mathfrak{g}_\C^*}\right) \;.
 $$
 Now, take $[\omega]\in H^{p,q}_{BC}\left(\mathfrak{g}_\C\right)$ such that $i\left[\omega\right]=0$ in $H^{p,q}_{BC}(X)$, that is, $\omega$ is a $G$-left-invariant $(p,q)$-form on $X$ and there exists a (possibly non-$G$-left-invariant) $(p-1,q-1)$-form $\eta$ on $X$ such that $\omega=\del\delbar\,\eta$. Up to zero terms in $H^{p,q}_{BC}\left(\mathfrak{g}_\C\right)$, we may assume that $\eta\in \left(i\left(\wedge^{p,q}\mathfrak{g}_\C^*\right)\right)^\perp\subseteq \wedge^{p,q}X$. Therefore, since $\delbar^*\del^*\del\delbar\eta$ is a $G$-left-invariant form (being $\del\delbar\eta$ a $G$-left-invariant form), we have that
$$ 0=\left\langle \left. \delbar^*\del^*\del\delbar\eta \right| \eta\right\rangle = \left\|\del\delbar\eta\right\|^2=\left\|\omega\right\|^2 $$
and therefore $\omega=0$.
\end{proof}

The second general result says that, if the Dolbeault and de Rham cohomologies of a solvmanifold are computed using just left-invariant forms, then also the Bott-Chern cohomology is computed using just left-invariant forms, \cite[Theorem 3.7]{angella}. The idea of the proof is inspired by \cite[\S1.c]{schweitzer}, where M. Schweitzer used a similar argument to explicitly compute the Bott-Chern cohomology in the special case of the Iwasawa manifold.

\begin{thm}
\label{thm:bc-invariant}
 Let $X=\left.\Gamma\right\backslash G$ be a solvmanifold endowed with a $G$-left-invariant complex structure $J$, and denote the Lie algebra naturally associated to $G$ by $\mathfrak{g}$. Suppose that
$$ i\colon H^\bullet_{dR}(\mathfrak{g};\C)\stackrel{\simeq}{\hookrightarrow} H^\bullet_{dR}(X;\C) \qquad \text{ and } \qquad i\colon H^{\bullet,\bullet}_{\delbar}\left(\mathfrak{g}_\C\right)\stackrel{\simeq}{\hookrightarrow} H^{\bullet,\bullet}_{\delbar}(X) \;.$$
Then also
$$ i\colon H^{\bullet,\bullet}_{BC}\left(\mathfrak{g}_\C\right) \stackrel{\simeq}{\hookrightarrow}H^{\bullet,\bullet}_{BC}(X) \;. $$
\end{thm}

\begin{proof}
 Fix $p,q\in\N$. We prove the theorem as a consequence of the following claims.

 \paragrafodclaim{1}{\itshape It suffices to prove that $\frac{\imm\de\cap\wedge^{p,q}X}{\imm\del\delbar}$ can be computed using just $G$-left-invariant forms}\\
Indeed, we have the exact sequence
$$ 0 \to \frac{\imm\de\cap\wedge^{p,q}X}{\imm\del\delbar} \to H^{p,q}_{BC}(X) \to H^{p+q}_{dR}(X;\C) $$
and, by hypothesis, $H^\bullet_{dR}(X;\C)$ can be computed using just $G$-left-invariant forms.

 \paragrafodclaim{2}{\itshape Under the hypothesis that the Dolbeault cohomology is computed using just $G$-left-invariant forms, if $\psi$ is a $G$-left-invariant $\delbar$-closed form, then every solution $\phi$ of $\delbar\phi=\psi$ is $G$-left-invariant up to $\delbar$-exact terms}\\
Indeed, since $[\psi]=0$ in $H^{\bullet,\bullet}_{\delbar}(X)$, there is a $G$-left-invariant form $\alpha$ such that $\psi=\delbar\alpha$. Hence, $\phi-\alpha$ defines a class in $H^{\bullet,\bullet}_{\delbar}(X)$ and hence $\phi-\alpha$ is $G$-left-invariant up to a $\delbar$-exact form, and so $\phi$ is.

 \paragrafodclaim{3}{\itshape Under the hypothesis that the Dolbeault cohomology is computed using just $G$-left-invariant forms, the space $\frac{\imm\de\cap\wedge^{p,q}X}{\imm\del\delbar}$ can be computed using just $G$-left-invariant forms}\\
Consider
\begin{equation}\label{eq:proof-step-3}
\omega^{p,q}\;=\;\de\eta\mod\imm\del\delbar\;\in\;\frac{\imm\de\cap\wedge^{p,q}X}{\imm\del\delbar} \;.
\end{equation}
Decomposing $\eta=:\sum_{p,q}\eta^{p,q}$ in pure-type components, the equality \eqref{eq:proof-step-3} is equivalent to the system
$$
\left\{
\begin{array}{cccccccc}
 && \del\eta^{p+q-1,0} &=&0 & \mod \imm\del\delbar && \\[5pt]
\delbar\eta^{p+q-\ell,\ell-1} &+& \del\eta^{p+q-\ell-1,\ell} &=& 0 &\mod\imm\del\delbar & \text{ for } & \ell\in\{1,\ldots,q-1\} \\[5pt]
\delbar\eta^{p,q-1} &+& \del\eta^{p-1,q} &=& \omega^{p,q} & \mod\imm\del\delbar && \\[5pt]
\delbar\eta^{\ell,p+q-\ell-1} &+& \del\eta^{\ell-1,p+q-\ell} &=& 0 &\mod\imm\del\delbar & \text{ for } & \ell\in\{1,\ldots,p-1\} \\[5pt]
\delbar\eta^{0,p+q-1} &&&=&0 &\mod\imm\del\delbar &&
\end{array}
\right. \;.
$$
Applying several times \texttt{Claim 2}, we may suppose that the forms $\eta^{\ell,p+q-\ell-1}$, with $\ell\in\{0,\ldots, p-1\}$, are $G$-left-invariant: indeed, they are $G$-left-invariant up to $\delbar$-exact terms, but $\delbar$-exact terms give no contribution in the system, since it is modulo $\imm\del\delbar$. Analogously, using the conjugate version of \texttt{Claim 2}, we may suppose that the forms $\eta^{p+q-\ell-1,\ell}$, with $\ell\in\{0,\ldots,q-1\}$, are $G$-left-invariant. Then we may suppose that $\omega^{p,q}=\delbar\eta^{p,q-1}+\del\eta^{p-1,q}$ is $G$-left-invariant.
\end{proof}

\begin{rem}
 Let $X=\left.\Gamma\right\backslash G$ be a solvmanifold endowed with a $G$-left-invariant complex structure $J$, and denote the Lie algebra naturally associated to $G$ by $\mathfrak{g}$. Note that, if the map of complexes $i\colon \left(\wedge^{p,\bullet}\duale{\mathfrak{g}}_\C,\, \delbar\right)\to \left(\wedge^{p,\bullet}X,\, \delbar\right)$ is a quasi-isomorphism for every $p\in\N$, that is,
 $$ i\colon H^{\bullet,\bullet}_{\delbar}\left(\mathfrak{g}_\C\right)\stackrel{\simeq}{\hookrightarrow} H^{\bullet,\bullet}_{\delbar}(X) \;,$$
 then also the map of complexes $i\colon \left(\wedge^\bullet\duale{\mathfrak{g}},\, \de\right) \to \left(\wedge^\bullet X,\, \de\right)$ is a quasi-isomorphism, that is,
 $$ i\colon H^\bullet_{dR}(\mathfrak{g};\C)\stackrel{\simeq}{\hookrightarrow} H^\bullet_{dR}(X;\C) \;.$$

 Indeed, the map of double complexes $i\colon \left(\wedge^{\bullet,\bullet}\duale{\mathfrak{g}}_\C,\, \del,\, \delbar\right)\to \left(\wedge^{\bullet,\bullet}X,\, \del,\, \delbar\right)$ induces a map between the corresponding Hodge and Fr\"olicher spectral sequences:
 $$ i \colon \left\{\left(E^{\bullet,\bullet}_r\left(\mathfrak{g}_\C\right), \de_r\right)\right\}_{r\in\N} \to \left\{\left(E^{\bullet,\bullet}_r\left(X\right), \de_r\right)\right\}_{r\in\N} \;.$$
 Since, see, e.g., \cite[Theorem 2.15]{mccleary},
 $$ E^{\bullet,\bullet}_1\left(\mathfrak{g}_\C\right) \;\simeq\; H^{\bullet, \bullet}_{\delbar}\left(\mathfrak{g}\right) \;\Rightarrow\; H^\bullet_{dR}\left(\mathfrak{g}_\C\right) \qquad \text{ and } \qquad E^{\bullet,\bullet}_1\left(X\right) \;\simeq\; H^{\bullet,\bullet}_{\delbar}(X) \;\Rightarrow\; H^{\bullet}_{dR}(X;\C) \;, $$
 one gets that, if $i\colon E^{\bullet,\bullet}_1\left(\mathfrak{g}_\C\right) \to E^{\bullet,\bullet}_1\left(X\right)$ is an isomorphism, then also $i \colon H^\bullet_{dR}\left(\mathfrak{g}_\C\right) \to H^{\bullet}_{dR}(X;\C)$ is an isomorphism, see, e.g., \cite[Theorem 3.5]{mccleary}.
\end{rem}

As a corollary of \cite[Theorem 1]{nomizu}, \cite[Theorem 1]{sakane}, \cite[Main Theorem]{cordero-fernandez-gray-ugarte}, \cite[Theorem 2, Remark 4]{console-fino}, \cite[Theorem 1.10]{rollenske}, \cite[Corollary 3.10]{rollenske-survey}, and Theorem \ref{thm:bc-invariant}, we get the following result, \cite[Theorem 3.8]{angella}.

\begin{thm}
\label{thm:bc-invariant-cor}
 Let $X=\left.\Gamma\right\backslash G$ be a nilmanifold endowed with a $G$-left-invariant complex structure $J$, and denote the Lie algebra naturally associated to $G$ by $\mathfrak{g}$. Suppose that one of the following conditions holds:
\begin{itemize}
 \item $X$ is holomorphically parallelizable;
 \item $J$ is an Abelian complex structure;
 \item $J$ is a nilpotent complex structure;
 \item $J$ is a rational complex structure;
 \item $\mathfrak{g}$ admits a torus-bundle series compatible with $J$ and with the rational structure induced by $\Gamma$;
 \item $\dim_\R\mathfrak{g}=6$ and $\mathfrak{g}$ is not isomorphic to $\mathfrak{h}_7:=\left(0^3,\, 12,\, 13,\, 23\right)$.
\end{itemize}
Then the de Rham, Dolbeault, Bott-Chern and Aeppli cohomologies can be computed as the cohomologies of the corresponding subcomplexes given by the space of $G$-left-invariant forms on $X$; in other words, the inclusions of the several subcomplexes of $G$-left-invariant forms on $X$ into the corresponding complexes of forms on $X$ are quasi-isomorphisms:
$$ i\colon H^\bullet_{dR}\left(\mathfrak{g};\R\right) \;\stackrel{\simeq}{\hookrightarrow}\; H^\bullet_{dR}(X;\R) \qquad \text{ and }\qquad i\colon H^{\bullet,\bullet}_{\sharp}\left(\mathfrak{g}_\C\right) \;\stackrel{\simeq}{\hookrightarrow}\; H^{\bullet,\bullet}_{\sharp}\left(X\right) \;,$$
for $\sharp\in\{\del,\,\delbar,\,BC,\,A\}$.
\end{thm}

\begin{rem}\label{rem:inv-harmonic}
 Note that Theorem \ref{thm:bc-invariant-cor}, and \cite[Corollary 4.2]{hattori}, allow to straightforwardly compute the de Rham, Dolbeault, Bott-Chern, and Aeppli cohomologies of nilmanifolds, endowed with certain left-invariant complex structures, respectively the de Rham cohomology of completely-solvable solvmanifolds, just by computing the space of left-invariant ($\Delta$, or $\overline\square$, or $\tilde\Delta_{BC}$, or $\tilde \Delta_A$-)harmonic forms with respect to a left-invariant Riemannian, or Hermitian, metric.

 Indeed, suppose that $X$ is a nilmanifold, endowed with a left-invariant complex structure, or a completely-solvable solvmanifold, satisfying $i\colon H^\bullet_{dR}\left(\mathfrak{g};\R\right) \stackrel{\simeq}{\hookrightarrow}\; H^\bullet_{dR}(X;\R)$, or $i\colon H^{\bullet,\bullet}_{\sharp}\left(\mathfrak{g}_\C\right) \;\stackrel{\simeq}{\hookrightarrow}\; H^{\bullet,\bullet}_{\sharp}\left(X\right)$, for some $\sharp\in\{\del,\,\delbar,\,BC,\,A\}$. Let $g$ be a left-invariant Riemannian, or Hermitian, metric on $X$. Hence, the operators $\Delta$, $\overline\square$, $\tilde\Delta_{BC}$, $\tilde\Delta_{A}$ send the subspace of left-invariant forms to the subspace of left-invariant forms, and induce the self-adjoint operators
 $$
   \Delta \;\in\; \End\left(\wedge^\bullet\duale{\mathfrak{g}}\right) \;,
   \qquad
   \overline\square \;\in\; \End\left(\wedge^{\bullet,\bullet} \duale{\mathfrak{g}}_\C\right) \;,
   \qquad
   \tilde\Delta_{BC} \;\in\; \End\left(\wedge^{\bullet,\bullet} \duale{\mathfrak{g}}_\C\right) \;,
   \qquad
   \tilde\Delta_{A} \;\in\; \End\left(\wedge^{\bullet,\bullet} \duale{\mathfrak{g}}_\C\right) \;,$$
with respect to the inner products $\left\langle \sspace,\, \ssspace \right\rangle$ induced by $g$ on the space $\wedge^{\bullet} \duale{\mathfrak{g}}$ and on the space $\wedge^{\bullet,\bullet} \duale{\mathfrak{g}}_\C$. Hence, one gets the orthogonal decompositions
\begin{align*}
  &\wedge^\bullet\duale{\mathfrak{g}} \;=\; \ker\Delta\oplus \imm\Delta \;,
  &\quad
  &\wedge^{\bullet,\bullet}\duale{\mathfrak{g}}_\C \;=\; \ker\overline\square \oplus \imm\overline\square \;,\\[5pt]
  &\wedge^{\bullet,\bullet}\duale{\mathfrak{g}}_\C \;=\; \ker\tilde\Delta_{BC} \oplus \imm\tilde\Delta_{BC} \;,
  &\quad
  &\wedge^{\bullet,\bullet}\duale{\mathfrak{g}}_\C \;=\; \ker\tilde\Delta_{A} \oplus \imm\tilde\Delta_{A}
\end{align*}
(one could argue also by using the F.~A. Belgun symmetrization trick \cite[Theorem 7]{belgun}). It follows that
 $$
  H^\bullet_{dR}\left(\mathfrak{g};\R\right) \;\simeq\; \ker\Delta \;,
  \qquad
  H^{\bullet,\bullet}_{\delbar}\left(\mathfrak{g}_\C\right) \;\simeq\; \ker\overline\square \;,
  \qquad
  H^{\bullet,\bullet}_{BC}\left(\mathfrak{g}_\C\right) \;\simeq\; \ker\tilde\Delta_{BC} \;,
  \qquad
  H^{\bullet,\bullet}_{A}\left(\mathfrak{g}_\C\right) \;\simeq\; \ker\tilde\Delta_{A} \;.
$$
\end{rem}

\begin{rem}\label{rem:aeppli-inv}
 Let $X=\left.\Gamma\right\backslash G$ be a $2n$-dimensional solvmanifold endowed with a $G$-left-invariant complex structure $J$, and denote the Lie algebra naturally associated to $G$ by $\mathfrak{g}$.
 The map of complexes \eqref{eq:a-complessi} induces an injective homomorphism
 $$ i\colon H^{\bullet,\bullet}_{A}\left(\mathfrak{g}_\C\right) \hookrightarrow H^{\bullet,\bullet}_{A}(X) \;.$$
 Furthermore, if $i\colon H^{\bullet,\bullet}_{BC}\left(\mathfrak{g}_\C\right) \stackrel{\simeq}{\hookrightarrow}H^{\bullet,\bullet}_{BC}(X)$,
 then the map of complexes \eqref{eq:a-complessi} is a quasi-isomorphism, that is,
 $$ i\colon H^{\bullet,\bullet}_{A}\left(\mathfrak{g}_\C\right) \stackrel{\simeq}{\hookrightarrow}H^{\bullet,\bullet}_{A}(X) \;. $$

 Indeed, fix a $G$-left-invariant Hermitian metric $g$ on $X$.
 Recall that
 $$ *\colon H^{\bullet_1,\bullet_2}_{A}(X)\stackrel{\simeq}{\to} H^{n-\bullet_2,n-\bullet_1}_{BC}(X) $$
 is an isomorphism, \cite[\S2.c]{schweitzer}.
 Analogously, note that, by Remark \ref{rem:inv-harmonic} and since $g$ is $G$-left-invariant, the map $*\colon \wedge^{\bullet_1,\bullet_2}\duale{\mathfrak{g}}_\C\stackrel{\simeq}{\to} \wedge^{n-\bullet_2,n-\bullet_1}\duale{\mathfrak{g}}_\C$ induces an isomorphism
 $$ *\colon H^{\bullet_1,\bullet_2}_{A}\left(\mathfrak{g}_\C\right)\stackrel{\simeq}{\to} H^{n-\bullet_2,n-\bullet_1}_{BC}\left(\mathfrak{g}_\C\right) \;.$$
 Note also that the diagram
 $$
  \xymatrix{
   H^{\bullet_1,\bullet_2}_{A}\left(\mathfrak{g}_\C\right) \ar[r]^{i} \ar[d]^{\simeq}_{*} & H^{\bullet_1,\bullet_2}_{A}(X) \ar[d]^{*}_{\simeq} \\
   H^{n-\bullet_2,n-\bullet_1}_{BC}\left(\mathfrak{g}_\C\right) \ar[r]_{i} & H^{n-\bullet_2,n-\bullet_1}_{BC}(X)
  }
 $$
 commutes, since $g$ is $G$-left-invariant. Since the map $i\colon H^{n-\bullet_2,n-\bullet_1}_{BC}\left(\mathfrak{g}_\C\right) \hookrightarrow H^{n-\bullet_2,n-\bullet_1}_{BC}(X)$ is injective by Lemma \ref{lemma:inj}, then also the map $H^{\bullet_1,\bullet_2}_{A}\left(\mathfrak{g}_\C\right)\to H^{\bullet_1,\bullet_2}_{A}(X)$ is injective. If $i\colon H^{n-\bullet_2,n-\bullet_1}_{BC}\left(\mathfrak{g}_\C\right) \hookrightarrow H^{n-\bullet_2,n-\bullet_1}_{BC}(X)$ is actually an isomorphism, then also $i\colon H^{\bullet_1,\bullet_2}_{A}\left(\mathfrak{g}_\C\right) \to H^{\bullet_1,\bullet_2}_{A}(X)$ is an isomorphism.
\end{rem}

\medskip

A slight modification of \cite[Theorem 1]{console-fino} by S. Console and A. Fino gives the following result, which says that the property of computing the Bott-Chern cohomology using just left-invariant forms is open in the space of left-invariant complex structures on solvmanifolds, \cite[Theorem 3.9]{angella}.

\begin{thm}
\label{thm:bc-invariant-open}
 Let $X=\left. \Gamma \right\backslash G$ be a solvmanifold endowed with a $G$-left-invariant complex structure $J$, and denote the Lie algebra naturally associated to $G$ by $\mathfrak{g}$. Let $\sharp\in\{\del,\,\delbar,\,BC,\,A\}$. Suppose that
$$ i\colon H^{\bullet,\bullet}_{\sharp_J}\left(\mathfrak{g}_\C\right) \;\stackrel{\simeq}{\hookrightarrow}\; H^{\bullet,\bullet}_{\sharp_J}\left(X\right) \;.$$
Then there exists an open neighbourhood $\mathcal{U}$ of $J$ in $\mathcal{C}\left(\mathfrak{g}\right)$ such that any $\tilde J\in \mathcal{U}$ still satisfies
$$ i\colon H^{\bullet,\bullet}_{\sharp_{\tilde{J}}}\left(\mathfrak{g}_\C\right) \;\stackrel{\simeq}{\hookrightarrow}\; H^{\bullet,\bullet}_{\sharp_{\tilde{J}}}\left(X\right) \;.$$
In other words, the set
$$ \mathcal{U} \;:=\; \left\{ J\in\mathcal{C}\left(\mathfrak{g}\right) \st i\colon H^{\bullet,\bullet}_{\sharp_J}\left(\mathfrak{g}_\C\right) \;\stackrel{\simeq}{\hookrightarrow}\; H^{\bullet,\bullet}_{\sharp_J}\left(X\right) \right\} $$
is open in $\mathcal{C}\left(\mathfrak{g}\right)$.
\end{thm}

\begin{proof}
 As a matter of notation, for $\varepsilon>0$ small enough, we consider
 $$ \left\{\left(X,\,J_t\right) \st t\in \Delta(0,\varepsilon)\right\} \twoheadrightarrow \Delta(0,\varepsilon) $$
 a complex-analytic family of $G$-left-invariant complex structures on $X$, where $\Delta(0,\varepsilon):=\left\{t\in\C^m\st \left|t\right|<\varepsilon\right\}$ for some $m\in\N\setminus\{0\}$; moreover, let $\left\{g_t\right\}_{t\in \Delta(0,\varepsilon)}$ be a family of $J_t$-Hermitian $G$-left-invariant metrics on $X$ depending smoothly on $t$. We will denote by $\delbar_t:=\delbar_{J_t}$ and $\delbar_t^*:=-*_{g_t}\del_{J_t}*_{g_t}$ the $\delbar$ operator and its $g_t$-adjoint respectively for the Hermitian structure $\left(J_t,\,g_t\right)$ and we set $\Delta_t:=\Delta_{\sharp_{J_t}}$ one of the differential operators involved in the definition of the Dolbeault, conjugate Dolbeault, Bott-Chern or Aeppli cohomologies with respect to $\left(J_t,\,g_t\right)$; we remark that $\Delta_t$ is a self-adjoint elliptic differential operator for all the considered cohomologies.

 By hypothesis, we have that $\left(H^{\bullet,\bullet}_{\sharp_{J_0}}\left(\mathfrak{g}_\C\right)\right)^\perp=\{0\}$, where the orthogonality is meant with respect to the inner product induced by $g_0$, and we have to prove the same replacing $0$ with $t\in\Delta(0,\varepsilon)$. Therefore, it will suffice to prove that
$$ \Delta(0,\varepsilon)\ni t \mapsto \dim_\C\left(H^{\bullet,\bullet}_{\sharp_{J_t}}\left(\mathfrak{g}_\C\right)\right)^\perp\in\N $$
is an upper-semi-continuous function at $0$.
 For any $t\in\Delta(0,\varepsilon)$, being $\Delta_t$ a self-adjoint elliptic differential operator, there exists a complete orthonormal basis $\{e_i(t)\}_{i\in I}$ of eigen-forms for $\Delta_t$ spanning $\left(\wedge^{\bullet,\bullet}_{J_t}\mathfrak{g}_\C^*\right)^\perp$, the orthogonal complement of the space of $G$-left-invariant forms, see \cite[Theorem 1]{kodaira-spencer-3}. For any $i\in I$ and $t\in\Delta(0,\varepsilon)$, let $a_i(t)$ be the eigen-value corresponding to $e_i(t)$; $\Delta_t$ depending differentiably on $t\in\Delta(0,\varepsilon)$, for any $i\in I$, the function $\Delta(0,\varepsilon)\ni t \mapsto a_i(t)\in\C$ is continuous, see \cite[Theorem 2]{kodaira-spencer-3}. Therefore, for any $t_0\in\Delta(0,\varepsilon)$, choosing a constant $c>0$ such that $c\not\in \overline{\left\{a_i(t_0) \st i\in I\right\}}$, the function
 $$ \Psi_c\colon\Delta(0,\varepsilon)\to\N \;,\qquad t \mapsto \dim \Span\left\{e_i(t) \st a_i(t)<c\right\} $$
is locally constant at $t_0$; moreover, for any $t\in\Delta(0,\varepsilon)$ and for any $c>0$, we have
$$ \Psi_c(t) \;\geq\; \dim_\C \left(H^{\bullet,\bullet}_{\sharp_{J_t}}\left(\mathfrak{g}_\C\right)\right)^\perp \;.$$
Since the spectrum of $\Delta_{t_0}$ has no accumulation point for any $t_0\in\Delta(0,\varepsilon)$, see \cite[Theorem 1]{kodaira-spencer-3}, the theorem follows choosing $c>0$ small enough so that $\Psi_c(0)=\dim_\C\left(H^{\bullet,\bullet}_{\sharp_{J_0}}\left(\mathfrak{g}_\C\right)\right)^\perp$.
\end{proof}

\medskip

In particular, the left-invariant complex structures on nilmanifolds belonging to the classes of Theorem \ref{thm:bc-invariant-cor}, and their small deformations satisfy the following conjecture, \cite[Conjecture 3.10]{angella}, which generalizes Conjecture \ref{conj:dolbeault}.

\begin{conj}
\label{conj:BC}
 Let $X=\left. \Gamma \right\backslash G$ be a nilmanifold endowed with a $G$-left-invariant complex structure $J$, and denote the Lie algebra naturally associated to $G$ by $\mathfrak{g}$. Then the de Rham, Dolbeault, Bott-Chern and Aeppli cohomologies can be computed as the cohomologies of the corresponding subcomplexes given by the space of $G$-left-invariant forms on $X$, that is,
$$ \dim_\R \left(H^\bullet_{dR}\left(\mathfrak{g};\R\right)\right)^{\perp} \;=\; 0 \qquad \text{ and }\qquad \dim_\C \left(H^{\bullet,\bullet}_{\sharp}\left(\mathfrak{g}_\C\right)\right)^{\perp}\;=\;0 \;,$$
where $\sharp\in\{\del,\,\delbar,\,BC,\,A\}$, and the orthogonality is meant with respect to the inner product induced by a given $J$-Hermitian $G$-left-invariant metric $g$ on $X$.
\end{conj}

\section{The cohomologies of the Iwasawa manifold and of its small deformations}\label{sec:computations-iwasawa}

The Iwasawa manifold is one of the simplest example of non-K\"ahler complex manifold: as such, it has been studied by several authors, and it has turned out to be a fruitful source of interesting behaviours, see, e.g., \cite{fernandez-gray, nakamura, alessandrini-bassanelli, bassanelli, abbena-garbiero-salamon, ketsetzis-salamon, ye, schweitzer, angella-tomassini-1, angella, franzini}.

In this section, we recall the construction of the Iwasawa manifold \S\ref{sec:iwasawa}, see, e.g., \cite{fernandez-gray}, \cite[\S2]{nakamura}, and of its Kuranishi space, \S\ref{sec:deformations-iwasawa}, see \cite[\S3]{nakamura}; then we write down the de Rham cohomology, \S\ref{sec:derham-iwasawa}, and the Dolbeault cohomology, \S\ref{sec:dolbeault-iwasawa}, (using \cite[Theorem 1]{nomizu}, and \cite[Theorem 1]{sakane} and \cite[Theorem 1]{console-fino}), and we compute the Bott-Chern and Aeppli cohomologies, \S\ref{sec:bott-chern-iwasawa}, (using Theorem \ref{thm:bc-invariant-cor} and Theorem \ref{thm:bc-invariant-open}), of the Iwasawa manifold and of its small deformations.

\subsection{The Iwasawa manifold and its small deformations}\label{subsec:iwasawa}

\subsubsection{The Iwasawa manifold}\label{sec:iwasawa}
Let $\mathbb{H}(3;\mathbb{C})$ be the $3$-dimensional \emph{Heisenberg group} over $\mathbb{C}$ defined by
$$
\mathbb{H}(3;\mathbb{C}) \;:=\; \left\{
\left(
\begin{array}{ccc}
 1 & z^1 & z^3 \\
 0 &  1  & z^2 \\
 0 &  0  &  1
\end{array}
\right) \in \mathrm{GL}(3;\mathbb{C}) \st z^1,\,z^2,\,z^3 \in\C \right\}
\;,
$$
where the product is the one induced by matrix multiplication.
(Equivalently, one can consider $\mathbb{H}(3;\mathbb{C})$ as isomorphic to $\left(\C^3,\,*\right)$, where the group structure $*$ on $\C^3$ is defined as
$$ \left. \left( z_1, \; z_2, \; z_3 \right) \;*\; \left( w_1, \; w_2, \; w_3 \right) \;:=\;
\left( z_1+w_1, \; z_2+w_2, \; z_3+z_1 w_2+w_3 \right) \;. \right) $$
It is straightforward to prove that $\mathbb{H}(3;\C)$ is a connected simply-connected complex $2$-step nilpotent Lie group, that is, the Lie algebra $\left(\mathfrak{h}_3,\,\left[\sspace,\ssspace\right]\right)$ naturally associated to $\mathbb{H}(3;\C)$ satisfies $\left[\mathfrak{h}_3,\mathfrak{h}_3\right]\neq 0$ and $\left[\mathfrak{h}_3,\left[\mathfrak{h}_3,\mathfrak{h}_3\right]\right]=0$.

One finds that
$$
\left\{
\begin{array}{rcl}
 \varphi^1 &:=& \de z^1 \\[5pt]
 \varphi^2 &:=& \de z^2 \\[5pt]
 \varphi^3 &:=& \de z^3-z^1\,\de z^2
\end{array}
\right.
$$
is a $\mathbb{H}(3;\C)$-left-invariant co-frame for the space of $(1,0)$-forms on $\mathbb{H}(3;\C)$, and that the structure equations with respect to this co-frame are
$$
\left\{
\begin{array}{rcl}
 \de\varphi^1 &=& 0 \\[5pt]
 \de\varphi^2 &=& 0 \\[5pt]
 \de\varphi^3 &=& -\varphi^1\wedge\varphi^2
\end{array}
\right. \;.
$$

\smallskip

Consider the action on the left of $\mathbb{H}\left(3;\Z\left[\im\right]\right):=\mathbb{H}(3;\C)\cap\GL\left(3;\Z\left[\im\right]\right)$ on $\mathbb{H}\left(3;\C\right)$ and take the compact quotient
$$ \mathbb{I}_3 := \left. \mathbb{H}\left(3;\Z\left[\im\right]\right) \right\backslash \mathbb{H}(3;\C)\;. $$
One gets that $\mathbb{I}_3$ is a $3$-dimensional complex nilmanifold, whose ($\mathbb{H}(3;\C)$-left-invariant) complex structure $J_{\zero}$ is the one inherited by the standard complex structure on $\C^3$; $\mathbb{I}_3$ is called the \emph{Iwasawa manifold}.

The forms $\varphi^1$, $\varphi^2$ and $\varphi^3$, being $\mathbb{H}(3;\C)$-left-invariant, define a co-frame also for $\left(T^{1,0}\I_3\right)^*$. Note that $\I_3$ is a holomorphically parallelizable manifold, that is, its holomorphic tangent bundle is holomorphically trivial. Since, for example, $\varphi^3$ is a non-closed holomorphic form, it follows that $\I_3$ admits no K\"ahler metric. In fact, one can show that $\I_3$ is not formal, having a non-zero Massey triple product, see \cite[page 158]{fernandez-gray}; therefore the underlying differentiable manifold of $\I_3$ has no complex structure admitting K\"ahler metrics, see \cite[Main Theorem]{deligne-griffiths-morgan-sullivan}, even though all the topological obstructions concerning the Betti numbers are satisfied. Nevertheless, $\mathbb{I}_3$ admits the balanced metric $\omega:=\sum_{j=1}^{3}\varphi^j\wedge\bar\varphi^j$.

We sketch in Figure \ref{fig:iwasawa} the structure of the finite-dimensional double complex $\left(\wedge^{\bullet,\bullet}\left(\mathfrak{h}_3\otimes_\R\C\right)^*,\,\del,\,\delbar\right)$: the dots denote a basis of $\wedge^{\bullet,\bullet}\left(\mathfrak{h}_3\otimes_\R\C\right)^*$, horizontal arrows are meant as $\del$, vertical ones as $\delbar$ and zero arrows are not depicted.

\smallskip
\begin{figure}[ht]
\begin{center}
 \includegraphics[width=8cm,natwidth=800,natheight=600]{iwasawa.eps}
 \caption{The double complex $\left(\wedge^{\bullet,\bullet}\left(\mathfrak{h}_3\otimes_\R \C\right)^*,\,\del,\,\delbar\right)$.}
 \label{fig:iwasawa}
\end{center}
\end{figure}
\smallskip

\subsubsection{Small deformations of the Iwasawa manifold}\label{sec:deformations-iwasawa}
I. Nakamura classified in \cite[\S2]{nakamura} the three-dimensional holomorphically parallelizable solvmanifolds into four classes by numerical invariants, giving the Iwasawa manifold $\mathbb{I}_3$ as an example in the second class. Moreover, he explicitly constructed the Kuranishi family of deformations of $\I_3$, showing that it is smooth and depends on $6$ effective parameters, \cite[pages 94--95]{nakamura}, compare also \cite[Corollary 4.9]{rollenske-jems}. In particular, he computed the Hodge numbers of the small deformations of $\mathbb{I}_3$ proving that they have not to remain invariant along a complex-analytic family of complex structures, \cite[Theorem 2]{nakamura}, compare also \cite[\S4]{ye}; moreover, he proved in this way that the property of being holomorphically parallelizable is not stable under small deformations, \cite[page 86]{nakamura}, compare also \cite[Theorem 5.1, Corollary 5.2]{rollenske-jems}.

\medskip

Firstly, we recall in the following theorem the results by I. Nakamura concerning the Kuranishi space of the Iwasawa manifold.

\begin{thm}[{\cite[pages 94--96]{nakamura}}]
Consider the Iwasawa manifold $\mathbb{I}_3 := \left. \mathbb{H}\left(3;\Z\left[\im\right]\right) \right\backslash \mathbb{H}(3;\C)$. There exists a locally complete complex-analytic family of complex structures $\left\{X_\tempo=\left(\I_3,\,J_\tempo\right)\right\}_{\tempo\in \Delta(\zero,\varepsilon)}$, deformations of $\I_3$, depending on six parameters
$$ \tempo\;=\;\left(t_{11},\,t_{12},\,t_{21},\,t_{22},\,t_{31},\,t_{32}\right)\;\in\;\Delta(\zero,\varepsilon)\;\subset\;\C^6 \;,$$
where $\varepsilon>0$ is small enough, $\Delta(\zero,\varepsilon):=\left\{\mathbf{s}\in\C^6 \st \left|\mathbf{s}\right|<\varepsilon\right\}$, and $X_\zero=\I_3$.

A set of holomorphic coordinates for $X_\tempo$ is given by
$$
\left\{
\begin{array}{rcccl}
 \zeta^1 &:=& \zeta^1(\tempo) &:=& z^1+\sum_{k=1}^{2}t_{1k}\,\bar z^k \\[5pt]
 \zeta^2 &:=& \zeta^2(\tempo) &:=& z^2+\sum_{k=1}^{2}t_{2k}\,\bar z^k \\[5pt]
 \zeta^3 &:=& \zeta^3(\tempo) &:=& z^3+\sum_{k=1}^{2}\left(t_{3k}+t_{2k}\,z^1\right)\bar z^k+A\left(\bar z^1,\,\bar z^2\right)-D\left(\tempo\right)\,\bar z^3
\end{array}
\right.
$$
where
$$ D\left(\tempo\right):=\det\left(
\begin{array}{cc}
 t_{11} & t_{12} \\
 t_{21} & t_{22}
\end{array}
\right)
\qquad
$$
and
$$ A\left(\bar z^1,\,\bar z^2\right):=\frac{1}{2}\left(t_{11}\,t_{21}\,\left(\bar z^1\right)^2+2\,t_{11}\,t_{22}\,\bar z^1\,\bar z^2+t_{12}\,t_{22}\,\left(\bar z^2\right)^2\right) \;. $$
For every $\tempo\in \Delta(\zero,\varepsilon)$, the universal covering of $X_\tempo$ is $\C^3$; more precisely,
$$ X_\tempo \;=\; \left. \Gamma_\tempo \right \backslash\C^3 \;,$$
where $\Gamma_\tempo$ is the subgroup generated by the transformations
$$ \left(\zeta^1,\,\zeta^2,\,\zeta^3\right) \stackrel{\left(\omega^1,\,\omega^2,\,\omega^3\right)}{\mapsto} \left(\tilde\zeta^1,\,\tilde\zeta^2,\,\tilde\zeta^3\right) \;,$$
varying $\left(\omega^1,\,\omega^2,\,\omega^3\right)\in\left(\Z\left[\im\right]\right)^3$, where
$$
\left\{
\begin{array}{rcl} 
 \tilde\zeta^1 &:=& \zeta^1+\left(\omega^1+t_{11}\,\bar\omega^1+t_{12}\,\bar\omega^2\right) \\[10pt]
 \tilde\zeta^2 &:=& \zeta^2+\left(\omega^2+t_{21}\,\bar\omega^1+t_{22}\,\bar\omega^2\right) \\[10pt]
 \tilde\zeta^3 &:=& \zeta^3+\left(\omega^3+t_{31}\,\bar\omega^1+t_{32}\,\bar\omega^2\right)+\omega^1\,\zeta^2\\[5pt]
&&+\left(t_{21}\,\bar\omega^1+ t_{22}\,\bar\omega^2\right)\left(\zeta^1+\omega^1\right)+A\left(\bar \omega^1,\,\bar\omega^2\right)-D\left(\tempo\right)\,\bar\omega^3
\end{array}
\right. \;.
$$
\end{thm}

\begin{rem}
Note that, by \cite[Theorem 4.5]{rollenske-jems}, if $X = \left.\Gamma \right\backslash G$ is a holomorphically parallelizable nilmanifold and $G$ is $\nu$-step nilpotent, then $\mathrm{Kur}(X)$ is cut out by polynomial equations of degree at most $\nu$; furthermore, by \cite[Corollary 4.9]{rollenske-jems}, the Kuranishi space of $X$ is smooth if and only if the associated Lie algebra $\mathfrak{g}$ to $G$ is a free $2$-step nilpotent Lie algebra, i.e., $\mathfrak{g} \simeq \mathfrak{b}_m$ with $m = \dim_\C H^{0,1}_\delbar (X)$, where $\mathfrak{b}_m:=\C^m\oplus\wedge^2\C^m$ with Lie bracket $\left[a_1+b_1\wedge c_1,\, a_2+b_2\wedge c_2\right]:=a_1\wedge a_2$ for $a_1,b_1,c_1,a_2,b_2,c_2\in\C^m$.
\end{rem}

\medskip

According to the classification by I. Nakamura, the small deformations of $\mathbb{I}_3$ are divided into three classes, {\itshape (i)}, {\itshape (ii)}, and {\itshape (iii)}, in terms of their Hodge numbers: such classes are explicitly described by means of polynomial relations in the parameters, see \cite[\S 3]{nakamura}. As we will see in \S\ref{sec:bott-chern-iwasawa}, it turns out that the Bott-Chern cohomology yields a finer classification of the Kuranishi space of $ \mathbb{I}_3$; more precisely, $h^{2,2}_{BC}$ assumes different values within class {\itshape (ii)}, respectively class {\itshape (iii)}, according to the rank of a certain matrix whose entries are related to the complex structure equations with respect to a suitable co-frame, whereas the numbers corresponding to class {\itshape (i)} coincide with those for $ \mathbb{I}_3$: this allows a further subdivision of  classes {\itshape (ii)} and {\itshape (iii)} into subclasses {\itshape (ii.a)}, {\itshape (ii.b)}, and {\itshape (iii.a)},
{\itshape (iii.b)}.

More precisely, the classes and subclasses of this classification are characterized by the following values of the parameters:
\begin{description}
 \item[class {\itshape (i)}] $t_{11}=t_{12}=t_{21}=t_{22}=0$;
 \item[class {\itshape (ii)}] $D\left(\tempo\right)=0$ and $\left(t_{11},\,t_{12},\,t_{21},\,t_{22}\right)\neq \left(0,\,0,\,0,\,0\right)$: 
    \begin{description}
     \item[subclass {\itshape (ii.a)}] $D\left(\tempo\right)=0$ and $\rk S=1$;
     \item[subclass {\itshape (ii.b)}] $D\left(\tempo\right)=0$ and $\rk S=2$;
    \end{description}
 \item[class {\itshape (iii)}] $D\left(\tempo\right)\neq 0$:
    \begin{description}
     \item[subclass {\itshape (iii.a)}] $D\left(\tempo\right)\neq 0$ and $\rk S=1$;
     \item[subclass {\itshape (iii.b)}] $D\left(\tempo\right)\neq 0$ and $\rk S=2$.
    \end{description}
\end{description}
The matrix $S$ is defined by
$$ S \;:=\;
\left(
\begin{array}{cccc}
 \overline{\sigma_{1\bar1}} & \overline{\sigma_{2\bar2}} & \overline{\sigma_{1\bar2}} & \overline{\sigma_{2\bar1}} \\
 \sigma_{1\bar1} & \sigma_{2\bar2} & \sigma_{2\bar1} & \sigma_{1\bar2}
\end{array}
\right)
$$
where $\sigma_{1\bar1},\,\sigma_{1\bar2},\,\sigma_{2\bar1},\,\sigma_{2\bar2}\in\C$ and $\sigma_{12}\in\C$ are complex numbers depending only on $\tempo$ such that
$$ \de\varphi^3_\tempo \;=:\; \sigma_{12}\,\varphi^1_{\tempo}\wedge\varphi^2_{\tempo}+\sigma_{1\bar1}\,\varphi^1_{\tempo}\wedge\bar\varphi^1_{\tempo}+\sigma_{1\bar2}\,\varphi^1_{\tempo}\wedge\bar\varphi^2_{\tempo}+\sigma_{2\bar1}\,\varphi^2_{\tempo}\wedge\bar\varphi^1_{\tempo}+\sigma_{2\bar2}\,\varphi^2_{\tempo}\wedge\bar\varphi^2_{\tempo} \;,$$
being
$$ \varphi^1_\tempo \;:=\; \de\zeta^1_{\tempo}\;,\quad \varphi^2_\tempo \;:=\; \de\zeta^2_{\tempo}\;,\quad \varphi^3_{\tempo} \;:=\; \de\zeta^3_{\tempo}-z_1\,\de\zeta^2_{\tempo}-\left(t_{21}\,\bar z^1+t_{22}\,\bar z^2\right)\de\zeta^1_{\tempo} \;,$$
see \S\ref{sec:structure-equations-iwasawa}.
As we will show, see \S\ref{sec:structure-equations-iwasawa}, the first order asymptotic behaviour of $\sigma_{12},\,\sigma_{1\bar1},\,\sigma_{1\bar2},\,\sigma_{2\bar1},\,\sigma_{2\bar2}$ for $\tempo$ near $0$ is the following:
\begin{equation}
\left\{
\begin{array}{rcl}
\sigma_{12} &=& -1 +\opiccolo{\tempo} \\[5pt]
\sigma_{1\bar1} &=& t_{21} +\opiccolo{\tempo}  \\[5pt]
\sigma_{1\bar2} &=& t_{22} +\opiccolo{\tempo}  \\[5pt]
\sigma_{2\bar1} &=& -t_{11} +\opiccolo{\tempo} \\[5pt]
\sigma_{2\bar2} &=& -t_{12} +\opiccolo{\tempo}
\end{array}
\right.
\qquad \text{ for } \qquad \tempo\in\text{ classes {\itshape (i)}, {\itshape (ii)} and {\itshape (iii)}}
\;,
\end{equation}
and, more precisely, for deformations in class {\itshape (ii)} we actually have that
\begin{equation}
\left\{
\begin{array}{rcl}
\sigma_{12} &=& -1 +\opiccolo{\tempo} \\[5pt]
\sigma_{1\bar1} &=& t_{21} \left(1+\opiccolouno\right)  \\[5pt]
\sigma_{1\bar2} &=& t_{22} \left(1+\opiccolouno\right)  \\[5pt]
\sigma_{2\bar1} &=& -t_{11} \left(1+\opiccolouno\right) \\[5pt]
\sigma_{2\bar2} &=& -t_{12} \left(1+\opiccolouno\right)
\end{array}
\right. \qquad \text{ for } \qquad \tempo\in\text{ class {\itshape (ii)}}\;.
\end{equation}

\medskip

The complex manifold $X_\tempo$ is endowed with the $J_\tempo$-Hermitian $\mathbb{H}(3;\C)$-left-invariant metric $g_\tempo$, which is defined as follows:
$$ g_\tempo \;:=\; \sum_{j=1}^3\varphi_\tempo^j\odot\bar\varphi_\tempo^j \;. $$

\subsubsection{Structure equations for small deformations of the Iwasawa manifold}\label{sec:structure-equations-iwasawa}
In this section, we give the structure equations for the small deformations of the Iwasawa manifold; we will use these computations in \S\ref{sec:dolbeault-iwasawa} and \S\ref{sec:bott-chern-iwasawa} to write the Bott-Chern cohomology of $X_\tempo$, and in Theorem \ref{thm:instability-iwasawa} to prove that the cohomological property of being \Cpf\ is not stable under small deformations of the complex structure.

\medskip

Fix $\tempo \in \Delta(\zero,\varepsilon)\subset \C^6$, and consider the small deformation $X_\tempo$ of the Iwasawa manifold $\mathbb{I}_3$.
Consider the system of complex coordinates on $X_\tempo$ given by
$$
\left\{
\begin{array}{rcl}
 \zeta^1_\mathbf{t} &:=& z^1 + \sum_{\lambda=1}^{2}t_{1\lambda}\bar{z}^{\lambda} \\[10pt]
 \zeta^2_\mathbf{t} &:=& z^2 + \sum_{\lambda=1}^{2}t_{2\lambda}\bar{z}^{\lambda} \\[10pt]
 \zeta^3_\mathbf{t} &:=& z^3 + \sum_{\lambda=1}^{2}(t_{3\lambda}+t_{2\lambda}z^1)\bar{z}^{\lambda}+
 A\left(\bar{z}\right)
\end{array}
\right. \;.
$$
Consider
$$
\left\{
\begin{array}{rcl}
 \varphi^1_{\tempo} &:=& \de\zeta^1_{\tempo} \\[10pt]
 \varphi^2_{\tempo} &:=& \de\zeta^2_{\tempo} \\[10pt]
 \varphi^3_{\tempo} &:=& \de\zeta^3_{\tempo}-z_1\,\de\zeta^2_{\tempo}-\left(t_{21}\,\bar z^1+t_{22}\,\bar z^2\right)\de\zeta^1_{\tempo}
\end{array}
\right.
$$
as a co-frame of $(1,0)$-forms on $X_{\tempo}$ (that is, as a $\Gamma_\tempo$-invariant co-frame of $(1,0)$-forms on $\C^3$). We want to write the structure equations for $X_{\tempo}$ with respect to this co-frame.

A straightforward computation gives
$$
\left\{
\begin{array}{rcl}
 z^1 &=& \gamma\,\left(\zeta^1_\mathbf{t}+\lambda_1\, \bar{\zeta}^1_\mathbf{t}+\lambda_2\, \zeta^2_\mathbf{t}+
 \lambda_3\, \bar{\zeta}^2_\mathbf{t}\right) \\[10pt]
 z^2 &=& \alpha\,\left(\mu_0\, \zeta^1_\mathbf{t}+\mu_1\, \bar{\zeta}^1_\mathbf{t}+\mu_2\, \zeta^2_\mathbf{t}+
 \mu_3\, \bar{\zeta}^2_\mathbf{t}\right)
\end{array}
\right. \;,
$$
where $\alpha$, $\beta$, $\gamma$, $\lambda_i$ (for $i\in\{1,2,3\}$), $\mu_j$ (for $j\in\{0,1,2,3\}$)
are complex numbers depending just on $\mathbf{t}$, and defined as follows:
$$
\left\{
\begin{array}{rcl}
 \alpha &:=& \displaystyle \frac{1}{1-\module{t_{22}}^2-t_{21}\, \bar{t}_{12}}\\[10pt]
 \beta  &:=& \displaystyle t_{21}\, \bar{t}_{11}+t_{22}\, \bar{t}_{21} \\[10pt]
 \gamma &:=& \displaystyle \frac{1}{1-\module{t_{11}}^2-\alpha\, \beta\, \left(t_{11}\, \bar{t}_{12}+
 t_{12}\, \bar{t}_{22}\right)-t_{12}\, \bar{t}_{21}} \\[10pt]
 \lambda_1 &:=& -t_{11}\, \left(1+\alpha\, \bar{t}_{12}\, t_{21}+\alpha\, \module{t_{22}}^2\right) \\[10pt]
 \lambda_2 &:=& \displaystyle \alpha\,  \left(t_{11}\, \bar{t}_{12}+ t_{12}\, \bar{t}_{22}\right) \\[10pt]
 \lambda_3 &:=& \displaystyle -t_{12}\, \left(1+\alpha\, \bar{t}_{12}\, t_{21}+\alpha\, \module{t_{22}}^2\right) \\[10pt]
 \mu_0 &:=& \displaystyle \beta\, \gamma \\[10pt]
 \mu_1 &:=& \displaystyle \lambda_1\, \beta\, \gamma-t_{21} \\[10pt]
 \mu_2 &:=& \displaystyle 1+\lambda_2\, \beta\, \gamma \\[10pt]
 \mu_3 &:=& \displaystyle \lambda_3\, \beta\, \gamma-t_{22}
\end{array}
\right. \;.
$$

For the complex structures in the class {\itshape (i)}, one checks that the structure equations (with respect to the co-frame $\left\{\varphi^1_\tempo,\, \varphi^2_\tempo,\, \varphi^3_\tempo\right\}$) are the same as the ones for $\I_3$, that is,
$$
\left\{
\begin{array}{rcl}
 \de\varphi^1_{\tempo} &=& 0 \\[10pt]
 \de\varphi^2_{\tempo} &=& 0 \\[10pt]
 \de\varphi^3_{\tempo} &=& -\varphi^1_{\tempo}\wedge\varphi^2_{\tempo}
\end{array}
\right. \qquad \text{ for } \quad \tempo\in\text{ class {\itshape (i)}}
\;.
$$
For small deformations in classes {\itshape (ii)} and {\itshape (iii)}, we have that
$$
\left\{
\begin{array}{rcl}
 \de\varphi^1_{\tempo} &=& 0 \\[10pt]
 \de\varphi^2_{\tempo} &=& 0 \\[10pt]
 \de\varphi^3_{\tempo} &=& \sigma_{12}\,\varphi^1_{\tempo}\wedge\varphi^2_{\tempo}\\[5pt]
 &&+\sigma_{1\bar1}\,\varphi^1_{\tempo}\wedge\bar\varphi^1_{\tempo}+\sigma_{1\bar2}\,\varphi^1_{\tempo}\wedge\bar\varphi^2_{\tempo}\\[5pt]
 &&+\sigma_{2\bar1}\,\varphi^2_{\tempo}\wedge\bar\varphi^1_{\tempo}+\sigma_{2\bar2}\,\varphi^2_{\tempo}\wedge\bar\varphi^2_{\tempo}
\end{array}
\right. \qquad \text{ for } \qquad \tempo\in\text{ classes {\itshape (ii)} and {\itshape (iii)}} \;,
$$
where $\sigma_{12},\,\sigma_{1\bar1},\,\sigma_{1\bar2},\,\sigma_{2\bar1},\,\sigma_{2\bar2}\in\C$ are complex numbers depending just on $\tempo$.
The asymptotic behaviour of $\sigma_{12},\,\sigma_{1\bar1},\,\sigma_{1\bar2},\,\sigma_{2\bar1},\,\sigma_{2\bar2}\in\C$ is the following:
\begin{equation}
\left\{
\begin{array}{rcl}
\sigma_{12} &=& -1 +\opiccolo{\tempo} \\[5pt]
\sigma_{1\bar1} &=& t_{21} +\opiccolo{\tempo}  \\[5pt]
\sigma_{1\bar2} &=& t_{22} +\opiccolo{\tempo}  \\[5pt]
\sigma_{2\bar1} &=& -t_{11} +\opiccolo{\tempo} \\[5pt]
\sigma_{2\bar2} &=& -t_{12} +\opiccolo{\tempo}
\end{array}
\right.
\qquad \text{ for } \qquad \tempo\in\text{ classes {\itshape (i)}, {\itshape (ii)} and {\itshape (iii)}}
\;,
\end{equation}
more precisely, for deformations in class {\itshape (ii)} we actually have that
\begin{equation}
\left\{
\begin{array}{rcl}
\sigma_{12} &=& -1 +\opiccolo{\tempo} \\[5pt]
\sigma_{1\bar1} &=& t_{21} \left(1+\opiccolouno\right)  \\[5pt]
\sigma_{1\bar2} &=& t_{22} \left(1+\opiccolouno\right)  \\[5pt]
\sigma_{2\bar1} &=& -t_{11} \left(1+\opiccolouno\right) \\[5pt]
\sigma_{2\bar2} &=& -t_{12} \left(1+\opiccolouno\right)
\end{array}
\right. \qquad \text{ for } \qquad \tempo\in\text{ class {\itshape (ii)}}\;.
\end{equation}
The explicit values of $\sigma_{12},\,\sigma_{1\bar1},\,\sigma_{1\bar2},\,\sigma_{2\bar1},\,\sigma_{2\bar2}\in\C$ in the case of class {\itshape (ii)} are the following, \cite[page 416]{angella-tomassini-1}:
$$
\left\{
\begin{array}{rcl}
 \sigma_{12} &:=& -\gamma+t_{21}\bar{\lambda}_3\bar{\gamma}+t_{22}\bar{\alpha}\bar{\mu}_3 \\[10pt]
 \sigma_{1\bar{1}} &:=& t_{21}\,\overline{\gamma\left(1+t_{21}\bar{t}_{12}\alpha+
 \module{t_{22}}^2\alpha\right)} \\[10pt]
 \sigma_{1\bar{2}} &:=& t_{22}\,\overline{\gamma\left(1+t_{21}\bar{t}_{12}\alpha+
 \module{t_{22}}^2\alpha\right)} \\[10pt]
 \sigma_{2\bar{1}} &:=& -t_{11}\,\gamma\left(1+t_{21}\bar{t}_{12}\alpha+
 \module{t_{22}}^2\alpha\right) \\[10pt]
 \sigma_{2\bar{2}} &:=& -t_{12}\,\gamma\left(1+t_{21}\bar{t}_{12}\alpha+
 \module{t_{22}}^2\alpha\right)
\end{array}
\right. \qquad \text{ for } \qquad \tempo\in\text{ class {\itshape (ii)}}\;.
$$
Note that, for small deformations in class {\itshape (ii)}, one has $\sigma_{12}\,\neq\, 0$ and
$\left(\sigma_{1\bar{1}},\,\sigma_{1\bar{2}},\,\sigma_{2\bar{1}},\,\sigma_{2\bar{2}}\right)\,\neq\,
\left(0,\,0,\,0,\,0\right)$.

\subsection{The de Rham cohomology of the Iwasawa manifold and of its small deformations}\label{sec:derham-iwasawa}
Recall that, by Ehresmann's theorem, every complex-analytic family of compact complex manifolds is locally trivial as a differentiable family of compact differentiable manifolds, see, e.g., \cite[Theorem 4.1]{kodaira-morrow}. Therefore the de Rham cohomology of small deformations of the Iwasawa manifold is the same as the de Rham cohomology of $\I_3$, which can be computed by using K. Nomizu's theorem \cite[Theorem 1]{nomizu}.

In the table below, we list the harmonic representatives with respect to the metric $g_\zero$ instead of their classes and, as usually, we shorten the notation writing, for example, $\varphi^{A\bar B}:=\varphi^A\wedge\bar\varphi^B$.

\smallskip
\begin{center}
\begin{tabular}{c||c|c}
\toprule
$H^k_{dR}\left(\I_3;\C\right)$ & $g_{\zero}$-harmonic representatives & $\dim_\C H^k_{dR}\left(\I_3;\C\right)$\\[5pt]
\midrule[0.02em]\addlinespace[1.5pt]\midrule[0.02em]
$k=1$ & $\varphi^1$, $\varphi^2$, $\bar{\varphi}^1$, $\bar{\varphi}^2$ & $4$\\[5pt]
$k=2$ & $\varphi^{13}$, $\varphi^{23}$, $\varphi^{1\bar{1}}$, $\varphi^{1\bar{2}}$, $\varphi^{2\bar{1}}$, $\varphi^{2\bar{2}}$, $\varphi^{\bar{1}\bar{3}}$, $\varphi^{\bar{2}\bar{3}}$ & $8$\\[5pt]
$k=3$ & $\varphi^{123}$, $\varphi^{13\bar{1}}$, $\varphi^{13\bar{2}}$, $\varphi^{23\bar{1}}$, $\varphi^{23\bar{2}}$, $\varphi^{1\bar{1}\bar{3}}$, $\varphi^{1\bar{2}\bar{3}}$, $\varphi^{2\bar{1}\bar{3}}$, $\varphi^{2\bar{2}\bar{3}}$, $\varphi^{\bar{1}\bar{2}\bar{3}}$ & $10$\\[5pt]
$k=4$ & $\varphi^{123\bar{1}}$, $\varphi^{123\bar{2}}$, $\varphi^{13\bar{1}\bar{3}}$, $\varphi^{13\bar{2}\bar{3}}$, $\varphi^{23\bar{1}\bar{3}}$, $\varphi^{23\bar{2}\bar{3}}$, $\varphi^{1\bar{1}\bar{2}\bar{3}}$, $\varphi^{2\bar{1}\bar{2}\bar{3}}$ & $8$\\[5pt]
$k=5$ & $\varphi^{123\bar{1}\bar{3}}$, $\varphi^{123\bar{2}\bar{3}}$, $\varphi^{13\bar{1}\bar{2}\bar{3}}$, $\varphi^{23\bar{1}\bar{2}\bar{3}}$ & $4$\\
\bottomrule
\end{tabular}
\end{center}
\smallskip

\begin{rem}
Note that all the $g_\zero$-harmonic representatives of $H^\bullet_{dR}(\mathbb{I}_3;\R)$ are of pure type with respect to $J_{\zero}$, that is, they are in $\left(\wedge^{p,q}\mathbb{I}_3\oplus \wedge^{q,p}\mathbb{I}_3\right) \cap \wedge^{p+q}\mathbb{I}_3$ for some $p,q\in\{0,\,1,\,2,\,3\}$; this is no more true for $J_\tempo$ with $\tempo\neq\zero$ small enough, see Theorem \ref{thm:instability-iwasawa}.
\end{rem}

\subsection{The Dolbeault cohomology of the Iwasawa manifold and of its small deformations}\label{sec:dolbeault-iwasawa}
The Hodge numbers of the Iwasawa manifold and of its small deformations have been computed by I. Nakamura in \cite[page 96]{nakamura}. The $g_\tempo$-harmonic representatives for $H^{\bullet,\bullet}_{\delbar}\left(X_{\tempo}\right)$, for $\tempo$ small enough, can be computed using the considerations in \S\ref{subsec:cohomology-computation-derham-dolbeault} and the structure equations given in \S\ref{sec:structure-equations-iwasawa}. We collect here the results of the computations.

In order to reduce the number of cases under consideration, recall that, on a compact complex Hermitian manifold $X$ of complex dimension $n$, for any $p,q\in\N$, the Hodge-$*$-operator and the conjugation induce an isomorphism
$$ H^{p,q}_{\delbar}(X) \stackrel{\simeq}{\to} H^{n-q,n-p}_{\del}(X) \stackrel{\simeq}{\to} \overline{H^{n-p,n-q}_{\delbar}(X)} \;.$$

\medskip

\begin{itemize}
 \item \textbf{$1$-forms.}
It is straightforward to check that
$$ H^{1,0}_{\delbar}(X_\tempo)\;=\;\C\left\langle \phit{1},\,\phit{2},\,\phit{3} \right\rangle \quad \text{ for } \quad \tempo\in\text{ class {\itshape (i)}} $$
and
$$ H^{0,1}_{\delbar}(X_\tempo)\;=\;\C\left\langle \bphit{1},\,\bphit{2}\right\rangle \quad \text{ for } \quad \tempo\in\text{ classes {\itshape (i)}, {\itshape (ii)} and {\itshape (iii)}} \;.$$
Since $\delbar\,\phit{3}\neq 0$ for $X_\tempo$ in class {\itshape (ii)} or in class {\itshape (iii)}, one has
$$ H^{1,0}_{\delbar}(X_\tempo)\;=\;\C\left\langle \phit{1},\,\phit{2} \right\rangle \quad \text{ for } \quad \tempo\in\text{ classes {\itshape (ii)} and {\itshape (iii)}} \;: $$
this means in particular that $X_{\tempo}$ is not holomorphically parallelizable for $\tempo$ in classes {\itshape (ii)} and {\itshape (iii)}, \cite[pages 86, 96]{nakamura}.

Summarizing,
$$
\dim_\C H^{1,0}_{\delbar}(X_\tempo)\;=\;
\left\{
\begin{array}{lcl}
3 & \quad \text{ for } \quad & \tempo\in\text{ class {\itshape (i)}}\\[5pt]
2 & \quad \text{ for } \quad & \tempo\in\text{ classes {\itshape (ii)} and {\itshape (iii)}}
\end{array}
\right. \;,
$$
and
$$
\dim_\C H^{0,1}_{\delbar}(X_\tempo)\;=\;2 \quad \text{ for } \quad \tempo\in\text{ classes {\itshape (i)}, {\itshape (ii)} and {\itshape (iii)}} \;.
$$

 \item \textbf{$2$-forms.}
A straightforward computation yields
$$ H^{2,0}_{\delbar}(X_\tempo)\;=\;\C\left\langle \phit{12},\, \phit{13},\,\phit{23} \right\rangle  \quad \text{ for } \quad \tempo\in\text{ class {\itshape (i)}}\;,$$
$$ H^{1,1}_{\delbar}(X_\tempo)\;=\;\C\left\langle \phit{1\bar1},\, \phit{1\bar2},\, \phit{2\bar1},\, \phit{2\bar2},\, \phit{3\bar1},\, \phit{3\bar2} \right\rangle \quad \text{ for } \quad \tempo\in\text{ class {\itshape (i)}} \;, $$
and
$$ H^{0,2}_{\delbar}(X_\tempo)\;=\;\C\left\langle \phit{\bar1\bar3},\, \phit{\bar2\bar3} \right\rangle  \quad \text{ for } \quad \tempo\in\text{ classes {\itshape (i)}, {\itshape (ii)} and {\itshape (iii)}} \;.$$

We now compute $H^{2,0}_{\delbar}(X_\tempo)$ for $\tempo \in \text{ classes {\itshape (ii)} and {\itshape (iii)}}$. The $\mathbb{H}(3;\C)$-left-invariant $(2,0)$-forms are of the type $A\,\phit{12}+B\,\phit{13}+C\,\phit{23}$ with $A,B,C\in\C$, so one has to solve the linear system
$$
\left(
\begin{array}{ccc}
 0 & 0 & 0 \\
 0 & -\sigma_{2\bar1} & \sigma_{1\bar1} \\
 0 & -\sigma_{2\bar2} & \sigma_{1\bar2}
\end{array}
\right)
\cdot
\left(
\begin{array}{c}
 A \\
 B \\
 C
\end{array}
\right)
=
\left(
\begin{array}{c}
 0 \\
 0 \\
 0
\end{array}
\right) \;;
$$
since the associated matrix to the system has rank $0$ for $\tempo \in \text{ class {\itshape (i)}}$, rank $1$ for $\tempo \in \text{ class {\itshape (ii)}}$ and rank $2$ for $\tempo \in \text{ class {\itshape (iii)}}$, one concludes that
$$ \dim_\C H^{2,0}_{\delbar}(X_\tempo) \;=\; 2 \quad \text{ for } \quad \tempo\in\text{ class {\itshape (ii)}} $$
(the generators being $\phit{12}$ and a linear combination of $\phit{13}$ and $\phit{23}$) and
$$ \dim_\C H^{2,0}_{\delbar}(X_{\tempo}) \;=\; 1 \quad \text{ for } \quad \tempo\in\text{ class {\itshape (iii)}} $$
(the generator being $\phit{12}$).

It remains to compute $H^{1,1}_{\delbar}(X_\tempo)$ for $\tempo \in \text{ classes {\itshape (ii)} and {\itshape (iii)}}$. For such $\tempo$, one has that: three independent $\overline\square_{J_\tempo}$-harmonic $(1,1)$-forms are of the type $\psi_1:=:A\,\phit{1\bar1}+B\,\phit{1\bar2}+C\,\phit{2\bar1}+D\,\phit{2\bar2}$ where $A,\,B,\,C,\,D\in\C$ satisfy the equation
$$ \left(
\begin{array}{cccc}
 \overline{\sigma_{1\bar1}} & -\overline{\sigma_{1\bar2}} & -\overline{\sigma_{2\bar1}} & \overline{\sigma_{2\bar2}}
\end{array}
\right)\cdot
\left(
\begin{array}{c}
 A \\
 B \\
 C \\
 D
\end{array}
\right)
\;=\;
0 \;,
$$
whose matrix has rank $1$ for $\tempo\in\text{ classes {\itshape (ii)} and {\itshape (iii)}}$ (while its rank is $0$ for $\tempo\in\text{ class {\itshape (i)}}$); two other independent $\overline\square_{J_\tempo}$-harmonic $(1,1)$-forms are of the type $\psi_2:=:E\,\phit{1\bar3}+F\,\phit{2\bar3}+G\,\phit{3\bar1}+H\,\phit{3\bar2}$ where $E,\,F,\,G,\,H\in\C$ are solution of the system
$$ \left(
\begin{array}{cccc}
 -\overline{\sigma_{12}} & 0 & -\overline{\sigma_{1\bar2}} & \overline{\sigma_{1\bar1}} \\
 0 & -\overline{\sigma_{12}} & -\overline{\sigma_{2\bar1}} & \overline{\sigma_{2\bar2}}
\end{array}
\right)\cdot
\left(
\begin{array}{c}
 E \\
 F \\
 G \\
 H
\end{array}
\right)
\;=\;
 \left(
\begin{array}{c}
 0 \\
 0
\end{array}
\right)
\;,
$$
whose matrix has rank $2$ for $\tempo\in\text{ classes {\itshape (i)}, {\itshape (ii)} and {\itshape (iii)}}$; note also that no $(1,1)$-form with a non-zero component in $\phit{3\bar3}$ can be $\overline\square_{J_t}$-harmonic. Hence, one can conclude that
$$ \dim_\C H^{1,1}_{\delbar}(X_{\tempo}) \;=\; 5 \quad \text{ for } \quad \tempo\in\text{ classes {\itshape (ii)} and {\itshape (iii)}} \;. $$

Summarizing,
$$
\dim_\C H^{2,0}_{\delbar}(X_\tempo)\;=\;
\left\{
\begin{array}{lcl}
 3 & \quad \text{ for } \quad & \tempo\in\text{ class {\itshape (i)}} \\[5pt]
 2 & \quad \text{ for } \quad & \tempo\in\text{ class {\itshape (ii)}} \\[5pt]
 1 & \quad \text{ for } \quad & \tempo\in\text{ class {\itshape (iii)}}
\end{array}
\right. \;,
$$
and
$$
\dim_\C H^{1,1}_{\delbar}(X_\tempo)\;=\;
\left\{
\begin{array}{lcl}
 6 & \quad \text{ for } \quad & \tempo\in\text{ class {\itshape (i)}} \\[5pt]
 5 & \quad \text{ for } \quad & \tempo\in\text{ classes {\itshape (ii)} and {\itshape (iii)}}
\end{array}
\right. \;,
$$
and
$$
\dim_\C H^{0,2}_{\delbar}(X_\tempo)\;=\;2 \quad \text{ for } \quad \tempo\in\text{ classes {\itshape (i)}, {\itshape (ii)} and {\itshape (iii)}} \;.
$$

 \item \textbf{$3$-forms.}
Finally, we have to compute $H^{3,0}_{\delbar}\left(X_\tempo\right)$ and $H^{2,1}_{\delbar}\left(X_\tempo\right)$.
A straightforward linear algebra computation yields to
$$ H^{3,0}_{\delbar}(X_{\tempo})\;=\;\C\left\langle \phit{123} \right\rangle \quad \text{ for } \quad \tempo\in\text{ classes {\itshape (i)}, {\itshape (ii)} and {\itshape (iii)}}$$
and
$$ H^{2,1}_{\delbar}(X_{\tempo})\;=\;\C\left\langle \phit{12\bar1},\, \phit{12\bar2},\, \phit{13\bar1},\, \phit{13\bar2},\, \phit{23\bar1},\, \phit{23\bar2} \right\rangle  \quad \text{ for } \quad \tempo\in\text{ class {\itshape (i)}}\;.$$

It remains to compute $H^{2,1}_{\delbar}(X_{\tempo})$ for $\tempo \in \text{ classes {\itshape (ii)} and {\itshape (iii)}}$. Firstly, one notes that four of the six generators of the space of $\mathbb{H}(3;\C)$-left-invariant $(2,1)$-forms that are $\overline\square_{J_\tempo}$-harmonic for $\tempo\in\text{ class {\itshape (i)}}$ can be slightly modified to get four $\delbar_{J_\tempo}$-holomorphic $(2,1)$-forms for $\tempo\in\text{ class {\itshape (ii)} or class {\itshape (iii)}}$: more precisely, one has
$$
H^{2,1}_{\delbar}(X_{\tempo}) \;\supseteq\; \C\left\langle \phit{13\bar1}-\frac{\sigma_{2\bar2}}{\overline{\sigma_{12}}}\phit{12\bar3},\, \phit{13\bar2}-\frac{\sigma_{2\bar1}}{\overline{\sigma_{12}}}\phit{12\bar3},\, \phit{23\bar1}-\frac{\sigma_{1\bar2}}{\overline{\sigma_{12}}}\phit{12\bar3}, \, \phit{23\bar2}-\frac{\sigma_{1\bar1}}{\overline{\sigma_{12}}}\phit{12\bar3} \right\rangle \;;
$$
in other words, four independent $\overline\square_{J_\tempo}$-harmonic $(2,1)$-forms are of the type $\psi_2:=:C\,\phit{12\bar3}+D\,\phit{13\bar1}+E\,\phit{13\bar2}+F\,\phit{23\bar1}+G\,\phit{23\bar2}$, where $C,\,D,\,E,\,F,\,G\in\C$ are solution of the linear system
$$
\left(
\begin{array}{ccccc}
 \overline{\sigma_{12}} & \sigma_{2\bar2} & -\sigma_{2\bar1} & -\sigma_{1\bar2} & \sigma_{1\bar1}
\end{array}
\right) \cdot
\left(
\begin{array}{c}
 C \\
 D \\
 E \\
 F \\
 G
\end{array}
\right)
\;=\; 0 \;,
$$
whose matrix has rank $1$ for every $\tempo\in\text{ classes {\itshape (i)}, {\itshape (ii)} and {\itshape (iii)}}$.
Note that one can reduce to study the $\overline\square$-harmonicity of the $(2,1)$-forms of the type $\psi_1:=:A\,\phit{12\bar1}+B\,\phit{12\bar2}$: indeed, a $(2,1)$-form $\psi:=:\psi_1+\psi_2+H\, \phit{13\bar3} + L\, \phit{23\bar3}$, where $H,L\in\C$, is $\overline\square$-harmonic if and only if $H=0=L$ and both $\psi_1$ and $\psi_2$ are $\overline\square$-harmonic. A $(2,1)$-form of the type $\psi_1$ is $\overline\square$-harmonic if and only if $A,\,B\in\C$ solve the linear system
$$
\left(
\begin{array}{cc}
 -\overline{\sigma_{1\bar1}} & \overline{\sigma_{1\bar2}} \\
 -\overline{\sigma_{2\bar1}} & \overline{\sigma_{2\bar2}}
\end{array}
\right)
\cdot
\left(
\begin{array}{c}
 A \\
 B
\end{array}
\right)
=
\left(
\begin{array}{c}
 0 \\
 0
\end{array}
\right) \;,
$$
whose matrix has rank $0$ for $\tempo\in\text{ class {\itshape (i)}}$, rank $1$ for $\tempo\in\text{ class {\itshape (ii)}}$ and rank $2$ for $\tempo\in\text{ class {\itshape (iii)}}$. In particular, one gets that
$$ \dim_\C H^{2,1}_{\delbar}(X_{\tempo}) \;=\; 5 \quad \text{ for } \quad \tempo\in\text{ class {\itshape (ii)}} $$
and
$$ \dim_\C H^{2,1}_{\delbar}(X_{\tempo}) \;=\; 4 \quad \text{ for } \quad \tempo\in\text{ class {\itshape (iii)}} \;. $$

Summarizing,
$$
\dim_\C H^{3,0}_{\delbar}(X_{\tempo}) \;=\; 1 \quad \text{ for } \quad \tempo\in\text{ classes {\itshape (i)}, {\itshape (ii)} and {\itshape (iii)}} \;,
$$
and
$$
\dim_\C H^{2,1}_{\delbar}(X_{\tempo})\;=\;
\left\{
\begin{array}{lcl}
6 & \quad \text{ for } \quad & \tempo\in\text{ class {\itshape (i)}} \\[5pt]
5 & \quad \text{ for } \quad & \tempo\in\text{ class {\itshape (ii)}} \\[5pt]
4 & \quad \text{ for } \quad & \tempo\in\text{ class {\itshape (iii)}}
\end{array}
\right. \;.
$$
\end{itemize}

\subsection{The Bott-Chern and Aeppli cohomologies of the Iwasawa manifold and of its small deformations}\label{sec:bott-chern-iwasawa}
In this section, using Theorem \ref{thm:bc-invariant-cor} and Theorem \ref{thm:bc-invariant-open}, we explicitly compute the dimensions of $H_{BC}^{\bullet,\bullet}(X_{\tempo})$, for $\tempo$ small enough, \cite[\S5.3]{angella}: such numbers are summarized in the tables in \S\ref{subsec:chart}.

\medskip

In order to reduce the number of cases under consideration, recall that, on a compact complex Hermitian manifold $X$ of complex dimension $n$, for every $p,q\in\N$, the conjugation induces an isomorphism $H^{p,q}_{BC}(X) \stackrel{\simeq}{\to} H^{q,p}_{BC}(X)$, and the Hodge-$*$-operator induces an isomorphism $H^{p,q}_{BC}(X) \stackrel{\simeq}{\to} H^{n-q,n-p}_{A}(X)$; furthermore, note that
$$ H^{p,0}_{BC}(X) \;\simeq\; \ker\left(\de\colon \wedge^{p,0}X\to\wedge^{p+1}(X;\C)\right) $$
and that
$$ H^{n,0}_{BC}(X) \;\simeq\; H^{n,0}_{\delbar}(X) \;.$$

\medskip

\begin{itemize}
\item \textbf{$1$-forms}
It is straightforward to check that
$$ H^{1,0}_{BC}(X_{\tempo})\;=\:\C\left\langle \phit{1},\,\phit{2}\right\rangle \quad \text{ for } \quad \tempo\in\text{ classes {\itshape (i)}, {\itshape (ii)} and {\itshape (iii)}} \;. $$

\item \textbf{$2$-forms}
It is straightforward to compute
$$ H^{2,0}_{BC}(X_{\tempo}) \;=\; \C\left\langle \phit{12},\,\phit{13},\,\phit{23} \right\rangle \quad \text{ for }\quad \tempo\in\text{ class {\itshape (i)}} \;.$$
The computations for $H^{2,0}_{BC}(X_{\tempo})$ reduce to find $\psi=A\,\phit{12}+B\,\phit{13}+C\,\phit{23}$ where $A,\,B,\,C\in\C$ satisfy the linear system
$$
\left(
\begin{array}{ccc}
 0 & 0 & 0 \\
 0 & -\sigma_{2\bar1} & \sigma_{1\bar1} \\
 0 & -\sigma_{2\bar2} & \sigma_{1\bar2}
\end{array}
\right)
\cdot
\left(
\begin{array}{c}
 A \\
 B \\
 C
\end{array}
\right)
=
\left(
\begin{array}{c}
 0 \\
 0 \\
 0
\end{array}
\right) \;,
$$
whose matrix has rank $0$ for $\tempo\in\text{ class {\itshape (i)}}$, rank $1$ for $\tempo\in\text{ class {\itshape (ii)}}$ and rank $2$ for $\tempo\in\text{ class {\itshape (iii)}}$; so, in particular, we get that
$$ \dim_\C H^{2,0}_{BC}(X_{\tempo}) \;=\; 2 \quad \text{ for } \quad \tempo\in\text{ class {\itshape (ii)}} $$
and
$$ \dim_\C H^{2,0}_{BC}(X_{\tempo}) \;=\; 1 \quad \text{ for } \quad \tempo\in\text{ class {\itshape (iii)}} $$
(more precisely, for $\tempo\in\text{ class {\itshape (iii)}}$ we have $H^{2,0}_{BC}(X_{\tempo})=\C\left\langle \phit{12}\right\rangle$).

It remains to compute $H^{1,1}_{BC}(X_\tempo)$ for $\tempo\in\text{ classes {\itshape (i)}, {\itshape (ii)} and {\itshape (iii)}}$. First of all, it is easy to check that
$$ H^{1,1}_{BC}(X_{\tempo}) \;\supseteq\; \C\left\langle \phit{1\bar1},\,\phit{1\bar2},\,\phit{2\bar1},\,\phit{2\bar2} \right\rangle \quad \text{ for } \quad \tempo\in\text{ classes {\itshape (i)}, {\itshape (ii)} and {\itshape (iii)}} \;,$$
and equality holds if $\tempo\in\text{ class {\itshape (i)}}$, hence, in particular, if $\tempo=\zero$. This immediately implies that
$$ H^{1,1}_{BC}(X_{\tempo}) \;=\; \C\left\langle \phit{1\bar1},\,\phit{1\bar2},\,\phit{2\bar1},\,\phit{2\bar2} \right\rangle \quad \text{ for } \quad \tempo\in\text{ classes {\itshape (i)}, {\itshape (ii)} and {\itshape (iii)}} \;;$$
indeed, the function $\tempo\mapsto \dim_\C H^{1,1}_{BC}(X_\tempo)$ is upper-semi-continuous at $0$, since $H^{1,1}_{BC}(X_\tempo)$ is isomorphic to the kernel of the self-adjoint elliptic differential operator $\tilde\Delta_{BC_{J_\tempo}}\lfloor_{\wedge^{1,1}X_{\tempo}}$. (One can explain this argument saying that the new parts appearing in the computations for $\tempo\neq \zero$ are ``too small'' to balance out the lack for the $\del$-closure or the $\delbar$-closure.) From another point of view, we can note that $(1,1)$-forms of the type $\psi=A\,\phit{1\bar3}+B\,\phit{2\bar3}+C\,\phit{3\bar1}+D\,\phit{3\bar2}+E\,\phit{3\bar3}$ are $\tilde\Delta_{BC_{J_\tempo}}$-harmonic if and only if $E=0$ and $A,\,B,\,C,\,D\in\C$ satisfy the linear system
$$
\left(
\begin{array}{cccc}
 -\overline{\sigma_{12}} & 0 & -\sigma_{1\bar2} & -\sigma_{1\bar1} \\
 0 & -\overline{\sigma_{12}} & -\sigma_{2\bar2} & -\sigma_{2\bar1} \\
\hline
 \overline{\sigma_{1\bar2}} & -\overline{\sigma_{1\bar1}} & \sigma_{12} & 0 \\
 \overline{\sigma_{2\bar2}} & -\overline{\sigma_{2\bar1}} & 0 & \sigma_{12} 
\end{array}
\right)
\cdot
\left(
\begin{array}{c}
 A \\
 B \\
 C \\
 D
\end{array}
\right)
\;=\;
\left(
\begin{array}{c}
 0 \\
 0 \\
 0 \\
 0
\end{array}
\right) \;,
$$
whose matrix has rank $4$ for every $\tempo\in\text{ classes {\itshape (i)}, {\itshape (ii)} and {\itshape (iii)}}$.

\item \textbf{$3$-forms}
It is straightforward to compute
$$ H^{3,0}_{BC}(X_{\tempo})\;=\;\C \left\langle \phit{123} \right\rangle \quad \text{ for } \quad \tempo\in\text{ classes {\itshape (i)}, {\itshape (ii)} and {\itshape (iii)}} \;.$$
Moreover,
\begin{eqnarray*}
H^{2,1}_{BC}(X_{\tempo}) &=& \C \left\langle \phit{12\bar1},\, \phit{12\bar2},\, \phit{13\bar1}-\frac{\sigma_{2\bar2}}{\overline{\sigma_{12}}}\,\phit{12\bar3},\, \phit{13\bar2}+\frac{\sigma_{2\bar1}}{\overline{\sigma_{12}}}\,\phit{12\bar3},\, \phit{23\bar1}+\frac{\sigma_{1\bar2}}{\overline{\sigma_{12}}}\,\phit{12\bar3},\, \phit{23\bar2}-\frac{\sigma_{1\bar1}}{\overline{\sigma_{12}}}\,\phit{12\bar3}  \right\rangle \\[5pt]
 && \text{ for } \quad \tempo\in\text{ classes {\itshape (i)}, {\itshape (ii)} and {\itshape (iii)}}\;;
\end{eqnarray*}
in particular,
$$ H^{2,1}_{BC}(X_{\tempo})\;=\;\C \left\langle \phit{12\bar1},\, \phit{12\bar2},\, \phit{13\bar1},\, \phit{13\bar2},\, \phit{23\bar1},\, \phit{23\bar2} \right\rangle  \quad \text{ for } \quad \tempo\in\text{ class {\itshape (i)}} \;.$$
From another point of view, one can easily check that
$$
H^{2,1}_{BC}(X_{\tempo}) \;\supseteq\; \C \left\langle \phit{12\bar1},\, \phit{12\bar2} \right\rangle \quad \text{ for } \quad \tempo\in\text{ classes {\itshape (i)}, {\itshape (ii)} and {\itshape (iii)}} \;,
$$
and that the $(2,1)$-forms of the type $\psi=A\,\phit{12\bar3}+B\,\phit{13\bar1}+C\,\phit{13\bar2}+D\,\phit{23\bar1}+E\,\phit{23\bar2}+F\,\phit{13\bar3}+G\,\phit{23\bar3}$ are $\tilde\Delta_{BC_{J_\tempo}}$-harmonic if and only if $F=0=G$ and $A,\,B,\,C,\,D,\,E\in\C$ satisfy the equation
$$
\left(
\begin{array}{ccccc}
 \overline{\sigma_{12}} & \sigma_{2\bar2} & -\sigma_{2\bar1} & \sigma_{1\bar2} & \sigma_{1\bar1}
\end{array}
\right)
\cdot
\left(
\begin{array}{c}
 A \\
 B \\
 C \\
 D \\
 E
\end{array}
\right)
\;=\;
 0 \;,
$$
whose matrix has rank $1$ for every $\tempo\in\text{ classes {\itshape (i)}, {\itshape (ii)} and {\itshape (iii)}}$. Note in particular that the dimensions of $H^{3,0}_{BC}(X_{\tempo})$ and of $H^{2,1}_{BC}(X_{\tempo})$ do not depend on $\tempo$.

\item \textbf{$4$-forms}
It is straightforward to compute
$$ H^{3,1}_{BC}(X_\tempo) \;=\; \C \left\langle \phit{123\bar1},\,\phit{123\bar2} \right\rangle \quad \text{ for } \quad \tempo\in\text{ classes {\itshape (i)}, {\itshape (ii)} and {\itshape (iii)}} $$
and
$$ H^{2,2}_{BC}(X_\tempo) \;=\; \C\left\langle \phit{12\bar1\bar3},\, \phit{12\bar2\bar3},\, \phit{13\bar1\bar2},\, \phit{13\bar1\bar3},\, \phit{13\bar2\bar3},\, \phit{23\bar1\bar2}, \, \phit{23\bar1\bar3},\, \phit{23\bar2\bar3} \right\rangle \quad \text{ for } \quad \tempo\in\text{ class {\itshape (i)}} \;. $$
Moreover, one can check that
$$ H^{2,2}_{BC}(X_\tempo) \;\supseteq\; \C\left\langle \phit{12\bar1\bar3},\, \phit{12\bar2\bar3},\, \phit{13\bar1\bar2},\, \phit{23\bar1\bar2} \right\rangle \quad \text{ for } \quad \tempo\in\text{ classes {\itshape (i)}, {\itshape (ii)} and {\itshape (iii)}} \;,$$
and that no $(2,2)$-form with a non-zero component in $\phit{12\bar1\bar2}$ can be $\tilde\Delta_{{BC}_{J_t}}$-harmonic.
For $H^{2,2}_{BC}(X_\tempo)$ with $\tempo\in\text{ classes {\itshape (ii)} and {\itshape (iii)}}$, we get a new behaviour: there are subclasses in both class {\itshape (ii)} and class {\itshape (iii)}, which can be distinguished by the dimension of $H^{2,2}_{BC}(X_\tempo)$. Indeed, consider $(2,2)$-forms of the type $\psi=A\,\phit{13\bar1\bar3}+B\,\phit{13\bar2\bar3}+C\,\phit{23\bar1\bar3}+D\,\phit{23\bar2\bar3}$; a straightforward computation shows that such a $\psi$ is $\tilde\Delta_{BC_{J_\tempo}}$-harmonic if and only if $A,\,B,\,C,\,D\in\C$ satisfy the linear system
$$
\left(
\begin{array}{cccc}
 \overline{\sigma_{2\bar2}} & -\overline{\sigma_{1\bar2}} & -\overline{\sigma_{2\bar1}} & \overline{\sigma_{1\bar1}} \\
 \sigma_{2\bar2} & -\sigma_{2\bar1} & -\sigma_{1\bar2} & \sigma_{1\bar1}
\end{array}
\right)
\cdot
\left(
\begin{array}{c}
 A \\
 B \\
 C \\
 D
\end{array}
\right)
=
\left(
\begin{array}{c}
 0 \\
 0
\end{array}
\right) \;.
$$
As one can straightforwardly note, the rank of the matrix involved is $0$ for $\tempo\in\text{ class {\itshape (i)}}$, while it is $1$ or $2$ depending on the values of the parameters within class {\itshape (ii)}, or within class {\itshape (iii)}. Therefore
$$ \dim_\C H^{2,2}_{BC}(X_\tempo) \;=\; 7 \quad \text{ for } \quad \tempo\in\text{ subclasses {\itshape (ii.a)} and {\itshape (iii.a)}} $$
and
$$ \dim_\C H^{2,2}_{BC}(X_\tempo) \;=\; 6 \quad \text{ for } \quad \tempo\in\text{ subclasses {\itshape (ii.b)} and {\itshape (iii.b)}} \;.$$

\item \textbf{$5$-forms}
Finally, let us compute $H^{3,2}_{BC}(X_\tempo)$.
It is straightforward to check that
$$ H^{3,2}_{BC}(X_{\tempo}) \;=\; \C\left\langle \phit{123\bar1\bar2},\, \phit{123\bar1\bar3},\, \phit{123\bar2\bar3} \right\rangle \quad \text{ for } \quad \tempo\in\text{ classes {\itshape (i)}, {\itshape (ii)} and {\itshape (iii)}} \;:$$
in particular, it does not depend on $\tempo\in\Delta(\zero,\varepsilon)$.
\end{itemize}

\medskip

We summarize the results of the computations above in the following theorem, \cite[Theorem 5.1]{angella}.
\begin{thm}
 Consider the Iwasawa manifold $\mathbb{I}_3 := \left. \mathbb{H}\left(3;\Z\left[\im\right]\right) \right\backslash \mathbb{H}(3;\C)$ and the family $\left\{X_\tempo=\left(\I_3,\,J_\tempo\right)\right\}_{\tempo\in\Delta(\zero,\varepsilon)}$ of its small deformations, where $\varepsilon>0$ is small enough and $X_\zero=\I_3$. Then the dimensions $h^{p,q}_{BC}:=h^{p,q}_{BC}\left(X_\tempo\right):=\dim_\C H^{p,q}_{BC}\left(X_\tempo\right)=\dim_\C H^{3-p,3-q}_{A}\left(X_\tempo\right)$ does not depend on $\tempo\in\Delta(\zero,\varepsilon)$ whenever $p+q$ is odd or $(p,q)\in\left\{(1,1),\,(3,1),\,(1,3)\right\}$, and they are equal to
$$
\begin{array}{ccccccccccc}
 h^{1,0}_{BC} &=& h^{0,1}_{BC} &=& 2 \;, & & & & & & \\[5pt]
 h^{2,0}_{BC} &=& h^{0,2}_{BC} &\in& \left\{1,\,2,\,3\right\}\;, & \qquad & & & h^{1,1}_{BC} & = & 4\;, \\[5pt]
 h^{3,0}_{BC} &=& h^{0,3}_{BC} &=& 1 \;, & \qquad & h^{2,1}_{BC} & = & h^{1,2}_{BC} & = & 6\;, \\[5pt]
 h^{3,1}_{BC} &=& h^{1,3}_{BC} &=& 2 \;, & \qquad & & & h^{2,2}_{BC} & \in & \left\{6,\,7,\,8\right\} \;, \\[5pt]
 h^{3,2}_{BC} &=& h^{2,3}_{BC} &=& 3 \;. & & & & & &
\end{array}
$$
\end{thm}

\begin{rem}
 As a consequence of the computations above, we notice that the Bott-Chern cohomology yields  a finer classification of the small deformations of $\I_3$ than the Dolbeault cohomology: indeed, note that $\dim_\C H^{2,2}_{BC}(X_\tempo)$ assumes different values according to different parameters in class {\itshape (ii)}, respectively in class {\itshape (iii)}; in a sense, this says that the Bott-Chern cohomology ``carries more informations'' about the complex structure that the Dolbeault one. Note also that most of the dimensions of Bott-Chern cohomology groups are invariant under small deformations: this happens for example for the odd-degree Bott-Chern cohomology groups.
\end{rem}

\begin{center}

\begin{landscape}
\subsection{Dimensions of the cohomologies of the Iwasawa manifold and of its small deformations}\label{subsec:chart}

\smallskip
\begin{center}
\begin{tabular}{c|ccccc}
\toprule
$\mathbf{H^\bullet_{dR}}$ & $\mathbf{b_1}$ & $\mathbf{b_2}$ & $\mathbf{b_3}$ & $\mathbf{b_4}$ & $\mathbf{b_5}$ \\[5pt]
\midrule[0.02em]\addlinespace[1.5pt]\midrule[0.02em]
$\mathbb{I}_3$ and {\itshape (i)}, {\itshape (ii)}, {\itshape (iii)} & $4$ & $8$ & $10$ & $8$ & $4$ \\
\bottomrule
\end{tabular} 
\end{center}
\smallskip

\medskip

\smallskip
\begin{center}
\begin{tabular*}{16.9cm}{c|cc|ccc|cccc|ccc|cc}
\toprule
$\mathbf{H^{\bullet,\bullet}_{\delbar}}$ & $\mathbf{h^{1,0}_{\delbar}}$ & $\mathbf{h^{0,1}_{\delbar}}$ & $\mathbf{h^{2,0}_{\delbar}}$ & $\mathbf{h^{1,1}_{\delbar}}$ & $\mathbf{h^{0,2}_{\delbar}}$ & $\mathbf{h^{3,0}_{\delbar}}$ & $\mathbf{h^{2,1}_{\delbar}}$ & $\mathbf{h^{1,2}_{\delbar}}$ & $\mathbf{h^{0,3}_{\delbar}}$ & $\mathbf{h^{3,1}_{\delbar}}$ & $\mathbf{h^{2,2}_{\delbar}}$ & $\mathbf{h^{1,3}_{\delbar}}$ & $\mathbf{h^{3,2}_{\delbar}}$ & $\mathbf{h^{2,3}_{\delbar}}$ \\[5pt]
\midrule[0.02em]\addlinespace[1.5pt]\midrule[0.02em]
$\I_3$ and \itshape{(i)} & $3$ & $2$ & $3$ & $6$ & $2$ & $1$ & $6$ & $6$ & $1$ & $2$ & $6$ & $3$ & $2$ & $3$ \\[5pt]
\itshape{(ii)} & $2$ & $2$ & $2$ & $5$ & $2$ & $1$ & $5$ & $5$ & $1$ & $2$ & $5$ & $2$ & $2$ & $2$ \\[5pt]
\itshape{(iii)} & $2$ & $2$ & $1$ & $5$ & $2$ & $1$ & $4$ & $4$ & $1$ & $2$ & $5$ & $1$ & $2$ & $2$ \\
\bottomrule
\end{tabular*}
\end{center}
\smallskip

\medskip

\smallskip
\begin{center}
\begin{tabular*}{16.9cm}{c|cc|ccc|cccc|ccc|cc}
\toprule
$\mathbf{H^{\bullet,\bullet}_{\textrm{BC}}}$ & $\mathbf{h^{1,0}_{\textrm{BC}}}$ & $\mathbf{h^{0,1}_{\textrm{BC}}}$ & $\mathbf{h^{2,0}_{\textrm{BC}}}$ & $\mathbf{h^{1,1}_{\textrm{BC}}}$ & $\mathbf{h^{0,2}_{\textrm{BC}}}$ & $\mathbf{h^{3,0}_{\textrm{BC}}}$ & $\mathbf{h^{2,1}_{\textrm{BC}}}$ & $\mathbf{h^{1,2}_{\textrm{BC}}}$ & $\mathbf{h^{0,3}_{\textrm{BC}}}$ & $\mathbf{h^{3,1}_{\textrm{BC}}}$ & $\mathbf{h^{2,2}_{\textrm{BC}}}$ & $\mathbf{h^{1,3}_{\textrm{BC}}}$ & $\mathbf{h^{3,2}_{\textrm{BC}}}$ & $\mathbf{h^{2,3}_{\textrm{BC}}}$ \\[5pt]
\midrule[0.02em]\addlinespace[1.5pt]\midrule[0.02em]
$\I_3$ and \itshape{(i)} & $2$ & $2$ & $3$ & $4$ & $3$ & $1$ & $6$ & $6$ & $1$ & $2$ & $8$ & $2$ & $3$ & $3$ \\[5pt]
\itshape{(ii.a)} & $2$ & $2$ & $2$ & $4$ & $2$ & $1$ & $6$ & $6$ & $1$ & $2$ & $7$ & $2$ & $3$ & $3$ \\[5pt]
\itshape{(ii.b)} & $2$ & $2$ & $2$ & $4$ & $2$ & $1$ & $6$ & $6$ & $1$ & $2$ & $6$ & $2$ & $3$ & $3$ \\[5pt]
\itshape{(iii.a)} & $2$ & $2$ & $1$ & $4$ & $1$ & $1$ & $6$ & $6$ & $1$ & $2$ & $7$ & $2$ & $3$ & $3$ \\[5pt]
\itshape{(iii.b)} & $2$ & $2$ & $1$ & $4$ & $1$ & $1$ & $6$ & $6$ & $1$ & $2$ & $6$ & $2$ & $3$ & $3$ \\
\bottomrule
\end{tabular*}
\end{center}
\smallskip

\medskip

\smallskip
\begin{center}
\begin{tabular*}{16.9cm}{c|cc|ccc|cccc|ccc|cc}
\toprule
$\mathbf{H^{\bullet,\bullet}_{\textrm{A}}}$ & $\mathbf{h^{1,0}_{\textrm{A}}}$ & $\mathbf{h^{0,1}_{\textrm{A}}}$ & $\mathbf{h^{2,0}_{\textrm{A}}}$ & $\mathbf{h^{1,1}_{\textrm{A}}}$ & $\mathbf{h^{0,2}_{\textrm{A}}}$ & $\mathbf{h^{3,0}_{\textrm{A}}}$ & $\mathbf{h^{2,1}_{\textrm{A}}}$ & $\mathbf{h^{1,2}_{\textrm{A}}}$ & $\mathbf{h^{0,3}_{\textrm{A}}}$ & $\mathbf{h^{3,1}_{\textrm{A}}}$ & $\mathbf{h^{2,2}_{\textrm{A}}}$ & $\mathbf{h^{1,3}_{\textrm{A}}}$ & $\mathbf{h^{3,2}_{\textrm{A}}}$ & $\mathbf{h^{2,3}_{\textrm{A}}}$ \\[5pt]
\midrule[0.02em]\addlinespace[1.5pt]\midrule[0.02em]
$\I_3$ and \itshape{(i)} & $3$ & $3$ & $2$ & $8$ & $2$ & $1$ & $6$ & $6$ & $1$ & $3$ & $4$ & $3$ & $2$ & $2$ \\[5pt]
\itshape{(ii.a)} & $3$ & $3$ & $2$ & $7$ & $2$ & $1$ & $6$ & $6$ & $1$ & $2$ & $4$ & $2$ & $2$ & $2$ \\[5pt]
\itshape{(ii.b)} & $3$ & $3$ & $2$ & $6$ & $2$ & $1$ & $6$ & $6$ & $1$ & $2$ & $4$ & $2$ & $2$ & $2$ \\[5pt]
\itshape{(iii.a)} & $3$ & $3$ & $2$ & $7$ & $2$ & $1$ & $6$ & $6$ & $1$ & $1$ & $4$ & $1$ & $2$ & $2$ \\[5pt]
\itshape{(iii.b)} & $3$ & $3$ & $2$ & $6$ & $2$ & $1$ & $6$ & $6$ & $1$ & $1$ & $4$ & $1$ & $2$ & $2$ \\
\bottomrule
\end{tabular*}
\end{center}
\smallskip

\end{landscape}

\end{center}

\section{Cohomology of orbifolds}\label{sec:orbifolds}

The notion of orbifold has been introduced by I. Satake in \cite{satake}, with the name of \emph{V-manifold}, and has been studied, among others, by W.~L. Baily, \cite{baily, baily-2}.

In this section, we start by recalling the main definitions and some classical results concerning complex orbifolds and their cohomology, and we are then interested in their Bott-Chern cohomology. Compact complex orbifolds of the type $\tilde X = \left. X \right\slash G$, where $X$ is a compact complex manifold and $G$ is a finite group of biholomorphisms of $X$, constitute one of the simplest examples of singular spaces: more precisely, we study the Bott-Chern cohomology for such orbifolds, proving that it can be defined using either currents or forms, or also by computing the $G$-invariant $\tilde \Delta_{BC}$-harmonic forms on $X$, Theorem \ref{thm:bc}.

\subsection{Orbifolds and cohomologies}

We first recall some classical definitions and results about orbifolds and their cohomologies, referring to \cite{joyce-red, joyce-ext, satake, baily, baily-2} (see, e.g., \cite[Definition 7.4.3]{joyce-red}).

\begin{defi}[{\cite[Definition 2]{satake}}]
 A \emph{complex orbifold of complex dimension $n$} is a singular complex space of complex dimension $n$ whose singularities are locally isomorphic to quotient singularities $\left. \C^n\right\slash G$, for finite subgroups $G\subset \GL(n;\C)$.
\end{defi}

By definition, an object (e.g., a \emph{differential form}, a \emph{Riemannian metric}, a \emph{Hermitian metric}) \emph{on a complex orbifold $\tilde X$} is defined locally at $x\in\tilde X$ as a $G_x$-invariant object on $\C^n$, where $G_x\subseteq\GL(n;\C)$ is such that $\tilde X$ is locally isomorphic to $\left. \C^n \right\slash G_x$ at $x$.

\medskip

In particular, one gets a differential complex $\left(\wedge^\bullet \tilde X, \, \de\right)$, and a double complex $\left(\wedge^{\bullet,\bullet}\tilde X,\, \del,\, \delbar\right)$. Define the de Rham, Dolbeault, Bott-Chern, and Aeppli cohomology groups of $\tilde X$ respectively as
\begin{eqnarray*}
   H^\bullet_{dR}\left(\tilde X;\R\right) \;:=\; \frac{\ker \de}{\imm\de} \;, &\qquad&
   H^{\bullet,\bullet}_{\delbar}\left(\tilde X\right) \;:=\; \frac{\ker \delbar}{\imm \delbar} \;, \\[5pt]
   H^{\bullet,\bullet}_{BC}\left(\tilde X\right) \;:=\; \frac{\ker \del \cap \ker \delbar}{\imm \del\delbar} \;, &\qquad&
   H^{\bullet,\bullet}_{A}\left(\tilde X\right) \;:=\; \frac{\ker \del\delbar}{\imm \del + \imm \delbar} \;.
\end{eqnarray*}
The structure of double complex of $\left(\wedge^{\bullet, \bullet}\tilde X,\, \del,\, \delbar\right)$ induces naturally a spectral sequence $\left\{\left(E_r^{\bullet,\bullet},\, \de_r\right)\right\}_{r\in\N}$, called \emph{Hodge and Fr\"olicher spectral sequence of $\tilde X$}, such that $E_1^{\bullet,\bullet}\simeq H^{\bullet,\bullet}_{\delbar}\left(\tilde X\right)$ (see, e.g., \cite[\S2.4]{mccleary}). Hence, one has the \emph{Fr\"olicher inequality}, see \cite[Theorem 2]{frolicher},
$$ \sum_{p+q=k} \dim_\C H^{p,q}_\delbar\left(\tilde X \right) \;\geq\; \dim_\C H^k_{dR}\left(\tilde X;\C\right) \;,$$
for any $k\in\N$.

\medskip

Given a Riemannian metric on a complex orbifold $\tilde X$ of complex dimension $n$, one can consider the $\R$-linear Hodge-$*$-operator $*_g\colon \wedge^\bullet \tilde X \to \wedge^{2n-\bullet}\tilde X$, and hence the \kth{2} order self-adjoint elliptic differential operator $\Delta:=\left[\de,\, \de^*\right]:=\de\,\de^*+\de^*\,\de \in \End{\wedge^\bullet \tilde X}$.

Analogously, given a Hermitian metric on a complex orbifold $\tilde X$ of complex dimension $n$, one can consider the $\C$-linear Hodge-$*$-operator $*_g\colon \wedge^{\bullet_1, \bullet_2} \tilde X \to \wedge^{n-\bullet_2, n-\bullet_1}\tilde X$, and hence the \kth{2} order self-adjoint elliptic differential operator $\overline\square:=\left[\delbar,\, \delbar^*\right]:=\delbar\,\delbar^*+\delbar^*\,\delbar \in \End{\wedge^{\bullet,\bullet}\tilde X}$. Furthermore, following \cite[\S2]{schweitzer}, see also \cite[Proposition 5]{kodaira-spencer-3}, one can define the \kth{4} order self-adjoint elliptic differential operators
$$ \tilde \Delta_{BC} \;:=\; \left(\del\delbar\right)\left(\del\delbar\right)^*+\left(\del\delbar\right)^*\left(\del\delbar\right)+\left(\delbar^*\del\right)\left(\delbar^*\del\right)^*+\left(\delbar^*\del\right)^*\left(\delbar^*\del\right)+\delbar^*\delbar+\del^*\del \;\in\; \End{\wedge^{\bullet, \bullet}\tilde X} $$
and
$$ \tilde\Delta_{A} \;:=\; \del\del^*+\delbar\delbar^*+\left(\del\delbar\right)^*\left(\del\delbar\right)+\left(\del\delbar\right)\left(\del\delbar\right)^*+\left(\delbar\del^*\right)^*\left(\delbar\del^*\right)+\left(\delbar\del^*\right)\left(\delbar\del^*\right)^* \;\in\; \End{\wedge^{\bullet, \bullet}\tilde X} \;. $$

As a matter of notation, given a compact complex orbifold $\tilde X$ of complex dimension $n$, denote the constant sheaf with coefficients in $\R$ over $\tilde X$ by $\underline{\R}_{\tilde X}$, the sheaf of germs of smooth functions over $\tilde X$ by $\mathcal{C}^{\infty}_{\tilde X}$, the sheaf of germs of $(p,q)$-forms (for $p,q\in\N$) over $\tilde X$ by $\mathcal{A}^{p,q}_{\tilde X}$, the sheaf of germs of $k$-forms (for $k\in\N$) over $\tilde X$ by $\mathcal{A}^{k}_{\tilde X}$, the sheaf of germs of bidimension-$(p,q)$-currents (for $p,q\in\N$) over $\tilde X$ by $\mathcal{D}_{\tilde X\, p,q}:=:\mathcal{D}^{n-p,n-q}_{\tilde X}$, the sheaf of germs of dimension-$k$-currents (for $k\in\N$) over $\tilde X$ by $\mathcal{D}_{\tilde X\, k}:=:\mathcal{D}^{2n-k}_{\tilde X}$, and the sheaf of holomorphic $p$-forms (for $p\in\N$) over $\tilde X$ by $\Omega^p_{\tilde X}$.

\medskip

The following result, concerning the de Rham cohomology of a complex orbifold, was proven by I. Satake, \cite{satake}, and by W.~L. Baily, \cite{baily}.

\begin{thm}[{\cite[Theorem 1]{satake}, \cite[Theorem H]{baily}}]
 Let $\tilde X$ be a compact complex orbifold of complex dimension $n$.
 There is a canonical isomorphism
 $$ H^\bullet_{dR}\left(\tilde X;\R\right) \;\simeq\; \check H^\bullet\left(\tilde X; \underline{\R}_{\tilde X}\right) \;,$$
 where $\underline{\R}_{\tilde X}$ is the constant sheaf with coefficients in $\R$ over $\tilde X$.

 Furthermore, given a Riemannian metric on $\tilde X$, there is a canonical isomorphism
 $$ H^\bullet_{dR}\left(\tilde X;\R\right) \;\simeq\; \ker \Delta \;.$$
 In particular, the Hodge-$*$-operator induces an isomorphism
 $$ H^\bullet_{dR}\left(\tilde X;\R\right) \;\simeq\; H^{2n-\bullet}_{dR}\left(\tilde X;\R\right) \;.$$
\end{thm}

The isomorphism $H^\bullet_{dR}\left(\tilde X;\R\right) \simeq \ker \Delta$ can be seen as a consequence of a more general decomposition theorem on orbifolds, \cite[Theorem D]{baily}, which holds for \kth{2} order self-adjoint elliptic differential operators. In particular, as regards the Dolbeault cohomology, the following result holds.

\begin{thm}[{\cite[page 807]{baily-2}, \cite[Theorem K]{baily}}]
 Let $\tilde X$ be a compact complex orbifold of complex dimension $n$.
 There is a canonical isomorphism
 $$ H^{\bullet_1,\bullet_2}_{\delbar}\left(\tilde X\right) \;\simeq\; \check H^{\bullet_2} \left(\tilde X; \Omega^{\bullet_1}_{\tilde X}\right) \;,$$
 where $\Omega^p_{\tilde X}$ is the sheaf of holomorphic $p$-forms over $\tilde X$, for $p\in\N$.

 Furthermore, given a Hermitian metric on $X$, there is a canonical isomorphism
 $$ H^{\bullet,\bullet}_{\delbar}\left(\tilde X\right) \;\simeq\; \ker \overline\square \;.$$
 In particular, the Hodge-$*$-operator induces an isomorphism
 $$ H^{\bullet_1,\bullet_2}_{\delbar}\left(\tilde X\right) \;\simeq\; H^{n-\bullet_1,n-\bullet_2}_{\delbar}\left(\tilde X\right) \;.$$
\end{thm}

\subsection{Bott-Chern cohomology of orbifolds of global-quotient-type}

Now, we will reduce to study complex orbifolds of the special type
$$ \tilde X \;=\; \left. X \right\slash G \;,$$
where $X$ is a complex manifold and $G$ is a finite group of biholomorphisms of $X$. Indeed, note that, by the S. Bochner linearization theorem \cite[Theorem 1]{bochner}, see, e.g., \cite[Theorem 2.2.1]{duistermaat-kolk}, see also \cite[Theorem 1.7.2]{raissy-master-thesis}, $\tilde X = \left. X \right\slash G$ is an orbifold according to the above definition.

Orbifolds of global-quotient-type have been considered and studied by D.~D. Joyce in constructing examples of compact $7$-dimensional manifolds with holonomy $G_2$, \cite{joyce-jdg-1-2-g2} and \cite[Chapters 11-12]{joyce-ext}, and examples of compact $8$-dimensional manifolds with holonomy ${\rm Spin}(7)$, \cite{joyce-invent, joyce-jdg-spin7} and \cite[Chapters 13-14]{joyce-ext}. See also \cite{fernandez-munoz, cavalcanti-fernandez-munoz} for the use of orbifolds of global-quotient-type to construct compact $8$-dimensional simply-connected non-formal symplectic manifolds (which do not satisfy, respectively satisfy, the Hard Lefschetz condition), answering to a question by I.~K. Babenko and I.~A. Ta\u{\i}manov, \cite[Problem]{babenko-taimanov}.

Since $G$ is a finite group of biholomorphisms, the singular set of $\tilde X$ is
$$ \Sing\left(\tilde X\right) \;=\; \left\{x\,G\in\left.X\right\slash G 
\st x\in X \text{ and }g\cdot x=x \text{ for some }g\in G\setminus\{\id_X\}\right\} \;.$$

\begin{rem}
Not all orbifolds are global quotients $\left. X \right\slash G$: a counterexample is provided by considering weighted projective spaces, see, e.g., \cite[Definition 6.5.4]{joyce-red}.
\end{rem}

\medskip

In particular, for the sake of completeness, we provide in this special case a straightforward proof of \cite[Theorem 1]{satake} and \cite[Theorem H]{baily} for the de Rham cohomology, and of \cite[page 807]{baily-2} and \cite[Theorem K]{baily} for the Dolbeault cohomology; furthermore, we extend these results to Bott-Chern and Aeppli cohomologies.

\begin{thm}[{\cite[Theorem 1]{satake}, \cite[Theorem H]{baily}}]\label{thm:derham-orbifold}
 Let $\tilde X=\left.X \right\slash G$ be a compact complex orbifold of complex dimension $n$, where $X$ is a complex manifold and $G$ is a finite group of biholomorphisms of $X$.
 There are canonical isomorphisms
 $$ H^\bullet_{dR}\left(\tilde X;\R\right) \;\simeq\; \check H^\bullet\left(\tilde X; \underline{\R}_{\tilde X}\right) \;\simeq\; \frac{\ker \left(\de \colon \correnti^\bullet \tilde X \to \correnti^{\bullet+1}\tilde X\right)}{\imm \left(\de \colon \correnti^{\bullet-1} \tilde X \to \correnti^{\bullet}\tilde X\right)} \;.$$

 Furthermore, given a Riemannian metric on $\tilde X$, there is a canonical isomorphism
 $$ H^\bullet_{dR}\left(\tilde X;\R\right) \;\simeq\; \ker \Delta \;.$$

 In particular, the Hodge-$*$-operator induces an isomorphism
 $$ H^\bullet_{dR}\left(\tilde X;\R\right) \;\simeq\; H^{2n-\bullet}_{dR}\left(\tilde X;\R\right) \;.$$
\end{thm}

\begin{proof}
We claim that
$$ 0 \to \underline{\R}_{\tilde X} \to \left(\mathcal{A}^\bullet_{\tilde X},\, \de\right) \qquad \text{ and } \qquad 0 \to \underline{\R}_{\tilde X} \to \left(\mathcal{D}^\bullet_{\tilde X},\, \de\right) $$
are fine resolutions of the constant sheaf $\underline{\R}_{\tilde X}$. Indeed, take $\phi$ a germ of a $\de$-closed $k$-form on $\tilde X$, with $k\in\N\setminus\{0\}$, that is, a germ of a $G$-invariant $k$-form on $X$; by the Poincaré lemma, see, e.g., \cite[I.1.22]{demailly-agbook}, there exists $\psi$ a germ of a $(k-1)$-form on $X$ such that $\phi=\de\psi$; since $\phi$ is $G$-invariant, one has
$$ \phi \;=\; \frac{1}{\ord G} \sum_{g\in G} g^*\phi \;=\; \frac{1}{\ord G} \sum_{g\in G} g^*\left(\de\psi\right) \;=\; \de\left(\frac{1}{\ord G} \sum_{g\in G} g^*\psi\right) \;, $$
that is, taking the germ of the $G$-invariant $(k-1)$-form
$$ \tilde\psi \;:=\; \frac{1}{\ord G} \sum_{g\in G} g^*\psi $$
on $X$, one gets a germ of a $(k-1)$-form on $\tilde X$ such that $\phi=\de\tilde\psi$. As regards the case $k=0$, it follows straightforwardly since every ($G$-invariant) $\de$-closed function on $X$ is locally constant. The same argument applies for the sheaves of currents, by using the Poincaré lemma for currents, see, e.g., \cite[Theorem I.2.24]{demailly-agbook}. Finally, note that, for every $k\in\N$, the sheaves $\mathcal{A}^k_{\tilde X}$ and $\mathcal{D}^k_{\tilde X}$ are fine: indeed, they are sheaves of $\mathcal{C}^{\infty}_{\tilde X}$-modules over a para-compact space.

Hence, one gets that
\begin{eqnarray*}
\check H^\bullet\left(\tilde X; \underline{\R}_{\tilde X}\right) &\simeq& \underbrace{\frac{\ker \left(\de \colon \wedge^\bullet \tilde X \to \wedge^{\bullet+1}\tilde X\right)}{\imm \left(\de \colon \wedge^{\bullet-1} \tilde X \to \wedge^{\bullet}\tilde X\right)}}_{=:\; H^\bullet_{dR}\left(\tilde X;\R\right)} \\[5pt]
&\simeq& \frac{\ker \left(\de \colon \correnti^\bullet \tilde X \to \correnti^{\bullet+1}\tilde X\right)}{\imm \left(\de \colon \correnti^{\bullet-1} \tilde X \to \correnti^{\bullet}\tilde X\right)} \;,
\end{eqnarray*}
see, e.g., \cite[Corollary IV.4.19, IV.6.4]{demailly-agbook}.

\smallskip

Consider now a Riemannian metric on $\tilde X$, that is, a $G$-invariant Riemannian metric on $X$. Since the elements of $G$ commute with both $\de$ and $\de^*$ (the Riemannian metric being $G$-invariant), and hence with $\Delta$, the decomposition
$$ \wedge^\bullet X \;=\; \ker \Delta \oplus \de \wedge^{\bullet-1} X \oplus \de^* \wedge^{\bullet+1} X $$
induces a decomposition of the space of $G$-invariant forms, namely,
$$ \wedge^\bullet \tilde X \;=\; \ker \Delta \oplus \de \wedge^{\bullet-1} \tilde X \oplus \de^* \wedge^{\bullet+1} \tilde X \;. $$
More precisely, let $\alpha$ be a $G$-invariant form on $X$; considering the decomposition $\alpha:=: h_\alpha + \de\beta+\de^*\gamma$ with $h_\alpha, \beta,\gamma\in\wedge^\bullet X$ such that $\Delta h_\alpha=0$, one has
$$ \alpha \;=\; \frac{1}{\ord G} \sum_{g\in G} g^*\alpha \;=\; \left(\frac{1}{\ord G} \sum_{g\in G} g^*h_\alpha\right) + \de\left(\frac{1}{\ord G} \sum_{g\in G} g^*\beta\right) + \de^*\left(\frac{1}{\ord G} \sum_{g\in G} g^*\gamma\right) \;, $$
where $\frac{1}{\ord G} \sum_{g\in G} g^*h_\alpha,\, \frac{1}{\ord G} \sum_{g\in G} g^*\beta,\, \frac{1}{\ord G} \sum_{g\in G} g^*\gamma \in\wedge^\bullet \tilde X$ and
$$ \Delta\left(\frac{1}{\ord G} \sum_{g\in G} g^*h_\alpha\right) \;=\; \frac{1}{\ord G} \sum_{g\in G} g^*\left(\Delta h_\alpha\right) \;=\; 0 \;.$$

\smallskip

Finally, note that the Hodge-$*$-operator $*\colon \wedge^{\bullet}\tilde X \to \wedge^{2n-\bullet}\tilde X$ sends $\Delta$-harmonic forms to $\Delta$-harmonic forms, and hence it induces an isomorphism
$$ *\colon H^\bullet_{dR}\left(\tilde X;\R\right)\stackrel{\simeq}{\to} H^{2n-\bullet}_{dR}\left(\tilde X;\R\right) \;, $$
concluding the proof.
\end{proof}

A similar argument can be repeated for the Dolbeault cohomology; more precisely, the following result holds.

\begin{thm}[{\cite[page 807]{baily-2}, \cite[Theorem K]{baily}}]\label{thm:dolbeault-orbifold}
 Let $\tilde X=\left.X \right\slash G$ be a compact complex orbifold of complex dimension $n$, where $X$ is a complex manifold and $G$ is a finite group of biholomorphisms of $X$.
 There are canonical isomorphisms
 $$ H^{\bullet_1,\bullet_2}_{\delbar}\left(\tilde X\right) \;\simeq\; \check H^{\bullet_2} \left(\tilde X; \Omega^{\bullet_1}_{\tilde X}\right) \;\simeq\; \frac{\ker\left(\delbar\colon \correnti^{\bullet_1, \bullet_2}\tilde X \to \correnti^{\bullet_1,\bullet_2+1}\tilde X\right)}{\imm\left(\delbar\colon \correnti^{\bullet_1, \bullet_2-1}\tilde X \to \correnti^{\bullet_1,\bullet_2}\tilde X\right)} \;. $$

 Furthermore, given a Hermitian metric on $\tilde X$, there is a canonical isomorphism
 $$ H^{\bullet,\bullet}_{\delbar}\left(\tilde X\right) \;\simeq\; \ker \overline\square \;.$$

 In particular, the Hodge-$*$-operator induces an isomorphism
 $$ H^{\bullet_1,\bullet_2}_{\delbar}\left(\tilde X\right) \;\simeq\; H^{n-\bullet_1,n-\bullet_2}_{\delbar}\left(\tilde X\right) \;.$$
\end{thm}

\begin{proof}
We claim that, for every $p\in\N$,
$$ 0 \to \Omega^p_{\tilde X} \to \left(\mathcal{A}^{p,\bullet}_{\tilde X},\, \delbar\right) \qquad \text{ and } \qquad 0 \to \Omega^p_{\tilde X} \to \left(\mathcal{D}^{p,\bullet}_{\tilde X},\, \delbar\right) $$
are fine resolutions of the constant sheaf $\Omega^p_{\tilde X}$. Indeed, take $\phi$ a germ of a $\delbar$-closed $(p,q)$-form (respectively, bidimension-$(p,q)$-current) on $\tilde X$, with $q\in\N\setminus\{0\}$, that is, a germ of a $G$-invariant $(p,q)$-form (respectively, bidimension-$(p,q)$-current) on $X$; by the Dolbeault and Grothendieck lemma, see, e.g., \cite[I.3.29]{demailly-agbook}, there exists $\psi$ a germ of a $(p,q-1)$-form (respectively, bidimension-$(p,q-1)$-current) on $X$ such that $\phi=\delbar\psi$; since $\phi$ is $G$-invariant, one has
$$ \phi \;=\; \frac{1}{\ord G} \sum_{g\in G} g^*\phi \;=\; \frac{1}{\ord G} \sum_{g\in G} g^*\left(\delbar\psi\right) \;=\; \delbar\left(\frac{1}{\ord G} \sum_{g\in G} g^*\psi\right) \;, $$
that is, taking the germ of the $G$-invariant $(p,q-1)$-form (respectively, bidimension-$(p,q-1)$-current)
$$ \tilde\psi \;:=\; \frac{1}{\ord G} \sum_{g\in G} g^*\psi $$
on $X$, one gets a germ of a $(p,q-1)$-form (respectively, bidimension-$(p,q-1)$-current) on $\tilde X$ such that $\phi=\delbar\tilde\psi$. As regards the case $q=0$, it follows by the fact that every ($G$-invariant) $\delbar$-closed bidimension-$(p,0)$-current on $X$ is locally a holomorphic $p$-form, see, e.g., \cite[I.3.29]{demailly-agbook}. Finally, note that, for every $q\in\N$, the sheaves $\mathcal{A}^{p,q}_{\tilde X}$ and $\mathcal{D}^{p,q}_{\tilde X}$ are fine: indeed, they are sheaves of $\left(\mathcal{C}^{\infty}_{\tilde X}\otimes_\R\C\right)$-modules over a para-compact space.

Hence, one gets that
\begin{eqnarray*}
\check H^{p,\bullet}\left(\tilde X; \Omega^p_{\tilde X}\right) &\simeq& \frac{\ker \left(\delbar \colon \wedge^{p,\bullet} \tilde X \to \wedge^{p,\bullet+1}\tilde X\right)}{\imm \left(\delbar \colon \wedge^{p,\bullet-1} \tilde X \to \wedge^{p,\bullet}\tilde X\right)} \\[5pt]
&\simeq& \frac{\ker \left(\delbar \colon \correnti^{p,\bullet} \tilde X \to \correnti^{p,\bullet+1}\tilde X\right)}{\imm \left(\delbar \colon \correnti^{p,\bullet-1} \tilde X \to \correnti^{p,\bullet}\tilde X\right)} \;,
\end{eqnarray*}
see, e.g., \cite[Corollary IV.4.19, IV.6.4]{demailly-agbook}.

\smallskip

Consider now a Hermitian metric on $\tilde X$, that is, a $G$-invariant Hermitian metric on $X$. Since the elements of $G$ commute with both $\delbar$ and $\delbar^*$ (the Hermitian metric being $G$-invariant), and hence with $\overline\square$, the decomposition
$$ \wedge^{\bullet,\bullet} X \;=\; \ker \overline\square \oplus \delbar \wedge^{\bullet,\bullet-1} X \oplus \delbar^* \wedge^{\bullet,\bullet+1} X $$
induces a decomposition on the space of $G$-invariant forms, namely,
$$ \wedge^{\bullet,\bullet} \tilde X \;=\; \ker \overline\square \oplus \delbar \wedge^{\bullet,\bullet-1} \tilde X \oplus \delbar^* \wedge^{\bullet,\bullet+1} \tilde X \;. $$
More precisely, let $\alpha$ be a $G$-invariant form on $X$; considering the decomposition $\alpha:=: h_\alpha + \delbar\beta+\delbar^*\gamma$ with $h_\alpha, \beta,\gamma\in\wedge^{\bullet,\bullet} X$ such that $\overline\square h_\alpha=0$, one has
$$ \alpha \;=\; \frac{1}{\ord G} \sum_{g\in G} g^*\alpha \;=\; \left(\frac{1}{\ord G} \sum_{g\in G} g^*h_\alpha\right) + \delbar\left(\frac{1}{\ord G} \sum_{g\in G} g^*\beta\right) + \delbar^*\left(\frac{1}{\ord G} \sum_{g\in G} g^*\gamma\right) \;, $$
where $\frac{1}{\ord G} \sum_{g\in G} g^*h_\alpha,\, \frac{1}{\ord G} \sum_{g\in G} g^*\beta,\, \frac{1}{\ord G} \sum_{g\in G} g^*\gamma \in\wedge^{\bullet,\bullet} \tilde X$ and
$$ \overline\square \left(\frac{1}{\ord G} \sum_{g\in G} g^*h_\alpha\right) \;=\; \frac{1}{\ord G} \sum_{g\in G} g^*\left(\overline\square h_\alpha\right) \;=\; 0 \;.$$

\smallskip

Finally, note that the Hodge-$*$-operator $*\colon \wedge^{\bullet_1,\bullet_2}\tilde X \to \wedge^{n-\bullet_2,n-\bullet_1}\tilde X$ sends $\overline\square$-harmonic forms to $\square$-harmonic forms, where $\square:=\left[\del,\,\del^*\right]:=\del\del^*+\del^*\del\in\End{\wedge^{\bullet,\bullet}\tilde X}$, and hence it induces an isomorphism
$$ *\colon H^{\bullet_1,\bullet_2}_{\delbar}\left(\tilde X\right)\stackrel{\simeq}{\to} H^{n-\bullet_1,n-\bullet_2}_{\delbar}\left(\tilde X\right) \;,$$
concluding the proof.
\end{proof}

\medskip

Finally, as done in Theorem \ref{thm:derham-orbifold} and Theorem \ref{thm:dolbeault-orbifold} for the de Rham cohomology and, respectively, the Dolbeault cohomology, we provide the following result, concerning Bott-Chern and Aeppli cohomologies of compact complex orbifolds of global-quotient-type.

\begin{thm}\label{thm:bc}
 Let $\tilde X=\left.X \right\slash G$ be a compact complex orbifold of complex dimension $n$, where $X$ is a complex manifold and $G$ is a finite group of biholomorphisms of $X$.
 For any $p,q\in\N$, there are canonical isomorphisms
 \begin{equation}\label{eq:bc-orbifolds}
  H^{p,q}_{BC}\left(\tilde X\right) \;\simeq\; \frac{\ker\left(\del \colon \correnti^{p,q}\tilde X \to \correnti^{p+1,q}\tilde X\right) \cap \ker \left(\delbar\colon \correnti^{p,q}\tilde X \to \correnti^{p,q+1}\tilde X\right)}{\imm\left(\del\delbar \colon \correnti^{p-1,q-1}\tilde X \to \correnti^{p,q}\tilde X\right)} \;.
 \end{equation}

 Furthermore, given a Hermitian metric on $\tilde X$, there are canonical isomorphisms
 $$ H^{\bullet,\bullet}_{BC}\left(\tilde X\right) \;\simeq\; \ker \tilde\Delta_{BC} \qquad \text{ and } \qquad H^{\bullet,\bullet}_{A}\left(\tilde X\right) \;\simeq\; \ker\tilde\Delta_A \;. $$

 In particular, the Hodge-$*$-operator induces an isomorphism
 $$ H^{\bullet_1,\bullet_2}_{BC}\left(\tilde X\right) \;\simeq\; H^{n-\bullet_2,n-\bullet_1}_{A}\left(\tilde X\right) \;.$$
\end{thm}

\begin{proof}
We use the same argument as in the proof of \cite[Theorem 3.7]{angella} to show that, since the de Rham cohomology and the Dolbeault cohomology of $\tilde X$ can be computed using either differential forms or currents, the same holds true for the Bott-Chern and the Aeppli cohomologies.

Indeed, note that, for any $p,q\in\N$, one has the exact sequence
\begin{eqnarray*}
\lefteqn{0 \to \frac{\imm \left(\de \colon \left(\correnti^{p+q-1}\tilde X\otimes_\R\C\right)\to \left(\correnti^{p+q}\tilde X\otimes_\R\C\right)\right)\cap \correnti^{p,q}\tilde X}{\imm \left(\del\delbar\colon \correnti^{p-1,q-1}\tilde X \to \correnti^{p,q}\tilde X\right)} } \\[5pt]
 && \to \frac{\ker \left(\de\colon \correnti^{p,q}\tilde X\to \correnti^{p+1,q+1}\tilde X\right)}{\imm \left(\del\delbar\colon \correnti^{p-1,q-1}\tilde X\to \correnti^{p,q}\tilde X\right)}\to \frac{\ker \left(\de \colon \left(\correnti^{p+q}\tilde X\otimes_\R\C\right)\to \left(\correnti^{p+q+1}\tilde X \otimes_\R\C\right)\right)}{\imm \left(\de\colon \left(\correnti^{p+q-1}\tilde X \otimes_\R\C\right) \to \left(\correnti^{p+q}\tilde X \otimes_\R \C\right)\right)} \;,
\end{eqnarray*}
where the maps are induced by the identity.
By \cite[Theorem 1]{satake}, see Theorem \ref{thm:derham-orbifold}, one has
$$ \frac{\ker \left(\de \colon \left(\correnti^{p+q}\tilde X\otimes_\R\C\right)\to \left(\correnti^{p+q+1}\tilde X \otimes_\R\C\right)\right)}{\imm \left(\de\colon \left(\correnti^{p+q-1}\tilde X \otimes_\R\C\right) \to \left(\correnti^{p+q}\tilde X \otimes_\R \C\right)\right)} \;\simeq\; \frac{\ker \left(\de \colon \left(\wedge^{p+q}\tilde X\otimes_\R\C\right)\to \left(\wedge^{p+q+1}\tilde X \otimes_\R\C\right)\right)}{\imm \left(\de\colon \left(\wedge^{p+q-1}\tilde X \otimes_\R\C\right) \to \left(\wedge^{p+q}\tilde X \otimes_\R \C\right)\right)} \;,$$
therefore it suffices to prove that the space
$$ \frac{\imm \left(\de \colon \left(\correnti^{p+q-1}\tilde X\otimes_\R\C\right)\to \left(\correnti^{p+q}\tilde X\otimes_\R\C\right)\right)\cap \correnti^{p,q}\tilde X}{\imm \left(\del\delbar\colon \correnti^{p-1,q-1}\tilde X \to \correnti^{p,q}\tilde X\right)}  $$
can be computed using just differential forms on $\tilde X$.

Firstly, we note that, since, by \cite[page 807]{baily-2}, see Theorem \ref{thm:dolbeault-orbifold},
$$ \frac{\ker \left(\delbar \colon \correnti^{p,q}\tilde X \to \correnti^{p,q+1}\tilde X\right)}{\imm \left(\delbar\colon \correnti^{p,q-1}\tilde X \to \correnti^{p,q}\tilde X \right)} \;\simeq\; \frac{\ker \left(\delbar \colon \wedge^{p,q}\tilde X \to \wedge^{p,q+1}\tilde X\right)}{\imm \left(\delbar\colon \wedge^{p,q-1}\tilde X \to \wedge^{p,q}\tilde X \right)} \;,$$
one has that, if $\psi\in\wedge^{r,s}\tilde X$ is a $\delbar$-closed differential form, then every solution $\phi\in\correnti^{r,s-1}$ of $\delbar\phi=\psi$ is a differential form up to $\delbar$-exact terms.
Indeed, since $[\psi]=0$ in $\frac{\ker\delbar\cap \correnti^{r,s}\tilde X}{\imm\delbar}$ and hence in $\frac{\ker\delbar\cap \wedge^{r,s}\tilde X}{\imm\delbar}$, there is a differential form $\alpha\in\wedge^{r,s-1}\tilde X$ such that $\psi=\delbar\alpha$. Hence, $\phi-\alpha\in\correnti^{r,s-1}\tilde X$ defines a class in $\frac{\ker\delbar \cap \correnti^{r,s-1}\tilde X}{\imm\delbar}\simeq \frac{\ker\delbar \cap \wedge^{r,s-1}\tilde X}{\imm\delbar}$, and hence $\phi-\alpha$ is a differential form up to a $\delbar$-exact form, and so $\phi$ is.

By conjugation, if $\psi\in\wedge^{r,s}\tilde X$ is a $\del$-closed differential form, then every solution $\phi\in\correnti^{r-1,s}$ of $\del\phi=\psi$ is a differential form up to $\del$-exact terms.

Now, let
$$ \omega^{p,q}\;=\;\de\eta\mod\imm\del\delbar\;\in\;\frac{\imm\de\cap\correnti^{p,q}X}{\imm\del\delbar} \;. $$
Decomposing $\eta=:\sum_{p,q}\eta^{p,q}$ in pure-type components, where $\eta^{p,q}\in\correnti^{p,q}\tilde X$, the previous equality is equivalent to the system
$$
\left\{
\begin{array}{cccccccc}
 && \del\eta^{p+q-1,0} &=&0 & \mod \imm\del\delbar && \\[5pt]
\delbar\eta^{p+q-\ell,\ell-1} &+& \del\eta^{p+q-\ell-1,\ell} &=& 0 &\mod\imm\del\delbar & \text{ for } & \ell\in\{1,\ldots,q-1\} \\[5pt]
\delbar\eta^{p,q-1} &+& \del\eta^{p-1,q} &=& \omega^{p,q} & \mod\imm\del\delbar && \\[5pt]
\delbar\eta^{\ell,p+q-\ell-1} &+& \del\eta^{\ell-1,p+q-\ell} &=& 0 &\mod\imm\del\delbar & \text{ for } & \ell\in\{1,\ldots,p-1\} \\[5pt]
\delbar\eta^{0,p+q-1} &&&=&0 &\mod\imm\del\delbar &&
\end{array}
\right. \;.
$$
By the above argument, we may suppose that, for $\ell\in\{0,\ldots, p-1\}$, the currents $\eta^{\ell,p+q-\ell-1}$ are differential forms: indeed, they are differential forms up to $\delbar$-exact terms, but $\delbar$-exact terms give no contribution in the system, which is modulo $\imm\del\delbar$. Analogously, we may suppose that, for $\ell\in\{0,\ldots,q-1\}$, the currents $\eta^{p+q-\ell-1,\ell}$ are differential forms. Then we may suppose that $\omega^{p,q}=\delbar\eta^{p,q-1}+\del\eta^{p-1,q}$ is a differential form. Hence \eqref{eq:bc-orbifolds} is proven.

\smallskip

Now, we prove that, fixed a $G$-invariant Hermitian metric on $\tilde X$, the Bott-Chern cohomology of $\tilde X$ is isomorphic to the space of $\tilde\Delta_{BC}$-harmonic $G$-invariant forms on $X$. Indeed, since the elements of $G$ commute with $\del$, $\delbar$, $\del^*$, and $\delbar^*$, and hence with $\tilde\Delta_{BC}$, the following decomposition, \cite[Théorème 2.2]{schweitzer},
$$ \wedge^{\bullet,\bullet} X \;=\; \ker \tilde\Delta_{BC} \oplus \del\delbar\wedge^{\bullet-1,\bullet-1}X \oplus \left(\del^*\wedge^{\bullet+1,\bullet}X + \delbar^*\wedge^{\bullet,\bullet+1}X \right) $$
induces a decomposition
$$ \wedge^{\bullet,\bullet} \tilde X \;=\; \ker \tilde\Delta_{BC} \oplus \del\delbar\wedge^{\bullet-1,\bullet-1}\tilde X \oplus \left(\del^*\wedge^{\bullet+1,\bullet}\tilde X + \delbar^*\wedge^{\bullet,\bullet+1}\tilde X \right) \;. $$
More precisely, let $\alpha\in\wedge^{\bullet,\bullet}\tilde X$, that is, $\alpha$ is a $G$-invariant form on $X$; if $\alpha$ has a decomposition $\alpha=h_{\alpha}+\del\delbar\beta+\left(\del^*\gamma+\delbar^*\eta\right)$ with $h_\alpha,\beta,\gamma,\eta\in\wedge^{\bullet,\bullet}X$ such that $\tilde\Delta_{BC}h_\alpha=0$, then one has
\begin{eqnarray*}
\alpha \;=\; \frac{1}{\ord G}\sum_{g\in G}g^*\alpha &=& \left(\frac{1}{\ord G}\sum_{g\in G}g^*h_{\alpha}\right) +\del\delbar\left(\frac{1}{\ord G}\sum_{g\in G}g^*\beta\right) \\[5pt]
 && + \left(\del^*\left(\frac{1}{\ord G}\sum_{g\in G}g^*\gamma\right)+\delbar^*\left(\eta\frac{1}{\ord G}\sum_{g\in G}g^*\right)\right) \;,
\end{eqnarray*}
where $\frac{1}{\ord G}\sum_{g\in G}g^*h_{\alpha},\, \frac{1}{\ord G}\sum_{g\in G}g^*\beta,\, \frac{1}{\ord G}\sum_{g\in G}g^*\gamma,\, \eta\frac{1}{\ord G}\sum_{g\in G}g^*\in\wedge^{\bullet,\bullet}\tilde X$ and
$$ \tilde\Delta_{BC}\left(\frac{1}{\ord G}\sum_{g\in G}g^*h_{\alpha}\right) \;=\; \frac{1}{\ord G}\sum_{g\in G}g^*\left(\tilde\Delta_{BC}h_{\alpha}\right) \;=\; 0 \;.$$

As regards the Aeppli cohomology, one has the decomposition, \cite[\S2.c]{schweitzer},
$$ \wedge^{\bullet,\bullet}X \;=\; \ker\tilde\Delta_A \oplus \left(\del\wedge^{\bullet-1,\bullet}X + \delbar\wedge^{\bullet,\bullet-1} X\right) \oplus \left(\del\delbar\right)^*\wedge^{\bullet+1,\bullet+1} X \;,$$
and hence the decomposition
$$ \wedge^{\bullet,\bullet} \tilde X \;=\; \ker\tilde\Delta_A \oplus \left(\del\wedge^{\bullet-1,\bullet}\tilde X + \delbar\wedge^{\bullet,\bullet-1} \tilde X\right) \oplus \left(\del\delbar\right)^*\wedge^{\bullet+1,\bullet+1} \tilde X \;,$$
from which one gets the isomorphism $H^{\bullet,\bullet}_A\left(\tilde X\right)\simeq\ker\tilde\Delta_A$.

\smallskip

Finally, note that the Hodge-$*$-operator $*\colon \wedge^{\bullet_1,\bullet_2}\tilde X \to \wedge^{n-\bullet_2,n-\bullet_1}\tilde X$ sends $\tilde\Delta_{BC}$-harmonic forms to $\tilde\Delta_{A}$-harmonic forms, and hence it induces an isomorphism
$$ *\colon H^{\bullet_1,\bullet_2}_{BC}\left(\tilde X\right)\stackrel{\simeq}{\to} H^{n-\bullet_2,n-\bullet_1}_{A}\left(\tilde X\right) \;, $$
concluding the proof.
\end{proof}

\begin{rem}\label{rem:hypoercoh-bc-orbifolds}
 We note that another proof of the isomorphism
 $$ H^{p,q}_{BC}\left(\tilde X\right) \;\simeq\; \frac{\ker\left(\del \colon \correnti^{p,q}\tilde X \to \correnti^{p+1,q}\tilde X\right) \cap \ker \left(\delbar\colon \correnti^{p,q}\tilde X \to \correnti^{p,q+1}\tilde X\right)}{\imm\left(\del\delbar \colon \correnti^{p-1,q-1}\tilde X \to \correnti^{p,q}\tilde X\right)} \;, $$
 and a proof of the isomorphism
 $$ H^{p,q}_{A}\left(\tilde X\right) \;\simeq\; \frac{\ker\left(\del\delbar \colon \correnti^{p,q}\tilde X \to \correnti^{p+1,q+1}\tilde X\right)}{\imm\left(\del \colon \correnti^{p-1,q}\tilde X \to \correnti^{p,q}\tilde X\right) + \imm\left(\delbar \colon \correnti^{p,q-1}\tilde X \to \correnti^{p,q}\tilde X\right)} $$
 follow from the sheaf-theoretic interpretation of the Bott-Chern and Aeppli cohomologies, developed by J.-P. Demailly,
 \cite[\S VI.12.1]{demailly-agbook} and M. Schweitzer, \cite[\S4]{schweitzer}, see also \cite[\S3.2]{kooistra}.

 We recall that, for any $p,q\in\N$, the complex $\left(\mathcal{L}^\bullet_{\tilde X\, p,q},\, \de_{\mathcal{L}^\bullet_{\tilde X\, p,q}}\right)$ of sheaves is defined as
$$
\left(\mathcal{L}^\bullet_{\tilde X\, p,q},\, \de_{\mathcal{L}^\bullet_{\tilde X\, p,q}}\right) \;: \quad
   \mathcal{A}^{0,0}_{\tilde X}
     \stackrel{\pr \circ \de}{\to}
   \bigoplus_{\substack{r+s=1 \\ r<p,\, s<q}} \mathcal{A}^{r,s}_{\tilde X}
     \to
   \cdots
     \stackrel{\pr \circ \de}{\to}
   \bigoplus_{\substack{r+s=p+q-2\\ r<p,\, s<q}} \mathcal{A}^{r,s}_{\tilde X}
     \stackrel{\del\delbar}{\to}
   \bigoplus_{\substack{r+s=p+q\\ r\geq p,\, s\geq q}} \mathcal{A}^{r,s}_{\tilde X}
     \stackrel{\de}{\to}
   \bigoplus_{\substack{r+s=p+q\\ r\geq p,\, s\geq q}} \mathcal{A}^{r,s}_{\tilde X}
     \to
   \cdots \;,
$$
and the complex $\left(\mathcal{M}^\bullet_{\tilde X\, p,q},\, \de_{\mathcal{M}^\bullet_{\tilde X\, p,q}}\right)$ of sheaves is defined as
$$
\left(\mathcal{M}^\bullet_{\tilde X\, p,q},\, \de_{\mathcal{M}^\bullet_{\tilde X\, p,q}}\right) \;: \quad
   \mathcal{D}^{0,0}_{\tilde X}
     \stackrel{\pr \circ \de}{\to}
   \bigoplus_{\substack{r+s=1 \\ r<p,\, s<q}} \mathcal{D}^{r,s}_{\tilde X}
     \to
   \cdots
     \stackrel{\pr \circ \de}{\to}
   \bigoplus_{\substack{r+s=p+q-2\\ r<p,\, s<q}} \mathcal{D}^{r,s}_{\tilde X}
     \stackrel{\del\delbar}{\to}
   \bigoplus_{\substack{r+s=p+q\\ r\geq p,\, s\geq q}} \mathcal{D}^{r,s}_{\tilde X}
     \stackrel{\de}{\to}
   \bigoplus_{\substack{r+s=p+q\\ r\geq p,\, s\geq q}} \mathcal{D}^{r,s}_{\tilde X}
     \to
   \cdots \;,
$$
where $\pr$ denotes the projection onto the appropriate space.

By the Poincaré lemma (see, e.g., \cite[I.1.22, Theorem I.2.24]{demailly-agbook}) and the Dolbeault and Grothendieck lemma (see, e.g., \cite[I.3.29]{demailly-agbook}), one gets M. Schweitzer's lemma \cite[Lemme 4.1]{schweitzer}, which can be extended also to the context of orbifolds by using the same trick as in the proof of Theorem \ref{thm:derham-orbifold} and Theorem \ref{thm:dolbeault-orbifold}; this allows to prove that the map
$$ \left(\mathcal{L}^\bullet_{\tilde X\, p,q},\, \de_{\mathcal{L}^\bullet_{\tilde X\, p,q}}\right) \to \left(\mathcal{M}^\bullet_{\tilde X\, p,q},\, \de_{\mathcal{M}^\bullet_{\tilde X\, p,q}}\right) $$
of complexes of sheaves is a quasi-isomorphism, and hence, see, e.g., \cite[Corollary IV.12.6]{demailly-agbook}, for every $\ell\in\N$,
$$ \mathbb{H}^{\ell} \left(\tilde X;\,\left(\mathcal{L}^{\bullet}_{\tilde X\, p,q},\, \de_{\mathcal{L}^{\bullet}_{\tilde X\, p,q}}\right)\right) \;\simeq\; \mathbb{H}^{\ell} \left(\tilde X;\,\left(\mathcal{M}^{\bullet}_{\tilde X\, p,q},\, \de_{\mathcal{L}^{\bullet}_{\tilde X\, p,q}}\right)\right) \;.$$

Since, for every $k\in\N$, the sheaves $\mathcal{L}^k_{\tilde X\, p,q}$ and $\mathcal{M}^k_{\tilde X\, p,q}$ are fine (indeed, they are sheaves of $\left(\mathcal{C}^{\infty}_{\tilde X}\otimes_\R\C\right)$-modules over a para-compact space), one has, see, e.g., \cite[Corollary IV.4.19, (IV.12.9)]{demailly-agbook},
$$
\mathbb{H}^{p+q-1} \left(\tilde X;\,\left(\mathcal{L}^{\bullet}_{\tilde X\, p,q},\, \de_{\mathcal{L}^{\bullet}_{\tilde X\, p,q}}\right)\right) 
\;\simeq\;
\frac{\ker \left(\del\colon \wedge^{p,q}\tilde X \to \wedge^{p+1,q}\tilde X\right) \cap \ker \left(\delbar \colon \wedge^{p,q}\tilde X \to \wedge^{p,q+1}\tilde X \right)}{\imm\left(\del\delbar \colon \wedge^{p-1,q-1}\tilde X \to \wedge^{p,q}\tilde X\right)}
$$
and
$$
\mathbb{H}^{p+q-1} \left(\tilde X;\,\left(\mathcal{M}^{\bullet}_{\tilde X\, p,q},\, \de_{\mathcal{L}^{\bullet}_{\tilde X\, p,q}}\right)\right)
\;\simeq\;
\frac{\ker \left(\del\colon \correnti^{p,q}\tilde X \to \correnti^{p+1,q}\tilde X\right) \cap \ker \left(\delbar \colon \correnti^{p,q}\tilde X \to \correnti^{p,q+1}\tilde X \right)}{\imm\left(\del\delbar \colon \correnti^{p-1,q-1}\tilde X \to \correnti^{p,q}\tilde X\right)} \;,
$$
and
$$
\mathbb{H}^{p+q-2} \left(\tilde X;\,\left(\mathcal{L}^{\bullet}_{\tilde X\, p,q},\, \de_{\mathcal{L}^{\bullet}_{\tilde X\, p,q}}\right)\right)
\;\simeq\;
\frac{\ker \left(\del\delbar\colon \wedge^{p-1,q-1}\tilde X \to \wedge^{p,q}\tilde X\right)}{\imm\left(\del \colon \wedge^{p-2,q-1}\tilde X \to \wedge^{p-1,q-1}\tilde X\right) + \imm\left(\delbar \colon \wedge^{p-1,q-2}\tilde X \to \wedge^{p-1,q-1}\tilde X\right)} $$
and
$$
\mathbb{H}^{p+q-2} \left(\tilde X;\,\left(\mathcal{M}^{\bullet}_{\tilde X\, p,q},\, \de_{\mathcal{L}^{\bullet}_{\tilde X\, p,q}}\right)\right)
\;\simeq\;
\frac{\ker \left(\del\delbar\colon \correnti^{p-1,q-1}\tilde X \to \correnti^{p,q}\tilde X\right)}{\imm\left(\del \colon \correnti^{p-2,q-1}\tilde X \to \correnti^{p-1,q-1}\tilde X\right) + \imm\left(\delbar \colon \correnti^{p-1,q-2}\tilde X \to \correnti^{p-1,q-1}\tilde X\right)} \;,
$$
proving the stated isomorphisms.
\end{rem}

\chapter{Cohomology of almost-complex manifolds}\label{chapt:almost-complex}

Let $X$ be a $2n$-dimensional (differentiable) manifold endowed with an almost-complex structure $J$. Note that if $J$ is not integrable, then the Dolbeault cohomology is not defined. In this section, we are concerned with studying some subgroups of the de Rham cohomology related to the almost-complex structure: these subgroups have been introduced by T.-J. Li and W. Zhang in \cite{li-zhang}, in order to study the relation between the compatible and the tamed symplectic cones on a compact almost-complex manifold, with the aim to throw light on a question by S.~K. Donaldson, \cite[Question 2]{donaldson} (see \S\ref{subsec:symplectic-cones}), and it would be interesting to consider them as a sort of counterpart of the Dolbeault cohomology groups in the non-integrable (or at least in the non-K\"ahler) case, see \cite[Lemma 2.15, Theorem 2.16]{draghici-li-zhang}. In particular, we are interested in studying when they let a splitting of the de Rham cohomology, and their relations with cones of metric structures.

More precisely, in \S\ref{sec:definition} we introduce the notions of \Cpf\ and \pf\ almost-complex structures, setting the notation and proving some useful relations between them. In \S\ref{sec:classes-Cpf}, we study \Cpf ness on several classes of (almost-)complex manifolds, e.g., solvmanifolds, semi-K\"ahler manifolds, almost-K\"ahler manifolds. In \S\ref{sec:deformations-cpf}, we study the behaviour of \Cpf ness under small deformations of the complex structure and along curves of almost-complex structures, investigating properties of stability, and of semi-continuity for the dimensions of the invariant, and anti-invariant subgroups of the de Rham cohomology with respect to the almost-complex structure. In \S\ref{sec:cones}, we study the cone of semi-K\"ahler structures on a compact almost-complex manifold and, in particular, we compare the cones of balanced metrics and of strongly-Gauduchon metrics on a compact complex manifold.

The results of this chapter have been obtained jointly with A. Tomassini, in \cite{angella-tomassini-1, angella-tomassini-2}, and with A. Tomassini and W. Zhang, in \cite{angella-tomassini-zhang}.

\section{Subgroups of the de Rham (co)homology of an almost-complex manifold}\label{sec:definition}
In this section, we set the notation concerning \Cpf\ and \pf\ almost-complex structures, as introduced in \cite{li-zhang}, and we study the relations between \Cpf ness and \pf ness.

\subsection{\Cpf\ and \pf\ almost-complex structures}\label{subsec:cpf-almost-complex}
In this section, we start by fixing some preliminary notation and recalling some definitions; then we briefly review some results to motivate the study of these topics, see Remark \ref{rem:storia}, which will be further discussed in the next sections.

\medskip

Let $S\subseteq \N\times\N$ and define
$$
H^S_J(X;\R) \;:=\; \left\{ \left[\alpha\right]\in H^{\bullet}_{dR}(X;\R) \st \alpha\in\left(\bigoplus_{(p,q)\in S} \wedge^{p,q}X\right)\cap\wedge^\bullet X \right\} \;;
$$
note that a real differential form $\alpha$ with a component of type $(p,q)$ has also a component of type $(q,p)$, and hence we are interested in studying the sets $S$ such that whenever $(p,q)\in S$, also $(q,p)\in S$.
As a matter of notation, we will usually list the elements of $S$ instead of writing $S$ itself.

Note that, for every $k\in\N$, one has
$$ \sum_{\substack{p+q=k\\p\leq q}} H^{(p,q),(q,p)}_J(X;\R) \;\subseteq\; H^k_{dR}(X;\R) \;, $$
but, in general, the sum is neither direct nor the equality holds: several examples of these facts will be provided in the sequel.

\medskip

The subgroups $H^{(2,0),(0,2)}_J(X;\R)$ and $H^{(1,1)}_J(X;\R)$ of $H^2_{dR}(X;\R)$ are of special interest for their interpretation as the $J$-anti-invariant, respectively, $J$-invariant part of the second de Rham cohomology group. Indeed, note that the endomorphism $J\lfloor_{\wedge^2X}\in\End\left(\wedge^\bullet X\right)$ naturally extending $J\in\End(TX)$ (that is, $J\alpha:=\alpha\left(J\sspace,\,J\ssspace\right)$ for every $\alpha\in\wedge^2X$) satisfies $\left(J\lfloor_{\wedge^2X}\right)^2=\id_{\wedge^2X}$; hence, one has the splitting
$$ \wedge^2 X \;=\; \wedge^+_JX \oplus \wedge^-_JX \;,$$
where, for $\pm\in\{+,\, -\}$,
$$ \wedge^\pm_JX \;:=\; \left\{ \alpha \in \wedge^2X \st J\alpha=\pm\alpha \right\} \;.$$
Since $H^2_{dR}(X;\R)$ contains, in particular, the classes represented by the symplectic forms, and $H^{(1,1)}_{J}(X;\R)$ contains, in particular, the classes represented by the $(1,1)$-forms associated to the Hermitian metrics on $X$, in \cite{li-zhang}, T.-J. Li and W. Zhang were interested in studying the \emph{$J$-invariant subgroup} of $H^2_{dR}(X;\R)$, namely,
$$ H^+_J(X) \;:=\; H^{(1,1)}_J(X;\R) \;=\; \left\{\left[\alpha\right]\in H^2_{dR}(X;\R) \st J\alpha=\alpha\right\} \;, $$
and the \emph{$J$-anti-invariant subgroup} of $H^2_{dR}(X;\R)$, namely,
$$ H^-_J(X) \;:=\; H^{(2,0),(0,2)}_J(X;\R) \;=\; \left\{\left[\alpha\right]\in H^2_{dR}(X;\R) \st J\alpha=-\alpha\right\} \;. $$
Note that, as in the general case, one has that
$$ H^+_J(X) + H^-_J(X) \;\subseteq\; H^2_{dR}(X;\R) $$
but, in general, the sum is neither direct nor equal to $H^2_{dR}(X;\R)$. The following definition, by T.-J. Li and W. Zhang, singles out the almost-complex structures whose subgroups $H^+_J(X)$ and $H^-_J(X)$ provide a decomposition of $H^2_{dR}(X;\R)$.

\begin{defi}[{\cite[Definition 2.2, Definition 2.3, Lemma 2.2]{li-zhang}}]
 An almost-complex structure $J$ on a manifold $X$ is said to be
\begin{itemize}
 \item \emph{\Cp} if $H^-_J(X) \cap H^+_J(X) = \left\{0\right\}$;
 \item \emph{\Cf} if $H^-_J(X) + H^+_J(X) = H^2_{dR}(X;\R)$;
 \item \emph{\Cpf} if it is both \Cp\ and \Cf, i.e., if the following cohomology decomposition holds:
$$ H^2_{dR}(X;\R) \;=\; H^-_J(X) \oplus H^+_J(X) \;.$$
\end{itemize}
\end{defi}

\medskip

We will also use the following definition, which is a natural generalization of the notion of \Cpf ness to higher degree cohomology groups.

\begin{defi}
 Let $X$ be a manifold endowed with an almost-complex structure $J$, and fix $k\in\N$. Consider $H^k_{dR}(X;\R)\supseteq\sum_{\substack{p+q=k\\p\leq q}}H^{(p,q),(q,p)}_J(X;\R)$:
  \begin{itemize}
   \item if
     $$ \bigoplus_{\substack{p+q=k\\p\leq q}} H^{(p,q),(q,p)}_J(X;\R) \;\subseteq\; H^k_{dR}(X;\R) $$
     (namely, the sum is direct), then $J$ is called \emph{\Cp\ at the \kth{k} stage};
   \item if
     $$ H^k_{dR}(X;\R) \;=\; \sum_{\substack{p+q=k\\p\leq q}}H^{(p,q),(q,p)}_J(X;\R) \;,$$
     then $J$ is called \emph{\Cf\ at the \kth{k} stage};
   \item if $J$ is both \Cp\ at the \kth{k} stage and \Cf\ at the \kth{k} stage, that is,
     $$ H^k_{dR}(X;\R) \;=\; \bigoplus_{\substack{p+q=k\\p\leq q}}H^{(p,q),(q,p)}_J(X;\R) \;;$$
     then $J$ is called \emph{\Cpf\ at the \kth{k} stage}.
  \end{itemize}
\end{defi}

\medskip

Analogous definitions can be given for the de Rham cohomology with complex coefficients.
More precisely, let $S\subseteq \N\times\N$ and define
$$ H^S_J(X;\C) \;:=\; \left\{ \left[\alpha\right]\in H^{\bullet}_{dR}(X;\C) \st \alpha\in\bigoplus_{(p,q)\in S} \wedge^{p,q}X \right\} $$
(as previously, we will usually list the elements of $S$ instead of writing $S$ itself); with such notation, one has in particular that $H^S_J(X;\R) = H^S_J(X;\C)\cap H^{\bullet}_{dR}(X;\R)$.

\begin{rem}
 Note that, when $X$ is a compact manifold endowed with an integrable almost-complex structure $J$, then, for any $(p,q)\in\N\times\N$,
 $$ H^{(p,q)}_J(X;\C) \;=\; \imm\left(H^{p,q}_{BC}\left(X\right)\to H^{p+q}_{dR}\left(X;\C\right)\right) \;,$$
 where the map $H^{p,q}_{BC}(X)\to H^{p+q}_{dR}(X;\C)$ is the one induced by the identity (note that $\ker\del\cap\ker\delbar\subseteq\ker\de$ and $\imm\del\delbar\subseteq\imm\de$). Indeed, any $\de$-closed $(p,q)$-form is both $\del$-closed and $\delbar$-closed.
\end{rem}

Note that, for every $k\in\N$, one has
$$ \sum_{p+q=k} H^{(p,q)}_J(X;\C) \;\subseteq\; H^k_{dR}(X;\C) \;,$$
but, in general, the sum is neither direct nor the equality holds. We can then give the following definition.

\begin{defi}
 Let $X$ be a manifold endowed with an almost-complex structure $J$, and fix $k\in\N$. Consider $H^k_{dR}(X;\C)\supseteq\sum_{p+q=k}H^{(p,q)}_J(X;\C)$:
  \begin{itemize}
   \item if
     $$ \bigoplus_{p+q=k} H^{(p,q)}_J(X;\C) \;\subseteq\; H^k_{dR}(X;\C) $$
     (namely, the sum is direct), then $J$ is called \emph{complex-\Cp\ at the \kth{k} stage};
   \item if
     $$ H^k_{dR}(X;\C) \;=\; \sum_{p+q=k}H^{(p,q)}_J(X;\C) \;,$$
     then $J$ is called \emph{complex-\Cf\ at the \kth{k} stage};
   \item if $J$ is both complex-\Cp\ at the \kth{k} stage and complex-\Cf\ at the \kth{k} stage, that is,
     $$ H^k_{dR}(X;\C) \;=\; \bigoplus_{p+q=k}H^{(p,q)}_J(X;\C) \;;$$
     then $J$ is called \emph{complex-\Cpf\ at the \kth{k} stage}.
  \end{itemize}
\end{defi}

\begin{rem}\label{rem:cpf-complex-cpf}
In general, being complex-\Cf\ at the \kth{2} stage is a stronger condition that being \Cf. Furthermore, if $J$ is integrable, then being complex-\Cpf\ at the \kth{2} stage is stronger than being \Cpf.
More precisely, for any (possibly non-integrable) almost-complex structure $J$, it holds, \cite[Lemma 2.11]{draghici-li-zhang},
$$
\left\{
\begin{array}{rcl}
 H^+_J(X) &=& H^{(1,1)}_J(X;\C)\cap H^2_{dR}(X;\R) \\[5pt]
 H^{(1,1)}_J(X;\C) &=& H^{+}_J(X) \otimes_\R \C
\end{array}
\right. \;,
$$
and
$$ H^{(2,0)}_J(X;\C)+H^{(0,2)}_J(X;\C) \;\subseteq\; H^{-}_J(X) \otimes_\R \C \;, $$
and, if $J$ is integrable, it holds
$$
\left\{
\begin{array}{l}
 H^-_J(X) \;=\; \left(H^{(2,0)}_J(X;\C)+H^{(0,2)}_J(X;\C)\right)\cap H^2_{dR}(X;\R) \\[5pt]
 H^{(2,0)}_J(X;\C)+H^{(0,2)}_J(X;\C) \;=\; H^{-}_J(X) \otimes_\R \C
\end{array}
\right. \;;
$$
indeed, $\de\wedge^{2,0}X\subseteq \wedge^{3,0}X\oplus \wedge^{2,1}X$ and $\de\wedge^{0,2}X\subseteq \wedge^{1,2}X\oplus\wedge^{0,3}X$.
(Compare also \cite[Lemma 2.12]{draghici-li-zhang} for further results in the case of $4$-dimensional manifolds.)

Note also that, if $J$ is \Cp, then
$$
H^{(1,1)}_J(X;\R)\,\cap\,\left(H^{(2,0)}_J(X;\R) \,+\, H^{(0,2)}_J(X;\R)\right)\,=\,\left\{0\right\} \;.
$$
\end{rem}

\medskip

The construction of the subgroups $H^{S}_J(X;\R)\subseteq H^\bullet_{dR}(X;\R)$ and the notion of \Cpf\ almost-complex structures can be repeated using the complex of currents $\left(\correnti_\bullet X:=:\correnti^{2n-\bullet}X,\, \de\right)$ instead of the complex of differential forms $\left(\wedge^\bullet X,\, \de\right)$ and the de Rham homology $H_\bullet^{dR}(X;\R)$ instead of the de Rham cohomology $H^\bullet_{dR}(X;\R)$. (We refer to \S\ref{sec:currents} for notations and references concerning currents and de Rham homology.)

As in the smooth case, accordingly to T.-J. Li and W. Zhang, \cite{li-zhang}, given $S\subseteq \N\times\N$, let
$$ H_S^J(X;\R) \;:=\; \left\{ \left[\alpha\right]\in H^{dR}_\bullet(X;\C) \st \alpha\in \left(\bigoplus_{(p,q)\in S}\correnti_{p,q}X\right) \cap \correnti_\bullet X\right\} \;.$$

In particular, the almost-complex structures on $X$ for which $H_{(2,0),(0,2)}^J(X;\R)$ and $H_{(1,1)}^J(X;\R)$ provide a decomposition of $H^{dR}_{2}(X;\R)$ are emphasized by the following definition by T.-J. Li and W. Zhang.

\begin{defi}[{\cite[Definition 2.15, Lemma 2.16]{li-zhang}}]
 An almost-complex structure $J$ on a manifold $X$ is said to be:
\begin{itemize}
 \item \emph{\p} if
$$ H_{(2,0),(0,2)}^J(X;\R)\;\cap\; H_{(1,1)}^J(X;\R) \;=\; \left\{0\right\} \;; $$
 \item \emph{\f} if
$$ H_{(2,0),(0,2)}^J(X;\R)\;+\; H_{(1,1)}^J(X;\R) \;=\; H^{dR}_2(X;\R)\;; $$
 \item \emph{\pf} if it is both \p\ and \f, i.e., if the following decomposition holds:
$$ H_{(2,0),(0,2)}^J(X;\R)\,\oplus\, H_{(1,1)}^J(X;\R) \,=\, H^{dR}_2(X;\R) \;.$$
\end{itemize}
\end{defi}

\medskip

The following are natural generalizations of the notion of \pf ness.

\begin{defi}
 Let $X$ be a manifold endowed with an almost-complex structure $J$, and fix $k\in\N$. Consider $H^k_{dR}(X;\R)\supseteq\sum_{\substack{p+q=k\\p\leq q}}H_{(p,q),(q,p)}^J(X;\R)$:
  \begin{itemize}
   \item if
     $$ \bigoplus_{\substack{p+q=k\\p\leq q}} H_{(p,q),(q,p)}^J(X;\R) \;\subseteq\; H_k^{dR}(X;\R) $$
     (namely, the sum is direct), then $J$ is called \emph{\p\ at the \kth{k} stage};
   \item if
     $$ H_k^{dR}(X;\R) \;=\; \sum_{\substack{p+q=k\\p\leq q}}H_{(p,q),(q,p)}^J(X;\R) \;,$$
     then $J$ is called \emph{\f\ at the \kth{k} stage};
   \item if $J$ is both \p\ at the \kth{k} stage and \f\ at the \kth{k} stage, that is,
     $$ H_k^{dR}(X;\R) \;=\; \bigoplus_{\substack{p+q=k\\p\leq q}}H_{(p,q),(q,p)}^J(X;\R) \;;$$
     then $J$ is called \emph{\pf\ at the \kth{k} stage}.
  \end{itemize}
\end{defi}

As regards de Rham homology with complex coefficients, given $S\subseteq \N\times\N$, let
$$ H_S^J(X;\C) \;:=\; \left\{ \left[\alpha\right]\in H^{dR}_\bullet(X;\C) \st \alpha\in\bigoplus_{(p,q)\in S}\correnti_{p,q}X\right\} \;,$$
so that $H_S^J(X;\R) = H_S^J(X;\C)\cap H^{dR}_\bullet(X;\R)$.

\begin{defi}
 Let $X$ be a manifold endowed with an almost-complex structure $J$, and fix $k\in\N$. Consider $H_k^{dR}(X;\C)\supseteq\sum_{p+q=k}H_{(p,q)}^J(X;\C)$:
  \begin{itemize}
   \item if
     $$ \bigoplus_{p+q=k} H_{(p,q)}^J(X;\C) \;\subseteq\; H_k^{dR}(X;\C) $$
     (namely, the sum is direct), then $J$ is called \emph{complex-\p\ at the \kth{k} stage};
   \item if
     $$ H_k^{dR}(X;\C) \;=\; \sum_{p+q=k}H_{(p,q)}^J(X;\C) \;,$$
     then $J$ is called \emph{complex-\f\ at the \kth{k} stage};
   \item if $J$ is both complex-\p\ at the \kth{k} stage and complex-\f\ at the \kth{k} stage, that is,
     $$ H_k^{dR}(X;\C) \;=\; \bigoplus_{p+q=k}H_{(p,q)}^J(X;\C) \;;$$
     then $J$ is called \emph{complex-\pf\ at the \kth{k} stage}.
  \end{itemize}
\end{defi}

\medskip

\begin{rem}\label{rem:storia}
 The study of the subgroups $H^{(p,q),(q,p)}_J(X;\R)$ and the notion of \Cpf\ almost-complex structure have been introduced by T.-J. Li and W. Zhang in \cite{li-zhang}, in order to study the relations between the compatible and the tamed symplectic cones on a compact almost-complex manifold, and inspired by a question by S.~K. Donaldson, \cite[Question 2]{donaldson}: whether, on a compact $4$-dimensional manifold endowed with an almost-complex structure $J$ tamed by a symplectic form, there exists also a symplectic form compatible with $J$, see \S\ref{subsec:symplectic-cones}. In \cite{draghici-li-zhang}, T. Dr\v{a}ghici, T.-J. Li, and W. Zhang investigated the $4$-dimensional case, proving, in particular, that every almost-complex structure on a compact $4$-dimensional manifold is \Cpf; they also obtained further results for $4$-dimensional almost-complex manifolds in \cite{draghici-li-zhang-2}, where they studied the dimensions of the subgroups $H^+_J(X)$ and $H^-_J(X)$. In \cite{fino-tomassini}, A. Fino 
and 
A. Tomassini studied the \Cpf ness in connection with other properties on almost-complex manifolds: in particular, by studying almost-complex solvmanifolds, they provided the first explicit example of a non-\Cpf\ almost-complex structure. Jointly with A. Tomassini, we studied in \cite{angella-tomassini-1} the behaviour of \Cpf ness under small deformations of the complex structure or along curves of almost-complex structures, proving in particular its instability. In \cite{angella-tomassini-2} we continued the study of the cohomological properties related to the existence of an almost-complex structure, focusing, in particular, on the study of the cone of semi-K\"ahler structures on a compact semi-K\"ahler manifold. In \cite{angella-tomassini-zhang}, jointly with A. Tomassini and W. Zhang, we further studied cohomological properties of almost-K\"ahler manifolds, especially in relation with W. Zhang's Lefschetz-type property; in particular, an example of a non-\Cf\ almost-K\"ahler structure on a compact 
manifold is provided. In \cite{draghici-zhang}, T. Dr\v{a}ghici and W. Zhang reformulated the S.~K. Donaldson ``tamed to compatible'' question in terms of spaces of exact forms, proving, in particular, that an almost-complex structure $J$ on a compact $4$-dimensional manifold admits a compatible symplectic form if and only if it admits tamed symplectic forms with any arbitrarily given $J$-anti-invariant component. Q. Tan, H. Wang, Y. Zhang, and P. Zhu, in \cite{tan-wang-zhang-zhu}, continued the study of the dimension of the $J$-anti-invariant subgroup $H^-_J(X)$ of the de Rham cohomology of a compact almost-complex manifold, considering almost-complex structures being metric related or fundamental form 
related, showing, for example, that $\dim_\R H^-_J(X)=0$ for a generic almost-complex structure $J$ on a compact $4$-dimensional manifold, as conjectured by T. Dr\v{a}ghici, T.-J. Li, and W. Zhang, \cite[Conjecture 2.4]{draghici-li-zhang-2}. For further results on the study of $J$-anti-invariant forms and $J$-anti-invariant de Rham cohomology classes on a (possibly non-compact) manifold endowed with an almost-complex structure $J$, see  \cite{hind-medori-tomassini} by R.~K. Hind, C. Medori, and A. Tomassini, where a result concerning analytic continuation for $J$-anti-invariant forms is proven.
In \cite{li-tomassini}, T.-J. Li and A. Tomassini studied the analogue of the above problems for linear
(possibly non-integrable) complex structures on $4$-dimensional unimodular Lie algebras; in particular, they proved that an analogue of the decomposition in \cite[Theorem 2.3]{draghici-li-zhang} holds for every $4$-dimensional unimodular Lie algebra endowed with a linear (possibly non-integrable) complex structure; furthermore, they considered the linear counterpart of Donaldson's ``tamed to compatible'' question, and of the tamed and compatible symplectic cones, studying, in particular, a sufficient condition on a $4$-dimensional Lie algebra $\mathfrak{g}$ (which holds, for example, for $4$-dimensional unimodular Lie algebras) in order that a linear (possibly non-integrable) complex structure admits a taming linear symplectic form if and only if it admits a compatible linear symplectic form. The paper \cite{draghici-li-zhang-survey} by T. Dr\v{a}ghici, T.-J. Li, and W. Zhang furnishes a survey on the known results concerning the subgroups $H^+_J(X)$ and $H^-_J(X)$, especially in dimension $4$, and their 
application to S.~K. Donaldson's ``tamed to compatible'' question.
\end{rem}

\subsection{Relations between \Cpf ness and \pf ness}
The following result summarizes the relations between \Cpf ness and \pf ness, see \cite[Theorem 2.1]{angella-tomassini-1}, see also \cite[Proposition 2.5]{li-zhang}, and between complex-\Cpf ness and complex-\pf ness. (Analogous results will be proven in Proposition \ref{prop:para-cpf-pf} for almost-$\mathbf{D}$-complex structures in the sense of F.~R. Harvey and H.~B. Lawson, and in Proposition \ref{prop:full-k-pure-2n-k} for symplectic structures.)

\begin{thm}[{see \cite[Proposition 2.5]{li-zhang}}]\label{thm:implicazioni}
 Let $J$ be an almost-complex structure on a compact $2n$-dimensional manifold $X$.
 The following relations between (complex-)\Cpf\ and (complex-)\pf\ notions hold: for any $k\in\N$,
$$
\xymatrix{
\text{\Cf\ at the \kth{k} stage}\ar@{=>}[r]\ar@{=>}[d] & \text{\p\ at the \kth{k} stage}\ar@{=>}[d] \\
\text{\f\ at the \kth{\left(2n-k\right)} stage}\ar@{=>}[r] & \text{\Cp\ at the \kth{\left(2n-k\right)} stage} \;,
}
$$
and
$$
\xymatrix{
\text{complex-\Cf\ at the \kth{k} stage}\ar@{=>}[r]\ar@{=>}[d] & \text{complex-\p\ at the \kth{k} stage}\ar@{=>}[d] \\
\text{complex-\f\ at the \kth{\left(2n-k\right)} stage}\ar@{=>}[r] & \text{complex-\Cp\ at the \kth{\left(2n-k\right)} stage} \;.
}
$$
\end{thm}

\begin{proof}
 The horizontal implications follow by considering the non-degenerate duality pairing
 $$ \left\langle \sspace,\,\ssspace \right\rangle\colon H^\bullet_{dR}(X;\R) \times H_\bullet^{dR}(X;\R) \to \R \;,
    \qquad \text{ respectively }\;
    \left\langle \sspace,\,\ssspace \right\rangle\colon H^\bullet_{dR}(X;\C) \times H_\bullet^{dR}(X;\C) \to \C \;,
 $$
 and noting that, for any $p,q\in\N$,
 \begin{eqnarray*}
 \lefteqn{\ker \left\langle H^{(p,q),(p,q)}_J(X;\R), \, \sspace \right\rangle \;\supseteq\; \sum_{\left\{(r,s),(s,r)\right\}\neq\left\{(p,q),(q,p)\right\}} H_{(r,s),(s,r)}^J(X;\R)} \\[5pt]
 &&\text{and }\qquad
 \ker \left\langle \sspace, \, H_{(p,q),(q,p)}^J(X;\R) \right\rangle \;\supseteq\; \sum_{\left\{(r,s),(s,r)\right\}\neq\left\{(p,q),(q,p)\right\}} H^{(r,s),(s,r)}_J(X;\R) \;,
 \end{eqnarray*}
 respectively
 $$ \ker \left\langle H^{(p,q)}_J(X;\C), \, \sspace \right\rangle \;\supseteq\; \sum_{(r,s)\neq(p,q)} H_{(r,s)}^J(X;\C) \qquad \text{ and } \qquad \ker \left\langle \sspace, \, H_{(p,q)}^J(X;\C) \right\rangle \;\supseteq\; \sum_{(r,s)\neq(p,q)} H^{(r,s)}_J(X;\C) \;.$$

 As an example, we give the details to prove that if $J$ is \Cf\ at the \kth{k} stage then it is also pure at the \kth{k} stage, when $k=2$. Let
 $$ \mathfrak{c} \;\in\;  H_{(2,0),(0,2)}^J(X;\R) \,\cap\, H_{(1,1)}^J(X;\R) \;,$$
 with $\mathfrak{c}\,\neq\,\left[0\right]$.
 Hence,
 $$ \left\langle \mathfrak{c},\, \sspace\right\rangle\lfloor_{H^{(2,0),(0,2)}_J(X;\R)} \;=\;0 \qquad \text{ and } \qquad \left\langle \mathfrak{c},\, \sspace\right\rangle\lfloor_{H^{(1,1)}_J(X;\R)} \;=\; 0 \;;$$
 since $J$ is \Cf, it follows that $\left\langle \mathfrak{c},\, \sspace\right\rangle\lfloor_{H^{2}_{dR}(X;\R)}=0$, and hence $\mathfrak{c}\,=\,\left[0\right]$.

 To prove the vertical implications, it is enough to note that the quasi-isomorphism $T_{\sspace}\colon \wedge^\bullet X \to \correnti_{2n-\bullet}X$ defined as $T_\varphi:=\int_X \varphi\wedge\sspace$ (see \S\ref{sec:currents}) induces an injective map
 $$ H^{(p,q),(q,p)}_J(X;\R) \to H_{(n-p,n-q), (n-q,n-p)}^J(X;\R) \;, \qquad \text{ respectively }\qquad  H^{(p,q)}_J(X;\C) \to H_{(n-p,n-q)}^J(X;\C) \;,$$
 for any $p,q\in\N$.
\end{proof}

\begin{rem}
 On a compact $2n$-dimensional manifold $X$ endowed with an almost-complex structure $J$,
 further linkings between $H^2_{dR}(X;\R)$ and $H^{2n-2}_{dR}(X;\R)$ could provide further relations between \Cpf\ and \pf\ notions: for example, A. Fino and A. Tomassini proved in \cite[Theorem 3.7]{fino-tomassini} that, given a $J$-Hermitian metric $g$ on $X$, if there exists a basis of $g$-harmonic representatives for $H^2_{dR}(X;\R)$ being of pure type with respect to $J$, then $J$ is both \Cpf\ and \pf. Furthermore, A. Fino and A. Tomassini proved in \cite[Theorem 4.1]{fino-tomassini} that, given a $J$-compatible symplectic form $\omega$ on $X$ satisfying the Hard Lefschetz Condition (that is, the map $\left[\omega^k\right]\cp\sspace \colon \; H^{n-k}_{dR}(X;\R) \stackrel{\simeq}{\to} H^{n+k}_{dR}(X;\R)$ is an isomorphism for every $k\in\N$), if $J$ is \Cpf, then $J$ is also \pf\ (compare Proposition \ref{prop:duality-semi-kahler} for a similar result).
\end{rem}

\medskip

Setting $2n=4$ and $k=2$ in Theorem \ref{thm:implicazioni}, it follows that, on compact $4$-dimensional almost-complex manifolds, \Cf ness implies \Cp ness. The following result states that, for higher dimensional manifolds, \Cp ness and \Cf ness are not, in general, related properties, \cite[Proposition 1.4]{angella-tomassini-2}.

\begin{prop}
\label{prop:Cf-Cp-non-related}
 There exist both examples of compact manifolds endowed with almost-complex structures being \Cf\ and non-\Cp, and examples of compact manifolds endowed with almost-complex structures being \Cp\ and non-\Cf.
\end{prop}

\begin{proof}
The proof follows from the following examples, \cite[Example 1.2, Example 1.3]{angella-tomassini-2}.

\paragrafod{1}{Being \Cf\ does not imply being \Cp}
Take a nilmanifold $N_1$ with associated Lie algebra
$$ \mathfrak{h}_{16} \;:=\; \left( 0^3,\;12,\;14,\;24 \right) \;. $$
Consider the left-invariant complex structure on $N_1$ whose space of $(1,0)$-forms is generated, as a $\mathcal{C}^\infty\left(N_1;\C\right)$-module, by
$$
\left\{
\begin{array}{rcl}
 \varphi^1 &:=& e^1+\im e^2 \\[5pt]
 \varphi^2 &:=& e^3+\im e^4 \\[5pt]
 \varphi^3 &:=& e^5+\im e^6
\end{array}
\right. \;.
$$
Writing the structure equations in terms of $\left\{\varphi^1,\,\varphi^2,\,\varphi^3\right\}$,
$$
\left\{
\begin{array}{rcl}
 2\,\de\varphi^1 &=& 0 \\[5pt]
 2\,\de\varphi^2 &=& \varphi^{1\bar1} \\[5pt]
 2\,\de\varphi^3 &=& -\im \varphi^{12}+\im\varphi^{1\bar2}
\end{array}
\right. \;,
$$
the integrability condition is easily verified.

K. Nomizu's theorem \cite[Theorem 1]{nomizu} makes the computation of the cohomology straightforward: in fact, listing the harmonic representatives with respect to the left-invariant Hermitian metric $g:=\sum_j\varphi^j\odot \bar\varphi^j$ instead of their classes, one finds
$$ H^2_{dR}(N_1;\C) \;=\; \C \left\langle \varphi^{13},\;\varphi^{\bar1\bar3}\right\rangle \,\oplus\, \C\left\langle \varphi^{1\bar3}-
\varphi^{3\bar1} \right\rangle \,\oplus\, \C \left\langle \varphi^{12}+\varphi^{1\bar2},\;\varphi^{2\bar1}-\varphi^{\bar1\bar2}
\right\rangle \;,$$
where
$$ H^{(2,0),(0,2)}_J\left(N_1;\C\right) \;=\; \C \left\langle \varphi^{13},\;\varphi^{\bar1\bar3}\right\rangle \,\oplus\, \C \left\langle \varphi^{12}+\varphi^{1\bar2},\;
\varphi^{2\bar1}-\varphi^{\bar1\bar2} \right\rangle $$
and
$$ H^{(1,1)}_J\left(N_1;\C\right) \;=\; \C\left\langle \varphi^{1\bar3}-\varphi^{3\bar1} \right\rangle \,\oplus\, \C \left\langle \varphi^{12}+\varphi^{1\bar2},\;
\varphi^{2\bar1}-\varphi^{\bar1\bar2} \right\rangle \;.$$
In particular, $J$ is a \Cf, non-\Cp\ complex structure.

\paragrafod{2}{Being \Cp\ does not imply being \Cf}
 Take a nilmanifold $N_2$ with associated Lie algebra
$$ \mathfrak{h}_{2} \;:=\; \left( 0^4,\;12,\;34 \right) \;. $$
and consider on it the left-invariant complex structure given requiring that the forms
$$
\left\{
\begin{array}{rcl}
 \varphi^1 &:=& e^1+\im e^2 \\[5pt]
 \varphi^2 &:=& e^3+\im e^4 \\[5pt]
 \varphi^3 &:=& e^5+\im e^6
\end{array}
\right.
$$
are of type $(1,0)$.

The integrability condition follows from the structure equations
$$
\left\{
\begin{array}{rcl}
 2\,\de\varphi^1 &=& 0 \\[5pt]
 2\,\de\varphi^2 &=& 0 \\[5pt]
 2\,\de\varphi^3 &=& \im \varphi^{1\bar1}-\im\varphi^{2\bar2}
\end{array}
\right. \;.
$$

K. Nomizu's theorem \cite[Theorem 1]{nomizu} gives
$$
H^2_{dR}\left(N_2;\C\right) \;=\; \C \left\langle \varphi^{12},\;\varphi^{\bar1\bar2} \right\rangle \,\oplus\, \C \left\langle \varphi^{1\bar 2},\;
\varphi^{2\bar1} \right\rangle \oplus \C\left\langle \varphi^{13}+\varphi^{1\bar3},\;\varphi^{3\bar1}-\varphi^{\bar1\bar3},\;\varphi^{3\bar2}-\varphi^{\bar2\bar3},\; \varphi^{23}-\varphi^{2\bar3} \right\rangle \;,
$$
where
$$ H^{(2,0),(0,2)}_J\left(N_2;\C\right) \;=\; \C \left\langle \varphi^{12},\;\varphi^{\bar1\bar2} \right\rangle $$
and
$$ H^{(1,1)}_J\left(N_2;\C\right) \;=\; \C \left\langle \varphi^{1\bar 2},\;\varphi^{2\bar1} \right\rangle \;; $$
this can be proven arguing as follows: with respect to the left-invariant Hermitian metric $g:=\sum_j \varphi^j\odot \bar\varphi^j$, one computes
$$ \del^*\,\varphi^{13}\;=\;\del^*\,\varphi^{23}\;=\;\del^*\,\varphi^{12}\;=\;0 \;,$$
that is, $\varphi^{13}$, $\varphi^{12}$ and $\varphi^{23}$ are $g$-orthogonal to the space $\del\wedge^{1,0}N_2$; in the same way,
one computes
$$ \del^*\,\varphi^{1\bar2}\;=\;\delbar^*\,\varphi^{1\bar2}\;=\;\del^*\,\varphi^{1\bar3}\;=\;\delbar^*\,\varphi^{1\bar3}\;=\;0 $$
(compare also Proposition \ref{prop:linear-cpf-invariant-cpf-J}).
In particular, $J$ is a \Cp, non-\Cf\ complex structure.
\end{proof}

\section{\Cpf ness for special manifolds}\label{sec:classes-Cpf}

In this section, we study the property of being \Cpf\ on special classes of (almost-)complex manifolds. After recalling some motivating results by T. Dr\v{a}ghici, T.-J. Li, and W. Zhang, we study \Cpf ness for left-invariant complex-structures on solvmanifolds, providing some examples in dimension $4$ or higher; furthermore, we consider almost-complex manifolds endowed with special metric structures, namely, \emph{semi-K\"ahler}, and \emph{almost-K\"ahler} structures.

\subsection{Special classes of \Cpf\ (almost-)complex manifolds}
In this section, we recall some results by T. Dr\v{a}ghici, T.-J. Li, and W. Zhang, providing classes of \Cpf\ and \pf\ (almost-)complex manifolds. They could be considered as motivations to study \Cpf ness: in fact, \cite[Lemma 2.15, Theorem 2.16]{draghici-li-zhang} suggests that the subgroups $H^{(\bullet,\bullet)}_J(X;\C)$ can be viewed as a generalization of the Dolbeault cohomology groups for non-K\"ahler, and non-integrable, almost-complex manifolds $X$. On the other hand, \cite[Theorem 2.3]{draghici-li-zhang} states that, on a compact $4$-dimensional almost-complex manifold $X$, the subgroups $H^+_J(X)$ and $H^-_J(X)$ induce always a decomposition of $H^2_{dR}(X;\R)$: this could be intended as a generalization of the Hodge decomposition theorem for compact $4$-dimensional almost-complex manifolds.

\medskip

According to the following result, the groups $H^{(\bullet,\bullet)}_J(X;\C)$ can be considered as the counterpart of the Dolbeault cohomology groups in the non-K\"ahler and non-integrable cases.

\begin{thm}[{\cite[Lemma 2.15, Theorem 2.16]{draghici-li-zhang}}]\label{thm:frolicher-Cinf}
Let $X$ be a compact complex manifold. If the Hodge and Fr\"{o}licher spectral sequence degenerates at the first step and the natural filtration associated with the structure of double complex of $\left(\wedge^{\bullet,\bullet}X,\,\del,\,\delbar\right)$ induces a Hodge decomposition of weight $k$ on $H^k_{dR}(X;\C)$ for some $k\in\N$, then $X$ is complex-\Cpf\ at the \kth{k} stage, and
$$ H^{(p,q)}_J(X;\C) \simeq H^{p,q}_{\delbar}(X) $$
for every $p,q\in\N$ such that $p+q=k$.
\end{thm}

A corollary of \cite[Lemma 2.15, Theorem 2.16]{draghici-li-zhang} is the following result.

\begin{cor}[{\cite[Proposition 2.1]{li-zhang}, \cite[Theorem 2.16, Proposition 2.17]{draghici-li-zhang}}]\label{cor:frolicher-Cinf}
One has that:
\begin{enumerate}[\itshape (i)]
 \item every compact complex surface is complex-\Cpf\ at the \kth{2} stage, and hence, in particular, \Cpf\ and \pf;
 \item every compact complex manifold satisfying the $\del\delbar$-Lemma is complex-\Cpf\ at every stage, and hence complex-\pf\ at every stage;
 \item every compact complex manifold admitting a K\"{a}hler structure is complex-\Cpf\ at every stage, and hence complex-\pf\ at every stage.
\end{enumerate}
\end{cor}

\begin{proof}
As regards the complex-\Cf ness at the \kth{2} stage for compact complex surfaces, one has that the assumptions of Theorem \ref{thm:frolicher-Cinf} with $k=2$ hold by \cite[Theorem IV.2.8, Proposition IV.2.9]{barth-hulek-peters-vandeven}.

As regards the complex-\Cf ness at every stage for compact complex manifolds satisfying the $\del\delbar$-Lemma, one has that the assumptions of Theorem \ref{thm:frolicher-Cinf} for any $k\in\N$ are satisfied by \cite[5.21]{deligne-griffiths-morgan-sullivan}.

As regards the complex-\Cf ness at every stage for compact K\"ahler manifolds, one has that a compact complex manifold admitting a K\"{a}hler metric satisfies the $\del\delbar$-Lemma, \cite[Lemma 5.11]{deligne-griffiths-morgan-sullivan}.

Finally, the other statements follow from Remark \ref{rem:cpf-complex-cpf} and Theorem \ref{thm:implicazioni}.
\end{proof}

Actually, T. Dr\v{a}ghici, T.-J. Li, and W. Zhang proved in \cite{draghici-li-zhang} the following result, which one can consider as a sort of Hodge decomposition theorem in the non-K\"ahler case.

\begin{thm}[{\cite[Theorem 2.3]{draghici-li-zhang}}]
 Every almost-complex structure on a compact $4$-dimensional manifold is \Cpf\ and \pf.
\end{thm}

\begin{proof}
The proof of the previous theorem rests on the very special properties of $4$-dimensional manifolds. For the sake of completeness, we recall here the argument by T. Dr\v{a}ghici, T.-J. Li, and W. Zhang in \cite{draghici-li-zhang}. Firstly, note that, by Theorem \ref{thm:implicazioni}, it suffices to prove that an almost-complex structure $J$ on a compact $4$-dimensional manifold is \Cf. Suppose that $J$ is not \Cf. Fix a Hermitian metric $g$ on $X$, and denote its associated $(1,1)$-form by $\omega$. Recall that the Hodge-$*$-operator $*_g\lfloor_{\wedge^2 X}\colon \wedge^2 X\to \wedge^2X$ satisfies $\left(*_g\lfloor_{\wedge^2 X}\right)^2=\id_{\wedge^2 X}$, hence it induces a splitting
$$ \wedge^2X \;=\; \wedge^+_gX \oplus \wedge^-_gX \;,$$
where $\wedge^\pm_gX := \left\{\varphi\in\wedge^2X \st *_g\varphi=\pm\varphi \right\}$, for $\pm\in\{+,-\}$. Setting $\Prim^\bullet X:=\ker\Lambda=\ker L^{2-\bullet+1}\lfloor_{\wedge^\bullet X}$ the space of primitive forms, where $\Lambda$ is the adjoint operator of the Lefschetz operator $L:=\omega\wedge\sspace \colon \wedge^\bullet X \to \wedge^{\bullet+2}X$ with respect to the pairing induced by $\omega$ (see \S\ref{sec:symplectic}), one has
$$ \wedge^+_gX \;=\; L\left(\mathcal{C}^\infty\left(X;\R\right)\right) \oplus \left(\left(\wedge^{2,0}X\oplus\wedge^{0,2}X\right)\cap\wedge^2X\right) \qquad \text{ and } \qquad \wedge^-_gX \;=\; \Prim^2 X \cap \wedge^{1,1}X \;; $$
indeed, recall that, on a compact $2n$-dimensional manifold $X$ endowed with an almost-complex structure $J$ and a Hermitian metric $g$ with associated $(1,1)$-form $\omega$, one has, for every $j\in\N$, for every $k\in\N$, the Weil identity, \cite[Théorème 2]{weil},
$$ *_g\, L^j\lfloor_{\Prim^kX} \;=\; \left(-1\right)^{\frac{k(k+1)}{2}} \, \frac{j!}{\left(n-k-j\right)!} \, L^{n-k-j}\, J \;, $$
see, e.g., \cite[Proposition 1.2.31]{huybrechts}. Since the Laplacian operator $\Delta$ and the Hodge-$*$-operator $*_g$ commute, the splitting $\wedge^2X=\wedge^+_gX\oplus\wedge^-_gX$ induces a decomposition in cohomology,
$$ H^2_{dR}(X;\R) \;=\; H^+_g(X) \oplus H^-_g(X) \;,$$
where $H^{\pm}_g(X):=\left\{\left[\varphi\right]\in H^2_{dR}(X;\R) \st \varphi\in\wedge^\pm_gX\right\}$ for $\pm\in\{+,-\}$.
Consider the non-degenerate pairing
$$ \left\langle \sspace, \, \ssspace \right\rangle \colon H^2_{dR}(X;\R) \times H^2_{dR}(X;\R) \to \R\;, \qquad \left\langle \varphi,\, \psi\right\rangle \;:=\; \int_X \varphi\wedge \psi \;,$$
and take $\mathfrak{a} \in \left(H^+_J(X)+H^-_J(X)\right)^\perp\subseteq H^2_{dR}(X;\R)$.
Since $\wedge^-_gX\subseteq \wedge^{1,1}X$, one can reduce to consider $\mathfrak{a}\in H^+_g(X)$; let $\alpha\in\wedge^+_gX$ be such that $\mathfrak{a}=\left[\alpha\right]$. According to the decomposition $\wedge^+_gX = L\left(\mathcal{C}^\infty\left(X;\R\right)\right) \oplus \left(\left(\wedge^{2,0}X\oplus\wedge^{0,2}X\right)\cap\wedge^2X\right)$, let $f\, \omega$ be the component of $\alpha$ in $L\left(\mathcal{C}^\infty\left(X;\R\right)\right)$. Consider the Hodge decomposition
$$ f\, \omega \;=\; h_{f\,\omega} + \de\vartheta + \de^*\eta $$
of $f\,\omega\in\wedge^2X$, where $h_{f \,\omega}\in\ker\Delta\cap\wedge^2X$, $\vartheta\in\wedge^1X$, and $\eta\in\wedge^3X$. Since $f\,\omega\in\wedge^+_gX$ and by the uniqueness of the Hodge decomposition, one has
$$ h_{f\, \omega} + 2\,\de\vartheta \;=\; f\, \omega + 2\, \pi_{\wedge^-_gX}\left(\de\vartheta\right) \;\in\; \wedge^{1,1}X\cap\wedge^2 X $$
(where $\pi_{\wedge^\pm_gX}\colon \wedge^2X\to \wedge^\pm_gX$ denotes the natural projection onto $\wedge^{\pm}_gX$, for $\pm\in\{+,-\}$). Therefore, noting also that $H^+_g(X)$ is orthogonal to $H^-_g(X)$ with respect to $\left\langle \sspace ,\, \ssspace\right\rangle$, one has
$$ 0 \;=\; \left\langle \mathfrak{a} ,\, \left[h_{f\,\omega} + 2\, \de\vartheta\right] \right\rangle \;=\; \left\langle \mathfrak{a} ,\, \left[f\, \omega + 2\pi_{\wedge^-_gX}\left(\de\vartheta\right)\right] \right\rangle \;=\; \int_X f^2\,\omega^2 \;, $$
from which it follows that $f=0$, and hence $\mathfrak{a}=0$.
\end{proof}

\begin{rem}
 The result in \cite[Theorem 2.3]{draghici-li-zhang} does not hold anymore true in dimension greater than or equal to $6$, or without the compactness assumption: the first example of a non-\Cp\ almost-complex structure has been provided by A. Fino and A. Tomassini in \cite[Example 3.3]{fino-tomassini} using a $6$-dimensional nilmanifold (for other examples, even in the integrable case, see Proposition \ref{prop:Cf-Cp-non-related}, Example \ref{ex:almost-kahler-non-Cf}, Theorem \ref{thm:instability-iwasawa}, Proposition \ref{prop:no-scs-h-}, Proposition \ref{prop:no-sci-h+}), while non-\Cpf\ almost-complex structures on non-compact $4$-dimensional manifolds arise from \cite[Theorem 3.24]{draghici-li-zhang-2} by T. Dr\v{a}ghici, T.-J. Li, and W. Zhang.
\end{rem}

\subsection{\Cpf\ solvmanifolds}
Let $X = \left. \Gamma \right\backslash G$ be a solvmanifold, and denote the Lie algebra naturally associated to $G$ by $\mathfrak{g}$, and its complexification by $\mathfrak{g}_\C:=\mathfrak{g}\otimes_\R\C$. (We refer to \S\ref{sec:solvmanifolds} for notations and results concerning solvmanifolds.)

We recall that if $X$ is a nilmanifold or, more in general, a completely-solvable solvmanifold, the inclusion of the sub-complex given by the $G$-left-invariant differential forms, which is isomorphic to the complex $\wedge^\bullet\duale{\mathfrak{g}}$ of linear forms on the dual of the Lie algebra $\mathfrak{g}$ associated to $G$, into the de Rham complex of $X$ turns out to be a quasi-isomorphism, in view of K. Nomizu's theorem \cite[Theorem 1]{nomizu}, respectively A. Hattori's theorem \cite[Corollary 4.2]{hattori}.

Let $J$ be a $G$-left-invariant almost-complex structure on $X$. In this case, one can study the problem of cohomological decomposition both on $X$ and on $\mathfrak{g}$: in this section, we investigate the relations between the cohomological decompositions at the level of the solvmanifold and at the level of the associated Lie algebra, Proposition \ref{prop:linear-cpf-invariant-cpf-J}, Corollary \ref{cor:linear-cpf-cpf}.

\medskip

Firstly, we set some notations. Consider $H_{dR}^\bullet\left(\mathfrak{g};\R\right):=H^\bullet\left(\wedge^\bullet\duale{\mathfrak{g}},\,\de\right)$. Being $J$ a $G$-left-invariant almost-complex structure, it induces a bi-graded splitting also on the vector space $\wedge^\bullet\duale{\mathfrak{g}}_\C$. For every $S\subset\N\times\N$, and for $\K\in\{\R,\, \C\}$, set
$$
H^{S}_J\left(\mathfrak{g};\K\right) \;:=\; \left\{ \left[\alpha\right]\in H_{dR}^{\bullet}\left(\mathfrak{g};\K\right) \st \alpha\in\bigoplus_{(p,q)\in S} \wedge^{p,q}\duale{\mathfrak{g}}_\C \cap \left(\wedge^\bullet \duale{\mathfrak{g}} \otimes_\R \K\right) \right\} \;,
$$
see \cite[Definition 0.3]{li-tomassini}.

The following are the natural linear counterparts of the corresponding definitions for manifolds.

\begin{defi}
 Let $X=\left.\Gamma\right\backslash G$ be a solvmanifold, and denote the Lie algebra naturally associated to $G$ by $\mathfrak{g}$. Fixed $k\in\N$, a $G$-left-invariant almost-complex structure $J$ on $X$ is called
\begin{itemize}
 \item \emph{linear-\Cp\ at the \kth{k} stage} if $$ \bigoplus_{\substack{p+q=k\\p\leq q}} H^{(p,q),(q,p)}_J\left(\mathfrak{g};\R\right) \;\subseteq\; H^k_{dR}\left(\mathfrak{g};\R\right) \;,$$
     namely, if the sum is direct;
 \item \emph{linear-\Cf\ at the \kth{k} stage} if $$ H^k_{dR}\left(\mathfrak{g};\R\right) \;=\; \sum_{\substack{p+q=k\\p\leq q}}H^{(p,q),(q,p)}_J\left(\mathfrak{g};\R\right) \;,$$
 \item \emph{linear-\Cpf\ at the \kth{k} stage} if $J$ is both linear-\Cp\ at the \kth{k} stage and linear-\Cf\ at the \kth{k} stage, that is, if the cohomological decomposition $$ H^k_{dR}\left(\mathfrak{g};\R\right) \;=\; \bigoplus_{\substack{p+q=k\\p\leq q}}H^{(p,q),(q,p)}_J\left(\mathfrak{g};\R\right) $$
 holds.
\end{itemize}

Furthermore, $J$ is called
\begin{itemize}
 \item \emph{linear-complex-\Cp\ at the \kth{k} stage} if $$ \bigoplus_{p+q=k} H^{(p,q)}_J\left(\mathfrak{g};\C\right) \;\subseteq\; H^k_{dR}\left(\mathfrak{g};\C\right) \;,$$
     namely, if the sum is direct;
 \item \emph{linear-complex-\Cf\ at the \kth{k} stage} if $$ H^k_{dR}\left(\mathfrak{g};\C\right) \;=\; \sum_{p+q=k}H^{(p,q)}_J\left(\mathfrak{g};\C\right) \;,$$
 \item \emph{linear-complex-\Cpf\ at the \kth{k} stage} if $J$ is both linear-complex-\Cp\ at the \kth{k} stage and linear-complex-\Cf\ at the \kth{k} stage, that is, if the cohomological decomposition $$ H^k_{dR}\left(\mathfrak{g};\C\right) \;=\; \bigoplus_{p+q=k}H^{(p,q)}_J\left(\mathfrak{g};\C\right) $$
 holds.
\end{itemize}

(In any case, when $k=2$, the specification ``at the \kth{2} stage'' will be understood.)
\end{defi}

\medskip

It is natural to ask what relations link the subgroups $H^{(\bullet,\bullet)}_J(X;\R)$ and the subgroups $H^{(\bullet,\bullet)}_J\left(\mathfrak{g};\R\right)$, and whether a $G$-left-invariant linear-\Cpf\ almost-complex structure on $X=\left.\Gamma\right\backslash G$ is also \Cpf.

The following lemma is the F.~A. Belgun symmetrization trick, \cite[Theorem 7]{belgun}, in the almost-complex setting.

\begin{lem}[{\cite[Theorem 7]{belgun}}]
\label{lemma:belgun-J}
Let $X=\left.\Gamma\right\backslash G$ be a solvmanifold, and denote the Lie algebra naturally associated to $G$ by $\mathfrak{g}$. Let $J$ be a $G$-left-invariant almost-complex structure on $X$.
Let $\eta$ be the $G$-bi-invariant volume form on $G$ given by J. Milnor's Lemma, \cite[Lemma 6.2]{milnor}, and such that $\int_X\eta=1$. Up to identifying $G$-left-invariant forms on $X$ and linear forms over $\duale{\mathfrak{g}}$ through left-translations, consider the Belgun symmetrization map
$$ \mu\colon \wedge^\bullet X \to \wedge^\bullet \duale{\mathfrak{g}}\;,\qquad \mu(\alpha)\;:=\;\int_X \alpha\lfloor_m \, \eta(m) \;.$$
Then one has that
$$ \mu\lfloor_{\wedge^\bullet \duale{\mathfrak{g}}}\;=\;\id\lfloor_{\wedge^\bullet \duale{\mathfrak{g}}} \;, $$
and that
$$ \de\left(\mu(\sspace)\right) \;=\; \mu\left(\de\sspace\right) \qquad \text{ and }\qquad  J\left(\mu(\sspace)\right) \;=\; \mu\left(J\sspace\right) \;.$$
\end{lem}

Using the previous lemma, we can prove the following Nomizu-type result, which relates the subgroups $H^{(r,s)}_J(X;\R)$ with their left-invariant part $H^{(r,s)}_J\left(\mathfrak{g};\R\right)$. (Analogous results will be proven in Proposition \ref{prop:linear-cpf-invariant-cpf-K} for almost-$\mathbf{D}$-complex structures in the sense of F.~R. Harvey and H.~B. Lawson, and in Proposition \ref{prop:linear-cpf-invariant-cpf-omega} for symplectic structures; compare also with \cite[Theorem 3.4]{fino-tomassini}, by A. Fino and A. Tomassini, for almost-complex structures.)

\begin{prop}[{\cite[Theorem 5.4]{angella-tomassini-zhang}}]\label{prop:linear-cpf-invariant-cpf-J}
Let $X=\left.\Gamma\right\backslash G$ be a solvmanifold endowed with a $G$-left-invariant almost-complex structure $J$, and denote the Lie algebra naturally associated to $G$ by $\mathfrak{g}$. For any $S\subset \N\times\N$, and for $\K\in\{\R,\,\C\}$, the map
$$ j\colon H^{S}_J(\mathfrak{g};\K) \to H^{S}_J(X;\K) $$
induced by left-translations is injective, and, if $H_{dR}^\bullet\left(\mathfrak{g};\K\right) \simeq H^\bullet_{dR}(X;\K)$ (for instance, if $X$ is a completely-solvable solvmanifold), then $j\colon H^{S}_J(\mathfrak{g};\K) \to H^{S}_J(X;\K)$ is in fact an isomorphism.
\end{prop}

\begin{proof}
Since $J$ is $G$-left-invariant, left-translations induce the map $j\colon H^{S}_J(\mathfrak{g};\K) \to H^{S}_J(X;\K)$.
Consider the Belgun symmetrization map $\mu\colon\wedge^\bullet X \otimes \K \to \wedge^\bullet \duale{\mathfrak{g}} \otimes_\R \K$, \cite[Theorem 7]{belgun}: since $\mu$ commutes with $\de$ by \cite[Theorem 7]{belgun}, it induces the map $\mu\colon H^\bullet_{dR}(X;\K) \to H_{dR}^\bullet\left(\mathfrak{g};\K\right)$, and, since $\mu$ commutes with $J$, it preserves the bi-graduation; therefore it induces the map $\mu\colon H^{S}_J(X;\K) \to H^{S}_J(\mathfrak{g};\K)$. Moreover, since $\mu$ is the identity on the space of $G$-left-invariant forms by \cite[Theorem 7]{belgun}, we get the commutative diagram
$$
\xymatrix{
H^{S}_{J}(\mathfrak{g};\K) \ar[r]^{j} \ar@/_1.5pc/[rr]_{\id} & H^{S}_{J}(X;\K) \ar[r]^{\mu} & H^{S}_{J}(\mathfrak{g};\K)
}
$$
hence $j\colon H^{S}_J(\mathfrak{g};\K) \to H^{S}_J(X;\K)$ is injective, and $\mu\colon H^{S}_J(X;\K) \to H^{S}_J(\mathfrak{g};\K)$ is surjective.

Furthermore, when $H_{dR}^\bullet\left(\mathfrak{g};\K\right) \simeq H^\bullet_{dR}(X;\K)$ (for instance, when $X$ is a completely-solvable solvmanifold, by A. Hattori's theorem \cite[Theorem 4.2]{hattori}), since
$$ \mu\lfloor_{\wedge^\bullet\duale{\mathfrak{g}}\otimes_\R\K} \;=\; \id\lfloor_{\wedge^\bullet\duale{\mathfrak{g}}\otimes_\R\K} $$
by \cite[Theorem 7]{belgun}, we get that $\mu\colon H^\bullet_{dR}(X;\K) \to H_{dR}^\bullet\left(\mathfrak{g};\K\right)$ is the identity map, and hence $\mu\colon H^{S}_J(X;\K) \to H^{S}_J(\mathfrak{g};\K)$ is also injective, and hence an isomorphism.
\end{proof}

As a straightforward consequence, we get the following result.

\begin{cor}\label{cor:linear-cpf-cpf}
Let $X=\left.\Gamma\right\backslash G$ be a solvmanifold endowed with a $G$-left-invariant almost-complex structure $J$, and denote the Lie algebra naturally associated to $G$ by $\mathfrak{g}$.
Suppose that $H_{dR}^\bullet\left(\mathfrak{g};\R\right) \simeq H^\bullet_{dR}(X;\R)$ (for instance, suppose that $X$ is a completely-solvable solvmanifold). For every $k\in\N$, the almost-complex structure $J$ is linear-\Cp\ (respectively, linear-\Cf, linear-\Cpf, linear-complex-\Cp, linear-complex-\Cf, linear-complex-\Cpf) at the \kth{k} stage if and only if it is \Cp\ (respectively, \Cf, \Cpf, complex-\Cp, complex-\Cf, complex-\Cpf) at the \kth{k} stage.
\end{cor}

\medskip

As an example, we provide here an explicit \Cpf\ almost-complex structure on a $6$-dimensional solvmanifold, \cite[Example 2.1]{angella-tomassini-1}.

\begin{ex}{\itshape A \Cpf\ and \pf\ almost-complex structure on a compact $6$-dimensional completely-solvable solvmanifold.}\\
 Let $G$ be the $6$-dimensional simply-connected completely-solvable Lie group defined by
$$
G \;:=\; \left\{\left(
\begin{array}{cccccc}
\esp^{x^1}  &  0               &  x^2\,\esp^{x^1}        &  0                 &  0    &  x^3\\
0           &  \esp^{-x^1}     &  0                      &  x^2\,\esp^{-x^1}  &  0    &  x^4\\
0           &  0               &  \esp^{x^1}             &  0                 &  0    &  x^5\\
0           &  0               &  0                      &  \esp^{-x^1}       &  0    &  x^6\\
0           &  0               &  0                      &  0                 &  1    &  x^1\\
0           &  0               &  0                      &  0                 &  0    &  1
\end{array}
\right)\,\in\,\GL(6;\R) \st x^1,\,\ldots,\,x^6\in\R \right\} \;.
$$
According to \cite[\S3]{fernandez-deleon-saralegui}, there exists a discrete co-compact subgroup $\Gamma\subset G$: therefore
$X:=\Gamma\backslash G$ is a $6$-dimensional completely-solvable solvmanifold.

The $G$-left-invariant $1$-forms on $G$ defined as
$$
\begin{array}{lcl}
e^1 \;:=\; \de x^1\;,& \qquad\qquad & e^2 \;:=\; \de x^2\;,\\[5pt]
e^3 \;:=\; \exp\left(-x^1\right) \cdot \left(\de x^3-x^2\,\de x^5\right) \;, & \qquad\qquad & e^4 \;:=\; \exp\left(x^1\right) \cdot \left(\de x^4-x^2\,\de x^6\right)\;, \\[5pt]
e^5 \;:=\; \exp\left(-x^1\right) \cdot \de x^5\,; & \qquad\qquad & e^6 \;:=\; \exp\left(x^1\right) \cdot \de x^6
\end{array}
$$
give rise to $G$-left-invariant $1$-forms on $X$. With respect to the co-frame $\left\{e^1,\ldots,e^6\right\}$, the structure equations are given by
$$
\left\{
\begin{array}{rcl}
 \de e^1 &=& 0 \\[5pt]
 \de e^2 &=& 0 \\[5pt]
 \de e^3 &=& -e^{1}\wedge e^3-e^{2}\wedge e^5 \\[5pt]
 \de e^4 &=&  e^{1}\wedge e^4-e^{2}\wedge e^6 \\[5pt]
 \de e^5 &=& -e^{1}\wedge e^5 \\[5pt]
 \de e^6 &=&  e^{1}\wedge e^6
\end{array}
\right. \;.
$$
Since $G$ is completely-solvable, by A. Hattori's theorem \cite[Corollary 4.2]{hattori}, it is straightforward to compute	
$$ H^2(X;\R)\;=\; \R\left\langle e^{1}\wedge e^{2},\; e^{5}\wedge e^{6},\; e^{3}\wedge e^{6}+e^{4}\wedge e^{5}\right\rangle \;. $$
Therefore, setting
$$
\left\{
\begin{array}{l}
 \varphi^1 \;:=\; e^1+\im e^2 \\[5pt]
 \varphi^2 \;:=\; e^3+\im e^4 \\[5pt]
 \varphi^3 \;:=\; e^5+\im e^6
\end{array}
\right. \;,
$$
we have that the almost-complex structure $J$ whose $\mathcal{C}^\infty\left(X;\C\right)$-module of complex $(1,0)$-forms is generated by $\left\{\varphi^1,\,\varphi^2,\,\varphi^3\right\}$ is \Cf: indeed,
$$
\begin{array}{rcl}
H^{(1,1)}_J\left(X;\R\right) &=& \R\left\langle-\frac{1}{2\im}\,\varphi^1\wedge\bar{\varphi}^1,\;-\frac{1}{2\im}\,\varphi^3\wedge\bar{\varphi}^3\right\rangle\;,\\[5pt]
H^{(2,0),(0,2)}_J\left(X;\R\right) &=& \R\left\langle \frac{1}{2\im}\, \left(\varphi^2\wedge\varphi^3-\bar{\varphi}^2\wedge\bar{\varphi}^3\right)\right\rangle\;.
\end{array}
$$
Since
$$ \de\wedge^1\duale{\mathfrak{g}}_\C \;=\; \C\left\langle \varphi^{13}-\varphi^{1\bar3},\, \varphi^{3\bar1}+\varphi^{\bar1\bar3},\, \varphi^{13}+\varphi^{1\bar3},\, \varphi^{3\bar1}-\varphi^{\bar1\bar3},\, \varphi^{12}-\varphi^{2\bar1},\, \varphi^{1\bar2}+\varphi^{\bar1\bar2} \right\rangle \;,$$
then $J$ is linear-\Cpf. Since $X$ is a completely-solvable solvmanifold, one gets that $J$ is also \Cp\ by Corollary \ref{cor:linear-cpf-cpf}. (Note that the \Cp ness of $J$ can be proven also by using a different argument: according to \cite[Theorem 3.7]{fino-tomassini}, since the above basis of harmonic representatives with respect to the $G$-left-invariant Hermitian metric $\sum_{j=1}^{3} \varphi^j\odot\bar\varphi^j$ consists of pure type forms with respect to the almost-complex structure, $J$ is both \Cpf\ and \pf.)
\end{ex}

\medskip

Further results concerning linear (possibly non-integrable) complex structures on $4$-dimensional unimodular Lie algebra and their cohomological properties have been obtained by T.-J. Li and A. Tomassini in \cite{li-tomassini}. In particular, they proved an analogous of \cite[Theorem 2.3]{draghici-li-zhang}, namely, that for every $4$-dimensional unimodular Lie algebra $\mathfrak{g}$ endowed with a linear (possibly non-integrable) complex structure $J$, one has the cohomological decomposition $H^2_{dR}\left(\mathfrak{g};\R\right)=H^{(2,0),(0,2)}_J\left(\mathfrak{g};\R\right)\oplus H^{(1,1)}_J\left(\mathfrak{g};\R\right)$, \cite[Theorem 3.3]{li-tomassini}. Furthermore, they studied the linear counterpart of S.~K. Donaldson's question \cite[Question 2]{donaldson} (see \S\ref{subsubsec:donaldson-question}), proving that, on a $4$-dimensional Lie algebra $\mathfrak{g}$ satisfying the condition $B\wedge B=0$, where $B\subseteq \wedge^2\mathfrak{g}$ denotes the space of boundary $2$-vectors, a linear
(possibly non-integrable) complex structure admits a taming linear symplectic form if and only if it admits a compatible linear symplectic form, \cite[Theorem 2.5]{li-tomassini}; note that $4$-dimensional unimodular Lie algebras satisfy the assumption $B\wedge B=0$. Finally, given a linear  (possibly non-integrable) complex structure on a $4$-dimensional Lie algebra, they studied the convex cones composed of the classes of $J$-taming, respectively $J$-compatible, linear symplectic forms, comparing them by means of $H^{(2,0),(0,2)}_J\left(\mathfrak{g};\R\right)$, \cite[Theorem 3.10]{li-tomassini}: this result is the linear counterpart of \cite[Theorem 1.1]{li-zhang}.

\subsection{Complex-\Cpf ness for $4$-dimensional manifolds}\label{subsec:cpf-solvmfds}

By \cite[Lemma 2.15, Theorem 2.16]{draghici-li-zhang}, or \cite[Proposition 2.1]{li-zhang}, every compact complex surface is complex-\Cpf\ at the \kth{2} stage; on the other hand, a compact complex surface is complex-\Cpf\ at the \kth{1} stage if and only if its first Betti number $b_1$ is even, that is, if and only if it admits a K\"ahler structure, see \cite{kodaira-structure-I, miyaoka, siu}, or \cite[Corollaire 5.7]{lamari}, or \cite[Theorem 11]{buchdahl}.

One may wonder about the relations between being complex-\Cpf\ and being integrable for an almost-complex structure on a compact $4$-dimensional manifold; this is the matter of the following result, \cite[Proposition 1.7]{angella-tomassini-2}.

\begin{prop}
\label{prop:complex-Cpf-4}
 There exist
\begin{itemize}
 \item non-complex-\Cpf\ at the \kth{1} stage non-integrable almost-complex structures, and
 \item complex-\Cpf\ at the \kth{1} stage non-integrable almost-complex structures
\end{itemize}
on compact $4$-dimensional manifolds with $b_1$ even.
\end{prop}

\begin{proof}
The proof follows from the following examples, \cite[Example 1.5, Example 1.6]{angella-tomassini-2}.

\paragrafod{1}{There exists a non-complex-\Cpf\ at the {$1^\text{st}$} stage non-integrable almost-complex structure on a $4$-dimensional manifold}
 Consider the standard K\"ahler structure $\left(J_0,\,\omega_0\right)$ on the $4$-dimensional torus $\T^4$ with coordinates $\left\{x^j\right\}_{j\in\{1,\ldots,4\}}$, that is,
$$ J_0 \;:=\;
\left(
\begin{array}{cc|cc}
 && -1 &\\
&&& -1 \\
\hline
1 &&&\\
&1&&
\end{array}
\right)
\;\in\; \End\left(\T^4\right)
\qquad \text{ and } \qquad
\omega_0 \;:=\; \de x^1 \wedge \de x^3 + \de x^2 \wedge \de x^4 \;\in\;\wedge^2\T^4 \;,
$$
and, for $\varepsilon>0$ small enough, let $\left\{J_t\right\}_{t\in\left(-\varepsilon,\, \varepsilon\right)}$ be the curve of almost-complex
 structures defined by
$$ J_t \;:=:\; J_{t,\,\ell} \;:=\; \left(\id - t\, L\right) \, J_0\, \left(\id-t\,L\right)^{-1} \;=\;
\left(
\begin{array}{cc|cc}
 && -\frac{1-t\,\ell}{1+t\,\ell} &\\
&&& -1 \\
\hline
\frac{1+t\,\ell}{1-t\,\ell} &&&\\
&1&&
\end{array}
\right)
\;\in\; \End\left(\T^4\right) \;,
$$
where
$$
L\;=\;
\left(
\begin{array}{cc|cc}
 \ell &&& \\
 & 0 &&\\
\hline
&&-\ell &\\
&&&0
\end{array}
\right)
\;\in\; \End\left(\T^4\right)
$$
and $\ell=\ell(x_2)\in\mathcal{C}^\infty(\R^4;\,\R)$ is a $\Z^4$-periodic non-constant function.

For $t\in\left(-\varepsilon,\, \varepsilon\right)\setminus\{0\}$, a straightforward computation yields
$$ H^{(1,0)}_{J_t} \left(\T^2_\C;\C\right) \;=\; \C\left\langle \de x^2+\im \de x^4 \right\rangle\;,\qquad H^{(0,1)}_{J_t}
\left(\T^2_\C;\C\right) \;=\; \C\left\langle \de x^2-\im \de x^4 \right\rangle $$
therefore
$$ \dim_\C H^{(1,0)}_{J_t}\left(\T^2_\C;\C\right)+\dim_\C H^{(0,1)}_{J_t}\left(\T^2_\C;\C\right) \;=\; 2 \;<\; 4\;=\; b_1\left(\T^2_\C\right) \;,
$$
that is, $J_t$ is not complex-\Cpf\ at the \kth{1} stage.

\paragrafod{2}{There exists a complex-\Cpf\ at the $1^\text{st}$ stage non-integrable almost-complex structure on a $4$-dimensional manifold}
 Consider a compact $4$-dimensional nilmanifold $X=\Gamma\backslash G$, quotient of the simply-connected nilpotent Lie group $G$ whose associated Lie algebra is
$$ \mathfrak{g} \;:=\; \left(0^2,\;14,\;12\right) \;; $$
let $J$ be the $G$-left-invariant almost-complex structure defined by
$$ J e^1 \;:=\; -e^2\; ,\qquad Je^3 \;:=\; -e^4 \;;$$
note that $J$ is not integrable, since $\textrm{Nij}(e_1,e_3)\neq 0$,
where $\left\{e_i\right\}_{i\in\{1,2,3,4\}}$ is the dual basis of $\left\{e^i\right\}_{i\in\{1,2,3,4\}}$. In fact, $X$ has no integrable almost-complex structure: indeed, since $b_1(X)=2$ is even, if there were a complex structure on $X$, then $X$ should carry a K\"ahler metric; this is not possible for compact non-tori nilmanifolds, by \cite[Theorem 1, Corollary]{hasegawa_pams}, or \cite[Theorem A]{benson-gordon-nilmanifolds}.

By K. Nomizu's theorem \cite[Theorem 1]{nomizu}, one computes
$$ H^1_{dR}(X;\C) \;=\; \C\left\langle \varphi^1,\,\bar\varphi^1 \right\rangle \qquad \text{ and } \qquad H^2_{dR}(X;\C) \;=\; \C\left\langle \varphi^{12} + \varphi^{\bar1\bar 2},\,\varphi^{1\bar2}-\varphi^{2\bar1} \right\rangle \;;$$
in particular, it follows that $J$ is complex-\Cpf\ at the $1^\text{st}$ stage.
Note that $J$ is not complex-\Cpf\ at the \kth{2} stage but just \Cpf: indeed, using Proposition \ref{prop:linear-cpf-invariant-cpf-J}, one can prove that the class $\left[\varphi^{12}+\varphi^{\bar1\bar2}\right]$ admits no pure type representative with respect to $J$. Moreover, observe that the $G$-left-invariant almost-complex structure
$$ J' e^1 \;:=\; -e^3\; ,\qquad J'e^2 \;:=\; -e^4 \;,$$
is complex-\Cpf\ at the \kth{2} stage and non-complex-\Cpf\ at the \kth{1} stage (obviously, in this case, $h^-_{J'}=0$, according to \cite[Corollary 2.14]{draghici-li-zhang}).
\end{proof}

\begin{rem}
 T. Dr\v{a}ghici, T.-J. Li, and W. Zhang proved in \cite[Corollary 2.14]{draghici-li-zhang} that an almost-complex structure on a compact $4$-dimensional manifold $X$ is complex-\Cpf\ at the \kth{2} stage if and only if $J$ is integrable or $\dim_\R H^-_J(X)=0$.
\end{rem}

\subsection{Almost-complex manifolds with large anti-invariant cohomology}
Given an almost-complex structure $J$ on a compact manifold $X$, it is natural to ask how large the cohomology subgroup $H^{-}_J(X)$ can be.

In \cite[Theorem 1.1]{draghici-li-zhang-2}, T. Dr\v{a}ghici, T.-J. Li, and W. Zhang, starting with a compact complex surface $X$ endowed with the complex structure $J$, proved that the dimension $h^-_{\tilde J}:=\dim_\R H^-_{\tilde{J}}(X)$ of the ${\tilde J}$-anti-invariant subgroup $H^-_{\tilde{J}}(X)$ of $H^2_{dR}(X;\R)$ associated to any {\em metric related} almost-complex structures $\tilde{J}$ on $X$ (that is, the almost-complex structures ${\tilde J}$ on $X$ inducing the same orientation as $J$ and with a common compatible metric with $J$), such that $\tilde{J}\neq\pm J$, satisfies $h^-_{\tilde J}\in\{0,\, 1,\, 2\}$, and they provided a description of such almost-complex structures $\tilde{J}$ having $h^-_{\tilde J}\in\{1,\,2\}$.

In this direction, T. Dr\v{a}ghici, T.-J. Li, and W. Zhang proposed the following conjecture.

\begin{conj}[{\cite[Conjecture 2.5]{draghici-li-zhang-2}}]
On a compact $4$-dimensional manifold endowed with an almost-complex structure $J$, if $\dim_\R H^-_J(X)\geq 3$, then $J$ is integrable.
\end{conj}

In \cite{tan-wang-zhang-zhu}, Q. Tan, H. Wang, Y. Zhang, and P. Zhu proved that, on a compact $4$-dimensional manifold endowed with an almost-complex structure $J$ and a $J$-Hermitian metric $g$, the dimension $\dim_\R H^-_{\tilde J}(X)$ is constant for all almost-complex structures $\tilde J$ being \emph{fundamental form related} to $J$, namely, such that $\omega\in\wedge^{1,1}_{\tilde J}X \cap \wedge^2X$, where $\omega:=g\left(J\,\sspace, \, \ssspace\right)\in\wedge^{1,1}_JX\cap\wedge^2X$ is the $(1,1)$-form with respect to $J$ associated to the $J$-Hermitian metric $g$, \cite[Theorem 1.2]{tan-wang-zhang-zhu}. Then, they proposed to modify \cite[Conjecture 2.5]{draghici-li-zhang-2} as follows.

\begin{conj}[{\cite[Question 1.5]{tan-wang-zhang-zhu}}]
Let $X$ be a compact $4$-dimensional manifold endowed with an almost-complex structure $J$ and a $J$-Hermitian metric $g$, and denote by $\omega:=g\left(J\,\sspace, \, \ssspace\right)$ the $(1,1)$-form associated to $g$. Suppose that $\dim_\R H^-_J(X)\geq 3$. Does there exist an integrable almost-complex structure $\tilde J$ such that $\omega\in\wedge^{1,1}_{\tilde J}X$?
\end{conj}

Furthermore, in \cite{draghici-li-zhang-2}, it was conjectured that $h^-_J=0$ for a generic almost-complex structure $J$ on a compact $4$-dimensional manifold, \cite[Conjecture 2.4]{draghici-li-zhang-2}. In \cite[Theorem 1.1]{tan-wang-zhang-zhu}, Q. Tan, H. Wang, Y. Zhang, and P. Zhu proved that this holds true, showing that, on a compact $4$-dimensional manifold $X$ admitting almost-complex structures, the set of almost-complex structures $J$ on $X$ with $\dim_\R H^-_J(X)=0$ is an open dense subset of the set of almost-complex structures on $X$.

\medskip

In \cite[\S5]{angella-tomassini-zhang}, a $1$-parameter family $\left\{J_t\right\}_{t\in\left(-\varepsilon,\, \varepsilon\right)}$ of (non-integrable) almost-complex structures on the $6$-dimensional torus $\T^6$, where $\varepsilon>0$ is small enough, having $\dim_\R H^{-}_{J_t}\left(\T^6\right)$ greater than $3$ has been provided. We recall here the construction, see also \cite[\S4]{angella-tomassini-1}.

\begin{ex}
\label{ex:toro-6-large-invariant}
\textit{A family of almost-complex structures on the $6$-dimensional torus with anti-invariant cohomology of dimension larger than $3$.}\\
 Consider the $6$-dimensional torus $\T^6$, with coordinates $\left\{x^j\right\}_{j\in\{1,\ldots,6\}}$. For $\varepsilon>0$ small enough, choose a function $\alpha\colon \left(-\varepsilon,\, \varepsilon\right)\times \T^6\to \R$ such that $\alpha_t:=:\alpha\left(t, \sspace\right)\in\mathcal{C}^\infty\left(\T^6\right)$ depends just on $x^3$ for any $t\in\left(-\varepsilon,\, \varepsilon\right)$, namely $\alpha_t=\alpha_t(x^3)$, and that $\alpha_0(x^3) = 1$. Define the almost-complex structure $J_t$ in such a way that
$$
\left\{
\begin{array}{l}
 \varphi^1_t \;:=\; \de x^1\,+\,\im\,\alpha_t\,\de x^4 \\[5pt]
 \varphi^2_t \;:=\; \de x^2\,+\,\im\,\de x^5 \\[5pt]
 \varphi^3_t \;:=\; \de x^3\,+\,\im\,\de x^6
\end{array}
\right.
$$
provides a co-frame for the $\mathcal{C}^\infty\left(\T^6;\C\right)$-module of $(1,0)$-forms on $\T^6$ with respect to $J_t$. In terms of this co-frame, the structure equations are
$$
\left\{
\begin{array}{l}
 \de\varphi^1_t \;=\; \im\,\de\alpha_t\,\wedge\,\de x^4 \\[5pt]
 \de\varphi^2_t \;=\; 0 \\[5pt]
 \de\varphi^3_t \;=\; 0
\end{array}
\right. \;.
$$
Straightforward computations give that the $J_t$-anti-invariant real closed $2$-forms are of the type
$$ 
\psi \;=\; \frac{C}{\alpha_t}\,\left(\de x^{13}-\alpha_t\,\de x^{46}\right) + D\,\left(\de x^{16} - \alpha_t\, \de x^{34}\right) + E\,\left(\de x^{23} - \de x^{56}\right) + F\,\left(\de x^{26} - \de x^{35}\right) \;,
$$
where $C,\;D,\;E,\;F\in\R$ (we shorten $\de x^{jk}:=\de x^j\wedge \de x^k$). Moreover, the forms $\de x^{23} - \de x^{56}$ and $\de x^{26} - \de x^{35}$ are clearly harmonic with respect to the standard Riemannian metric $\sum_{j=1}^{6} \de x^j\otimes \de x^j$, while the classes of $\de x^{16} - \alpha_t\, \de x^{34}$ and $\de x^{13}-\alpha_t\,\de x^{46}$ are non-zero, being their harmonic parts non-zero. Therefore, we get that
$$ h^-_{J_t} \;=\; 4 \qquad \text{ for small }t\neq0 \;,$$
while $h^-_{J_0}=6$.
\end{ex}

\medskip

The natural generalization of \cite[Conjecture 2.5]{draghici-li-zhang} to higher dimensional manifolds yields the following question, \cite[Question 5.2]{angella-tomassini-zhang}.
\begin{question}
 Are there compact $2n$-dimensional manifolds $X$ endowed with non-integrable almost-complex structures $J$ with $\dim_\R H^-_J(X) > n\,\left(n-1\right)$?
\end{question}

Note that, when $X = \left.\Gamma\right\backslash G$ is a $2n$-dimensional completely-solvable solvmanifold endowed with a $G$-left-invariant almost-complex structure $J$, then, by Proposition \ref{prop:linear-cpf-invariant-cpf-J}, it follows that
$$ \dim_\R H^-_J(X) \;\leq\; n\, (n-1) \qquad \text{ and } \qquad \dim_\R H^+_J(X) \;\leq\; n^2 \;. $$

\subsection{Semi-K\"ahler manifolds}
As already recalled, A. Fino and A. Tomassini's \cite[Theorem 4.1]{fino-tomassini} proves that, given an almost-K\"ahler structure on a compact manifold, if the almost-complex structure is \Cpf\ and the symplectic structure satisfies the Hard Lefschetz Condition, then the almost-complex structure is \pf\ too; moreover, by \cite[Proposition 3.2]{fino-tomassini}, see also \cite[Proposition 2.8]{draghici-li-zhang}, the almost-complex structure of every almost-K\"ahler structure on a compact manifold is \Cp.

To study the cohomology of balanced manifolds $X$ and the duality between $H^{(\bullet,\bullet)}_J(X;\C)$ and $H^J_{(\bullet,\bullet)}(X;\C)$, we get the following result, \cite[Proposition 3.1]{angella-tomassini-2}, which can be considered as the semi-K\"ahler counterpart of \cite[Theorem 4.1]{fino-tomassini}.

\begin{prop}
\label{prop:duality-semi-kahler}
 Let $X$ be a compact $2n$-dimensional manifold endowed with an almost-complex structure $J$ and a semi-K\"ahler form $\omega$. Suppose that $\left[\omega^{n-1}\right]\cp\sspace\colon H^1_{dR}(X;\R) \to H^{2n-1}_{dR}(X;\R)$ is an isomorphism. If $J$ is complex-\Cpf\ at the \kth{1} stage, then it is also complex-\pf\ at the \kth{1} stage, and
 $$ H^{(1,0)}_J(X;\C)\simeq H_{(0,1)}^J(X;\C) \;.$$
\end{prop}

\begin{proof}
Firstly, note that $J$ is complex-\p\ at the \kth{1} stage. Indeed, if
$$ \mathfrak{a} \;\in\; H^J_{(1,0)}(X;\C)\cap H^J_{(0,1)}(X;\C) \;, $$
then
$$ \mathfrak{a}\lfloor_{H^{(1,0)}_J(X;\C)} \;=\; 0 \;=\; \mathfrak{a}\lfloor_{H^{(0,1)}_J(X;\C)} \;. $$
Therefore, by the assumption
$$ H^{1}_{dR}(X;\C) \;=\; H^{(1,0)}_J(X;\C)\oplus H^{(0,1)}_J(X;\C) \;, $$
we get that
$$ \mathfrak{a} \;=\; 0 \;. $$

Now, note that, since
$$
   \left[\omega^{n-1}\right]\cp H^{(1,0)}_{J}(X;\C)\subseteq H^{(n,n-1)}_{J}(X;\C)
   \qquad \text{ and } \qquad
   \left[\omega^{n-1}\right]\cp H^{(0,1)}_{J}(X;\C)\subseteq H^{(n-1,n)}_{J}(X;\C) \;,
$$
the isomorphism
$$ H^{1}_{dR}(X;\C) \stackrel{\left[\omega^{n-1}\right]\cp\sspace}{\longrightarrow} H^{2n-1}_{dR}(X;\C) \stackrel{T_\sspace}{\longrightarrow} H^1_{dR}(X;\C) $$
yields the injective maps
$$ H^{(1,0)}_J(X;\C)\hookrightarrow H^J_{(0,1)}(X;\C) \qquad \text{ and } \qquad  H^{(0,1)}_J(X;\C)\hookrightarrow H^J_{(1,0)}(X;\C) \;.$$
Since, by hypothesis, $J$ is complex-\Cpf\ at the \kth{1} stage, namely, $H^1_{dR}(X;\C) = H^{(1,0)}_J(X;\C)\oplus H^{(0,1)}_J(X;\C)$, we get the proof.
\end{proof}

\medskip

We provide here some explicit examples, \cite[Example 3.2, Example 3.3]{angella-tomassini-2}, checking the validity of the hypothesis of $\left[\omega^{n-1}\right]\cp\sspace: H^1_{dR}(X;\R) \to H^{2n-1}_{dR}(X;\R)$ being an isomorphism in Proposition \ref{prop:duality-semi-kahler}.

\begin{ex}
 {\itshape A balanced structure on the Iwasawa manifold.}\\
On the Iwasawa manifold $\mathbb{I}_3$ (see \S\ref{subsec:iwasawa}), consider the balanced structure
$$ \omega \;:=\; \frac{\im}{2}\left(\varphi^1\wedge\bar\varphi^1+\varphi^2\wedge\bar\varphi^2+\varphi^3\wedge\bar\varphi^3\right) \;. $$
Since
$$ H^1_{dR}\left(\mathbb{I}_3;\C\right) \;=\; \C\left\langle \varphi^1,\;\varphi^2,\;\bar\varphi^1,\;\bar\varphi^2 \right\rangle \qquad \text{ and } \qquad  H^5_{dR}\left(\mathbb{I}_3;\C\right) \;=\; \C\left\langle \varphi^{123\bar1\bar3},\;\varphi^{123\bar2\bar3},\;\varphi^{13\bar1\bar2\bar3},
\;\varphi^{23\bar1\bar2\bar3} \right\rangle\;, $$
it is straightforward to check that
$$ \left[\omega^2\right]\cp\sspace \colon H^1_{dR}\left(\mathbb{I}_3;\C\right)\to H^5_{dR}\left(\mathbb{I}_3;\C\right) $$
is an isomorphism.
Therefore, by Proposition \ref{prop:duality-semi-kahler}, $\mathbb{I}_3$ is complex-\Cpf\ at the \kth{1} stage and complex-\pf\ at the \kth{1} stage (the same result follows also arguing as in \cite[Theorem 3.7]{fino-tomassini}, the above harmonic representatives of $H^1_{dR}\left(\mathbb{I}_3;\C\right)$, with respect to the Hermitian metric $\sum_{j=1}^3\varphi^j\odot\bar\varphi^j$, being of pure type with respect to the complex structure).
\end{ex}

\begin{ex}
 {\itshape A $6$-dimensional manifold endowed with a semi-K\"ahler structure not inducing an isomorphism in cohomology.}\\
Consider the $6$-dimensional nilmanifold
$$ X \;=\; \Gamma\backslash G \;:=\; \left(0^4,\; 12,\; 13\right)\;. $$ 
In \cite[Example 3.3]{fino-tomassini}, the almost-complex structure
$$ J'\,e^1 \;:=\;-e^2\;, \qquad J'\,e^3 \;:=\; -e^4\;, \qquad J'\,e^5 \;:=\; -e^6 $$
is provided as a first example of non-\Cp\ almost-complex structure. Note that $J'$ is not even \Cf: indeed, the cohomology class $\left[e^{15}+e^{16}\right]$ admits neither $J'$-invariant nor $J'$-anti-invariant $G$-left-invariant representatives, and hence, by Proposition \ref{prop:linear-cpf-invariant-cpf-J}, it admits neither $J'$-invariant nor $J'$-anti-invariant representatives.

Consider now the almost-complex structure
$$ J\,e^1 \;:=\; -e^5\;, \qquad J\,e^2 \;:=\; -e^3\;, \qquad J\,e^4 \;:=\; -e^6 $$
and the non-degenerate $J$-invariant $2$-form
$$ \omega \;:=\; e^{15}+e^{23}+e^{46} \;. $$
A straightforward computation shows that
$$ \de\omega \;=\; -e^{134} \;\neq\; 0 \quad \text{ and } \quad \de\omega^2 \;=\; \de\left(e^{1235}-e^{1456}+e^{2346}\right) \;=\; 0\;.$$
By K. Nomizu's theorem \cite[Theorem 1]{nomizu}, it is straightforward to compute
$$ H^{1}_{dR}(X;\R) \;=\; \R\left\langle e^1,\; e^2,\; e^3,\; e^4 \right\rangle \;.$$
Since
$$ \omega^2\,e^1 \;=\; e^{12346} \;=\; \de e^{3456} \;, $$
we get that $\left[\omega^2\right]\cp\sspace\colon H^{1}_{dR}(X;\R)\to H^{5}_{dR}(X;\R)$ is not injective.
\end{ex}

We give two explicit examples of $2n$-dimensional complex manifolds endowed with a balanced structure, with $2n=10$, respectively $2n=6$, such that the \kth{\left(n-1\right)} power of the associated $(1,1)$-form induces an isomorphism in cohomology, and admitting small balanced deformations, \cite[Example 3.4, Example 3.5]{angella-tomassini-2}.

\begin{ex}
\label{ex:etabeta5}
 {\itshape A curve of balanced structures on $\eta\beta_5$ inducing an isomorphism in cohomology.}\\
We recall the construction of the $10$-dimensional nilmanifold $\eta\beta_5$, introduced and studied in \cite{alessandrini-bassanelli} by L. Alessandrini and G. Bassanelli to prove that being $p$-K\"ahler is not a stable property under small deformations of the complex structure; more in general, in \cite{alessandrini-bassanelli-p-kahler}, the manifold $\eta\beta_{2n+1}$, for any $n\in\N\setminus\{0\}$, has been provided as a generalization of the Iwasawa manifold $\mathbb{I}_3$, and the existence of $p$-K\"ahler structures on $\eta\beta_{2n+1}$ has been investigated. (For definitions and results concerning $p$-K\"ahler structures, see \cite{alessandrini-bassanelli-p-kahler}, or, e.g., \cite{silva, alessandrini-cras}.)

For $n\in\N\setminus\{0\}$, consider the complex Lie group
$$
G_{2n+1} \;:=\; \left\{ A\in \GL(n+2;\C) \st 
A \;=\;
\left(
\begin{array}{c|ccc|c}
 1 & x^1 & \cdots & x^n & z \\
 \hline
 0      & 1 &        &   & y^1    \\
 \vdots &   & \ddots &   & \vdots \\
 0      &   &        & 1 & y^n    \\
 \hline
 0 & 0 & \cdots & 0 & 1
\end{array}
\right)
 \right\} \;;
$$
equivalently, one can identify $G_{2n+1}$ with $\left(\C^{2n+1}, \, *\right)$, where the group structure $*$ is defined as
\begin{eqnarray*}
\lefteqn{\left(x^1,\, \ldots,\, x^n,\, y^1,\, \ldots,\, y^n,\, z\right) \,*\, \left(u^1,\, \ldots,\, u^n,\, v^1,\, \ldots,\, v^n,\, w\right)} \\[5pt]
 && :=\; \left(x^1+u^1,\, \ldots,\, x^n+u^n,\, y^1+v^1,\, \ldots,\, y^n+v^n,\, z+w+x^1\cdot v^1+\cdots+x^n\cdot v^n \right) \;.
\end{eqnarray*}
Since the subgroup
$$ \Gamma_{2n+1} \;:=\; G_{2n+1} \cap \GL\left(n+2;\Z\left[\im\right]\right) \;\subset\; G_{2n+1} $$
is a discrete co-compact subgroup of the nilpotent Lie group $G_{2n+1}$, one gets a compact complex manifold, of complex dimension $2n+1$,
$$ \eta\beta_{2n+1} \;:=\; \left. \Gamma_{2n+1} \right\backslash G_{2n+1} \;,$$
which is a holomorphically parallelizable nilmanifold and admits no K\"ahler metric, \cite[Corollary 2]{wang}, or \cite[Theorem A]{benson-gordon-nilmanifolds}, or \cite[Theorem 1, Corollary]{hasegawa_pams}; note that $\eta\beta_3=\mathbb{I}_3$ is the Iwasawa manifold (see \S\ref{subsec:iwasawa}). In fact, one has that $\eta\beta_{2n+1}$ is not $p$-K\"ahler for $1<p\leq n$ and it is $p$-K\"ahler for $n+1\leq p\leq 2n+1$, \cite[Theorem 4.2]{alessandrini-bassanelli-p-kahler}; furthermore, $\eta\beta_{2n+1}$ has complex submanifolds of any complex dimension less than or equal to $2n+1$, and hence it follows that the $p$-K\"ahler forms on $\eta\beta_{2n+1}$ can never be exact, \cite[\S4.4]{alessandrini-bassanelli-p-kahler}.

Setting
$$
 \left\{
 \begin{array}{lclcl}
 \varphi^{2j-1} &:=& \de x^j \;, &\text{ for }& j\in\{1,\ldots,n\} \;, \\[5pt]
 \varphi^{2j}   &:=& \de y^j \;, &\text{ for }& j\in\{1,\ldots,n\} \;, \\[5pt]
 \varphi^{2n+1} &:=& \de z-\sum_{j=1}^{n}x^j\,\de y^j \;, &&
 \end{array}
 \right.
$$
one gets the global co-frame $\left\{\varphi^j\right\}_{j\in\{1,\ldots,2n+1\}}$ for the space of holomorphic $1$-forms, with respect to which the structure equations are
$$
\left\{
\begin{array}{l}
 \de\varphi^1 \;=\; \cdots \;=\; \de\varphi^{2n} \;=\; 0 \\[5pt]
 \de\varphi^{2n+1} \;=\; -\sum_{j=1}^{n}\varphi^{2j-1}\wedge\varphi^{2j}
\end{array}
\right. \;.
$$

Now, take $2n+1=5$. With respect to the co-frame $\left\{\varphi^j\right\}_{j\in\{1,\ldots,5\}}$ for the space of holomorphic $1$-forms on $\eta\beta_5$, the structure equations are
$$
\left\{
\begin{array}{l}
 \de\varphi^1 \;=\; \de\varphi^2 \;=\; \de\varphi^3 \;=\; \de\varphi^4 \;=\; 0 \\[5pt]
 \de\varphi^5 \;=\; -\varphi^{12}-\varphi^{34}
\end{array}
\right.
$$
(where, as usually, we shorten, e.g., $\varphi^{12}:=\varphi^1\wedge\varphi^2$).

Consider on $\eta\beta_5$ the balanced structure
$$ \omega \;:=\; \frac{\im}{2}\,\sum_{j=1}^5\varphi^j\wedge\bar\varphi^j \;.$$

By K. Nomizu's theorem \cite[Theorem 1]{nomizu}, it is straightforward to compute
$$ H^1_{dR}\left(\eta\beta_5;\C\right) \;=\; \C\left\langle \varphi^1,\; \varphi^2,\; \varphi^3,\; \varphi^4,\; \bar\varphi^1,\; \bar\varphi^2,\; \bar\varphi^3,\; \bar\varphi^4\right\rangle $$
and
\begin{eqnarray*}
H^9_{dR}\left(\eta\beta_5;\C\right) &=& \C\left\langle \varphi^{12345\bar2\bar3\bar4\bar5},\; \varphi^{12345\bar1\bar3\bar4\bar5},\; \varphi^{12345\bar1\bar2\bar4\bar5},\; \varphi^{12345\bar1\bar2\bar3\bar5},\right.\\[5pt]
&& \left.\varphi^{2345\bar1\bar2\bar3\bar4\bar5},\; \varphi^{1345\bar1\bar2\bar3\bar4\bar5},\; \varphi^{1245\bar1\bar2\bar3\bar4\bar5},\; \varphi^{1235\bar1\bar2\bar3\bar4\bar5}\right\rangle \;;
\end{eqnarray*}
therefore, $\eta\beta_5$ is complex-\Cpf\ at the \kth{1} stage and
$$ \left[\omega^4\right]\cp\sspace\colon H^1_{dR}\left(\eta\beta_5;\R\right)\to H^9_{dR}\left(\eta\beta_5;\R\right) $$
is an isomorphism, and so $\eta\beta_5$ is also complex-\pf\ at the \kth{1} stage by Proposition \ref{prop:duality-semi-kahler} (note that, the above pure type representatives being harmonic with respect to the metric $\sum_{j=1}^5\varphi^j\odot\bar\varphi^j$, the same result follows also arguing as in \cite[Theorem 3.7]{fino-tomassini}).

Now, let $\left\{J_t\right\}_{t\in\Delta\left(0,\varepsilon\right)\subset\C}$, where $\varepsilon>0$ is small enough, be a family of small deformations of the complex structure such that
$$
\left\{
\begin{array}{rcl}
 \varphi^1_t &:=& \varphi^1+t\,\bar\varphi^1 \\[5pt]
 \varphi^2_t &:=& \varphi^2 \\[5pt]
 \varphi^3_t &:=& \varphi^3 \\[5pt]
 \varphi^4_t &:=& \varphi^4 \\[5pt]
 \varphi^5_t &:=& \varphi^5
\end{array}
\right.
$$
is a co-frame for the $J_t$-holomorphic cotangent bundle.
With respect to $\left\{\varphi^j_t\right\}_{j\in\{1,\ldots,5\}}$, the structure equations are written as
$$
\left\{
\begin{array}{l}
 \de\varphi^1_t \;=\; \de\varphi^2_t \;=\; \de\varphi^3_t \;=\; \de\varphi^4_t \;=\; 0 \\[5pt]
 \de\varphi^5_t \;=\; -\frac{1}{1-\left|t\right|^2}\,\varphi^{12}_t-\varphi^{34}_t-\frac{t}{1-\left|t\right|^2}\,
\varphi^{2\bar1}_t
\end{array}
\right. \;.
$$
Setting, for $t\in\Delta\left(0,\varepsilon\right)\subset \C$,
$$ \omega_t \;:=\; \frac{\im}{2} \,\sum_{j=1}^5\varphi^j_t\wedge\overline{\varphi}^j_t \;, $$
one gets a curve of balanced structures $\left\{\left(J_t,\,\omega_t\right)\right\}_{t\in\Delta\left(0,\varepsilon\right)}$ on the smooth manifold underlying $\eta\beta_5$. Furthermore, for every $t\in\Delta\left(0,\varepsilon\right)$, the complex structure $J_t$ is complex-\Cpf\ at the \kth{1} stage and
$$ \left[\omega_t^{4}\right]\cp\sspace \colon H^1_{dR}\left(\eta\beta_5;\R\right) \to H^{9}_{dR}\left(\eta\beta_5;\R\right) $$
is an isomorphism. Therefore, according to Proposition \ref{prop:duality-semi-kahler}, it follows that, for every $t\in\Delta\left(0,\varepsilon\right)$, the complex structure $J_t$ is complex-\pf\ at the \kth{1} stage, and that $H^{(1,0)}_{J_t}\left(\eta\beta_5;\C\right)\simeq H_{(0,1)}^{J_t}\left(\eta\beta_5;\C\right)$.
\end{ex}

\begin{ex}
 {\itshape A curve of semi-K\"ahler structures on a $6$-dimensional completely-solvable solvmanifold inducing an isomorphism in cohomology.}\\
Consider a completely-solvable solvmanifold
$$ X \;=\; \Gamma\backslash G \;:=\; \left(0,\; -12,\; 34,\; 0,\; 15,\; 46\right) $$
endowed with the almost-complex structure $J_0$ whose holomorphic cotangent bundle has co-frame generated by
$$
\left\{
\begin{array}{rcl}
 \varphi^1 &:=& e^1+\im e^4 \\[5pt]
 \varphi^2 &:=& e^2+\im e^5 \\[5pt]
 \varphi^3 &:=& e^3+\im e^6
\end{array}
\right.
$$
and with the $J_0$-compatible symplectic form
$$ \omega_0 \;:=\; e^{14}+e^{25}+e^{36} $$
(see also \cite[\S6.3]{fino-tomassini}).
The structure equations with respect to $\left\{\varphi^1,\, \varphi^2,\, \varphi^3\right\}$ are
$$
\left\{
\begin{array}{rcl}
 \de\varphi^1 &=& 0 \\[5pt]
 2\,\de\varphi^2 &=& -\varphi^{1\bar2}-\varphi^{\bar1\bar2} \\[5pt]
 2\im \,\de\varphi^3 &=& -\varphi^{1\bar3}+\varphi^{\bar1\bar3}
\end{array}
\right. \;;
$$
using A. Hattori's theorem \cite[Corollary 4.2]{hattori}, one computes
\begin{eqnarray*}
H^1_{dR}(X;\R) &=& \R\left\langle e^1,\; e^4\right\rangle\;, \\[5pt]
H^5_{dR}(X;\R) &=& \R\left\langle *_{g_0}\,e^1,\; *_{g_0}\,e^4 \right\rangle \;=\; \R\left\langle e^{23456},\; e^{12356} \right\rangle\;,
\end{eqnarray*}
where $g_0$ is the $J_0$-Hermitian metric induced by $\left(J_0,\,\omega_0\right)$.

Now, consider the curve $\left\{J_t\right\}_{t\in\left(-\varepsilon,\,\varepsilon\right)\subset\R}$ of almost-complex structures on $X$, where $\varepsilon>0$ is small enough and $J_t$ is defined requiring that the $J_t$-holomorphic cotangent bundle is generated by
$$
\left\{
\begin{array}{rcl}
 \varphi^1_t &:=& \varphi^1 \\[5pt]
 \varphi^2_t &:=& \varphi^2+\im\,t\,e^6 \\[5pt]
 \varphi^3_t &:=& \varphi^3
\end{array}
\right. \;;
$$
for every $t\in\left(-\varepsilon,\,\varepsilon\right)$, consider also the non-degenerate $J_t$-compatible $2$-form
$$ \omega_t \;:=\; e^{14}+e^{25}+e^{36}+t\,e^{26} \;;$$
for $t\neq0$, one has that $\de\omega\neq 0$, but
$$ \de \omega_t^2 \;=\; \de\left(\omega_0^2-t\,e^{1246}\right) \;=\; 0 \;,$$
hence $\left\{\left(J_t,\,\omega_t\right)\right\}_{t\in\left(-\varepsilon,\,\varepsilon\right)}$ gives rise to a curve of semi-K\"ahler structures on $X$. Moreover, note that
$$ \omega_t^2 \wedge e^1\;=\;e^{12356}\;,\qquad \omega_t^2 \wedge e^4\;=\;e^{23456}\;,$$
therefore $\left[\omega_t^2\right]\cp\sspace\colon H^1_{dR}(X;\R)\to H^5_{dR}(X;\R)$ is an isomorphism, for every $t\in\left(-\varepsilon,\,\varepsilon\right)$.
\end{ex}

\subsection{Almost-K\"ahler manifolds and Lefschetz-type property}\label{subsec:alm-kahler-lefschetz-type}
Recall that every compact manifold $X$ endowed with a K\"ahler structure $\left(J,\,\omega\right)$ is \Cpf, in fact, complex-\Cpf\ at every stage, \cite[Lemma 2.15, Theorem 2.16]{draghici-li-zhang}, or \cite[Proposition 2.1]{li-zhang}. A natural question is whether or not the same holds true even for almost-K\"ahler structures, namely, without the integrability assumption on $J$.

In this section, we study cohomological properties for almost-K\"ahler structures, in connection with a Lefschetz-type property, Theorem \ref{thm:almost-kahler-cpf}, and we describe some explicit examples.

The results in this section have been obtained in a joint work with A. Tomassini and W. Zhang, \cite{angella-tomassini-zhang}.

\medskip

Let $X$ be a compact $2n$-dimensional manifold endowed with an \emph{almost-K\"ahler structure} $\left(J,\,\omega,\,g\right)$, that is, $J$ is an almost-complex structure on $X$ and $g$ is a $J$-Hermitian metric whose associated $(1,1)$-form $\omega:=g\left(J\,\sspace, \, \ssspace\right)\in\wedge^{1,1}X\cap\wedge^2X$ is $\de$-closed.

\medskip

Firstly, we recall the following result on decomposition in cohomology for almost-K\"ahler manifolds, proven by T. Dr\v{a}ghici, T.-J. Li, and W. Zhang in \cite{draghici-li-zhang} and, in a different way, by A. Fino and A. Tomassini in \cite{fino-tomassini}.

\begin{prop}[{\cite[Proposition 2.8]{draghici-li-zhang}, \cite[Proposition 3.2]{fino-tomassini}}] 
 Let $X$ be a compact manifold and let $\left(J,\,\omega,\,g\right)$ be an almost-K\"ahler structure on $X$. Then $J$ is \Cp.
\end{prop}

Hence, one is brought to study the \Cf ness of almost-K\"ahler structures.

\medskip

Note that $\omega$ is in particular a symplectic form on $X$. We recall that, given a compact $2n$-dimensional manifold $X$ endowed with a symplectic form $\omega$, and fixed $k\in\N$, the Lefschetz-type operator on $(n-k)$-forms associated with $\omega$ is the operator
$$
L^k \;:=:\; L^k_\omega \colon\wedge^{n-k} X \to\wedge^{n+k} X\,,\qquad L^k(\alpha)\;:=\;\omega^{k}\wedge \alpha
$$
(see \S\ref{sec:symplectic} for notations concerning symplectic structures); since $\de\omega=0$, the map $L^k \colon\wedge^{n-k} X \to\wedge^{n+k} X$ induces a map in cohomology, namely,
$$
L^k \colon H^{n-k}_{dR}(X;\R) \to H^{n+k}_{dR}(X;\R)\,,\qquad L^k \left(\mathfrak{a}\right)\;:=\;\left[\omega^{k}\right]\cp\mathfrak{a}\,,
$$

Initially motivated by studying, in \cite{zhang}, Taubes currents, which have been introduced by C.~H. Taubes in \cite{taubes} in order to study S.~K. Donaldson's ``tamed to compatible'' question, \cite[Question 2]{donaldson}, W. Zhang considered the following Lefschetz-type property, see also \cite[\S2.2]{draghici-li-zhang-survey}.

\begin{defi}
 Let $X$ be a compact $2n$-dimensional manifold endowed with an almost-complex structure $J$ and with a $J$-Hermitian metric $g$; denote by $\omega$ the $(1,1)$-form associated to $g$. One says that the {\em Lefschetz-type property (on $2$-forms)} holds on $X$ if
 $$ L^{n-2}_\omega \colon \wedge^2X\to\wedge^{2n-2}X $$
 takes $g$-harmonic $2$-forms to $g$-harmonic $(2n-2)$-forms.
\end{defi}

Since the map $L^k \colon \wedge^{n-k} X \to \wedge^{n+k} X$ is an isomorphism for every $k\in\N$, \cite[Corollary 2.7]{yan}, it follows that the Lefschetz-type property on $2$-forms is stronger than the Hard Lefschetz Condition on $2$-classes, namely, the property that $\left[\omega\right]^{n-2} \cp \sspace \colon H^2_{dR}(X;\R) \to H^{2n-2}_{dR}(X;\R)$ is an isomorphism.

\medskip

In order to study the relation between the Lefschetz-type property on $2$-forms and the \Cf ness, we prove here the following result, \cite[Theorem 2.3]{angella-tomassini-zhang}, which states that the Lefschetz-type property on $2$-forms is satisfied provided that the almost-K\"ahler structure admits a basis of pure type harmonic representatives for the second de Rham cohomology group. (Recall that A. Fino and A. Tomassini proved in \cite[Theorem 3.7]{fino-tomassini} that an almost-K\"ahler manifold admitting a basis of harmonic $2$-forms of pure type with respect to the almost-complex structure is \Cpf\ and \pf; they also described several examples of non-K\"ahler solvmanifolds satisfying the above assumption, \cite[\S5, \S6]{fino-tomassini}.)

\begin{thm}\label{thm:almost-kahler-cpf}
 Let $X$ be a compact manifold endowed with an almost-K\"ahler structure $\left(J,\,\omega,\,g\right)$. Suppose that there exists a basis of $H^2_{dR}(X;\R)$ represented by $g$-harmonic $2$-forms which are of pure type with respect to $J$. Then the Lefschetz-type property on $2$-forms holds on $X$.
\end{thm}

\begin{proof}
We recall that, on a compact $2n$-dimensional symplectic manifold, using the symplectic form $\omega$ instead of a Riemannian metric and miming the Hodge theory for Riemannian manifolds, one can define a symplectic-$\star$-operator $\star_\omega\colon \wedge^\bullet X \to \wedge^{2n-\bullet}X$ such that $\alpha\wedge\star_\omega \beta = \left(\omega^{-1}\right)^k\left(\alpha,\beta\right)\,\frac{\omega^n}{n!}$ for every $\alpha,\,\beta\in\wedge^kX$,
see \cite[\S2]{brylinski}. (See \S\ref{sec:symplectic} for further details on symplectic structures, and see \S\ref{subsec:symplectic-hodge-theory} for definitions and results concerning the Hodge theory for symplectic manifolds.) In particular, on a compact manifold $X$ endowed with an almost-K\"ahler structure $\left(J,\, \omega,\, g\right)$, the Hodge-$*$-operator $*_g$ and the symplectic-$\star$-operator $\star_\omega$ are related by
$$ \star_\omega \;=\; *_g\,J \;,$$
see \cite[Theorem 2.4.1, Remark 2.4.4]{brylinski}. Therefore, for forms of pure type with respect to $J$, the properties of being $g$-harmonic and of being $\omega$-symplectically-harmonic (that is, both $\de$-closed and $\de^\Lambda$-closed, where $\de^\Lambda$ is the symplectic co-differential operator, which is defined, for every $k\in\N$, as $\de^\Lambda\lfloor_{\wedge^kX}:=\left(-1\right)^{k+1}\,\star_\omega\,\de\,\star_\omega$) are equivalent. The statement follows noting that
$$ \left[\de,\, L \right]\;=\;0 \qquad \text{ and }\qquad \left[\de^\Lambda,\, L\right]\;=\;-\de \;,$$
see, e.g., \cite[Lemma 1.2]{yan}: hence $L$ sends $\omega$-symplectically-harmonic $2$-forms (of pure type with respect to $J$) to $\omega$-symplectically-harmonic $(2n-2)$-forms (of pure type with respect to $J$).
\end{proof}

\begin{rem} Note that, if $X$ is a compact $2n$-dimensional manifold endowed with an almost-K\"ahler structure $\left(J,\, \omega,\, g\right)$ satisfying the Lefschetz-type property on $2$-forms and $J$ is \Cf, then $J$ is \Cpf\ and \pf, too, \cite[Remark 2.4]{angella-tomassini-zhang}.

Indeed, we have already noticed that $J$ is \Cp\ by \cite[Proposition 2.8]{draghici-li-zhang} or \cite[Proposition 3.2]{fino-tomassini}.
Moreover, since $J$ is \Cf, $J$ is also \p\ by \cite[Proposition 2.5]{li-zhang}. We recall now the argument in \cite[Theorem 4.1]{fino-tomassini} to prove that $J$ is also \f.
Firstly, note that if the Lefschetz-type property on $2$-forms holds, then 
$\left[\omega^{n-2}\right] \cp \sspace\colon H^{2}_{dR}\left(X;\R\right)\to H^{2n-2}_{dR}\left(X;\R\right)$ 
is an isomorphism.
Therefore, we get that
$$ H^{2n-2}_{dR}(X;\R) \;=\; H^{(n,n-2),(n-2,n)}_J(X;\R) + H^{(n-1,n-1)}_J(X;\R) \;; $$
indeed, (following the argument in \cite[Theorem 4.1]{fino-tomassini},) since 
$\left[\omega^{n-2}\right] \cp \sspace\colon H^2_{dR}(X;\R)\to H^{2n-2}_{dR}(X;\R)$ is in particular surjective, we have
 \begin{eqnarray*}
  H^{2n-2}_{dR}(X;\R) &=& \left[\omega^{n-2}\right]\cp H^{2}_{dR}(X;\R) \\[5pt]
  &=& \left[\omega^{n-2}\right] \cp \left(H^{(2,0),(0,2)}_J(X;\R)\oplus H^{(1,1)}_J(X;\R)\right) \\[5pt]
  &\subseteq& H^{(n,n-2),(n-2,n)}_J(X;\R) + H^{(n-1,n-1)}_J(X;\R) \;,
 \end{eqnarray*}
yielding the above decomposition of $H^{2n-2}_{dR}(X;\R)$.
Then, it follows that $J$ is also \f\ by Theorem \ref{thm:implicazioni}.
\end{rem}

\medskip

We describe here some examples, from \cite{angella-tomassini-zhang}, of almost-K\"ahler manifolds, studying Lefschetz-type property and \Cf ness on them.

In the following example, we give a family of \Cf\ almost-K\"ahler manifolds satisfying the Lefschetz-type property on $2$-forms, \cite[\S2.2]{angella-tomassini-zhang}.

\begin{ex}
{\itshape A family of \Cf\ almost-K\"ahler manifolds satisfying the Lefschetz-type property on $2$-forms.}\\
Consider the $6$-dimensional Lie algebra
$$ \mathfrak{h}_7 \;:=\; \left(0^3,\, 23,\,13,\,12\right) \;.$$
By Mal'tsev's theorem \cite[Theorem 7]{malcev}, the connected simply-connected Lie group $G$ associated with $\mathfrak{h}_7$ admits a discrete co-compact subgroup $\Gamma$: let $N:=\left.\Gamma\right\backslash G$ be the nilmanifold obtained as a quotient of $G$ by $\Gamma$. Note that $N$ is not formal by K. Hasegawa's theorem \cite[Theorem 1, Corollary]{hasegawa_pams}.

Fix $\alpha>1$ and consider the $G$-left-invariant symplectic form $\omega_\alpha$ on $N$ defined by
$$ \omega_\alpha \;:=\; e^{14}+\alpha\cdot e^{25}+\left(\alpha-1\right)\cdot e^{36} \;. $$
Consider the left-invariant almost-complex structure $J$ on $N$ defined by
$$
\begin{array}{lll}
J_\alpha \, e_1 \;:=\; e_4 \;, &\qquad J_\alpha\, e_2\;:=\;\alpha\, e_5\;, &\qquad
J_\alpha \, e_3\;:=\;(\alpha -1)\,e_6\;,\,\,\,\,\vspace{.2cm}\\
J_\alpha \,e_4\;:=\;-e_1\;,&\qquad
J_\alpha \,e_5\;:=\;-\frac{1}{\alpha}\,e_2\;, &\qquad
J_\alpha\, e_6\;:=\;-\frac{1}{\alpha -1}\,e_3\;,
\end{array}
$$
where $\{e_1,\, \ldots,\, e_6\}$ denotes the global dual frame of the $G$-left-invariant co-frame $\{e^1,\, \ldots,\, e^6\}$ associated to the structure equations.

Finally, define the $G$-left-invariant symmetric tensor
$$ g_{\alpha}(\sspace,\, \ssspace) \;:=\; \omega_\alpha\left(\sspace,\, J_\alpha \ssspace\right) \;. $$

It is straightforward to check that $\left\{\left(J_\alpha,\,\omega_\alpha,\,g_{\alpha}\right)\right\}_{\alpha>1}$ is a family of $G$-left-invariant almost-K\"ahler structures on $N$; moreover, setting
\begin{eqnarray*}
E_\alpha^1 \;:=\; e^1\;, &\qquad E_\alpha^2\;:=\; \alpha\, e^2\;, &\qquad E_\alpha^3 \;:=\; (\alpha-1)\, e^3\;, \\[5pt]
E_\alpha^4 \;:=\; e^4\;, &\qquad E_\alpha^5\;:=\; e^5\;,          &\qquad E_\alpha^6 \;:=\; e^6\;,
\end{eqnarray*}
we get the $G$-left-invariant $g_{\alpha}$-orthonormal co-frame $\left\{E_\alpha^1,\, \ldots,\, E_\alpha^6\right\}$ on $N$.
The structure equations with respect to the co-frame $\left\{E_\alpha^1,\, \ldots,\, E_\alpha^6\right\}$ read as follows:
$$
\left\{
\begin{array}{rcl}
 \de E_\alpha^1 &=& 0 \\[5pt]
 \de E_\alpha^2 &=& 0 \\[5pt]
 \de E_\alpha^3 &=& 0 \\[5pt]
 \de E_\alpha^4 &=& \frac{1}{\alpha(\alpha -1)}\, E_\alpha^{23} \\[5pt]
 \de E_\alpha^5 &=& \frac{1}{\alpha -1}\, E_\alpha^{13} \\[5pt]
 \de E_\alpha^6 &=& \frac{1}{\alpha}\,E_\alpha^{12}
\end{array}
\right. \;.
$$
Then
$$
 \varphi_\alpha^1 \;:=\; E_\alpha^1+\im E_\alpha^4 \;, \qquad
 \varphi_\alpha^2 \;:=\; E_\alpha^2+\im E_\alpha^5 \;, \qquad
 \varphi_\alpha^3 \;:=\; E_\alpha^3+\im E_\alpha^6
$$
are $(1,0)$-forms for the almost-complex structure $J_\alpha$, and
$$
\omega_\alpha \;=\; E_\alpha^1\wedge E_\alpha^4 + E_\alpha^2\wedge E_\alpha^5 +E_\alpha^3\wedge E_\alpha^6\,.
$$
By K. Nomizu's theorem \cite[Theorem 1]{nomizu}, the de Rham cohomology of $N$ is straightforwardly computed:
\begin{eqnarray*}
H^2_{dR}(N;\R) &=& \R\left\langle E_\alpha^{15},\; E_\alpha^{16},\; E_\alpha^{24},\; E_\alpha^{26},\; E_\alpha^{34},\; E_\alpha^{35},\; E_\alpha^{14}+\frac{1}{\alpha}\,E_\alpha^{25},\; \frac{1}{\alpha}\,E_\alpha^{25}+\frac{1}{\alpha-1}\,E_\alpha^{36} \right\rangle \\[5pt]
& = & \R\left\langle \im\,\alpha\,\varphi_\alpha^{1\bar1}+\im\,\varphi_\alpha^{2\bar2},\; \im\,(\alpha-1)\,\varphi_\alpha^{2\bar2}+
 \im\,\alpha\,\varphi_\alpha^{3\bar3},\;\Im\,\varphi_\alpha^{1\bar2},\;\Im\,\varphi_\alpha^{1\bar3},\;\Im\,\varphi_\alpha^{3\bar2} \right\rangle \\[5pt]
&& \oplus \left\langle \Im\,\varphi_\alpha^{12},\;\Im\,\varphi_\alpha^{13},\;\Im\,\varphi_\alpha^{23} \right\rangle \;.
\end{eqnarray*}
Note that the $g_{\alpha}$-harmonic representatives of the above basis of $H^2_{dR}(N;\R)$ are of pure type with respect to $J_\alpha$: hence, the almost-complex structure $J_\alpha$ is \Cpf\ and \pf\ by \cite[Theorem 3.7]{fino-tomassini};
furthermore, by Theorem \ref{thm:almost-kahler-cpf}, the Lefschetz-type property on $2$-forms holds on $N$ endowed with the almost-K\"ahler structure $\left(J_\alpha,\,  \omega_\alpha,\, g_\alpha\right)$, where $\alpha>1$.
Moreover, we get
$$ h^+_{J_\alpha}(N)\;=\; 5\;, \qquad h^-_{J_\alpha}(N)\;=\; 3\;.$$
On the other hand, one can explicitly note that
$$
\begin{array}{rcccl}
 L_{\omega_\alpha}E_\alpha^{15} &=& E_\alpha^{1536} &=& *_{g_{\alpha}}\, E_\alpha^{24} \;,\\[5pt]
 L_{\omega_\alpha}E_\alpha^{16} &=& E_\alpha^{1625} &=& *_{g_{\alpha}}\, E_\alpha^{34} \;,\\[5pt]
 L_{\omega_\alpha}E_\alpha^{24} &=& E_\alpha^{2436} &=& *_{g_{\alpha}}\, E_\alpha^{15} \;,\\[5pt]
 L_{\omega_\alpha}E_\alpha^{26} &=& E_\alpha^{2614} &=& *_{g_{\alpha}}\, E_\alpha^{35} \;,\\[5pt]
 L_{\omega_\alpha}E_\alpha^{34} &=& E_\alpha^{3425} &=& *_{g_{\alpha}}\, E_\alpha^{16} \;,\\[5pt]
 L_{\omega_\alpha}E_\alpha^{35} &=& E_\alpha^{3514} &=& *_{g_{\alpha}}\, E_\alpha^{26} \;,
\end{array}
$$
and
$$
\de\, *_{g_{\alpha}}\, L_{\omega_\alpha}\,\left(E_\alpha^{14} + \frac{1}{\alpha}\,E_\alpha^{25}\right) \;=\; \de\left(-\frac{\alpha+1}{\alpha}\, E_\alpha^{36} - E_\alpha^{25} - \frac{1}{\alpha}\,E_\alpha^{14}\right) \;=\; 0 \;,
$$
and
$$
\de\, *_{g_{\alpha}}\, L_{\omega_\alpha}\,\left(e^{25}+e^{36}\right) =0 \;;
$$
this proves explicitly that the the Lefschetz-type property on $2$-forms holds on $N$ endowed with the almost-K\"ahler structure $\left(J_\alpha,\,  \omega_\alpha,\, g_\alpha\right)$, where $\alpha>1$.

Note that, while $\omega_\alpha \wedge \sspace \colon \wedge^2N \to\wedge^4N$ induces an isomorphism $\left[\omega_\alpha\right]\cp\sspace \colon H^2_{dR}(N;\R)\stackrel{\simeq}{\to} H^4_{dR}(N;\R)$ in cohomology, the map $\left[\omega_\alpha\right]^2 \cp \sspace \colon H^1_{dR}(N;\R)\to H^5_{dR}(N;\R)$ is not an isomorphism, according to \cite[Theorem A]{benson-gordon-nilmanifolds}.

We show explicitly that the nilmanifold $N$ is not formal, without using K. Hasegawa's theorem \cite[Theorem 1, Corollary]{hasegawa_pams}. By \cite[Corollary 1]{deligne-griffiths-morgan-sullivan}, every Massey product on a formal manifold is zero. Since
$$
\left[E_\alpha^1\right]\cp \left[E_\alpha^3\right] \;=\; \left(\alpha-1\right)\,\left[\de E_\alpha^5\right]\;=\; 0
\quad \text{ and }\quad
\left[E_\alpha^3\right]\cp \left[E_\alpha^2\right] \;=\; -\alpha\,\left(\alpha-1\right)\,\left[\de E_\alpha^4\right]\;=\; 0 \;,
$$
the triple Massey product
$$
\left\langle [E_\alpha^1],\, [E_\alpha^3],\, [E_\alpha^2] \right\rangle \;=\; -\left(\alpha-1\right) \, \left[ E_\alpha^{25}+\alpha\,E_\alpha^{14}\right]
$$
is not zero, and hence $N$ is not formal.
\end{ex}

Summarizing, we state the following result, \cite[Proposition 2.5]{angella-tomassini-zhang}.
\begin{prop}
There exists a non-formal $6$-dimensional nilmanifold endowed with an $1$-parameter family $\left\{\left(J_\alpha,\, \omega_\alpha,\, g_{\alpha}\right)\right\}_{\alpha>1}$ of left-invariant almost-K\"ahler structures, such that $J_\alpha$ is \Cpf\ and \pf, and for which the Lefschetz-type property on $2$-forms holds.
\end{prop}

In the following example, we give a \Cpf\ almost-K\"ahler structure on the completely-solvable Nakamura manifold, \cite[\S3]{angella-tomassini-zhang}.

\begin{ex}
{\itshape A \Cpf\ almost-K\"ahler structure on the completely-solvable Nakamura manifold.}\\
Firstly, we recall the construction of the \emph{completely-solvable Nakamura manifold}: it is a completely-solvable solvmanifold diffeomorphic to the \emph{Nakamura manifold} studied by I. Nakamura in \cite[page 90]{nakamura}, and it is an example of a cohomologically-K\"ahler non-K\"ahler manifold, \cite{deandres-fernandez-deleon-mencia}, \cite[Example 3.1]{fernandez-munoz-santisteban}, \cite[\S3]{debartolomeis-tomassini}.

Take $A\in\SL(2;\Z)$ with two different real positive eigenvalues $\esp^\lambda$ and $\esp^{-\lambda}$ with $\lambda>0$,
and fix $P\in\GL(2;\R)$ such that $PAP^{-1}=\mathrm{diag}\left(\esp^{\lambda},\,\esp^{-\lambda}\right)$. For example, take
$$ A \;:=\;
\left(
\begin{array}{cc}
 2 & 1 \\
 1 & 1
\end{array}
\right) \;,
\quad \text{ and } \quad
P \;:=\;
\left(
\begin{array}{cc}
 \frac{1-\sqrt{5}}{2} & 1 \\
 1 & \frac{\sqrt{5}-1}{2}
\end{array}
\right) \;,
$$
and consequently $\lambda=\log\frac{3+\sqrt{5}}{2}$.

Let $M^6:=:M^6(\lambda)$ be the $6$-dimensional completely-solvable solvmanifold
$$ M^6 := \mathbb{S}^1_{x^2} \times \frac{\R_{x^1}\times\T^2_{\C,\,\left(x^3,\,x^4,\,x^5,\,x^6\right)}}
{\left\langle T_1\right\rangle} $$
where $\T^2_\C$ is the $2$-dimensional complex torus
$$
\T^2_\C := \frac{\C^2}{P\Z\left[\im\right]^2}
$$
and $T_1$ acts on $\R\times \T^2_\C$ as
$$ T_1\left(x^1,\,x^3,\,x^4,\,x^5,\,x^6\right) :=\left(x^1+\lambda,\, \esp^{-\lambda}x^3,\, \esp^{\lambda}x^4,\,
\esp^{-\lambda}x^5,\, \esp^{\lambda}x^6\right) \;.$$
Using coordinates $x^2$ on $\mathbb{S}^1$, $x^1$ on $\R$ and $\left(x^3,x^4,x^5,x^6\right)$ on $\T^2_\C$, we set
$$
\begin{array}{ll}
 e^1 \;:=\; \de x^1\;,           \qquad & \qquad e^2 \;:=\; \de x^2\;,\\[5pt]
 e^3 \;:=\; \esp^{x^1}\de x^3\;, \qquad & \qquad e^4 \;:=\; \esp^{-x^1}\de x^4\;,\\[5pt]
 e^5 \;:=\; \esp^{x^1}\de x^5\;, \qquad & \qquad e^6 \;:=\; \esp^{-x^1}\de x^6\;.
\end{array}
$$
as a basis for $\duale{\mathfrak{g}}$, where $\mathfrak{g}$ denotes the Lie algebra naturally associated to $M^6$; therefore, with respect to $\{e^i\}_{i\in\{1,\ldots,6\}}$, the structure equations are the following:
$$
\left\{
\begin{array}{rcl}
 \de e^1 &=& 0 \\[5pt]
 \de e^2 &=& 0 \\[5pt]
 \de e^3 &=& e^1\wedge e^3 \\[5pt]
 \de e^4 &=& -e^1\wedge e^4 \\[5pt]
 \de e^5 &=& e^1\wedge e^5 \\[5pt]
 \de e^6 &=& -e^1\wedge e^6
\end{array}
\right. \;.
$$

Let $J$ be the almost-complex structure on $M^6$ defined requiring that a co-frame for the space of complex $(1,0)$-forms is given by
$$
\left\{
\begin{array}{rcl}
 \varphi^1 &:=& \frac{1}{2}\left(e^1+\im e^2\right) \\[5pt]
 \varphi^2 &:=& e^3+\im e^5 \\[5pt]
 \varphi^3 &:=& e^4+\im e^6
\end{array}
\right.\;.
$$
It is straightforward to check that $J$ is integrable.

Being $M^6$ a compact quotient of a completely-solvable Lie group, one computes the de Rham cohomology of $M^6$ by
A. Hattori's theorem \cite[Corollary 4.2]{hattori}:
\begin{eqnarray*}
H^0_{dR}\left(M^6;\C\right) &=& \C\left\langle 1 \right\rangle \;, \\[5pt]
H^1_{dR}\left(M^6;\C\right) &=& \C\left\langle \varphi^1, \, \bar\varphi^1 \right\rangle \;, \\[5pt]
H^2_{dR}\left(M^6;\C\right) &=& \C\left\langle \varphi^{1\bar1}, \, \varphi^{2\bar3}, \, \varphi^{3\bar2}, \,
\varphi^{23}, \, \varphi^{\bar2\bar3}  \right\rangle \;, \\[5pt]
H^3_{dR}\left(M^6;\C\right) &=& \C\left\langle \varphi^{12\bar3}, \, \varphi^{13\bar2}, \, \varphi^{123}, \,
\varphi^{1\bar2\bar3}, \, \varphi^{2\bar1\bar3}, \, \varphi^{3\bar1\bar2}, \, \varphi^{23\bar1}, \, \varphi^{\bar1\bar2\bar3}
\right\rangle
\end{eqnarray*}
(as usually, for the sake of clearness, we write, for example, $\varphi^{A\bar B}$ in place of $\varphi^A\wedge\bar\varphi^B$, and we list
the harmonic representatives with respect to the metric $g:=\sum_{j=1}^{3} \varphi^j\odot\bar\varphi^j$ instead of their classes).
Therefore, \cite[Proposition 3.2]{fernandez-munoz-santisteban}:
\begin{inparaenum}[\itshape (i)]
 \item $M^6$ is {\em geometrically formal}, that is, the product of $g$-harmonic forms is still $g$-harmonic, and therefore it
is formal;
 \item furthermore,
$$
\omega \;:=\; e^{12}+e^{34}+e^{56}
$$
is a symplectic form on $M^6$ satisfying the Hard Lefschetz Condition.
\end{inparaenum}

Note also that $\tilde\omega:=\frac{\im}{2}\left(\varphi^{1\bar1}+\varphi^{2\bar2}+\varphi^{3\bar3}\right)$
is not closed but $\de\tilde\omega^{2}=0$, from which it follows that the manifold $M^6$ admits a balanced metric.

Since $M^6$ is a compact quotient of a completely-solvable Lie group, by K. Hasegawa's theorem \cite[Main Theorem]{hasegawa-3},
the manifold $M^6$, endowed with any integrable almost-complex structure (e.g., the $J$ defined above), admits no K\"ahler structure, and it is not in class $\mathcal{C}$ of Fujiki, see also \cite[Theorem 3.3]{fernandez-munoz-santisteban}.

Therefore, we consider the (non-integrable) almost-complex structure $J'$ defined by
$$ J'\, e^1 \;:=\; -e^2\;, \qquad J'\, e^3 \;:=\; -e^4\;, \qquad J'\, e^5 \;:=\; -e^6 \;.$$
Set
$$
\left\{
\begin{array}{rcl}
 \psi^1 &:=& \frac{1}{2}\left(e^1+\im e^2\right) \\[5pt]
 \psi^2 &:=& e^3+\im e^4 \\[5pt]
 \psi^3 &:=& e^5+\im e^6
\end{array}
\right.
$$
as a co-frame for the space of $(1,0)$-forms on $M^6$ with respect to $J'$; the structure equations with respect to this co-frame are
$$
\left\{
\begin{array}{rcl}
 \de\psi^1 &=& 0 \\[5pt]
 \de\psi^2 &=& \psi^{1\bar2}+\psi^{\bar1\bar2} \\[5pt]
 \de\psi^3 &=& \psi^{1\bar3}+\psi^{\bar1\bar3}
\end{array}
\right. \;,
$$
from which it is clear that $J'$ is not integrable.

The $J'$-compatible $2$-form
$$ \omega' \;:=\; e^{12}+e^{34}+e^{56} $$
is $\de$-closed; hence $\left(J',\,\omega'\right)$ is an almost-K\"ahler structure on $M^6$.

Moreover, as already remarked, using A. Hattori's theorem \cite[Corollary 4.2]{hattori}, one gets
\begin{eqnarray*}
H^2_{dR}\left(M^6;\R\right) &=& \R\left\langle e^{12},\, e^{34},\, e^{56},\, -e^{36}+ e^{45},\, e^{36}+e^{45} \right\rangle \\[5pt]
&=& \underbrace{\R\left\langle \im\psi^{1\bar1},\, \im\psi^{2\bar2},\, \im\psi^{3\bar3},\, \im\left(\psi^{2\bar3}+\psi^{3\bar2}\right) \right\rangle}_{\subseteq H^+_{J'}\left(M^6\right)} \oplus \underbrace{\R\left\langle \im\left(\psi^{23}-\psi^{\bar2\bar3}\right)\right\rangle}_{\subseteq H^-_{J'}\left(M^6\right)} \;,
\end{eqnarray*}
where we have listed the harmonic representatives with respect to the metric $g':=\sum_{j=1}^6 e^{j}\odot e^{j}$
instead of their classes; note that the above $g'$-harmonic representatives are of pure type with respect to $J'$. Therefore, $J'$ is obviously \Cf; it is also \Cp\ by \cite[Proposition 3.2]{fino-tomassini}, or \cite[Proposition 2.8]{draghici-li-zhang}. Moreover, since any cohomology class in $H^+_{J'}\left(M^6\right)$, respectively in $H^-_{J'}\left(M^6\right)$, has a $\de$-closed $g'$-harmonic representative in $\wedge^{1,1}_{J'}M^6\cap\wedge^2 M^6$, respectively in $\left(\wedge^{2,0}_{J'}M^6\oplus\wedge^{0,2}_{J'}M^6\right)\cap\wedge^2 M^6$, then $J'$ is also \pf, by \cite[Theorem 3.7]{fino-tomassini}, and the Lefschetz-type property on $2$-forms holds, by Theorem \ref{thm:almost-kahler-cpf}.

One can explicitly check that the Lefschetz-type operator
$$
L_{\omega'}\colon\wedge^2M^6\to\wedge^4M^6
$$
takes $g'$-harmonic $2$-forms to $g'$-harmonic $4$-forms, since
$$
\begin{array}{ccccc}
 L_{\omega'}\,e^{12} &=& e^{1234}+e^{1256} &=& *_{g'}\,\left(e^{34}+e^{56}\right) \;,\\[5pt]
 L_{\omega'}\,e^{34} &=& e^{1234}+e^{3456} &=& *_{g'}\,\left(e^{12}+e^{56}\right) \;,\\[5pt]
 L_{\omega'}\,e^{56} &=& e^{1256}+e^{3456} &=& *_{g'}\,\left(e^{12}+e^{34}\right) \;,\\[5pt]
 L_{\omega'}\,e^{36} &=& e^{1236} &=& *_{g'}\,e^{45} \;,\\[5pt]
 L_{\omega'}\,e^{45} &=& e^{1245} &=& *_{g'}\,e^{36} \;.
\end{array}
$$
\end{ex}

Summarizing, the content of the last example yields the following result, \cite[Proposition 3.3]{angella-tomassini-zhang}.
\begin{prop}
The completely-solvable Nakamura manifold $M^6$ admits
\begin{itemize}
 \item both a \Cpf\ and \pf\ complex structure $J$, and
 \item a \Cpf\ and \pf\ almost-K\"ahler structure $\left(J',\,\omega',\,g'\right)$, for which the Lefschetz-type property on $2$-forms holds.
\end{itemize}
\end{prop}

\medskip

Finally, in the following example, we give a non-\Cf\ almost-K\"ahler structure, \cite[\S4]{angella-tomassini-zhang}. In particular, this provides another strong difference between the (non-integrable) almost-K\"ahler case and the (integrable) K\"ahler case, all the compact K\"ahler manifolds being \Cpf\ by \cite[Lemma 2.15, Theorem 2.16]{draghici-li-zhang}, or \cite[Proposition 2.1]{li-zhang}.

\begin{ex}
\label{ex:almost-kahler-non-Cf}
{\itshape An almost-K\"ahler non-\Cf\ structure for which the Lefschetz-type property on $2$-forms does not hold.}\\
Consider the Iwasawa manifold $\mathbb{I}_3 := \left. \mathbb{H}\left(3;\Z\left[\im\right]\right) \right\backslash \mathbb{H}(3;\C)$, see \S\ref{subsec:iwasawa}.
Recall that, given the standard complex structure induced by the one on $\C^3$ and setting $\left\{\varphi^1,\,\varphi^2,\,\varphi^3\right\}$ as a global co-frame for the $(1,0)$-forms on $\mathbb{I}_3$, by K. Nomizu's theorem \cite[Theorem 1]{nomizu} one gets
\begin{eqnarray*}
 H^2_{dR}\left(\mathbb{I}_3;\C\right) &=& \R\left\langle \varphi^{13}+\varphi^{\bar{1}\bar{3}},\; \im\left(\varphi^{13}-\varphi^{\bar{1}\bar{3}}\right),\;\varphi^{23}+\varphi^{\bar{2}\bar{3}},\; \im\left(\varphi^{23}-\varphi^{\bar{2}\bar{3}}\right),\; \varphi^{1\bar{2}}-\varphi^{2\bar{1}}, \right. \\[5pt]
 && \left. \im\left(\varphi^{1\bar{2}}+\varphi^{2\bar{1}}\right),\;\im\varphi^{1\bar{1}},\;\im\varphi^{2\bar{2}} \right\rangle \,\otimes_\R\, \C \;,
\end{eqnarray*}
where we have listed the harmonic representatives with respect to the metric $g:=\sum_{h=1}^3\varphi^h\odot\bar\varphi^h$ instead of their classes.
Using the co-frame $\left\{e^1,\, \ldots,\, e^6\right\}$ of the cotangent bundle defined by
$$
\varphi^{1}=:e^1+\im e^2\;,\qquad \varphi^{2}=:e^3+\im e^4\;,\qquad \varphi^{3}=:e^5+\im e^6\;,
$$
one computes the structure equations
$$
\de e^1 \;=\; \de e^2 \;=\; \de e^3 \;=\; \de e^4 \;=\; 0\;, \qquad \de e^5 \;=\; -e^{13}+e^{24}\;, \qquad \de e^6 \;=\; -e^{14}-e^{23}\;.
$$
Therefore
\begin{eqnarray*}
H^2_{dR}\left(\mathbb{I}_3;\R\right) &=& \R\left\langle e^{15}-e^{26},\; e^{16}+e^{25},\; e^{35}-e^{46},\; e^{36}+e^{45},\; e^{13}+e^{24},\; e^{23}-e^{14},\; e^{12},\; e^{34} \right\rangle \;.
\end{eqnarray*}
Consider the almost-complex structure $J$ on $X$ defined by
$$
J\,e^1\;:=\; -e^6\;,\qquad  J\,e^2\;:=\; -e^5\;,\qquad J\,e^3\;:=\; -e^4 \;,
$$
and set
$$
\omega \;:=\; e^{16}+e^{25}+e^{34}\;.
$$
Then $\left(J,\,\omega,\,g\right)$ is an almost-K\"ahler structure on the Iwasawa manifold $\mathbb{I}_3$. We easily get that
$$
 \R\left\langle e^{16}+e^{25},\, \left(e^{35}-e^{46}\right)+\left(e^{13}+e^{24}\right),\,\left(e^{36}+e^{45}\right)-\left(e^{23}-e^{14}\right),\,e^{34}\right\rangle \subseteq H^+_J\left(\mathbb{I}_3\right)
$$
and
$$
 \R\left\langle e^{15}-e^{26},\,\left(e^{35}-e^{46}\right)-\left(e^{13}+e^{24}\right),\,\left(e^{36}+e^{45}\right)+\left(e^{23}-e^{14}\right)\right\rangle \subseteq H^-_J\left(\mathbb{I}_3\right)\;.
$$

We claim that the previous inclusions are actually equalities, and in particular that $J$ is a non-\Cf\ almost-K\"ahler structure on $\mathbb{I}_3$.
Indeed, we firstly note that, by \cite[Proposition 3.2]{fino-tomassini} or \cite[Proposition 2.8]{draghici-li-zhang}, $J$ is \Cp, since it admits a symplectic structure compatible with it. Moreover, we recall that a \Cf\ almost-complex structure is also \p\ by \cite[Proposition 2.5]{li-zhang}, and therefore it is also \Cp\ at the \kth{4} stage, by Theorem \ref{thm:implicazioni}, that is,
$$
H^{(3,1),(1,3)}_J\left(\mathbb{I}_3;\R\right)\cap H^{(2,2)}_J\left(\mathbb{I}_3;\R\right) \;=\; \{0\} \;.
$$
Therefore, our claim reduces to prove that $J$ is not \Cp\ at the \kth{4} stage. Note that
\begin{eqnarray*}
 0 \;\neq\; \left[e^{3456}\right] &=& \left[ e^{3456}-\de e^{135}\right]=\left[e^{3456}+e^{1234}\right] \\[5pt]
                                  &=& \left[ e^{3456}+\de e^{135}\right]=\left[e^{3456}-e^{1234}\right] \;,
\end{eqnarray*}
and that $e^{3456}+e^{1234}\in\left(\wedge^{3,1}_J\mathbb{I}_3\oplus\wedge^{1,3}_J\mathbb{I}_3\right)\cap \wedge^4\mathbb{I}_3$, while $e^{3456}-e^{1234}\in\wedge^{2,2}_J\mathbb{I}_3\cap \wedge^4\mathbb{I}_3$, and so
$H^{(3,1),(1,3)}_J\left(\mathbb{I}_3;\R\right) \cap H^{(2,2)}_J\left(\mathbb{I}_3;\R\right) \ni \left[e^{3456}\right]$, therefore $J$ is not \Cp\ at the \kth{4} stage, and hence it is not \Cf.

Let $L_\omega$ be the Lefschetz-type operator associated to the almost-K\"ahler structure $\left(J,\, \omega,\, g\right)$. Then, we have
$$ L_\omega \left(e^{12}\right) \;=\; e^{1234} \;=\; \de\left(e^{245}\right) \;,$$
namely, $L_\omega$ does not take $g$-harmonic $2$-forms to $g$-harmonic $4$-forms.
\end{ex}

The previous example proves the following result, \cite[Proposition 4.1]{angella-tomassini-zhang}.
\begin{prop}\label{prop:almost-kahler-non-cf}
The differentiable manifold $X$ underlying the Iwasawa manifold $\mathbb{I}_3 := \left. \mathbb{H}\left(3;\Z\left[\im\right]\right) \right\backslash \mathbb{H}(3;\C)$ admits an almost-K\"ahler structure $\left(J,\,\omega,\,g\right)$ which is \Cp\ and non-\Cf, and for which the Lefschetz-type property on $2$-forms does not hold.
\end{prop}

\medskip

The argument of the proof of \cite[Theorem 2.3]{draghici-li-zhang} suggests the following question, \cite[Question 3.4]{angella-tomassini-zhang}, compare also \cite[\S2]{draghici-li-zhang-survey}, in accordance with Proposition \ref{prop:almost-kahler-non-cf}.

\begin{question}
 Let $X$ be a compact $2n$-dimensional manifold endowed with an almost-K\"ahler structure $\left(J,\,\omega,\,g\right)$ such that the Lefschetz-type property on $2$-forms holds. Is $J$ \Cf?
\end{question}

\section{\Cpf ness and deformations of (almost-)complex structures}\label{sec:deformations-cpf}

In this section, we are interested in studying the behaviour of the cohomological decomposition of the de Rham cohomology of an (almost-)complex manifold under small deformations of the complex structure and along curves of almost-complex structures.

More precisely, we prove that being \Cpf\ is not a stable property under small deformations of the complex structure, Theorem \ref{thm:instability-iwasawa}, as a consequence of the study of the \Cpf ness for small deformations of the Iwasawa manifold, Theorem \ref{thm:instability-iwasawa}. Then we study some explicit examples of curves of almost-complex structures on compact manifolds: by using a construction introduced by J. Lee, \cite[\S1]{lee}, we construct a curve of almost-complex structures along which the property of being \Cpf\ remains satisfied, Theorem \ref{thm:lee-curves}. In \S\ref{subsec:semicontinuity}, we provide counterexamples to the upper-semi-continuity of $t\mapsto H^-_{J_t}(X)$, Proposition \ref{prop:no-scs-h-}, and to the lower-semi-continuity of $t\mapsto H^+_{J_t}(X)$, Proposition \ref{prop:no-sci-h+}, where $\left\{J_t\right\}_t$ is a curve of almost-complex structures on a compact manifold $X$ of dimension greater than $4$; we also study a stronger semi-continuity problem,
\S\ref{subsubsec:stronger-semicont}.

The results in this section have been obtained in joint work with A. Tomassini, \cite{angella-tomassini-1, angella-tomassini-2}.

\subsection{Deformations of \Cpf\ almost-complex structures}

In this section, we consider the problem of the stability of the \Cpf ness under small deformations of the complex structure and along curves of almost-complex structures.

\subsubsection{Instability of \Cpf\ property}

We recall that a property concerning compact complex (respectively, almost-complex) manifolds (e.g., admitting K\"ahler metrics, admitting balanced metrics, satisfying the $\del\delbar$-Lemma, admitting compatible symplectic structures) is called \emph{stable under small deformations of the complex} (respectively, \emph{almost-complex}) \emph{structure} if, for every complex-analytic family $\left\{X_t:=:\left(X,\, J_t\right)\right\}_{t\in B}$ of compact complex manifolds (respectively, for every smooth curve $\left\{J_t\right\}_{t\in B}$ of almost-complex structures on a compact differentiable manifold $X$), whenever the property holds for $\left(X,\, J_t\right)$ for some $t\in B$, it holds also for $\left(X,\, J_s\right)$ for any $s$ in a neighbourhood of $t$ in $B$.

\medskip

The main result in the context of stability under small deformations of the complex structure is the following classical theorem by K. Kodaira and D.~C. Spencer, \cite{kodaira-spencer-3}, which actually holds for differentiable families of compact complex manifolds.

\begin{thm}[{\cite[Theorem 15]{kodaira-spencer-3}}]
 For a compact manifold, admitting a K\"ahler structure is a stable property under small deformations of the complex structure.
\end{thm}

\begin{rem}
 Conditions under which the property of admitting a balanced metric is stable under small deformations of the complex structure have been studied by C.-C. Wu \cite[\S5]{wu}, and by J. Fu and S.-T. Yau \cite{fu-yau}.
\end{rem}

\medskip

Note that, by \cite[Theorem 5.4]{draghici-li-zhang-2}, see also \cite{donaldson}, on compact almost-complex manifolds of dimension $4$, the property of admitting an almost-K\"ahler structure is stable under small deformations of the almost-complex structure. This result stands on the very special properties of $4$-dimensional manifolds, and does not hold true in higher dimension. More precisely, we provide here an explicit example, in dimension $6$, showing that, relaxing the integrability condition in the previous theorem (namely, starting with an almost-K\"ahler structure), we lose the stability under small deformations of the almost-complex structure.

\begin{ex}\label{ex:instability-alm-kahler}
\textit{A curve $\left\{J_t\right\}_t$ of almost-complex structures on a compact $6$-dimensional manifold such that $J_0$ admits an almost-K\"ahler structure and $J_t$, for $t\neq 0$, admits no almost-K\"ahler structure.}\\
For $c\in\R$, consider the completely-solvable Lie group
$$
\mathrm{Sol}(3)_{(x^1,y^1,z^1)} \;:=\;
\left\{
\left(
\begin{array}{cccc}
 \mathrm{e}^{c\,z^1} &   &   & x^1 \\
   & \mathrm{e}^{-c\,z^1} &   & y^1 \\
   &   & 1 & z^1 \\
   &   &   & 1
\end{array}
\right)\in \GL(4;\,\R)
\;\st\; x^1,\,y^1,\,z^1\in\R
\right\} \;.
$$
Choose a suitable $c\in\R$, for which there exists a co-compact discrete subgroup $\Gamma(c)\subset \mathrm{Sol}(3)$ such that
$$ M(c)_{(x^1,y^1,z^1)} \;:=\; \Gamma(c) \left\backslash \mathrm{Sol}(3)_{(x^1,y^1,z^1)} \right. $$
is a compact $3$-dimensional completely-solvable solvmanifold, \cite[\S3]{auslander-green-hahn}.

The manifold
$$ N^6(c) \;:=\; M(c)_{(x^1,y^1,z^1)} \,\times\, M(c)_{(x^2,y^2,z^2)} \;.$$
is cohomologically-K\"ahler, see \cite[Example 1]{benson-gordon-solvmanifolds}, is formal and has a symplectic structure satisfying the Hard Lefschetz Condition, but it admits no K\"{a}hler structure, see \cite[Theorem 3.5]{fernandez-munoz-santisteban}.

Consider $\left\{e^i\right\}_{i\in\{1,\ldots,6\}}$ as a $\left(\mathrm{Sol}(3)\times\mathrm{Sol}(3)\right)$-left-invariant co-frame for $N^6(c)$, where
\begin{eqnarray*}
 e^1 \,:=\, \mathrm{e}^{-c\,z^1}\,\de x^1 ,&\quad e^2 \,:=\, \mathrm{e}^{-c\,z^1}\,\de y^1 ,&\quad e^3 \,:=\, \de z^1 , \\[5pt]
 e^4 \,:=\, \mathrm{e}^{-c\,z^2}\,\de x^2 ,&\quad e^5 \,:=\, \mathrm{e}^{-c\,z^2}\,\de y^2 ,&\quad e^6 \,:=\, \de z^2 ;
\end{eqnarray*}
with respect to it, the structure equations are
$$
\left\{
\begin{array}{l}
 \de e^1 \;=\; \phantom{+}c\, e^1\wedge e^3 \\[5pt]
 \de e^2 \;=\; -c\, e^2\wedge e^3\\[5pt]
 \de e^3 \;=\; \phantom{+}0\\[5pt]
 \de e^4 \;=\; \phantom{+}c\, e^4\wedge e^6 \\[5pt]
 \de e^5 \;=\; -c\, e^5\wedge e^6 \\[5pt]
 \de e^6 \;=\; \phantom{+}0
\end{array}
\right. \;.
$$

By A. Hattori's theorem \cite[Corollary 4.2]{hattori}, it is straightforward to compute
$$ H^2_{dR}\left(N^6(c);\R\right) \;=\; \R\left\langle e^{1}\wedge e^{2},\; e^{3}\wedge e^{6},\; e^{4}\wedge e^{5} \right\rangle \;,$$
hence the space of $\left(\mathrm{Sol}(3)\times\mathrm{Sol}(3)\right)$-left-invariant $\de$-closed $2$-forms is
$$
\R\left\langle e^{12},\, e^{36},\, e^{45} \right\rangle \oplus \R\left\langle e^{13},\,e^{23},\,e^{45},\,e^{46} \right\rangle $$
(where, as usually, we shorten $e^{AB}:=e^A\wedge e^B$).

Let $J_0\in\End\left(TN^6(c)\right)$ be the almost-complex structure given, with respect to the frame $\left\{e_1,\ldots,e_6\right\}$ dual to $\left\{e^1,\ldots,e^6\right\}$, by
$$
J_0 \;:=\;
\left(
\begin{array}{cc|cccc}
 & -1 & & & & \\
 \phantom{+}1 & & & & & \\
\hline
 & & & & & -1 \\
 & & & & -1 & \\
 & & & \phantom{+}1 & & \\
 & & \phantom{+}1 & & &
\end{array}
\right) \;\in\; \End\left(TN^6(c)\right) \;.
$$
It is straightforward to check that $J_0$ admits almost-K\"ahler structures: more precisely, the cone $\mathcal{K}^c_{J_0,\,\textrm{inv}}$ of $\left(\mathrm{Sol}(3)\times\mathrm{Sol}(3)\right)$-left-invariant almost-K\"ahler structures on $\left(N^6(c),\,J_0\right)$ is
$$ \mathcal{K}^c_{J_0,\,\textrm{inv}} \;=\; \left\{ \alpha\,e^1\wedge e^2+\beta\,e^3\wedge e^6+\gamma\,e^4\wedge e^5 \st \alpha,\,\beta,\,\gamma >0 \right\} \;.$$

Take now
$$
L \;:=\;
\left(
\begin{array}{cccccc}
0 & 0 & 1 & 0 & 0 & 0 \\
0 & 0 & 0 & 0 & 0 & -1 \\
0 & 0 & 0 & 1 & 0 & 0 \\
0 & 0 & 0 & 0 & 0 & 0 \\
0 & 0 & 0 & 0 & 0 & 0 \\
0 & 0 & 0 & 0 & -1 & 0
\end{array}
\right) \;\in\; \End\left(TN^6(c)\right)
$$
and define, for $t\in\R$, the almost-complex structure
$$
J_t \;:=\; \left(\id_{TN^6(c)}-t\,L\right)\,J_0\,\left(\id_{TN^6(c)}-t\,L\right)^{-1} \;=\;
\left(
\begin{array}{cc|cccc}
 & 1 & & & 2\,t^2 & -2\,t \\
 -1 & & -2\,t & -2\,t^2 & & \\
\hline
 & & & & -2\,t & 1 \\
 & & & & 1 & \\
 & & & -1 & & \\
 & & -1 & -2\,t & &
\end{array}
\right)
\;\in\; \End\left(T^*N^6(c)\right) 
\;.
$$
We first prove that $J_t$ admits no $\left(\mathrm{Sol}(3)\times\mathrm{Sol}(3)\right)$-left-invariant almost-K\"ahler structure for $t\neq 0$. Indeed, for $t\neq 0$, the space of $\left(\mathrm{Sol}(3)\times\mathrm{Sol}(3)\right)$-left-invariant $\de$-closed $J_t$-invariant $2$-forms is
$$ \R\left\langle e^{36}+2t\,e^{46},\, e^{45} \right\rangle $$
and
$$ \left(\beta\,e^3\wedge e^6+\gamma\, e^4\wedge e^5+2t\,\beta\,e^4\wedge e^6\right)^3 \;=\; 0 \quad \text{ for every } \beta,\,\gamma\in\R  \;,$$
hence
$$ \mathcal{K}^c_{J_t,\,\textrm{inv}} \;=\; \varnothing  \qquad \text{ for } t\neq 0\;.$$
Now, using F.~A. Belgun's symmetrization trick, \cite[Theorem 7]{belgun}, we get that, if $J_t$ admits an almost-K\"ahler structure $\omega$, then it should admits a $\left(\mathrm{Sol}(3)\times\mathrm{Sol}(3)\right)$-left-invariant almost-K\"ahler structure
$$ \mu(\omega)\;:=\;\int_{N^6(c)} \omega\lfloor_m \, \eta(m) \;,$$
where $\eta$ is a $\left(\mathrm{Sol}(3)\times\mathrm{Sol}(3)\right)$-bi-invariant volume form on $N^6(c)$, whose existence is guaranteed by \cite[Lemma 6.2]{milnor}.
\end{ex}

We resume the content of the previous example in the following result.
\begin{thm}\label{thm:instability-almK}
 Being almost-K\"ahler is not a stable property along curves of almost-complex structures.
\end{thm}

\medskip

In view of K. Kodaira and D.~C. Spencer's theorem \cite[Theorem 15]{kodaira-spencer-3}, a natural question in non-K\"ahler geometry is what properties, weaker that the property of being K\"ahler, still remain stable under small deformations of the complex structure. This does not hold true, for example, for the balanced property, as proven in \cite[Proposition 4.1]{alessandrini-bassanelli} by L. Alessandrini and G. Bassanelli; on the other hand, the cohomological property of satisfying the $\del\delbar$-Lemma is stable under small deformations of the complex structure, as we have seen in Corollary \ref{cor:stab-del-delbar-lemma}, see also \cite[Proposition 9.21]{voisin}, or \cite[Theorem 5.12]{wu}, or \cite[\S B]{tomasiello}. We show now that the cohomological property of \Cpf ness turns out to be non-stable under small deformations of the complex structure, \cite[Theorem 3.2]{angella-tomassini-1}.

\begin{thm}
\label{thm:instability}
The properties of being \Cpf, or \Cp, or \Cf, or \pf, or \p, or \f\ are not stable under small deformations of the complex structure.
\end{thm}

The proof of Theorem \ref{thm:instability} follows studying explicitly \Cpf ness for small deformations of the standard complex structure on the Iwasawa manifold $\mathbb{I}_3$, \cite[Theorem 3.1]{angella-tomassini-1}. (We refer to \S\ref{subsec:iwasawa} for notations and results concerning the Iwasawa manifold and its Kuranishi space; we recall here that $\mathbb{I}_3$ is a holomorphically parallelizable nilmanifold of complex dimension $3$, and its Kuranishi space is smooth and depends on $6$ effective parameters; the small deformations of $\mathbb{I}_3$ can be divided into three classes, {\itshape (i)}, {\itshape (ii)}, and {\itshape (iii)}, according to their Hodge numbers; in particular, the Hodge numbers of the small deformations in class {\itshape (i)} are equal to the Hodge numbers of $\mathbb{I}_3$.)

\begin{thm}
\label{thm:instability-iwasawa}
 Let $\mathbb{I}_3 := \left. \mathbb{H}\left(3;\Z\left[\im\right]\right) \right\backslash \mathbb{H}(3;\C)$
be the Iwasawa manifold, endowed with the complex structure inherited by the standard complex structure on $\C^3$, and consider the small deformations in its Kuranishi space. Then:
\begin{itemize}
 \item the natural complex structure on $\I_3$ is \Cpf\ at every stage and \pf\ at every stage;
 \item the small deformations in class {\itshape (i)} are \Cpf\ at every stage and \pf\ at every stage;
 \item the small deformations in classes {\itshape (ii)} and {\itshape (iii)} are neither \Cp\ nor \Cf\ nor \p\ nor \f.
\end{itemize}
\end{thm}

\begin{proof}
We follow the notation introduced in \S\ref{subsec:iwasawa}; in particular, we recall that the structure equations with respect to a certain co-frame $\left\{\varphi^1_\tempo,\, \varphi^2_\tempo,\, \varphi^3_\tempo\right\}$ of the space of $(1,0)$-forms on $X_\tempo:=:\left(\mathbb{I}_3,\, J_\tempo\right)$, for $\tempo\in \Delta(0,\varepsilon)\subset\C^6$ with $\varepsilon>0$ small enough, are the following:
$$
\left\{
\begin{array}{rcl}
 \de\varphi^1_{\tempo} &=& 0 \\[10pt]
 \de\varphi^2_{\tempo} &=& 0 \\[10pt]
 \de\varphi^3_{\tempo} &=& \sigma_{12}\,\varphi^1_{\tempo}\wedge\varphi^2_{\tempo} + \sigma_{1\bar1}\,\varphi^1_{\tempo}\wedge\bar\varphi^1_{\tempo} + \sigma_{1\bar2}\,\varphi^1_{\tempo}\wedge\bar\varphi^2_{\tempo} + \sigma_{2\bar1}\,\varphi^2_{\tempo}\wedge\bar\varphi^1_{\tempo} + \sigma_{2\bar2}\,\varphi^2_{\tempo}\wedge\bar\varphi^2_{\tempo}
\end{array}
\right. \;,
$$
where $\sigma_{12},\,\sigma_{1\bar1},\,\sigma_{1\bar2},\,\sigma_{2\bar1},\,\sigma_{2\bar2}\in\C$ are complex numbers depending just on $\tempo$. The asymptotic behaviour of $\sigma_{12}$, $\sigma_{1\bar1}$, $\sigma_{1\bar2}$, $\sigma_{2\bar1}$, and $\sigma_{2\bar2}$ for $\tempo$ near $\zero$  is the following, see \S\ref{sec:structure-equations-iwasawa}:
$$
\left\{
\begin{array}{rcl}
\sigma_{12} &=& -1 +\opiccolo{\tempo} \\[5pt]
\sigma_{1\bar1} &=& t_{21} +\opiccolo{\tempo}  \\[5pt]
\sigma_{1\bar2} &=& t_{22} +\opiccolo{\tempo}  \\[5pt]
\sigma_{2\bar1} &=& -t_{11} +\opiccolo{\tempo} \\[5pt]
\sigma_{2\bar2} &=& -t_{12} +\opiccolo{\tempo}
\end{array}
\right.
\;;
$$
more precisely, for $\tempo$ in class {\itshape (i)}, respectively class {\itshape (ii)}, we actually have
$$
\left\{
\begin{array}{rcl}
\sigma_{12} &=& -1 \\[5pt]
\sigma_{1\bar1} &=& 0 \\[5pt]
\sigma_{1\bar2} &=& 0 \\[5pt]
\sigma_{2\bar1} &=& 0 \\[5pt]
\sigma_{2\bar2} &=& 0
\end{array}
\right. \qquad \text{ for } \qquad \tempo\in\text{ class {\itshape (i)}}\;,
$$
and
$$
\left\{
\begin{array}{rcl}
\sigma_{12} &=& -1 +\opiccolo{\tempo} \\[5pt]
\sigma_{1\bar1} &=& t_{21} \left(1+\opiccolouno\right)  \\[5pt]
\sigma_{1\bar2} &=& t_{22} \left(1+\opiccolouno\right)  \\[5pt]
\sigma_{2\bar1} &=& -t_{11} \left(1+\opiccolouno\right) \\[5pt]
\sigma_{2\bar2} &=& -t_{12} \left(1+\opiccolouno\right)
\end{array}
\right. \qquad \text{ for } \qquad \tempo\in\text{ class {\itshape (ii)}}\;.
$$

By K. Nomizu's theorem \cite[Theorem 1]{nomizu}, one computes straightforwardly the de Rham cohomology of $\mathbb{I}_3$ and of its small deformations; for the sake of clearness, we recall in the following table a basis of the space of the harmonic representatives of the de Rham cohomology classes with respect to the metric $g_\zero:=\sum_{j=1}^{3} \varphi^j_\zero\odot\bar\varphi^j_\zero$.

\smallskip
\begin{scriptsize}
\begin{center}
\begin{tabular}{>{$}c<{$} >{$}c<{$} | >{$}c<{$}}
\toprule
k & \K & g_\zero \text{\bfseries-harmonic representatives of } H^k_{dR}\left(\I_3;\K\right) \\
\midrule[0.02em]\addlinespace[1.5pt]\midrule[0.02em]
1 & \C & \varphi^1,\; \varphi^2,\; \bar{\varphi}^1,\; \bar{\varphi}^2 \\[10pt]
  & \R & \varphi^1+\bar{\varphi}^1,\; \im\left(\varphi^1-\bar{\varphi}^1\right),\; \varphi^2+\bar{\varphi}^2,\; \im\left(\varphi^2-\bar{\varphi}^2\right) \\
\midrule
2 & \C & \varphi^{13},\; \varphi^{23},\;\varphi^{1\bar{1}},\; \varphi^{1\bar{2}},\; \varphi^{2\bar{1}},\;\varphi^{2\bar{2}},\; \varphi^{\bar{1}\bar{3}},\; \varphi^{\bar{2}\bar{3}} \\[10pt]
  & \R & \varphi^{13}+\varphi^{\bar{1}\bar{3}},\; \im\left(\varphi^{13}-\varphi^{\bar{1}\bar{3}}\right),\;\varphi^{23}+\varphi^{\bar{2}\bar{3}},\; \im\left(\varphi^{23}-\varphi^{\bar{2}\bar{3}}\right),\, \varphi^{1\bar{2}}-\varphi^{2\bar{1}},\; \im\left(\varphi^{1\bar{2}}+\varphi^{2\bar{1}}\right),\;\im\varphi^{1\bar{1}},\;\im\varphi^{2\bar{2}} \\
\midrule
3 & \C & \varphi^{123},\; \varphi^{13\bar{1}},\; \varphi^{13\bar{2}},\; \varphi^{23\bar{1}},\; \varphi^{23\bar{2}},\; \varphi^{1\bar{1}\bar{3}},\; \varphi^{1\bar{2}\bar{3}},\; \varphi^{2\bar{1}\bar{3}},\; \varphi^{2\bar{2}\bar{3}},\; \varphi^{\bar{1}\bar{2}\bar{3}} \\[10pt]
  & \R & \varphi^{123}+\varphi^{\bar{1}\bar{2}\bar{3}},\; \im\left(\varphi^{123}-\varphi^{\bar{1}\bar{2}\bar{3}}\right),\;\varphi^{13\bar{1}}+\varphi^{1\bar{1}\bar{3}},\; \im\left(\varphi^{13\bar{1}}-\varphi^{1\bar{1}\bar{3}}\right),\, \varphi^{13\bar{2}}+\varphi^{2\bar{1}\bar{3}},\; \im\left(\varphi^{13\bar{2}}-\varphi^{2\bar{1}\bar{3}}\right), \\[5pt]
&&  \varphi^{23\bar{1}}+\varphi^{1\bar{2}\bar{3}},\; \im\left(\varphi^{23\bar{1}}-\varphi^{1\bar{2}\bar{3}}\right),\, \varphi^{23\bar{2}}+\varphi^{2\bar{2}\bar{3}},\; \im\left(\varphi^{23\bar{2}}-\varphi^{2\bar{2}\bar{3}}\right) \\
\midrule
4 & \C & \varphi^{123\bar{1}},\; \varphi^{123\bar{2}},\; \varphi^{13\bar{1}\bar{3}},\; \varphi^{13\bar{2}\bar{3}},\; \varphi^{23\bar{1}\bar{3}},\; \varphi^{23\bar{2}\bar{3}},\; \varphi^{1\bar{1}\bar{2}\bar{3}},\; \varphi^{2\bar{1}\bar{2}\bar{3}} \\[10pt]
  & \R & \varphi^{123\bar{1}}-\varphi^{1\bar{1}\bar{2}\bar{3}},\; \im\left(\varphi^{123\bar{1}}+\varphi^{1\bar{1}\bar{2}\bar{3}}\right),\;\varphi^{123\bar{2}}-\varphi^{2\bar{1}\bar{2}\bar{3}},\; \im\left(\varphi^{123\bar{2}}+\varphi^{2\bar{1}\bar{2}\bar{3}}\right),\, \varphi^{13\bar{1}\bar{3}},\;\varphi^{13\bar{2}\bar{3}}+\varphi^{23\bar{1}\bar{3}},\; \im\left(\varphi^{13\bar{2}\bar{3}}-\varphi^{23\bar{1}\bar{3}}\right),\;\varphi^{23\bar{2}\bar{3}} \\
\midrule
5 & \C & \varphi^{123\bar{1}\bar{3}},\; \varphi^{123\bar{2}\bar{3}},\; \varphi^{13\bar{1}\bar{2}\bar{3}},\; \varphi^{23\bar{1}\bar{2}\bar{3}} \\[10pt]
  & \R & \varphi^{123\bar{1}\bar{3}}+\varphi^{13\bar{1}\bar{2}\bar{3}},\; \im\left(\varphi^{123\bar{1}\bar{3}}-\varphi^{13\bar{1}\bar{2}\bar{3}}\right),\,\varphi^{123\bar{2}\bar{3}}+\varphi^{23\bar{1}\bar{2}\bar{3}},\; \im\left(\varphi^{123\bar{2}\bar{3}}-\varphi^{23\bar{1}\bar{2}\bar{3}}\right) \\
\bottomrule
\end{tabular}
\end{center}
\end{scriptsize}
\smallskip

Note that the above harmonic representatives of the classes in $H^\bullet_{dR}\left(\mathbb{I}_3;\R\right)$ are of pure type with respect to $J_\zero$ and to $J_\tempo$ with $\tempo$ in class {\itshape (i)}: hence, by Theorem \ref{thm:implicazioni} (or arguing as in \cite[Theorem 3.7]{fino-tomassini}), one gets that $\I_3$ and its small deformations in class {\itshape (i)} are \Cpf\ at every stage and \pf\ at every stage.

Concerning small deformations $J_{\mathbf{t}}$ in class {\itshape (ii)} and in class {\itshape (iii)}, using the asymptotic behaviour of the structure equations, we obtain that
$$ \left[\sigma_{12}\,\phit{12}\right] \;=\; \left[\sigma_{1\bar{1}}\,\phit{1\bar{1}}+
\sigma_{1\bar{2}}\,\phit{1\bar{2}}+\sigma_{2\bar{1}}\,\phit{2\bar{1}}+\sigma_{2\bar{2}}\,\phit{2\bar{2}}\right] \;
\neq\; 0 $$
in $H^2_{dR}\left(\I_3;\C\right)$. Therefore
$$ H^{(1,1)}_{J_\mathbf{t}}\left(\I_3;\C\right) \cap\left(H^{(2,0)}_{J_\mathbf{t}}\left(\I_3;\C\right) 
+ H^{(0,2)}_{J_\mathbf{t}}\left(\I_3;\C\right)\right) \;\neq\; \left\{0\right\} \;, $$
and in particular $J_\mathbf{t}$ is not complex-\Cp. It follows from Remark \ref{rem:cpf-complex-cpf} that $J_\mathbf{t}$ cannot be \Cp; from \cite[Proposition 2.30]{li-zhang}, or Theorem \ref{thm:implicazioni}, it follows that $J_\mathbf{t}$ cannot be \f.

To prove that small deformations in class {\itshape (ii)} and in class {\itshape (iii)} are non-\p\ and non-\Cf, fix $\mathbf{t}$ small enough and choose two positive complex numbers $A:=:A\left(\mathbf{t}\right)\in\C$ and $B:=:B\left(\mathbf{t}\right)\in\C$, depending just on $\mathbf{t}$, such that
$$ \left(A\,\sigma_{1\bar{2}}-B\,\sigma_{1\bar{1}},\,A\,\sigma_{2\bar{2}}-B\,\sigma_{2\bar{1}}\right) \;
\neq\; \left(0,\,0\right) \;;$$
computing $-\de\left(A\,\varphi_{\mathbf{t}}^{13\bar{3}}+B\,\varphi_{\mathbf{t}}^{23\bar{3}}\right)$,
note that
\begin{eqnarray*}
\lefteqn{\left[\left(A\,\sigma_{2\bar{1}}-B\,\sigma_{1\bar{1}}\right)\phit{12\bar{1}\bar{3}}+
\left(A\,\sigma_{2\bar{2}}-B\,\sigma_{1\bar{2}}\right)\phit{12\bar{2}\bar{3}}
-A\,\bar{\sigma}_{12}\phit{13\bar{1}\bar{2}}-B\,\bar{\sigma}_{12}\phit{23\bar{1}\bar{2}}\right]} \\[5pt]
&=& \left[\left(A\,\bar{\sigma}_{1\bar{2}}-B\,\bar{\sigma}_{1\bar{1}}\right)\phit{123\bar{1}}+
\left(A\,\bar{\sigma}_{2\bar{2}}-B\,\bar{\sigma}_{2\bar{1}}\right)\phit{123\bar{2}}\right] \;\neq\; 0 \;,
\end{eqnarray*}
in $H^4_{dR}\left(\I_3;\C\right)$. As before, it follows that $J_\mathbf{t}$ is not \Cp\ at the \kth{4} stage,
and consequently it is neither \p\ nor \Cf, by Theorem \ref{thm:implicazioni}.
\end{proof}

\subsubsection{Curves of \Cpf\ almost-complex structures}

We study here some explicit examples of curves of almost-complex structures on compact manifolds, along which the property of being \Cpf\ remains satisfied. The aim of this section is to better understand the behaviour of \Cpf ness along curves of almost-complex structures.

\medskip

Firstly, we recall some general results concerning curves of almost-complex structures on compact manifolds, referring, e.g., to \cite{audin-lafontaine}.

Let $J$ be an almost-complex structure on a compact $2n$-dimensional manifold $X$. Every curve $\left\{J_t\right\}_{t\in \left(-\varepsilon,\varepsilon\right)\subset \R}$ of almost-complex structures on $X$ such that $J_0=J$ can be written, for $\varepsilon>0$ small enough, as
$$ J_t \;=\; \left(\id\,-\,L_t\right)\,J\,\left(\id\,-\,L_t\right)^{-1} \;\in\; \End\left(TX\right) \;, $$
where $L_t\in\End\left(TX\right)$, see, e.g., \cite[Proposition 1.1.6]{audin-lafontaine}; the endomorphism $L_t$ is uniquely determined further requiring that $L_t\in T^{1,0}_JX\otimes \duale{\left(T^{0,1}_JX\right)}$, namely,
$$ L_t\,J\,+\,J\,L_t \;=\; 0 \;;$$
furthermore, set $L_t\,=:\, t\,L+\mathrm{o}(t)$: if $J$ is compatible with a symplectic form $\omega$, then the curves consisting of $\omega$-compatible almost-complex structures $J_t$ are exactly those ones for which $\trasposta{L}=L$.

In \cite[Proposition 3.3]{debartolomeis-meylan}, P. de Bartolomeis and F. Meylan computed $\left.\frac{\de}{\de t}\right\lfloor_{t=0} \Nij_J$, getting a characterization in terms of $L$ of the curves of complex structures starting at a given integrable almost-complex structure $J$. 

\medskip

A. Fino and A. Tomassini, in \cite[\S6, \S7]{fino-tomassini}, studied several examples of families of almost-complex structures constructed in such a way. We provide here some further examples, starting with a curve of almost-complex structures on the $4$-dimensional torus, \cite[pages 420--422]{angella-tomassini-1}.

\begin{ex}
\textit{A curve of almost-complex structures through the standard K\"{a}hler structure on the $4$-dimensional torus.}\\
Let $\left(J_0,\, \omega_0\right)$ be the standard K\"ahler structure on the $4$-dimensional torus $\T^4$ with coordinates $\{x^j\}_{j\in\{1,\ldots,4\}}$, that is,
$$
 J_{0} \;:=\;
\left(
\begin{array}{cc|cc}
 & & -1 & \\
 &  &  & -1 \\
\hline
 1 &  &  & \\
 & 1 &  & \\
\end{array}
\right)
\;\in\; \End\left(T\T^4\right)
\qquad \text{ and } \qquad
 \omega_0 \;:=\; \de x^1 \wedge \de x^3 + \de x^2 \wedge \de x^4 \;\in\; \wedge^2\T^4 \;.
$$

Set
$$ L\;:=\;\left(
\begin{array}{cc|cc}
 \phantom{+}\ell &  &  &  \\
 & \phantom{+}0 &  &  \\
\hline
  & & -\ell &\\
 &  &  & \phantom{+}0
\end{array}
\right)
\;\in\; \End\left(T\T^4\right) \;,$$
where $\ell\in\mathcal{C}^\infty(\T^4;\,\R)$, that is, $\ell\in\mathcal{C}^\infty(\R^4;\,\R)$ is a $\Z^4$-periodic function. For $t\in \left(-\varepsilon,\, \varepsilon\right)$ with $\varepsilon>0$ small enough, define
$$
 J_{t,\,\ell} \;:=\; \left(\id\,-\,t\,L\right)\,J_0\,\left(\id\,-\,t\,L\right)^{-1}
 \;=\; \left(
\begin{array}{cc|cc}
 & & -\frac{1\,-\,t\,\ell}{1\,+\,t\,\ell} & \\
 &  &  & -1 \\
\hline
 \phantom{+}\frac{1\,+\,t\,\ell}{1\,-\,t\,\ell} &  &  & \\
 & \phantom{+}1 &  & \\
\end{array}
\right)
\;\in\; \End\left(T\T^4\right) \;,
$$
obtaining a curve of $\omega_0$-compatible almost-complex structures on $\T^4$, see also Proposition \ref{prop:complex-Cpf-4}. To simplify the notation, set
$$ \alpha\;:=:\;\alpha(t,\ell)\;:=\; \frac{1\,-\,t\,\ell}{1\,+\,t\,\ell} \;. $$
A co-frame for the holomorphic cotangent bundle of $\T^4$ with respect to $J_{t,\,\ell}$ is given by
$$
\left\{
\begin{array}{l}
 \varphi^1_{t,\ell} \;:=\; \de x^1\,+\,\im\,\alpha\,\de x^3 \\[5pt]
 \varphi^2_{t,\ell} \;:=\; \de x^2\,+\,\im\,\de x^4
\end{array}
\right. \;,
$$
with respect to which we compute the structure equations
$$
\left\{
\begin{array}{l}
 \de\varphi^1_{t,\ell} \;=\; \im\,\de\alpha\,\wedge\,\de x^3 \\[5pt]
 \de\varphi^2_{t,\ell} \;=\; 0
\end{array}
\right. \;.
$$
Note that, taking $\ell\,=\,\ell\left(x^1,x^3\right)$, the corresponding almost-complex structure $J_{t,\,\ell}$
is integrable, in fact, $\left(J_{t,\,\ell},\,\omega_0\right)$ is a K\"{a}hler structure on $\T^4$.
Recall that, $\T^4$ being $4$-dimensional, $J_{t,\,\ell}$ is \Cpf\ by \cite[Theorem 2.3]{draghici-li-zhang}. For the sake of
simplicity, assume $\ell\,=\,\ell\left(x^2\right)$ depending just on $x^2$ and non-constant.
Set
\begin{eqnarray*}
 v_1 &:=& \de x^1\wedge \de x^2 - \alpha\, \de x^3 \wedge \de x^4 \;, \\[5pt]
 v_2 &:=& \de x^1\wedge \de x^4 - \alpha\, \de x^2 \wedge \de x^3 \;, \\[10pt]
 w_1 &:=& \alpha\,\de x^1\wedge \de x^3 \;, \\[5pt]
 w_2 &:=& \de x^2\wedge \de x^4 \;, \\[5pt]
 w_3 &:=& \de x^1\wedge \de x^2 + \alpha\, \de x^3 \wedge \de x^4 \;, \\[5pt]
 w_4 &:=& \de x^1\wedge \de x^4 + \alpha\, \de x^2 \wedge \de x^3 \;.
\end{eqnarray*}
Using this notation, an arbitrary $J_{t,\,\ell}$-anti-invariant real $2$-form $\psi \,:=:\, A\,v_1+ B\, v_2$, with $A,B\in\mathcal{C}^\infty\left(\T^4;\R\right)$, is $\de$-closed if and only if
\begin{equation}\label{eq:system-T^2}
\left\{
\begin{array}{rcl}
 \frac{\del A}{\del x^3}-\frac{\del B}{\del x^1}\,\alpha &=& 0 \\[5pt]
 \frac{\del A}{\del x^4}-\frac{\del B}{\del x^2} &=& 0 \\[5pt]
 -\frac{\del A}{\del x^1}\,\alpha-\frac{\del B}{\del x^3} &=& 0 \\[5pt]
 -\frac{\del B}{\del x^4}\,\alpha-\frac{\del A}{\del x^2}\,\alpha-A\,\frac{\del \alpha}{\del x^2} &=& 0
\end{array}
\right. \;.
\end{equation}
By solving \eqref{eq:system-T^2}, we obtain the solutions
$$ \psi \;=\; \frac{A}{\alpha}\,v_1 + B\, v_2 \qquad \text{ where } \qquad A,\, B\in\R \;.$$
Therefore, for $t\in \left(-\varepsilon,\, \varepsilon\right)$ with $\varepsilon>0$ small enough, we have
$$ \dim_\R H^{(2,0),(0,2)}_{J_{t,\,\ell}}\left(\T^4;\R\right) \;\leq\; 2 \;=\; \dim_\R H^{(2,0),(0,2)}_{J_0}
\left(\T^4;\R\right)\;, $$
and hence
$$
\dim_\R H^{(1,1)}_{J_{t,\,\ell}}\left(\T^4;\R\right) \;\geq\; 4 \;=\; \dim_\R H^{(1,1)}_{J_0}\left(\T^4;\R\right)\;,
$$
accordingly to the upper-semi-continuity, respectively lower-semi-continuity, property proven in \cite[Theorem 2.6]{draghici-li-zhang-2} for $4$-dimensional almost-complex manifolds.
\end{ex}

Now, we turn our attention to the case of dimension greater than $4$, \cite[pages 422--423]{angella-tomassini-1}.

\begin{ex}
\textit{A curve of almost-complex structures through the standard K\"{a}hler structure on the $6$-dimensional torus.}\\
Let $\left(J_0,\, \omega_0\right)$ be the standard K\"ahler structure on the $6$-dimensional torus $\T^6$ with coordinates $\{x^j\}_{j\in\{1,\ldots,6\}}$, that is,
$$
  J_{0} \;:=\;
 \left(
\begin{array}{ccc|ccc}
 & & & -1 & & \\
 &  &  &  & -1 & \\
 &  &  &  &  & -1 \\
\hline
 1 &  &  &  &  & \\
 & 1 &  &  &  & \\
 &  & 1 &  &  &
\end{array}
\right)
\;\in\; \End\left(T\T^6\right)
\qquad \text{ and } \qquad
\omega_0 \;:=\; \de x^1 \wedge \de x^4 + \de x^2 \wedge \de x^5 + \de x^3 \wedge \de x^6 \;.
$$

Set
$$ L\;=\;\left(
\begin{array}{ccc|ccc}
 \phantom{+}\ell &  &  &  &  &  \\
 & \phantom{+}0 &  &  &  &  \\
 &  & \phantom{+}0 &  &  &  \\
\hline
  & & & -\ell & & \\
 &  &  &  & \phantom{+}0 & \\
 &  &  &  &  & \phantom{+}0
\end{array}
\right)
\;\in\; \End\left(T\T^6\right) \;,$$
where $\ell\in\mathcal{C}^\infty(\T^6;\,\R)$, that is, $\ell\in\mathcal{C}^\infty(\R^6;\,\R)$ is a $\Z^6$-periodic function. For $t\in \left(-\varepsilon,\, \varepsilon\right)$ with $\varepsilon>0$ small enough, define
$$
  J_{t,\,\ell} \;:=\; \left(\id\,-\,t\,L\right)\,J_0\,\left(\id\,-\,t\,L\right)^{-1}
 \;=\; \left(
\begin{array}{ccc|ccc}
 & & & -\frac{1\,-\,t\,\ell}{1\,+\,t\,\ell} & & \\
 &  &  &  & -1 & \\
 &  &  &  &  & -1 \\
\hline
 \phantom{+}\frac{1\,+\,t\,\ell}{1\,-\,t\,\ell} &  &  &  &  & \\
 & \phantom{+}1 &  &  &  & \\
 &  & \phantom{+}1 &  &  &
\end{array}
\right)
\;\in\; \End\left(T\T^6\right) \;,
$$
obtaining a curve of $\omega_0$-compatible almost-complex structures on $\T^6$, see also Example \ref{ex:toro-6-large-invariant}. Setting
$$
\alpha\;:=:\;\alpha(t,\ell) \;:=\; \frac{1\,-\,t\,\ell}{1\,+\,t\,\ell} \;,
$$
a co-frame for the holomorphic cotangent bundle of $\T^6$ with respect to $J_{t,\,\ell}$ is given by
$$
\left\{
\begin{array}{l}
 \varphi^1_{t,\ell} \;:=\; \de x^1\,+\,\im\,\alpha\,\de x^4 \\[5pt]
 \varphi^2_{t,\ell} \;:=\; \de x^2\,+\,\im\,\de x^5 \\[5pt]
 \varphi^3_{t,\ell} \;:=\; \de x^3\,+\,\im\,\de x^6
\end{array}
\right. \;,
$$
with respect to which the structure equations are
$$
\left\{
\begin{array}{l}
 \de\varphi^1_{t,\ell} \;=\; \im\,\de\alpha\,\wedge\,\de x^4 \\[5pt]
 \de\varphi^2_{t,\ell} \;=\; 0 \\[5pt]
 \de\varphi^3_{t,\ell} \;=\; 0
\end{array}
\right. \;.
$$
Note that if $\ell\,=\,\ell\left(x^1,x^4\right)$, then we get a curve of integrable almost-complex structures, in fact, of K\"{a}hler structures, on $\T^6$: in particular, in such a case, $J_{t,\,\ell}$ is \Cpf. Therefore, as an example, assume that
$\ell\,=\,\ell\left(x^3\right)$ depends just on $x^3$ and is non-constant.

An arbitrary $J_{t,\,\ell}$-anti-invariant real $2$-form
\begin{eqnarray*}
\psi &:=:& A\,\left(\de x^1\wedge \de x^2-\alpha\,\de x^{4}\wedge \de x^5\right) + B\, \left(\de x^1\wedge\de x^5-\alpha\,\de x^2\wedge\de x^4\right) + C\, \left(\de x^1\wedge \de x^3-\alpha\, \de x^4\wedge\de x^6\right) \\[5pt]
 && + D\, \left(\de x^1\wedge\de x^6-\alpha\,\de x^3\wedge\de x^4\right) + E\, \left(\de x^2\wedge\de x^3-\de x^5\wedge\de x^6\right) + F\, \left(\de x^2\wedge\de x^6-\de x^3\wedge \de x^5\right) \;,
\end{eqnarray*}
with $A,B,C,D,E,F\in\mathcal{C}^\infty\left(\T^6;\R\right)$, is $\de$-closed if and only if
\begin{equation}\label{eq:system-T^4}
\left\{
\begin{array}{rcl}
 \frac{\del A}{\del x^3} - \frac{\del C}{\del x^2} + \frac{\del E}{\del x^1} &=& 0 \\[5pt]
 \frac{\del A}{\del x^4} - \frac{\del B}{\del x^1}\, \alpha &=& 0 \\[5pt]
 \frac{\del A}{\del x^5} - \frac{\del B}{\del x^2} &=& 0 \\[5pt]
 \frac{\del A}{\del x^6} - \frac{\del D}{\del x^2} + \frac{\del F}{\del x^1} &=& 0 \\[5pt]
 \frac{\del C}{\del x^4} - \frac{\del D}{\del x^1}\, \alpha &=& 0 \\[5pt]
 -\frac{\del B}{\del x^3} + \frac{\del C}{\del x^5} - \frac{\del F}{\del x^1} &=& 0 \\[5pt]
 \frac{\del C}{\del x^6} - \frac{\del D}{\del x^3} &=& 0 \\[5pt]
 -\frac{\del A}{\del x^1}\, \alpha - \frac{\del B}{\del x^4} &=& 0 \\[5pt]
 -\frac{\del C}{\del x^1}\, \alpha - \frac{\del D}{\del x^4} &=& 0 \\[5pt]
 \frac{\del B}{\del x^6} - \frac{\del D}{\del x^5} - \frac{\del E}{\del x^1} &=& 0 \\[5pt]
 \frac{\del \left(B\,\alpha\right)}{\del x^3} - \frac{\del D}{\del x^2}\, \alpha + \frac{\del E}{\del x^4} &=& 0 \\[5pt]
 \frac{\del E}{\del x^5} - \frac{\del F}{\del x^2} &=& 0 \\[5pt]
 \frac{\del E}{\del x^6} - \frac{\del F}{\del x^3} &=& 0 \\[5pt]
 -\frac{\del A}{\del x^2}\, \alpha - \frac{\del B}{\del x^5}\, \alpha &=& 0 \\[5pt]
 -\frac{\del B}{\del x^6}\, \alpha - \frac{\del C}{\del x^2}\, \alpha - \frac{\del F}{\del x^4} &=& 0 \\[5pt]
 -\frac{\del E}{\del x^2} - \frac{\del F}{\del x^5} &=& 0 \\[5pt]
 -\frac{\del \left(A\, \alpha\right)}{\del x^3} - \frac{\del D}{\del x^5}\, \alpha + \frac{\del F}{\del x^4} &=& 0 \\[5pt]
 -\frac{\del \left(C\, \alpha\right)}{\del x^3} - \frac{\del D}{\del x^6}\, \alpha &=& 0 \\[5pt]
 -\frac{\del E}{\del x^3} - \frac{\del F}{\del x^6} &=& 0 \\[5pt]
 -\frac{\del A}{\del x^6}\, \alpha + \frac{\del C}{\del x^5}\, \alpha - \frac{\del E}{\del x^4} &=& 0
\end{array}
\right. \;.
\end{equation}
For $t \neq 0$ small enough, by solving \eqref{eq:system-T^4}, we obtain that the $J_{t,\,\ell}$-anti-invariant real $\de$-closed $2$-forms are
$$
\psi \;=\; \frac{C}{\alpha}\left(\de x^{13}-\alpha\,\de x^{46}\right) + D\,\left(\de x^{16} - \alpha\, \de x^{34}\right) + E\, \left(\de x^{23} - \de x^{56}\right) + F\,\left(\de x^{26} - \de x^{35}\right) \;,
$$
where $C,\, D,\, E,\, F\in\R$.

For $t\neq0$ small enough, we have
$$ \dim_\R H^{(2,0),(0,2)}_{J_{t,\,\ell}}\left(\T^6;\R\right) \;\leq\; 4 \;<\; 6 \;=\; \dim_\R H^{(2,0),(0,2)}_{J_0}\left(\T^6;\R\right) \;, $$
and hence the function $t\mapsto \dim_\R H^{(2,0),(0,2)}_{J_{t,\,\ell}}\left(\T^6;\R\right)$ is upper-semi-continuous at $0$. On the other hand, the explicit computations for $H^{(1,1)}_{J_{t,\,\ell}}\left(\T^6;\R\right)$ are not so straightforward. In particular, it is not clear if $J_{t,\,\ell}$ remains still \Cf; note that $J_{t,\,\ell}$ is \Cp\ by \cite[Proposition 2.7]{draghici-li-zhang} or \cite[Proposition 3.2]{fino-tomassini}.
\end{ex}

\medskip

We recall here the construction of curves of almost-complex structures through an almost-complex structure $J$ by means of a $J$-anti-invariant real $2$-form, as introduced by J. Lee in \cite[\S1]{lee}, in the context of holomorphic curves on symplectic manifolds and Gromov and Witten invariants.

Let $J$ be an almost-complex structure on a compact manifold $X$; let $g$ be a $J$-Hermitian metric on $X$ and fix $\gamma\in \left(\wedge^{2,0}X\oplus\wedge^{0,2}X\right)\cap\wedge^2X$. Define $V_\gamma\in\End\left(TX\right)$ such that
\begin{equation}\label{eq:representation}
\gamma\left(\sspace,\,\ssspace\right) \;=\; g\left(V_\gamma\, \sspace,\, \ssspace\right) \;;
\end{equation}
a direct computation shows that $V_\gamma\,J\,+\,J\,V_\gamma\,=\,0$. Therefore, setting
$$ L_\gamma \;:=\; \frac{1}{2}\,V_\gamma\,J\; \in \; \End\left(TX\right) \;,$$
one gets that $L_\gamma\,J\,+\,J\,L_\gamma\,=\,0$. For $t\in \left(-\varepsilon,\, \varepsilon\right)$ with $\varepsilon>0$ small enough, define
$$
J_{t,\,\gamma} \;:=\; \left(\id\,-\,t\,L_\gamma\right)\,J\,\left(\id\,-\,t\,L_\gamma\right)^{-1} \;\in\; \End\left(TX\right) \;,
$$
obtaining a curve $\left\{J_{t,\,\gamma}\right\}_{t\in\left(-\varepsilon,\varepsilon\right)}$ of almost-complex structures associated with $\gamma$.

\medskip

We give an example of a \Cpf\ structure on a non-K\"{a}hler manifold such that the stability property of the \Cpf ness holds along a curve obtained using the construction by J. Lee, \cite[pages 423--425]{angella-tomassini-1}.

\begin{ex}
\label{ex:n6c}
\textit{A curve of \Cpf\ almost-complex structures on the completely-solvable solvmanifold $N^6(c)$.}\\
We recall that the manifold $N^6(c)$ is a compact $6$-dimensional completely-solvable solvmanifold defined, for suitable $c\in\R$, as the product
$$ N^6(c) \;:=\; \left(\Gamma(c) \left\backslash \mathrm{Sol}(3)\right.\right) \,\times\, \left(\Gamma(c) \left\backslash \mathrm{Sol}(3)\right.\right) \;,$$
where $\mathrm{Sol}(3)$ is a completely-solvable Lie group and $\Gamma(c)$ is a co-compact discrete subgroup of $\mathrm{Sol}(3)$, \cite[\S3]{auslander-green-hahn}, see Example \ref{ex:instability-alm-kahler}. It has been studied in \cite[Example 1]{benson-gordon-solvmanifolds} as an example of a cohomologically K\"{a}hler manifold, and in \cite[Example 3.4]{fernandez-munoz-santisteban} by M. Fern\'{a}ndez, V. Mu\~{n}oz, and J. A. Santisteban, as an example of a formal manifold admitting a symplectic structure satisfying the Hard Lefschetz Condition and with no K\"{a}hler structure, \cite[Theorem 3.5]{fernandez-munoz-santisteban}. A. Fino and A. Tomassini provided in \cite[\S6.3]{fino-tomassini} a family of \Cpf\ structures on $N^6(c)$. We construct here a curve of \Cpf\ almost-complex structures on $N^6(c)$ using the construction by J. Lee, \cite[\S1]{lee}.

Let $\left\{e^i\right\}_{i\in\{1,\ldots,6\}}$ be a co-frame for $N^6(c)$ such that the structure equations are
$$
\left\{
\begin{array}{l}
 \de e^1 \;=\; \phantom{+}c\, e^1\wedge e^3 \\[5pt]
 \de e^2 \;=\; -c\, e^2\wedge e^3\\[5pt]
 \de e^3 \;=\; \phantom{+}0\\[5pt]
 \de e^4 \;=\; \phantom{+}c\, e^4\wedge e^6 \\[5pt]
 \de e^5 \;=\; -c\, e^5\wedge e^6 \\[5pt]
 \de e^6 \;=\; \phantom{+}0
\end{array}
\right. \;.
$$
Take the almost-complex structure
$$
J\;=\;
\left(
\begin{array}{cc|cc|cc}
 & -1 & & & & \\
 \phantom{+}1 & & & & & \\
\hline
 & & & -1& & \\
 & & \phantom{+}1& & & \\
\hline
 & & & & &-1 \\
 & & & &\phantom{+}1 &
\end{array}
\right)
\;\in\; \End\left(TN^6(c)\right) \;.
$$

By A. Hattori's theorem \cite[Corollary 4.2]{hattori}, one computes
$$ H^2_{dR}\left(N^6(c);\R\right) \;=\; \R\left\langle e^{1}\wedge e^{2},\; e^{3}\wedge e^6-e^{4}\wedge e^{5},\; e^{3}\wedge e^{6}+e^{4}\wedge e^{5}\right\rangle \;,$$
proving that $\left(N^6(c),\,J\right)$ is \Cpf\ and \pf: indeed, the above harmonic representatives with respect to the $\left(\mathrm{Sol}(3)\times\mathrm{Sol}(3)\right)$-left-invariant metric $g:=\sum_{j=1}^{6} e^j\odot e^j$ are of pure type with respect to $J$, and hence \cite[Theorem 3.7]{fino-tomassini} assures the \Cpf ness and the \pf ness. Note that
$$
H^{(2,0),(0,2)}_J\left(N^6(c);\R\right) \;=\; \R\left\langle e^{3}\wedge e^{6}+e^{4}\wedge e^{5}\right\rangle \;;
$$
apply J. Lee's construction \cite[\S1]{lee} to the real $J$-anti-invariant $2$-form
$$ \gamma\;:=\;e^{3}\wedge e^{6}+e^{4}\wedge e^{5} \;:$$
the linear map $V\in\End(TX)$ representing $\gamma$ as in \eqref{eq:representation} is
$$
V\;=\;
\left(
\begin{array}{cc|cc|cc}
 \phantom{+}0& & & & & \\
 &\phantom{+}0 & & & & \\
\hline
 & & & & & -1 \\
 & & & & -1 & \\
\hline
 & & & \phantom{+}1 & & \\
 & & \phantom{+}1 & & &
\end{array}
\right)
\;\in\; \End\left(TN^6(c)\right) \;,
$$
and then it is straightforward to compute
$$
L\;=\;
\left(
\begin{array}{cc|cc|cc}
 \phantom{+}0& & & & & \\
 &\phantom{+}0 & & & & \\
\hline
 & & & & -\frac{1}{2} & \\
 & & & & & \phantom{+}\frac{1}{2} \\
\hline
 & & \phantom{+}\frac{1}{2} & & & \\
 & & & -\frac{1}{2} & &
\end{array}
\right)
\;\in\; \End\left(TN^6(c)\right) \;,
$$
and
$$
J_t \;:=:\; J_{t,\,\gamma} \;=\;
\left(
\begin{array}{cc|cc|cc}
 & -1 & & & & \\
 \phantom{+}1 & & & & & \\
\hline
 & & & -\frac{4-t^2}{4+t^2} & & -\frac{4t}{4+t^2} \\
 & & \phantom{+}\frac{4-t^2}{4+t^2} & & -\frac{4t}{4+t^2} & \\
\hline
 & & & \frac{4t}{4+t^2} & &-\frac{4-t^2}{4+t^2} \\
 & & \frac{4t}{4+t^2} & &\frac{4-t^2}{4+t^2} &
\end{array}
\right)
\;\in\; \End\left(TN^6(c)\right) \;.
$$
To shorten the notation, set
$$ \alpha \;:=:\; \alpha(t) \;:=\; \frac{4-t^2}{4+t^2}\;,\qquad \beta\;:=:\;\beta(t) \;:=\;
\frac{4t}{4+t^2}\;.$$
A co-frame for the $J_t$-holomorphic cotangent bundle is given by
$$
\left\{
\begin{array}{l}
 \varphi^1_t \;:=\; e^1+\im e^2 \\[5pt]
 \varphi^2_t \;:=\; e^3+\im \left(\alpha\, e^4+\beta\, e^6\right) \\[5pt]
 \varphi^3_t \;:=\; e^5+\im \left(-\beta\, e^4+\alpha\, e^6\right)
\end{array}
\right. \;.
$$
Since the real $\de$-closed $2$-forms
$$ \frac{1}{2\im}\,\varphi_t^{1\bar{1}}\;,\qquad \frac{1}{2\im}\varphi_t^{3\bar{3}}-\frac{\alpha}{c}\,\de e^{5}\;,
\qquad \frac{1}{2\im}\left(\beta\,\varphi_t^{2\bar{2}}+\alpha\left(\varphi_t^{2\bar{3}}-\varphi_t^{\bar{2}3}\right)
\right)+\frac{1}{2\im}\varphi_t^{3\bar{3}} $$
generate three different cohomology classes, we get that, for $t\neq0$ small enough,
$$
H^{2}_{dR}\left(N^6(c);\R\right) \;=\; H^{(1,1)}_{J_t}\left(N^6(c);\R\right) \;,
$$
and so, in particular, $J$ is \Cf\ and \p. A straightforward computation yields
\begin{eqnarray*}
H^4_{dR}\left(N^6(c);\R\right) &=& \R\left\langle *_g\left(\frac{1}{2\im}\,\varphi_t^{1\bar{1}}\right), \; *_g\left(\varphi_t^{3\bar{3}}-\frac{\alpha}{c}\,
\de e^{5}\right)+\frac{\alpha}{c}\,\de \left(e^{125}\right), \right. \\[5pt]
 && \left. \frac{\alpha}{4}\left(\varphi_t^{12\bar{1}\bar{3}}+
\varphi_t^{\bar{1}\bar{2}13}\right)+\frac{\beta}{4}\,\varphi_t^{12\bar{1}\bar{2}}+\frac{\alpha\,\beta}{c}\,
\de\left(e^{125}\right)\right\rangle \\[5pt]
&=& H^{(2,2)}_{J_t}\left(N^6(c);\R\right) \;,
\end{eqnarray*}
therefore $N^6(c)$ is also \Cf\ at the \kth{4} stage and hence \f\ and \Cp.
\end{ex}

We resume the content of the last example in the following theorem, \cite[Theorem 4.1]{angella-tomassini-1}.

\begin{thm}\label{thm:lee-curves}
 There exists a compact manifold $N^6(c)$ endowed with an almost-complex structure $J$ and a $J$-Hermitian metric $g$ such that:
\begin{enumerate}
\item[(i)] $J$ is \Cpf;
\item[(ii)] each $J$-anti-invariant $g$-harmonic form gives rise to a curve
$\left\{J_t\right\}_{t\in\left(-\varepsilon,\varepsilon\right)}$ of \Cpf\ almost-complex structures on $N^6(c)$, where $\varepsilon>0$ is small enough, using J. Lee's construction;
\item[(iii)] furthermore, the function
$$
\left(-\varepsilon,\varepsilon\right) \;\ni\; t\mapsto \dim_\R H^{(2,0),(0,2)}_{J_t}\left(N^6(c);\R\right) \;\in\; \N
$$
is upper-semi-continuous at $0$.
\end{enumerate}
\end{thm}

\subsection{The semi-continuity problem}\label{subsec:semicontinuity}
Given a compact $4$-dimensional manifold $X$ and a family $\left\{J_t\right\}_{t}$ of almost-complex structures on $X$, T. Dr\v{a}ghici, T.-J. Li, and W. Zhang studied in \cite{draghici-li-zhang-2} the semi-continuity properties of the functions $t\mapsto \dim_\R H^+_{J_t}(X)$ and $t\mapsto \dim_\R H^-_{J_t}(X)$. They proved the following result.

\begin{thm}[{\cite[Theorem 2.6]{draghici-li-zhang-2}}]
 Let $X$ be a compact $4$-dimensional manifold and let $\{J_t\}_{t\in I\subseteq \R}$ be a family of (\Cpf) almost-complex structures on $X$, for $I\subseteq\R$ an interval. Then the function
$$ I \;\ni\; t\mapsto \dim_\R H^{-}_{J_t}(X) \;\in\; \N $$
is upper-semi-continuous, and therefore the function
$$ I \;\ni\; t\mapsto \dim_\R H^{+}_{J_t}(X) \;\in\; \N $$
is lower-semi-continuous.
\end{thm}

The previous result is closely related to the geometry of $4$-dimensional manifolds; more precisely, it follows from M. Lejmi's result in \cite[Lemma 4.1]{lejmi-mrl} that a certain operator is a self-adjoint strongly elliptic linear operator with kernel the harmonic $J$-anti-invariant $2$-forms. In this section, we are concerned with establishing if a similar semi-continuity result could occur in dimension higher than $4$, possibly assuming further hypotheses.

\subsubsection{Counterexamples to semi-continuity}

First of all, we provide two examples showing that, in general, no semi-continuity property holds in dimension higher than $4$.

\medskip

The following result provides a counterexample to the upper-semi-continuity of $t\mapsto\dim_\R H^-_{J_t}$ in dimension greater than $4$, \cite[Proposition 4.1]{angella-tomassini-2}.

\begin{prop}
\label{prop:no-scs-h-}
 The compact $10$-dimensional manifold $\eta\beta_5$ is endowed with a \Cpf\ complex structure $J$ and a curve
$\left\{J_t\right\}_{t\in\Delta\left(0,\varepsilon\right)\subset\C}$ of complex structures (which are non-\Cp\ for $t\neq 0$), with $J_0=J$, and $\varepsilon>0$, such that the function
$$ \Delta\left(0,\varepsilon\right) \;\ni\; t \mapsto \dim_\R H^-_{J_t}\left(\eta\beta_5\right) \;\in\; \N $$
is not upper-semi-continuous.
\end{prop}

\begin{proof}
The proof follows from the following example, \cite[Example 4.2]{angella-tomassini-2}.

Consider the nilmanifold $\eta\beta_5$ endowed with its natural complex structure $J$, as described in Example \ref{ex:etabeta5}.
We recall that, chosen a suitable co-frame $\left\{\varphi^j\right\}_{j\in\{1,\ldots,5\}}$ of the holomorphic cotangent bundle, the complex structure equations are
$$
\left\{
\begin{array}{l}
 \de\varphi^1 \;=\; \de\varphi^2 \;=\; \de\varphi^3 \;=\; \de\varphi^4 \;=\; 0 \\[5pt]
 \de\varphi^5 \;=\; -\varphi^{12}-\varphi^{34}
\end{array}
\right. \;.
$$

By K. Nomizu's theorem \cite[Theorem 1]{nomizu}, it is straightforward to compute
\begin{eqnarray*}
 H^2_{dR}\left(\eta\beta_5;\C\right) &=& \C\left\langle \varphi^{13},\; \varphi^{14},\; \varphi^{23},\; \varphi^{24},\; \varphi^{\bar1\bar3},\; \varphi^{\bar1\bar4},\; \varphi^{\bar2\bar3},\; \varphi^{\bar2\bar4},\; \varphi^{12}-\varphi^{34},\; \varphi^{\bar1\bar2}-\varphi^{\bar3\bar4} \right\rangle\\[5pt]
 && \oplus\; \C\left\langle \varphi^{1\bar1},\; \varphi^{1\bar2},\; \varphi^{1\bar3},\; \varphi^{1\bar4},\; \varphi^{2\bar1},\; \varphi^{2\bar2},\; \varphi^{2\bar3},\; \varphi^{2\bar4}, \; \varphi^{3\bar1},\; \varphi^{3\bar2},\; \varphi^{3\bar3},\; \varphi^{3\bar4},\; \varphi^{4\bar1},\; \varphi^{4\bar2},\; \varphi^{4\bar3},\; \varphi^{4\bar4} \right\rangle
\end{eqnarray*}
(where, as usually, we have listed the harmonic representatives with respect to the left-invariant Hermitian metric $\sum_{j=1}^{5}\varphi^j\odot\bar\varphi^j$ instead of their classes, and we have shortened, e.g.,  $\varphi^{A\bar B}:=\varphi^A\wedge\bar\varphi^B$). Hence the complex structure $J$ is \Cpf\ by \cite[Theorem 3.7]{fino-tomassini}, and
$$ \dim_\R H^-_J\left(\eta\beta_5\right) \;=\; 10\;,\qquad\qquad \dim_\R H^+_J\left(\eta\beta_5\right)\;=\;16\;.$$

Now, for $\varepsilon>0$ small enough, consider the curve $\left\{J_t\right\}_{t\in\Delta\left(0,\varepsilon\right)}$ of complex structures such that a co-frame for the $J_t$-holomorphic cotangent bundle is given by $\left\{\varphi^j_t\right\}_{j\in\{1,\ldots,5\}}$, where, for any $t\in\Delta\left(0,\varepsilon\right)$,
$$
\left\{
\begin{array}{rcl}
 \varphi^1_t &:=& \varphi^1+t\,\bar\varphi^1 \\[5pt]
 \varphi^2_t &:=& \varphi^2 \\[5pt]
 \varphi^3_t &:=& \varphi^3 \\[5pt]
 \varphi^4_t &:=& \varphi^4 \\[5pt]
 \varphi^5_t &:=& \varphi^5
\end{array}
\right. \;,
$$
see Example \ref{ex:etabeta5}. The structure equations with respect to $\left\{\varphi^j_t\right\}_{j\in\{1,\ldots,5\}}$ are
$$
\left\{
\begin{array}{l}
 \de\varphi^1_t \;=\; \de\varphi^2_t \;=\; \de\varphi^3_t \;=\; \de\varphi^4_t \;=\; 0 \\[5pt]
 \de\varphi^5_t \;=\; -\frac{1}{1-\left|t\right|^2}\,\varphi^{12}_t-\varphi^{34}_t-\frac{t}{1-\left|t\right|^2}\,
\varphi^{2\bar1}_t
\end{array}
\right. \;.
$$
When $\varepsilon>0$ is small enough, for $t\in\Delta\left(0,\varepsilon\right)\setminus\{0\}$, the complex structure $J_t$ is not \Cp: indeed,
$$ H^{(1,1)}_{J_t}\left(\eta\beta_5;\C\right) \;\ni\;
\left[\frac{t}{1-\left|t\right|^2}\,\varphi^{2\bar1}_t+\de\varphi^5_t\right] \;=\; \left[-\frac{1}{1-\left|t\right|^2}\,
\varphi^{12}_t-\varphi^{34}_t\right]\;\in\; H^{(2,0)}_{J_t}\left(\eta\beta_5;\C\right) \;, $$
where $\left[\frac{t}{1-\left|t\right|^2}\,\varphi^{2\bar1}_t\right]\in H^2_{dR}\left(\eta\beta_5;\C\right)$ is a non-zero cohomology class by K. Nomizu's theorem \cite[Theorem 1]{nomizu}.
Moreover, note that
$$ H^{(2,0),(0,2)}_{J_t}\left(\eta\beta_5;\C\right) \;\supseteq \; \C\left\langle \varphi^{13}_t,\; \varphi^{14}_t,\; \varphi^{23}_t,\; \varphi^{24}_t,\;  \varphi^{\bar1\bar3}_t,\; \varphi^{\bar1\bar4}_t,\; \varphi^{\bar2\bar3}_t,\; \varphi^{\bar2\bar4}_t,\; \varphi^{12}_t,\; \varphi^{34}_t,\; \varphi^{\bar1\bar2}_t,\; \varphi^{\bar3\bar4}_t \right\rangle \;, $$
hence, for every $t\in\Delta\left(0,\varepsilon\right)\setminus\{0\}$,
$$ \dim_\R H^-_{J_0}\left(\eta\beta_5\right) \;=\; 10 \;<\; 12 \;\leq\; \dim_\R H^-_{J_t}\left(\eta\beta_5\right) \;,$$
and in particular $t\mapsto h^-_{J_t}$ is not upper-semi-continuous at $0$.
\end{proof}

\medskip

The following result provides a counterexample to the lower-semi-continuity of $t\mapsto\dim_\R H^+_{J_t}$ in dimension greater than $4$, \cite[Proposition 4.3]{angella-tomassini-2}.

\begin{prop}
\label{prop:no-sci-h+}
 The compact $6$-dimensional manifold $\mathbb{S}^3\times\T^3$ is endowed with a \Cf\ (non-integrable) almost-complex structure $J$ and a curve $\left\{J_t\right\}_{t\in\Delta\left(0,\varepsilon\right)\subset\C}$, where $\varepsilon>0$, of (non-integrable) almost-complex structures (which are not \Cp), with $J_0=J$, such that
$$ \Delta\left(0,\varepsilon\right) \;\ni\; t \mapsto \dim_\R H^+_{J_t}\left(\mathbb{S}^3\times\T^3\right) \;\in\; \N $$
is not lower-semi-continuous.
\end{prop}

\begin{proof}
 The proof follows from the following example, \cite[Example 4.4]{angella-tomassini-2}.

 Consider the compact $6$-dimensional manifold $\mathbb{S}^3\times\T^3$, and set a global co-frame $\left\{e^j\right\}_{j\in\{1,\ldots,6\}}$ with respect to which the structure equations are
$$ \left( 23, \; -13, \; 12, \; 0^3 \right) \;; $$
 consider the (non-integrable) almost-complex structure $J$ defined requiring that
$$
\left\{
\begin{array}{rcl}
 \varphi^1 &:=& e^1+\im e^4 \\[5pt]
 \varphi^2 &:=& e^2+\im e^5 \\[5pt]
 \varphi^3 &:=& e^3+\im e^6
\end{array}
\right.
$$
generate the $\mathcal{C}^\infty\left(\mathbb{S}^3\times\T^3;\C\right)$-module of $(1,0)$-forms on $\mathbb{S}^3\times\T^3$. By the K\"unneth formula, one computes
\begin{eqnarray*}
H^2_{dR}(\mathbb{S}^3\times\T^3;\C)&=&\C\left\langle e^{45},\; e^{46},\; e^{56} \right\rangle \\[5pt]
&=&  \left\langle \varphi^{12}+\varphi^{\bar1\bar2},\; \varphi^{13}+\varphi^{\bar1\bar3},\; \varphi^{23}+\varphi^{\bar2\bar3} \right\rangle =H^{-}_{J}\left(\mathbb{S}^3\times\T^3\right)\\[5pt]
&=&  \left\langle \varphi^{1\bar2}-\varphi^{2\bar1},\; \varphi^{1\bar3}-\varphi^{3\bar1},\; \varphi^{2\bar3}-\varphi^{3\bar2} \right\rangle =H^{+}_{J}\left(\mathbb{S}^3\times\T^3\right) \;.
\end{eqnarray*}
For $\varepsilon>0$ small enough, consider the curve $\left\{J_t\right\}_{t\in\Delta\left(0,\varepsilon\right)\subset\C}$ of (non-integrable) almost-complex structures defined requiring that, for any $t\in\Delta\left(0,\varepsilon\right)$, the $J_t$-holomorphic cotangent bundle has co-frame
$$
\left\{
\begin{array}{rcl}
 \varphi^1_t &:=& \varphi^1 + t \,\bar\varphi^1 \\[5pt]
 \varphi^2_t &:=& \varphi^2 \\[5pt]
 \varphi^3_t &:=& \varphi^3
\end{array}
\right. \;.
$$
By using the F.~A. Belgun symmetrization trick, \cite[Theorem 7]{belgun}, we have that, for $t\in\Delta\left(0,\varepsilon\right)\setminus\R$,
$$
\left[\varphi^{1\bar2}-\varphi^{2\bar1}\right] \;=\;
\left[\frac{1}{1-|t|^2}\left(\varphi_t^{1\bar2}-\varphi_t^{2\bar1}\right)
- \frac{1}{1-|t|^2}\left(\bar t\,\varphi_t^{12}+t\,\varphi_t^{\bar1\bar2}\right)\right]
\;\not\in\; H^+_{J_t}\left(\mathbb{S}^3\times\T^3\right)
$$
and
$$
\left[\varphi^{1\bar3}-\varphi^{3\bar1}\right] \;=\;
\left[\frac{1}{1-|t|^2}\left(\varphi_t^{1\bar3}-\varphi_t^{3\bar1}\right)
- \frac{1}{1-|t|^2}\left(\bar t\,\varphi_t^{13}+t\,\varphi_t^{\bar1\bar3}\right)\right]
\;\not\in\; H^+_{J_t}\left(\mathbb{S}^3\times\T^3\right) \;;
$$
indeed, the terms $\psi_1:=\bar t\,\varphi_t^{12}+t\,\varphi_t^{\bar1\bar2}$, respectively $\psi_2:=\bar t \, \varphi^{13}+ t\, \varphi^{\bar1\bar3}$, cannot be
written as the sum of a $J_t$-invariant form and a $\de$-exact form: on the contrary, since $\psi_1$, and $\psi_2$ are left-invariant, applying Belgun's symmetrization map, \cite[Theorem 7]{belgun}, we can suppose that the $J_t$-anti-invariant component of the $\de$-exact term is actually the $J_t$-anti-invariant component of the differential of a left-invariant $1$-form; but the image of the differential on the space of left-invariant $1$-forms is
\begin{eqnarray*}
\de\wedge^{1}\duale{\mathfrak{g}}_\C &=& \C \left\langle
      \varphi^{23}_t + \varphi^{2\bar3}_t - \varphi^{3\bar2}_t+\varphi^{\bar2\bar3}_t,\,
      (1-\bar t)\, \varphi^{13}_t + (1-\bar t)\, \varphi^{1\bar3} - (1-t)\, \varphi^{3\bar1} + (1-t)\, \varphi^{\bar1\bar3}, \right. \\[5pt]
      && \left. (1-\bar t)\, \varphi^{12}_t + (1-\bar t)\, \varphi^{1\bar2} - (1-t)\, \varphi^{2\bar1} + (1-t)\, \varphi^{\bar1\bar2}
   \right\rangle \;,
\end{eqnarray*}
and hence one should have $t\in\R$.
Hence, we have that, for $t\in\Delta\left(0,\varepsilon\right)\setminus\R$,
$$ \dim_\R H^+_{J_t}\left(\mathbb{S}^3\times\T^3\right) \;=\; 1 \;<\; 3 \;=\; \dim_\R H^+_{J_0}\left(\mathbb{S}^3\times\T^3\right) \;,$$
and consequently, in particular, $t\mapsto \dim_\R H^+_{J_t}\left(\mathbb{S}^3\times\T^3\right)$ is not lower-semi-continuous at $0$.
\end{proof}

\subsubsection{Semi-continuity in a stronger sense}\label{subsubsec:stronger-semicont}

Note that Proposition \ref{prop:no-scs-h-} and Proposition \ref{prop:no-sci-h+} force us to consider stronger conditions under which semi-continuity may occur, or to slightly modify the statement of the semi-continuity problem.

\medskip

We turn our attention to the aim of giving a more precise statement of the semi-continuity problem. We notice that, for a compact $4$-dimensional manifold $X$ endowed with a family $\left\{J_t\right\}_{t\in\Delta(0,\varepsilon)}$ of almost-complex structures, one does not have only the semi-continuity properties of $t\mapsto \dim_\R H^+_{J_t}(X)$ and $t\mapsto \dim_\R H^-_{J_t}(X)$, but one gets also that every $J_0$-invariant class admits a $J_t$-invariant class {\em close to it}. This is also a sufficient condition to assure that, if $\alpha$ is a $J_0$-compatible symplectic structure on $X$, then there is a $J_t$-compatible symplectic structure $\alpha_t$ on $X$ for $t$ small enough. Therefore, we are interested in the following problem.

\smallskip
\noindent{\itshape
 Let $X$ be a compact manifold endowed with an almost-complex structure $J$ and with a curve $\left\{J_t\right\}_{t\in \left(-\varepsilon,\,\varepsilon\right)\subset\R}$ of almost-complex structures, where $\varepsilon>0$ is small enough, such that $J_0=J$. Suppose that
$$ H^+_J(X) \;=\; \C\left\langle \left[\alpha^1\right],\; \ldots,\; \left[\alpha^k\right]\right\rangle \;, $$
where $\alpha^1,\,\ldots,\,\alpha^k$ are forms of type $(1,1)$ with respect to $J$. We look for further hypotheses assuring that, for every $t\in\left(-\varepsilon,\,\varepsilon\right)$,
$$ H^+_{J_t}(X) \;\supseteq\; \C\left\langle \left[\alpha^1_t\right],\; \ldots,\; \left[\alpha^k_t\right]\right\rangle $$
with 
$$ \alpha^j_t \;=\; \alpha^j+\opiccolouno \;.$$
In this case, $\left(-\varepsilon,\, \varepsilon\right) \ni t\mapsto \dim_\R H^+_{J_t}(X) \in \N$ is a lower-semi-continuous function at $0$.
}
\smallskip

\medskip

Concerning this problem, we have the following result, \cite[Proposition 4.5]{angella-tomassini-2}.

\begin{prop}
\label{prop:semicont-forte}
 Let $X$ be a compact manifold endowed with an almost-complex structure $J$.
 Take $L\in\End(TX)$ and consider the curve $\left\{J_t\right\}_{t\in\left(-\varepsilon,\, \varepsilon\right)\subset \R}$ of almost-complex structures defined by
 $$ J_t \;:=\; \left(\id-t\,L\right)\,J\,\left(\id-t\,L\right)^{-1} \;\in\; \End(TX) \;, $$
 where $\varepsilon>0$ is small enough. For every $\left[\alpha\right]\in H^+_J(X)$ with $\alpha\in\wedge^{1,1}_{J}(X)\cap\wedge^2X$, the following conditions are equivalent:
 \begin{enumerate}[\itshape (i)]
  \item there exists a family $\left\{ \eta_t=\alpha+\opiccolouno \right\}_{t\in\left(-\varepsilon,\,\varepsilon\right)} \subseteq \wedge^{1,1}_{J_t}(X)\cap\wedge^2X$ of real $2$-forms, with $\varepsilon>0$ small enough, depending real-analytically in $t$ and such that $\de\eta_t=0$, for every $t\in\left(-\varepsilon,\,\varepsilon\right)$;
  \item there exists $\left\{\beta_j\right\}_{j\in\N\setminus\{0\}}\subseteq \wedge^2X$ solution of the system
 \begin{eqnarray}\label{eq:condizione-generale}
  \nonumber
  \lefteqn{\de \left( \beta_j + 2\,\alpha\left(L^j\sspace,\,\ssspace\right) + 4\,\sum_{k=1}^{j-1}\alpha\left(L^{j-k}\sspace, \, L^k\ssspace\right) + 2\,\alpha \left(\sspace,\, L^j\ssspace\right) \right.} \\[5pt]
  && + \left. \sum_{h=1}^{j-1} \left(2\, \beta_h \left(L^{j-h} \sspace,\, \ssspace\right) + 4\, \sum_{k=1}^{j-h-1} \alpha \left(L^{j-h-k} \sspace,\, L^k\ssspace\right) + 2\, \alpha \left(\sspace,\, L^{j-h}\ssspace\right) \right) \right) \;=\; 0 \;,
 \end{eqnarray}
 varying $j\in\N\setminus\{0\}$, such that $\sum_{j\geq 1} t^j\,\beta_j$ converges.
\end{enumerate}
In particular, the first order obstruction to the existence of $\eta_t$ as in {\itshape (i)} reads: there exists $\beta_1\in\wedge^2X$ such that
 \begin{eqnarray}\label{eq:condizione-1}
  \de\left(\beta_1+2\,\alpha(L\sspace,\,\ssspace)+2\,\alpha(\sspace,\,L\ssspace)\right) \;=\;0 \;.
 \end{eqnarray}
\end{prop}

\begin{proof}
Expanding $J_t$ in power series with respect to $t$, one gets
$$ J_t \;=\; J+\sum_{j\geq 1}2\,t^j\,J\,L^j \;, $$
and then, for every $\varphi\in\wedge^2X$, one computes
$$ J_t\,\varphi(\sspace,\,\ssspace) \;=\; J\,\varphi(\sspace,\,\ssspace)+2t\,J\,\left(\varphi(L\sspace,\, \ssspace)+\varphi(\sspace,\, L\ssspace)\right)+\opiccolo{t} $$
and
$$ \de^c_{J_t}\,\varphi \;=\; J_t^{-1}\,\de\,J_t \,\varphi\;=\; \de^c_{J}\,\varphi +2t\,J_t\,\de\,J\, \left(\varphi(L\sspace,\, \ssspace)+\varphi(\sspace,\,L\ssspace)\right)+\opiccolo{t} \;. $$

Now, given $\left[\alpha\right]\in H^+_J(X)$ with $\alpha\in\wedge^{1,1}_{J}(X)\cap\wedge^2X$, let $\left\{\beta_j\right\}_j$ be such that \eqref{eq:condizione-generale} holds and $\sum_{j\geq 1} t^j\,\beta_j$ converges, for $t\in\left(-\varepsilon,\, \varepsilon\right)$ with $\varepsilon>0$ small enough; we define
$$ \alpha_t \;:=\; \alpha+\sum_{j\geq 1} t^j\,\beta_j \;\in\; \wedge^2X $$
and
$$ \eta_t \;:=\; \frac{\alpha_t+J_t\,\alpha_t}{2} \;\in\; \wedge^{1,1}_{J_t}X \cap \wedge^2X \;.$$
By construction, $\eta_t$ is a $J_t$-invariant real $2$-form, real-analytic in $t$, and such that $\eta_t=\alpha+\opiccolouno$. A straightforward computation yields
\begin{eqnarray*}
\de\eta_t &=& \sum_{j\geq1} t^j\, \de\Biggl(\beta_j+2\,\alpha\left(L^j\sspace,\,\ssspace\right)+4\,
\sum_{k=1}^{j-1}\alpha\left(L^{j-k}\sspace,\, L^k\ssspace\right)+2\,\alpha\left(\sspace,\,L^j\ssspace\right)\\[5pt]
&&+\sum_{h=1}^{j-1}\left(2\,\beta_h\left(L^{j-h}\sspace,\,\ssspace\right)+4\,\sum_{k=1}^{j-h-1}\alpha\left(L^{j-h-k}
\sspace,\, L^k\ssspace\right)+2\,\alpha\left(\sspace,\,L^{j-h}\ssspace\right)\right)
\Biggr)
\end{eqnarray*}
therefore $\de\eta_t=0$.

Conversely, given $\left[\alpha\right]\in H^+_J(X)$ with $\alpha\in\wedge^{1,1}_{J}(X)\cap\wedge^2X$, let $\eta_t \in \wedge^{1,1}_{J_t}(X)\cap\wedge^2X$ be real-analytic in $t$ and such that $\eta_t=\alpha+\opiccolouno$ and $\de\eta_t=0$, for every $t\in\left(-\varepsilon,\,\varepsilon\right)$ with $\varepsilon>0$ small enough. Defining $\beta_j\in\wedge^2X$, for every $j\in\N\setminus\{0\}$, such that
$$ \eta_t \;=:\; \alpha+\sum_{j\geq 1}t^j\,\beta_j \;,$$
by the same computation we have that \eqref{eq:condizione-generale} holds, being $\de\eta_t=\de\left(\frac{\eta_t+J_t\,\eta_t}{2}\right)=0$.
\end{proof}

\begin{rem}
 We notice that, if $\de\, J_t \,=\, \pm J_t\,\de$ on $\wedge^2M$ for any $t$, then one can simply let
$$ \eta_t \;:=\; \frac{\alpha+J_t\,\alpha}{2} $$
so that $\eta_t\in\wedge^{1,1}_{J_t}$ and $\de\eta_t=0$. This is the case, for example, if any $J_t$
is an Abelian complex structure; C. Maclaughlin, H. Pedersen, Y.~S. Poon, and S. Salamon characterized in \cite[Theorem 6]{maclaughlin-pedersen-poon-salamon} the $2$-step nilmanifolds whose complex deformations are Abelian.
\end{rem}

\subsubsection{Counterexample to the stronger semi-continuity}\label{subsubsec:stronger-semicont-counterex}

In the following example, we provide an application of Proposition \ref{prop:semicont-forte}, showing a curve of almost-complex structures that does not have the semi-continuity property in the stronger sense described above, \cite[Example 4.8]{angella-tomassini-2}.

\begin{ex}
{\itshape A curve of almost-complex structures that does not satisfy \eqref{eq:condizione-1}.}\\
As in Example \ref{ex:instability-alm-kahler} and in Example \ref{ex:n6c}, consider, for suitable $c\in\R$, \cite[\S3]{auslander-green-hahn}, the solvmanifold
$$ N^6(c) \;:=\; \left(\Gamma(c) \left\backslash \mathrm{Sol}(3)\right.\right) \,\times\, \left(\Gamma(c) \left\backslash \mathrm{Sol}(3)\right.\right) \;,$$
which has been studied in \cite[Example 3.4]{fernandez-munoz-santisteban} as an example of a cohomologically K\"ahler manifold without K\"ahler structures, see also \cite[Example 1]{benson-gordon-solvmanifolds}. In the following, we consider $N^6:=N^6(1)$. We recall that, with respect to a suitable co-frame, the structure equations of $N^6$ are
$$ \left( 12,\; 0,\; -36,\; 24,\; 56,\; 0 \right) \;.$$	

We look for a curve $\left\{J_t\right\}_{t\in\left(-\varepsilon, \varepsilon\right)\subset \R}$ of almost-complex structures on $N^6$, where $\varepsilon>0$ is small enough, and for a $J_0$-invariant form $\alpha$ that do not satisfy the first-order obstruction \eqref{eq:condizione-1} to the stronger semi-continuity problem stated above: therefore, there will not be a $J_t$-invariant class close to $\alpha$, for any $t\in\left(-\varepsilon,\, \varepsilon\right)$.

Consider the almost-complex structure represented by
$$ J\;=\;
\left(
\begin{array}{ccc|ccc}
 &&& -1 && \\
 &&&& -1 & \\
 &&&&& -1 \\
 \hline
 1 &&&&& \\
 & 1 &&&& \\
 && 1 &&&
\end{array}
\right)
\;\in\; \End\left(TN^6\right) \;,$$
and
$$ L\;=\;
\left(
\begin{array}{c|c}
 \mathbf{A} & \mathbf{B} \\
 \hline
 \mathbf{B} & -\mathbf{A}
\end{array}
\right)
\;\in\; \End\left(TN^6\right) \;,$$
where
$$ \mathbf{A}\;=\; \left(a_{i}^{j}\right)_{i,j\in\{1,2,3\}}\;,\qquad\qquad\mathbf{B}\;=\;
\left(b_{i}^{j}\right)_{i,j\in\{1,2,3\}} $$
are constant matrices;
for
$$ \alpha \;=\; e^{14} $$
we have
$$
\de\left(\alpha(L\sspace,\,\ssspace)+\alpha(\sspace,\,L\ssspace)\right) \;=\; b_1^3\, e^{123} + a_1^2\, e^{125} - a_1^3\, e^{126} + b_1^3\, e^{136} - a_1^2\, e^{156}  + a_1^3\, e^{234} - b_1^2\, e^{245} - b_1^3\, e^{246} + a_1^3\, e^{346} + b_1^2\, e^{456} \;.
$$
Then, choosing
$$ L\;=\;
\left(
\begin{array}{ccc|ccc}
&&&&&b_1^3 \\
&&&&0& \\
&&&0&& \\
\hline
&&b_1^3&&& \\
&0&&&& \\
0&&&&&
\end{array}
\right)
\;\in\; \End\left(TN^6\right) $$
with $b_1^3\in\R\setminus\{0\}$, it is straightforward to check that there is no $\left(\mathrm{Sol}(3)\times\mathrm{Sol}(3)\right)$-left-invariant $\beta\in\wedge^2N^6$ such that
\begin{equation}\label{eq:counterex-strong-semicont}
\de \beta \;=\; b_1^3\, e^{123}+b_1^3\,e^{136}-b_1^3\,e^{246} \;;
\end{equation}
hence, by applying the F.~A. Belgun symmetrization trick, \cite[Theorem 7]{belgun}, there is no (possibly non-$\left(\mathrm{Sol}(3)\times\mathrm{Sol}(3)\right)$-left-invariant) $\beta\in\wedge^2N^6$ satisfying \eqref{eq:counterex-strong-semicont}.
\end{ex}

We resume the content of the last example in the following proposition, \cite[Proposition 4.9]{angella-tomassini-2}.

\begin{prop}\label{prop:semi-cont-strong-counterex}
 There exist a compact manifold $X$ endowed with a \Cpf\ almost-complex structure $J_0$, and a curve $\{J_t\}_{t\in\left(-\varepsilon, \varepsilon\right)}$ of almost-complex structures on $X$, where $\varepsilon>0$ is small enough, such that, for every $t\in\left(-\varepsilon, \varepsilon\right)$, there is no $J_t$-invariant class, real-analytic in $t$, close to any fixed $J_0$-invariant class.
\end{prop}

\section{Cones of metric structures}\label{sec:cones}

In introducing and studying the subgroups $H^{(\bullet,\bullet)}_J(X;\R)$ on a compact manifold $X$ endowed with an almost-complex structure $J$, T.-J. Li and W. Zhang were mainly aimed by the problem of investigating the relations between the $J$-taming and the $J$-compatible symplectic cones on $X$. As follows by their theorem \cite[Theorem 1.1]{li-zhang}, whenever $J$ is \Cf, then the subgroup $H^-_J(X)$ measures the difference between the $J$-taming cone and the $J$-compatible cone.

In this section, we discuss some results obtained in \cite{angella-tomassini-2}, jointly with A. Tomassini, giving a counterpart of T.-J. Li and W. Zhang's theorem \cite[Theorem 1.1]{li-zhang} in the semi-K\"ahler case, Theorem \ref{thm:kbc-kbt}, and, in particular, comparing the cones of balanced metrics and of strongly-Gauduchon metrics on a compact complex manifold. Furthermore, concerning the search of a holomorphic-tamed non-K\"ahler example, \cite[page 678]{li-zhang}, \cite[Question 1.7]{streets-tian}, we show that no such example can exist among $6$-dimensional nilmanifolds endowed with left-invariant complex structures, Theorem \ref{thm:nilmfd-are-not-tamed}, as proven in a joint work with A. Tomassini, \cite{angella-tomassini-1}.

\subsection{Sullivan's results on cone structures}

Firstly, we recall some results by D.~P. Sullivan, \cite[\S I.1]{sullivan}, concerning cone structures on a (differentiable) manifold $X$.

\medskip

Fixed $p\in\N$, a \emph{cone structure} of $p$-directions on $X$ is a continuous field $C:=:\{C(x)\}_{x\in X}$, with $C(x)$ a compact convex cone in $\wedge^p(T_xX)$ for every $x\in X$.

A $p$-form $\omega$ on $X$ is called \emph{transverse} to a cone structure $C$ if $\omega\lfloor_x(v)>0$ for all $v\in C(x)\setminus\{0\}$ and for all $x\in X$; using the partitions of unity, a transverse form could be constructed for any given $C$, \cite[Proposition I.4]{sullivan}.

Every cone structure $C$ gives rise to a cone $\mathfrak{C}$ of \emph{structure currents}, which are by definition the currents generated by the Dirac currents associated to the elements in $C(x)$, see \cite[Definition I.4]{sullivan}; the set $\mathfrak{C}$ is a compact convex cone in $\correnti_{p}X$.

The cone $\mathcal{Z}\mathfrak{C}$ of the \emph{structure cycles} is defined as the sub-cone of $\mathfrak{C}$ consisting of $\de$-closed currents; denote with $\mathcal{B}$ the set of $\de$-exact currents.

Define the cone $H\mathfrak{C}$ in $H^{dR}_{p}(X;\R)$ as the set of the classes of the structure cycles.

The dual cone of $H\mathfrak{C}$ in $H^p_{dR}(X;\R)$ is denoted by $\breve{H}\mathfrak{C}$ and is characterized by the relation
$$ \left(\breve{H}\mathfrak{C},H\mathfrak{C}\right) \;\geq\; 0 \;;$$
its interior is denoted by $\interno\breve{H}\mathfrak{C}$ and is characterized by the relation
$$ \left(\interno\breve{H}\mathfrak{C},H\mathfrak{C}\right) \;>\; 0 \;.$$
A cone structure of $2$-directions is said to be \emph{ample} if, for every $x\in X$, it satisfies that
$$ C(x)\cap \textrm{span}\{e\in S_\tau\st \tau \text{ is a }2\text{-plane}\} \;\neq\; \{0\} \;,$$
where $S_\tau$ is the Schubert variety, given by the set of $2$-planes intersecting $\tau$ in at least one line; by \cite[Theorem III.2]{sullivan}, an ample cone structure admits non-trivial structure cycles.

\medskip

When the $2n$-dimensional manifold $X$ is endowed with an almost-complex structure $J$, the following cone structures turn out to be particularly interesting.

For a fixed $p\in\{0,\ldots, n\}$, let $C_{p,J}$ be the cone $\left\{C_{p,J}(x)\right\}_{x\in X}$, where, for every $x\in X$, the compact convex cone $C_{p,J}(x)$ is generated by the positive combinations of $p$-dimensional complex subspaces of $T_xX\otimes_\R\C$ belonging to $\wedge^{2p}\left(T_xX\otimes_\R\C\right)$.

The cone $\mathfrak{C}_{p,J}$ of \emph{complex currents} is defined as the compact convex cone, see \cite[\S III.4]{sullivan}, of the structure currents.

The cone $\mathcal{Z}\mathfrak{C}_{p,J}$ of \emph{complex cycles} is defined as the compact convex cone, see \cite[\S III.7]{sullivan}, of the structure cycles.

The structure cone $C_{1,J}$ is ample, \cite[p. 249]{sullivan}, therefore it admits non-trivial cycles.

\medskip

We recall the following theorem by D.~P. Sullivan, which follows by Hahn and Banach's theorem.

\begin{thm}[{\cite[Theorem I.7]{sullivan}}]
Let $X$ be a compact differentiable manifold (with or without boundary) and let $C$ be a cone structure of $p$-vectors defined on a compact subspace $Y$ in the interior of $X$.
\begin{enumerate}[(\itshape i\upshape)]
 \item There are always non-trivial structure cycles in $Y$ or closed $p$-forms on $X$ transversal to the cone structure.
 \item If no closed transverse form exists, some non-trivial structure cycle in $Y$ is homologous to zero in $X$.
 \item If no non-trivial structure cycle exists, some transversal closed form is cohomologous to zero.
 \item If there are both structure cycles and transversal closed forms, then
  \begin{enumerate}[(\itshape a\upshape)]
   \item the natural map
    $$ \left\{\text{structure cycles on }Y \right\} \to \left\{\text{homology classes in }X\right\} $$
    is proper and the image is a compact cone $\mathcal{C}\subseteq H^{dR}_p(X;\R)$, and
    \item the interior of the dual cone $\breve{\mathcal{C}}\subseteq H_{dR}^p(X;\R)$ (that is, $\breve{\mathcal{C}}$ is the cone defined by the relation $\left(\breve{\mathcal{C}},\,\mathcal{C}\right)\geq 0$) consists precisely of the classes of closed forms transverse to $C$.
  \end{enumerate}
\end{enumerate}
\end{thm}

\subsection{The cones of compatible, and tamed symplectic structures}\label{subsec:symplectic-cones}
Let $X$ be a manifold endowed with an almost-complex structure $J$.

We recall that a symplectic form $\omega$ is said to \emph{tame} $J$ if it is positive on the $J$-lines, that is, if $\omega_x\left(v_x, \, J_xv_x\right)>0$ for every $v_x\in T_xX\setminus\{0\}$ and for every $x\in X$, equivalently, if
$$ \tilde{g}_J\left(\sspace,\, \ssspace\right) \;:=\; \frac{1}{2}\,\left(\omega\left(\sspace,\, J\ssspace\right)-\omega\left(J\sspace,\, \ssspace\right)\right) $$
is a $J$-Hermitian metric on $X$ with $\pi_{\wedge^{1,1}X}\omega$ as the associated $(1,1)$-form (the map $\pi_{\wedge^{1,1}X}\colon\wedge^\bullet X\to\wedge^{1,1}X$ being the natural projection onto $\wedge^{1,1}X$). A symplectic form $\omega$ is called \emph{compatible} with $J$ if it tames $J$ and it is $J$-invariant, equivalently, if $\omega$ is the $(1,1)$-form associated to the $J$-Hermitian metric $g_J\left(\sspace,\, \ssspace\right) := \omega(\sspace,\, J\ssspace)$. In particular, an integrable almost-complex structure $J$ is called \emph{holomorphic-tamed} if it admits a taming symplectic form; on the other hand, the datum of an integrable almost-complex structure and a compatible symplectic form gives a K\"ahler structure.

\subsubsection{Symplectic cones and Donaldson's question}\label{subsubsec:donaldson-question}

Consider the \emph{$J$-tamed cone} $\mathcal{K}^t_J$, which is defined as the set of the cohomology classes of the $J$-taming symplectic forms, namely,
$$ \mathcal{K}^t_J \;:=\; \left\{ \left[\omega\right]\in H^2_{dR}(X;\R) \st \omega \text{ is a }J\text{-taming symplectic form on } X\right\} \;, $$
and the \emph{$J$-compatible cone} $\mathcal{K}^c_J$, which is defined as the set of the cohomology classes of the $J$-compatible symplectic forms, namely,
$$ \mathcal{K}^c_J \;:=\; \left\{ \left[\omega\right]\in H^2_{dR}(X;\R) \st \omega \text{ is a }J\text{-compatible symplectic form on } X\right\} \;. $$
The set $\mathcal{K}^t_J$ is an open convex cone in $H^2_{dR}(X;\R)$, and the set $\mathcal{K}^c_J$ is a convex sub-cone of $\mathcal{K}^t_J$ and it is contained in $H^{(1,1)}_J(X;\R)$; moreover, both the sets are sub-cones of the \emph{symplectic cone}
$$ \mathcal{S} \;:=\; \left\{\left[\omega\right]\in H^2_{dR}(X;\R) \st \omega \text{ is a symplectic form on }X\right\} $$
in $H^2_{dR}(X;\R)$. 

\medskip

T.-J. Li and W. Zhang proved the following result in \cite{li-zhang}, concerning the relation between the $J$-tamed and the $J$-compatible cones.

\begin{thm}[{\cite[Proposition 3.1, Theorem 1.1, Corollary 1.1]{li-zhang}}]
Let $X$ be a compact manifold endowed with an almost-K\"ahler structure $J$ (namely, $J$ is an almost-complex structure on $X$ such that $\mathcal{K}^c_J\neq\varnothing$). Then
$$ \mathcal{K}^t_J\cap H^{(1,1)}_J(X;\R) \;=\; \mathcal{K}^c_J \qquad \text{ and }\qquad \mathcal{K}^c_J+H^{(2,0),(0,2)}_J(X;\R) \subseteq \mathcal{K}^t_J \;. $$
Moreover, if $J$ is \Cf, then
$$ \mathcal{K}^t_J\;=\; \mathcal{K}^c_J+ H^{(2,0),(0,2)}_J(X;\R) \;. $$
In particular, if $\dim X=4$ and $b^+(X)=1$, then $\mathcal{K}^t_J=\mathcal{K}^c_J$.
\end{thm}

The proof is essentially based on \cite[Theorem I.7]{sullivan}. Note indeed that the closed forms transverse to the cone $C_{1,J}$ are exactly the $J$-taming symplectic forms. By \cite[Theorem I.7(iv)(b)]{sullivan}, it follows that $\mathcal{K}^t_J$ is the interior of the dual cone $\breve{H}\mathfrak{C}_{1,J}\subseteq H_{dR}^2(X;\R)$ of $H\mathfrak{C}_{1,J}\subseteq H^{dR}_2(X;\R)$, \cite[Theorem 3.2]{li-zhang}. On the other hand, assumed that $\mathcal{K}^c_J$ is non-empty, by the Hahn and Banach separation theorem, $\mathcal{K}^c_J$ is the interior of the dual cone of $H\mathfrak{C}_{1,J}\subseteq H_{(1,1)}^J(X;\R)$, \cite[Theorem 3.4]{li-zhang}. Finally, when $\dim X=4$, chosen a $J$-Hermitian metric $g$ on $X$ with associated $(1,1)$-form $\omega$, one has $\wedge^+_gX=\R\left\langle\omega\right\rangle \oplus \wedge^-_JX$, hence, in the almost-K\"ahler case, if $b^+(X):=\dim_\R H^{+}_g(X)=1$, then $H^{(2,0),(0,2)}_J(X;\R)=\{0\}$, see \cite[Proposition 3.1]{draghici-li-zhang}.

\medskip

Whereas the previous theorem by T.-J. Li and W. Zhang could be intended as a ``quantitative comparison'' between the $J$-taming and the $J$-compatible symplectic cones on a compact manifold $X$ endowed with an almost-complex structure $J$, one could ask what about their ``qualitative comparison'', namely, one could ask whether $\mathcal{K}^c_J$ being empty implies $\mathcal{K}^t_J$ being empty, too. The following question has been arisen by S.~K. Donaldson in \cite{donaldson}.

\begin{quest}[{\cite[Question 2]{donaldson}}]
 Let $X$ be a compact $4$-dimensional manifold endowed with an almost-complex structure $J$. If $J$ is tamed by a symplectic form, is there a symplectic form compatible with $J$?
\end{quest}

\begin{rem}
S.~K. Donaldson's ``tamed to compatible'' question has a positive answer for $\CP^2$ by the works by M. Gromov \cite{gromov} and by C.~H. Taubes, \cite{taubes-seiberg-witten}. When $b^+(X)=1$ (where $b^+$ is the number of positive eigenvalues of the intersection pairing on $H_2(X;\R)$), a possible positive answer to \cite[Question 2]{donaldson}, see also \cite[Conjecture 1.2]{tosatti-weinkove-yau}, would be provided as a consequence of \cite[Conjecture 1]{donaldson}, see also \cite[Conjecture 1.1]{tosatti-weinkove-yau}, concerning the study of the symplectic Calabi and Yau equation, which aims to generalize S.-T. Yau's theorem \cite{yau-proc, yau}, solving the Calabi conjecture, \cite{calabi}, to the non-integrable case. Some results concerning this problem have been recently obtained by several authors, see, e.g., \cite{weinkove, tosatti-weinkove-yau, tosatti-weinkove_kodaira-thurston, taubes, zhang, li-tomassini, fino-li-salamon-vezzoni, buzano-fino-vezzoni}, see also \cite{tosatti-weinkove}. More 
precisely, in \cite{weinkove}, all the estimates for the closedness argument of the continuity method applied to the symplectic Calabi and Yau equation, \cite[Conjecture 1]{donaldson}, are reduced to a $\mathcal{C}^0$ {\itshape a priori} estimate of a scalar potential function, \cite[Theorem 1]{weinkove}; then, the existence of a solution of the symplectic Calabi and Yau equation is proven for compact $4$-dimensional manifolds $X$ endowed with an almost-K\"ahler structure $\left(J,\, \omega,\, g\right)$ satisfying $\left\|\Nij_J\right\|_{\mathrm{L}^1}<\varepsilon$, where $\varepsilon>0$ depends just on the data, \cite[Theorem 2]{weinkove}. In \cite{tosatti-weinkove-yau}, it is shown that the $\mathcal{C}^\infty$ {\itshape a priori} estimates can be reduced to an integral estimate of a scalar function potential, \cite[Theorem 1.3]{tosatti-weinkove-yau}; furthermore, it is shown that \cite[Conjecture 1]{donaldson} holds under a positive curvature assumption, \cite[Theorem 1.4]{tosatti-weinkove-yau}. In \cite{
tosatti-weinkove_kodaira-thurston}, the symplectic Calabi-Yau equation is solved on the Kodaira-Thurston manifold $\mathbb{S}^1\times\left(\left. \mathbb{H}(3;\Z) \right\backslash \mathbb{H}(3;\R)\right)$ for any given left-invariant volume form, \cite[Theorem 1.1]{tosatti-weinkove_kodaira-thurston}; further results on the Calabi-Yau equation for torus-bundles over a $2$-dimensional torus have been provided in \cite{fino-li-salamon-vezzoni, buzano-fino-vezzoni}. In \cite{taubes}, it is shown that, on a compact $4$-dimensional manifold with $b^+=1$ and endowed with a symplectic form $\omega$, a generic $\omega$-tamed almost-complex structure on $X$ is compatible with a symplectic form on $X$, \cite[Theorem 1]{taubes}, which is defined by integrating over a space of currents that are defined by pseudo-holomorphic curves. The Taubes currents have been studied, both in dimension $4$ and higher, also by W. Zhang in \cite{zhang}. In \cite{li-zhang-2}, T.-J. Li and W. Zhang were concerned with studying Donaldson's 
``tamed to compatible'' question for almost-complex structures on rational $4$-dimensional manifolds; they provided, in particular, an affirmative answer to \cite[Question 2]{donaldson} for $\mathbb{S}^2\times\mathbb{S}^2$ and for $\CP^2\sharp\overline{\CP^2}$, see \cite[Theorem 4.11]{draghici-li-zhang-survey}. In \cite{li-tomassini}, a positive answer to S.~K. Donaldson's question \cite[Question 2]{donaldson} is provided in the Lie algebra setting, proving that, given a $4$-dimensional Lie algebra $\mathfrak{g}$ such that $B\wedge B=0$ (where $B\subseteq \wedge^2\mathfrak{g}$ is the space of boundary $2$-vectors), e.g., a $4$-dimensional unimodular Lie algebra, a linear (possibly non-integrable) complex structure is tamed by a linear symplectic form if and only if it is compatible with a linear symplectic form, \cite[Theorem 0.2]{li-tomassini}.
\end{rem}

In a sense, \cite[Corollary 1.1]{li-zhang} provides evidences towards an affirmative answer for \cite[Question 2]{donaldson}, especially in the case $b^+=1$; confirmed in their opinion by the computations in \cite{draghici-li-zhang} in the case $b^+>1$, T.-J. Li and W. Zhang speculated in \cite[page 655]{li-zhang} that the equality $\mathcal{K}^t_J=\mathcal{K}^c_J$ holds for a generic almost-complex structure $J$ on a $4$-dimensional manifold.

\medskip

The analogous of \cite[Question 2]{donaldson} in dimension higher than $4$ has a negative answer: counterexamples in the (non-integrable) almost-complex case can be found in \cite{migliorini-tomassini} by M. Migliorini and A. Tomassini, and in \cite{tomassini-forummath} by A. Tomassini. Notwithstanding, since examples of non-K\"ahler holomorphic-tamed complex structures are not known, T.-J. Li and W. Zhang speculated a negative answer for the following question, also addressed by J. Streets and G. Tian in \cite{streets-tian}.

\begin{quest}[{\cite[page 678]{li-zhang}, \cite[Question 1.7]{streets-tian}}]
 Do there exist non-K\"ahler holomorphic-tamed complex manifolds, of complex dimension greater than $2$?
\end{quest}

\subsubsection{Tameness conjecture for $6$-dimensional nilmanifolds}

In view of the speculation in \cite[page 678]{li-zhang}, and of \cite[Question 1.7]{streets-tian}, one could ask whether small deformations of the Iwasawa manifold, see \S\ref{subsec:iwasawa}, may provide examples of non-K\"ahler holomorphic-tamed complex structures. In this section, we prove that this is not the case: more precisely, we prove that no example of left-invariant non-K\"ahler holomorphic-tamed complex structure can be found on $6$-dimensional nilmanifolds. The same holds true, more in general, for higher dimensional nilmanifolds, as proven by N. Enrietti, A. Fino, and L. Vezzoni, \cite[Theorem 1.3]{enrietti-fino-vezzoni}.

\medskip

We recall that a Hermitian metric $g$ on a complex manifold $X$ is called \emph{pluri-closed} (or \emph{strong K\"ahler with torsion}, shortly \emph{\textsc{skt}}), \cite{bismut--math-ann-1989}, if the $(1,1)$-form $\omega$ associated to $g$ satisfies $\del\delbar\omega=0$.

By the following result, holomorphic-tamed manifolds admit pluri-closed metrics, \cite[Proposition 3.1]{angella-tomassini-1}.

\begin{prop}
\label{prop:hol-tamed-skt}
Let $X$ be a manifold endowed with a symplectic structure $\omega$ and an $\omega$-tamed complex structure $J$.
Then the $(1,1)$-form $\tilde\omega := \tilde{g}_J\left(J\,\sspace, \, \ssspace\right)$ associated to the Hermitian metric $\tilde{g}_J\left(\sspace,\, \ssspace\right) := \frac{1}{2}\,\left(\omega\left(\sspace,\, J\ssspace\right)-\omega\left(J\sspace,\, \ssspace\right)\right)$ is $\del\delbar$-closed, namely, $\tilde{g}$ is a pluri-closed metric on $X$.
\end{prop}

\begin{proof}
Decomposing $\omega$ in pure type components, set
$$ 
\omega \;=:\; \omega^{2,0} + \omega^{1,1} + \overline{\omega^{2,0}} $$
where $\omega^{2,0}\in\wedge^{2,0}X$ and $\omega^{1,1}=\overline{\omega^{1,1}}\in\wedge^{1,1}X$.
Since, by definition, $\tilde{\omega} = \frac{1}{2}\,\left(\omega+ J\omega\right)$, we have $\tilde{\omega} = \omega^{1,1}$.
We get that
$$
\de\omega \;=\; 0 \quad\Leftrightarrow \quad \left\{
\begin{array}{l}
\del\omega^{2,0} \;=\;  0 \\[5pt]
\del\omega^{1,1} + \delbar\omega^{2,0} \;=\; 0
\end{array}
\right. \;,
$$
and hence
$$
 \del\delbar \tilde{\omega} \;=\; \del\delbar \omega^{1,1} \;=\; 
-\delbar\del\omega^{1,1} \;=\; \delbar^2\omega^{2,0} \;=\; 0\;,
$$
proving that $\tilde{g}$ is a pluri-closed metric on $X$.
\end{proof}

Now, we can prove the announced theorem, \cite[Theorem 3.3]{angella-tomassini-1}.

\begin{thm}\label{thm:nilmfd-are-not-tamed}
Let $X=\Gamma\backslash G$ be a $6$-dimensional nilmanifold endowed with a $G$-left-invariant complex structure $J$.
If $X$ is not a torus, then there is no symplectic structure $\omega$ on $X$ taming $J$.
\end{thm}

\begin{proof}
Let $\omega$ be a (non-necessarily $G$-left-invariant) symplectic form on $X$ taming $J$. By F.~A. Belgun's symmetrization trick, \cite[Theorem 7]{belgun}, setting
$$ \mu(\omega)\;:=\;\int_X \omega\lfloor_m \, \eta(m) \;,$$
where $\eta$ is a $G$-bi-invariant volume form on $G$ such that $\int_X\eta=1$, whose existence follows from J. Milnor's lemma \cite[Lemma 6.2]{milnor}, we get a $G$-left-invariant symplectic form on $X$ taming $J$. Then, it suffices to prove that, on a non-torus $6$-dimensional nilmanifold, there is no left-invariant symplectic structure taming a left-invariant complex structure.

Hence, let $\omega$ be such a $G$-left-invariant symplectic structure. Then, by Proposition \ref{prop:hol-tamed-skt}, $X$ should admit a $G$-left-invariant pluri-closed Hermitian metric $g$. Hence, by \cite[Theorem 1.2]{fino-parton-salamon}, there exists a co-frame $\{\varphi^1,\;\varphi^2,\;\varphi^3\}$ for the $J$-holomorphic cotangent bundle such that
\begin{equation*}
\left\{
\begin{array}{l}
\de\varphi^1 \;=\; 0 \\[5pt]
\de\varphi^{2} \;=\; 0 \\[5pt]
\de\varphi^3 \;=\; A\,\overline{\varphi}^1\wedge\varphi^2+B\,\overline{\varphi}^2\wedge\varphi^2+
C\,\varphi^1\wedge\overline{\varphi}^1+
D\,\varphi^1\wedge\overline{\varphi}^2+E\,\varphi^1\wedge\varphi^2
\end{array}
\right. \;,
\end{equation*}
where $A,\,B,\,C,\,D,\,E\in \C$ are complex numbers such that $\left|A\right|^2+\left|D\right|^2+\left|E\right|^2+2\,\Re\left(\bar B C\right)=0$.
Set
$$
\omega \;=:\; \omega^{2,0}+\omega^{1,1}+\overline{\omega^{2,0}}\;,
$$
where
$$
\omega^{2,0}\;=\;\sum_{i<j}a_{ij}\,\varphi^i\wedge\varphi^j\;,
\qquad
\omega^{1,1}\;=\;\frac{\im}{2}\,\sum_{i,j=1}^3b_{i\,\overline{j}}\,\varphi^{i}\wedge\overline{\varphi}^j\;,
$$
with $\left\{a_{ij},\, b_{i\,\overline{j}}\right\}_{i,j}\subset\C$ such that $\omega^{1,1}=\overline{\omega^{1,1}}$. A
straightforward computation yields
$$ \de\omega\;=\; 0 \qquad \Leftrightarrow \qquad \left(A\;=\;B\;=\;C\;=\;D\;=\;E\;=\;0 \quad \text{ or } \quad b_{3\overline{3}}\;=\;0\right)\;.$$
Since $b_{3\bar3}\neq0$, we get $A=B=C=D=E=0$, namely, $X$ is a torus.
\end{proof}

As a corollary, we get the following result, \cite[Theorem 3.4]{angella-tomassini-1}, concerning the speculation in \cite[p. 678]{li-zhang}, and \cite[Question 1.7]{streets-tian}.

\begin{thm}
No small deformation of the complex structure of the Iwasawa manifold $\mathbb{I}_3 := \left. \mathbb{H}\left(3;\Z\left[\im\right]\right) \right\backslash \mathbb{H}(3;\C)$ can be tamed by any symplectic form.
\end{thm}

\subsection{The cones of semi-K\"ahler, and strongly-Gauduchon metrics}\label{subsec:balanced-cones}
Let $X$ be a compact $2n$-dimensional manifold endowed with an almost-complex structure $J$. We recall that a non-degenerate $2$-form $\omega$ on $X$ is called \emph{semi-K\"ahler}, \cite[page 40]{gray-hervella}, if $\omega$ is the $(1,1)$-form associated to a $J$-Hermitian metric on $X$ (that is, $\omega(\sspace,\,J\sspace)>0$ and $\omega(J\sspace,\,\,J\ssspace)=\omega(\sspace,\,\ssspace)$) and $\de\left(\omega^{n-1}\right) = 0$;
when $J$ is integrable, a semi-K\"ahler structure is called \emph{balanced}, \cite[Definition 1.4, Theorem 1.6]{michelsohn}.

\medskip

We set
\begin{eqnarray*}
\mathcal{K}b^c_J &:=& \left\{\left[\Omega\right]\in H^{2n-2}_{dR}(X;\R) \st \Omega\in\wedge^{n-1,n-1}X \text{ is positive on the}\right.\\[5pt]
&& \left.\text{ complex }(n-1)\text{-subspaces of } T_xX\otimes_\R\C, \text{ for every }x\in X \right\} \;,
\end{eqnarray*}
and
\begin{eqnarray*}
\mathcal{K}b^t_J &:=& \left\{\left[\Omega\right]\in H^{2n-2}_{dR}(X;\R) \st \Omega\in\wedge^{2n-2}X \text{ is positive on the}\right.\\[5pt]
&& \left.\text{ complex }(n-1)\text{-subspaces of } T_xX\otimes_\R\C, \text{ for every }x\in X \right\} \;.
\end{eqnarray*}
We note that $\mathcal{K}b^c_J$ and $\mathcal{K}b^t_J$ are convex cones in $H^{2n-2}_{dR}(X;\R)$, and that $\mathcal{K}b^c_J$ is a sub-cone of $\mathcal{K}b^t_J$ and is contained in $H^{(n-1,n-1)}_J(X;\R)$.

We recall the following trick by M.~L. Michelsohn.

\begin{lem}[{\cite[pp. 279-280]{michelsohn}}]
 Let $X$ be a compact $2n$-dimensional manifold endowed with an almost-complex structure $J$.
 Let $\Phi$ be a real $(n-1,n-1)$-form such that it is positive on the complex $(n-1)$-subspaces of $T_xX\otimes_\R\C$, for every $x\in X$. Then $\Phi$ can be written as $\Phi=\varphi^{n-1}$, where $\varphi$ is a $J$-taming real $(1,1)$-form. In particular, if $\Phi$ is $\de$-closed, then $\varphi$ is a semi-K\"ahler form.
\end{lem}

The previous Lemma allows us to confuse the cone $\mathcal{K}b^{c}_J$ with the cone generated by the \kth{(n-1)} powers of the semi-K\"ahler forms, namely,
$$ \mathcal{K}b^c_J = \left\{\left[\omega^{n-1}\right] \st \omega \text{ is a semi-K\"ahler form on }X\right\} \;. $$

In particular, if $J$ is integrable, then the cone $\mathcal{K}b^c_J$ is just the cone of balanced structures on $X$. On the other hand, in the integrable case, $\mathcal{K}b^t_J$ is the cone of strongly-Gauduchon metrics on $X$. We recall that a \emph{strongly-Gauduchon metric} on $X$, \cite[Definition 3.1]{popovici-proj}, is a positive-definite $(1,1)$-form $\gamma$ on $X$ such that the $(n,n-1)$-form $\del\left(\gamma^{n-1}\right)$ is $\delbar$-exact. These metrics have been introduced by D. Popovici in \cite{popovici-proj} in studying the limits of projective manifolds under deformations of the complex structure, and they turn out to be special cases of \emph{Gauduchon metrics}, \cite{gauduchon}, for which $\del\left(\gamma^{n-1}\right)$ is just $\delbar$-closed; note that the notions of Gauduchon metric and of strongly-Gauduchon metric coincide if the $\del\delbar$-Lemma holds, \cite[page 15]{popovici-proj}. D. Popovici proved in \cite[Lemma 3.2]{popovici-proj} that a compact complex manifold $X$, of 
complex dimension $n$, carries a strongly-Gauduchon metric if and only if there exists a real $\de$-closed $(2n-2)$-form $\Omega$ such that its component $\Omega^{(n-1,n-1)}$ of type $(n-1,n-1)$ satisfies $\Omega^{(n-1,n-1)}>0$.

The aim of this section is to compare the cones $\mathcal{K}b^c_J$ and $\mathcal{K}b^t_J$, Theorem \ref{thm:kbc-kbt}, in the same way as \cite[Theorem 1.1]{li-zhang} does for $\mathcal{K}^c_J$ and $\mathcal{K}^t_J$ in the almost-K\"ahler case.

\medskip

Note that $\mathcal{K}b^t_J$ can be identified with the set of the classes of $\de$-closed $(2n-2)$-forms transverse to $C_{n-1,J}$.
On the other hand, we recall the following lemma.

\begin{lem}[{see, e.g., \cite[Proposition I.1.3]{silva}}]
 Let $X$ be a compact manifold endowed with an almost-complex structure $J$, and fix $p\in\N$.
 A structure current in $\mathfrak{C}_{p,J}$ is a positive current of bi-dimension $(p,p)$.
\end{lem}

As a direct consequence of \cite[Theorem I.7]{sullivan}, we get the following result, \cite[Theorem 2.6]{angella-tomassini-2}.

\begin{thm}
\label{thm:characterization-kbt}
 Let $X$ be a compact $2n$-dimensional manifold endowed with an almost-complex structure $J$.
 Then $\mathcal{K}b^t_J$ is non-empty if and only if there is no non-trivial $\de$-closed positive currents of bi-dimension $(n-1,n-1)$ that is a boundary, i.e.,
 $$ \mathcal{Z}\mathfrak{C}_{n-1,J} \cap \mathcal{B} \;=\; \{0\} \;. $$
 Furthermore, if we suppose that $0\not\in\mathcal{K}b^t_J$, then  $\mathcal{K}b^t_J\subseteq H^{2n-2}_{dR}(X;\R)$ is the interior of the dual cone $\breve{H}\mathfrak{C}_{n-1,J}\subseteq H^{2n-2}_{dR}(X;\R)$ of $H\mathfrak{C}_{n-1,J}\subseteq H_{2n-2}^{dR}(X;\R)$.
\end{thm}

\begin{proof}
 Note that if $\omega\in\mathcal{K}b^t_J\neq\varnothing$, and if $\eta:=:\de\xi$ is a non-trivial $\de$-closed positive current of bi-dimension $(n-1,n-1)$ being a boundary, then
 $$ 0 \;<\; \left(\eta,\; \pi_{\wedge^{n-1,n-1}X}\omega\right) \;=\; \left(\eta,\; \omega\right) \;=\; \left(\de\xi,\; \omega\right) \;=\; \left(\xi,\; \de\omega\right) \;=\; 0 $$
(where $\pi_{\wedge^{n-1,n-1}X}\colon\wedge^\bullet X\to \wedge^{n-1,n-1}X$ is the natural projection onto $\wedge^{n-1,n-1}X$) yields an absurd.

To prove the converse, suppose that no non-trivial $\de$-closed positive currents of bi-dimension $(n-1,n-1)$ is a boundary; then, by \cite[Theorem I.7(ii)]{sullivan}, there exists a $\de$-closed form that is transverse to $C_{n-1,J}$, that is, $\mathcal{K}b^t_J$ is non-empty.

The last statement follows from \cite[Theorem I.7(iv)]{sullivan}: indeed, by the assumption $0\not\in\mathcal{K}b^t_J$, no $\de$-closed transverse form is cohomologous to zero, therefore, by \cite[Theorem I.7(iii)]{sullivan}, there exists a non-trivial structure cycle.
\end{proof}

We provide a similar characterization for $\mathcal{K}b^c_J$, \cite[Theorem 2.7]{angella-tomassini-2}.

\begin{thm}
\label{thm:characterization-kbc}
 Let $X$ be a compact $2n$-dimensional manifold endowed with an almost-complex structure $J$.
 Suppose that $\mathcal{K}b^c_J\neq\varnothing$ and that $0\not\in\mathcal{K}b^c_J$. Then  $\mathcal{K}b^c_J\subseteq H^{(n-1,n-1)}_J(X;\R)$ is the interior of the dual cone $\breve{H}\mathfrak{C}_{n-1,J}\subseteq H_J^{(n-1,n-1)}(X;\R)$ of $H\mathfrak{C}_{n-1,J}\subseteq H^J_{(n-1,n-1)}(X;\R)$.
\end{thm}

\begin{proof}
 By the hypothesis $0\not\in\mathcal{K}^c_J$, we have that $\left(\mathcal{K}b^c_J,H\mathfrak{C}_{n-1,J}\right)>0$, and therefore the inclusion  $\mathcal{K}b^c_J\subseteq\interno\breve{H}\mathfrak{C}_{n-1,J}$ holds.

 To prove the other inclusion, let $e\in H^{(n-1,n-1)}_J(X;\R)$ be an element in the interior of the dual cone in $H_J^{(n-1,n-1)}$ of $H\mathfrak{C}_{n-1,J}$, i.e., $e$ is such that $\left(e,H\mathfrak{C}_{n-1,J}\right)>0$.
 Consider the isomorphism
 $$ \overline{\sigma}^{n-1,n-1}\colon H^{(n-1,n-1)}_J(X;\R) \stackrel{\simeq}{\to} \duale{\left(\frac{\overline{\pi_{\correnti_{n-1,n-1}X}\mathcal{Z}}}{\overline{\pi_{\correnti_{n-1,n-1}X}\mathcal{B}}} \right)} \;, $$
 \cite[Proposition 2.4]{li-zhang} (where $\pi_{\correnti_{n-1,n-1}X}\colon \correnti_\bullet X\to\correnti_{n-1,n-1}X$ denotes the natural projection onto $\correnti_{n-1,n-1}X$): hence, $\overline{\sigma}^{n-1,n-1}(e)$ gives rise to a functional on $\frac{\overline{\pi_{\correnti_{n-1,n-1}X}\mathcal{Z}}}{\overline{\pi_{\correnti_{n-1,n-1}X}\mathcal{B}}}$, namely, to a functional on $\overline{\pi_{\correnti_{n-1,n-1}X}\mathcal{Z}}$ vanishing on $\overline{\pi_{\correnti_{n-1,n-1}X}\mathcal{B}}$; such a functional, in turn, gives rise to a hyperplane $L$ in $\overline{\pi_{\correnti_{n-1,n-1}X}\mathcal{Z}}$ containing $\overline{\pi_{\correnti_{n-1,n-1}X}\mathcal{B}}$. Being a kernel hyperplane in a closed set, $L$ is closed in $\correnti_{n-1,n-1}X\cap\correnti_{2n-2}X$; furthermore, $L$ is disjoint from $\mathfrak{C}_{n-1,J}\setminus\{0\}$, by the choice made for $e$.
 Pick a $J$-Hermitian metric and let $\varphi$ be its associated $(1,1)$-form; consider
 $$ K \;:=\; \left\{ T\in \mathfrak{C}_{n-1,J} \st T\left(\varphi^{n-1}\right)=1 \right\} \;, $$
 which is a compact set.
 Now, in the space $\correnti_{n-1,n-1}X\cap\correnti_{2n-2}X$, consider the closed set $L$, and the compact convex non-empty set $K$, which have empty intersection. By the Hahn and Banach separation theorem, there exists a hyperplane containing $L$, and then containing also $\overline{\pi_{\correnti_{n-1,n-1}X}\mathcal{B}}$, and disjoint from $K$. The functional on $\correnti_{n-1,n-1}X\cap\correnti_{2n-2}X$ associated to this hyperplane is a real $(n-1,n-1)$-form being $\de$-closed, since it vanishes on $\overline{\pi_{\correnti_{n-1,n-1}X}\mathcal{B}}$, and positive on the complex $(n-1)$-subspaces of $T_xX\otimes_\R\C$, for every $x\in X$, that is, a $J$-compatible symplectic form.
\end{proof}

The same argument as in \cite[Proposition 12, Theorem 14]{harvey-lawson} yields the following result, \cite[Theorem 2.8]{angella-tomassini-2}, which generalizes \cite[Proposition 12, Theorem 14]{harvey-lawson}, \cite[page 671]{li-zhang}, see also \cite[Theorem 4.7]{michelsohn}.

\begin{thm}[{\cite[Proposition 12, Theorem 14]{harvey-lawson}, \cite[page 671]{li-zhang}, \cite[Theorem 2.8]{angella-tomassini-2}}]
\label{thm:caratterizzazione-intrinseca}
 Let $X$ be a compact $2n$-dimensional manifold endowed with an almost-complex structure $J$, and denote by $\pi_{\correnti_{k,k}X} \colon \correnti_\bullet X \to \correnti_{k,k}X$ the natural projection onto $\correnti_{k,k}X$, for every $k\in\N$.
\begin{enumerate}[(\itshape i\upshape)]
 \item If $J$ is integrable, then there exists a K\"ahler metric if and only if $\mathfrak{C}_{1,J}\cap\pi_{\correnti_{1,1}X}\mathcal{B}=\{0\}$.
 \item There exists an almost-K\"ahler metric if and only if $\mathfrak{C}_{1,J}\cap \overline{\pi_{\correnti_{1,1}X}\mathcal{B}}=\{0\}$.
 \item There exists a semi-K\"ahler metric if and only if $\mathfrak{C}_{n-1,J}\cap\overline{\pi_{\correnti_{n-1,n-1}X}\mathcal{B}}=\{0\}$.
\end{enumerate}
\end{thm}

\begin{proof}
 Note that {\itshape(i)}, namely, \cite[Proposition 12, Theorem 14]{harvey-lawson}, is a consequence of {\itshape(ii)}: indeed, if $J$ is integrable, then $J$ is closed, \cite[Lemma 6]{harvey-lawson}, that is, $\pi_{\correnti_{1,1}X}\mathcal{B}$ is a closed set.

 The proof of {\itshape(ii)}, namely, \cite[page 671]{li-zhang}, being similar, we prove {\itshape(iii)}, following closely the proof of {\itshape(i)} in \cite[Proposition 12, Theorem 14]{harvey-lawson}.

 Firstly, note that if $\omega$ is a semi-K\"ahler form and
 $$ 0 \;\neq\; \eta \;:=:\; \lim_{k\to+\infty} \pi_{\correnti_{n-1,n-1}X}\left(\de\alpha_k\right) \;\in\; \mathfrak{C}_{n-1,J}\cap\overline{\pi_{\correnti_{n-1,n-1}X}\mathcal{B}} \;\neq\; \{0\} \;, $$
 where $\left\{\alpha_k\right\}_{k\in\N}\subset \correnti_{2n-1}X$, then
 \begin{eqnarray*}
  0 &<& \left(\eta,\; \omega^{n-1}\right) \;=\; \left(\lim_{k\to+\infty}\pi_{\correnti_{n-1,n-1}X}\left(\de\alpha_k\right),\; \omega^{n-1}\right) \;=\; \lim_{k\to+\infty}\left(\de\alpha_k,\; \omega^{n-1}\right) \;=\; \lim_{k\to+\infty}\left(\alpha_k,\; \de\omega^{n-1}\right) \;=\; 0 \;,
 \end{eqnarray*}
 yielding an absurd.

 For the converse, fix a $J$-Hermitian metric and let $\varphi$ be its associated $(1,1)$-form; the set
 $$ K \;:=\; \left\{T\in\mathfrak{C}_{n-1,J}\st T\left(\varphi^{n-1}\right)=1\right\} $$
 is a compact convex non-empty set in $\correnti_{n-1,n-1}X\cap\correnti_{2n-2}X$. By the Hahn and Banach separation theorem, there exists a hyperplane in $\correnti_{n-1,n-1}X\cap\correnti_{2n-2}X$ containing the closed subspace $\overline{\pi_{\correnti_{n-1,n-1}X}\mathcal{B}}$ and disjoint from $K$; hence, the real $(n-1,n-1)$-form associated to this hyperplane is a real $\de$-closed $(n-1,n-1)$-form and is positive on the complex $(n-1)$-subspaces, namely, it is a semi-K\"ahler form.
\end{proof}

Now, we can prove the semi-K\"ahler counterpart, \cite[Theorem 2.9]{angella-tomassini-2}, of T.-J. Li and W. Zhang's \cite[Proposition 3.1, Theorem 1.1]{li-zhang}.

\begin{thm}
\label{thm:kbc-kbt}
 Let $X$ be a compact $2n$-dimensional manifold endowed with an almost-complex structure $J$.
 Assume that $\mathcal{K}b^c_J\neq\varnothing$ (that is, there exists a semi-K\"ahler structure on $X$) and that $0\not\in\mathcal{K}b^t_J$. Then
 \begin{equation*}
  \mathcal{K}b^t_J\cap H^{(n-1,n-1)}_J(X;\R) \;=\; \mathcal{K}b^c_J
 \end{equation*}
 and
 \begin{equation*}
  \mathcal{K}b^c_J+H^{(n,n-2),(n-2,n)}_J(X;\R) \;\subseteq\; \mathcal{K}b^t_J \;. 
 \end{equation*}
 Moreover, if $J$ is \Cf\ at the \kth{(2n-2)} stage, then
 \begin{equation*}
 \mathcal{K}b^c_J+H^{(n,n-2),(n-2,n)}_J(X;\R) \;=\; \mathcal{K}b^t_J \;.
 \end{equation*}
\end{thm}

\begin{proof}
By Theorem \ref{thm:characterization-kbt}, $\mathcal{K}b^t_J \subseteq H^{2n-2}_{dR}(X;\R)$ is the interior of the dual cone $\breve{H}\mathfrak{C}_{n-1,J}\subseteq H^{2n-2}_{dR}(X;\R)$ of $H\mathfrak{C}_{n-1,J} \subseteq H_{2n-2}^{dR}(X;\R)$, and, by Theorem \ref{thm:characterization-kbc}, $\mathcal{K}b^c_J \subseteq H^{(n-1,n-1)}_J(X;\R)$ is the interior of the dual cone $\breve{H}\mathfrak{C}_{n-1,J} \subseteq H^{(n-1,n-1)}_J(X;\R)$ of $H\mathfrak{C}_{n-1,J} \subseteq H_{(n-1,n-1)}^J(X;\R)$; therefore $\mathcal{K}b^t_J\cap H^{(n-1,n-1)}_J(X;\R) = \mathcal{K}b^c_J$.

The inclusion $\mathcal{K}b^c_J+H^{(n,n-2),(n-2,n)}_J(X;\R) \subseteq \mathcal{K}b^t_J$ follows straightforwardly noting that the sum of a semi-K\"ahler form and a $J$-anti-invariant $(2n-2)$-form is still $\de$-closed and positive on the complex $(n-1)$-subspaces.

Finally, if $J$ is \Cf\ at the \kth{(2n-2)} stage, then
\begin{eqnarray*}
 \mathcal{K}b^t_J &=& \interno\breve{H}\mathfrak{C}_{n-1,J} \;=\; \interno\breve{H}\mathfrak{C}_{n-1,J} \cap H^{2n-2}_{dR}(X;\R) \\[5pt]
 &=&\interno\breve{H}\mathfrak{C}_{n-1,J} \cap \left(H^{(n-1,n-1)}_J(X;\R)+H^{(n,n-2),(n-2,n)}_J(X;\R)\right) \\[5pt]
 &\subseteq& \mathcal{K}b^c_J + H^{(n,n-2),(n-2,n)}_J(X;\R) \;,
\end{eqnarray*}
and hence $\mathcal{K}b^c_J+H^{(n,n-2),(n-2,n)}_J(X;\R) = \mathcal{K}b^t_J$.
\end{proof}

\begin{rem}
 We note that, while the de Rham cohomology class of an almost-K\"ahler metric cannot be trivial, the hypothesis $0\not\in\mathcal{K}b^t_J$ in Theorem \ref{thm:kbc-kbt} is not trivial: J. Fu, J. Li, and S.-T. Yau proved in \cite[Corollary 1.3]{fu-li-yau} that, for any $k\geq 2$, the connected sum $\left(\mathbb{S}^3\times\mathbb{S}^3\right)^{\sharp k}$, endowed with the complex structure constructed from the conifold transitions, admits balanced metrics.
\end{rem}

\chapter{Cohomology of manifolds with special structures}\label{chapt:special}

In this chapter, we continue in studying the cohomological properties of (differentiable) manifolds endowed with special structures, other than (almost-)complex structures. More precisely, in Section \ref{sec:sympl}, we recall the results obtained jointly with A. Tomassini in \cite{angella-tomassini-4}, concerning the cohomology of symplectic manifolds; in Section \ref{sec:paracomplex}, we study cohomological decompositions on \para-complex manifolds in the sense of F.~R. Harvey and H.~B. Lawson: this has been the matter of a joint work with F.~A. Rossi, \cite{angella-rossi}; finally, in Section \ref{sec:p-convex}, we consider domains in $\R^n$ endowed with a smooth proper strictly $p$-convex exhaustion function, and, using $\mathrm{L}^2$-techniques, we give another proof of a consequence of J.-P. Sha's theorem \cite[Theorem 1]{sha}, and H. Wu's theorem \cite[Theorem 1]{wu-indiana}, on the vanishing of the higher degree de Rham cohomology groups, which has been obtained in a joint work with S. Calamai,
\cite{angella-calamai}.

\section{Cohomology of symplectic manifolds}\label{sec:sympl}

The K\"ahler manifolds have special cohomological properties from both the complex and the symplectic point of view, the Hodge decomposition theorem providing a decomposition of the complex de Rham cohomology in terms of the Dolbeault cohomology groups, and the Lefschetz decomposition theorem providing a decomposition of the de Rham cohomology in terms of primitive cohomology groups. Then, in order to better understand the geometry of non-K\"ahler manifolds, it may be interesting to investigate both the contribution of the complex structure and the contribution of the symplectic structure.

In this section, we develop the symplectic counterpart of the theory introduced by T.-J. Li and W. Zhang in \cite{li-zhang} to study the cohomology of almost-complex manifolds. The results in this section have been obtained jointly with A. Tomassini in \cite{angella-tomassini-4}.

\subsection{Hodge theory on symplectic manifolds}\label{subsec:symplectic-hodge-theory}

Cohomological properties of symplectic manifolds have been studied starting from the works by J.-L. Koszul, \cite{koszul}, and by J.-L. Brylinski, \cite{brylinski}.
Drawing a parallel between the symplectic and the Riemannian cases, J.-L. Brylinski proposed in \cite{brylinski} a Hodge theory for compact symplectic manifolds $\left(X,\,\omega\right)$, introducing a symplectic Hodge-$\star$-operator $\star_\omega$ and the notion of \emph{$\omega$-symplectically-harmonic form} (i.e., a form being both $\de$-closed and $\de^\Lambda$-closed, where the symplectic co-differential is defined as $\de^\Lambda\lfloor_{\wedge^kX}:=(-1)^{k+1}\,\star_\omega\de\star_\omega$ for every $k\in\N$): in this context, O. Mathieu in \cite{mathieu}, and D. Yan in \cite{yan}, proved that any de Rham cohomology class admits an $\omega$-symplectically-harmonic representative if and only if the Hard Lefschetz Condition is satisfied.
Recently, L.-S. Tseng and S.-T. Yau, in \cite{tseng-yau-1, tseng-yau-2}, see also \cite{tseng-yau-3}, introduced new cohomologies for symplectic manifolds $\left(X,\,\omega\right)$: among them, in particular, they defined and studied
$$ H^\bullet_{\de+\de^\Lambda}(X;\R) \;:=\; \frac{\ker\left(\de+\de^\Lambda\right)}{\imm\de\de^\Lambda} \;,$$
developing a Hodge theory for this cohomology; furthermore, they studied the dual currents of Lagrangian and co-isotropic submanifolds, and they defined a homology theory on co-isotropic chains, which turns out to be naturally dual to a primitive cohomology. In the context of Generalized Geometry, \cite{gualtieri-phd, gualtieri, cavalcanti, cavalcanti-impa}, the cohomology $H^\bullet_{\de+\de^\Lambda}(X;\R)$ can be interpreted as the symplectic counterpart of the Bott-Chern cohomology of a complex manifold, see \cite{tseng-yau-3}. Inspired also by their works, Y. Lin developed in \cite{lin} a new approach to the symplectic Hodge theory, proving in particular that, on any compact symplectic manifold satisfying the Hard Lefschetz Condition, there is a Poincaré duality between the primitive homology on co-isotropic chains and the primitive cohomology.

In this section, we recall some notions and results concerning Hodge theory for compact symplectic manifolds; we refer to \cite{brylinski, mathieu, yan, cavalcanti, tseng-yau-1, tseng-yau-2, lin} for further details. (See \S\ref{sec:symplectic} for basic definitions and results on symplectic manifolds.)

\subsubsection{Symplectic cohomologies}

Let $X$ be a compact $2n$-dimensional manifold endowed with a symplectic structure $\omega$.

\medskip

We recall, see \S\ref{sec:symplectic}, that, $\omega$ being non-degenerate, it induces a natural isomorphism $I\colon TX\to T^*X$, namely, $I(\sspace)(\ssspace)=\omega(\sspace,\ssspace)$, and hence a bi-$\mathcal{C}^\infty(X;\R)$-linear form $\left(\omega^{-1}\right)^k\colon \wedge^k X\otimes \wedge^kX \to \mathcal{C}^\infty(X;\R)$; the \emph{symplectic-$\star$-operator}, is defined, for every $\alpha,\,\beta\in\wedge^kX$, by, \cite[\S2]{brylinski},
$$
\star_\omega\colon \wedge^\bullet X\to \wedge^{2n-\bullet}X \;, \qquad \alpha\wedge\star_\omega \beta \;=\; \left(\omega^{-1}\right)^k\left(\alpha,\beta\right)\,\frac{\omega^n}{n!} \;,
$$
and satisfies $\star_\omega^2=\id_{\wedge^\bullet X}$, \cite[Lemma 2.1.2]{brylinski}.

We recall that the operators
\begin{eqnarray*}
L &:=& \omega\wedge\sspace \colon \wedge^\bullet X\to \wedge^{\bullet+2}X \;, \\[5pt]
\Lambda &:=& -\iota_\Pi = -\star_\omega\,L\,\star_\omega \colon \wedge^\bullet X\to \wedge^{\bullet-2}X\;, \\[5pt]
H &:=& \sum_k \left(n-k\right)\,\pi_{\wedge^kX} \colon \wedge^\bullet X\to \wedge^\bullet X \;,
\end{eqnarray*}
(where $\Pi:=\omega^{-1}\in\wedge^2TX$ is the canonical Poisson bi-vector associated to $\omega$, the interior product with $\xi\in\wedge^2\left(TX\right)$ is denoted by $\iota_{\xi}\colon\wedge^{\bullet}X\to\wedge^{\bullet-2}X$, and, for $k\in\N$, the map $\pi_{\wedge^kX}\colon\wedge^\bullet X\to\wedge^kX$ denotes the natural projection onto $\wedge^kX$) yields an $\mathfrak{sl}(2;\R)$-representation on $\wedge^\bullet X$ having finite $H$-spectrum, and hence one has the \emph{Lefschetz decomposition} on differential forms, \cite[Corollary 2.6]{yan},
$$ \wedge^\bullet X \;=\; \bigoplus_{r\in\N} L^r \, \Prim^{\bullet-2r}X \;, $$
where the space of \emph{primitive forms} is
$$ \Prim^\bullet X \;:=\; \ker\Lambda \;=\; \ker L^{n-\bullet+1}\lfloor_{\wedge^\bullet X} \;. $$

\medskip

Consider now the \emph{symplectic co-differential operator} $\de^\Lambda\colon \wedge^\bullet X\to \wedge^{\bullet-1}X$, defined, for every $k\in \N$, by
$$
 \de^\Lambda\lfloor_{\wedge^kX} \;:=\; (-1)^{k+1}\star_\omega \de \star_\omega \;;
$$
it has been introduced, in general for a Poisson manifold, by J.-L. Koszul, \cite{koszul}, and studied also by J.-L. Brylinski, \cite[\S1.2]{brylinski}.
The basic symplectic identity
$$ \left[\de,\,\Lambda\right]\;=\;\de^\Lambda $$
holds, see, e.g., \cite[Corollary 1.3]{yan}; by the graded-Jacobi identity, it follows that $\left[\de,\, \de^\Lambda\right] = \left[\de,\, \left[\de,\, \Lambda\right] \right] = \left[\de,\, \left[\Lambda,\, \de\right]\right] - \left[\Lambda,\, \left[\de,\, \de\right]\right] = -\left[\de,\, \de^\Lambda\right]$, since $\left[\de,\, \de\right]=0$ and $\left[\Lambda,\, \de\right]=-\left[\de,\, \Lambda\right]$, and hence, \cite[page 265]{koszul}, \cite[Theorem 1.3.1]{brylinski},
$$ \de\de^\Lambda +\de^\Lambda\de \;=\; 0 \;.$$
Hence, interpreting, as in \cite{brylinski}, $\de^\Lambda$ as the symplectic counterpart of the Riemannian co-differential operator $\de^*$ associated to a Riemannian metric $g$ on $X$, then the symplectic counterpart of the Laplacian operator $\Delta:=\de\de^*+\de^*\de$ vanishes.

We recall that, if $\left(J,\,\omega,\,g\right)$ is an almost-K\"ahler structure on $X$, then the symplectic-$\star$-operator $\star_\omega$ and the Hodge-$*$-operator $*_g$ are related by
$$ \star_\omega \;=\; J\,*_g \;,$$
\cite[Theorem 2.4.1]{brylinski}, and hence $\de^\Lambda$ and $\de^c:=J^{-1}\,\de\,J$ are related by
$$ \de^\Lambda \;=\; -\left(\de^c\right)^{*} \;.$$
The previous identity, together with the identity $\de\de^\Lambda+\de^\Lambda\de=0$, suggests that $\de^\Lambda$ should be interpreted as the symplectic counterpart of the operator $\de^c$ in Complex Geometry; this guess can be made more precise using Complex Generalized Geometry, \cite{gualtieri-phd, gualtieri, cavalcanti, cavalcanti-impa}.

\medskip

The symplectic co-differential operator satisfies $\left(\de^\Lambda\right)^2=0$, and hence it gives rise to a differential complex $\left(\wedge^\bullet X, \, \de^\Lambda\right)$. This complex has been introduced, more in general, on a Poisson manifold, with the name of \emph{canonical complex}, by J.-L. Koszul, \cite{koszul}, and studied also by J.-L. Brylinski, \cite[\S1]{brylinski}, and, more recently, by L.-S. Tseng and S.-T. Yau, \cite[\S3.1]{tseng-yau-1}. The homology of the complex $\left(\wedge^\bullet X, \, \de^\Lambda\right)$ is, in J.-L. Koszul's terminology, the \emph{canonical homology} of $X$,
$$ H^\bullet_{\de^\Lambda}(X;\R) \;:=\; \frac{\ker \de^\Lambda}{\imm\de^\Lambda} \;. $$
Note that, \cite[Corollary 2.2.2]{brylinski},
$$ \star_\omega\colon H^\bullet_{dR}(X;\R) \stackrel{\simeq}{\to} H^{2n-\bullet}_{\de^\Lambda}(X;\R) \;,$$
hence, for a compact symplectic manifold, the canonical homology groups and the de Rham cohomology groups are isomorphic.

\medskip

In \cite{tseng-yau-1}, L.-S. Tseng and S.-T. Yau introduced also the \emph{$\left(\de+\de^\Lambda\right)$-cohomology}, \cite[\S3.2]{tseng-yau-1},
$$ H^\bullet_{\de+\de^\Lambda}(X;\R) \;:=\; \frac{\ker\left(\de+\de^\Lambda\right)}{\imm\de\de^\Lambda} \;, $$
and the \emph{$\left(\de\de^\Lambda\right)$-cohomology}, \cite[\S3.3]{tseng-yau-1},
$$ H^\bullet_{\de\de^\Lambda}(X;\R) \;:=\; \frac{\ker \de\de^\Lambda}{\imm\de+\imm\de^\Lambda} \;; $$
such cohomologies are, in a sense, the symplectic counterpart of the Aeppli and Bott-Chern cohomologies of complex manifolds, see \cite[\S5]{tseng-yau-1} and \cite{tseng-yau-3} for further discussions.

Furthermore, they provided a Hodge theory for such cohomologies, proving the following result.

\begin{thm}[{\cite[Theorem 3.5, Corollary 3.6]{tseng-yau-1}}]
 Let $X$ be a compact manifold endowed with a symplectic structure $\omega$. Let $\left(J,\, \omega,\, g\right)$ be an almost-K\"ahler structure on $X$. For a fixed $\lambda>0$, the \kth{4} order self-adjoint differential operator
 \begin{eqnarray*}
  D_{\de+\de^\Lambda} &:=& \left(\de\de^\Lambda\right) \left(\de\de^\Lambda\right)^* + \left(\de\de^\Lambda\right)^* \left(\de\de^\Lambda\right) + \left(\de^*\de^\Lambda\right) \left(\de^*\de^\Lambda\right)^* + \left(\de^*\de^\Lambda\right)^* \left(\de^*\de^\Lambda\right) \\[5pt]
  && + \lambda \, \left(\de^*\de+\left(\de^\Lambda\right)^*\de^\Lambda\right) \;.
 \end{eqnarray*}
 is elliptic, with $\ker D_{\de+\de^\Lambda} = \ker\de\cap\ker\de^\Lambda\cap\ker\left(\de\de^\Lambda\right)^*$.

 Furthermore, there exist an orthogonal decomposition
 $$ \wedge^\bullet X \;=\; \ker D_{\de+\de^\Lambda} \oplus \de\de^\Lambda \wedge^\bullet X \oplus \left(\de^*\wedge^{\bullet+1}X+\left(\de^\Lambda\right)^*\wedge^{\bullet-1}X\right) $$
 and an isomorphism
 $$ H^\bullet_{\de+\de^\Lambda} (X;\R) \;\simeq\; \ker D_{\de+\de^\Lambda} \;.$$

 In particular, $\dim_\R H^\bullet_{\de+\de^\Lambda}(X;\R)<+\infty$.
\end{thm}

An analogous statement holds for the $\left(\de\de^\Lambda\right)$-cohomology.

\begin{thm}[{\cite[Theorem 3.16, Corollary 3.17]{tseng-yau-1}}]
 Let $X$ be a compact manifold endowed with a symplectic structure $\omega$. Let $\left(J,\, \omega,\, g\right)$ be an almost-K\"ahler structure on $X$. For a fixed $\lambda>0$, the \kth{4} order self-adjoint differential operator
 \begin{eqnarray*}
  D_{\de\de^\Lambda} &:=& \left(\de\de^\Lambda\right) \left(\de\de^\Lambda\right)^* + \left(\de\de^\Lambda\right)^* \left(\de\de^\Lambda\right) + \left(\de\left(\de^\Lambda\right)^*\right) \left(\de\left(\de^\Lambda\right)^*\right)^* + \left(\de\left(\de^\Lambda\right)^*\right)^* \left(\de\left(\de^\Lambda\right)^*\right) \\[5pt]
  && + \lambda \, \left(\de\de^*+\de^\Lambda\left(\de^\Lambda\right)^*\right) \;.
 \end{eqnarray*}
 is elliptic, with $\ker D_{\de\de^\Lambda} = \ker\left(\de\de^\Lambda\right)\cap\ker\de^*\cap\ker\left(\de^\Lambda\right)^*$.

 Furthermore, there exist an orthogonal decomposition
 $$ \wedge^\bullet X \;=\; \ker D_{\de\de^\Lambda} \oplus \left(\de\wedge^{\bullet-1}X + \de^\Lambda\wedge^{\bullet+1}X \right) \oplus \left(\de\de^\Lambda\right)^* \wedge^\bullet X $$
 and an isomorphism
 $$ H^\bullet_{\de\de^\Lambda} (X;\R) \;\simeq\; \ker D_{\de\de^\Lambda} \;.$$

 In particular, $\dim_\R H^\bullet_{\de\de^\Lambda}(X;\R)<+\infty$.
\end{thm}

As for the Bott-Chern and the Aeppli cohomologies, the $\left(\de+\de^\Lambda\right)$-cohomology and the $\left(\de\de^\Lambda\right)$-cohomology groups turn out to be isomorphic by means of the Hodge-$*$-operator associated to any Riemannian metric being compatible with $\omega$.

\begin{thm}[{\cite[Lemma 3.23, Proposition 3.24, Corollary 3.25]{tseng-yau-1}}]
 Let $X$ be a $2n$-dimensional compact manifold endowed with a symplectic structure $\omega$. Let $\left(J,\, \omega,\, g\right)$ be an almost-K\"ahler structure on $X$. The operators $D_{\de+\de^\Lambda}$ and $D_{\de\de^\Lambda}$ satisfy
$$ *_g\, D_{\de+\de^\Lambda} \;=\; D_{\de\de^\Lambda} \, *_g \;,$$
and hence $*_g$ induces an isomorphism
$$ *_g \colon H^\bullet_{\de+\de^\Lambda}(X;\R) \stackrel{\simeq}{\to} H^{2n-\bullet}_{\de\de^\Lambda}(X;\R) \;.$$
\end{thm}

Moreover, the cohomology $H^\bullet_{\de+\de^\Lambda}(X;\R)$ is invariant under symplectomorphisms and Hamiltonian isotopies, \cite[Proposition 2.8]{tseng-yau-1}.

\medskip

One has the following commutation relations between the differential operators $\de$, $\de^\Lambda$, and $\de\de^\Lambda$, and the elements $L$, $\Lambda$, and $H$ of the $\mathfrak{sl}(2;\R)$-triple, see, e.g., \cite[Lemma 2.3]{tseng-yau-1}:
$$
\begin{array}{rclrclrcl}
\left[\de,\,L\right] &=& 0 \;,\qquad & \left[\de^\Lambda,\,L\right] &=& -\de \;,\qquad & \left[\de\de^\Lambda,\,L\right] &=& 0 \;, \\[5pt]
\left[\de,\,\Lambda\right] &=& \de^\Lambda \;,\qquad & \left[\de^\Lambda,\,\Lambda\right] &=& 0 \;,\qquad & \left[\de\de^\Lambda,\,\Lambda\right] &=& 0 \;,\\[5pt]
\left[\de,\,H\right] &=& \de \;,\qquad & \left[\de^\Lambda,\,H\right] &=& -\de^\Lambda \;,\qquad & \left[\de\de^\Lambda,\,H\right] &=& 0 \;.
\end{array}
$$
Hence, by setting
$$ \PH^\bullet_{\de+\de^\Lambda}(X;\R) \;:=\; \frac{\ker\de\cap\ker\de^\Lambda\cap\Prim^\bullet X}{\imm\de\de^\Lambda\cap\Prim^\bullet X} \;=\; \frac{\ker\de\cap\Prim^\bullet X}{\imm\de\de^\Lambda\lfloor_{\Prim^\bullet X}} $$
(where the second equality follows from \cite[Lemma 3.9]{tseng-yau-1}), one gets the following result.

\begin{thm}[{\cite[Theorem 3.11]{tseng-yau-1}}]
Let $X$ be a $2n$-dimensional compact manifold endowed with a symplectic structure $\omega$.
Then there exist a decomposition
$$ H^\bullet_{\de+\de^\Lambda}(X;\R) \;=\; \bigoplus_{r\in\N} L^r \, \PH^{\bullet-2r}_{\de+\de^\Lambda}(X;\R) $$
and, for every $k\in\N$, an isomorphism
$$ L^k\colon H^{n-k}_{\de+\de^\Lambda}(X;\R) \stackrel{\simeq}{\to} H^{n+k}_{\de+\de^\Lambda}(X;\R) \;,$$
\end{thm}

Analogously, by setting
$$ \PH^\bullet_{\de\de^\Lambda}(X;\R) \;:=\; \frac{\ker\de\de^\Lambda\cap\Prim^\bullet X}{\left(\imm\de+\imm\de^\Lambda\right)\cap\Prim^\bullet X} \;=\; \frac{\ker\de\de^\Lambda\cap\Prim^\bullet X}{\imm\left.\left(\de+LH^{-1}\de^\Lambda\right)\right\lfloor_{\Prim^{\bullet-1}X}+\imm\left.\de^\Lambda\right\lfloor_{\Prim^{\bullet+1}X}} $$
(where the second equality follows from \cite[Lemma 3.20]{tseng-yau-1}), one gets the following result.

\begin{thm}[{\cite[Theorem 3.21]{tseng-yau-1}}]
Let $X$ be a $2n$-dimensional compact manifold endowed with a symplectic structure $\omega$.
Then there exist a decomposition
$$ H^\bullet_{\de\de^\Lambda}(X;\R) \;=\; \bigoplus_{r\in\N} L^r \, \PH^{\bullet-2r}_{\de\de^\Lambda}(X;\R) $$
and, for every $k\in\N$, an isomorphism
$$ L^k\colon H^{n-k}_{\de\de^\Lambda}(X;\R) \stackrel{\simeq}{\to} H^{n+k}_{\de\de^\Lambda}(X;\R) \;,$$
\end{thm}

\medskip

The identity map induces the following natural maps in cohomology:
$$
\xymatrix{
  & H_{\de+\de^\Lambda}^{\bullet}(X;\R) \ar[ld]\ar[rd] &   \\
 H_{dR}^{\bullet}(X;\R) \ar[rd] &  & H_{\de^\Lambda}^{\bullet}(X;\R) \ar[ld] \\
  & H_{\de\de^\Lambda}^{\bullet}(X;\R) &
}
$$

Recall that a symplectic manifold is said to satisfy the \emph{$\de\de^\Lambda$-Lemma} if every $\de$-exact $\de^\Lambda$-closed form is $\de\de^\Lambda$-exact, \cite{deligne-griffiths-morgan-sullivan}, namely, if $H_{\de+\de^\Lambda}^{\bullet}(X;\R)\to H_{dR}^{\bullet}(X;\R)$ is injective.

\begin{rem}
 Note that
 $$ \ker \de^\Lambda \cap \imm \de \;=\; \imm\de\de^\Lambda \qquad \text{ if and only if } \qquad \ker \de \cap \imm \de^\Lambda \;=\; \imm\de\de^\Lambda \;.$$
 Indeed, since $\star_\omega^2=\id_{\wedge^\bullet X}$, \cite[Lemma 2.1.2]{brylinski}, and $\de\de^\Lambda+\de^\Lambda\de=0$, \cite[Theorem 1.3.1]{brylinski}, one has
 $$ \star_\omega\ker\de \;=\; \ker \de^\Lambda\;, \qquad \star_\omega\imm\de \;=\; \imm\de^\Lambda\;, \qquad \star_\omega\imm\de\de^\Lambda\;=\; \imm\de\de^\Lambda \;.$$
\end{rem}

Another cohomological property that can be defined on a $2n$-dimensional compact manifold $X$ endowed with a symplectic form $\omega$ is the \emph{Hard Lefschetz Condition}, that is,
\begin{equation}\label{eq:hlc2}
\tag{HLC}
 \text{for every } k\in\N\;, \qquad L^k\colon H^{n-k}_{dR}(X;\R) \stackrel{\simeq}{\to} H^{n+k}_{dR}(X;\R) \;.
\end{equation}

In fact, the following result relates the $\de\de^\Lambda$-Lemma, the Hard Lefschetz Condition, and the existence of $\omega$-symplectically harmonic representatives in any de Rham cohomology class.

\begin{thm}[{\cite[Corollary 2]{mathieu}, \cite[Theorem 0.1]{yan}, \cite[Proposition 1.4]{merkulov}, \cite{guillemin}, \cite[Theorem 5.4]{cavalcanti}}]
 Let $X$ be a compact manifold endowed with a symplectic structure $\omega$. The following conditions are equivalent:
 \begin{enumerate}[\itshape (i)]
  \item every de Rham cohomology class admits a representative being both $\de$-closed and $\de^\Lambda$-closed (i.e., Brylinski's conjecture \cite[Conjecture 2.2.7]{brylinski} holds on $X$);
  \item the Hard Lefschetz Condition holds on $X$;
  \item the natural homomorphism $H_{\de+\de^\Lambda}^{\bullet}(X;\R)\to H_{dR}^{\bullet}(X;\R)$ induced by the identity is actually an isomorphism;
  \item $X$ satisfies the $\de\de^\Lambda$-Lemma.
 \end{enumerate}
\end{thm}

Note that, by the Lefschetz decomposition theorem, the compact K\"ahler manifolds satisfy the Hard Lefschetz Condition; in other terms, note that, given a K\"ahler structure $\left(J,\, \omega,\, g\right)$ on a compact manifold $X$, one has $\star_\omega = J\,*_g$, \cite[Theorem 2.4.1]{brylinski}, and hence every de Rham cohomology class admits an $\omega$-symplectically-harmonic representative.

\subsubsection{Primitive currents}
Let $X$ be a $2n$-dimensional compact manifold endowed with a symplectic structure $\omega$. Denote by $\correnti_\bullet X:=:\correnti^{2n-\bullet}X$ the space of currents, and consider the de Rham homology $H_\bullet^{dR}(X;\R):=H^\bullet\left(\correnti_\bullet X,\, \de\right)$. (See \S\ref{sec:currents} for definitions and results concerning currents and de Rham homology.)

\medskip

Following \cite[Definition 5.1]{lin}, set, by duality,
\begin{eqnarray*}
L\colon \correnti_\bullet X\to \correnti_{\bullet-2}X\;, \quad && S \mapsto S\left(L\, \sspace\right) \;,\\[5pt]
\Lambda\colon \correnti_\bullet X\to \correnti_{\bullet+2}X\;, \quad && S \mapsto S\left(\Lambda\, \sspace\right) \;,\\[5pt]
H\colon \correnti_\bullet X\to \correnti_\bullet X\;, \quad && S \mapsto S\left(-H\, \sspace\right) \;;
\end{eqnarray*}
note that
$$ \left[L,\, H\right] \;=\; 2\, L \;, \qquad \left[\Lambda,\, H\right] \;=\; -2\, \Lambda \;, \qquad  \left[L,\, \Lambda\right] \;=\; H \;. $$

A current $S\in \correnti^k X$ is said \emph{primitive} if $\Lambda S=0$, equivalently, if $L^{n-k+1}S=0$, see, e.g., \cite[Proposition 5.3]{lin}; denote by $\Primcorrenti^\bullet X:=:\Primcorrenti_{2n-\bullet} X$ the space of primitive currents on $X$.

In \cite{lin}, Y. Lin proved the following result.

\begin{thm}[{\cite[Lemma 5.2, Proposition 5.3, Lemma 5.12]{lin}}]
 Let $X$ be a compact manifold endowed with a symplectic structure $\omega$. Then $\left\langle L,\, \Lambda,\, H\right\rangle$ gives an $\mathfrak{sl}(2;\R)$-module structure on $\correnti^\bullet X$. In particular, one has the Lefschetz decomposition on the space of currents,
 $$ \correnti^\bullet X \;=\; \bigoplus_{r\in\N} L^r\, \Primcorrenti^{\bullet-2r}X \;:=:\; \bigoplus_{r\in\N} L^r\, \Primcorrenti_{2n-\bullet+2r}X \;.$$

 Furthermore, the space of flat currents is an $\mathfrak{sl}(2;\R)$-submodule of the space of currents.
\end{thm}

Finally, if $j\colon Y\hookrightarrow X$ is a compact oriented submanifold of $X$ of codimension $k$ (possibly with non-empty boundary), then the {\em dual current} $[Y]\in\correnti_k X$ associated with $Y$ is defined, by setting, for every $\varphi\in\wedge^kX$,
$$
 \left[Y\right](\varphi) \;:=\; \int_Y j^*(\varphi) \:.
$$
If $Y$ is a closed oriented submanifold, then the dual current $\left[Y\right]$ is $\de$-closed. According to \cite[Lemma 4.1]{tseng-yau-1}, the dual current $\left[Y\right]$ is primitive if and only if $Y$ is co-isotropic.

\subsection{Symplectic subgroups of (co)homology}
In this section, we provide a symplectic counterpart to T.-J. Li and W. Zhang's theory on cohomology of almost-complex manifolds developed in \cite{li-zhang}.
More precisely, we define the subgroups $H^{(\bullet,\bullet)}_\omega(X;\R)$ of the de Rham cohomology $H^\bullet_{dR}(X;\R)$ of a symplectic manifold $\left(X,\, \omega\right)$, and, analogously, the subgroups $H_{(\bullet,\bullet)}^\omega(X;\R)$ of the de Rham homology $H_\bullet^{dR}(X;\R)$; then, we study some of their properties: in particular, we prove that, for every compact symplectic manifold, the decomposition $H^2_{dR}(X;\R)=H^{(1,0)}_{\omega}(X;\R)\oplus H^{(0,2)}_{\omega}(X;\R)$ holds, Theorem \ref{thm:sympl-decomp-H2}, which provides a symplectic counterpart of \cite[Theorem 2.3]{draghici-li-zhang}.

\medskip

Let $X$ be a $2n$-dimensional compact manifold endowed with a symplectic structure $\omega$. For any $r,s\in\N$, define
$$ H^{(r,s)}_\omega(X;\R) \;:=\; \left\{\left[L^r\,\beta^{(s)}\right]\in H^{2r+s}_{dR}(X;\R) \st \beta^{(s)}\in \Prim^sX \right\} \;\subseteq\; H^{2r+s}_{dR}(X;\R) \;.$$
Obviously, for every $k\in\N$, one has
$$ \sum_{2r+s=k} H^{(r,s)}_\omega(X;\R) \;\subseteq\; H^k_{dR}(X;\R) \;:$$
we are concerned with studying when the above inclusion is actually an equality, and when the sum is actually a direct sum.

\begin{rem}
We underline the relations between the above subgroups and the primitive cohomologies introduced by L.-S. Tseng and S.-T. Yau in \cite{tseng-yau-1}.
As regards L.-S. Tseng and S.-T. Yau's primitive $\left(\de+\de^\Lambda\right)$-cohomology $\PH^{\bullet}_{\de+\de^{\Lambda}}(X;\R)$, note that, for every $r,s\in\N$,
$$ \imm\left(L^r\,\PH^{s}_{\de+\de^{\Lambda}}(X;\R)\to H^\bullet_{dR}(X;\R)\right) \;=\; L^r\, H^{(0,s)}_\omega(X;\R) \;\subseteq\; H^{(r,s)}_\omega(X;\R) \;. $$
In \cite[\S4.1]{tseng-yau-1}, L.-S. Tseng and S.-T. Yau have introduced also the primitive cohomology groups
$$ \PH^s_{\de}(X;\R) \;:=\; \frac{\ker\de\cap\ker\de^\Lambda\cap\Prim^sX}{\imm\de\lfloor_{\Prim^{s-1}X\cap\ker\de^\Lambda}} \;, $$
where $s\in\N$, proving that the homology on co-isotropic chains is naturally dual to $\PH^{2n-\bullet}_{\de}(X;\R)$, see \cite[pages 40--41]{tseng-yau-1};
in \cite[Proposition 2.7]{lin}, Y. Lin proved that, if the Hard Lefschetz Condition holds on $X$, then
$$ H^{(0,\bullet)}_\omega(X;\R) \;=\; \PH^\bullet_{\de}(X;\R) \;. $$
\end{rem}

\begin{rem}
 In \cite{conti-tomassini}, D. Conti and A. Tomassini studied the notion of half-flat structure on a $6$-dimensional manifold $X$, see \cite{chiossi-salamon}. Namely, an \emph{$\mathrm{SU}(3)$-structure} $\left(\omega,\, \psi\right)$ on $X$, where $\omega$ is a non-degenerate real $2$-form and $\psi$ is a decomposable complex $3$-form such that $\psi \wedge \omega = 0$ and $\psi\wedge\bar\psi = -\frac{4\im}{3}\,\omega^ 3$, is called \emph{half-flat} if both $\omega\wedge\omega$ and $\Re \psi$ are $\de$-closed. Note in particular that, if $\left(\omega,\, \psi\right)$ is a symplectic half-flat structure on $X$, then $\left[\Re\psi\right]\in H^{(0,3)}_\omega(X;\R)$.
\end{rem}

\begin{rem}
A class of examples of compact symplectic manifolds $\left(X,\, \omega\right)$ satisfying the cohomology decomposition by means of the above subgroups $H^{\bullet,\bullet}_\omega(X;\R)$ (actually, satisfying an even stronger cohomology decomposition) is provided by the compact symplectic manifolds satisfying the $\de\de^\Lambda$-Lemma, equivalently, as already recalled, the Hard Lefschetz Condition, \cite[Proposition 1.4]{merkulov}, \cite{guillemin}, \cite[Theorem 5.4]{cavalcanti}.

More precisely, on a compact manifold $X$ endowed with a symplectic structure $\omega$, the following conditions are equivalent:
 \begin{itemize}
  \item $X$ satisfies the $\de\de^\Lambda$-Lemma;
  \item it holds the decomposition
    \begin{equation}\label{eq:dedelambda-lemma-dec}
      H^\bullet_{dR}(X;\R) \;=\; \bigoplus_{r\in\N} L^r\, H^{(0,\bullet-2r)}_\omega(X;\R) \;.
    \end{equation}
 \end{itemize}

Indeed, recall that the decomposition
$$ H^\bullet_{\de+\de^\Lambda}(X;\R) \;=\; \bigoplus_{r\in\N} L^r \, \PH^{\bullet-2r}_{\de+\de^\Lambda}(X;\R) $$
holds on any compact symplectic manifold, \cite[Theorem 3.11]{tseng-yau-1}; moreover, the $\de\de^\Lambda$-Lemma holds on a compact symplectic manifold if and only if the natural homomorphism
$$ H_{\de+\de^\Lambda}^{\bullet}(X;\R)\to H_{dR}^{\bullet}(X;\R) $$
induced by the identity is actually an isomorphism; recall also that
$$ \imm\left(L^r\,\PH^{s}_{\de+\de^{\Lambda}}(X;\R)\to H^\bullet_{dR}(X;\R)\right) = L^r\, H^{(0,s)}_\omega(X;\R) \;;$$
hence, if the $\de\de^\Lambda$-Lemma holds, then one has the decomposition \eqref{eq:dedelambda-lemma-dec}.
Conversely, if \eqref{eq:dedelambda-lemma-dec} holds, then, straightforwardly, $X$ satisfies the Hard Lefschetz condition, and hence also the $\de\de^\Lambda$-Lemma, \cite[Proposition 1.4]{merkulov}, \cite{guillemin}, \cite[Theorem 5.4]{cavalcanti}.
\end{rem}

\medskip

Analogously, considering the space $\correnti^\bullet X:=:\correnti_{2n-\bullet} X$ of currents and the de Rham homology $H_\bullet^{dR}(X;\R)$, for every $r,s\in\N$, define
$$ H_{(r,s)}^\omega(X;\R) \;:=\; \left\{\left[L^r\, B_{(s)}\right]\in H_{-2r+s}^{dR}(X;\R) \st B_{(s)}\in \Primcorrenti_sX \right\} \;\subseteq\; H_{-2r+s}^{dR}(X;\R) \;;$$
as previously, for every $k\in\N$, we have just the inclusion
$$ \sum_{-2r+s=k} H_{(r,s)}^\omega(X;\R) \;\subseteq\; H_k^{dR}(X;\R) \;,$$
but, in general, neither the sum is direct nor the inclusion is an equality.

\medskip

We prove that, fixed $k\in\N$, if the sum $\sum_{2r+s=2n-k}H^{(r,s)}_\omega(X;\R)$ gives the whole \kth{\left(2n-k\right)} de Rham cohomology group, then the sum of the subgroups of the \kth{k} de Rham cohomology group is direct, \cite[Proposition 2.4]{angella-tomassini-4} (this result should be compared with \cite[Proposition 2.30]{li-zhang} and Theorem \ref{thm:implicazioni} in the almost-complex case, and with Proposition \ref{prop:para-cpf-pf} in the \para-complex case).

\begin{prop}
\label{prop:full-k-pure-2n-k}
Let $X$ be a $2n$-dimensional compact manifold endowed with a symplectic structure $\omega$.
For every $k\in\N$, the following implications hold:
$$
\xymatrix{
 H^{k}_{dR}(X;\R) = \sum_{2r+s=k}H^{(r,s)}_\omega(X;\R) \ar@{=>}[r] \ar@{=>}[d] & \bigoplus_{-2r+s=k} H_{(r,s)}^\omega(X;\R) \;\subseteq\; H_k^{dR}(X;\R) \ar@{=>}[d] \\
 H_{2n-k}^{dR}(X;\R) = \sum_{-2r+s=2n-k}H_{(r,s)}^\omega(X;\R) \ar@{=>}[r] & \bigoplus_{2r+s=2n-k} H^{(r,s)}_\omega(X;\R) \;\subseteq\; H^{2n-k}_{dR}(X;\R) \;. \\
}
$$
\end{prop}

\begin{proof}
 Note that the quasi-isomorphism $T_\sspace\colon \wedge^\bullet X \ni \varphi \mapsto \int_X \varphi \wedge \sspace \in \correnti^\bullet X$ satisfies
 $$ T_{L\, \sspace} \;=\; L T_\sspace \;, $$
 and hence, in particular, it preserves the bi-graduation,
 $$ T \left(L^{\bullet_1}\, \Prim^{\bullet_2}X\right) \;\subseteq\; L^{\bullet_1}\, \Primcorrenti^{\bullet_2}X \;:=:\; L^{\bullet_1}\ \Primcorrenti_{2n-\bullet_2}X \;, $$
 and it induces, for every $r,s\in\N$, an injective map
 $$ H^{(r,s)}_\omega(X;\R) \hookrightarrow H_{(r,2n-s)}^\omega(X;\R) \;. $$
 Therefore the two vertical implications are proven.

 Consider now the non-degenerate duality pairing
 $$ \left\langle \sspace,\,\ssspace\right\rangle \colon H^{\bullet}_{dR}(X;\R) \times H_{\bullet}^{dR}(X;\R) \to \R\;, $$
 and note that, for every $r,s\in\N$,
 $$ \ker\left\langle H^{(r,s)}_{\omega}(X;\R),\;\sspace\right\rangle \;\supseteq\; \sum_{(p,q)\neq(n-r-s,2n-s)} H_{(p,q)}^{\omega}(X;\R) \;,$$
 and, analogously, for every $p,q\in\N$,
 $$ \ker\left\langle \sspace ,\; H_{(p,q)}^{\omega}(X;\R) \right\rangle \;\supseteq\; \sum_{(r,s)\neq(n-p-s,2n-q)} H^{(r,s)}_{\omega}(X;\R) \;;$$
 this suffices to prove the two horizontal implications.
\end{proof}

A straightforward consequence of \cite[Corollary 2]{mathieu}, or \cite[Theorem 0.1]{yan}, and Proposition \ref{prop:full-k-pure-2n-k} is the following result, \cite[Corollary 2.5]{angella-tomassini-4}, which should be compared with \cite[Theorem 2.16, Proposition 2.17]{draghici-li-zhang}.

\begin{cor}
\label{cor:sympl-dedelambda-dec}
 Let $X$ be a compact manifold endowed with a symplectic structure $\omega$. Suppose that the Hard Lefschetz Condition holds on $X$, equivalently, that $X$ satisfies the $\de\de^\Lambda$-Lemma. Then
 $$
    H^\bullet_{dR}(X;\R) \;=\; \bigoplus_{r\in\N} H^{(r,\bullet-2r)}_\omega(X;\R)
     \qquad \text{ and } \qquad
    H_\bullet^{dR}(X;\R) \;=\; \bigoplus_{r\in\N} H_{(r,\bullet+2r)}^\omega(X;\R) \;.
 $$
\end{cor}

In particular, when $\dim X=4$ and taking $k=2$ in Proposition \ref{prop:full-k-pure-2n-k}, one gets that, if $H^{2}_{dR}(X;\R)=H^{(1,0)}_\omega(X;\R) + H^{(0,2)}_\omega(X;\R)$ holds, then actually $H^2_{dR}(X;\R) = H^{(1,0)}_\omega(X;\R) \oplus H^{(0,2)}_\omega(X;\R)$ holds. In fact, the following result states that $H^2_{dR}(X;\R)$ always decomposes as direct sum of $H^{(1,0)}_\omega(X;\R)$ and $H^{(0,2)}_\omega(X;\R)$, also in dimension higher than $4$, \cite[Theorem 2.6]{angella-tomassini-4}: this gives a symplectic counterpart to T. Dr\v{a}ghici, T.-J. Li and W. Zhang's decomposition theorem \cite[Theorem 2.3]{draghici-li-zhang} in the complex setting, in fact, without the restriction to dimension $4$.

\begin{thm}
\label{thm:sympl-decomp-H2}
 Let $X$ be a compact manifold endowed with a symplectic structure $\omega$. Then
 \begin{eqnarray*}
  H^2_{dR}(X;\R) &=& H^{(1,0)}_\omega(X;\R) \oplus H^{(0,2)}_\omega(X;\R) \;.
 \end{eqnarray*}

 In particular, if $\dim X=4$, then
 $$ 
  H^{\bullet}_{dR}(X;\R) \;=\; \bigoplus_{r\in\N} H^{(r,\bullet-2r)}_\omega(X;\R)
   \qquad \text{ and } \qquad
  H_{\bullet}^{dR}(X;\R) \;=\; \bigoplus_{r\in\N} H_{(r,\bullet+2r)}^\omega(X;\R) \;.$$
\end{thm}

\begin{proof}
 Let $2n:=\dim X$.
 Firstly, we prove that $H^{(1,0)}_\omega(X;\R) \cap H^{(0,2)}_{\omega}(X;\R)=\{0\}$. Let
 $$
 \mathfrak{c}:=:\left[f\,\omega\right]:=:\left[\beta^{(2)}\right]\in H^{(1,0)}_\omega(X;\R) \cap H^{(0,2)}_{\omega}(X;\R)\,,
 $$ where $f\in\mathcal{C}^\infty(X;\R)$ and $\beta^{(2)}\in \Prim^{2}X$. Being $\Prim^{2}X=\ker L^{n-1}\lfloor_{\wedge^{2}X}$, one has
 $$ 0 \;=\; \int_X f\,L^{n-1} \beta^{(2)} \;=\; \int_X f\,\omega \wedge \beta^{(2)} \wedge \omega^{n-2} \;=\; \int_X f\,\omega \wedge f\,\omega \wedge \omega^{n-2} \;=\; \int_X f^2 \omega^n $$
 hence $f=0$, that is, $\mathfrak{c}=0$.

 Now, we prove that $H^{2}_{dR}(X;\R)=H^{(1,0)}_{\omega}(X;\R)+H^{(0,2)}_{\omega}(X;\R)$. Let $\mathfrak{a}:=:\left[\alpha\right]\in H^2_{dR}(X;\R)$. Then $L^{n-1}\mathfrak{a}\in H^{2n}_{dR}(X;\R) = \R\left\langle \left[\omega^n\right]\right\rangle$, that is, there exist $\lambda\in\R$ and $\gamma_{2n-1}\in\wedge^{2n-1}X$ such that $L^{n-1}\alpha=\lambda\,\omega^n+\de\gamma_{2n-1}$. Since $L^{n-1}\colon \wedge^1X\stackrel{\simeq}{\to} \wedge^{2n-1}X$ is an isomorphism, there exists $\gamma_1\in\wedge^1X$ such that $L^{n-1}\gamma_1=\gamma_{2n-1}$. Hence, since $\left[\de,\, L^{n-1}\right]=0$, we get that $L^{n-1}\left(\alpha-\de\gamma_1-\lambda\,\omega\right)=0$, that is, $\alpha-\de\gamma_1-\lambda\,\omega\in\Prim^2X$; therefore we get that
 $$ \mathfrak{a} \;:=:\; \left[\alpha\right] \;=\; \left[\alpha-\de\gamma_1\right] \;=\; \underbrace{\lambda\,\left[\omega\right]}_{\in \, H^{(1,0)}_\omega(X;\R)} \,+\, \underbrace{\left[\alpha-\de\gamma_1-\lambda\,\omega\right]}_{\in \, H^{(0,2)}_\omega(X;\R)} \;,$$
 concluding the proof.
\end{proof}

Part of the argument in the proof of Theorem \ref{thm:sympl-decomp-H2} can be generalized to prove the following result, \cite[Remark 2.7]{angella-tomassini-4}.

\begin{prop}
Let $X$ be a $2n$-dimensional compact manifold endowed with a symplectic structure $\omega$. For every $k\in\left\{1,\ldots,\left\lfloor\frac{n}{2}\right\rfloor\right\}$, it holds
$$ H^{(k,0)}_\omega(X;\R) \cap H^{(0,2k)}_{\omega}(X;\R) \;=\; \left\{0\right\} \;.$$
\end{prop}

\begin{proof}
 Let $\mathfrak{c}:=:\left[f\,\omega^{k}\right]:=:\left[\beta^{(2k)}\right]\in H^{(k,0)}_\omega(X;\R) \cap H^{(0,2k)}_{\omega}(X;\R)$, where $f\in\mathcal{C}^\infty(X;\R)$ and $\beta^{(2k)}\in \Prim^{2k}X$. Being $\Prim^{2k}X=\ker L^{n-2k+1}\lfloor_{\wedge^{2k}X}$, one has
 $$
  0 \;=\; \int_X f\,L^{n-2k+1} \beta^{(2k)} \wedge \omega^{k-1} \;=\; \int_X f\,\omega^k \wedge \beta^{(2k)} \wedge \omega^{n-2k}
    \;=\; \int_X f\,\omega^k \wedge f\,\omega^k \wedge \omega^{n-2k} \;=\; \int_X f^2 \omega^n
 $$
 hence $f=0$, that is, $\mathfrak{c}=0$.
\end{proof}

\medskip

In some cases, in studying $H^{(r,s)}_\omega(X;\R)$, one can reduce to study $H^{(0,s)}_\omega(X;\R)$: this is the matter of the following result, \cite[Proposition 2.8]{angella-tomassini-4}.

\begin{prop}
 Let $X$ be a $2n$-dimensional compact manifold endowed with a symplectic structure $\omega$. Then, for every $r,s\in\N$ such that $2r+s\leq n$, one has
$$ H^{(r,s)}_\omega(X;\R) \;=\; L^r H^{(0,s)}_\omega(X;\R) \;.$$
\end{prop}

\begin{proof}
 Since $L\colon\wedge^{j}X\to\wedge^{j+2}X$ is injective for $j\leq n-1$, \cite[Corollary 2.8]{yan}, (in fact, an isomorphism for $j=n-1$, \cite[Corollary 2.7]{yan},) and $\left[\de,\,L\right]=0$, we get that
 \begin{eqnarray*}
  H^{(r,s)}_{\omega}(X;\R) &=& \left\{ \left[\omega^r\,\beta^{(s)}\right]\in H^{2r+s}_{dR}(X;\R) \st \beta^{(s)}\in\wedge^{s}X\cap\ker\Lambda \text{ such that }L^r\de \beta^{(s)}=0\right\} \\[5pt]
  &=& \left\{ \left[\omega^r\right]\smile\left[\beta^{(s)}\right]\in H^{2r+2}_{dR}(X;\R) \st \beta^{(s)}\in\wedge^sX\cap\ker\Lambda \right\} \,,
 \end{eqnarray*}
 assumed that $2r+s\leq n$.
\end{proof}

In particular, for every $r\in\left\{1,\ldots,\left\lfloor\frac{n}{2}\right\rfloor\right\}$, the spaces $H^{(r,0)}_\omega(X;\R)$ are $1$-dimensional $\R$-vector spaces, more precisely, $H^{(r,0)}_{\omega}(X;\R)=\R\left\langle \left[\omega^r\right] \right\rangle$.

Furthermore, by the previous proposition, it follows that, for $k\leq \frac{1}{2}\,\dim X$, the condition 
$$
H^k_{dR}(X;\R) = \bigoplus_{r\in\N}H^{(r,k-2r)}_\omega(X;\R)
$$ 
is in fact equivalent to $H^k_{dR}(X;\R) = \bigoplus_{r\in\N} L^r\, H^{(0,k-2r)}_\omega(X;\R)$.

\subsection{Symplectic cohomological decomposition on solvmanifolds}\label{subsec:symplectic-solvmfds}
As shown in Corollary \ref{cor:sympl-dedelambda-dec}, whenever $X$ is a compact manifold endowed with a symplectic structure $\omega$ satisfying the Hard Lefschetz Condition, the de Rham cohomology $H^\bullet_{dR}(X;\R)$, respectively the de Rham homology $H_\bullet^{dR}(X;\R)$, decomposes as direct sum of the subgroups $H^{(\bullet,\bullet)}_{\omega}(X;\R)$, respectively $H_{(\bullet,\bullet)}^{\omega}(X;\R)$. Hence, it should be interesting to study cohomological properties for classes of symplectic manifolds not satisfying the Hard Lefschetz property, e.g., non-tori nilmanifolds, \cite[Theorem A]{benson-gordon-nilmanifolds}.

In this section, we study a Nomizu-type theorem for the subgroups $H^{(\bullet,\bullet)}_{\omega}(X;\R)$ on completely-solvable solvmanifolds endowed with left-invariant symplectic structures, Proposition \ref{prop:linear-cpf-invariant-cpf-omega}, providing explicit examples and studying their cohomological properties.
(As regards notations, definitions, and results concerning nilmanifolds and solvmanifolds, we refer to \S\ref{sec:solvmanifolds}.)

\subsubsection{Left-invariant symplectic structures on solvmanifolds}

Let $X=\left.\Gamma\right\backslash G$ be a completely-solvable solvmanifold endowed with a $G$-left-invariant symplectic structure $\omega$.

Recall that, by A. Hattori's theorem \cite[Corollary 4.2]{hattori}, the complex $\left(\wedge^\bullet\duale{\mathfrak{g}},\, \de\right)$, which is isomorphic to the sub-complex composed of the $G$-left-invariant forms of $\left(\wedge^\bullet X,\, \de\right)$, turns out to be quasi-isomorphic to the de Rham complex $\left(\wedge^\bullet X,\, \de\right)$, that is, $H_{dR}^\bullet\left(\mathfrak{g};\R\right) \simeq H^\bullet_{dR}(X;\R)$.

Since $\omega$ is $G$-left-invariant, $\left\langle L,\, \Lambda,\, H\right\rangle$ induces a $\mathfrak{sl}(2;\R)$-representation both on $\wedge^\bullet X$ and on
$\wedge^\bullet\duale{\mathfrak{g}}$. Hence, for any $r,s\in\N$, we can consider both the subgroup $H^{(r,s)}_\omega(X;\R)$ of $H^\bullet_{dR}(X;\R)$, and the subgroup
$$ H^{(r,s)}_\omega(\mathfrak{g};\R) \;:=\; \left\{ \left[L^r\,\beta^{(s)}\right]\in H^{\bullet}_{dR}\left(\mathfrak{g};\R\right) \st \Lambda \beta^{(s)}=0 \right\} $$
of $H^{\bullet}_{dR}\left(\mathfrak{g};\R\right) \simeq H^\bullet_{dR}(X;\R)$, namely, the subgroup constituted of the de Rham cohomology classes admitting $G$-left-invariant representatives in $L^r \, \Prim^sX$.

In this section, we are concerned with studying the linking between $H^{(\bullet,\bullet)}_\omega(X;\R)$ and $H^{(\bullet,\bullet)}_\omega(\mathfrak{g};\R)$. This will let us study explicit examples in \S\ref{subsubsec:ex-sympl-cohom}.

\medskip

In the following lemma, we adapt the F.~A. Belgun symmetrization trick, \cite[Theorem 7]{belgun}, to the symplectic case, \cite[Lemma 3.2]{angella-tomassini-4}.

\begin{lem}
\label{lemma:belgun-sympl}
Let $X=\left.\Gamma\right\backslash G$ be a solvmanifold, and denote the Lie algebra naturally associated to $G$ by $\mathfrak{g}$. Let $\omega$ be a $G$-left-invariant symplectic structure on $X$.
Let $\eta$ be a $G$-bi-invariant volume form on $G$, given by J. Milnor's Lemma \cite[Lemma 6.2]{milnor}, such that $\int_X\eta=1$. Up to identifying $G$-left-invariant forms on $X$ and linear forms over $\duale{\mathfrak{g}}$ through left-translations, consider the F.~A. Belgun symmetrization map, \cite[Theorem 7]{belgun},
$$ \mu\colon \wedge^\bullet X \to \wedge^\bullet \duale{\mathfrak{g}}\;,\qquad \mu(\alpha)\;:=\;\int_X \alpha\lfloor_m \, \eta(m) \;.$$
One has that
$$ \mu\lfloor_{\wedge^\bullet \duale{\mathfrak{g}}}\;=\;\id\lfloor_{\wedge^\bullet \duale{\mathfrak{g}}} $$
and that
$$ \de\left(\mu(\sspace)\right) \;=\; \mu\left(\de\sspace\right) \qquad \text{ and } \qquad L\left(\mu(\sspace)\right) \;=\; \mu\left(L\sspace\right) \;.$$
In particular, $\mu$ sends primitive forms to $G$-left-invariant primitive forms.
\end{lem}

\begin{proof}
 It has to be shown just that $\mu\left(L\,\alpha\right)=L\,\mu\left(\alpha\right)$ for every $\alpha\in\wedge^\bullet X$. Note that, $\omega$ being a $G$-left-invariant form, one has $\mu\left(L\,\alpha\right)=\int_X \left(\omega\wedge\alpha\right)\lfloor_m \, \eta(m) =\int_X \omega\lfloor_m \wedge \alpha\lfloor_m \, \eta(m) = \omega \wedge \int_X \alpha\lfloor_m \, \eta(m) = L\,\mu\left(\alpha\right)$, for every $\alpha\in\wedge^\bullet X$.
\end{proof}

As a consequence of the previous lemma, we can prove the following result, which relates the subgroups $H^{(r,s)}_\omega(X;\R)$ with their $G$-left-invariant part $H^{(r,s)}_\omega(\mathfrak{g};\R)$, \cite[Proposition 3.3]{angella-tomassini-4} (compare with Proposition \ref{prop:linear-cpf-invariant-cpf-J}, and also with \cite[Theorem 3.4]{fino-tomassini}, for almost-complex structures, and with Proposition \ref{prop:linear-cpf-invariant-cpf-K} for almost-$\mathbf{D}$-complex structures in the sense of F.~R. Harvey and H.~B. Lawson).

\begin{prop}
\label{prop:linear-cpf-invariant-cpf-omega}
Let $X=\left.\Gamma\right\backslash G$ be a solvmanifold endowed with a $G$-left-invariant symplectic structure $\omega$, and denote the Lie algebra naturally associated to $G$ by $\mathfrak{g}$.
For every $r,s\in\N$, the map
$$ j\colon H^{(r,s)}_\omega(\mathfrak{g};\R) \to H^{(r,s)}_\omega(X;\R) $$
induced by left-translations is injective, and, if $H_{dR}^\bullet\left(\mathfrak{g};\R\right) \simeq H^\bullet_{dR}(X;\R)$ (for instance, if $X$ is a completely-solvable solvmanifold), then it is in fact an isomorphism.
\end{prop}

\begin{proof}
Left-translations induce the map $j\colon H^{(r,s)}_\omega(\mathfrak{g};\R) \to H^{(r,s)}_\omega(X;\R)$.
Consider the F.~A. Belgun's symmetrization map $\mu\colon\wedge^\bullet X \to \wedge^\bullet \duale{\mathfrak{g}}$: since it commutes with $\de$ by \cite[Theorem 7]{belgun}, it induces the map $\mu\colon H^\bullet_{dR}(X;\R) \to H_{dR}^\bullet\left(\mathfrak{g};\R\right)$, and, since it commutes with $L$ by Lemma \ref{lemma:belgun-sympl}, it induces the map $\mu\colon H^{(r,s)}_\omega(X;\R) \to H^{(r,s)}_\omega(\mathfrak{g};\R)$. Moreover, since $\mu$ is the identity on the space of $G$-left-invariant forms, we get the commutative diagram
$$
\xymatrix{
H^{(r,s)}_{\omega}(\mathfrak{g};\R) \ar[r]^{j} \ar@/_1.5pc/[rr]_{\id} & H^{(r,s)}_{\omega}(X;\R) \ar[r]^{\mu} & H^{(r,s)}_{\omega}(\mathfrak{g};\R) \;.
}
$$
Hence $j\colon H^{(r,s)}_\omega(\mathfrak{g};\R) \to H^{(r,s)}_\omega(X;\R)$ is injective, and $\mu\colon H^{(r,s)}_\omega(X;\R) \to H^{(r,s)}_\omega(\mathfrak{g};\R)$ is surjective.

Furthermore, when $H_{dR}^\bullet\left(\mathfrak{g};\R\right) \simeq H^\bullet_{dR}(X;\R)$ (for instance, when $X$ is a completely-solvable solvmanifold, by A. Hattori's theorem \cite[Theorem 4.2]{hattori}), since $\mu\lfloor_{\wedge^\bullet\duale{\mathfrak{g}}} = \id\lfloor_{\wedge^\bullet\duale{\mathfrak{g}}}$ by \cite[Theorem 7]{belgun}, we get that the map $\mu\colon H^\bullet_{dR}(X;\R) \to H_{dR}^\bullet\left(\mathfrak{g};\R\right)$ is the identity map, and hence $\mu\colon H^{(r,s)}_\omega(X;\R) \to H^{(r,s)}_\omega(\mathfrak{g};\R)$ is also injective, hence an isomorphism.
\end{proof}

\subsubsection{Symplectic (co)homology decomposition on solvmanifolds}\label{subsubsec:ex-sympl-cohom}
Proposition \ref{prop:linear-cpf-invariant-cpf-omega} is a useful tool to study explicit examples, \cite[Example 3.4, Example 3.5, Example 3.6]{angella-tomassini-4}.

\begin{ex}{\itshape A $6$-dimensional symplectic nilmanifold such that $H^{(0,3)}_\omega(X;\R)+H^{(1,1)}_\omega(X;\R)\subsetneq H^3_{dR}(X;\R)$ and $H^{(0,3)}_\omega(X;\R)\cap H^{(1,1)}_\omega(X;\R)\neq\{0\}$.}\\
Take a $6$-dimensional nilmanifold
$$ X \;=\; \left.\Gamma \right\backslash G \;:=\; \left(0^3,\; 12,\; 14-23,\; 15+34 \right) $$
endowed with the $G$-left-invariant symplectic structure
$$ \omega \;:=\; e^{16}+e^{35}+e^{24} \;.$$

By K. Nomizu's theorem \cite[Theorem 1]{nomizu}, one computes
\begin{eqnarray*}
H^1_{dR}(X;\R) &=& \underbrace{\R\left\langle e^1,\; e^2,\; e^3\right\rangle}_{=H^{(0,1)}_\omega(X;\R)} \;,\\[5pt]
H^2_{dR}(X;\R) &=& \underbrace{\R\left\langle e^{16}+e^{35}+e^{24} \right\rangle}_{=H^{(1,0)}_\omega(X;\R)} \oplus \underbrace{\R\left\langle e^{13},\; e^{14}+e^{23},\; 2\cdot e^{24}-e^{16}-e^{35}\right\rangle}_{=H^{(0,2)}_\omega(X;\R)} \;, \\[5pt]
H^3_{dR}(X;\R) &=& \R\left\langle e^{126}-e^{145}-2\cdot e^{235},\; e^{136},\; e^{146}+\frac{1}{2}\cdot e^{236}+\frac{1}{2}\cdot e^{345},\; e^{245}\right\rangle
\end{eqnarray*}
(where, as usual, we have listed the harmonic representatives with respect to the $G$-left-invariant metric $\sum_{j=1}^{6}e^j\odot e^j$ instead of their classes, and we have shortened, for example, $e^{hk}:=e^h\wedge e^k$).

Since the Lefschetz decompositions of the $g$-harmonic representatives of $H^3_{dR}(X;\R)$ are
\begin{eqnarray*}
 e^{126}-e^{145}-2\cdot e^{235} &=& \underbrace{\left(-\frac{1}{2}\cdot e^{126}-\frac{1}{2}\cdot e^{235}-e^{145}\right)}_{\in \Prim^3X} + \underbrace{\left(\frac{3}{2}\cdot e^{126}-\frac{3}{2}\cdot e^{235}\right)}_{=\; L\left(-\frac{3}{2}\cdot e^2\right)} \;, \\[5pt]
 e^{136} &=& \underbrace{\left(\frac{1}{2}\cdot e^{136}-\frac{1}{2}\cdot e^{234}\right)}_{\in \Prim^3X} + \underbrace{\left(\frac{1}{2}\cdot e^{136}+\frac{1}{2}\cdot e^{234}\right)}_{=\; L\left(-\frac{1}{2}\cdot e^3\right)} \;, \\[5pt]
 e^{146}+\frac{1}{2}\cdot e^{236}+\frac{1}{2}\cdot e^{345} &=& \underbrace{\left(\frac{1}{4}\cdot e^{146}-\frac{1}{4}\cdot e^{345}+\frac{1}{2}\cdot e^{236}\right)}_{\in \Prim^3X} + \underbrace{\left(\frac{3}{4}\cdot e^{146}+\frac{3}{4}\cdot e^{345}\right)}_{=\; L\left(-\frac{3}{4}\cdot e^4\right)} \;, \\[5pt]
 e^{245} &=& \underbrace{\left(\frac{1}{2}\cdot e^{156}+\frac{1}{2}\cdot e^{245}\right)}_{\in \Prim^3X} + \underbrace{\left(-\frac{1}{2}\cdot e^{156}+\frac{1}{2}\cdot e^{245}\right)}_{=\; L\left(\frac{1}{2}\cdot e^5\right)} \;,
\end{eqnarray*}
and since
$$ \de\wedge^2\duale{\mathfrak{g}} \;=\; \R\left\langle e^{123},\, e^{124},\; e^{125},\; e^{126}+e^{145},\; e^{134},\; e^{135},\; e^{146}-e^{236}-e^{345},\; e^{234} \right\rangle \;, $$
we get that
$$
 \left[e^{126}-e^{145}-2\cdot e^{235}\right] \;=\; \left[e^{126}-e^{145}-2\cdot e^{235}+\de e^{46}\right] \;=\; \left[2\cdot e^{126}-2\cdot e^{235}\right] \;=\; \left[L \left(-2\cdot e^2\right)\right] \;\in\; H^{(1,1)}_\omega(X;\R)
$$
and
\begin{eqnarray*}
 \left[e^{136}\right] &=& \left[e^{136}+\de\left(\frac{1}{2}\cdot e^{45}-\frac{1}{2}\cdot e^{26}\right)\right] \;=\; \left[e^{136} + e^{234}\right] \;=\; \left[L\left(-e^3\right)\right] \;\in\; H^{(1,1)}_\omega(X;\R) \;, \\[5pt]
 \left[e^{136}\right] &=& \left[e^{136}-\de\left(\frac{1}{2}\cdot e^{45}-\frac{1}{2}\cdot e^{26}\right)\right] \;=\; \left[e^{136} - e^{234}\right] \;\in\; H^{(0,3)}_\omega(X;\R) \;,
\end{eqnarray*}
while it is straightforward to check that
$$ \R\left\langle \left[ e^{146}+\frac{1}{2}\cdot e^{236}+\frac{1}{2}\cdot e^{345} \right],\; \left[ e^{245} \right] \right\rangle \cap \left(H^{(0,3)}_\omega(X;\R)+H^{(1,1)}_\omega(X;\R)\right) \;=\; \left\{0\right\} \;; $$
in particular, $H^{(0,3)}_\omega(X;\R)+H^{(1,1)}_\omega(X;\R)\subsetneq H^3_{dR}(X;\R)$ and $H^{(0,3)}_\omega(X;\R)\cap H^{(1,1)}_\omega(X;\R)\neq\{0\}$.
\end{ex}

\begin{ex}{\itshape A $6$-dimensional symplectic solvmanifold satisfying the decomposition
$$ H_{dR}^\bullet\left(X;\R\right) = \bigoplus_{r\in\N} L^r\, H^{(0,\bullet-2r)}_\omega\left(X;\R\right) \;. $$}
Take the $6$-dimensional solvable Lie algebra
$$ \mathfrak{g}_{3.4}^{-1}\oplus\mathfrak{g}_{3.5}^0 \;:=\; \left(-13,\; 23,\; 0,\; -56,\; 46,\; 0 \right) $$
endowed with the linear symplectic structure
$$ \omega \;:=\; e^{12}+e^{36}+e^{45} \;.$$
The corresponding connected simply-connected Lie group $G_{3.4}^{-1}\times G_{3.5}^0$ admits a compact quotient $X$, whose de Rham cohomology is the same as the cohomology of $\left(\wedge^\bullet\duale{\left(\mathfrak{g}_{3.4}^{-1}\oplus\mathfrak{g}_{3.5}^0\right)},\, \de\right)$, see \cite[Table 5]{bock}: indeed, note that $\dim_\R H^k_{dR}\left(X;\R\right)=\dim_\R H^k\left(\mathfrak{g}_{3.4}^{-1}\oplus\mathfrak{g}_{3.5}^0;\R\right)$ for every $k\in\N$.

It is straightforward to compute
\begin{eqnarray*}
H_{dR}^1\left(X;\R\right) &=& \underbrace{\R\left\langle e^3,\; e^6 \right\rangle}_{=\; H^{(0,1)}_\omega\left(X;\R\right)} \;, \\[5pt]
H_{dR}^2\left(X;\R\right) &=& \underbrace{\R\left\langle e^{12}+e^{36}+e^{45} \right\rangle}_{=\; H^{(1,0)}_\omega\left(X;\R\right)} \oplus \underbrace{\R\left\langle e^{12}-e^{36},\; e^{12}-e^{45} \right\rangle}_{=\; H^{(0,2)}_\omega\left(X;\R\right)} \;, \\[5pt]
H_{dR}^3\left(X;\R\right) &=& \underbrace{\R\left\langle e^{123}+e^{345},\; e^{126}+e^{456} \right\rangle}_{=\; H^{(1,0)}_\omega\left(X;\R\right) \;=\; L\,H^{(0,1)}_\omega\left(X;\R\right)} \oplus \underbrace{\R\left\langle e^{123}-e^{345},\; e^{126}-e^{456} \right\rangle}_{=\; H^{(0,3)}_\omega\left(X;\R\right)}  \;, \\[5pt]
H_{dR}^4\left(X;\R\right) &=& \underbrace{\R\left\langle e^{1236}+e^{1245}+e^{3456} \right\rangle}_{=\; H^{(2,0)}_\omega\left(X;\R\right)} \oplus \underbrace{\R\left\langle e^{1236}-e^{1245},\; e^{1236}-e^{3456} \right\rangle}_{=\; H^{(1,2)}_\omega\left(X;\R\right) \;=\; L\, H^{(0,2)}_\omega\left(X;\R\right)} \;, \\[5pt]
H_{dR}^5\left(X;\R\right) &=& \underbrace{\R\left\langle e^{12456},\; e^{12345} \right\rangle}_{=\; H^{(2,1)}_\omega\left(X;\R\right) \;=\; L^2\, H^{(0,1)}_\omega\left(X;\R\right)} \;, \\[5pt]
\end{eqnarray*}
hence we get a decomposition
$$ H_{dR}^\bullet\left(X;\R\right) \;=\; \bigoplus_{r\in\N} L^r\, H^{(0,\bullet-2r)}_\omega\left(X;\R\right) \;.$$
In particular, it follows that the Hard Lefschetz Condition holds on $\left(X,\, \omega\right)$.
\end{ex}

\begin{ex}{\itshape A $6$-dimensional completely-solvable solvmanifold such that $H^{(0,3)}_\omega\left(X;\R\right) + H^{(1,1)}_\omega\left(X;\R\right) \subsetneq H^3_{dR}(X;\R)$.}\\
Take a $6$-dimensional completely-solvable solvmanifold with Lie algebra $\mathfrak{g}_{3.1}\oplus \mathfrak{g}_{3.4}^{-1}$,
$$ X \;=\; \left. \Gamma \right\backslash G \;:=\; \left(-23,\; 0,\; 0,\; -46,\; 56,\; 0 \right) \;, $$
see \cite[Table 5]{bock}, endowed with the $G$-left-invariant symplectic structure
$$ \omega \;:=\; e^{12}+e^{36}+e^{45} \;.$$

By A. Hattori's theorem \cite[Corollary 4.2]{hattori}, one computes
\begin{eqnarray*}
H^1_{dR}(X;\R) &=& \underbrace{\R\left\langle e^2,\; e^3,\; e^6\right\rangle}_{=H^{(0,1)}_\omega(X;\R)} \;,\\[5pt]
H^2_{dR}(X;\R) &=& \underbrace{\R\left\langle e^{12}+e^{36}+e^{45} \right\rangle}_{=H^{(1,0)}_\omega(X;\R)} \oplus \underbrace{\R\left\langle e^{12}-e^{36},\; e^{12}-e^{45},\; e^{13},\; e^{26} \right\rangle}_{=H^{(0,2)}_\omega(X;\R)} \;, \\[5pt]
H^3_{dR}(X;\R) &=& \R\left\langle e^{123},\; e^{126},\; e^{136},\; e^{245},\; e^{345},\; e^{456} \right\rangle \;. \\[5pt]
\end{eqnarray*}

Note that $e^{123}-e^{345}$, $e^{126}-e^{456}$ and $e^{245}+\de e^{16}$ are primitive, and consequently
$$ H^{(0,3)}_\omega\left(X;\R\right) \;\supseteq\; \R\left\langle e^{123}-e^{345},\; e^{126}-e^{456},\; e^{245} \right\rangle \;; $$
since $e^{123}+e^{345}=L e^3$, $e^{126}+e^{456}=Le^6$, and $e^{245}-\de e^{16}=L\, e^2$, it follows that
$$ H^{(1,1)}_\omega\left(X;\R\right) \;=\; L\, H^{(0,3)}_\omega\left(X;\R\right) \;\supseteq \; \R\left\langle e^{123}+e^{345},\; e^{126}+e^{456},\; e^{245} \right\rangle \;; $$
since
$$ e^{136} \;=\; \underbrace{\frac{1}{2}\,\left(e^{136}+e^{145}\right)}_{\in L\,\Prim^1X} + \underbrace{\frac{1}{2}\,\left(e^{136}-e^{145}\right)}_{\in \Prim^3X} \;, $$
and
$$ \de\wedge^2\duale{\mathfrak{g}} \;=\; \R \left\langle e^{146}-e^{234},\; e^{156}+e^{235},\; e^{236},\; e^{246},\; e^{256},\; e^{346},\; e^{356} \right\rangle \:, $$
it follows that
$$ \left\langle e^{136} \right\rangle \not\in H^{(0,3)}_\omega\left(X;\R\right) + H^{(1,1)}_\omega\left(X;\R\right) \;,$$
and hence $H^{(0,3)}_\omega\left(X;\R\right) + H^{(1,1)}_\omega\left(X;\R\right) \subsetneq H^3_{dR}(X;\R)$.
\end{ex}

\medskip

Finally, we give explicit examples of dual currents on a compact symplectic half-flat manifold, \cite[Example 3.7]{angella-tomassini-4}.
\begin{ex}{\itshape Dual currents of oriented special Lagrangian submanifolds of the Nakamura manifold.}\\
Let $\C^3$ be endowed with the product $*$ defined by
$$
\left(w^1,\, w^2,\, w^3\right)
\,*\,
\left(z^1,\, z^2,\, z^3\right)
\;:=\;
\left(z^1+w^1,\, \esp^{-w^1}z^2+w^2,\, \esp^{w^1}z^3+w^3\right)
$$
for every $\left(w^1, \, w^2, \, w^3\right),\, \left(z^1, \, z^2, \, z^3\right)\in \C^3$.
Then $\left(\C^3,\,*\right)$ is a complex solvable (non-nilpotent) Lie group and, according to \cite{nakamura}, it admits a lattice
$\Gamma\subset\C^3$, such that $X:=\Gamma\backslash (\C^3,*)$ is a solvmanifold, which is called the \emph{Nakamura manifold}, see also \cite[\S3]{debartolomeis-tomassini}. Setting
$$
\varphi^1 \;:=\; \de z^1\;, \qquad \varphi^2 \;:=\; \esp^{z^1}\de z^2\;, \qquad \varphi^3 \;:=\; \esp^{-z^1}\de z^3\;,
$$
then $\left\{\varphi^1,\, \varphi^2,\, \varphi^3\right\}$ is a global complex co-frame on $X$ satisfying the following complex structure equations:
$$
\left\{
\begin{array}{lcl}
 \de\varphi^1 &=& 0 \\[5pt]
 \de\varphi^2 &=& \varphi^{12} \\[5pt]
 \de\varphi^3 &=& -\varphi^{13}
\end{array}
\right. \;.
$$
If we set $\varphi^j =: e^j+\im e^{3+j}$, for $j\in\{1,\, 2,\, 3\}$, then the last equations yield to
\begin{equation}\label{differenzialireali}
\left\{
\begin{array}{lcl}
\de e^{1} &=& 0\\[5pt]
\de e^{2} &=& e^{12}-e^{45}\\[5pt]
\de e^{3} &=& -e^{13}+e^{46}\\[5pt]
\de e^{4} &=& 0 \\[5pt]
\de e^{5} &=& e^{15}-e^{24}\\[5pt]
\de e^{6} &=& -e^{16}+e^{34}
\end{array}
\right. \;.
\end{equation}
Then, \cite[\S5]{debartolomeis-tomassini},
$$
\omega \;:=\; e^{14}+e^{35}+e^{62}\;,
$$
and
$$
\begin{array}{lll}
Je^1 \;:=\; -e^{4}\;, & \quad Je^{3} \;:=\; -e^{5}\;, & \quad Je^{6} \;:=\; -e^{2}\;,\\[5pt]
Je^4 \;:=\;  e^{1}\;, & \quad Je^{5} \;:=\;  e^{3}\;, & \quad Je^{2} \;:=\; e^{6}\;,
\end{array}
$$
and
$$
\psi \;:=\; \left(e^1+\im e^4\right) \wedge \left(e^3+\im e^5\right) \wedge \left(e^6+\im e^2\right)
$$
give rise to a symplectic half-flat structure on $X$, where 
$$
\Re\,\psi \;=\; e^{136} + e^{125} + e^{234} - e^{456} \;.
$$

Note that the Hard Lefschetz Condition holds on $\left(X,\, \omega\right)$, \cite[Theorem 5.1]{debartolomeis-tomassini}.

Setting $z^j=:x^j+\im y^j$, for $j\in\{1,\, 2,\, 3\}$, and denoting by $\pi\colon \C^3\to X$ the canonical projection, we easily check that 
$$
\begin{array}{l}
Y_1 \;:=\; \pi\left(\left\{\left(x^1,\, x^2,\, x^3,\, y^1,\, y^2,\, y^3\right) \in \C^3 \st x^2=y^4=y^5=0\right\}\right) \;,\\[5pt]
Y_2 \;:=\; \pi\left(\left\{\left(x^1,\, x^2,\, x^3,\, y^1,\, y^2,\, y^3\right) \in \C^3 \st x^3=y^4=y^6=0\right\}\right)
\end{array}
$$
are special Lagrangian submanifolds of $\left(X,\, \omega,\, \psi\right)$, namely, for $j\in\{1,\, 2\}$, it holds $\Re\,\psi\lfloor_{{Y_j}}=\Vol_{Y_j}$, and, consequently, the associated dual currents $\left[Y_j\right]$ are primitive.
\end{ex}

\section{Cohomology of \para-complex manifolds}\label{sec:paracomplex}

In this section, we provide some results obtained in a joint work with F.~A. Rossi, \cite{angella-rossi}, concerning the de Rham cohomology of almost-\para-complex manifolds. \para-complex Geometry is, in a sense, the ``hyperbolic analogue'' of Complex Geometry: an almost-\para-complex structure on a manifold $X$ is given by an endomorphism $K\in\End(TX)$ such that $K^2=\id_{TX}$ and the eigen-bundles of $TX$ with respect to the eigenvalues $1$ and $-1$ of $K\in\End(TX)$ have the same rank. Recently, \para-complex Geometry appeared to be related with many other problems and notions in Mathematics and Physics (in particular, with product structures, bi-Lagrangian geometry, and optimal transport theory).

It is natural to ask what properties from Complex Geometry can be translated in the \para-complex setting. We refer to the work by F.~A. Rossi, e.g., \cite{rossi-1, rossi-2}, for problems and results in this direction. Here, we are mainly concerned in cohomological properties. In fact, it turns out that the \para-complex counterpart of the Dolbeault cohomology is not well-behaved, not being finite-dimensional. This fact leads us to study some subgroups of the de Rham cohomology related to the almost-\para-complex structure, miming the theory introduced by T.-J. Li and W. Zhang in \cite{li-zhang} for almost-complex manifolds. More precisely, we study the subgroups of the de Rham cohomology of an almost-\para-complex manifold consisting of the classes admitting invariant, respectively anti-invariant, representatives with respect to the almost-\para-complex structure; in particular, we prove that, on a $4$-dimensional nilmanifold endowed with a left-invariant \para-complex structure, such subgroups provide a 
decomposition at the level of the real second de Rham cohomology group, Theorem \ref{thm:4-dim}; counterexamples without the hypothesis on dimension, respectively nilpotency, respectively integrability, are provided. Moreover, we consider deformations of \para-complex structures: in particular, we show that admitting a \para-K\"ahler structure is not a stable property under small deformations of the \para-complex structure, Theorem \ref{thm:para-kahler-deformations}, providing another strong difference with the complex case (indeed, recall that admitting a K\"ahler structure is a stable property under small deformations of the complex structure by K. Kodaira and D.~C. Spencer's theorem \cite[Theorem 15]{kodaira-spencer-3}).

\subsection{\para-complex structures on manifolds}

We start by recalling the basic definitions in \para-complex Geometry. We refer, e.g., to \cite{harvey-lawson, alekseevsky-medori-tomassini, cortes-mayer-mohuapt-saueressig-1, cortes-mayer-mohuapt-saueressig-2, cortes-mayer-mohuapt-saueressig-3, cruceanu-fortuny-gadea, kim-mccann-warren, andrada-barberis-dotti-ovando, andrada-salamon, krahe, rossi-1, rossi-2} and the references therein for more results about (almost-)\para-complex structures and for motivations for their study.

\medskip

Let $X$ be a $2n$-dimensional manifold. Consider $K\in\End(TX)$ such that $K^2=\lambda\,\id_{TX}$ where $\lambda\in\left\{-1,\,0,\,1\right\}$: if $\lambda=-1$, then by definition $K$ is an \emph{almost-complex} structure; if $\lambda=0$, then the structure $K$ is called an \emph{almost-subtangent} structure; if $\lambda=1$, then $K$ is said to be an \emph{almost-product} structure; according to \cite[\S1]{vaisman}, these three structures are called \emph{almost-c.p.s.} structures.

In the case $K^2=\id_{TX}$, one gets that $K$ has eigen-values $\left\{1,\,-1\right\}$ and hence there is a decomposition $TX=T^+X\oplus T^-X$ where $T^\pm X$ is given, point by point, by the eigen-space of $K$ corresponding to the eigen-value $\pm 1$, where $\pm\in\left\{+,\,-\right\}$.
By definition, an \emph{almost-\para-complex} structure (also called \emph{almost-para-complex} structure) on $X$ is an endomorphism $K\in\End(TX)$ such that
$$ K^2\;=\;\id_{TX} \qquad \text{ and } \qquad \rk T^+X\;=\;\rk T^-X\;=\;\frac{1}{2}\dim X \;;$$
a \emph{\para-holomorphic map} between two almost-\para-complex manifolds $\left(X_1,\, K_1\right)$ and $\left(X_2,\, K_2\right)$ is a smooth map $f\colon X_1 \to X_2$ such that $\de f \circ K_1 = K_2 \circ \de f$.

An almost-\para-complex structure is said to be \emph{integrable} (and hence it is called \emph{\para-complex}, or also \emph{para-complex}) if
$$ \left[T^+X,\,T^+X\right] \;\subseteq\; T^+X \qquad \text{ and } \qquad \left[T^-X,\,T^-X\right] \;\subseteq\; T^-X \;.$$

The integrability condition is, straightforwardly, equivalent to the vanishing of the \emph{Nijenhuis tensor} $N_K$ of $K$, where
$$ N_K(\sspace ,\, \ssspace) \;:=\; \left[\sspace ,\, \ssspace\right] + \left[K\sspace ,\, K \ssspace\right] - K\left[K\sspace ,\, \ssspace\right] - K\left[\sspace ,\, K\ssspace \right] \;; $$
furthermore, as in the complex case, one has that an almost-\para-complex structure on an $n$-dimensional manifold $X$ is integrable if and only if it is naturally associated to a structure on $X$ defined in terms of local homeomorphisms with open sets in $\mathbf{D}^n$ and \para-holomorphic changes of coordinates, see, e.g., \cite[Proposition 3]{cortes-mayer-mohuapt-saueressig-1}, where $\mathbf{D}^n:=\R^n+\tau\,\R^n$, with $\tau^2=1$, is the algebra of \emph{double numbers}.

We recall that, given a $2n$-dimensional manifold endowed with an almost-\para-complex structure $K$, a \emph{\para-Hermitian metric} on $X$ is a pseudo-Riemannian metric of signature $(n,n)$ such that $g(K\sspace,\, K\ssspace)=-g(\sspace,\, \ssspace)$.
A \emph{\para-K\"ahler} structure on a manifold $X$ is the datum of an integrable \para-complex structure $K$ and a \para-Hermitian metric $g$ such that its associated $K$-anti-invariant form $\omega:=g\left(K\,\sspace, \, \ssspace\right)$ is $\de$-closed, equivalently, the datum of a \emph{$K$-compatible} (that is, a $K$-anti-invariant) symplectic form on $X$, see, e.g., \cite[\S5.1]{alekseevsky-medori-tomassini}, \cite[Theorem 1]{cortes-mayer-mohuapt-saueressig-1}.

\medskip

The basic example of \para-complex structure is given on the product of two manifolds of the same dimension: given $X^+$ and $X^-$ two manifolds with $\dim X^+=\dim X^-$, the product $X^+\times X^-$ inherits a natural \para-complex structure $K$, given by the decomposition
$$ T\left(X^+\times X^-\right) \;=\; TX^+ \,\oplus\, TX^- \;,$$
where $K\lfloor_{TX^+}=\id_{T\left(X^+\times X^-\right)}$ and $K\lfloor_{TX^-}=-\id_{T\left(X^+\times X^-\right)}$.
Every \para-complex manifold is locally of this form, see, e.g., \cite[Proposition 2]{cortes-mayer-mohuapt-saueressig-1}.

\medskip

Starting from $K\in\End(TX)$ such that $K^2=\id_{TX}$, one can define, by duality, an endomorphism $K\in\End(T^*X)$ such that $K^2=\id_{T^*X}$, and hence one gets a natural decomposition $T^*X=\duale{\left(T^{+}X\right)}\oplus \duale{\left(T^{-}X\right)}$ into eigen-bundles. Extending $K\in\End(T^*X)$ to $K\in\End\left(\wedge^\bullet X\right)$, one gets the following decomposition on the bundle of differential $\ell$-forms, for $\ell\in\N$:
$$ \wedge^\ell X \;=\; \bigoplus_{p+q=\ell} \formepm{p}{q}X \qquad \text{ where } \qquad \formepm{p}{q}X \;:=\; \wedge^p\duale{\left(T^{+}X\right)}\otimes\wedge^q\duale{\left(T^{-}X\right)} \;;$$
note that, for any $p,q\in\N$, the structure $K$ acts on $\formepm{p}{q}X$ as $\left(+1\right)^p\left(-1\right)^q\id_{\formepm{p}{q}X}$.
In particular, for any $\ell\in\N$, one has
$$ \wedge^\ell X \;=\; \wedge^{\ell\,+}_KX \oplus \wedge^{\ell\,-}_KX $$
where
$$ \wedge^{\ell\,+}_KX \;:=\; \bigoplus_{p+q=\ell,\;q=0 \modulo 2} \formepm{p}{q}X \qquad \text{ and } \qquad \wedge^{\ell\,-}_KX \;:=\; \bigoplus_{p+q=\ell,\;q=1\modulo 2} \formepm{p}{q}X  \;; $$
note that $K\lfloor_{\wedge^{\ell\,+}_KX} = \id_{\wedge^{\ell\,+}_KX}$ and $K\lfloor_{\wedge^{\ell\,-}_KX} = -\id_{\wedge^{\ell\,-}_KX}$.

If a \para-complex structure $K$ is given, then the exterior differential splits as
$$ \de \;=\; \del_+ + \del_- $$
where
$$ \del_+\;:=\;\proj_{\formepm{p+1}{q}X} \circ \de \colon \formepm{p}{q}X\to\formepm{p+1}{q}X $$
and
$$ \del_-\;:=\;\proj_{\formepm{p}{q+1}X} \circ \de \colon \formepm{p}{q}X\to\formepm{p}{q+1}X $$
(where $\proj_{\formepm{r}{s}X}\colon \formepm{\bullet}{\bullet}X\to \formepm{r}{s}X$ denotes the natural projection onto $\formepm{r}{s}X$, for every $r,s\in\N$).
In particular, the condition $\de^2=0$ can be rewritten as
$$
\left\{
\begin{array}{rcl}
 \del_+^2 &=& 0 \\[5pt]
 \del_+\del_-+\del_-\del_+ &=& 0 \\[5pt]
 \del_-^2 &=& 0
\end{array}
\right.
$$
and hence one can define a \para-complex counterpart of the Dolbeault cohomology by considering the cohomology of the differential complex $\left(\formepm{\bullet}{q}X,\, \del_+\right)$ for every $q\in\N$, that is,
$$ H^{\bullet,\bullet}_{\del_+}(X;\R) \;:=\; \frac{\ker\del_+}{\imm\del_+} \;.$$
Unfortunately, one cannot hope to adjust the Hodge theory of the complex case to this non-elliptic context. For example, take $X^+$ and $X^-$ two manifolds having the same dimension and consider the natural \para-complex structure on $X^+\times X^-$; one has that
$$ H^{0,0}_{\del_+}\left(X^+\times X^-\right) \;\simeq\; \mathcal{C}^\infty\left(X^-\right) \;,$$
hence the space $H^{0,0}_{\del_+}\left(X^+\times X^-\right)$ of $\del_+$-closed functions on $X^+\times X^-$ is not finite-dimensional, even if $X^+$ and $X^-$ are compact.

\subsection{\para-complex subgroups of (co)homology}
In this section, we adapt T.-J. Li and W. Zhang's theory on cohomology of almost-complex manifolds, \cite{li-zhang}, to the almost-\para-complex case. More precisely, let $X$ be a $2n$-dimensional compact manifold endowed with an almost-\para-complex structure $K$; we are interested in studying when the decomposition
$$ \wedge^\bullet X \;=\; \bigoplus_{p,q} \formepm{p}{q}X \;=\; \wedge^{\bullet\,+}_KX \oplus \wedge^{\bullet\,-}_KX $$
gives rise to a cohomological decomposition.

\medskip

We start by giving some definitions.
For any $p,q\in\N$, we define the subgroup
$$ H^{(p,q)}_K\left(X;\R\right) \;:=\; \left\{\left[\alpha\right]\in H^{p+q}_{dR}\left(X;\R\right)\st \alpha\in\formepm{p}{q}X\right\} \;\subseteq\; H^\bullet_{dR}(X;\R) \;, $$
and, for any $\ell\in\N$ and for $\pm\in\{+,-\}$, the subgroup
$$
H^{\ell\,\pm}_{K}\left(X;\R\right) \;:=\; \left\{\left[\alpha\right]\in H^\ell_{dR}\left(X;\R\right) \st K\alpha=\pm\alpha\right\} \;=\; \left\{\left[\alpha\right]\in H^\ell_{dR}\left(X;\R\right) \st \alpha\in\wedge^{\ell\,\pm}_KX\right\} \;\subseteq\; H^\bullet_{dR}(X;\R) \;.
$$

Note, that, if $K$ is integrable, then, for any $\ell\in\N$,
$$ H^{\ell\,+}_{K} \;=\; \bigoplus_{p+q=\ell,\;q= 0 \modulo 2} H^{(p,q)}_K(X;\R) \qquad \text{ and } \qquad H^{\ell\,-}_{K} \;=\; \bigoplus_{p+q=\ell,\;q=1 \modulo 2} H^{(p,q)}_K(X;\R) \;. $$

As in \cite[Definition 2.2, Definition 2.3, Lemma 2.2]{li-zhang}, see \S\ref{subsec:cpf-almost-complex}, for almost-complex structures, we introduce the following definition, \cite[Definition 1.2]{angella-rossi}.

\begin{defi}
 For $\ell\in\N$, an almost-\para-complex structure $K$ on the manifold $X$ is said to be
\begin{itemize}
 \item \emph{\Cp\ at the \kth{\ell} stage} if
  $$ H^{\ell\,+}_K\left(X;\R\right)\,\cap\,H^{\ell\,-}_K\left(X;\R\right) \;=\; \left\{0\right\} \;;$$
 \item \emph{\Cf\ at the \kth{\ell} stage} if
  $$ H^{\ell\,+}_K\left(X;\R\right)\,+\,H^{\ell\,-}_K\left(X;\R\right) \;=\; H^\ell_{dR}\left(X;\R\right) \;;$$
 \item \emph{\Cpf\ at the \kth{\ell} stage} if it is both \Cp\ at the \kth{\ell} stage and \Cf\ at the \kth{\ell} stage, namely, if it satisfies the cohomological decomposition
  $$ H^\ell_{dR}\left(X;\R\right) \;=\; H^{\ell\,+}_K\left(X;\R\right)\,\oplus\,H^{\ell\,-}_K\left(X;\R\right) \;.$$
 \end{itemize}
\end{defi}

\medskip

Consider now the space $\correnti^\bullet X:=:\correnti_{2n-\bullet}X$ of currents on $X$ and the de Rham homology $H_\bullet^{dR}(X;\R)$ (we refer to \S\ref{sec:currents}, and references therein, for definitions and results concerning currents and de Rham homology).
The action of $K$ on $\wedge^\bullet X$ induces, by duality, an action, still denoted by $K$, on $\correnti_\bullet X$, and hence, for any $\ell\in\N$, a decomposition
$$ \correnti_\ell X \;=:\; \bigoplus_{p+q=\ell} \correntipm{p}{q}X \;.$$
For any $p,q\in\N$, note that the space $\correntipm{p}{q}X:=:\correntipmalto{n-p}{n-q}$ is the topological dual space of the topological subspace $\formepm{p}{q}X$ of $\wedge^\bullet X$, and that the quasi-isomorphism $T_\sspace\colon \wedge^\bullet X \ni \alpha\mapsto \int_X\alpha\wedge\sspace\in\correnti^\bullet X$ yields the inclusion $T_\sspace\colon\formepm{p}{q}\hookrightarrow\correntipmalto{p}{q}X$. As before, we have
$$ \correnti_\bullet X \;=\; \correnti_{\bullet\,+}^{K}X \,\oplus\, \correnti_{\bullet\,-}^{K}X $$
where
$$ \correnti_{\bullet\,+}^{K}X \;:=\; \bigoplus_{q=0 \modulo 2}\correntipm{\bullet}{q}X \qquad \text{ and } \qquad \correnti_{\bullet\,-}^{K}X \;:=\; \bigoplus_{q=1\modulo 2}\correntipm{\bullet}{q}X \;,$$
and $K\lfloor_{\correnti_{\bullet\,\pm}^{K}X}=\pm\id_{\correnti_{\bullet\,\pm}^{K}X}$ for $\pm\in\{+,\,-\}$.

For any $p,q\in\N$, we define the subgroup
$$ H_{(p,q)}^K\left(X;\R\right) \;:=\; \left\{\left[\alpha\right]\in H_{p+q}^{dR}\left(X;\R\right)\st \alpha\in\correntipm{p}{q}X\right\} \;\subseteq\; H_\bullet^{dR}(X;\R) \;, $$
and, for any $\ell\in\N$ and for $\pm\in\{+,-\}$, the subgroup
$$ H_{\ell\,\pm}^{K}\left(X;\R\right) \;:=\; \left\{\left[\alpha\right]\in H_\ell^{dR}\left(X;\R\right) \st K\alpha=\pm\alpha\right\} \;=\; \left\{\left[\alpha\right]\in H_\ell^{dR}\left(X;\R\right) \st \alpha\in\correnti_{\ell\,\pm}^KX\right\} \;\subseteq\; H_\bullet^{dR}(X;\R) \;.$$

We are particularly interested in the almost-\para-complex structures admitting a homological decomposition through the subgroups $H_{\bullet\,+}^{K}\left(X;\R\right)$ and $H_{\bullet\,-}^{K}\left(X;\R\right)$, \cite[Definition 1.3]{angella-rossi}.

\begin{defi}
 For $\ell\in\N$, an almost-\para-complex structure $K$ on the manifold $X$ is said to be
\begin{itemize}
 \item \emph{\p\ at the \kth{\ell} stage} if
  $$ H_{\ell\,+}^K\left(X;\R\right)\,\cap\,H_{\ell\,-}^K\left(X;\R\right) \;=\; \left\{0\right\} \;;$$
 \item \emph{\f\ at the \kth{\ell} stage} if
  $$ H_{\ell\,+}^K\left(X;\R\right)\,+\,H_{\ell\,-}^K\left(X;\R\right) \;=\; H_\ell\left(X;\R\right) \;;$$
 \item \emph{\pf\ at the \kth{\ell} stage} if it is both \p\ at the \kth{\ell} stage and \f\ at the \kth{\ell} stage, namely, if it satisfies the homological decomposition
  $$ H_\ell\left(X;\R\right) \;=\; H_{\ell\,+}^K\left(X;\R\right)\,\oplus\,H_{\ell\,-}^K\left(X;\R\right) \;.$$
 \end{itemize}
\end{defi}

\medskip

The introduced notions are not completely independent. Using the same argument as in Theorem \ref{thm:implicazioni}, see \cite[Proposition 2.5]{li-zhang}, and in Proposition \ref{prop:full-k-pure-2n-k}, we prove the following relations between \Cpf ness and \pf ness for almost-\para-complex structures, \cite[Proposition 1.4]{angella-rossi}.

\begin{prop}
\label{prop:para-cpf-pf}
 Let $K$ be an almost-\para-complex structure on a $2n$-dimensional compact manifold $X$. Then, for every $\ell\in\N$, the following implications hold:
$$
\xymatrix{
\text{\Cf\ at the \kth{\ell} stage}\ar@{=>}[r]\ar@{=>}[d] & \text{\p\ at the \kth{\ell} stage}\ar@{=>}[d] \\
\text{\f\ at the \kth{\left(2n-\ell\right)} stage}\ar@{=>}[r] & \text{\Cp\ at the \kth{\left(2n-\ell\right)} stage}
}
$$
\end{prop}

\begin{proof}
 We recall that the quasi-isomorphism $T_\sspace\colon \wedge^\bullet X \ni \alpha\mapsto \int_X\alpha\wedge\sspace\in\correnti_{2n-\bullet} X$ induces, for every $p,q\in\N$, the inclusion
 $$ H^{(p,q)}_K(X;\R) \hookrightarrow H_{(n-p,n-q)}^K(X;\R) \;:$$
 this fact proves the two vertical implications.

 To prove the horizontal implications, consider the duality paring $\left\langle \sspace,\,\ssspace\right\rangle\colon D_{\ell}X\times \wedge^{\ell}X\to\R$ and the induced non-degenerate pairing
 $$ \left\langle \sspace,\,\ssspace\right\rangle\colon H^\ell_{dR}(X;\R)\times H_\ell^{dR}(X;\R)\to \R \;.$$
 
 Suppose that $K$ is \Cf\ at the \kth{\ell} stage, that is, $H^\ell_{dR}(X;\R)=H^{\ell\,+}_K(X;\R)+H^{\ell\,-}_K(X;\R)$, and let $\mathfrak{c}=\left[\gamma_+\right]=\left[\gamma_-\right]\in H_{\ell\,+}^K(X;\R)\cap H_{\ell\,-}^K(X;\R)$ with $\gamma_+\in D_{\ell\,+}^KX$ and $\gamma_-\in D_{\ell\,-}^KX$; since
 $$ \left\langle H^\ell(X;\R),\, \mathfrak{c} \right\rangle \;=\; \left\langle H^{\ell\,+}_K(X;\R) + H^{\ell\,-}_K(X;\R) ,\, \mathfrak{c} \right\rangle \;=\; \left\langle H^{\ell\,+}_K(X;\R),\, \left[\gamma_-\right]\right\rangle + \left\langle H^{\ell\,-}_K(X;\R),\, \left[\gamma_+\right]\right\rangle \;=\; 0 \;, $$
 one has $\mathfrak{c}=0$ in $H^{dR}_\ell(X;\R)$; hence $K$ is \p\ at the \kth{\ell} stage.

 Similarly, since
 $$ \left\langle H^{\ell\,+}_K(X;\R) \cap H^{\ell\,-}_K(X;\R) ,\, H_{\ell\,+}^K(X;\R) + H_{\ell\,-}^K(X;\R) \right\rangle \;=\; 0 \;,$$
 we get that, if $K$ is \f\ at the \kth{\ell} stage, then it is  \Cp\ at the \kth{\ell} stage.
\end{proof}

In particular, by applying Proposition \ref{prop:para-cpf-pf} with $2n=4$ and $k=2$, one gets that, on a compact $4$-dimensional manifold endowed with an almost-\para-complex structure, being \Cf\ at the \kth{2} stage is stronger than being \Cp\ at the \kth{2} stage.

A straightforward consequence of Proposition \ref{prop:para-cpf-pf} is the following result, \cite[Corollary 1.5]{angella-rossi}.

\begin{cor}
\label{cor:cf-every-stage}
 Let $K$ be an almost-\para-complex structure on a compact manifold $X$.  If $K$ is \Cf\ at every stage, then it is also \Cpf\ at every stage and \pf\ at every stage.
\end{cor}

As an application of the K\"unneth formula, T. Dr\v{a}ghici, T.-J. Li, and W. Zhang noted that, given $X_1$, respectively $X_2$, a compact manifold endowed with a \Cpf\ almost-complex structure $J_1$, respectively $J_2$, and assuming that $b_1\left(X_1\right)=0$, or $b_1\left(X_2\right) = 0$, then the almost-complex structure $J_1+J_2$ on $X_1\times X_2$ is \Cpf, \cite[Proposition 2.6]{draghici-li-zhang-survey}. In the \para-complex case, we have the following, \cite[Theorem 1.6]{angella-rossi}.

\begin{thm}
\label{thm:products}
 Let $X^+$ and $X^-$ be two compact manifolds of the same dimension. Then the natural \para-complex structure on the product $X^+\times X^-$ is \Cpf\ at every stage and \pf\ at every stage.
\end{thm}

\begin{proof}
For any $\ell\in\N$, using the K\"unneth formula, one gets
\begin{eqnarray*}
\lefteqn{H_{dR}^\ell\left(X^+\times X^-;\R\right) \;\simeq\; \bigoplus_{p+q=\ell} H_{dR}^p\left(X^+;\R\right)\,\otimes\, H_{dR}^q\left(X^-;\R\right)}  \\[10pt]
&=& \underbrace{\left(\bigoplus_{p+q=\ell,\;q=0\modulo 2} H_{dR}^p\left(X^+;\R\right)\,\otimes\, H_{dR}^q\left(X^-;\R\right)\right)}_{\subseteq H^{\ell\,+}_K\left(X^+\times X^-;\,\R\right)} \;\oplus\; \underbrace{\left(\bigoplus_{p+q=\ell,\;q=1\modulo 2} H_{dR}^p\left(X^+;\R\right)\,\otimes\, H_{dR}^q\left(X^-;\R\right)\right)}_{\subseteq H^{\ell\,-}_K\left(X^+\times X^-;\,\R\right)} \\[10pt]
&\subseteq& H^{\ell\,+}_K\left(X^+\times X^-;\R\right) \,+\, H^{\ell\,-}_K\left(X^+\times X^-;\R\right) \;;
\end{eqnarray*}
hence, by using Corollary \ref{cor:cf-every-stage}, one gets the theorem.
\end{proof}

\subsection{\para-complex cohomological decomposition on solvmanifolds}

In this section, we consider left-invariant \para-complex structures on solvmanifolds, as in \S\ref{subsec:cpf-solvmfds} for almost-complex structures, and in \S\ref{subsec:symplectic-solvmfds} for symplectic structures. We recall that, given a Lie algebra $\mathfrak{g}$, one has the differential operator $\de\colon \wedge^\bullet\duale{\g}\to\wedge^{\bullet+1}\duale{\g}$ naturally induced by the Lie bracket $\left[\sspace,\,\ssspace\right]$, and hence the cohomology $H_{dR}^\bullet(\mathfrak{g};\R) := H^\bullet\left(\wedge^\bullet\duale{\mathfrak{g}},\,\de\right)$. Hence, we are concerned with studying the linear counterpart of \para-complex structures on Lie algebras, and the corresponding decomposition problem for the cohomology $H_{dR}^\bullet(\mathfrak{g};\R)$. In particular, we prove a Nomizu-type result for the subgroups $H^{\ell\, \pm}_K(X;\R)$, Proposition \ref{prop:linear-cpf-invariant-cpf-K}; it will allow to explicitly study several examples of \para-complex solvmanifolds:
in \S\ref{subsubsec:examples-para}, we provide some examples of \para-complex structures on solvmanifolds, even admitting a \para-K\"ahler structure, that do not satisfy the cohomology decomposition by means of the subgroups $H^{\ell\, \pm}_K(X;\R)$; then, we prove that, for every left-invariant \para-complex structure on a $4$-dimensional nilmanifold, it holds $H^2_{dR}(X;\R)=H^{2\, +}_K(X;\R)\oplus H^{2\, -}_K(X;\R)$, Theorem \ref{thm:4-dim}, which provides a partial \para-complex counterpart of \cite[Theorem 2.3]{draghici-li-zhang}. (We refer to \S\ref{sec:solvmanifolds} for definitions and results concerning solvmanifolds.)

\medskip

We recall that a \emph{linear almost-\para-complex} structure $K$ on $\mathfrak{g}$ is an endomorphism $K\in\End(\mathfrak{g})$ such that
$$ K^2 \;=\; \id_\mathfrak{g} \quad  \text{ and } \quad \dim_\R\mathfrak{g}^+\;=\;\dim_\R\mathfrak{g}^-\;=\;\frac{1}{2}\,\dim_\R\mathfrak{g} \;,$$ where $\mathfrak{g}^\pm$ is the eigen-space of $K$ corresponding to the eigen-value $\pm1$, for $\pm\in\left\{+,\,-\right\}$.
A linear almost-\para-complex structure on $\mathfrak{g}$ is said to be \emph{integrable} (and hence it is called a \emph{linear \para-complex} structure on $\mathfrak{g}$) if $\mathfrak{g}^+$ and $\mathfrak{g}^-$ are Lie subalgebras of $\mathfrak{g}$, that is,
$$ \left[\mathfrak{g}^+,\,\mathfrak{g}^+\right] \;\subseteq\; \mathfrak{g}^+ \qquad \text{ and }\qquad \left[\mathfrak{g}^-,\,\mathfrak{g}^-\right] \;\subseteq\; \mathfrak{g}^- \;.$$

As a matter of notation, with respect to a given basis $\left\{e_j\right\}_{j\in\{1,\ldots,\dim_\R\mathfrak{g}\}}$ of $\mathfrak{g}$, in writing a(n almost-)\para-complex structure $K$, e.g., (suppose $\dim_\R\mathfrak{g}=6$,) as
$$ K \;:=\; \left(-\;+\;+\;-\;-\;+\right) $$
we mean that
$$ \mathfrak{g}^+ \;:=\; \R\left\langle e_2,\;e_3,\; e_6\right\rangle \qquad \text{ and }\qquad \mathfrak{g}^- \;:=\; \R\left\langle e_1,\;e_4,\; e_5\right\rangle \;. $$

\medskip

By considering the dual map $K\in\End\left(\duale{\mathfrak{g}}\right)$ of $K\in\End(\g)$ and by extending it to $K\in\End\left(\wedge^\bullet\duale{\g}\right)$, the splitting $\mathfrak{g}=\mathfrak{g}^+\oplus \mathfrak{g}^-$ into eigen-spaces given by the linear almost-\para-complex structure $K$ on $\mathfrak{g}$ induces also a splitting $\duale{\mathfrak{g}}=\duale{\left(\mathfrak{g}^+\right)}\oplus \duale{\left(\mathfrak{g}^-\right)}$, and hence, for every $\ell\in\N$, a splitting on the space of $\ell$-forms on $\duale{\g}$:
$$ \wedge^\ell\duale{\mathfrak{g}} \;=\; \bigoplus_{p+q=\ell}\formepm{p}{q}\duale{\mathfrak{g}} \qquad \text{ where } \qquad \formepm{p}{q}\duale{\mathfrak{g}} \;:=\; \bigoplus_{p+q=\ell}\wedge^{p}\duale{\left(\mathfrak{g}^+\right)}\otimes\wedge^{q}\duale{\left(\mathfrak{g}^-\right)} \;;$$
for any $p,q\in\N$, one has $K\lfloor_{\formepm{p}{q}\duale{\mathfrak{g}}}=(+1)^p\,(-1)^q\,\id_{\formepm{p}{q}\duale{\mathfrak{g}}}$. Consider also the splitting of the space of forms into its $K$-invariant and $K$-anti-invariant components:
$$ \wedge^\bullet\duale{\mathfrak{g}} \;=\; \wedge^{\bullet\,+}_K\duale{\mathfrak{g}}\,\oplus\,\wedge^{\bullet\,-}_K\duale{\mathfrak{g}} $$
where
$$ \wedge^{\bullet\,+}_K\duale{\mathfrak{g}} \;:=\; \bigoplus_{q=0\modulo 2}\formepm{\bullet}{q}\duale{\mathfrak{g}} \qquad \text{ and } \qquad \wedge^{\bullet\,-}_K\duale{\mathfrak{g}} \;:=\; \bigoplus_{q=1\modulo 2}\formepm{\bullet}{q}\duale{\mathfrak{g}} \;. $$

\medskip

As for manifolds, we define, for every $p,q\in\N$, the subgroup
$$ H^{(p,q)}_K\left({\mathfrak{g}};\R\right) \;:=\; \left\{\left[\alpha\right]\in H^{p+q}_{dR}\left(\mathfrak{g};\R\right) \st \alpha\in\formepm{p}{q}\duale{\mathfrak{g}}\right\} \;\subseteq\; H^{\bullet}_{dR}\left(\mathfrak{g};\R\right) \;, $$
and, for any $\ell\in\N$ and for $\pm\in\{+,-\}$, the subgroup
$$
H^{\ell\,\pm}_{K}\left(\mathfrak{g};\R\right) \;:=\; \left\{\left[\alpha\right]\in H^\ell\left({\mathfrak{g}};\R\right) \st K\alpha=\pm\alpha\right\}
\;=\; \left\{\left[\alpha\right]\in H^\ell\left({\mathfrak{g}};\R\right) \st \alpha\in\wedge^{\bullet\,\pm}_K\duale{\mathfrak{g}}\right\} \;\subseteq\; H^{\bullet}_{dR}\left(\mathfrak{g};\R\right) \;,
$$
and we give the following definition, \cite[Definition 2.1]{angella-rossi}.

\begin{defi}
 For $\ell\in\N$, a linear almost-\para-complex structure on the Lie algebra $\mathfrak{g}$ is said to be
\begin{itemize}
 \item \emph{linear-\Cp\ at the \kth{\ell} stage} if
  $$ H^{\ell\,+}_K\left(\mathfrak{g};\R\right)\,\cap\,H^{\ell\,-}_K\left(\mathfrak{g};\R\right) \;=\; \left\{0\right\} \;;$$
 \item \emph{linear-\Cf\ at the \kth{\ell} stage} if
  $$ H^{\ell\,+}_K\left(\mathfrak{g};\R\right)\,+\,H^{\ell\,-}_K\left(\mathfrak{g};\R\right) \;=\; H^\ell_{dR}\left(\mathfrak{g};\R\right) \;;$$
 \item \emph{linear-\Cpf\ at the \kth{\ell} stage} if it is both \Cp\ at the \kth{\ell} stage and \Cf\ at the \kth{\ell} stage, namely, if it satisfies the cohomological decomposition
  $$ H^\ell_{dR}\left(\mathfrak{g};\R\right) \;=\; H^{\ell\,+}_K\left(\mathfrak{g};\R\right)\,\oplus\,H^{\ell\,-}_K\left(\mathfrak{g};\R\right) \;.$$
 \end{itemize}
\end{defi}

\medskip

Given a $2n$-dimensional solvmanifold $X=\left.\Gamma\right\backslash G$, one can consider the associated Lie algebra $\mathfrak{g}$ to the Lie group $G$.
Note that a $G$-left-invariant almost-\para-complex structure on $X$ is uniquely determined, through left-translations on $G$, by a linear almost-\para-complex structure on $\mathfrak{g}$; furthermore, a $G$-left-invariant almost-\para-complex structure on $X$ is integrable if and only if the corresponding linear almost-\para-complex structure on $\mathfrak{g}$ is integrable. Hence, in the following we will confuse a $G$-left-invariant (almost-)\para-complex structure $K$ on the solvmanifold $X=\left.\Gamma\right\backslash G$ and the corresponding linear (almost-)\para-complex structure on the naturally associated Lie algebra $\mathfrak{g}$.

We recall that the left-translations induce an injective map in cohomology,
$$ H_{dR}^\bullet\left({\mathfrak{g}};\R\right) \;\hookrightarrow\; H^\bullet_{dR}(X;\R) \;, $$
where $H_{dR}^\bullet\left({\mathfrak{g}};\R\right)$ can be interpreted as the cohomology of the sub-complex composed of the $G$-left-invariant forms of $\left(\wedge^\bullet X,\, \de\right)$,
and that this map is actually an isomorphism if $G$ is nilpotent, respectively completely-solvable, by K. Nomizu's theorem \cite[Theorem 1]{nomizu}, respectively by A. Hattori's theorem \cite[Corollary 4.2]{hattori}.

Hence, given a $G$-left-invariant \para-complex structure on $X$, we may study \para-complex decomposition in cohomology both on $H_{dR}^\bullet\left({\mathfrak{g}};\R\right)$ and on $H^\bullet_{dR}(X;\R)$. The aim of this section is to make clear the connection between the \Cpf ness of a left-invariant almost-\para-complex structure on a completely-solvable solvmanifold and the linear-\Cpf ness of the corresponding linear almost-\para-complex structure on the associated Lie algebra.

\medskip

The following lemma adapt F.~A. Belgun's symmetrization trick, \cite[Theorem 7]{belgun}, to the \para-complex case, \cite[Lemma 2.3]{angella-rossi}.

\begin{lem}
\label{lemma:belgun-para-cplx}
Let $X=\left.\Gamma\right\backslash G$ be a solvmanifold, and denote the Lie algebra naturally associated to $G$ by $\mathfrak{g}$. Let $K$ be a $G$-left-invariant almost-\para-complex structure on $X$, or equivalently the associated linear almost-\para-complex structure on $\mathfrak{g}$.
Let $\eta$ be the $G$-bi-invariant volume form on $G$ given by J. Milnor's Lemma \cite[Lemma 6.2]{milnor}, and such that $\int_X\eta=1$. Up to identifying $G$-left-invariant forms on $X$ and linear forms over $\duale{\mathfrak{g}}$ through left-translations, consider the Belgun symmetrization map, \cite[Theorem 7]{belgun},
$$ \mu\colon \wedge^\bullet X \to \wedge^\bullet \duale{\mathfrak{g}}\;,\qquad \mu(\alpha)\;:=\;\int_X \alpha\lfloor_m \, \eta(m) \;.$$
One has that
$$ \mu\lfloor_{\wedge^\bullet \duale{\mathfrak{g}}}\;=\;\id\lfloor_{\wedge^\bullet \duale{\mathfrak{g}}} $$
and that
$$ \de\left(\mu(\sspace)\right) \;=\; \mu\left(\de\sspace\right) \qquad \text{ and } \qquad K\left(\mu(\sspace)\right) \;=\; \mu\left(K\sspace\right) \;.$$
\end{lem}

As a consequence, we get the following result, \cite[Proposition 2.4]{angella-rossi} (compare with Proposition \ref{prop:linear-cpf-invariant-cpf-J}, see also \cite[Theorem 3.4]{fino-tomassini}, in the almost-complex case, and with Proposition \ref{prop:linear-cpf-invariant-cpf-omega} in the symplectic case).

\begin{prop}
\label{prop:linear-cpf-invariant-cpf-K}
Let $X=\left.\Gamma\right\backslash G$ be a solvmanifold, and denote the Lie algebra naturally associated to $G$ by $\mathfrak{g}$. Suppose that $H_{dR}^\bullet\left(\mathfrak{g};\R\right) \;\simeq\; H^\bullet_{dR}(X;\R)$ (e.g., suppose that $X$ is a completely-solvable solvmanifold). Let $K$ be a $G$-left-invariant almost-\para-complex structure on $X$.
Then, for every $\ell\in\N$ and for $\pm\in\{+,\,-\}$, the injective map
$$ H^{\ell\,\pm}_K(\mathfrak{g};\R) \to H^{\ell\,\pm}_K(X;\R) $$
induced by left-translations is an isomorphism.
\end{prop}

\begin{proof}
Consider the F.~A. Belgun symmetrization map $\mu\colon \wedge^\bullet X \to \wedge^\bullet \duale{\mathfrak{g}}$, \cite[Theorem 7]{belgun}. It is enough to observe the following three facts.
\begin{enumerate}[\itshape (i)]
\item Since $\de\left(\mu(\sspace)\right)=\mu\left(\de\sspace\right)$, \cite[Theorem 7]{belgun}, one has that $\mu$ sends $\de$-closed,  respectively $\de$-exact, forms to $\de$-closed, respectively $\de$-exact, $G$-left-invariant forms, and so it induces a map
$$ \mu\colon H^\bullet_{dR}(X;\R)\to H_{dR}^\bullet\left(\mathfrak{g};\R\right) \;.$$
\item Since $K\left(\mu(\sspace)\right)=\mu\left(K\sspace\right)$, Lemma \ref{lemma:belgun-para-cplx}, for $\pm\in\{+,\,-\}$, one has
$$ \mu\lfloor_{\wedge^{\bullet\,\pm}_KX} \colon \wedge^{\bullet\,\pm}_KX \to \wedge^{\bullet\,\pm}_K\duale{\mathfrak{g}} \;,$$
and hence
$$ \mu\lfloor_{H^{\bullet\,\pm}_K(X;\R)} \colon H^{\bullet\,\pm}_K(X;\R) \to H^{\bullet\,\pm}_K\left(\mathfrak{g};\R\right) \;.$$
\item Finally, if $H^\bullet_{dR}(X;\R)\simeq H_{dR}^\bullet\left(\mathfrak{g};\R\right)$ (e.g., if $X$ is a completely-solvable solvmanifold, \cite[Corollary 4.2]{hattori}), then the condition $\mu\lfloor_{\wedge^\bullet_{\text{inv}}X}=\id\lfloor_{\wedge^\bullet_{\text{inv}}X}$, \cite[Theorem 7]{belgun}, gives that $\mu$ is the identity in cohomology.\qedhere
\end{enumerate}
\end{proof}

As a straightforward corollary, we get the following result, \cite[Proposition 2.4]{angella-rossi} (compare with Corollary \ref{cor:linear-cpf-cpf} in the almost-complex case).

\begin{cor}\label{cor:linear-cpf-invariant-cpf-K}
 Let $X=\left.\Gamma\right\backslash G$ be a solvmanifold such that $H^\bullet_{dR}(X;\R)\simeq H_{dR}^\bullet\left(\mathfrak{g};\R\right)$ (e.g., a completely-solvable solvmanifold), and denote the Lie algebra naturally associated to $G$ by $\mathfrak{g}$. Let $K$ be a $G$-left-invariant almost-\para-complex structure on $X$.
 For every $\ell\in\N$, the associated linear almost-\para-complex structure $K\in\End(\mathfrak{g})$ is linear-\Cp\ (respectively, linear-\Cf, linear-\Cpf) at the \kth{\ell} stage if and only if the $G$-left-invariant almost-\para-complex structure $K\in\End(TX)$ is \Cp\ (respectively, \Cf, \Cpf) at the \kth{\ell} stage.
\end{cor}

\subsubsection{Non-\Cpf\ (almost-)\para-complex nilmanifolds}\label{subsubsec:examples-para}
We provide here some explicit examples of left-invariant (almost-)\para-complex structures on nilmanifolds, studying the corresponding subgroups in cohomology, and providing differences between the \para-complex and the complex cases, Proposition \ref{prop:para-kahler-non-cpf}.

\medskip

More precisely, recall that every K\"ahler structure on a compact manifold is \Cpf, \cite[Lemma 2.15, Theorem 2.16]{draghici-li-zhang}, or \cite[Proposition 2.1]{li-zhang}, and that every almost-complex structure on a $4$-dimensional compact manifold is \Cpf, \cite[Theorem 2.3]{draghici-li-zhang}: we give instead an example of a \para-complex structure on a $6$-dimensional nilmanifold such that it is non-\Cf\ at the \kth{2} stage, \cite[Example 3.1]{angella-rossi}, respectively non-\Cp\ at the \kth{2} stage, \cite[Example 3.2]{angella-rossi}, despite it admits a \para-K\"ahler structure; furthermore, we provide a non-\Cpf\ at the \kth{2} stage almost-\para-complex structure on a $4$-dimensional manifold, proving that no almost-\para-complex counterpart of \cite[Theorem 2.3]{draghici-li-zhang} could exist.

\begin{ex}
\label{es 2.5}
{\itshape A \para-complex structure on a $6$-dimensional nilmanifold that is \Cp\ at the \kth{2} stage and non-\Cf\ at the \kth{2} stage and that admits a \para-K\"ahler structure.}\\
Consider a nilmanifold
 $$ X \;=\; \left.\Gamma\right\backslash G \;:=\; \left(0^4,\; 12,\; 13\right) $$
and define the $G$-left-invariant \para-complex structure $K$ by setting
$$ K \;:=\; \left(-\,+\,+\,-\,-\,+\right) \;.$$
By K. Nomizu's theorem \cite[Theorem 1]{nomizu}, the de Rham cohomology of $X$ is given by
$$
H^2_{dR}(X;\R) \;=\; \R\left\langle e^{14},\; e^{15},\; e^{16},\; e^{23},\; e^{24},\; e^{25},\; e^{34},\; e^{36},\; e^{26}+e^{35} \right\rangle
$$
(where, as usual, we list the harmonic representatives with respect to the $G$-left-invariant metric $\sum_{j=1}^{6} e^j\odot e^j$ instead of their classes, and we write, e.g., $e^{hk}$ to shorten $e^h\wedge e^k$).
Note that
$$ H^{2\,+}_K\left(\mathfrak{g};\R\right) \;=\; \R\left\langle e^{14},\;e^{15},\;e^{23},\;e^{36} \right\rangle \qquad \text{ and } \qquad H^{2\,-}_K\left(\mathfrak{g};\R\right) \;=\; \R\left\langle e^{16},\;e^{24},\;e^{25},\;e^{34} \right\rangle \;, $$
since the space of $G$-left-invariant $\de$-exact $2$-forms is $\R\left\langle e^{12},\, e^{13}\right\rangle$, and hence no $G$-left-invariant representative in the class $\left[e^{26}+e^{35}\right]$ is of pure type with respect to $K$. It follows that $K\in\End\left(\mathfrak{g}\right)$ is linear-\Cp\ at the \kth{2} stage and linear non-\Cf\ at the \kth{2} stage, and hence $K\in\End(TX)$ is \Cp\ at the \kth{2} stage and non-\Cf\ at the \kth{2} stage, by Corollary \ref{cor:linear-cpf-invariant-cpf-K}. (Note that, $K$ being Abelian, one can deduce the \Cp ness at the \kth{2} stage also by Corollary \ref{cor:abelian}.)

Moreover, we observe that
$$ \omega \;:=\; e^{16}+e^{25}+e^{34} $$
is a ($G$-left-invariant) symplectic form compatible with $K$, hence $\left(K,\,\omega\right)$ is a \para-K\"ahler structure on $X$.
\end{ex}

\begin{ex}
\label{es 2.6}
{\itshape A \para-complex structure on a $6$-dimensional nilmanifold that is non-\Cp\ at the \kth{2} stage, and hence non-\Cf\ at the \kth{4} stage, and that admits a \para-K\"ahler structure.}\\
Consider a nilmanifold
 $$ X \;=\; \left.\Gamma\right\backslash G \;:=\; \left(0^3,\;12,\;13+14,\;24\right) $$
and define the $G$-left-invariant \para-complex structure
 $$ K \;:=\; \left(+\;-\;+\;-\;+\;-\right) \;.$$
(Note that $\left[\mathfrak{g}^-,\,\mathfrak{g}^-\right]\neq\left\{0\right\}$, since $\left[e_2,e_4\right]=-e_6$, hence $K$ is not Abelian.)\\
We have
$$ H^{2\,+}_K\left(\mathfrak{g};\R\right) \;\ni\; \left[e^{13}\right] \;=\; \left[e^{13}-\de e^{5}\right] \;=\; -\left[e^{14}\right] \;\in\; H^{2\,-}_K\left(\mathfrak{g};\R\right) $$
and therefore $0 \neq \left[e^{13}\right] \in H^{2\,+}_K\left(\mathfrak{g};\R\right) \cap H^{2\,-}_K\left(\mathfrak{g};\R\right)$, namely, $K\in\End(\mathfrak{g})$ is not linear-\Cp\ at the \kth{2} stage, hence, by Corollary \ref{cor:linear-cpf-invariant-cpf-K}, $K\in\End(TX)$ is not \Cp\ at the \kth{2} stage; moreover, by Proposition \ref{prop:para-cpf-pf}, we have also that $K$ is not \Cf\ at the \kth{4} stage.

Furthermore,
$$ \omega \;:=\; e^{16}+e^{25}+e^{34} $$
is a ($G$-left-invariant) symplectic form compatible with $K$, hence $\left(K,\,\omega\right)$ is a \para-K\"ahler structure on $X$.
\end{ex}

It is straightforward to obtain higher-dimensional examples of \para-K\"ahler non-\Cf, respectively non-\Cp, at the \kth{2} stage structures, taking products with standard \para-complex tori.

The contents of the previous two examples are resumed in the following result, \cite[Proposition 3.3]{angella-rossi}, which gives a difference with the complex case, \cite[Proposition 2.1]{li-zhang}, or \cite[Lemma 2.15, Theorem 2.16]{draghici-li-zhang}.

\begin{prop}\label{prop:para-kahler-non-cpf}
 Admitting a \para-K\"ahler structure is not a sufficient condition for either being \Cp\ at the \kth{2} stage or being \Cf\ at the \kth{2} stage.
\end{prop}

\medskip

We provide now a counterexample showing that T. Dr\v{a}ghici, T.-J. Li, and W. Zhang's decomposition theorem for compact $4$-dimensional almost-complex manifolds, \cite[Theorem 2.3]{draghici-li-zhang}, does not hold, in general, in the almost-\para-complex case, \cite[Example 3.4]{angella-rossi}.

\begin{ex}
\label{es 2.8}
{\itshape An almost-\para-complex structure on a $4$-dimensional nilmanifold that is non-\Cpf\ at the \kth{2} stage.}\\
Consider a nilmanifold
 $$ X \;=\; \left.\Gamma\right\backslash G \;:=\; \left(0,\; 0,\; 12,\; 0\right) $$
and define the $G$-left-invariant almost-\para-complex structure $K$ requiring that $K\lfloor_{\mathfrak{g}^+}=\id_{\mathfrak{g}^+}$ and $K\lfloor_{\mathfrak{g}^-}=-\id_{\mathfrak{g}^-}$ where
\begin{equation*}
\mathfrak{g}^+\;:=\;\R\left\langle e_1,\, e_4-e_2\right\rangle \quad \text{ and }\quad \mathfrak{g}^-\;:=\;\R\left\langle e_2,\, e_3\right\rangle.
\end{equation*}
Note that $K$ is not integrable, since $\left[\g^+,\,\g^+\right]\ni\left[e_1,e_4-e_2\right]=e_3\not\in\g^+$.

Note that we have
$$ H^{2\,+}_K\left(\mathfrak{g};\R\right)\ni\left[e^{14}\right]=\left[e^{14}+\de e^{3}\right]=\left[e^{14}+e^{12}\right]= \left[e^1\wedge(e^4+e^2)\right]\in H^{2\,-}_K\left(\mathfrak{g};\R\right) \;, $$
and therefore we get that $0 \neq \left[e^{14}\right] \in H^{2\,+}_K\left(\mathfrak{g};\R\right) \cap H^{2\,-}_K\left(\mathfrak{g};\R\right)$; then, $K$ is non-\Cp\ at the \kth{2} stage and non-\Cf\ at the \kth{2} stage, by Corollary \ref{cor:linear-cpf-invariant-cpf-K}, and Proposition \ref{prop:para-cpf-pf}.
\end{ex}

\subsubsection{\Cpf ness of low-dimensional \para-complex solvmanifolds}

In this section, we state and prove Theorem \ref{thm:4-dim}, providing a partial \para-complex counterpart of \cite[Theorem 2.3]{draghici-li-zhang} in the almost-complex case. We start by fixing some notations and by proving some preliminary results.

\medskip

Given a linear \para-complex structure $K$ on a Lie algebra $\mathfrak{g}$, consider the induced eigen-spaces decomposition $\mathfrak{g}=\mathfrak{g}^+\oplus\mathfrak{g}^-$, and consider the nilpotent steps
$$ s^+ \;:=\; s\left(\mathfrak{g}^+\right) \qquad \text{ and }\qquad s^- \;:=\; s\left(\mathfrak{g}^-\right) \;. $$

(As a matter of notation, recall that, given a Lie algebra $\left(\mathfrak{a},\,\left[\sspace,\ssspace\right]\right)$, the \emph{lower central series} $\left\{\mathfrak{a}^n\right\}_{n\in\N}$ is defined, by induction on $n\in\N$, as
$$
\left\{
\begin{array}{rcll}
 \mathfrak{a}^0 &:=& \mathfrak{a} & \\[5pt]
 \mathfrak{a}^{n+1} &:=& \left[\mathfrak{a}^n,\,\mathfrak{a}\right] & \text{ for }n\in\N
\end{array}
\right. \;;
$$
note that $\left\{\mathfrak{a}_n\right\}_{n\in\N}$ is a descending sequence of Lie algebras:
$$ \mathfrak{a}\;=\; \mathfrak{a}^0 \;\supseteq \mathfrak{a}^1 \;\supseteq\; \cdots \;\supseteq\; \mathfrak{a}^{j-1} \;\supseteq\; \mathfrak{a}^j \;\supseteq\; \cdots \;;$$
recall that the \emph{nilpotent step} of $\mathfrak{a}$ is defined as
$$ s\left(\mathfrak{a}\right) \;:=\; \inf\left\{m\in\N \st \mathfrak{a}^m=0\right\} \;; $$
in particular, if $s\left(\mathfrak{a}\right)<+\infty$, then, by definition, $\mathfrak{a}$ is nilpotent.)

Since $\mathfrak{g}^+\subset\mathfrak{g}$ and $\mathfrak{g}^-\subset\mathfrak{g}$, we have obviously that
$$ s^+\;\leq\; s\left(\mathfrak{g}\right) \qquad \text{ and } \qquad s^- \;\leq\; s\left(\mathfrak{g}\right)  \;. $$
In fact, we have the following lemma, \cite[Lemma 3.5]{angella-rossi}.

\begin{lem}
\label{lemma:s+-}
 Let $\mathfrak{g}$ be a $2n$-dimensional nilpotent Lie algebra, namely, $s\left(\mathfrak{g}\right)<+\infty$. Let $K$ be a linear \para-complex structure on $\mathfrak{g}$, inducing the eigen-spaces decomposition $\mathfrak{g}=\mathfrak{g}^+\oplus\mathfrak{g}^-$. Then, setting $s^\pm:=s\left(\mathfrak{g}^{\pm}\right)$ for $\pm\in\left\{+,\,-\right\}$, we have
$$ 1 \;\leq\; s^+ \;\leq\; n-1 \qquad \text{ and } \qquad 1 \;\leq\; s^- \;\leq\; n-1 \;.$$
\end{lem}

\begin{proof}
 It suffices to note that, for $\pm\in\left\{+,\;-\right\}$, we have
$$
\left\{
\begin{array}{rcll}
 \dim_\R \left(\mathfrak{g}^\pm\right)^0 &=& n & \\[5pt]
 \dim_\R \left(\mathfrak{g}^\pm\right)^k &\leq& \max\left\{n-k-1,\,0\right\} & \text{ for } k\in\N\setminus\{0\} \\[5pt]
\end{array}
\right. \;,
$$
as a consequence of the nilpotency and of the integrability properties.
\end{proof}

The following result, \cite[Proposition 3.6]{angella-rossi}, should be compared with Theorem \ref{thm:products}.

\begin{prop}
 Let $\mathfrak{g}$ be a Lie algebra. If $K$ is a linear \para-complex structure on $\mathfrak{g}$ with eigen-spaces $\mathfrak{g}^+$ and $\mathfrak{g}^-$ such that $\left[\mathfrak{g}^+,\,\mathfrak{g}^-\right]=\{0\}$, then $K$ is linear-\Cpf\ at every stage.
\end{prop}

\begin{proof}
 Since $\left[\mathfrak{g}^+,\,\mathfrak{g}^-\right]=\{0\}$, one has $\mathfrak{g}=\mathfrak{g}^+ \times \mathfrak{g}^-$ and, using the K\"unneth formula, one gets the statement, as in the proof of Theorem \ref{thm:products}.
\end{proof}

Therefore, from Corollary \ref{cor:linear-cpf-invariant-cpf-K}, one gets the following corollary, \cite[Corollary 3.7]{angella-rossi}.

\begin{cor}
 Let $X=\left.\Gamma\right\backslash G$ be a completely-solvable solvmanifold endowed with a $G$-left-invariant \para-complex structure $K$, and denote the Lie algebra naturally associated to the Lie group $G$ by $\mathfrak{g}$. Consider the linear \para-complex structure $K\in\End(\mathfrak{g})$ induced by $K\in\End(TX)$. Suppose that the eigen-spaces $\mathfrak{g}^+$ and $\mathfrak{g}^-$ of $K\in\End(\mathfrak{g})$ satisfy $\left[\mathfrak{g}^+,\,\mathfrak{g}^-\right]=\{0\}$. Then $K$ is \Cpf\ at every stage and \pf\ at every stage.
\end{cor}

\medskip

We recall that a \para-complex structure on a manifold $X$ is said to be \emph{Abelian} if the induced eigen-bundle decomposition $TX=T^+X\oplus T^-X$ satisfies $\left[T^+X,\,T^+X\right]=\{0\}=\left[T^-X,\,T^-X\right]$; analogously, a linear \para-complex structure on a Lie algebra $\mathfrak{g}$ is said to be \emph{Abelian} if the induced decomposition $\mathfrak{g}=\mathfrak{g}^+\oplus\mathfrak{g}^-$ satisfies $\left[\mathfrak{g}^+,\,\mathfrak{g}^+\right]=\{0\}=\left[\mathfrak{g}^-,\,\mathfrak{g}^-\right]$, namely, $s\left(\g^+\right)=1=s\left(\g^-\right)$.
Obviously, if $X=\left.\Gamma\right\backslash G$ is a solvmanifold endowed with a $G$-left-invariant \para-complex structure, then $K\in\End\left(TX\right)$ is Abelian if and only if the associated linear \para-complex structure $K\in\End\left(\g\right)$ is Abelian.

\begin{rem}
\label{oss 2.11}
 Note that every linear \para-complex structure on a $4$-dimensional nilpotent Lie algebra is Abelian, as a consequence of Lemma \ref{lemma:s+-}.
\end{rem}

We prove that every linear Abelian \para-complex structure is linear-\Cp\ at the \kth{2} stage, \cite[Theorem 3.10]{angella-rossi}.

\begin{thm}
\label{thm 2.12}
 Let $\mathfrak{g}$ be a Lie algebra and $K$ be a linear Abelian \para-complex structure on $\mathfrak{g}$. Then $K$ is linear-\Cp\ at the \kth{2} stage.
\end{thm}

\begin{proof}
 Denote by $\pi_{\wedge^{\bullet\,+}_K\duale{\mathfrak{g}}}\colon\wedge^\bullet\duale{\mathfrak{g}}\to\wedge^{\bullet\,+}_K\duale{\mathfrak{g}}$ the natural projection onto the space $\wedge^{\bullet\,+}_K\duale{\mathfrak{g}}$.
 Recall that $\de\eta:=-\eta\left(\left[\sspace,\,\ssspace\right]\right)$ for every $\eta\in\wedge^1\duale{\mathfrak{g}}$; therefore, since $\left[\mathfrak{g}^+,\,\mathfrak{g}^+\right]=0$ and $\left[\mathfrak{g}^-,\,\mathfrak{g}^-\right]=0$ by hypothesis, we have that $$ \pi_{\wedge^{\bullet\,+}_K\duale{\mathfrak{g}}} \left(\imm\left(\de\colon\wedge^1\duale{\mathfrak{g}}\to\wedge^2\duale{\mathfrak{g}}\right)\right)\;=\;\{0\} \;.$$
 Suppose that there exists $\left[\gamma^+\right]=\left[\gamma^-\right]\in H^{2\,+}_K\left(\mathfrak{g};\,\R\right)\cap H^{2\,-}_K\left(\mathfrak{g};\,\R\right)$, where $\gamma^+\in\wedge^{2\,+}_K\duale{\mathfrak{g}}$ and $\gamma^-\in\wedge^{2\,-}_K\duale{\mathfrak{g}}$, and $\gamma^+=\gamma^-+\de\alpha$ for some $\alpha\in\wedge^1\duale{\mathfrak{g}}$. Since $\pi_{\wedge^{\bullet\,+}_K\duale{\mathfrak{g}}}\left(\de\alpha\right)=0$, we have that $\gamma^+=0$ and hence $\left[\gamma^+\right]=0$, so $K$ is linear-\Cp\ at the \kth{2} stage.
\end{proof}

\begin{rem}
\label{rem:partial-abelian}
 We note that the condition of $K$ being Abelian in Theorem \ref{thm 2.12} cannot be dropped or weakened, in general. In fact, Example \ref{es 2.17} shows that the Abelian assumption just on $\mathfrak{g}^-$ is not sufficient to have \Cp ness at the \kth{2} stage. Another example of this fact, on a (non-unimodular) solvable Lie algebra, is given below, \cite[Example 3.12]{angella-rossi}.
\end{rem}

\begin{ex}
{\itshape A $4$-dimensional (non-unimodular) solvable Lie algebra with a non-Abelian \para-complex structure that is not linear-\Cp\ at the \kth{2} stage.}\\
Consider the $4$-dimensional solvable Lie algebra defined by
$$ \mathfrak{g} \;:=\; \left(0^3,\; 13+34\right) \;;$$
note that $\mathfrak{g}$ is not unimodular, since $\de e^{124}=e^{1234}$, see Lemma \ref{lemma:unimod}, \cite[\S III]{koszul-homologie}.

Set the linear \para-complex structure
$$ K\;:=\; \left( +\;+\;-\;- \right) \;;$$
$K$ is not Abelian, since $\left[\mathfrak{g}^+,\,\mathfrak{g}^+\right]=0$ but $\left[\mathfrak{g}^-,\,\mathfrak{g}^-\right]=\R\left\langle e_4\right\rangle\neq \{0\}$.

It is straightforward to check that $K$ is linear-\Cf\ at the \kth{2} stage: in fact,
$$ H^2_{dR}\left(\mathfrak{g};\R\right) \;=\; \underbrace{\R\left\langle e^{12},\,e^{34}\right\rangle}_{H^+_K\left(\mathfrak{g};\R\right)} \oplus \underbrace{\left\langle e^{23}\right\rangle}_{H^-_K\left(\mathfrak{g};\R\right)} \;; $$
on the other hand, $K$ is not linear-\Cp\ at the \kth{2} stage, since
$$ H^{2\,+}_K\left(\mathfrak{g};\R\right) \;\ni\; \left[e^{34}\right] \;=\; \left[e^{34}-\de e^{4}\right] \;=\; -\left[e^{13}\right] \;\in\; H^{2\,-}_K\left(\mathfrak{g};\R\right) $$
and $\left[e^{34}\right]\neq 0$.
\end{ex}

A direct consequence of Theorem \ref{thm 2.12} and Corollary \ref{cor:linear-cpf-invariant-cpf-K} is the following result, \cite[Corollary 3.13]{angella-rossi}.

\begin{cor}\label{cor:abelian-cp}
\label{cor:abelian}
 Let $X=\left.\Gamma\right\backslash G$ be a completely-solvable solvmanifold endowed with a $G$-left-invariant Abelian \para-complex structure $K$. Then $K$ is \Cp\ at the \kth{2} stage.
\end{cor}

\begin{rem}
 We remark that, for a \para-complex structure on a compact manifold of dimension greater than $4$, being Abelian or being \Cp\ at the \kth{2} stage is not a sufficient condition to have \Cf ness at the \kth{2} stage. Indeed, Example \ref{es 2.5} provides a $G$-left-invariant \para-complex structure $K$ on a $6$-dimensional nilmanifold $X=\left.\Gamma\right\backslash G$ such that $K$ is Abelian, \Cp\ at the \kth{2} stage and non-\Cf\ at the \kth{2} stage.
\end{rem}

\medskip

As observed in Remark \ref{oss 2.11}, any left-invariant \para-complex structure on a $4$-dimensional nilmanifold is Abelian, and hence \Cp\ at the \kth{2} stage by Corollary \ref{cor:abelian-cp}. In general, a left-invariant Abelian \para-complex structure on a nilmanifold of dimension greater than $4$ may be non-\Cf\ at the \kth{2} stage, Example \ref{es 2.5}. Notwithstanding, we prove that every left-invariant \para-complex structure on a $4$-dimensional nilmanifold is \Cf, in fact \Cpf, at the \kth{2} stage, Theorem \ref{thm:4-dim}. To prove this fact, we need the following lemmata: the first one is a classical result, the second one is \cite[Lemma 3.16]{angella-rossi}.

\begin{lem}[{\cite[\S III]{koszul-homologie}}]\label{lemma:unimod}
 Let $\mathfrak{g}$ be a unimodular Lie algebra of dimension $n$. Then
$$ \de\lfloor_{\wedge^{n-1}\duale{\mathfrak{g}}} \;=\; 0 \;.$$
\end{lem}

\begin{lem}
\label{lemma:n0-puro}
 Let $\mathfrak{g}$ be a unimodular Lie algebra of dimension $2n$ endowed with an Abelian linear \para-complex structure $K$. Then
$$ \de\lfloor_{\formepm{n}{0}\duale{\mathfrak{g}}\,\oplus\,\formepm{0}{n}\duale{\mathfrak{g}}} \;=\; 0 \;.$$
\end{lem}

\begin{proof}
 Consider the eigen-spaces decomposition $\g=\g^+\oplus\g^-$ induced by $K$, and fix two bases for $\duale{\left(\mathfrak{g}^+\right)}$ and $\duale{\left(\mathfrak{g}^-\right)}$:
$$ \duale{\left(\mathfrak{g}^+\right)} \;=\; \R\left\langle e^1,\, \ldots,\, e^n\right\rangle \qquad \text{ and }\qquad \duale{\left(\mathfrak{g}^-\right)} \;=\; \R\left\langle f^1,\, \ldots,\, f^n\right\rangle \;. $$
Since $K$ is Abelian, the general structure equations, in terms of these bases, are
$$
\left\{
\begin{array}{rcll}
 \de e^j &=:& \sum_{h,\,k=1}^n a^j_{hk}\, e^h\wedge f^k & \quad \text{ for } j\in\{1,\ldots,n\}\\[5pt]
 \de f^j &=:& \sum_{h,\,k=1}^n b^j_{hk}\, e^h\wedge f^k & \quad \text{ for } j\in\{1,\ldots,n\}
\end{array}
\right. \;,
$$
where $\left\{a^j_{hk},\,b^j_{hk}\right\}_{j,h,k\in\{1,\ldots,n\}}\subset\R$.

By \cite[\S III]{koszul-homologie}, see Lemma \ref{lemma:unimod}, for any $k\in\{1,\ldots,n\}$, one has that
$$ \de\left( e^1\wedge\cdots\wedge e^n\wedge f^1\wedge\cdots\wedge f^{k-1}\wedge f^{k+1}\wedge\cdots\wedge f^n\right) \;=\; 0 \;; $$
by a straightforward computation, we get that,
$$ \sum_{\ell=1}^n a^\ell_{\ell k} \;=\; 0 \;,$$

Hence, we get that
$$ \de\left(e^1\wedge\cdots\wedge e^n\right) \;=\; \left(-1\right)^n \sum_{k=1}^n \left(\sum_{\ell=1}^n a^\ell_{\ell k}\right) e^1\wedge\cdots\wedge e^n\wedge f^k \;=\; 0 \;. $$

Arguing in the same way, we prove also that
$$ \de\left(f^1\wedge\cdots\wedge f^n\right) \;=\; 0 \;,$$
completing the proof.
\end{proof}

We can now prove the following result, \cite[Theorem 3.17]{angella-rossi}.

\begin{thm}
\label{thm:4-dim}
 Every left-invariant \para-complex structure on a $4$-dimensional nilmanifold is \Cpf\ at the \kth{2} stage, and hence also \pf\ at the \kth{2} stage. 
\end{thm}

\begin{proof}
 By Remark \ref{oss 2.11} and Corollary \ref{cor:abelian}, we get the \Cp ness at the \kth{2} stage.

 We recall that, by J. Milnor's lemma \cite[Lemma 6.2]{milnor}, the Lie algebra associated to any nilmanifold is unimodular. From Lemma \ref{lemma:n0-puro} one gets that, on every $4$-dimensional \para-complex nilmanifold, the \para-complex invariant component of a left-invariant $2$-form is $\de$-closed. Hence both the \para-complex invariant component and the \para-complex anti-invariant component of a $\de$-closed left-invariant $2$-form is $\de$-closed. Hence the linear \para-complex structure is linear-\Cf\ at the \kth{2} stage. Therefore, by Corollary \ref{cor:linear-cpf-invariant-cpf-K}, the left-invariant \para-complex structure is \Cf\ at the \kth{2} stage.

 Finally, Proposition \ref{prop:para-cpf-pf} gives the \pf ness at the \kth{2} stage.
\end{proof}

\begin{rem}\label{rem:4-dim-hp-necessarie}
 We note that Theorem \ref{thm:4-dim} is optimal. Indeed, we cannot grow the dimension, Example \ref{es 2.5} and Example \ref{es 2.6}, nor change the nilpotent hypothesis with a solvable condition, Example \ref{es 2.17}, nor drop the integrability condition on the \para-complex structure, Example \ref{es 2.8}.
\end{rem}

\subsection{Small deformations of \para-complex structures}

In this section, we study explicit examples of small deformations of the \para-complex structure on nilmanifolds and solvmanifolds, studying the behaviour of being \para-K\"ahler, Theorem \ref{thm:para-kahler-deformations}, the behaviour of \Cpf ness, Proposition \ref{prop:para-cpf-instability}, and the semi-continuity problem for the dimensions of the \para-(anti-)invariant subgroups of the second de Rham cohomology group, Proposition \ref{prop:para-semi-cont}.

We refer to \cite{medori-tomassini, rossi-1} for more results concerning deformations of (almost-)\para-complex structures.

\medskip

In the following example, \cite[Example 4.1]{angella-rossi}, we construct a curve $\left\{K_t\right\}_{t\in\R}$ of left-invariant \para-complex structures on a $4$-dimensional solvmanifold such that
\begin{inparaenum}[(\itshape i\upshape)] 
 \item $K_0$ is \Cpf\ at the \kth{2} stage and admits a \para-K\"ahler structure, and
 \item $K_t$, for $t\neq 0$, is neither \Cp\ at the \kth{2} stage nor \Cf\ at the \kth{2} stage and does not admit any \para-K\"ahler structure. 
\end{inparaenum}
In particular, this example proves that being \para-K\"ahler is not a stable property under small deformations of the \para-complex structure, Theorem \ref{thm:para-kahler-deformations}, and it shows also that the nilpotency condition in Theorem \ref{thm:4-dim} cannot be dropped out, Remark \ref{rem:4-dim-hp-necessarie}.

\begin{ex}
\label{es 2.17}
{\itshape There exists a $4$-dimensional solvmanifold endowed with a left-invariant \para-complex structure $K$ such that
\begin{inparaenum}[(\itshape i\upshape)] 
 \item $K$ is \Cpf\ at the \kth{2} stage,
 \item it admits a \para-K\"ahler structure, and
 \item it has small \para-complex deformations that are neither \para-K\"ahler nor \Cpf\ at the \kth{2} stage.
\end{inparaenum}
}\\
Consider a $4$-dimensional completely-solvable solvmanifold
$$ X \;=\; \left.\Gamma\right\backslash G \;:=\; \left(0^2,\;23,\;-24 \right) $$
(for its existence, see, e.g., \cite[Table 8]{bock}).

By A. Hattori's theorem \cite[Corollary 4.2]{hattori}, it is straightforward to compute
$$ H^2_{dR}(X;\R) \;=\; \R\left\langle e^{12},\,e^{34}\right\rangle \;.$$
For every $t\in\R$, consider the $G$-left-invariant \para-complex structure
$$
K_t \;:=\;
\left(
\begin{array}{cccc}
 -1 & 0 & 0 & 0\\
 0 & 1 & 0 & -2t \\
 0 & 0 & 1 & 0 \\
 0 & 0 & 0 & -1
\end{array}
\right) \;.
$$
For every $t\in\R$, we have that
\begin{equation*}
 \mathfrak{g}^+_{K_t}\;=\;\R\left\langle e_2, \, e_3 \right\rangle \quad \text{ and }\quad \mathfrak{g}^-_{K_t}\;=\;\R\left\langle e_1,\,e_4+t\,e_2\right\rangle \;:
\end{equation*}
in particular, $\left[\mathfrak{g}^+_{K_t},\,\mathfrak{g}^+_{K_t}\right]=\R\left\langle e_3\right\rangle\subseteq \mathfrak{g}^+_{K_t}$ and $\left[\mathfrak{g}^-_{K_t},\,\mathfrak{g}^-_{K_t}\right]=\left\{0\right\}$, which proves the integrability of $K_t$, for every $t\in\R$.

In particular, for $t=0$, we have the (non-Abelian) \para-complex structure
$$ K_0 \;=\; \left( - \, + \, + \, -\right) \;. $$
It is straightforward to check that $K_0$ is linear-\Cpf\ at the \kth{2} stage, and hence \Cpf\ at the \kth{2} stage by Corollary \ref{cor:linear-cpf-invariant-cpf-K}: in fact, by Proposition \ref{prop:linear-cpf-invariant-cpf-K},
\begin{equation*}
 H^{2\,+}_{K_0}(X;\R) \;=\; \left\{0\right\} \qquad \text{ and } \qquad  H^{2\,-}_{K_0}(X;\R) \;=\; H^2_{dR}(X;\R) \;;
\end{equation*}
in particular, we have
$$ \dim_\R H^{2\,+}_{K_0}(X;\R) \;=\; 0\;, \qquad \dim_\R H^{2\,-}_{K_0}(X;\R) \;=\; 2.$$
Furthermore, for $t\neq0$, we get
$$
H^{2\,-}_{K_t}\left(X;\R\right)\ni\left[e^{34}\right] \;=\; \left[e^{34}+\frac{1}{t}\,\de e^{3}\right] \;=\; \left[e^{34}+\frac{1}{t}\,(e^{23}+t\,e^{43}-t\,e^{43})\right] \;=\; \left[\frac{1}{t}\,(e^2-t\,e^4)\wedge e^3\right] \in H^{2\,+}_{K_t}\left(X;\R\right) \;,
$$
and therefore $0 \neq \left[e^{34}\right] \in H^{2\,-}_{K_t}\left(X;\R\right) \cap H^{2\,+}_{K_t}\left(X;\R\right)$, namely, $K_t$ is neither \Cp\ at the \kth{2} stage nor \Cf\ at the \kth{2} stage, by Proposition \ref{prop:para-cpf-pf} (in fact, since the space of $G$-left-invariant $\de$-exact $2$-forms is
$$ \de\wedge^1\duale{\mathfrak{g}} \;=\; \R\left\langle (e^2-t\,e^4)\wedge e^3-t\,e^{34},\, (e^2-t\,e^4)\wedge e^4\right\rangle \;,$$
no $G$-left-invariant representative of the class $\left[e^{12}\right]=\left[e^1\wedge(e^2-te^4)+te^{14}\right]$ is of pure type with respect to $K_t$). Therefore, for $t\neq 0$, we have
$$ \dim_\R H^{2\,+}_{K_t}(X;\R) \;=\; 1\;, \qquad \dim_\R H^{2\,-}_{K_t}(X;\R) \;=\; 1.$$

Note that, for every $t\in\R$, one has $s\left(\mathfrak{g}^-_{K_t}\right)=0$ and $s\left(\mathfrak{g}^+_{K_t}\right)=1$, but, for $t\neq0$, the \para-complex structure $K_t$ is not \Cp\ at the \kth{2} stage: therefore the Abelian condition on just $\mathfrak{g}^-$ in Theorem \ref{thm 2.12} is not sufficient to have \Cp ness at the \kth{2} stage, as observed in Remark \ref{rem:partial-abelian}.

Note moreover that, in this example, the functions
$$ \R\ni t\mapsto \dim_\R H^{2\,+}_{K_t}(X;\R)\in\N \qquad \text{ and } \qquad \R\ni t\mapsto \dim_\R H^{2\,-}_{K_t}(X;\R)\in\N $$
are, respectively, lower-semi-continuous and upper-semi-continuous.

Furthermore, we note that $X$ admits a ($G$-left-invariant) symplectic form
$$ \omega \;:=\; e^{12}+e^{34} \;, $$
which is compatible with the \para-complex structure $K_0$: therefore, $\left(K_0,\,\omega\right)$ is a \para-K\"ahler structure on $X$. On the other hand, for $t\neq0$, one has $H^-_{K_t}\left(X;\R\right)=\R\left\langle e^{34}\right\rangle$ and therefore, if a $K_t$-compatible symplectic form $\omega_t$ existed, then it should be in the same cohomology class as $e^{34}$, and then it should satisfy
$$ \Vol(X) \;=\; \int_X \omega_t\wedge\omega_t \;=\; \int_X e^{34}\wedge e^{34} \;=\; 0 \;, $$
which is not possible;
therefore, for $t\neq0$, there is
no symplectic structure compatible with the \para-complex structure $K_t$: in particular, $\left(X,\,K_t\right)$ admits no \para-\kal\ structure.
\end{ex}

In particular, the previous example proves the following result, \cite[Theorem 4.2]{angella-rossi}, providing another strong difference between the \para-complex and the complex cases (recall that being K\"ahler is a stable property under small deformations of the complex structure by K. Kodaira and D.~C. Spencer's stability theorem \cite[Theorem 15]{kodaira-spencer-3}).

\begin{thm}
\label{thm:para-kahler-deformations}
 The property of being \para-K\"ahler is not stable under small deformations of the \para-complex structure.
\end{thm}

Furthermore, Example \ref{es 2.17} proves also the following instability result, \cite[Proposition 4.3]{angella-rossi}, analogous to Theorem \ref{thm:instability}, which proves the instability of \Cpf ness in the complex case.

\begin{prop}
\label{prop:para-cpf-instability}
 The properties of being \Cp\ at the \kth{2} stage, or \Cf\ at the \kth{2} stage are not stable under small deformations of the \para-complex structure.
\end{prop}

\medskip

We have already recalled, see \S\ref{subsec:semicontinuity}, that T. Dr\v{a}ghici, T.-J. Li, and W. Zhang proved in \cite[Theorem 2.6]{draghici-li-zhang-2} that, given a curve $\left\{J_t\right\}_{t}$ of (\Cpf) almost-complex structures on a $4$-dimensional compact manifold $X$, the dimension of $H^+_{J_t}(X;\R)$ is upper-semi-continuous in $t$ and hence, as a consequence of \cite[Theorem 2.3]{draghici-li-zhang}, the dimension of $H^-_{J_t}(X;\R)$ is lower-semi-continuous in $t$; this result holds no more true for almost-complex manifolds of higher dimension, Proposition \ref{prop:no-sci-h+}, Proposition \ref{prop:no-scs-h-}.
In the next two examples, \cite[Example 4.4]{angella-rossi}, respectively \cite[Example 4.5]{angella-rossi}, we study the behaviour of the dimensions of the \para-complex invariant and \para-complex anti-invariant subgroups of the cohomology along curves of \para-complex structures.

\begin{ex}
\label{ex:def-jump-sci}
{\itshape A curve of \para-complex structures on a $6$-dimensional nilmanifold such that the dimensions of the \para-complex invariant and anti-invariant subgroups of the second de Rham cohomology group jump (lower-semi-continuously) along the curve.}\\
Consider a $6$-dimensional nilmanifold
$$ X \;=\; \left.\Gamma\right\backslash G \;:=\; \left(0^3,\;12,\;13,\;24\right) \;.$$

By K. Nomizu's theorem \cite[Theorem 1]{nomizu}, it is straightforward to compute
$$ H^2_{dR}(X;\R) \;=\; \R\left\langle e^{14},\, e^{15},\, e^{23},\, e^{26},\, e^{35},\, e^{25}+e^{34} \right\rangle \;.$$

For every $t\in\left[0,\,1\right]$, consider the left-invariant \para-complex structure
$$
K_t \;:=\;
\left(
\begin{array}{cc|cc|cc}
 1 & & & & & \\
 & -1 & & & & \\
\hline
 & & \frac{(1-t)^2-t^2}{(1-t)^2+t^2} & \frac{2t(1-t)}{(1-t)^2+t^2} & & \\ 
 & & \frac{2t(1-t)}{(1-t)^2+t^2} & -\frac{(1-t)^2-t^2}{(1-t)^2+t^2} & & \\ 
\hline
 & & & & 1 & \\
 & & & & & -1
\end{array}
\right) \;.
$$
For $0\leq t\leq 1$, one checks that
$$ \mathfrak{g}^+_{K_t} \;=\; \R\left\langle e_1,\, (1-t)\,e_3+t\,e_4,\, e_5 \right\rangle \quad \text{ and } \quad \mathfrak{g}^-_{K_t} \;=\; \R\left\langle e_2,\, t\,e_3-(1-t)\,e_4,\, e_6 \right\rangle \;;$$
therefore, it is straightforwardly checked that the integrability condition of $K_t$ is satisfied for every $t\in\left[0,\,1\right]$.
\begin{description}
 \item[{[Case $t=0$]}] For $t=0$, the \para-complex structure
 $$ K_0 \;=\; \left( + \, - \, + \, - \, + \, - \right) $$
 is \Cpf\ at the \kth{2} stage: in fact,
 $$ H^2_{dR}(X;\R) \;=\; \underbrace{\R\left\langle e^{15},\,e^{26},\,e^{35}\right\rangle}_{=\;H^{2\,+}_{K_0}(X;\R)}\,\oplus\,\underbrace{\R\left\langle e^{14},\,e^{23},\,e^{25}+e^{34} \right\rangle}_{=\;H^{2\,-}_{K_0}(X;\R)} \;;$$
 therefore
 $$ \dim_\R H^{2\,+}_{K_0}(X;\R) \;=\; 3 \qquad \text{ and } \qquad \dim H^{2\,-}_{K_0}(X;\R) \;=\; 3 \;. $$
 \item[{[Case $0<t<1$]}] For $0<t<1$, one has
 $$ H^{2\,+}_{K_t}(X;\R) \;=\; \R\left\langle e^{14},\, e^{15},\, e^{23},\, e^{26}\right\rangle $$
 and
 $$ H^{2\,-}_{K_t}(X;\R) \;=\; \R\left\langle e^{14},\, e^{23},\, e^{25}+e^{34}\right\rangle \;; $$
 it follows that the \para-complex structure $K_t$ is neither \Cp\ at the \kth{2} stage nor \Cf\ at the \kth{2} stage; moreover,
 $$ \dim_\R H^{2\,+}_{K_t}(X;\R) \;=\; 4 \qquad \text{ and } \qquad \dim H^{2\,-}_{K_t}(X;\R) \;=\; 3 \;. $$
 \item[{[Case $t=1$]}] For $t=1$, the \para-complex structure
 $$ K_1 \;=\; \left( + \, - \, - \, + \, + \, - \right) $$
 is \Cpf\ at the \kth{2} stage: in fact,
 $$ H^2_{dR}(X;\R) \;=\; \underbrace{\R\left\langle e^{14},\,e^{15},\,e^{23},\,e^{26}\right\rangle}_{=\;H^{2\,+}_{K_1}(X;\R)}\,\oplus\,\underbrace{\R\left\langle e^{35},\,e^{25}+e^{34} \right\rangle}_{=\;H^{2\,-}_{K_1}(X;\R)} \;;$$
 therefore
 $$ \dim_\R H^{2\,+}_{K_1}(X;\R) \;=\; 4 \qquad \text{ and } \qquad \dim H^{2\,-}_{K_1}(X;\R) \;=\; 2 \;. $$
\end{description}

In particular, it follows that the functions
$$ \left[0,\,1\right] \ni t \mapsto \dim_\R H^{2\,+}_{K_t}(X;\R) \in \N  \quad \text{ and } \quad \left[0,\,1\right] \ni t \mapsto \dim_\R H^{2\,-}_{K_t}(X;\R) \in \N  $$
are non-constant and lower-semi-continuous.
\end{ex}

Example \ref{es 2.17} and Example \ref{ex:def-jump-sci} show that the dimension of the \para-complex anti-invariant subgroup of the de Rham cohomology in general is not upper-semi-continuous (as it is in Example \ref{es 2.17}) or lower-semi-continuous (as it is in Example \ref{ex:def-jump-sci}) along curves of \para-complex structures. We give now an example showing that also the dimension of the \para-complex invariant subgroup of the de Rham cohomology in general is not lower-semi-continuous (as it is in Example \ref{es 2.17} and in Example \ref{ex:def-jump-sci}) along curves of \para-complex structures, \cite[Example 4.5]{angella-rossi}.

\begin{ex}
\label{ex:def-jump-scs}
{\itshape A curve of \para-complex structures on a $6$-dimensional nilmanifold such that the dimensions of the \para-complex invariant and anti-invariant subgroups of the second de Rham cohomology group jump (upper-semi-continuously) along the curve.}\\
Consider a $6$-dimensional nilmanifold
$$ X \;=\; \left.\Gamma\right\backslash G \;:=\; \left(0^3,\;12,\;13,\;24\right) \;.$$

By K. Nomizu's theorem \cite[Theorem 1]{nomizu}, it is straightforward to compute
$$ H^2_{dR}(X;\R) \;=\; \R\left\langle e^{14},\; e^{15},\; e^{23},\; e^{26},\; e^{35},\; e^{25}+e^{34} \right\rangle \;.$$

For every $t\in\left[0,\,1\right]$, consider the $G$-left-invariant \para-complex structure
$$
K_t \;:=\;
\left(
\begin{array}{cccc|cc}
 1 & & & & & \\
 & -1 & & & & \\
 & & -1 & & & \\
 & & & 1 & & \\
 \hline
 & & & & \frac{(1-t)^2-t^2}{(1-t)^2+t^2} & \frac{2t(1-t)}{(1-t)^2+t^2} \\
 & & & & \frac{2t(1-t)}{(1-t)^2+t^2} & -\frac{(1-t)^2-t^2}{(1-t)^2+t^2} \\
\end{array}
\right) \;.
$$
For $0\leq t\leq 1$, one checks that
$$ \mathfrak{g}^+_{K_t} \;=\; \R\left\langle e_1,\, e_4,\, (1-t)\,e_5+t\, e_6 \right\rangle \quad \text{ and } \quad \mathfrak{g}^-_{K_t} \;=\; \R\left\langle e_2,\, e_3,\, t\, e_5-(1-t)\, e_6 \right\rangle \;.$$
Therefore one straightforwardly checks that, for every $t\in\left[0,\,1\right]$, the structure $K_t$ is integrable, in fact Abelian: hence it is in particular \Cp\ at the \kth{2} stage by Corollary \ref{cor:abelian}.
\begin{description}
 \item[{[Case $t=0$]}] For $t=0$, the \para-complex structure
 $$ K_0 \;=\; \left( + \, - \, - \, + \, + \, - \right) $$
 is \Cpf\ at the \kth{2} stage, and
 $$ H^2_{dR}(X;\R) \;=\; \underbrace{\R\left\langle e^{14},\,e^{15},\,e^{23},\, e^{26}\right\rangle}_{=\;H^{2\,+}_{K_0}(X;\R)}\,\oplus\,\underbrace{\R\left\langle e^{35},\,e^{25}+e^{34} \right\rangle}_{=\;H^{2\,-}_{K_0}(X;\R)} \;;$$
 in particular,
 $$ \dim_\R H^{2\,+}_{K_0}(X;\R) \;=\; 4 \qquad \text{ and } \qquad \dim H^{2\,-}_{K_0}(X;\R) \;=\; 2 \;. $$
 \item[{[Case $0<t<1$]}] For $0<t<1$, one has
 $$ H^{2\,+}_{K_t}(X;\R) \;=\; \R\left\langle e^{14},\, e^{23} \right\rangle $$
 and
 $$ H^{2\,-}_{K_t}(X;\R) \;=\; \R\left\langle t\, e^{26}+(1-t)\, e^{25}+(1-t)\, e^{34} \right\rangle \;, $$
 while
 $$ \R\left\langle e^{15},\, e^{35},\, e^{26} \right\rangle \,\cap\, \left(H^{2\,+}_{K_t}(X;\R)\oplus H^{2\,-}_{K_t}(X;\R)\right)\;=\;\left\{0\right\} \;; $$
 it follows that the \para-complex structure $K_t$ is \Cp\ at the \kth{2} stage and non-\Cf\ at the \kth{2} stage, and that
 $$ \dim_\R H^{2\,+}_{K_t}(X;\R) \;=\; 2 \qquad \text{ and } \qquad \dim H^{2\,-}_{K_t}(X;\R) \;=\; 1 \;. $$
 \item[{[Case $t=1$]}] For $t=1$, the \para-complex structure
 $$ K_1 \;=\; \left( + \, - \, - \, + \, - \, + \right) \;.$$
 is \Cp\ at the \kth{2} stage and non-\Cf\ at the \kth{2} stage, and
 $$ H^2_{dR}(X;\R) \;=\; \underbrace{\R\left\langle e^{14},\,e^{23},\,e^{35} \right\rangle}_{=\;H^{2\,+}_{K_1}(X;\R)}\,\oplus\,\underbrace{\R\left\langle e^{15},\,e^{26} \right\rangle}_{=\;H^{2\,-}_{K_1}(X;\R)} \,\oplus\, \R\left\langle e^{25}+e^{34}\right\rangle \;,$$
 where
 $$ \R\left\langle e^{25}+e^{34}\right\rangle \,\cap\,\left(H^{2\,+}_{K_1}(X;\R)\oplus H^{2\,-}_{K_1}(X;\R)\right)\;=\;\left\{0\right\} \;; $$
 in particular,
 $$ \dim_\R H^{2\,+}_{K_1}(X;\R) \;=\; 3 \qquad \text{ and } \qquad \dim H^{2\,-}_{K_1}(X;\R) \;=\; 2 \;. $$
\end{description}

In particular, the functions
$$ \left[0,\,1\right] \ni t \mapsto \dim_\R H^{2\,+}_{K_t}(X;\R) \in \N  \quad \text{ and } \quad \left[0,\,1\right] \ni t \mapsto \dim_\R H^{2\,-}_{K_t}(X;\R) \in \N  $$
are non-constant and upper-semi-continuous.
\end{ex}

Example \ref{ex:def-jump-sci} and Example \ref{ex:def-jump-scs} yield the following result, \cite[Proposition 4.6]{angella-rossi}.

\begin{prop}
\label{prop:para-semi-cont}
 Let $X$ be a compact manifold and let $\left\{K_t\right\}_{t\in I \subseteq \R}$ be a curve of \para-complex structures on $X$. Then, in general, the functions
$$ I\ni t \mapsto \dim_\R H^{2\,+}_{K_t}(X;\R)\in\N \qquad \text{ and } \qquad  I\ni t \mapsto \dim_\R H^{2\,-}_{K_t}(X;\R)\in\N $$
are not upper-semi-continuous or lower-semi-continuous.
\end{prop}

\section{Cohomology of strictly $p$-convex domains in $\R^n$}\label{sec:p-convex}

In Complex Analysis, properties concerning the existence of exhaustion functions with convexity properties may have consequences on the vanishing of the cohomology.
Indeed, recall, for example, that the Dolbeault cohomology groups $H^{p,q}_\delbar(D)$ of a \emph{strictly pseudo-convex} domain $D$ in $\C^n$ (that is, a domain admitting a smooth proper strictly pluri-sub-harmonic exhaustion function) vanish for $q\geq 1$, for any $p\in\N$. In fact, the following result holds.

\begin{thm}[{\cite[Theorem 2.2.4, Theorem 2.2.5]{hormander-acta}, \cite[Theorem 4.2.2, Corollary 4.2.6]{hormander}}]
 Let $D\subseteq \C^n$ be a strictly pseudo-convex domain.
 Then, for any $q>0$, every $\delbar$-closed $(p,q)$-form $\eta\in\Lebloc{p,q}$ (respectively, $\eta\in\wedge^{p,q}X$) is $\delbar$-exact, namely, there exists $\alpha\in\Lebloc{p,q-1}$ (respectively, $\alpha\in\wedge^{p,q-1}X$) such that $\eta=\delbar\alpha$.
\end{thm}

Generalizing the notion of strictly pseudo-convex domain, A. Andreotti and H. Grauert, \cite{andreotti-grauert}, studied \emph{$q$-complete} domains in $\C^n$ (that is, domains in $\C^n$ admitting a smooth proper exhaustion function whose Levi form has at least $n-q+1$ positive eigen-values), and provided an analogue of the L. H\"ormander theorem.

\begin{thm}[{\cite[Proposition 27]{andreotti-grauert}, \cite[Theorem 5]{andreotti-vesentini}}]
 Let $D\in\C^n$ be a $q$-complete domain. Then $H^{r,s}_{\delbar}(X)=\{0\}$, for any $r\in\N$ and for any $s\geq q$.
\end{thm}

Recently, F.~R. Harvey and H.~B. Lawson, \cite{harvey-lawson-1, harvey-lawson-2}, and references therein, raised the interest on generalizations of the concept of convexity for Riemannian manifolds, studying the existence of exhaustion functions whose Hessian is positive definite or satisfies weaker positivity conditions; in this context, holomorphic convexity and $q$-completeness motivate the notion of \emph{geometric convexity}.

J.-P. Sha, \cite[Theorem 1]{sha}, and H. Wu, \cite[Theorem 1]{wu-indiana}, (see also \cite[Proposition 5.7]{harvey-lawson-2},) proved, using Morse theory, that the existence of a smooth proper strictly $p$-pluri-sub-harmonic exhaustion function on a domain in $\R^n$ has consequences on the homotopy type of the domain; hence, vanishing results for the de Rham cohomology hold for strictly $p$-convex domains in $\R^n$ in the sense of F.~R. Harvey and H.~B. Lawson.

In this section, we re-prove, using different techniques, the vanishing result by J.-P. Sha, and H. Wu for the de Rham cohomology of strictly $p$-convex domains in $\R^n$ in the sense of F.~R. Harvey and H.~B. Lawson;
more precisely, we use the $\mathrm{L}^2$-techniques developed by L. H\"ormander, \cite{hormander-acta}, and used also by A. Andreotti and E. Vesentini, \cite{andreotti-vesentini, andreotti-vesentini-erratum} (see also \cite{vesentini-tata}); such $\mathrm{L}^2$-techniques could be hopefully applied in a wider context.

The results in this section have been obtained in a joint work with S. Calamai, \cite{angella-calamai}.

\subsection{The notion of $p$-convexity by F.~R. Harvey and H.~B. Lawson}

In this section, following F.~R. Harvey and H.~B. Lawson, \cite{harvey-lawson-2, harvey-lawson-1}, we recall the basic notions and results concerning $p$-convexity, starting from the definition of $p$-positive symmetric endomorphism, and then recalling the notions of (strictly) $p$-pluri-sub-harmonic function and (strictly) $p$-convex domain.

\medskip

Let $V$ be an $n$-dimensional $\R$-vector space endowed with an inner product $\scalardL{\sspace}{\ssspace}$.

Let $G\colon V\stackrel{\simeq}{\to} \duale{V}$ denote the isomorphism induced by $\scalardL{\sspace}{\ssspace}$, defined as
$$ G\colon V \;\ni\; v \mapsto \scalardL{v}{\sspace} \;\in\; \duale{V} \;.$$
One gets an isomorphism
$$ G^{-1}\colon \duale{V}\otimes\duale{V}\stackrel{\simeq}{\to}\Homom{V}{V} \;;$$
this isomorphism sends the space of the symmetric elements of $\duale{\left(V \otimes V\right)}$, namely,
$$ \Sym{V} \;:=\; \left\{A\in\duale{\left(V \otimes V\right)} \st A(v \otimes w) = A(w \otimes v)\;, \text{ for any }v,w\in V \right\} \;.$$
to the space of the $\scalardL{\sspace}{\ssspace}$-symmetric endomorphisms of $V$.

Given $A\in\duale{V}\otimes\duale{V}$, the endomorphism $G^{-1}A\in\Homom{V}{V}$ extends to
$$ D_{G^{-1}A}^{[p]}\in\Homom{\wedge^p V}{\wedge^p V} $$
by setting, for any simple vector $v_{i_1} \wedge \cdots \wedge v_{i_p}\in\wedge^pV$,
\begin{align*}
 D_{G^{-1}A}^{[p]} \left( v_{i_1} \wedge \cdots \wedge v_{i_p} \right) \;:=\; 
 \sum_{\ell=1}^{p} v_{i_1}\wedge\cdots\wedge v_{i_{\ell-1}}\wedge G^{-1}A \left(v_{i_\ell}\right) \wedge v_{i_{\ell+1}}\wedge\cdots\wedge v_{i_p} \;;
\end{align*}
note that $D_{G^{-1}A}^{[p]}\in\Homom{\wedge^p V}{\wedge^p V}$ is a symmetric endomorphism with respect to the inner product on $\wedge^p V$ induced by $\scalar{\sspace}{\ssspace}$.

Note that, given $A\in \Sym{V}$, if the set of the eigenvalues of $G^{-1}A$ is
$$ \spec\left(G^{-1}A\right) \;=\; \left\{\lambda_1, \, \ldots , \, \lambda_n \right\} \;,$$
then the set of the eigenvalues of $D_{G^{-1}A}^{[p]}$ is
$$ \spec\left(D_{G^{-1}A}^{[p]}\right) \;=\; \left\{ \lambda_{i_1} + \cdots + \lambda_{i_p} \st i_1, \ldots, i_p \in \In{n}  \text{ s.t. } i_1 < \cdots < i_p \right\} \;. $$

Finally, given a $\scalar{\sspace}{\ssspace}$-symmetric endomorphism  $E\in \Homom{V}{V}$, let $\sgn{E}$ denote the number of non-negative eigenvalues of $E$:
$$ \sgn{E} \;:=\; \card\left\{\lambda \in \spec(E) \st \lambda\geq 0 \right\} \;.$$

Note that, given two inner products  on $V$ inducing the isomorphisms $G_1\colon V\stackrel{\simeq}{\to}\duale{V}$ and $G_2\colon V\stackrel{\simeq}{\to}\duale{V}$ respectively, then there holds $\sgn{G_1^{-1}A} = \sgn{G_2^{-1}A}$, but, for $p>1$, it might hold
$$ \sgn{D_{G_1^{-1}A}^{[p]}} \;\neq\;  \sgn{D_{G_2^{-1}A}^{[p]}} \;.$$

\medskip

As said, $\sgn{D^{[p]}_{G^{-1}A}}$ counts the non-negative sums of $p$ eigenvalues of $G^{-1}A\in \Homom{V}{V}$. As a natural generalization of the notion of convexity, one is interested in studying symmetric endomorphisms having at least a certain number of non-negative sums of $p$ eigenvalues. (Compare also \cite[Definition 2.1]{harvey-lawson-1}, concerning the notion of positivity with respect to a sub-bundle of the Grassmannian bundle $\Grass{\R}{p}{TX}$ over a Riemannian manifold $X$.)

\begin{defi}[{\cite[Definition 2.1, \S3]{harvey-lawson-2}}]
\mbox{}\\
\begin{itemize}
 \item Let $V$ be an $n$-dimensional $\R$-vector space endowed with an inner product $\scalardL{\sspace}{\ssspace}$.
 For $p\in\In{n}$, and for $k\in\In{\binom{n}{p}}$, define the space of \emph{$p$-positive forms of \kth{k} branch} on $V$ as
\begin{align*}
 \mathcal{P}_p^{(k)}\left(V,\, \scalardL{\sspace}{\ssspace}\right):= 
\left\{ 
A \in \Sym{V} \st \sgn{ D_{G^{-1}A}^{[p]}} \geq \binom{n}{p} -k +1  
\right\}\;.
\end{align*}
 \item Let $X$ be an $n$-dimensional manifold endowed with a Riemannian metric $g$.
For $p\in\In{n}$, and for $k\in\In{\binom{n}{p}}$, define the space of \emph{$p$-positive sections of \kth{k} branch} of the bundle $\Sym{TX}$ of symmetric endomorphisms of $TX$ as
\begin{align*}
 \mathcal{P}_p^{(k)}\left(X,\, g \right) \;:=\;
\left\{
 A \in \Sym{T X} \st \forall x\in X, \; A_x\in\mathcal{P}_p^{(k)}\left(T_xX, g_x\right) 
\right\}.
\end{align*}
\end{itemize}
\end{defi}

\medskip

In order to introduce exhaustion functions on a given Riemannian manifold, we focus on special $p$-positive symmetric $2$-forms: those arising from the Hessian of smooth functions.

Let $\left(X,\, g\right)$ be a Riemannian manifold, and denote the Levi Civita connection associated to the Riemannian metric $g$ by $\nabla^{LC}$. For every $u\in\Cinfk{0}$, let
$$ \Hess u \;\in\; \Sym{TX} $$
be defined, for any $V, W\in\mathcal{C}^\infty(X;TX)$, as
$$ \Hess u \left(V, W\right) \;:=\; V\, W\, u - \left(\nabla^{LC}_V W\right)\, u \;. $$

\begin{defi}[{\cite[Definition 2.2', \S3]{harvey-lawson-2}}]
Let $X$ be an $n$-dimensional manifold endowed with a Riemannian metric $g$. Fix $p\in\In{n}$, and $k\in\In{\binom{n}{k}}$.
\begin{itemize}
 \item The space
\begin{align*}
\PSH^{(k)}_p\left(X,\, g\right) \;:=\;
\left\{ u\in \Cinfk{0} \st \Hess u \in \mathcal{P}^{(k)}_p \left( X,\, g\right) \right\} \;,
\end{align*}
is called the space of \emph{$p$-pluri-sub-harmonic functions of \kth{k} branch} on $X$.
\item The space
\begin{align*}
\intern{\PSH^{(k)}_p\left(X,\, g\right)} \;:=\;
\left\{ u\in \Cinfk{0} \st \Hess u \in \intern{\mathcal{P}^{(k)}_p \left( X,\, g\right)} \right\} \;,
\end{align*}
(where $\intern{\mathcal{P}^{(k)}_p \left( X,\, g\right)}$ denotes the interior of $\mathcal{P}^{(k)}_p \left( X,\, g\right)$) is called the space of \emph{strictly $p$-pluri-sub-harmonic functions of \kth{k} branch} on $X$.
\end{itemize}
\end{defi}

(Compare also \cite[Definition 2.1]{harvey-lawson-1} for the notion of (strictly) $p$-pluri-sub-harmonicity with respect to a sub-bundle of the Grassmannian bundle $\Grass{\R}{p}{TX}$ over a Riemannian manifold $X$.)

\medskip

We can now define (strictly) $p$-convexity in terms of the $p$-convex hulls (and of the $p$-core).

Let $X$ be an $n$-dimensional Riemannian manifold endowed with a Riemannian metric $g$, and fix $p\in\In{n}$. Let $K\subseteq X$ be a subset of $X$; the \emph{$p$-convex hull} of $K$, \cite[Definition 4.1]{harvey-lawson-2}, is defined as
\begin{align*}
{\widetilde{K}}^{\PSH^{(1)}_p\left( X ,\, g\right)} \;:=\;
\left\{ x\in X \st
\forall \phi \in PSH^{(1)}_p \left( X ,\, g\right) ,\;
\phi(x) \leq \max_{y\in K}   \phi (y)
\right\} \;.
\end{align*}
(Compare also \cite[Definition 4.3]{harvey-lawson-1} for the notion of convex hull with respect to a sub-bundle of the Grassmannian bundle $\Grass{\R}{p}{TX}$ over a Riemannian manifold $X$.)

\begin{defi}[{\cite[Definition 4.3]{harvey-lawson-2}}]\label{defi:p-convexity}
Let $X$ be an $n$-dimensional Riemannian manifold endowed with a Riemannian metric $g$, and fix $p\in\In{n}$.
One says that $X$ is \emph{$p$-convex} if, for any subset $K\subseteq X$ that is relatively compact in $X$, then ${\widetilde{K}}^{\PSH^{(1)}_p\left( X ,\, g\right)}$ is relatively compact in $X$.
\end{defi}

(Compare also \cite[Definition 4.5]{harvey-lawson-1} for the notion of convexity with respect to a sub-bundle of the Grassmannian bundle $\Grass{\R}{p}{TX}$ over a Riemannian manifold $X$.)

Define the \emph{$p$-core} of $X$, \cite[Definition 5.3]{harvey-lawson-2}, as
$$ \mathrm{Core}_p\left(X,\, g\right) \;:=\; \left\{ x\in X \st \text{for all }u\in\PSH^{(1)}_p\left(X,\, g\right),\; \Hess u(x)\not\in\intern{\mathcal{P}^{(1)}_p\left(T_xX,\, g_x\right)} \right\} \;. $$
(Compare also \cite[Definition 4.1]{harvey-lawson-1} for the notion of core with respect to a sub-bundle of the Grassmannian bundle $\Grass{\R}{p}{TX}$ over a Riemannian manifold $X$.)

\begin{defi}[{\cite[Definition 5.2, Theorem 5.4]{harvey-lawson-2}}]
 Let $X$ be an $n$-dimensional Riemannian manifold endowed with a Riemannian metric $g$, and fix $p\in\In{n}$.
 One says that the manifold $X$ is \emph{strictly $p$-convex} if
\begin{inparaenum}[(\itshape i\upshape)]
 \item $\mathrm{Core}_p\left(X,\, g\right)=\varnothing$, and
 \item for any subset $K\subseteq X$ that is relatively compact in $X$, then ${\widetilde{K}}^{\PSH^{(1)}_p\left( X ,\, g\right)}$ is relatively compact in $X$.
\end{inparaenum}
\end{defi}

(Compare also \cite[Definition 4.9]{harvey-lawson-1} for the notion of strictly convexity with respect to a sub-bundle of the Grassmannian bundle $\Grass{\R}{p}{TX}$ over a Riemannian manifold $X$.)

\medskip

The relations between (strictly) $p$-convexity and the existence of smooth proper (strictly) $p$-pluri-sub-harmonic exhaustion functions were proven by F.~R. Harvey and H.~B. Lawson in \cite{harvey-lawson-2, harvey-lawson-1}. Namely, the following result holds.

\begin{thm}[{\cite[Theorem 4.4, Theorem 5.4]{harvey-lawson-2}}]\label{thm:p-positive-exhaustion-functions}
 Let $X$ be an $n$-dimensional Riemannian manifold endowed with a Riemannian metric $g$, and fix $p\in\In{n}$.
 Then $X$ is $p$-convex (respectively, strictly $p$-convex) if and only if $X$ admits a smooth proper exhaustion function $u\in \PSH^{(1)}_p\left( X ,\, g\right)$ (respectively, $u\in \intern{\PSH^{(1)}_p\left( X ,\, g\right)}$).
\end{thm}

(Compare also \cite[Theorem 4.4, Theorem 4.8]{harvey-lawson-1} for the relations between (strictly) convexity and the existence of smooth proper (strictly) pluri-sub-harmonic exhaustion functions with respect to a sub-bundle of the Grassmannian bundle $\Grass{\R}{p}{TX}$ over a Riemannian manifold $X$.)

(We recall that a function $u\colon X\to \R$, where $X$ is a manifold, is said to be an \emph{exhaustion} if, for any $c\in\R$, the set $u^{-1}\left(\left(-\infty, \, c\right)\right)=\left\{x\in X \st u(x)<c\right\}\subseteq X$ is relatively compact in $X$.)

\medskip

The previous definitions are motivated by the classical notions of (strictly) ($q$-)pseudo-convex functions, and $q$-complete and pseudo-convex domains, in Complex Analysis.

\begin{defi}[{\cite[\S4]{andreotti}, \cite[\S10]{andreotti-grauert}}]
Let $D \subseteq \mathbb{C}^n$ be a domain, and let $\phi$ be a smooth real-valued function on $D$. The function $\phi$ is called \emph{$q$-pseudo-convex} or \emph{$q$-pluri-sub-harmonic} (respectively, \emph{strictly $q$-pseudo-convex} or \emph{strictly-$q$-pluri-sub-harmonic}), if and only if, for any $z\in D$, the Hermitian form $\Levi(\phi)_z$ defined, for $\xi:=:\left(\xi^a\right)_{a\in\In{n}} \in \mathbb{C}^n$, as
\begin{align*}
 \Levi(\phi)_z\,(\xi) \;:=\; \sum_{a , b =1}^n 
\frac{\partial^2 \phi }{\partial z^a \partial \bar z^{b}}(z)\, \xi^a\, \overline{\xi^{b}} \;,
\end{align*}
has, at least, $n-q+1$ non-negative (respectively, positive) eigenvalues. When $q=1$, (strictly) $1$-pseudo-convex functions are called \emph{(strictly) pseudo-convex}, or \emph{(strictly) pluri-sub-harmonic}.
\end{defi}

\begin{defi}[{\cite{rothstein}, \cite[\S16.c]{andreotti-grauert}}]
A domain $D \subseteq \mathbb{C}^n$ is called \emph{$q$-complete} if there exists a smooth proper strictly $q$-pseudo-convex exhaustion function. When $q=1$, the $1$-complete domains are called \emph{strictly pseudo-convex}.
\end{defi}

\medskip

A. Andreotti and H. Grauert, in \cite{andreotti-grauert}, proved a vanishing theorem for the higher-degree Dolbeault cohomology groups of $q$-complete domains; the same result was proven by A. Andreotti and E. Vesentini, in \cite{andreotti-vesentini}, (see also \cite[Theorem 4.2]{vesentini-tata},) extending the $\mathrm{L}^2$-techniques by L. H\"ormander, \cite{hormander-acta}. More precisely, \cite[Proposition 27]{andreotti-grauert}, and \cite[Theorem 5]{andreotti-vesentini}, state that, given a $q$-complete domain $D\in\C^n$, it holds $H^{r,s}_{\delbar}(X)=\{0\}$, for any $r\in\N$ and for any $s\geq q$.

\subsection{Vanishing of the de Rham cohomology for strictly $p$-convex domains}
In this section, motivated by A. Andreotti and H. Grauert's vanishing result for the Dolbeault cohomology of $q$-complete domains in $\C^n$, \cite{andreotti-grauert}, and by A. Andreotti and E. Vesentini's proof using $\mathrm{L}^2$-techniques, \cite{andreotti-vesentini}, we consider domains $X$ in $\R^n$ endowed with a proper exhaustion function $u\in\mathcal{C}^\infty(X;\R)$ whose Hessian is in $\intern{\mathcal{P}^{(1)}_p(X,g)}$, re-proving, with $\mathrm{L}^2$-techniques, the vanishing result for the higher-degree de Rham cohomology groups for strictly $p$-convex domains in the sense of F.~R. Harvey and H.~B. Lawson, Theorem \ref{thm:vanishing}, yet shown by J.-P. Sha, \cite{sha}, and by H. Wu, \cite[Theorem 1]{wu-indiana}, using Morse theory, as a consequence of results on the homotopy type of $X$.
Firstly, we recall some definitions and we set some notations; then, we prove some preliminary lemmata and estimates; finally we prove Theorem \ref{thm:cauchy}, stating that, on a strictly $p$-convex domain in $\R^n$, every $\de$-closed $k$-form with $k\geq p$ is $\de$-exact.

\medskip

Let $X$ be an oriented Riemannian manifold of dimension $n$, and denote its Riemannian metric by $g$ and its volume by $\vol$. The Riemannian metric $g$ induces, for every $x\in X$, a point-wise inner product $\left\langle \sspace \left| \ssspace \right. \right\rangle_{g_x}\colon \wedge^\bullet T^*_xX \times \wedge^\bullet T^*_xX \to \R$.

Fix $\phi\in\mathcal{C}^0(X;\R)$ a continuous function. For every $\varphi,\, \psi \in \Cinfkc{\bullet}$, let
$$ \scalarL{\varphi}{\psi}{\phi} \;:=\; \int_X \left\langle \varphi \left| \psi \right. \right\rangle_{g_x} \, \expp{-\phi} \vol \;\in\; \R \;, $$
and, for $k\in\N$, define $\Leb{k}{\phi}$ as the completion of the space $\Cinfkc{k}$ of smooth $k$-forms with compact support with respect to the metric induced by $\normaL{\sspace}{\phi}:=\scalarL{\sspace}{\sspace}{\phi}$.
Therefore, the space $\Leb{k}{\phi}$ is a Hilbert space, endowed with the inner product $\scalarL{\sspace}{\ssspace}{\phi}$, and $\Cinfkc{k}$ is dense in $\Leb{k}{\phi}$. For any $k\in\N$, let $\Lebloc{k}$ denote the space of $k$-forms $\varphi$ whose restriction $\varphi\lfloor_{K}$ to every compact set $K\subseteq X$ belongs to $\LebK{k}{}$.

\medskip

For every $\phi_1,\,\phi_2\in\mathcal{C}^0(X;\R)$, the operator
$$ \de\colon \Leb{\bullet}{\phi_1} \dashrightarrow \Leb{\bullet+1}{\phi_2} $$
is densely-defined and closed; denote by
$$ \destar{\phi_2}{\phi_1}\colon \Leb{\bullet+1}{\phi_2} \dashrightarrow \Leb{\bullet}{\phi_1} $$
its adjoint, which is a densely-defined closed operator, see, e.g., \cite[Theorem 7.55]{dellasala-saracco-simioniuc-tomassini}.

\medskip

Moreover, for a domain $X$ in $\R^n$, with set of coordinates $\left\{x^1, \ldots, x^n\right\}$, fixed $k\in\N$, $s\in\N$, and $\phi\in\Cinfk{0}$, one can consider the Sobolev space $\Sob{k}{\phi}{s}$, which is defined as the space of $k$-forms $\varphi:=:\ssum{|I|=k}\varphi_I\,\de x^I$ such that $\frac{\del^{\ell_1+\cdots+\ell_n} \varphi_I}{\del^{\ell_1} x^1\cdots \del^{\ell_n}x^n}\in\Leb{k}{\phi}$ for every multi-index $\left(\ell_1,\ldots,\ell_n\right)\in\N^n$ satisfying $\ell_1+\cdots+\ell_n\leq s$ and for every strictly increasing multi-index $I$ such that $|I|=k$. The space $\Sobloc{s}{k}$ is defined as the space of $k$-forms $\varphi$ whose restriction $\varphi\lfloor_K$ to every compact set $K\subseteq X$ belongs to $\SobK{k}{}{s}$.

\medskip

As a matter of notation, the symbol $\ssum{|I|=k}$ denotes the sum over the strictly increasing multi-indices $I:=:\left(i_1,\ldots,i_k\right)\in\N^k$ (that is, the multi-indices such that $0<i_1<\cdots<i_k$) of length $k$. We use $\left\{x^1, \ldots, x^n\right\}$ as a set of coordinates on $\R^n$, and, given a multi-index $I:=:\left(i_1,\ldots,i_k\right)\in\N^k$, we shorten $\de x^I:=\de x^{i_1}\wedge\cdots\wedge \de x^{i_k}$. Given $I_1$ and $I_2$ two multi-indices of length $k$, let $\sign{I_1}{I_2}$ be the sign of the permutation $\left(\begin{array}{c}I_1\\I_2\end{array}\right)$ if $I_1$ is a permutation of $I_2$, and zero otherwise.

\medskip

Let $X$ be a domain in $\R^n$, that is, an open connected subset of $\R^n$, endowed with the metric and the volume induced, respectively, by the Euclidean metric and the standard volume of $\R^n$.

For $\phi_1,\, \phi_2\in\Cinfk{0}$, consider $\de\colon \Leb{k-1}{\phi_1} \dashrightarrow \Leb{k}{\phi_2}$.
The following lemma gives an explicit expression of the adjoint $\destar{\phi_2}{\phi_1}\colon \Leb{k}{\phi_2} \dashrightarrow \Leb{k-1}{\phi_1}$, \cite[Lemma 2.1]{angella-calamai} (compare with, e.g., \cite[\S8.2.1]{dellasala-saracco-simioniuc-tomassini}, \cite[Lemma O.2]{gunning-1} in the complex case).

\begin{lem}
\label{lemma:d*}
 Let $X$ be a domain in $\R^n$. Let $\phi_1,\,\phi_2\in\Cinfk{0}$ and consider
 $$
 \xymatrix{
 \Leb{k-1}{\phi_1} \ar@{-->}@/^1pc/[r]^{\de} & \Leb{k}{\phi_2} \ar@{-->}@/^1pc/[l]^{\destar{\phi_2}{\phi_1}}
 } \;.
 $$
 Let
 $$v \;:=:\; \ssum{|I|=k} v_I\, \de x^I \in \Leb{k}{\phi_2} $$
 and suppose that $v\in\dom \destar{\phi_2}{\phi_1}$. Then
 \begin{eqnarray*}
 \destar{\phi_2}{\phi_1} v &=& \expp{\phi_1} \destar{0}{0}\left(\expp{-\phi_2} v\right) \\[5pt]
 &=& \ssum{|J|=k-1} \left(-\expp{\phi_1} \ssum{|I|=k}\sum_{\ell=1}^{n} \sign{\ell J}{I}\, \der{\left(v_I\expp{-\phi_2}\right)}{\ell} \right) \, \de x^J \;.
 \end{eqnarray*}
\end{lem}

\begin{proof}
By definition of $\destar{\phi_2}{\phi_1}$, for every $u\in\dom \de$, one has $\scalarL{\de u}{v}{\phi_2} = \scalarL{u}{\destar{\phi_2}{\phi_1} v}{\phi_1}$.
Hence, consider
$$ u \;:=:\; \ssum{|J|=k-1} u_J\, \de x^J \;\in\; \Cinfkc{k-1} \;, $$
and compute
$$ \de u \;=\; \ssum{\substack{|J|=k-1\\ |I|=k}}\sum_{\ell=1}^{n} \sign{\ell J}{I} \der{u_J}{\ell} \de x^I \;. $$
The statement follows by computing
\begin{eqnarray*}
\scalarL{\de u}{v}{\phi_2} &=& \int_X \ssum{\substack{|J|=k-1\\|I|=k}} \sum_{\ell=1}^{n} \sign{\ell J}{I} \der{u_J}{\ell} v_I \, \expp{-\phi_2}\, \vol \\[5pt]
 &=& -\int_X \ssum{\substack{|J|=k-1\\|I|=k}} \sum_{\ell=1}^{n} \sign{\ell J}{I} \der{\left(v_I \, \expp{-\phi_2}\right)}{\ell} u_J\, \vol
\end{eqnarray*}
and
$$ \scalarL{u}{\destar{\phi_2}{\phi_1} v}{\phi_1} \;=\; \int_X \ssum{|J|=k-1} \left(\destar{\phi_2}{\phi_1} v\right)_J\, u_J\, \expp{-\phi_1}\, \vol \;, $$
where $\destar{\phi_2}{\phi_1} v=:\ssum{|J|=k-1} \left(\destar{\phi_2}{\phi_1} v\right)_J\,\de x^J$.
\end{proof}

\medskip

For any fixed $\phi\in\Cinfk{0}$ and for any $j\in\In{n}$, define the operator
$$ \delta^\phi_j \colon \Cinfk{0} \to \Cinfk{0} \;, $$
where
$$ \delta^\phi_j(f) \;:=\; -\expp{\phi}\,\der{\left(f\,\expp{-\phi}\right)}{j} \;=\; \der{\phi}{j}\cdot f-\der{f}{j} \;. $$
The following lemma states that $\delta^\phi_j$ is the adjoint of $\der{}{j}$ in $\Leb{0}{\phi}$, and computes the commutator between $\delta^\phi_j$ and $\der{}{k}$, \cite[Lemma 2.2]{angella-calamai} (compare with, e.g., \cite[pages 83--84]{hormander}).

\begin{lem}
\label{lemma:delta}
 Let $X$ be a domain in $\R^n$. Let $\phi\in\Cinfk{0}$ and $j\in\In{n}$, and consider the operator $\delta^\phi_j\colon \Cinfk{0} \to \Cinfk{0}$. Then:
\begin{itemize}
 \item for every $w_1,w_2\in\Cinfkc{0}$,
$$ \int_X w_1\cdot\der{w_2}{k}\expp{-\phi}\vol \;=\; \int_X\delta^\phi_k(w_1)\cdot w_2\,\expp{-\phi}\vol \;;$$
 \item for any $k\in\In{n}$, the following commutation formula holds in $\Endom{\Cinfkc{0}}$:
$$ \left[\delta^\phi_j,\, \der{}{k}\right] \;=\; -\frac{\del^2\phi}{\del x^j\del x^k}\cdot \;.$$
\end{itemize}
\end{lem}

\begin{proof}
 As regards the first item, one has
 \begin{eqnarray*}
  \int_X w_1\cdot\der{w_2}{k}\expp{-\phi}\vol &=& -\int_X w_2 \cdot \frac{\del}{\del x^k}\left(w_1\,\expp{-\phi}\right) \vol \\[5pt]
    &=& \int_X w_2\cdot \left(w_1\,\der{\phi}{k}-\der{w_1}{k} \right) \, \expp{-\phi}\vol \\[5pt]
    &=& \int_X\delta^\phi_k(w_1)\cdot w_2\,\expp{-\phi}\vol \;.
 \end{eqnarray*}

 As regards the second item, one has, for every $f\in\Cinfk{0}$,
 \begin{eqnarray*}
  \left[\delta^\phi_j,\, \der{}{k}\right](f) &=& \delta^\phi_j\left(\frac{\del f}{\del x^k}\right)-\frac{\del}{\del x^k}\left(\delta^\phi_j(f)\right) \\[5pt]
    &=& \der{\phi}{j} \cdot \der{f}{k} - \frac{\del^2 f}{\del x^j \del x^k} - \frac{\del}{\del x^k}\left(\der{\phi}{j}\cdot f - \der{f}{j}\right)     \\[5pt]
    &=& \der{\phi}{j} \cdot \der{f}{k} - \frac{\del^2 f}{\del x^j \del x^k} - \frac{\del^2 \phi}{\del x^k \del x^j}\cdot f - \der{\phi}{j}\cdot \der{f}{k} + \frac{\del^2 f}{\del x^k\del x^j} \\[5pt]
    &=& - \frac{\del^2 \phi}{\del x^k \del x^j}\cdot f \;,
 \end{eqnarray*}
 concluding the proof of the lemma.
\end{proof}

\medskip

Finally, we prove the following estimate, \cite[Proposition 2.3, Remark 2.4]{angella-calamai}, which will be used in the proof of Theorem \ref{thm:cauchy} (we refer to \cite[\S4.2]{hormander}, or, e.g., \cite[Lemma O.3]{gunning-1} and \cite[\S8.3.1]{dellasala-saracco-simioniuc-tomassini} for its complex counterpart).

\begin{prop}
\label{prop:stima}
 Let $X$ be a domain in $\R^n$ and $\phi,\,\psi\in\Cinfk{0}$. Consider
 $$
 \xymatrix{
 \Leb{k-1}{\phi-2\psi} \ar@{-->}@/^1pc/[r]^{\de} & \Leb{k}{\phi-\psi} \ar@{-->}@/^1pc/[r]^{\de} \ar@{-->}@/^1pc/[l]^{\destar{\phi-\psi}{\phi-2\psi}} & \Leb{k+1}{\phi} \ar@{-->}@/^1pc/[l]^{\destar{\phi}{\phi-\psi}} \;.
 }
 $$
 Then, for any $\eta :=: \ssum{|I|=k}\eta_I \, \de x^I \in \Cinfkc{k}$, one has
\begin{eqnarray*}
 \lefteqn{\int_X \ssum{\substack{|J|=k-1\\|I_1|=k\\|I_2|=k}}\sum_{\ell_1,\,\ell_2=1}^{n} \sign{\ell_1 J}{I_1}\sign{\ell_2 J}{I_2} \frac{\del^2\phi}{\del x^{\ell_1}\,\del x^{\ell_2}}\, \eta_{I_1}\, \eta_{I_2} \, \expp{-\phi} \vol} \\[5pt]
 &\leq& \;  \lefteqn{\int_X \left(\ssum{\substack{|J|=k-1\\|I_1|=k\\|I_2|=k}}\sum_{\ell_1,\,\ell_2=1}^{n} \sign{\ell_1 J}{I_1}\sign{\ell_2 J}{I_2} \frac{\del^2\phi}{\del x^{\ell_1}\,\del x^{\ell_2}}\, \eta_{I_1}\, \eta_{I_2}+\ssum{|I|=k}\sum_{\ell=1}^{n}\left|\frac{\del \eta_I}{\del x^\ell}\right|^2\right)} \\[5pt]
 && \lefteqn{\cdot \expp{-\phi} \vol} \\[5pt]
 &\leq& \; C \cdot \left( \normaL{\destar{\phi-\psi}{\phi-2\psi} \eta}{\phi-2\psi}^2 + \normaL{\de\eta}{\phi}^2 + \int_X \ssum{|I|=k}\sum_{\ell=1}^{n}\left|\frac{\del\psi}{\del x^\ell}\right|^2\,\left|\eta_I\right|^2\, \expp{-\phi} \vol \right) \;,
\end{eqnarray*}
where $C:=:C(k,n)\in\N$ is a constant depending just on $k$ and $n$.
\end{prop}

\begin{proof}
It is straightforward to compute
$$ \de \eta \;=\; \ssum{\substack{|I|=k\\|H|=k+1}}\sum_{\ell=1}^{n} \sign{\ell I}{H}\der{\eta_I}{\ell}\de x^H $$
and, using Lemma \ref{lemma:d*},
\begin{eqnarray*}
 \destar{\phi-\psi}{\phi-2\psi} \eta &=& -\expp{-\psi} \ssum{\substack{|J|=k-1\\|I|=k}}\sum_{\ell=1}^{n}\sign{\ell J}{I} \left(\der{\eta_I}{\ell}-\der{\left(\phi-\psi\right)}{\ell}\eta_I\right) \, \de x^J \\[5pt]
 &=& \expp{-\psi} \ssum{\substack{|J|=k-1\\|I|=k}}\sum_{\ell=1}^{n}\sign{\ell J}{I} \left(\delta^\phi_\ell\left(\eta_I\right)-\der{\psi}{\ell}\,\eta_I\right) \, \de x^J \;.
\end{eqnarray*}
For every $J$ such that $|J|=k-1$, the previous equality gives
$$
\ssum{|I|=k}\sum_{\ell=1}^{n}\sign{\ell J}{I} \delta^\phi_\ell\left(\eta_I\right)\;=\; \expp{\psi} \left(\destar{\phi-\psi}{\phi-2\psi} \eta\right)_J  + \ssum{|I|=k}\sum_{\ell=1}^{n}\sign{\ell J}{I} \der{\psi}{\ell}\,\eta_I \;,
$$
where $\destar{\phi-\psi}{\phi-2\psi}\eta =: \ssum{|J|=k-1}\left(\destar{\phi-\psi}{\phi-2\psi}\eta\right)_J\de x^J$.

By the inequality between the geometric mean and the arithmetic mean, one gets
\begin{eqnarray}
 \lefteqn{\int_X \ssum{|J|=k-1} \left| \ssum{|I|=k}\sum_{\ell=1}^{n}\sign{\ell J}{I} \delta^\phi_\ell\left(\eta_I\right) \right|^2 \, \expp{-\phi} \vol} \nonumber\\[5pt]
 && \leq\; 2\, \int_X \ssum{|J|=k-1}\left(\left|\left(\destar{\phi-\psi}{\phi-2\psi}\eta\right)_J\right|^2\, \expp{2\psi} + \left|\ssum{|I|=k}\sum_{\ell=1}^n\sign{\ell J}{I} \der{\psi}{\ell}\eta_I\right|^2 \right)\, \expp{-\phi} \vol \nonumber\\[5pt]
 && \leq\; C \, \left(\normaL{\destar{\phi-\psi}{\phi-2\psi}\eta}{\phi-2\psi}^2+\int_X\ssum{|I|=k}\sum_{\ell=1}^{n}\left|\der{\psi}{\ell}\right|^2\cdot\left|\eta_I\right|^2\, \expp{-\phi} \vol\right) \;,\label{eq:stima-delta}
\end{eqnarray}
where $C:=:C(k,n)\in\N$ depends on $k$ and $n$ only.

Now, using Lemma \ref{lemma:delta}, one computes
\begin{eqnarray}
 \lefteqn{\int_X \ssum{|J|=k-1} \left| \ssum{|I|=k}\sum_{\ell=1}^{n} \sign{\ell J}{I} \delta^\phi_\ell\left(\eta_I\right) \right|^2 \, \expp{-\phi} \vol} \nonumber\\[5pt]
 &=& \ssum{|J|=k-1} \ssum{\substack{|I_1|=k\\|I_2|=k}}\sum_{\ell_1,\,\ell_2=1}^{n} \sign{\ell_1 J}{I_1} \sign{\ell_2 J}{I_2} \int_X \delta^\phi_{\ell_1}\left(\eta_{I_1}\right) \cdot \delta^\phi_{\ell_2}\left(\eta_{I_2}\right) \expp{-\phi} \vol \nonumber\\[5pt]
 &=& \ssum{\substack{|J|=k-1\\|I_1|=k\\|I_2|=k}}\sum_{\ell_1,\,\ell_2=1}^{n} \sign{\ell_1 J}{I_1} \sign{\ell_2 J}{I_2} \nonumber\\[5pt]
 && \cdot \int_X \left(\der{\eta_{I_1}}{\ell_2}\,\der{\eta_{I_2}}{\ell_1}+\frac{\del^2\phi}{\del x^{\ell_1}\,\del x^{\ell_2}}\,\eta_{I_1}\,\eta_{I_2}\right) \, \expp{-\phi} \vol \;.\label{eq:normale-delta}
\end{eqnarray}

Now, note that
\begin{eqnarray}
 \left|\de\eta\right|^2 &=& \ssum{|H|=k+1} \left|\ssum{|I|=k}\sum_{\ell=1}^{n}\sign{\ell I}{H}\der{\eta_I}{\ell}\right|^2 \nonumber\\[5pt]
 &=& \ssum{|H|=k+1} \left(\ssum{\substack{|I_1|=k\\|I_2|=k}}\sum_{\ell_1,\,\ell_2=1}^{n}\sign{\ell_1 I_1}{H}\sign{\ell_2 I_2}{H}\der{\eta_{I_1}}{\ell_1}\der{\eta_{I_2}}{\ell_2}\right) \nonumber\\[5pt]
 &=& \ssum{\substack{|I_1|=k\\|I_2|=k}}\sum_{\ell_1,\,\ell_2=1}^{n}\sign{\ell_1 I_1}{\ell_2 I_2}\der{\eta_{I_1}}{\ell_1}\der{\eta_{I_2}}{\ell_2} \nonumber\\[5pt]
 &=& \ssum{|I|=k}\sum_{\ell=1}^{n} \left|\der{\eta_I}{\ell}\right|^2 - \ssum{\substack{|J|=k-1\\|I_1|=k\\|I_2|=k}} \sum_{\ell_1,\,\ell_2=1}^{n} \sign{\ell_1 J}{I_1} \sign{\ell_2 J}{I_2} \der{\eta_{I_1}}{\ell_2} \der{\eta_{I_2}}{\ell_1} \;.\label{eq:norma-de}
\end{eqnarray}

Hence, in view of \eqref{eq:norma-de}, \eqref{eq:normale-delta}, \eqref{eq:stima-delta}, we get
\begin{eqnarray*}
 \lefteqn{\int_X \ssum{\substack{|J|=k-1\\|I_1|=k\\|I_2|=k}}\sum_{\ell_1,\,\ell_2=1}^{n} \sign{\ell_1 J}{I_1}\sign{\ell_2 J}{I_2} \frac{\del^2\phi}{\del x^{\ell_1}\,\del x^{\ell_2}}\, \eta_{I_1}\, \eta_{I_2} \, \expp{-\phi} \vol} \\[5pt]
 && \leq\; \int_X \left(\ssum{\substack{|J|=k-1\\|I_1|=k\\|I_2|=k}}\sum_{\ell_1,\,\ell_2=1}^{n} \sign{\ell_1 J}{I_1}\sign{\ell_2 J}{I_2} \frac{\del^2\phi}{\del x^{\ell_1}\,\del x^{\ell_2}}\, \eta_{I_1}\, \eta_{I_2}+\ssum{|I|=k}\sum_{\ell=1}^{n}\left|\frac{\del \eta_I}{\del x^\ell}\right|^2\right) \, \expp{-\phi} \vol \\[5pt]
 && =\; \int_X \left(\ssum{|J|=k-1} \left| \ssum{|I|=k}\sum_{\ell=1}^{n}\sign{\ell J}{I} \delta^\phi_\ell\left(\eta_I\right) \right|^2 + \ssum{|H|=k+1}\left|\left(\de \eta\right)_H\right|^2\right) \, \expp{-\phi} \vol \\[5pt]
 && \leq \; C \cdot \left( \normaL{\destar{\phi-\psi}{\phi-2\psi}\eta}{\phi-2\psi}^2 + \normaL{\de\eta}{\phi}^2 + \int_X \ssum{|I|=k}\sum_{\ell=1}^{n}\left|\frac{\del\psi}{\del x^\ell}\right|^2\,\left|\eta_I\right|^2\, \expp{-\phi} \vol \right) \;,
\end{eqnarray*}
concluding the proof.
\end{proof}

\medskip

Using the previous result, we prove here the following theorem, \cite[Theorem 3.1]{angella-calamai}.

\begin{thm}
\label{thm:cauchy}
 Let $X$ be a strictly $p$-convex domain in $\R^n$, and fix $k\in\N$ such that $k\geq p$. Then, every $\de$-closed $k$-form $\eta\in\wedge^kX$ is $\de$-exact, namely, there exists $\alpha\in\wedge^{k-1}X$ such that $\eta=\de\alpha$.
\end{thm}

\begin{proof}
Let us split the proof in the following steps.

\paragrafod{1}{Preliminary definitions}
$X$ being a strictly $p$-convex domain in $\R^n$, by F.~R. Harvey and H.~B. Lawson's \cite[Theorem 4.8]{harvey-lawson-1} (see also \cite[Theorem 5.4]{harvey-lawson-2}), there exists a smooth proper strictly $p$-pluri-sub-harmonic exhaustion function
$$ \rho \;\in\; \intern{\PSH^{(1)}_p\left( X,\, g\right)} \cap \Cinfk{0} \;,$$
where $g$ is the metric on $X$ induced by the Euclidean metric on $\R^n$.

For every $m\in\N$, consider the compact set
$$ K^{(m)} \;:=\; \left\{x\in X \st \rho(x)\leq m \right\} \;,$$
and define
$$ L^{(m)} \;:=\; \min_{K^{(m)}} \lambda_{1}^{[k]} \;>\; 0 \;,$$
where, for every $x\in X$, the real numbers $\lambda_{1}^{[k]}(x)\leq\cdots\leq \lambda_{\binom{n}{k}}^{[k]}(x)$ are the ordered eigen-values of the endomorphism $D^{[k]}_{g^{-1}\Hess \rho(x)}\in\Homom{\wedge^kT_xX}{\wedge^kT_xX}$, and $\lambda_1(x)\leq\cdots\leq\lambda_n(x)$ are the ordered eigen-values of the endomorphism $g^{-1}\Hess \rho(x)\in\Homom{T_xX}{T_xX}$; indeed, note that, for every $x\in X$, since $\rho$ is strictly $p$-pluri-sub-harmonic,
$$ \lambda_{1}^{[k]}(x) \;=\; \lambda_1(x)+\cdots+\lambda_{k}(x) \;\geq\; \lambda_1(x)+\cdots+\lambda_p(x) \;>\; 0 \;, $$
and that the function $X\ni x\mapsto \lambda_{1}^{[k]}(x)\in\R$ is continuous.

Fix $\left\{\rho_\nu\right\}_{\nu\in\N}\subset \Cinfkc{0}$ such that
\begin{inparaenum}[(\itshape i\upshape)]
 \item $0\leq \rho_\nu \leq 1$ for every $\nu\in\N$, and
 \item for every compact set $K\subseteq X$, there exists $\nu_0:=:\nu_0(K)\in\N$ such that $\rho_\nu\lfloor_{K}=1$ for every $\nu\geq \nu_0$.
\end{inparaenum}

Then, we can choose $\psi\in\Cinfk{0}$ such that, for every $\nu\in\N$,
$$ \left|\de \rho_\nu\right|^2 \;\leq\; \expp{\psi} \;.$$

For every $m\in\N$, set
$$ \gamma^{(m)} \;:=\; \max_{K^{(m)}} \left(C\cdot \left|\de\psi\right|^2+\expp{\psi}\right) \;, $$
where $C:=:C(n,k)$ is the constant in Proposition \ref{prop:stima}.

Fix $\chi\in\mathcal{C}^\infty\left(\R;\R\right)$ such that 
\begin{inparaenum}[(\itshape i\upshape)]
 \item $\chi'>0$,
 \item $\chi''>0$, and
 \item $\chi'\lfloor_{\left(-\infty,\,m\right]}>\frac{\gamma^{(m)}}{L^{(m)}}$, for every $m\in\N$.
\end{inparaenum}
Define
$$ \phi \;:=\; \chi \circ \rho \;; $$
then, $\phi\in \intern{\PSH^{(1)}_p\left( X,\, g\right)} \cap \Cinfk{0}$; furthermore
$$ \frac{\del^2\phi}{\del x^{\ell_1}\del x^{\ell_2}} \;=\; \chi''\circ\rho \cdot \frac{\del\rho}{\del x^{\ell_1}} \cdot \frac{\del\rho}{\del x^{\ell_2}}+\chi'\circ\rho \cdot \frac{\del^2\rho}{\del x^{\ell_1}\del x^{\ell_2}} \;.$$

Choose $\mu\in\Cinfk{0}$ such that, for every $m\in\N$,
$$ \chi'\circ \rho\lfloor_{K^{(m)}}\cdot L^{(m)} \;\geq\; \mu\lfloor_{K^{(m)}} \;\geq\; \gamma^{(m)} \;.$$

\paragrafod{2}{For every $\eta\in\mathcal{C}^\infty_{\mathrm{c}}\left(X;\wedge^kT^*X\right)$, it holds $\normaL{\eta}{\phi-\psi}^2\leq C \cdot \left(\normaL{\destar{\phi-\psi}{\phi-2\psi}\eta}{\phi-2\psi}^2 + \normaL{\de\eta}{\phi}^2 \right)$}
Since
$$ D^{\left[k\right]}_{g^{-1}\Hess \rho} \;=\; \left(\ssum{|J|=k-1}\sum_{\ell_1,\, \ell_2=1}^{n} \sign{\ell_1 J}{I_1} \sign{\ell_2 J}{I_2} \frac{\del^2\rho}{\del x^{\ell_1}\del x^{\ell_2}}\right)_{I_1, I_2} \;\in\; \Homom{\wedge^{k}TX}{\wedge^{k}TX} \;,$$
then, by \texttt{Step 1}, one has the estimate
\begin{eqnarray*}
 \lefteqn{\ssum{\substack{|J|=k-1\\|I_1|=k\\|I_1|=k}}\sum_{\ell_1,\, \ell_2=1}^{n} \sign{\ell_1 J}{I_1} \sign{\ell_2 J}{I_2} \frac{\del^2\phi}{\del x^{\ell_1} \del x^{\ell_2}}\,\eta_{I_1}\,\eta_{I_2}} \\[5pt]
 &=& \ssum{\substack{|J|=k-1\\|I_1|=k\\|I_1|=k}}\sum_{\ell_1,\, \ell_2=1}^{n} \sign{\ell_1 J}{I_1} \sign{\ell_2 J}{I_2} \chi''\circ \rho \cdot \der{\rho}{\ell_1}\der{\rho}{\ell_2}\eta_{I_1}\,\eta_{I_2} \\[5pt]
 && +\; \ssum{\substack{|J|=k-1\\|I_1|=k\\|I_1|=k}}\sum_{\ell_1,\, \ell_2=1}^{n} \sign{\ell_1 J}{I_1} \sign{\ell_2 J}{I_2} \chi'\circ \rho\cdot \frac{\del^2\rho}{\del x^{\ell_1}\del x^{\ell_2}}\, \eta_{I_1}\, \eta_{I_2} \\[5pt]
 &=& \ssum{|J|=k-1}\chi''\circ \rho \cdot \left|\ssum{|I|=k}\sum_{\ell=1}^{n}\sign{\ell J}{I}\der{\rho}{\ell}\eta_I \right|^2 \\[5pt]
 && +\; \chi'\circ\rho \cdot \ssum{\substack{|J|=k-1\\|I_1|=k\\|I_1|=k}}\sum_{\ell_1,\, \ell_2=1}^{n} \sign{\ell_1 J}{I_1} \sign{\ell_2 J}{I_2} \frac{\del^2\rho}{\del x^{\ell_1}\del x^{\ell_2}}\, \eta_{I_1}\, \eta_{I_2} \\[5pt]
 &\geq& \chi'\circ \rho \cdot \lambda_{1}^{[k]}(x) \cdot \ssum{|I|=k}\left|\eta_I\right|^2 \\[5pt]
 &\geq& \mu \cdot \ssum{|I|=k}\left|\eta_I\right|^2 \;.
\end{eqnarray*}

Hence, using Proposition \ref{prop:stima}, we get that, for every $\eta\in\Cinfkc{k}$,
\begin{eqnarray*}
 \normaL{\eta}{\phi-\psi}^2 &=& \int_X \ssum{|I|=k}\left|\eta_I\right|^2\, \expp{-\left(\phi-\psi\right)}\vol \\[5pt]
 &\leq& \int_X \ssum{|I|=k} \left(\mu-C\cdot\sum_{\ell=1}^{n}\left|\frac{\del\psi}{\del x^\ell}\right|^2\right)\cdot \left|\eta_I\right|^2 \, \expp{-\phi} \vol \\[5pt]
 &\leq& \int_X \left(\ssum{\substack{|J|=k-1\\|I_1|=k\\|I_2|=k}}\sum_{\ell_1,\,\ell_2=1}^{n} \sign{\ell_1 J}{I_1}\sign{\ell_2 J}{I_2} \frac{\del^2\phi}{\del x^{\ell_1}\,\del x^{\ell_2}}\, \eta_{I_1}\, \eta_{I_2} \right. \\[5pt]
 && \left. - C\cdot \ssum{|I|=k}\sum_{\ell=1}^{n}\left|\frac{\del\psi}{\del x^\ell}\right|^2\,\left|\eta_I\right|^2\right)\, \expp{-\phi} \vol \\[5pt]
 &\leq&  C \cdot \left( \normaL{\destar{\phi-\psi}{\phi-2\psi}\eta}{\phi-2\psi}^2 + \normaL{\de\eta}{\phi}^2 \right) \;, 
\end{eqnarray*}
where $C:=:C(k,n)\in\N$ is the constant in Proposition \ref{prop:stima}, depending just on $k$ and $n$.

\paragrafod{3}{The space $\mathcal{C}^\infty_{\mathrm{c}}\left(X;\,\wedge^kT^*X\right)$ is dense in the space $\dom\de\cap\dom\destar{\phi-\psi}{\phi-2\psi}$ endowed with the norm $\normaL{\sspace}{\phi-\psi}+\normaL{\destar{\phi-\psi}{\phi-2\psi}\sspace}{\phi-2\psi}+\normaL{\de\sspace}{\phi}$}
Consider
$$
\xymatrix{
\Leb{k-1}{\phi-2\psi} \ar@{-->}@/^1pc/[r]^{\de} & \Leb{k}{\phi-\psi} \ar@{-->}@/^1pc/[r]^{\de} \ar@{-->}@/^1pc/[l]^{\destar{\phi-\psi}{\phi-2\psi}} & \Leb{k+1}{\phi} \ar@{-->}@/^1pc/[l]^{\destar{\phi-\psi}{\phi-2\psi}} \;.
}
$$
Fix $\eta\in\dom\de\cap\dom\destar{\phi-\psi}{\phi-2\psi}\subseteq\Leb{k}{\phi-\psi}$.

Firstly, we prove that $\left\{\rho_\nu\,\eta\right\}_{\nu\in\N}\subset\dom\de\cap\dom\destar{\phi-\psi}{\phi-2\psi}\subseteq\Leb{k}{\phi-\psi}$,  where $\left\{\rho_\nu\right\}_{\nu\in\N}\subset\Cinfkc{0}$ has been defined in \texttt{Step 1}, is a sequence of functions having compact support and converging to $\eta$ in the graph norm $\normaL{\sspace}{\phi-\psi}+\normaL{\destar{\phi-\psi}{\phi-2\psi}\sspace}{\phi-2\psi}+\normaL{\de\sspace}{\phi}$. Indeed,
\begin{eqnarray*}
 \left| \de\left(\rho_\nu\,\eta\right)-\rho_\nu\,\de\eta\right|^2 \,\expp{-\phi} &=& \left|\eta\right|^2\cdot\left|\de\rho_\nu\right|^2 \,\expp{-\phi} \\[5pt]
 &\leq& \left|\eta\right|^2\,\expp{-\left(\phi-\psi\right)} \;\in\; \Leb{k}{0} \;,
\end{eqnarray*}
hence, by the Lebesgue dominated convergence theorem, $\normaL{\de\left(\rho_\nu\,\eta\right)-\rho_\nu\,\de\eta}{\phi}\to 0$ as $\nu\to+\infty$. Furthermore, for every $\nu\in\N$, note that $\rho_\nu\,\eta\in\dom\destar{\phi-\psi}{\phi-2\psi}$: indeed, the map
$$ \Leb{k-1}{\phi-2\psi} \;\supseteq\; \dom \de \;\ni\; u \mapsto \scalarL{\rho_\nu\,\eta}{\de u}{\phi-\psi} \;\in\; \R $$
is continuous, being
\begin{eqnarray*}
 \scalarL{\rho_\nu\,\eta}{\de u}{\phi-\psi} &=& \scalarL{\eta}{\de\left(\rho_\nu\, u\right)}{\phi-\psi} - \scalarL{\eta}{\de\rho_\nu\wedge u}{\phi-\psi} \\[5pt]
 &=& \scalarL{\rho_\nu\,\destar{\phi-\psi}{\phi-2\psi}\eta}{u}{\phi-2\psi} - \scalarL{\eta}{\de\rho_\nu\wedge u}{\phi-\psi} \;,
\end{eqnarray*}
hence, by the Riesz representation theorem, there exists
$$ \tilde\eta \;=:\; \destar{\phi-\psi}{\phi-2\psi}\left(\rho_\nu\,\eta\right)\in \Leb{k-1}{\phi-2\psi} $$
such that, for every $u\in \dom\de \subseteq \Leb{k-1}{\phi-2\psi}$, it holds
$$ \scalarL{\rho_\nu\,\eta}{\de u}{\phi-\psi} \;=\; \scalarL{\tilde\eta}{u}{\phi-2\psi} \;.$$
Finally, note that, for every $u\in\dom\de\subseteq \Leb{k-1}{\phi-2\psi}$,
\begin{eqnarray*}
 \left|\scalarL{\destar{\phi-\psi}{\phi-2\psi}\left(\rho_\nu\,\eta\right)-\rho_\nu\,\destar{\phi-\psi}{\phi-2\psi}\,\eta}{u}{\phi-2\psi}\right| &=& \left|\scalarL{\rho_\nu\,\eta}{\de u}{\phi-\psi}-\scalarL{\destar{\phi-\psi}{\phi-2\psi}\eta}{\rho_\nu\,u}{\phi-2\psi}\right| \\[5pt]
 &=& \left|\scalarL{\eta}{\de\rho_\nu\wedge u}{\phi-\psi}\right| \\[5pt]
 &\leq& \normaL{\eta}{\phi-\psi}\cdot\normaL{\de\rho_\nu\wedge u}{\phi-\psi} \,,
\end{eqnarray*}
hence, by the Lebesgue dominated convergence theorem, $\normaL{\destar{\phi-\psi}{\phi-2\psi}\left(\rho_\nu\,\eta\right)-\rho_\nu\,\destar{\phi-\psi}{\phi-2\psi}\eta}{\phi-2\psi}\to 0$ as $\nu\to+\infty$. This shows that $\rho_\nu\,\eta\to \eta$ as $\nu\to+\infty$ with respect to the graph norm.

Hence, we may suppose that $\eta\in\dom\de\cap\dom\destar{\phi-\psi}{\phi-2\psi}\subseteq\Leb{k}{\phi-\psi}$ has compact support. Let $\left\{\Phi_\varepsilon\right\}_{\varepsilon\in\R\setminus\{0\}}\subseteq\mathcal{C}^\infty\left(\R^n;\R\right)$ be a family of positive mollifiers, that is, for every $\varepsilon\in\R\setminus\{0\}$,
$$ \Phi_\varepsilon \;:=\; \varepsilon^{-n}\,\Phi\left(\frac{\sspace}{\varepsilon}\right) \;\in\; \mathcal{C}^\infty\left(\R^n;\R\right) \;,$$
where
\begin{inparaenum}[(\itshape i\upshape)]
 \item $\Phi\in\mathcal{C}^\infty_{\textrm{c}}\left(\R^n;\R\right)$,
 \item $\int_{\R^n}\Phi \vol_{\R^n} =1$,
 \item $\lim_{\varepsilon\to 0}\Phi_\varepsilon=\delta$, where $\delta$ is the Dirac delta function, and
 \item $\Phi\geq 0$.
\end{inparaenum}
Consider the convolution
$$ \left\{\eta * \Phi_\varepsilon\right\}_{\varepsilon\in \R}\subset \Cinfkc{k} \;;$$
we prove that $\eta * \Phi_\varepsilon\to \eta$ as $\varepsilon\to 0$ with respect to the graph norm. Clearly,
$$ \lim_{\varepsilon\to 0} \normaL{\eta-\eta*\Phi_\varepsilon}{\phi-\psi} \;=\; 0 \;. $$
Since $\de\left(\eta*\Phi_\varepsilon\right)=\de\eta*\Phi_\varepsilon$, one has that
$$ \lim_{\varepsilon\to 0} \normaL{\de\left(\eta*\Phi_\varepsilon\right)-\de\eta}{\phi} \;=\; 0 \;. $$
Finally, write
$$ \destar{\phi-\psi}{\phi-2\psi} \;=:\; \expp{-\psi} \left(\destar{0}{0}+A_{\phi-\psi,\,\phi-2\psi}\right) \;,$$
where $\destar{0}{0}$
is a differential operator with constant coefficients, and $A_{\phi-\psi,\,\phi-2\psi}$ is a differential operator of order zero defined, for every $v\in \Leb{k}{\phi-\psi}$, as
$$ A_{\phi-\psi,\,\phi-2\psi}\left(v\right) \;:=\;  \ssum{\substack{|J|=k-1\\|I|=k}}\sum_{\ell=1}^{n}\sign{\ell J}{I} \der{\left(\phi-\psi\right)}{\ell}\cdot \eta\de x^J \;; $$
hence
\begin{eqnarray*}
\left(\destar{0}{0} +A_{\phi-\psi,\,\phi-2\psi}\right) \left(\eta*\Phi_\varepsilon\right) &=& \left(\left(\destar{0}{0}+A_{\phi-\psi,\,\phi-2\psi}\right) \left(\eta\right)\right) *\Phi_\varepsilon - \left(A_{\phi-\psi,\,\phi-2\psi}\eta\right)*\Phi_\varepsilon + A_{\phi-\psi,\,\phi-2\psi}\left(\eta*\Phi_\varepsilon\right) \\[5pt]
 &\to& \left(\destar{0}{0}+A_{\phi-\psi,\,\phi-2\psi}\right)\left(\eta\right)
\end{eqnarray*}
as $\varepsilon\to0$ in $\Leb{k-1}{\phi-2\psi}$; since $\eta$ has compact support, it follows that
$$ \destar{\phi-\psi}{\phi-2\psi}\left(\eta*\Phi_\varepsilon \right)\to \destar{\phi-\psi}{\phi-2\psi}\left(\eta\right) $$
as $\varepsilon\to 0$ in $\Leb{k-1}{\phi-2\psi}$, that is,
$$ \lim_{\varepsilon\to 0} \normaL{\destar{\phi-\psi}{\phi-2\psi}\left(\eta*\Phi_\varepsilon \right)-\destar{\phi-\psi}{\phi-2\psi}\left(\eta\right)}{\phi-2\psi} \;=\; 0 \;.$$

\paragrafod{4}{If $\normaL{\eta}{\phi-\psi}^2\leq C\cdot\left(\normaL{\destar{\phi-\psi}{\phi-2\psi}\eta}{\phi-2\psi}^2+\normaL{\de\eta}{\phi}^2\right)$ holds for every $\mathcal{C}^\infty_{\mathrm{c}}\left(X;\,\wedge^kT^*X\right)$, then it holds for every $\eta\in\dom\de\cap\dom\destar{\phi-\psi}{\phi-2\psi}$}
Let $\eta\in\dom\de\cap\dom\destar{\phi-\psi}{\phi-2\psi}$. By \texttt{Step 3}, take $\left\{\eta_j\right\}_{j\in\N}\subset \Cinfkc{k}$ such that $\eta_j\to \eta$ as $j\to+\infty$ in the graph norm. Since, for every $j\in\N$, one has
$$ \normaL{\eta_j}{\phi-\psi}^2 \;\leq\; C\cdot\left(\normaL{\destar{\phi-\psi}{\phi-2\psi}\eta_j}{\phi-2\psi}^2+\normaL{\de\eta_j}{\phi}^2\right) \;,$$
and since, for $j\to+\infty$,
$$ \normaL{\eta_j-\eta}{\phi-\psi}\to 0 \;, \qquad \normaL{\destar{\phi-\psi}{\phi-2\psi}\eta_j-\destar{\phi-\psi}{\phi-2\psi}\eta}{\phi-2\psi}\to 0 \;, \qquad \text{ and } \qquad \normaL{\de\eta_j-\de\eta}{\phi}\to 0 \;, $$
we get that also
$$ \normaL{\eta}{\phi-\psi}^2 \;\leq\; C\cdot\left(\normaL{\destar{\phi-\psi}{\phi-2\psi}\eta}{\phi-2\psi}^2+\normaL{\de\eta}{\phi}^2\right) \;.$$

\paragrafod{5}{Existence of a solution in $\mathrm{L}^2_{\mathrm{loc}}\left(X;\wedge^{k}T^*X\right)$}
We prove here that the operator
$$ \de\colon \Leb{k-1}{\phi-2\psi} \dashrightarrow \ker\left(\de\colon \Leb{k}{\phi-\psi}\dashrightarrow\Leb{k+1}{\phi}\right) $$
is surjective, hence, for every $\eta\in\ker\left(\de\colon \Leb{k}{\phi-\psi}\dashrightarrow\Leb{k+1}{\phi}\right)$, the equation $\de\alpha=\eta$ has a solution $\alpha$ in $\Leb{k-1}{\phi-\psi}\subseteq \Lebloc{k-1}$.

We recall, see, e.g., \cite[Lemma 4.1.1]{hormander}, that, given the Hilbert spaces $\left(H_1,\,\left\langle \sspace,\, \ssspace\right\rangle_{H_1}\right)$ and $\left(H_2,\,\left\langle \sspace,\, \ssspace\right\rangle_{H_2}\right)$, and a densely-defined closed operator $T\colon H_1 \dashrightarrow H_2$, whose adjoint is $T^*\colon H_2\dashrightarrow H_1$, if $F\subseteq H_2$ is a closed subspace such that $\imm T\subseteq F$, then the following conditions are equivalent:
\begin{enumerate}[(\itshape i\upshape)]
 \item $\imm T=F$;
 \item there exists $C>0$ such that, for every $y\in \dom T^* \cap F$,
 $$ \left\|y\right\|_{H_2} \;\leq\; C\cdot \left\|T^*y\right\|_{H_1} \;. $$
\end{enumerate}
Hence, consider
$$ \de\colon \Leb{k-1}{\phi-2\psi} \dashrightarrow \Leb{k}{\phi-\psi} $$
and
\begin{eqnarray*}
\Leb{k}{\phi-\psi} \;\supseteq\; F &:=& \ker\left(\de\colon \Leb{k}{\phi-\psi}\dashrightarrow\Leb{k+1}{\phi}\right) \\[5pt]
  &\supseteq& \imm \left(\de\colon \Leb{k-1}{\psi-2\psi}\dashrightarrow\Leb{k}{\phi-\psi}\right) \;.
\end{eqnarray*}
By \texttt{Step 4}, for every $\eta\in \dom\destar{\phi-\psi}{\phi-2\psi}\cap F \subseteq \dom\de\cap\dom\destar{\phi-\psi}{\phi-2\psi}$, it holds that
$$ \normaL{\eta}{\phi-\psi}^2 \;\leq\; C \, \normaL{\destar{\phi-\psi}{\phi-2\psi}\eta}{\phi-2\psi}^2 \;,$$
from which it follows that
$$ F \;=\; \imm \left(\de\colon \Leb{k-1}{\psi-2\psi}\dashrightarrow\Leb{k}{\phi-\psi}\right) \;.$$

\paragrafod{6}{Sobolev regularity of the solutions with compact support}
We prove that, for every $\alpha\in\Leb{k-1}{0}$ with compact support, if $\de\alpha\in\Leb{k}{0}$ and $\destar{0}{0}\alpha\in\Leb{k-2}{0}$, then $\alpha\in\Sob{k-1}{0}{1}$. Indeed, take $\left\{\Phi_\varepsilon\right\}_{\varepsilon\in\R}$ a family of positive mollifiers and, for every $\varepsilon\in\R$, consider $\alpha*\Phi_\varepsilon\in\Cinfkc{k-1}$; by Proposition \ref{prop:stima} with $\phi:=0$ and $\psi:=0$, we get that, for any multi-index $I$ such that $|I|=k-1$ and for any $\ell\in\In{n}$,
$$
\int_X\left|\frac{\del\left(\alpha_I*\Phi_\varepsilon\right)}{\del x^\ell}\right|^2\,\vol \;\leq\; C\cdot\left(\normazeroL{\destar{0}{0} \left(\alpha*\Phi_\varepsilon\right)}^2+\normazeroL{\de \left(\alpha*\Phi_\varepsilon\right)}^2\right) \;,
$$
where $C:=:C(k,n)$ is a constant depending just on $k$ and $n$; since, for every multi-index $I$ such that $|I|=k-1$, and for every $\ell\in\In{n}$, it holds that
$$ \lim_{\varepsilon\to0}\int_X\left|\frac{\del\left(\alpha_I*\Phi_\varepsilon\right)}{\del x^\ell} - \frac{\del \alpha_I}{\del x^\ell}\right|^2\,\vol \;=\; \lim_{\varepsilon\to0}\normazeroL{\destar{0}{0} \left(\alpha*\Phi_\varepsilon\right) - \destar{0}{0} \alpha} \;=\; \lim_{\varepsilon\to0}\normazeroL{\de \left(\alpha*\Phi_\varepsilon\right) - \de \alpha} \;=\; 0 \;,$$
we get that
$$
\int_X\left|\frac{\del\alpha_I}{\del x^\ell}\right|^2\,\vol \;\leq\; C\cdot\left(\normazeroL{\destar{0}{0} \alpha}^2+\normazeroL{\de \alpha}^2\right) \;,
$$
proving the claim.

\paragrafod{7}{Regularization of the solution} By \texttt{Step 5}, if $\eta\in\wedge^kX$ is such that $\de\eta=0$, then
the equation $\de\alpha=\eta$ has a solution $\alpha\in\Lebloc{k-1}$; we prove that actually $\alpha\in\wedge^{k-1}X$.

Note that we may suppose that the solution $\alpha\in\Lebloc{k-1}$ satisfies
$$ \alpha \;\in\; \left(\ker\de\right)^{\perp_{\Lebloc{k-1}}} \;=\; \overline{\imm\destar{0}{0}} \;=\; \imm\destar{0}{0} \;\subseteq\; \ker\destar{0}{0} \;;$$
hence, $\alpha$ satisfies the system of differential equation
$$
\left\{
\begin{array}{rcl}
 \de \alpha &=& \eta \\[5pt]
 \destar{0}{0} \alpha &=& 0
\end{array}
\right. \;.
$$

We prove, by induction on $s\in\N$, that $\alpha\in\Sobloc{s}{k-1}$ for every $s\in\N$. Indeed, we have by \texttt{Step 5} that $\alpha\in\Sobloc{0}{k-1}=\Lebloc{k-1}$. Suppose now that $\alpha\in\Sobloc{s}{k-1}$ and prove that $\alpha\in\Sobloc{s+1}{k-1}$. Clearly, $\eta\in\wedge^kX\subseteq\Sobloc{\sigma}{k}$ for every $\sigma\in\N$. Take $K$ a compact subset of $X$, and choose $\widehat\chi\in\Cinfkc{0}$ such that $\supp\widehat\chi\supset K$. For any multi-index $L:=:\left(\ell_1,\ldots,\ell_n\right)\in\N^n$ such that $\ell_1+\cdots+\ell_n=s$, being
$$ \de\left(\widehat\chi\cdot\frac{\del^{s}\alpha}{\del^{\ell_1} x^1\cdots\del^{\ell_n}x^n}\right) \;=\; \de\widehat\chi\wedge\frac{\del^{s}\alpha}{\del^{\ell_1} x^1\cdots\del^{\ell_n}x^n}+\widehat\chi \cdot \frac{\del^{s}\eta}{\del^{\ell_1} x^1\cdots\del^{\ell_n}x^n} \;\in\; \LebK{k}{0} $$
and
$$ \destar{0}{0}\left(\widehat\chi \cdot \frac{\del^{s}\alpha}{\del^{\ell_1} x^1\cdots\del^{\ell_n}x^n}\right) \;=\;
- \ssum{\substack{|J|=k-1\\|I|=k}}\sum_{\ell=1}^{n}\sign{\ell J}{I}\frac{\del \widehat\chi}{\del x^\ell}\cdot \frac{\del^{s}\alpha_I}{\del^{\ell_1} x^1\cdots\del^{\ell_n}x^n} \de x^J \;\in\; \LebK{k-2}{0} $$
we get that $\widehat\chi\cdot\frac{\del^{s}\alpha}{\del^{\ell_1} x^1\cdots\del^{\ell_n}x^n}\in\SobK{k-1}{0}{1}$, that is, $\alpha\in\SobK{k-1}{0}{s+1}$. Hence, we get that $\alpha\in\Sobloc{s+1}{k-1}$.

Since $\Sobloc{\sigma}{k-1} \hookrightarrow \mathcal{C}^m\left(X;\wedge^{k-1}T^*X\right)$ for every $0\leq m< \sigma-\frac{n}{2}$, see, e.g., \cite[Corollary 7.11]{gilbarg-trudinger}, we get that $\alpha\in\wedge^{k-1}X$, concluding the proof of the theorem.
\end{proof}

As a straightforward corollary, we get the following vanishing theorem for the higher-degree de Rham cohomology groups of a strictly $p$-convex domain in $\R^n$, \cite[Theorem 3.1]{angella-calamai}; for a different proof, involving Morse theory, compare \cite[Theorem 1]{sha} by J.-P. Sha, and \cite[Theorem 1]{wu-indiana} by H. Wu, see also \cite[Proposition 5.7]{harvey-lawson-2}.

\begin{thm}[{\cite[Theorem 3.1]{angella-calamai}, see \cite[Theorem 1]{sha}, \cite[Theorem 1]{wu-indiana}, \cite[Proposition 5.7]{harvey-lawson-2}}]\label{thm:vanishing}
 Let $X$ be a strictly $p$-convex domain in $\R^n$. Then $H^k_{dR}(X;\R)=\{0\}$ for every $k\geq p$.
\end{thm}

\backmatter

\bibliographystyle{amsalpha}
\bibliography{./Bibliografia/bibliografia.bib}

\end{document}